\documentclass[a4paper,11pt]{amsart}
\usepackage[latin1]{inputenc}
\usepackage[greek,english]{babel}
\usepackage[all]{xy}
\usepackage{amscd}
\usepackage{amssymb}
\usepackage{amsthm}
\usepackage{enumitem}
\usepackage{mathrsfs,bbm}
\usepackage[dvipdfmx]{xcolor,graphicx}
\usepackage{graphics}
\usepackage{comment}
\usepackage[all]{xy}
\usepackage{amscd}
\usepackage{amssymb,amsmath,latexsym}
\usepackage{amsthm}
\usepackage{enumitem}
\usepackage{mathrsfs,bbm}
\usepackage[dvipdfmx]{graphicx}
\usepackage[dvipdfmx]{xcolor,graphicx}
\usepackage[latin1]{inputenc}
\usepackage{graphics}
\usepackage{hyperref}  
\usepackage{tikz}

\newtheoremstyle{component}{}{}{}{}{\itshape}{.}{.5em}{\thmnote{#3}{#1}}
\theoremstyle{plain}
\newtheorem{thm}{Theorem}[section]
\newtheorem{pro}[thm]{Proposition}
\newtheorem{lem}[thm]{Lemma}
\newtheorem{cor}[thm]{Corollary}
\newtheorem{rem}[thm]{Remark}

\theoremstyle{definition}

\newtheorem{thmA}{Theorem}

\newtheorem{proA}[thmA]{Proposition}

\newtheorem{thmN}{Theorem}
\newtheorem{proN}{Proposition}

\DeclareMathOperator{\ord}{ord}
\DeclareFontEncoding{OT2}{}{}
\DeclareFontSubstitution{OT2}{cmr}{m}{n}
\DeclareFontFamily{OT2}{cmr}{\hyphenchar\font45}
\DeclareFontShape{OT2}{cmr}{m}{n}{%
<5><6><7><8><9>gen*wncyr%
<10><10.95><12><14.4><17.28><20.74><24.88>wncyr10}{}
\DeclareFontShape{OT2}{cmr}{b}{n}{%
<5><6><7><8><9>gen*wncyb%
<10><10.95><12><14.4><17.28><20.74><24.88>wncyb10}{}
\newcommand{\HH}{\mathcal{H}}

\newcommand{\K}{\mathcal{K}} 

\DeclareMathAlphabet{\mathcyr}{OT2}{cmr}{m}{n}
\DeclareMathAlphabet{\mathcyb}{OT2}{cmr}{b}{n}
\SetMathAlphabet{\mathcyr}{bold}{OT2}{cmr}{b}{n}
\title{Multi-variable Admissible distributions}
\author{Kengo Fukunaga}
\address[K.~Fukunaga]{Department of Mathematics\\
Institute of Science Tokyo\\
2-12-1 Ookayama, Meguro-ku, Tokyo 152-8551, JAPAN}
\email[K.~Fukunaga]{fukunaga.k.492a@m.isct.ac.jp }
\author{Tadashi Ochiai}
\address[T.~Ochiai]{Department of Mathematics\\
Institute of Science Tokyo\\
2-12-1 Ookayama, Meguro-ku, Tokyo 152-8551, JAPAN}
\email[T.~Ochiai]{ochiai.t.1998@m.isct.ac.jp }
\thanks{The second-named author is partially supported by KAKENHI (
Grant-in-Aid for Grant-in-Aid for Challenging Exploratory Research : Grant Number 23K17561 and Grant-in-Aid for Scientic Research (B): Grant Number 23K25763) 
of JSPS. }
\usepackage{}	
\usepackage{amsmath}	
\begin{document}
\maketitle
\begin{abstract}
The theory of admissible distributions over a weight-space of one-variable was studied by Amice--V\'{e}lu and played important roles in the cyclotomic Iwasawa theory of 
non-ordinary $p$-adic Galois representations. 
In this article, we establish the multi-variable generalization of the theory of admissible distributions over a weight-space of several variables. {As an application of our theory, we construct a two-variable $p$-adic Rankin Selberg $L$-series attached to a cuspidal eigenform $f$ with slope $\alpha\geq 0$ and a Hida family $G$ as a two-variable admissible distribution with growth $(2\alpha,\alpha)$.}
\end{abstract}
 \setcounter{tocdepth}{1}
\tableofcontents 
\section{Introduction} \label{sc:Introduction}
Let us fix a prime number $p$, an embedding of $\overline{\mathbb{Q}}$ into $\mathbb{C}$, and an embedding of $\overline{\mathbb{Q}}$ into the completion $\mathbb{C}_p$ of the algebraic closure of the $p$-adic field $\mathbb{Q}_p$. 
The Iwasawa theory of the cyclotomic deformation of a motive was originally studied when the given motive is ordinary at $p$.  
Later, the theory was generalized to the situation where the given motive is non-ordinary at $p$. 
Let $\mathcal{M}$ be a motive defined over a number field $F$ and let $k$ be the fraction field of the 
ring of coefficients of the motive $\mathcal{M}$, which is a number field.  
We denote by $\widehat{k}$ the completion of $k$ in 
$\mathbb{C}_p$. We denote by $\Gamma_{F ,\mathrm{cyc}}$ the Galois group of the cyclotomic $\mathbb{Z}_p$-extension of $F$, which is isomorphic to 
the $p$-Sylow subgroup of $\mathbb{Z}^\times_p$. 
\par 
When $\mathcal{M}$ is ordinary at $p$, the conjectural $p$-adic $L$-function of $\mathcal{M}$ is an element of 
the Iwasawa algebra $\mathcal{O}_{\widehat{k}}[[\Gamma_{F ,\mathrm{cyc}}]] \otimes_{\mathcal{O}_{\widehat{k}}} \widehat{k}$, 
which is isomorphic to the algebra of bounded measures on $\Gamma_{F ,\mathrm{cyc}}$ with values in $ \widehat{k}$. 
When $\mathcal{M}$ is non-ordinary at $p$, the conjectural $p$-adic $L$-function of $\mathcal{M}$ is not necessarily contained in 
$\mathcal{O}_{\widehat{k}}[[\Gamma_{F ,\mathrm{cyc}}]] \otimes_{\mathcal{O}_{\widehat{k}}} \widehat{k}$ and it is conjectured to be 
an element of the module of admissible distributions of growth $h$ with values in $\widehat{k}$, which is much larger than 
$\mathcal{O}_{\widehat{k}}[[\Gamma_{F ,\mathrm{cyc}}]] \otimes_{\mathcal{O}_{\widehat{k}}} \widehat{k}$. 
Here, $h$ is a certain non-negative rational number. 
Note that the theory of admissible distributions on $\Gamma_{F ,\mathrm{cyc}}$ was classically established by Amice--V\'{e}lu as we will recall later. 
\par 
Thanks to the theory of Hida deformations (for ordinary cusp forms) and the theory of Coleman families (for non-ordinary cusp forms), 
the setting of Iwasawa theory was generalized and enlarged in order to cover the situation associated to these Galois deformations. 
\par 
Let $\mathcal{F}$ be a Hida deformation for $\mathrm{GL}_2 (\mathbb{Q})$ 
defined over a local algebra $R$, which is finite over a certain one-variable Iwasawa algebra over $\mathbb{Z}_p$. 
Then the two-variable $p$-adic $L$-function associated $\mathcal{F}$ constructed by Kitagawa and Greenberg--Stevens
is an element of $R [[\Gamma_{\mathbb{Q} ,\mathrm{cyc}}]]  \otimes_{\mathbb{Z}_p} \mathbb{Q}_p$. 
The ring $R [[\Gamma_{\mathbb{Q} ,\mathrm{cyc}}]]  \otimes_{\mathbb{Z}_p} \mathbb{Q}_p$ is identified with the algebra of 
bounded measures on $\Gamma_{\mathbb{Q} ,\mathrm{cyc}}$ with values in $R  \otimes_{\mathbb{Z}_p} \mathbb{Q}_p$. 
A recipe to construct an element in the algebra of bounded measures on $\Gamma_{\mathbb{Q} ,\mathrm{cyc}}$ with values in 
$R  \otimes_{\mathbb{Z}_p} \mathbb{Q}_p$ and the way to characterize an element of $R [[\Gamma_{\mathbb{Q} ,\mathrm{cyc}}]] 
\otimes_{\mathbb{Z}_p} \mathbb{Q}_p$ is more or less parallel to the above case of bounded measures on $\Gamma_{\mathbb{Q} ,\mathrm{cyc}}$ with values in 
${\mathcal{O}_{\widehat{k}}}$ which was classically well-known. The case of $R$ which is associated to a general nearly ordinary deformation is also similar.   
\par 
In the non-ordinary situation, the situation is quite different. In this case, 
the $p$-adic $L$-function is {a one-variable admissible distribution} when the ring of coefficients $R$ is 
a discrete valuation ring ${\mathcal{O}_{\widehat{k}}}$, and a recipe to construct {a one-variable admissible distribution} and the way to characterize {a one-variable admissible distribution} is more complicated than the case of $\mathcal{O}_{\widehat{k}}[[\Gamma_{F ,\mathrm{cyc}}]] \otimes_{\mathcal{O}_{\widehat{k}}} \widehat{k}$, but this was already studied by the classical theory of Amice--V\'{e}lu.  
However, if we consider the situation of a more general non-ordinary Galois deformation over the deformation ring $R$ which is not 
a discrete valuation ring, the theory of the space where the $p$-adic $L$-function is contained, 
as well as a recipe to construct an element of this space and the way to characterize the element are not found in any references and it seems that the theory which we can use to construct a $p$-adic $L$-function for a non-ordinary Galois deformation was still missing. 
Also, we would like to establish the theory which will be a multi-variable generalization of the theory of 
Amice--V\'{e}lu and which we can use to construct a $p$-adic $L$-function  
of a non-ordinary Galois deformation. In \S\ref{Construction of a two-variable padic rankin selber l series}, we will apply our theory to construct a $p$-adic $L$-function 
of a non-ordinary Galois deformation space.  
\\   
\  
\par 
In order to state our main results, we first recall some notation and the classical theory of Amice--V\'{e}lu. 
Let $\mathcal{K}$ be a complete subfield of $\mathbb{C}_p$ and $\mathcal{O}_\mathcal{K}$ the ring of integers in $\mathcal{K}$. Typical examples of such fields $\mathcal{K}$ 
are $\mathbb{C}_p$, a finite extension of $\mathbb{Q}_p$ or the completion $\widehat{\mathbb{Q}^{\mathrm{ur}}_p}$ of the maximal unramified extension 
$\mathbb{Q}^{\mathrm{ur}}_p$ of $\mathbb{Q}_p$. 
Let $\ord_{p}$ be the $p$-adic order on $\mathbb{C}_{p}$ such that $\ord_{p}(p)=1$. For $h \in \ord_{p}(\mathcal{O}_{\K}\backslash \{0\})$, we define 
\begin{equation}\label{equation:thedefinition_of_H_h}
\HH_{h/\K} =\Big\{\sum^{+ \infty}_{n = 0} a_n X^n \in \K[ [X] ] 
\ \Big\vert \ 
\inf \big\{\ord_p (a_n ) +h \tfrac{\log n }{\log p} 
\big \}_{n \in\mathbb{Z}_{> 0}}
  >-\infty\Big\}
\end{equation}
and call an element of $\HH_{h/\K}$ a power series of logarithmic order $h$. We note that we have $fg \in \HH_{h\slash \K}$ for $f\in \mathcal{O}_{\K}[[X]]\otimes_{\mathcal{O}_{\K}}\K$ and $g\in \HH_{h\slash \K}$. Hence, $\HH_{h\slash \K}$ is an $\mathcal{O}_{\K}[[X]]\otimes_{\mathcal{O}_{\K}}\K$-module.
\par 
Let $\Gamma$ be a $p$-adic Lie group which is isomorphic to $1+2p\mathbb{Z}_p\subset \mathbb{Q}_{p}^{\times}$ via a continuous character $\chi : \Gamma \longrightarrow \mathbb{Q}_{p}^{\times}$. When $p$ is odd, we have $1+2p\mathbb{Z}_p = 1+p\mathbb{Z}_p$ and $1+p\mathbb{Z}_p$ is a pro-cyclic group. 
We note that we can regard $\Gamma$ as a subgroup of $\mathcal{K}^\times$ 
through the character $\chi$. 
We take a topological generator $\gamma\in \Gamma$ and put $u=\chi(\gamma)$. We denote by $\mu_{p^{m}}$ the subgroup of $\overline{\mathbb{Q}}^{\times}$ consisting of $p^{m}$-power roots of unity with $m\in\mathbb{Z}_{\geq 0}$ and put $\mu_{p^{\infty}}=\cup_{m\geq 0}\mu_{p^{m}}$.
\par 
Let $d,e$ be integers satisfying $e\geq d$. We put $[d,e]=\{d,d+1,\ldots, e\}$. Denote by $\lfloor h\rfloor$ the largest integer which is equal to or smaller than $h$. 
\par 
The following results on the admissible distribution and the $p$-adic power series of logarithmic order 
in the one-variable situation which we will state in Theorem \ref{one variable classical uniqueness}, Theorem \ref{onevariable projectresult}, Proposition \ref{prop1}, Proposition \ref{prop2} 
below might look like copies from well-known references such as 
those of Amice--V\'{e}lu \cite{amicevelu}, Vishik \cite{vishik1976}, Perrin--Riou \cite{perrinriou1994}, Colmez \cite{colmez2010} and Bella{\"\i}che \cite{Bell2021}. 
However, the results and their proofs which we present in this article are not fully recovered by these references. 
Later in Remark \ref{remark:H_h}, we will explain more precisely about 
the relation between Theorem \ref{one variable classical uniqueness}, Theorem \ref{onevariable projectresult}, Proposition \ref{prop1}, Proposition \ref{prop2} 
in this paper and the results in the above references. Another important reason why we give detailed proofs of these results which might be 
supposed to be classical, is that, we must allow the coefficients field $\K$ of 
Theorem \ref{one variable classical uniqueness}, Theorem \ref{onevariable projectresult}, Proposition \ref{prop1}, Proposition \ref{prop2}
to be a more general Banach space. In fact, we need such extended versions as a partial step of the proof of Theorem \ref{intro main theroem 1}, Theorem \ref{intro main theorem 2}, Proposition \ref{multivariable Iii necesarry sufficient condition}, 
Theorem \ref{multi variabl admissible intro}.
\par 
The following classical theorem gives a characterization of an element of $\HH_{h/\K}$. 
\begin{thmN}\label{one variable classical uniqueness}
If $f\in\HH_{h/\K}$ satisfies $f(u^{i}\epsilon-1)=0$ for every $i\in [d,d+\lfloor h\rfloor]$ and for every $\epsilon\in \mu_{p^{\infty}}$, then $f$ is zero.
\end{thmN}
Theorem \ref{one variable classical uniqueness} is essentially due to Amice--V\'{e}lu \cite[Lemme I\hspace{-1.2pt}I. 2.5]{amicevelu}, 
but, it is a variant of \cite[Lemme I\hspace{-1.2pt}I. 2.5]{amicevelu} (see also Remark \ref{remark:H_h} (1) for a more precise situation). {Theorem \ref{one variable classical uniqueness} is a special case of Proposition \ref{generalization of uniqueness on M} which will be proved later in this paper}. 
\par 
For each $m\in \mathbb{Z}_{\geq 0}$, we put $\Omega_{m}^{[d,e]}(X)=\prod_{i=d}^{e}((1+X)^{p^{m}}-u^{ip^{m}})\in \mathcal{O}_{\K}[X]$. We define an $\mathcal{O}_{\K}[[X]]\otimes_{\mathcal{O}_{\K}}\K$-module $J_{h}^{[d,e]}$ to be
\begin{multline}\label{intro powe Ihd1d2}
J_{h}^{[d,e]}=\Bigg\{(s_{m}^{[d,e]})_{m}\in \varprojlim_{m\in\mathbb{Z}_{\geq 0}}\left(\frac{\mathcal{O}_{\K}[[X]]}{(\Omega_{m}^{[d,e]}(X))}\otimes_{\mathcal{O}_{\K}}\K\right)
\Bigg\vert\\
 (p^{hm}s_{m}^{[d,e]})_{m}\in \left(\prod_{m=0}^{+\infty}\frac{\mathcal{O}_{\K}[[X]]}{(\Omega_{m}^{[d,e]}(X))}\right)\otimes_{\mathcal{O}_{\K}}\K 
\Bigg\}, 
\end{multline}
{where we regard $\varprojlim_{m\in\mathbb{Z}_{\geq 0}}\left(\frac{\mathcal{O}_{\K}[[X]]}{(\Omega_{m}^{[d,e]}(X))}\otimes_{\mathcal{O}_{\K}}\K\right)$ and $\left(\prod_{m=0}^{+\infty}\frac{\mathcal{O}_{\K}[[X]]}{(\Omega_{m}^{[d,e]}(X))}\right)\otimes_{\mathcal{O}_{\K}}\K$ as submodules of $\prod_{m=0}^{+\infty}\left(\frac{\mathcal{O}_{\K}[[X]]}{(\Omega_{m}^{[d,e]}(X))}\otimes_{\mathcal{O}_{\K}}\K\right)$}. The following classical theorem gives a recipe to construct an element of $\HH_{h/\K}$. 
\begin{thmN}\label{onevariable projectresult}
Assume that $e-d\geq \lfloor h\rfloor$. For $s^{[d,e]}=(s_{m}^{[d,e]})_{m\in\mathbb{Z}_{\geq 0}}\in J_{h}^{[d,e]}$, there exists a unique element $f_{s^{[d,e]}}\in \HH_{h\slash \K}$ such that 
$$f_{s^{[d,e]}}-\tilde{s}^{[d,e]}_{m}\in \Omega_{m}^{[d,e]}\HH_{h\slash \K}$$
for each $m\in\mathbb{Z}_{\geq 0}$, where $\tilde{s}^{[d,e]}_{m}\in \mathcal{O}_{\K}[[X]]\otimes_{\mathcal{O}_{\K}}\K$ is a lift of $s_{m}^{[d,e]}$. Further, the correspondence $s^{[d,e]}\mapsto f_{s^{[d,e]}}$ from $J_{h}^{[d,e]}$ to $\HH_{h\slash \K}$ induces an $\mathcal{O}_\mathcal{K}[ [X] ]\otimes_{\mathcal{O}_{\K}}\K$-module isomorphism 
$$
J_{h}^{[d,e]} \overset{\sim}{\longrightarrow} \HH_{h/\K}.
$$ 
\end{thmN}
Theorem \ref{onevariable projectresult} is also essentially due to Amice--V\'{e}lu \cite[Proposition I\hspace{-1.2pt}V. 1]{amicevelu}, but, it is a variant of \cite[Proposition I\hspace{-1.2pt}V. 1]{amicevelu} 
 (see also Remark \ref{remark:H_h} (1) for a more precise situation). {Theorem \ref{onevariable projectresult} is a special case of Proposition \ref{isomorphism HH and project lim Banach} which will be proved later in this paper.}
 \par  
 For each $m\in \mathbb{Z}_{\geq 0}$, we put $\Omega_{m}^{[d,e]}(\gamma)=\prod_{i=d}^{e}([\gamma]^{p^{m}}-u^{ip^{m}})\in \mathcal{O}_{\K}[[\Gamma]]$, where $[\ ]:\Gamma\rightarrow \mathcal{O}_{\K}[[\Gamma]]^{\times}$ is the natural inclusion. We remark that the ideal $(\Omega_{m}^{[d,e]}(\gamma))$ of $\mathcal{O}_{\K}[[\Gamma]]$ is independent of the choice of the topological generator $\gamma$ of $\Gamma$. In a similar way to $J_{h}^{[d,e]}$, we define an $\mathcal{O}_{\K}[[\Gamma ]]\otimes_{\mathcal{O}_{\K}}\K$-module $I_{h}^{[d,e]}$ to be
\begin{multline}\label{intro powe Ihd1d2+}
I_{h}^{[d,e]}=\Bigg\{(s^{[d,e]}_{m})_{m}\in \varprojlim_{m\in\mathbb{Z}_{\geq 0}}\left(\frac{\mathcal{O}_{K}[[\Gamma ]]}{(\Omega_{m}^{[d,e]}(\gamma ))}\otimes_{\mathcal{O}_{\K}}\K\right)\\
\Bigg\vert (p^{hm}s_{m}^{[d,e]})_{m}\in \left(\prod_{m=0}^{+\infty}\frac{\mathcal{O}_{K}[[\Gamma ]]}{(\Omega_{m}^{[d,e]}(\gamma ))}\right)\otimes_{\mathcal{O}_{\K}}\K 
\Bigg\}.
\end{multline}
By definition, we have a non-canonical $\K$-linear isomorphism $I_{h}^{[d,e]}\stackrel{\sim}{\rightarrow}J_{h}^{[d,e]}$ 
which extends the non-canonical continuous $\mathcal{O}_{\K}$-algebra isomorphism $\mathcal{O}_{\K}[[\Gamma]]\stackrel{\sim}{\rightarrow}\mathcal{O}_{\K}[[X]]$ characterized by $[\gamma]\mapsto 1+X$.  
\par 
{To simplify the notation, we denote $[i,i]$ by $[i]$ for each $i\in \mathbb{Z}$}. If we have a system $s^{[d,e]}=(s_{m}^{[d,e]})_{m\in\mathbb{Z}_{\geq 0}}\in I_{h}^{[d,e]}$, we obtain a system 
$(s_{m}^{[i]})_{m\in\mathbb{Z}_{\geq 0}}\in I_{h}^{[i]}$ for each integer $i \in [d,e]$ by setting $s_{m}^{[i]} \in \frac{\mathcal{O}_{K}[[\Gamma ]]}{(\Omega_{m}^{[i]}(\gamma ))}\otimes_{\mathcal{O}_{\K}}\K$ to be the image of $s_m^{[d,e]}$ by the natural projection 
$\frac{\mathcal{O}_{K}[[\Gamma ]]}{(\Omega_{m}^{[d,e]}(\gamma ))}\otimes_{\mathcal{O}_{\K}}\K\rightarrow \frac{\mathcal{O}_{K}[[\Gamma ]]}{(\Omega_{m}^{[i]}(\gamma ))}\otimes_{\mathcal{O}_{\K}}\K$. On the other hand, when we want to construct a $p$-adic $L$-function of a given motive, 
we are often given $(s_{m}^{[i]})_{m\in\mathbb{Z}_{\geq 0}}\in I_{h}^{[i]}$ for each integer $i$ contained in a fixed range $[d,e]$ related to the given motive, 
and we need to construct a projective system $s^{[d,e]}=(s_{m}^{[d,e]})_{m\in\mathbb{Z}_{\geq 0}}\in I_{h}^{[d,e]}$ 
whose projection gives the given projective system $(s_{m}^{[i]})_{m\in\mathbb{Z}_{\geq 0}}\in I_{h}^{[i]}$ for each $i \in [d,e]$. 
The following proposition gives a necessary and sufficient condition for the existence of such a system $s^{[d,e]}=(s_{m}^{[d,e]})_{m\in\mathbb{Z}_{\geq 0}}\in I_{h}^{[d,e]}$.
\begin{proN}\label{prop1}
Let $s^{[i]}=(s^{[i]}_{m})_{m\in \mathbb{Z}_{\geq 0}}\in I_{h}^{[i]}$, and let $\tilde{s}^{[i]}_{m}\in \mathcal{O}_{\K}[[\Gamma]]\otimes_{\mathcal{O}_{\K}}\K$ be a lift of $s^{[i]}_{m}$ for each $m\in \mathbb{Z}_{\geq 0}$ and for each $i\in [d,e]$.  If there exists a non-negative integer $n$ which satisfies 
\begin{equation}\label{intro onevariabble admissible I}
p^{m(h-(j-d))}\displaystyle{\sum_{i=d}^{j}}\begin{pmatrix}j-d\\i-d\end{pmatrix}(-1)^{j-i}\tilde{s}^{[i]}_{m}\in \mathcal{O}_{\K}[[\Gamma]]\otimes_{\mathcal{O}_{\K}}p^{-n}\mathcal{O}_{\K}
\end{equation}
for each $m\in \mathbb{Z}_{\geq 0}$ and for each $j\in [d,e]$, we have a unique element $s^{[d,e]}\in I_{h}^{[d,e]}$ such that the image of $s^{[d,e]}$ by the natural projection $I_{h}^{[d,e]}\rightarrow I_{h}^{[i]}$ is $s^{[i]}$ for each $i\in [d,e]$.
\end{proN}
Let $s^{[d,e]}=(s^{[d,e]}_{m})_{m\in \mathbb{Z}_{\geq 0}}$ be an element of $I_{h}^{[d,e]}$. 
For every integer $i\in [d,e]$, we denote by $s^{[i]}=(s_{m}^{[i]})_{m\in\mathbb{Z}_{\geq 0}}\in I_{h}^{[i]}$ the projection of the element $s^{[d,e]}$ to the $(i)$-component. 
Then, there exists a non-negative integer $n$ and a lift $\tilde{s}^{[i]}_{m}$ of $s^{[i]}_{m}$ for each $m\in \mathbb{Z}_{\geq 0}$ and for each $i\in [d,e]$ which satisfy \eqref{intro onevariabble admissible I}. Indeed, by the definition of $s^{[d,e]}$, there exists a non-negative integer $n$ such that $s^{[d,e]}\in \left(\prod_{m\in \mathbb{Z}_{\geq 0}}\frac{\mathcal{O}_{\K}[[\Gamma]]}{(\Omega_{m}^{[d,e]}(\gamma))}\otimes_{\mathcal{O}_{\K}}p^{-mh}\mathcal{O}_{\K}\right)\otimes_{\mathcal{O}_{\K}}p^{-n}\mathcal{O}_{\K}$. Then, 
for every $m\in \mathbb{Z}_{\geq 0}$, we have a lift
$\tilde{s}_{m}^{[d,e]}\in \mathcal{O}_{\K}[[\Gamma]]\otimes_{\mathcal{O}_{\K}}p^{-hm-n}\mathcal{O}_{\K}$ of $s_{m}^{[d,e]}$. If we take a lift $\tilde{s}_{m}^{[i]}$ of $s_{m}^{[i]}$ to be $\tilde{s}_{m}^{[d,e]}$ for each $m\in \mathbb{Z}_{\geq 0}$ and $i\in [d,e]$, we see that $\tilde{s}_{m}^{[i]}$ satisfies \eqref{intro onevariabble admissible I}.
\par 
In \cite{amicevelu}, Amice--V\'{e}lu developed a similar argument as Proposition \ref{prop1} in a more special setting. 
In fact, Amice--V\'{e}lu constructed a one-variable $p$-adic $L$-function for an elliptic eigen cusp form with positive slope in \cite[Theorem I\hspace{-1.2pt}I\hspace{-1.2pt}I]{amicevelu}. 
We can find a similar argument as Proposition \ref{prop1} in the proof of \cite[Theorem I\hspace{-1.2pt}I\hspace{-1.2pt}I]{amicevelu}. 
%
Proposition \ref{prop1} is a special case of Proposition \ref{one variable I(i) banach sufficient cond} which will be proved later in this paper.
\par 
Let us explain an interpretation of $I_{h}^{[d,e]}$ as a space of distributions. 
We denote by $C^{[d,e]} (\Gamma , \mathcal{O}_\mathcal{K})$ the $\mathcal{O}_\mathcal{K}$-module of functions $f:\Gamma\rightarrow \mathcal{O}_{\K}$ such that $\chi(x)^{-d}f(x)$ is a locally polynomial function of degree at most $e-d$ (see \S \ref{preparation} for the precise definition of 
locally polynomial functions). Let $\mathcal{D}^{[d,e]}_h (\Gamma , \mathcal{K})$ be the $\K$-vector space 
of elements of $\mathrm{Hom}_{\mathcal{O}_{\mathcal{K}}}( C^{[d,e]} (\Gamma,\linebreak \mathcal{O}_\mathcal{K}), \mathcal{K})$
which are $[d,e]$-admissible distributions of growth $h$ (see \eqref{defof admissible distri space} of this paper for the precise definition of $[d,e]$-admissible distributions of growth $h$). Put $LC(\Gamma,\mathcal{O}_{\K})=C^{[0,0]}(\Gamma,\mathcal{O}_{\K})$ and $\mathrm{Meas}(\Gamma,\mathcal{O}_{\K})=\mathrm{Hom}_{\mathcal{O}_{\K}}(LC(\Gamma,\mathcal{O}_{\K}),\mathcal{O}_{\K})$. 
The $\mathcal{O}_{\K}$-module $\mathrm{Meas}(\Gamma,\mathcal{O}_{\K})$ is an $\mathcal{O}_{\K}$-algebra by the convolution product of measures 
and we regard $\mathcal{D}^{[d,e]}_h (\Gamma , \mathcal{K})$ as a $\mathrm{Meas}(\Gamma,\mathcal{O}_{\K})\otimes_{\mathcal{O}_{\K}}\K$-module naturally. 
It is well-known that there exists a natural $\mathcal{O}_{\K}$-algebra isomorphism $\mathrm{Meas}(\Gamma,\mathcal{O}_{\K})\stackrel{\sim}{\rightarrow}\mathcal{O}_{\K}[[\Gamma]]$. Thus, we can regard $\mathcal{D}_{h}^{[d,e]}(\Gamma,\mathcal{\K})$ as an $\mathcal{O}_{\K}[[\Gamma]]\otimes_{\mathcal{O}_{\K}}\K$-module. Let $\mathfrak{X}_{\mathcal{O}_{\K}[[\Gamma]]}^{[d,e]}$ be the set of arithmetic specializations $\kappa$ on $\mathcal{O}_{\K}[[\Gamma]]$ with the weight $w_{\kappa}\in [d,e]$. For each $\kappa\in \mathfrak{X}_{\mathcal{O}_{\K}[[\Gamma]]}^{[d,e]}$, we denote by $\phi_{\kappa}$ and $m_{\kappa}$ the finite part of $\kappa$ and the smallest integer $m$ such that $\phi_{\kappa}$ factors through $\Gamma\slash \Gamma^{p^{m}}$.
\begin{proN}\label{prop2}
We have an $\mathcal{O}_{\K}[[\Gamma]]\otimes_{\mathcal{O}_{\K}}\K$-module isomorphism 
\begin{equation}\label{equation:mapfromDtoH}
I_{h}^{[d,e]}\stackrel{\sim}{\rightarrow} \mathcal{D}^{[d,e]}_h (\Gamma , \K)
\end{equation}
such that the image $\mu_{s^{[d,e]}}\in \mathcal{D}^{[d,e]}_h (\Gamma , \K)$ of each element $s^{[d,e]}=(s_{m}^{[d,e]})_{m\in \mathbb{Z}_{\geq 0}}\in I_{h}^{[d,e]}$ is characterized 
by  
$$
\kappa(\tilde{s}_{m_{\kappa}}^{[d,e]})= \int_\Gamma \chi^{w_{\kappa}} \phi_{\kappa} d\mu_{s^{[d,e]}} 
$$
for every $\kappa\in\mathfrak{X}_{\mathcal{O}_{\K}[[\Gamma]]}^{[d,e]}$, where $\tilde{s}_{m_{\kappa}}^{[d,e]}$ is a lift of $s_{m_{\kappa}}^{[d,e]}$.
\end{proN}
Proposition \ref{prop2} is essentially due to Vishik \cite[2.3. Theorem]{vishik1976}. Vishik essentially proved that there exists an injective map from $\mathcal{D}_{h}^{[0,h]}(\Gamma,\K)$ into $\HH_{h\slash \K}$ for each $h\in \mathbb{Z}_{\geq 0}$ in \cite[2.3. Theorem]{vishik1976}, 
and Perrin-Riou showed that this map is surjective in \cite[1.2.7. Proposition]{perrinriou1994} 
(see Remark \ref{remark:H_h} for the precise situation). {Proposition \ref{prop2} is a special case of Proposition \ref{onevariable isom Ih from Dh for banach} which will be proved later in this paper.}
\par 
Below we give several historical remarks on the relation of the above results to the classical references. 
As mentioned earlier, our notations and definitions in the above statements are different from those given in Amice--V\'{e}lu \cite{amicevelu}, Vishik \cite{vishik1976}, 
Perrin--Riou \cite{perrinriou1994}. Hence, these results in the one-variable situation are missing in classical references such as \cite{amicevelu}, 
\cite{perrinriou1994} and \cite{vishik1976} and we also need to prove it in \S \ref{section: one-variable}. 
In later references such as \cite{Bell2021}, \cite{colmez2010}, some of notations and definitions are common with ours compared to  
\cite{amicevelu}, \cite{perrinriou1994} and \cite{vishik1976}, but the statements presented in Theorem \ref{one variable classical uniqueness}, Theorem \ref{onevariable projectresult}, Proposition \ref{prop1}, Proposition \ref{prop2} in this paper are found in none of these articles, or even if we find it in one of these articles, the proof is missing. In the remark below, we will also explain these points. 
\begin{rem}\label{remark:H_h}
\begin{enumerate}
\item There might be another option of the definition of $\HH_{h/\K}$ which is obtained by replacing the condition 
$$ 
\text{``$
\inf \big\{\ord_p (a_n ) +h \tfrac{\log n }{\log p} 
\big \}_{n \in\mathbb{Z}_{> 0}}
  >-\infty$''}
  $$ 
in \eqref{equation:thedefinition_of_H_h} 
by the condition $$ 
\text{``$\ord_p (a_n ) +h \tfrac{\log n }{\log p} \rightarrow +\infty 
$ when $n\rightarrow +\infty$''.} 
$$  
 We call the latter version of $\HH_{h/\K}$ the small $o$ version and we call our version of 
$\HH_{h/\K}$ the big O version. We do not know references which prove Theorem \ref{one variable classical uniqueness} and Theorem 
\ref{onevariable projectresult} in the big O version, hence we prove these theorems in our paper. 
However, the classical reference \cite[Lemme I\hspace{-1.2pt}I. 2.5, Proposition I\hspace{-1.2pt}V. 1]{amicevelu}
already proves the small o versions of Theorem \ref{one variable classical uniqueness} and Theorem 
\ref{onevariable projectresult}. 
%
%
\par 
Similarly to the case of $\HH_{h\slash \K}$, Vishik proved that there exists an injective map from the small o version of $\mathcal{D}_{h}^{[0,h]}(\Gamma,\K)$ into the small o version $\HH_{h\slash \K}$ for each $h\in\mathbb{Z}_{\geq 0}$ in \cite[2.3. Theorem]{vishik1976}.
In Proposition \ref{prop2}, we give a slightly more general result with $\mathcal{D}_{h}^{[d,e]}(\Gamma,\K)$ for more general $d$, $e$ and $h$, 
but we work with the big O version of $\mathcal{D}_{h}^{[d,e]}(\Gamma,\K)$. 
%
%
\item 
We believe that the big O version as it is presented here will be more suitable to the future study of multi-variable Iwasawa theory 
because the module $\HH_{h/\K}$ with $h=0$ recovers the Iwasawa algebra $\mathcal{O}_\mathcal{K} [ [X] ]\otimes_{\mathcal{O}_{\K}}\K$ 
which is standard algebra in the study of (nearly) ordinary setting (If we work with the small o version, we recovers the Tate algebra $\mathcal{O}_\mathcal{K} \langle X \rangle\otimes_{\mathcal{O}_{\K}}\K$ which is not compatible with a lot of research in Iwasawa theory). Then, we prove the above theorems and propositions as special cases of our results.\label{remark:H_h1} 
%
\item In the classical references \cite{amicevelu} and \cite{vishik1976}, they discuss only the case where $h\in \mathbb{Z}_{\geq 0}$, $d=0$ and $e=h$.\label{remark:H_h3}
\item 
Even after the publication of classical references such as Amice--V\'{e}lu \cite{amicevelu}, Vishik \cite{vishik1976}, 
Perrin--Riou \cite{perrinriou1994}, there also appeared some other references related to this subject such as an article by Colmez \cite{colmez2010} and a book by Bella{\"\i}che \cite{Bell2021}. 
Compared to  \cite{amicevelu}, \cite{vishik1976}, and \cite{perrinriou1994}, 
the notations of \cite{Bell2021} and \cite{colmez2010} are closer to ours. However, in \cite{colmez2010}, 
there are no statements corresponding to 
Theorem \ref{one variable classical uniqueness}, Theorem \ref{onevariable projectresult}, Proposition \ref{prop1}, Proposition \ref{prop2}. 
In \cite[Theorem V. 6.10]{Bell2021}, we find a statement corresponding to Theorem \ref{onevariable projectresult}, but its proof is reduced to an online article which seems to be lost and missing. 
\item 
In the classical references \cite{amicevelu} and \cite{vishik1976}, they discuss only the case where $\mathcal{K}$ is equal to $\mathbb{C}_p$. 
In \cite{Bell2021}, the field of coefficients $\mathcal{K}$ is assumed to be a complete discrete valuation field. Only, the reference \cite{colmez2010} allows 
$\mathcal{K}$ to be any closed subfield of $\mathbb{C}_p$ as we do in this paper.\label{remark:H_h2} 
\end{enumerate}
\end{rem}

%
%
%
%
%
%
%
%
%
%
%
%
%

As mentioned earlier, Amice--V\'{e}lu and Vishik applied the above mentioned theory to construct the one-variable cyclotomic $p$-adic $L$-function 
associated to an elliptic cusp form which is not necessarily ordinary at $p$. 
On the other hand, we sometimes consider more general $p$-adic families of motives which is not necessarily ordinary at $p$ and which is not the cyclotomic deformation 
of a fixed motive. The most typical example of such $p$-adic families is the Coleman family mentioned earlier. 
Hence we will need the multi-variable version of the above theories in order to develop a theory of multi-variable $p$-adic $L$-functions attached to 
such general $p$-adic families of non-ordinary motives. 
In order to state our result on such multi-variable generalizations, we will prepare some notation. 
\par 
For each $i\in\mathbb{Z}_{\geq 0}$, we denote by $\ell(i)$ the smallest non-negative integer $n$ which satisfies $p^{n}>i$. By definition, we have $\ell(0)=0$ and $\ell (i)=\lfloor\frac{\log i}{\log p}\rfloor+1$ if $i\geq 1$. Let $k\in \mathbb{Z}_{\geq 1}$. Throughout this paper, for each $k$-tupule $\boldsymbol{a}$ of a set $X$, we denote by $a_{j}\in X$ the $j$-th component of $\boldsymbol{a}$. Let $\langle\ ,\ \rangle_{k}$ be the Euclidean inner product on $\mathbb{R}^{k}$ defined by $\langle \boldsymbol{a},\boldsymbol{b}\rangle_{k}=a_{1}b_{1}+\cdots+a_{k}b_{k}$ for each $\boldsymbol{a},\boldsymbol{b}\in \mathbb{R}^{k}$. Let $\boldsymbol{h}\in \ord_{p}(\mathcal{O}_{\mathcal{K}}\backslash \{0\})^{k}$. We define 
a multi-variable variant of \eqref{equation:thedefinition_of_H_h} as follows: 
 \begin{equation}\label{equation:themultivariableversion_of_H_h}
\HH_{\boldsymbol{h}/\K} =\Big\{\sum_{\boldsymbol{n}\in\mathbb{Z}_{\geq 0}^{k}} a_{\boldsymbol{n}} X^{\boldsymbol{n}} \in \K[ [X_{1},\ldots, X_{k}] ] 
\ \Big\vert \ 
\inf \big\{\ord_p (a_{\boldsymbol{n}} ) +\langle \boldsymbol{h},\ell(\boldsymbol{n})\rangle_{k} 
\big \}_{\boldsymbol{n}\in \mathbb{Z}_{\geq 0}^{k}}
  >-\infty\Big\},
\end{equation}
where $X^{\boldsymbol{n}}=X_{1}^{n_{1}}\cdots X_{k}^{n_{k}}$ and $\ell(\boldsymbol{n})=(\ell(n_{1}),\ldots, \ell(n_{k}))$. We call an element $f$ of $\HH_{\boldsymbol{h}/\K}$ a $k$-variable power series of logarithmic order $\boldsymbol{h}$. We remark that $fg\in \HH_{\boldsymbol{h}\slash \K}$ for each $f\in \mathcal{O}_{\K}[[X_{1},\ldots, X_{k}]]\otimes_{\mathcal{O}_{\K}}\K$ and $g\in \HH_{\boldsymbol{h}\slash \K}$. Then, $\HH_{\boldsymbol{h}\slash \K}$ is an $\mathcal{O}_{\K}[[X_{1},\ldots, X_{k}]]\otimes_{\mathcal{O}_{\K}}\K$-module. Further, if $k=1$ and $\boldsymbol{h}=h$, the module defined in \eqref{equation:themultivariableversion_of_H_h} is equal to the module defined in $\eqref{equation:thedefinition_of_H_h}$. This is checked by using the inequality $\frac{\log n}{\log p}\leq \ell(n)\leq \frac{\log n}{\log p}+1$ for each $n\in\mathbb{Z}_{\geq 1}$.

Let $\Gamma_{i}$ be a $p$-adic Lie group which is isomorphic to $1+2p\mathbb{Z}_p\subset \mathbb{Q}_{p}^{\times}$ via a continuous character $\chi_{i} : \Gamma_{i} \longrightarrow \mathbb{Q}_{p}^{\times}$ for each $1\leq i\leq k$. We define $\Gamma=\Gamma_{1}\times \cdots \times\Gamma_{k}$. Let $\boldsymbol{d},\boldsymbol{e}\in \mathbb{Z}^{k}$ such that $\boldsymbol{e}\geq\boldsymbol{d}$. Here the order $\geq$ on $\mathbb{Z}^{k}$ is the componentwise order. Put $[\boldsymbol{d},\boldsymbol{e}]=\prod_{i=1}^{k}[d_{i},e_{i}]$. Let $\gamma_{i}\in \Gamma_{i}$ be a topological generator and put $u_{i}=\chi_{i}(\gamma_{i})$ with $1\leq i\leq k$. For each $\boldsymbol{m}\in \mathbb{Z}_{\geq 0}^{k}$, we put $(\Omega_{\boldsymbol{m}}^{[\boldsymbol{d},\boldsymbol{e}]}(X_{1},\ldots, X_{k}))=(\Omega_{m_{1}}^{[d_{1},e_{1}]}(X_{1}),\ldots, \Omega_{m_{k}}^{[d_{k},e_{k}]}(X_{k}))\subset \mathcal{O}_{\K}[[X_{1},\ldots, X_{k}]]$. If there is no risk of confution, we write $(\Omega_{\boldsymbol{m}}^{[\boldsymbol{d},\boldsymbol{e}]})$ for $(\Omega_{\boldsymbol{m}}^{[\boldsymbol{d},\boldsymbol{e}]}(X_{1},\ldots, X_{k}))$. We define 
a multi-variable version $J_{\boldsymbol{h}}^{[\boldsymbol{d},\boldsymbol{e}]}$ of \eqref{intro powe Ihd1d2} to be
\begin{align}\label{intro mult powe Ihde}
\begin{split}
J_{\boldsymbol{h}}^{[\boldsymbol{d},\boldsymbol{e}]}=\Bigg\{(s_{\boldsymbol{m}}^{[\boldsymbol{d},\boldsymbol{e}]})_{\boldsymbol{m}}\in \varprojlim_{\boldsymbol{m}\in\mathbb{Z}_{\geq 0}^{k}}&\left(\frac{\mathcal{O}_{\mathcal{K}}[[X_{1},\ldots, X_{k}]]}{(\Omega_{\boldsymbol{m}}^{[\boldsymbol{d},\boldsymbol{e}]}(X_{1},\ldots, X_{k}))}\otimes_{\mathcal{O}_{\K}}\K\right)\\
&\Bigg\vert (p^{\langle \boldsymbol{h},\boldsymbol{m}\rangle_{k}}s_{\boldsymbol{m}}^{[\boldsymbol{d},\boldsymbol{e}]})_{\boldsymbol{m}}\in \left(\prod_{\boldsymbol{m}\in \mathbb{Z}_{\geq 0}^{k}}\frac{\mathcal{O}_{\mathcal{K}}[[X_{1},\ldots, X_{k}]]}{(\Omega_{\boldsymbol{m}}^{[\boldsymbol{d},\boldsymbol{e}]}(X_{1},\ldots, X_{k}))}\right)\otimes_{\mathcal{O}_{\K}}\K\Bigg\}. 
\end{split}
\end{align}
Here, we remark that $\varprojlim_{\boldsymbol{m}\in\mathbb{Z}_{\geq 0}^{k}}\left(\frac{\mathcal{O}_{\mathcal{K}}[[X_{1},\ldots, X_{k}]]}{(\Omega_{\boldsymbol{m}}^{[\boldsymbol{d},\boldsymbol{e}]}(X_{1},\ldots, X_{k}))}\otimes_{\mathcal{O}_{\K}}\K\right)$ and $\left(\prod_{\boldsymbol{m}\in \mathbb{Z}_{\geq 0}^{k}}\frac{\mathcal{O}_{\mathcal{K}}[[X_{1},\ldots, X_{k}]]}{(\Omega_{\boldsymbol{m}}^{[\boldsymbol{d},\boldsymbol{e}]}(X_{1},\ldots, X_{k}))}\right)\linebreak\otimes_{\mathcal{O}_{\K}}\K$ are regarded as submodules of $\prod_{\boldsymbol{m}\in\mathbb{Z}_{\geq 0}^{k}}\left(\frac{\mathcal{O}_{\mathcal{K}}[[X_{1},\ldots, X_{k}]]}{(\Omega_{\boldsymbol{m}}^{[\boldsymbol{d},\boldsymbol{e}]}(X_{1},\ldots, X_{k}))}\otimes_{\mathcal{O}_{\K}}\K\right)$.
\\ \\ 
\ 
\par 
As multi-variable generalizations of Theorem \ref{one variable classical uniqueness} and Theorem \ref{onevariable projectresult}, 
we will state Theorem \ref{intro main theroem 1} and Theorem \ref{intro main theorem 2} below. 
Their proofs are given in \S \ref{sc:ordinary}, but these proofs are not obtained immediately simply replacing the field of coefficients $\K$ in 
Theorem \ref{one variable classical uniqueness} and Theorem \ref{onevariable projectresult} 
in the one-variable case by a general Banach space and by applying the same argument as the case with coefficients in $\K$. 
We really need a step-by-step induction argument with respect to the number of variables. 
In fact, to prove these results by induction argument,  
we will prove the following isomorphisms 
\begin{align*}
\HH_{\boldsymbol{h}\slash \K}&\simeq \HH_{h_{k}}(\HH_{\boldsymbol{h}^{\prime}\slash \K}),\\
J_{\boldsymbol{h}}^{[\boldsymbol{d},\boldsymbol{e}]}&\simeq J_{h_{k}}^{[d_{k},e_{k}]}(J_{\boldsymbol{h}^{\prime}}^{[\boldsymbol{d}^{\prime},\boldsymbol{e}^{\prime}]}),
\end{align*}
for every $\boldsymbol{h}\in \ord_{p}(\mathcal{O}_{\K}\backslash \{0\})^{k}$, $\boldsymbol{d},\boldsymbol{e}\in \mathbb{Z}^{k}$, $\boldsymbol{e}\geq \boldsymbol{d}$ 
in Proposition\ \ref{isometry ofHh for induction} and Proposition \ref{multi J induction pro}. 
Here we set $\boldsymbol{h}^{\prime}=(h_{1},\ldots, h_{k-1})$, $\boldsymbol{d}^{\prime}=(d_{1},\ldots, d_{k-1})$, $\boldsymbol{e}^{\prime}=(e_{1},\ldots, e_{k-1})$ 
and we denote by 
$\HH_{h_{k}}(\HH_{\boldsymbol{h}^{\prime}\slash \K})$ and $J_{h_{k}}^{[d_{k},e_{k}]}(J_{\boldsymbol{h}^{\prime}}^{[\boldsymbol{d}^{\prime},\boldsymbol{e}^{\prime}]})$ 
the spaces obtained by replacing the field of coefficients $\K$ of $\HH_{h_{k}\slash \K}$ and $J_{h_{k}}^{[d_{k},e_{k}]}$ by 
$\HH_{\boldsymbol{h}^{\prime}\slash \K}$ and $J_{\boldsymbol{h}^{\prime}}^{[\boldsymbol{d}^{\prime},\boldsymbol{e}^{\prime}]}$ respectively (see \eqref{themultivariable Banach version_of_H_h} 
and \eqref{generalization of the project lim for deformation ring Jboldsymbol} for the precise definitions of $\HH_{h_{k}}(\HH_{\boldsymbol{h}^{\prime}\slash \K})$ and $J_{h_{k}}^{[d_{k},e_{k}]}(J_{\boldsymbol{h}^{\prime}}^{[\boldsymbol{d}^{\prime},\boldsymbol{e}^{\prime}]})$). 
\par 
Put $\lfloor\boldsymbol{h}\rfloor=(\lfloor h_{1}\rfloor,\ldots,\lfloor h_{k}\rfloor)$. Here is one of our main results 
which is a multi-variable variant of Theorem \ref{one variable classical uniqueness}. 
\begin{thmA}[Theorem \ref{main theorem 1 and proof}]\label{intro main theroem 1}
 If $f\in\HH_{\boldsymbol{h}/\K}$ satisfies $f(u_{1}^{i_{1}}\epsilon_{1}-1,\ldots, u_{k}^{i_{k}}\epsilon_{k}-1)=0$ for each $k$-tuple $\boldsymbol{i}\in
 [\boldsymbol{d},\boldsymbol{d}+\lfloor\boldsymbol{h}\rfloor]$ and $(\epsilon_{1},\ldots, \epsilon_{k})\in \mu_{p^{\infty}}^{k}$, then $f$ is zero.
\end{thmA}
We can define a valuation $v_{\HH_{\boldsymbol{h}}}$ on $\HH_{\boldsymbol{h}\slash \K}$
by $v_{\HH_{\boldsymbol{h}}}(f)=\inf\{\ord_p (a_{\boldsymbol{n}} ) +\langle \boldsymbol{h},\ell(\boldsymbol{n})\rangle_{k}\}_{\boldsymbol{n}\in \mathbb{Z}_{\geq 0}^{k}}$ for each $f=\sum_{\boldsymbol{n}\in\mathbb{Z}_{\geq 0}^{k}}a_{\boldsymbol{n}}X^{\boldsymbol{n}}\in\HH_{\boldsymbol{h}\slash \K}$. We define an integral structure $\left(J_{\boldsymbol{h}}^{[\boldsymbol{d},\boldsymbol{e}]}\right)^{0}$ of $J_{\boldsymbol{h}}^{[\boldsymbol{d},\boldsymbol{e}]}$ to be
$$
\left(J_{\boldsymbol{h}}^{[\boldsymbol{d},\boldsymbol{e}]}\right)^{0}=\left\{(s_{\boldsymbol{m}}^{[\boldsymbol{d},\boldsymbol{e}]})_{\boldsymbol{m}}\in J_{\boldsymbol{h}}^{[\boldsymbol{d},\boldsymbol{e}]} \ 
\left\vert \  (p^{\langle \boldsymbol{h},\boldsymbol{m}\rangle_{k}}s_{\boldsymbol{m}}^{[\boldsymbol{d},\boldsymbol{e}]})_{\boldsymbol{m}}\in \prod_{\boldsymbol{m}\in \mathbb{Z}_{\geq 0}^{k}}\frac{\mathcal{O}_{\mathcal{K}}[[X_{1},\ldots, X_{k}]]}{(\Omega_{\boldsymbol{m}}^{[\boldsymbol{d},\boldsymbol{e}]}(X_{1},\ldots, X_{k}))}\right. \right\} .
$$
%
%
%
We also prove the following theorem which is a  multi-variable variant of Theorem \ref{onevariable projectresult}. 
\begin{thmA}[Theorem \ref{main thm 2 and proof}]\label{intro main theorem 2}
Assume that $\boldsymbol{e}-\boldsymbol{d}\geq \lfloor\boldsymbol{h}\rfloor$.
For $s^{[\boldsymbol{d},\boldsymbol{e}]}=(s_{\boldsymbol{m}}^{[\boldsymbol{d},\boldsymbol{e}]})_{\boldsymbol{m}\in \mathbb{Z}_{\geq 0}^{k}}\in J_{\boldsymbol{h}}^{[\boldsymbol{d},\boldsymbol{e}]}$, there exists a unique element $f_{s^{[\boldsymbol{d},\boldsymbol{e}]}}\in \HH_{\boldsymbol{h}\slash \K}$ such that 
$$f_{s^{[\boldsymbol{d},\boldsymbol{e}]}}-\tilde{s}_{\boldsymbol{m}}^{[\boldsymbol{d},\boldsymbol{e}]}\in (\Omega_{\boldsymbol{m}}^{[\boldsymbol{d},\boldsymbol{e}]})\HH_{\boldsymbol{h}\slash\K}$$
for each $\boldsymbol{m}\in \mathbb{Z}_{\geq 0}^{k}$, where $\tilde{s}_{\boldsymbol{m}}^{[\boldsymbol{d},\boldsymbol{e}]}\in \mathcal{O}_{\K}[[X_{1},\ldots, X_{k}]]\otimes_{\mathcal{O}_{\K}}\K$ is a lift of $s_{\boldsymbol{m}}^{[\boldsymbol{d},\boldsymbol{e}]}$. Further, the correspondence $s^{[\boldsymbol{d},\boldsymbol{e}]}\mapsto f_{s^{[\boldsymbol{d},\boldsymbol{e}]}}$ from $J_{\boldsymbol{h}}^{[\boldsymbol{d},\boldsymbol{e}]}$ to $\HH_{\boldsymbol{h}\slash \K}$ induces an $\mathcal{O}_\mathcal{K}[ [X_{1},\ldots, X_{k}] ]\otimes_{\mathcal{O}_{\K}}\K$-module isomorphism 
$
J_{\boldsymbol{h}}^{[\boldsymbol{d},\boldsymbol{e}]} \overset{\sim}{\longrightarrow} \HH_{\boldsymbol{h}/\K}.
$ 
Via the above isomorphism, we have
\begin{align}\label{error of Ih and integral Hh}
\{f\in \HH_{\boldsymbol{h}\slash\K}\vert v_{\HH_{\boldsymbol{h}}}(f)\geq \alpha_{\boldsymbol{h}}^{[\boldsymbol{d},\boldsymbol{e}]}\}\subset \left(J_{\boldsymbol{h}}^{[\boldsymbol{d},\boldsymbol{e}]}\right)^{0}\subset \{f\in \HH_{\boldsymbol{h}\slash\K}\vert v_{\HH_{\boldsymbol{h}}}(f)\geq \beta_{\boldsymbol{h}}\},\end{align}
where $\alpha_{\boldsymbol{h}}^{[\boldsymbol{d},\boldsymbol{e}]}=\sum_{i=1}^{k}\alpha_{h_{i}}^{[d_{i},e_{i}]}$ and $\beta_{\boldsymbol{h}}=\sum_{i=1}^{k}\beta_{h_{i}}$ with
\begin{align*}
\alpha_{h_{i}}^{[d_{i},e_{i}]}&=\begin{cases}\lfloor\frac{(e_{i}-d_{i}+1)}{p-1}+\max\{0, h_{i}-\frac{h_{i}}{\log p}(1+\log \frac{\log p}{(p-1)h_{i}})\}\rfloor+1\ &\mathrm{if}\ h_{i}>0,\\ 0\ &\mathrm{if}\ h_{i}=0,\end{cases}\\
\beta_{h_{i}}&=\begin{cases}-\lfloor\max\{h_{i},\frac{p}{p-1}\}\rfloor-1\ \ \ \  \ \ \ \ \ \ \ \ \ \ \ \ \ \ \ \ \ \ \ \ \ \ \ \ \ \ \ \ \ \ \ \ \ \ \ \, &\mathrm{if}\ h_{i}>0,\\ 0\ &\mathrm{if}\ h_{i}=0.\end{cases}
\end{align*}
\end{thmA}
Next, we will give the multi-variable generalizations of Proposition \ref{prop1} and Proposition \ref{prop2} (Proposition C and Theorem D respectively).

\par 
To state the results, we introduce some notation before we state these results. We put $(\Omega_{\boldsymbol{m}}^{[\boldsymbol{d},\boldsymbol{e}]}(\gamma_{1},\ldots, \gamma_{k}))=(\Omega_{m_{1}}^{[d_{1},e_{1}]}(\gamma_{1}),\ldots, \Omega_{m_{k}}^{[d_{k},e_{k}]}(\gamma_{k}))\subset \mathcal{O}_{\K}[[\Gamma]]$ for each $\boldsymbol{m}\in \mathbb{Z}_{\geq 0}^{k}$. We remark that the ideal $(\Omega_{\boldsymbol{m}}^{[\boldsymbol{d},\boldsymbol{e}]}(\gamma_{1},\ldots, \gamma_{k}))$ of $\mathcal{O}_{\K}[[\Gamma]]$ is independent of the choice of the topological generators $\gamma_{i}$ of $\Gamma_{i}$ for each $1\leq i\leq k$. If there is no risk of confusion, we write $(\Omega_{\boldsymbol{m}}^{[\boldsymbol{d},\boldsymbol{e}]})$ for $(\Omega_{\boldsymbol{m}}^{[\boldsymbol{d},\boldsymbol{e}]}(\gamma_{1},\ldots, \gamma_{k}))$. In a similar way to $J_{\boldsymbol{h}}^{[\boldsymbol{d},\boldsymbol{e}]}$, we define an $\mathcal{O}_{\K}[[\Gamma ]]\otimes_{\mathcal{O}_{\K}}\K$-module $I_{\boldsymbol{h}}^{[\boldsymbol{d},\boldsymbol{e}]}$ to be
\begin{align}
\begin{split}
I_{\boldsymbol{h}}^{[\boldsymbol{d},\boldsymbol{e}]}=\Bigg\{(s_{\boldsymbol{m}}^{[\boldsymbol{d},\boldsymbol{e}]})_{\boldsymbol{m}}\in \varprojlim_{\boldsymbol{m}\in\mathbb{Z}_{\geq 0}^{k}}&\left(\frac{\mathcal{O}_{\K}[[\Gamma]]}{(\Omega_{\boldsymbol{m}}^{[\boldsymbol{d},\boldsymbol{e}]}(\gamma_{1},\ldots, \gamma_{k}))}\otimes_{\mathcal{O}_{\K}}\K\right)\\
&\Bigg\vert (p^{\langle \boldsymbol{h},\boldsymbol{m}\rangle_{k}}s_{\boldsymbol{m}}^{[\boldsymbol{d},\boldsymbol{e}]})_{\boldsymbol{m}}\in \left(\prod_{\boldsymbol{m}\in \mathbb{Z}_{\geq 0}^{k}}\frac{\mathcal{O}_{\K}[[\Gamma]]}{(\Omega_{\boldsymbol{m}}^{[\boldsymbol{d},\boldsymbol{e}]}(\gamma_{1},\ldots, \gamma_{k}))}\right)\otimes_{\mathcal{O}_{\K}}\K\Bigg\}.
\end{split}
\end{align}
Further, we put $\left(I_{\boldsymbol{h}}^{[\boldsymbol{d},\boldsymbol{e}]}\right)^{0}=\left\{(s_{\boldsymbol{m}}^{[\boldsymbol{d},\boldsymbol{e}]})_{\boldsymbol{m}}\in I_{\boldsymbol{h}}^{[\boldsymbol{d},\boldsymbol{e}]}\Big\vert (p^{\langle \boldsymbol{h},\boldsymbol{m}\rangle_{k}}s_{\boldsymbol{m}}^{[\boldsymbol{d},\boldsymbol{e}]})_{\boldsymbol{m}}\in \prod_{\boldsymbol{m}\in \mathbb{Z}_{\geq 0}^{k}}\frac{\mathcal{O}_{\mathcal{K}}[[\Gamma]]}{(\Omega_{\boldsymbol{m}}^{[\boldsymbol{d},\boldsymbol{e}]}(\gamma_{1},\ldots, \gamma_{k}))}\right\}$. 
\par 
{Let us consider the non-canonical continuous $\mathcal{O}_{\K}$-algebra isomorphism $\alpha^{(k)}:\mathcal{O}_{\K}[[\Gamma]]\stackrel{\sim}{\rightarrow}\mathcal{O}_{\K}[[X_{1},\ldots, X_{k}]]$ characterized by $\alpha^{(k)} ([(\gamma_{1}^{n_{1}},\ldots, \gamma_{k}^{n_{k}})]) = \prod_{i=1}^{k}(1+X_{i})^{n_{i}}$ for each $\boldsymbol{n}\in \mathbb{Z}^{k}$. Then, we have a non-canonical $\K$-linear isomorphism 
\begin{equation}\label{intrononcanonical between Ioldysmbolhde and J}
I_{\boldsymbol{h}}^{[\boldsymbol{d},\boldsymbol{e}]}\simeq J_{\boldsymbol{h}}^{[\boldsymbol{d},\boldsymbol{e}]}
\end{equation}
which extends the isomorphism $\alpha^{(k)}$. 
\par 
We denote by $C^{[\boldsymbol{d},\boldsymbol{e}]} (\Gamma , \mathcal{O}_\mathcal{K})$ the $\mathcal{O}_\mathcal{K}$-module of $k$-variable functions $f:\Gamma\rightarrow \mathcal{O}_{\K}$ such that $\left(\prod_{i=1}^{k}\chi_{i}(x_{i})^{-d_{i}}\right)f(x_{1},\ldots, x_{k})$ is a locally polynomial function of degree at most $\boldsymbol{e}-\boldsymbol{d}$ (see \S \ref{preparation} for the precise definition of 
locally polynomial functions). Let $\mathcal{D}^{[\boldsymbol{d},\boldsymbol{e}]}_{\boldsymbol{h} }(\Gamma , \mathcal{K})$ be the $\K$-vector space 
of elements of $\mathrm{Hom}_{\mathcal{O}_{\mathcal{K}}}( C^{[\boldsymbol{d},\boldsymbol{e}]} (\Gamma , \mathcal{O}_\mathcal{K}) , \mathcal{K})$
which are $[\boldsymbol{d},\boldsymbol{e}]$-admissible distributions of growth $\boldsymbol{h}$ (see \eqref{defof admissible distri space} of this paper for the precise definition of $[\boldsymbol{d},\boldsymbol{e}]$-admissible distributions of growth $\boldsymbol{h}$). In the same way as the space of one-variable admissible distributions, we can regard $\mathcal{D}_{\boldsymbol{h}}^{[\boldsymbol{d},\boldsymbol{e}]}(\Gamma,\K)$ as an $\mathcal{O}_{\K}[[\Gamma]]\otimes_{\mathcal{O}_{\K}}\K$-module naturally. 
\par 
{
Similarly as what we remarked earlier before Theorem \ref{intro main theroem 1} and Theorem \ref{intro main theorem 2},  
Proposition \ref{multivariable Iii necesarry sufficient condition} and Theorem \ref{multi variabl admissible intro} 
are not obtained immediately simply replacing the field of coefficients $\K$ in Proposition \ref{prop1} and Proposition \ref{prop2} 
in the one-variable case by a general Banach space and by applying the same argument as the case with coefficients in $\K$. 
We really need a step-by-step induction argument with respect to the number of variables. 
\par 
In fact, to prove Theorem \ref{multi variabl admissible intro} by induction argument, we will prove the following isomorphisms 
\begin{align*}
I_{\boldsymbol{h}}^{[\boldsymbol{d},\boldsymbol{e}]}&\simeq I_{h_{k}}^{[d_{k},e_{k}]}(I_{\boldsymbol{h}^{\prime}}^{[\boldsymbol{d}^{\prime},\boldsymbol{e}^{\prime}]}),\\
\mathcal{D}_{\boldsymbol{h}}^{[\boldsymbol{d},\boldsymbol{e}]}(\Gamma,\K)&\simeq\mathcal{D}_{h_{k}}^{[d_{k},e_{k}]}(\Gamma_{k}, \mathcal{D}_{\boldsymbol{h}^{\prime}}^{[\boldsymbol{d}^{\prime},\boldsymbol{e}^{\prime}]}(\Gamma^{\prime},\K)), 
\end{align*}
for every $\boldsymbol{h}\in \ord_{p}(\mathcal{O}_{\K}\backslash \{0\})^{k}$, $\boldsymbol{d},\boldsymbol{e}\in \mathbb{Z}^{k}$, $\boldsymbol{e}\geq \boldsymbol{d}$ 
in Proposition \ref{for induction admisible admissible} and Proposition \ref{multi J induction pro}. 
Here we set $\boldsymbol{h}^{\prime}=(h_{1},\ldots, h_{k-1})$, $\boldsymbol{d}^{\prime}=(d_{1},\ldots, d_{k-1})$, $\boldsymbol{e}^{\prime}=(e_{1},\ldots, e_{k-1})$, $\Gamma^{\prime}=\Gamma_{1}\times\cdots\times \Gamma_{k-1}$ we denote by $I_{h_{k}}^{[d_{k},e_{k}]}(I_{\boldsymbol{h}^{\prime}}^{[\boldsymbol{d}^{\prime},\boldsymbol{e}^{\prime}]})$ and 
$\mathcal{D}_{h_{k}}^{[d_{k},e_{k}]}(\Gamma_{k}, \mathcal{D}_{\boldsymbol{h}^{\prime}}^{[\boldsymbol{d}^{\prime},\boldsymbol{e}^{\prime}]}(\Gamma^{\prime},\K))$ 
the spaces obtained by replacing the field of coefficients $\K$ of $I_{h_{k}}^{[d_{k},e_{k}]}$ and $\mathcal{D}_{h_{k}}^{[d_{k},e_{k}]}(\Gamma_{k},\K)$ by 
$I_{\boldsymbol{h}^{\prime}}^{[\boldsymbol{d}^{\prime},\boldsymbol{e}^{\prime}]}$ and $\mathcal{D}_{\boldsymbol{h}^{\prime}}^{[\boldsymbol{d}^{\prime},\boldsymbol{e}^{\prime}]}(\Gamma^{\prime},\K)$ 
respectively (see \eqref{generalization of the project lim for deformation ring} and \eqref{defof admissible distri space} 
for the precise definitions of $I_{h_{k}}^{[d_{k},e_{k}]}(I_{\boldsymbol{h}^{\prime}}^{[\boldsymbol{d}^{\prime},\boldsymbol{e}^{\prime}]})$ and 
$\mathcal{D}_{h_{k}}^{[d_{k},e_{k}]}(\Gamma_{k}, \mathcal{D}_{\boldsymbol{h}^{\prime}}^{[\boldsymbol{d}^{\prime},\boldsymbol{e}^{\prime}]}(\Gamma^{\prime},\K))$).}  
\par 
{The induction argument in the proof of Proposition \ref{multivariable Iii necesarry sufficient condition} requires further more complicated process than that of 
Theorem \ref{multi variabl admissible intro}. Hence we stress again that the multi-variable variant of the one-variable theory of Amice--V\'{e}lu and Vishik obtained in this article 
is not done simply replacing the field of coefficients in the one-variable case by a general Banach space.}
\par 
{Let us denote $[\boldsymbol{i},\boldsymbol{i}]$ by $[\boldsymbol{i}]$ for each $\boldsymbol{i}\in \mathbb{Z}^{k}$}. The following proposition is a multi-variable variant of Proposition \ref{prop1}. 
\begin{proA}[Proposition \ref{multivariable iwasawa I(ii) sufficient}]\label{multivariable Iii necesarry sufficient condition}
Let $s^{[\boldsymbol{i}]}=(s^{[\boldsymbol{i}]}_{\boldsymbol{m}})_{\boldsymbol{m}\in \mathbb{Z}_{\geq 0}^{k}}\in I_{\boldsymbol{h}}^{[\boldsymbol{i}]}$ and $\tilde{s}_{\boldsymbol{m}}^{[\boldsymbol{i}]}$ a lift of $s_{\boldsymbol{m}}^{[\boldsymbol{i}]}$ for each $\boldsymbol{m}\in \mathbb{Z}_{\geq 0}^{k}$ and $\boldsymbol{i}\in [\boldsymbol{d},\boldsymbol{e}]$. If there exists a non-negative integer $n$ which satisfies
$$p^{\langle \boldsymbol{m},\boldsymbol{h}-(\boldsymbol{j}-\boldsymbol{d})\rangle_{k}}\displaystyle{\sum_{\boldsymbol{i}\in [\boldsymbol{d},\boldsymbol{j}]}}\left(\prod_{t=1}^{k}\begin{pmatrix}j_{t}-d_{t}\\i_{t}-d_{t}\end{pmatrix}\right)(-1)^{\sum_{t=1}^{k}(j_{t}-i_{t})}\tilde{s}_{\boldsymbol{m}}^{[\boldsymbol{i}]}\in \mathcal{O}_{\K}[[\Gamma]]\otimes_{\mathcal{O}_{\K}}p^{-n}\mathcal{O}_{\K}$$
for each $\boldsymbol{m}\in \mathbb{Z}_{\geq 0}^{k}$ and $\boldsymbol{j}\in [\boldsymbol{d},\boldsymbol{e}]$, we have a unique element $s^{[\boldsymbol{d},\boldsymbol{e}]}\in \left(I_{\boldsymbol{h}}^{[\boldsymbol{d},\boldsymbol{e}]}\right)^{0}\otimes_{\mathcal{O}_{\K}}p^{-c^{[\boldsymbol{d},\boldsymbol{e}]}-n}\mathcal{O}_{\K}$ such that the image of $s^{[\boldsymbol{d},\boldsymbol{e}]}$ by the natural projection $I_{\boldsymbol{h}}^{[\boldsymbol{d},\boldsymbol{e}]}\rightarrow I_{\boldsymbol{h}}^{[\boldsymbol{i}]}$ is $s^{[\boldsymbol{i}]}$ for each $\boldsymbol{i}\in [\boldsymbol{d},\boldsymbol{e}]$, where $c^{[\boldsymbol{d},\boldsymbol{e}]}=\sum_{i=1}^{k}c^{[d_{i},e_{i}]}$ is the constant defined by
\begin{equation}\label{intro multivariable cde constant}
c^{[d_{i},e_{i}]}=\begin{cases}\ord_{p}((e_{i}-d_{i})!)+2(e_{i}-d_{i})+\lfloor\frac{e_{i}-d_{i}+1}{p-1}\rfloor+1\ &\mathrm{if}\ d_{i}<e_{i},\\
0\ &\mathrm{if}\ d_{i}=e_{i}.\end{cases}
\end{equation}
\end{proA}
Let $\mu\in \mathcal{D}_{\boldsymbol{h}}^{[\boldsymbol{d},\boldsymbol{e}]}(\Gamma,\K)$. We define
\begin{multline*}
\begin{split}
v_{\boldsymbol{h}}^{[\boldsymbol{d},\boldsymbol{e}]}(\mu)=\inf_{\substack{\boldsymbol{a}\in \Gamma,\boldsymbol{m}\in \mathbb{Z}_{\geq 0}^{k}\\ \boldsymbol{i}\in [\boldsymbol{d},\boldsymbol{e}]}}\Biggl\{\ord_{p}\left(\displaystyle{\int_{\boldsymbol{a}\Gamma^{p^{\boldsymbol{m}}}}}\prod_{j=1}^{k}\left((\chi_{j}(x_{j})-\chi_{j}(a_{j}))^{i_{j}-d_{j}}\chi_{j}(x_{j})^{d_{j}}\right)d\mu\right)\\
+\langle \boldsymbol{h}-(\boldsymbol{i}-\boldsymbol{d}),\boldsymbol{m}\rangle_{k}\Bigg\}>-\infty,
\end{split}
\end{multline*}
where $\boldsymbol{a}\Gamma^{p^{\boldsymbol{m}}}=\prod_{j=1}^{k}a_{j}\Gamma_{j}^{p^{m_{j}}}$. Let $\mathfrak{X}_{\mathcal{O}_{\K}[[\Gamma]]}^{[\boldsymbol{d},\boldsymbol{e}]}$ be the set of $k$-variable arithmetic specializations of weight $\boldsymbol{w}_{\kappa}\in [\boldsymbol{d},\boldsymbol{e}]$ over $\mathcal{O}_{\K}[[\Gamma]]$. For each $\kappa\in \mathfrak{X}_{\mathcal{O}_{\K}[[\Gamma]]}^{[\boldsymbol{d},\boldsymbol{e}]}$, we denote by $\boldsymbol{\phi}_{\kappa}=(\phi_{\kappa,1},\ldots, \phi_{\kappa,k})$ the finite character of $\kappa$ and put $\boldsymbol{m}_{\kappa}=(m_{\kappa,1},\ldots, m_{\kappa,k})$, where $m_{\kappa,i}$ is the smallest integer $m$ such that $\phi_{\kappa,i}$ factors through $\Gamma_{i}/(\Gamma_{i})^{p^m}$ with $1\leq i\leq k$.
The following theorem is a multi-variable variant of Proposition \ref{prop2}.
\begin{thmA}[Theorem \ref{multi-variable results on admissible distributions}]\label{multi variabl admissible intro}
We have a unique $\mathcal{O}_{\K}[[\Gamma]]\otimes_{\mathcal{O}_{\K}}\K$-module isomorphism 
\begin{equation}\label{equation:multi-variable results on admissible distributions intro}
I_{\boldsymbol{h}}^{[\boldsymbol{d},\boldsymbol{e}]} \stackrel{\sim}{\rightarrow} \mathcal{D}^{[\boldsymbol{d},\boldsymbol{e}]}_{\boldsymbol{h}} (\Gamma, \K)
\end{equation}
such that the image $\mu_{s^{[\boldsymbol{d},\boldsymbol{e}]}}\in \mathcal{D}^{[\boldsymbol{d},\boldsymbol{e}]}_{\boldsymbol{h}} (\Gamma,\K)$ of each element $s^{[\boldsymbol{d},\boldsymbol{e}]}=(s_{\boldsymbol{m}}^{[\boldsymbol{d},\boldsymbol{e}]})_{\boldsymbol{m}\in \mathbb{Z}_{\geq 0}^{k}} \in I_{\boldsymbol{h}}^{[\boldsymbol{d},\boldsymbol{e}]}$ is characterized by the interpolation property  
\begin{align}
\kappa(\tilde{s}_{\boldsymbol{m}_{\kappa}}^{[\boldsymbol{d},\boldsymbol{e}]})= \int_{\Gamma} \prod_{j=1}^{k}(\chi_{j}^{w_{\kappa,j}}\phi_{\kappa,j})(x_{j})d\mu_{s^{[\boldsymbol{d},\boldsymbol{e}]}} \end{align}
 for every $\kappa\in \mathfrak{X}_{\mathcal{O}_{\K}[[\Gamma]]}^{[\boldsymbol{d},\boldsymbol{e}]}$, where $\tilde{s}_{\boldsymbol{m}_{\kappa}}^{[\boldsymbol{d},\boldsymbol{e}]}$ is a lift of $s_{\boldsymbol{m}_{\kappa}}^{[\boldsymbol{d},\boldsymbol{e}]}$. In addition, via the above isomorphism, we have 
\begin{equation}\label{multi variabl admissible intro constant c}
\{\mu\in \mathcal{D}^{[\boldsymbol{d},\boldsymbol{e}]}_{\boldsymbol{h}} (\Gamma ,\K)\vert v_{\boldsymbol{h}}^{[\boldsymbol{d},\boldsymbol{e}]}(\mu)\geq c^{[\boldsymbol{d},\boldsymbol{e}]}\}\subset \left(I_{\boldsymbol{h}}^{[\boldsymbol{d},\boldsymbol{e}]}\right)^{0}
\subset \{\mu\in \mathcal{D}^{[\boldsymbol{d},\boldsymbol{e}]}_{\boldsymbol{h}} (\Gamma ,\K)\vert v_{\boldsymbol{h}}^{[\boldsymbol{d},\boldsymbol{e}]}(\mu)\geq 0\},
\end{equation}
where $c^{[\boldsymbol{d},\boldsymbol{e}]}=\sum_{i=1}^{k}c^{[d_{i},e_{i}]}$ is the constant defined in \eqref{intro multivariable cde constant}.
\end{thmA}
{In \S\ref{sc:ordinary deformation}, we generalize the results of Theorem \ref{intro main theroem 1}, Theorem \ref{intro main theorem 2}, Proposition \ref{multivariable Iii necesarry sufficient condition} and Theorem \ref{multi variabl admissible intro} to results on deformation spaces. We prove the generalizations of Theorem \ref{intro main theroem 1}, Theorem \ref{intro main theorem 2}, Proposition \ref{multivariable Iii necesarry sufficient condition} and Theorem \ref{multi variabl admissible intro} on deformation spaces in Theorem \ref{main theorem 1 for deformation space}, Theorem \ref{deformtion Jhtheorem multivariable}, Proposition \ref{deformation Ih[d,e] lift from Ihi,i} and Theorem \ref{multi-variable results on admissible distributions deformation ver} respectively.}

As mentioned above, our results are multi-variable generalizations of the results of Amice--V\'{e}lu \cite{amicevelu} and Vishik \cite{vishik1976}. However, even if we restrict our results to the one-variable case, our results still have several advantages compared to the 
results obtained in \cite{amicevelu} and \cite{vishik1976}. 
In addition to Remark \ref{remark:H_h}, we explain below a few more advantages of our results 
which are not proved in the classical results obtained in \cite{amicevelu} and \cite{vishik1976}. 
\begin{rem}\label{remark:H_h2}
\begin{enumerate}
\item  
From the Iwasawa theoretical viewpoint, it is important to study the integral structures of given modules. 
Let $\HH_{\boldsymbol{h}\slash \K}^{0}=\{f\in \HH_{\boldsymbol{h}\slash \K}\vert v_{\HH_{\boldsymbol{h}}}(f)\geq 0\}$. 
We estimated the difference of the integral lattice $\left(J_{\boldsymbol{h}}^{[\boldsymbol{d},\boldsymbol{e}]}\right)^{0}$ of 
$J_{\boldsymbol{h}}^{[\boldsymbol{d},\boldsymbol{e}]}$ and the integral lattice $\HH_{\boldsymbol{h}\slash \K}^{0}$ of $\HH_{\boldsymbol{h}\slash \K}$ 
in the isomorphism $J_{\boldsymbol{h}}^{[\boldsymbol{d},\boldsymbol{e}]} \overset{\sim}{\longrightarrow} \HH_{\boldsymbol{h}/\K}$ of Theorem \ref{intro main theorem 2}. 
In the classical one-variable setting, Amice--V\'{e}lu \cite[Proposition I\hspace{-1.2pt}V. 1]{amicevelu} did not really study 
such an error between the integral structures of the both sides of the isomorphism. 
Hence our estimate \eqref{error of Ih and integral Hh} on the difference of the integral structures in the isomorphism $J_{\boldsymbol{h}}^{[\boldsymbol{d},\boldsymbol{e}]} \overset{\sim}{\longrightarrow} \HH_{\boldsymbol{h}/\K}$
 gives a new and finer result even if we restrict ourselves to the one-variable situation.  
%
\item 
Let $s^{[i]}=(s^{[i]}_{m})_{m\in \mathbb{Z}_{\geq 0}}\in I_{h}^{[i]}$ and let $\tilde{s}^{[i]}_{m}\in \mathcal{O}_{\K}[[\Gamma]]\otimes_{\mathcal{O}_{\K}}\K$ be a lift of $s^{[i]}_{m}$ for each $m\in \mathbb{Z}_{\geq 0}$ and $i\in [d,e]$, where $I_{h}^{[i]}$ is the 
module defined in \eqref{intro powe Ihd1d2+}. We assume that there exists a non-negative integer $n$ which satisfies \eqref{intro onevariabble admissible I} in Proposition \ref{prop1}. Then, by the classical result of Proposition \ref{prop1}, we see that there exists a unique element $s^{[d,e]}\in I_{h}^{[d,e]}$ such that the image of $s^{[d,e]}$ by the natural projection $I_{h}^{[d,e]}\rightarrow I_{h}^{[i]}$ is $s^{[i]}$ for each $d\leq i\leq e$. In this case also, our result gives an 
integral refinement of this classical result. 
In fact, when we restrict our result of Proposition \ref{multivariable Iii necesarry sufficient condition} to the classical 
one-vaiable setting, we can prove that $s^{[d,e]}$ is in $(I_{h}^{[d,e]})^{0}\otimes_{\mathcal{O}_{\K}}p^{-c^{[d,e]}-n}\mathcal{O}_{\K}$ 
provided that  $s^{[i]}=(s^{[i]}_{m})_{m\in \mathbb{Z}_{\geq 0}}$ is contained in the integral part $(I_{h}^{[i]})^0$ for every $i\in [d,e]$, 
where $c^{[d,e]}$ is the constant in \eqref{intro multivariable cde constant}. 

\item 
 We also estimate the error between the integral structure $\left(I_{\boldsymbol{h}}^{[\boldsymbol{d},\boldsymbol{e}]}\right)^{0}$ and 
the integral structure $\mathcal{D}_{\boldsymbol{h}\slash \K}^{[\boldsymbol{d},\boldsymbol{e}]}(\Gamma,\K)^{0}$ in the isomorphism $I_{\boldsymbol{h}}^{[\boldsymbol{d},\boldsymbol{e}]}\simeq \mathcal{D}_{\boldsymbol{h}
\slash \K}^{[\boldsymbol{d},\boldsymbol{e}]}(\Gamma,\K)$ in Theorem \ref{multi variabl admissible intro}, where $\mathcal{D}_{\boldsymbol{h}\slash \K}^{[\boldsymbol{d},\boldsymbol{e}]}(\Gamma,\K)^{0}=\{\mu\in \mathcal{D}_{\boldsymbol{h}\slash \K}^{[\boldsymbol{d},\boldsymbol{e}]}(\Gamma,\K)\vert v_{\boldsymbol{h}}^{[\boldsymbol{d},\boldsymbol{e}]}\geq 0\}$. 
In this case also, our result restricted to the the classical one-variable setting gives 
a new and finer result compared to the classical result of Vishik \cite[2.3. Theorem]{vishik1976}. 
%
\end{enumerate}
\end{rem}

As an application of our theory developed in this paper, 
we construct a two-variable $p$-adic Rankin Selberg $L$-series in \S\ref{sc:application}.
To state the application, we recall some notation of Rankin Selberg $L$-series and Hida families. We denote by $S_{l}(N,\psi)$ the space of cusp forms of weight $l\in \mathbb{Z}_{\geq 1}$, level $N$ and character $\psi$, where $N\in \mathbb{Z}_{\geq 1}$ and $\psi$ is a Dirichlet character modulo $N$. For each $f\in S_{l_{1}}(N,\psi)$ and $g\in S_{l_{2}}(N,\xi)$, we define the Rankin Selberg $L$-sereis $\mathscr{D}_{N}(s,f,g)$ to be
$$\mathscr{D}_{N}(s,f,g)=L_{N}(2s+2-l_{1}-l_{2},\psi\xi)\sum_{n=1}^{+\infty}a_{n}(f)a_{n}(g)n^{-s},\ \mathrm{Re}(s)>\frac{l_{1}+l_{2}}{2},$$
where $a_{n}(f)$ and $a_{n}(g)$ are the $n$-th Fourier coefficients of $f$ of $g$ respectively and $L_{N}(s,\psi\xi)=\sum_{n=1}^{+\infty}\psi\xi(n)n^{-s}$. Assume that $l_{1}>l_{2}$. It is known that $\mathscr{D}_{N}(s,f,g)$ has a holomorphic continuation to the whole complex plane. Further, when $f$ is a primitive form whose conductor divides $N$ and the Fourier coefficients of $g$ are algebraic, Shimura \cite{Shimura1976} and \cite{Shimura1977} proved that $\frac{\mathscr{D}_{N}(m,f,g)}{\pi^{2m-l_{2}+1}\langle f,f\rangle_{l_{1},N}}$ is\ algebraic
for each integer $m$ satisfying $l_{2}\leq m<l_{1}$. Here $\langle f,f\rangle_{l_{1},N}$ is defined by 
$$\langle f,f\rangle_{l_{1},N}=\int_{\Gamma_{0}(N)\backslash \mathfrak{H}}\vert f(z)\vert^{2}y^{l_{1}}\frac{dxdy}{y^{2}},$$
where $\mathfrak{H}$ is the upper half plane and
 $\Gamma_{0}(N)=\left\{ \left. \begin{pmatrix}a&b\\ c&d\end{pmatrix}\in SL_{2}(\mathbb{Z}) \  \right\vert \ c\equiv 0\ \mathrm{mod}\ N\right  \}$. The values $\mathscr{D}_{N}(m,f,g)$ with $l_{2}\leq m<l_{1}$ are called the critical values of $\mathscr{D}_{N}(s,f,g)$. For each normalized Hecke eigenforms $f\in S_{l_{1}}(N,\psi)$ and $g\in S_{l_{2}}(N,\xi)$, we put
$$\Lambda(s,f,g)=\Gamma_{\mathbb{C}}\left(s-l_{2}+1\right)\Gamma_{\mathbb{C}}\left(s\right)\mathscr{D}_{M}(s,f^{0},g^{0})$$
where $f^{0}$ and $g^{0}$ are primtive forms attached to $f$ and $g$ respectively, $M$ is the least common multiple of the conductor of $f$ and the conductor of $g$ and $\Gamma_{\mathbb{C}}(s)=2(2\pi)^{-s}\Gamma(s)$.
 
We assume that $p\geq 5$. Let $\K$ be a finite extension of $\mathbb{Q}_{p}$ and $\omega$ the Teichm\"{u}ller character modulo $p$. Let $N$ be a positive integer which is prime to $p$ and $\xi$ a Dirichlet character modulo $Np$. We say that a power series $G=\sum_{n=1}^{+\infty}a_{n}(G)q^{n}\in \mathcal{O}_{\K}[[\Gamma_{2}]][[q]]$ is an $\mathcal{O}_{\K}[[\Gamma_{2}]]$-adic Hida family of tame level $N$ and character $\xi$ if the specialization $\kappa(G)=\sum_{n=1}^{+\infty}\kappa(a_{n}(G))q^{n}$ is a $q$-expansion of a normalized cuspidal Hecke eigenform of weight $w_{\kappa}$, level $Np^{m_{\kappa}+1}$ and character $\xi\phi_{\kappa}\omega^{-w_{\kappa}}$ which is ordinary at $p$ for each $\kappa\in \mathfrak{X}_{\mathcal{O}_{\K}[[\Gamma_{2}]]}$ such that $w_{\kappa}\geq 2$. 

As an application of our theorems, we have the following two-variable $p$-adic Rankin Selberg $L$-series. 
\begin{thmA}[Theorem \ref{two variable rankin selberg l series of hida family}]\label{intro two variable p-aic}
Let $f\in S_{k}(p^{m(f)},\psi;\K)$ with $k,m(f)\in \mathbb{Z}_{\geq 1}$ be a normalized Hecke eigenform, and let $G$ be an $\mathcal{O}_{\K}[[\Gamma_{2}]]$-adic Hida family of tame level $1$ and character $\xi$. Here, $\psi$ and $\xi$ are Dirichlet characters modulo $p^{m(f)}$ and $p$ respectively. Put $\boldsymbol{h}=(2\alpha,\alpha)$ with $\alpha=\ord_{p}(a_{p}(f))$, $\boldsymbol{d}=(0,2)$ and $\boldsymbol{e}=(k-3,k-1)$. We assume the following conditions:
\begin{enumerate}
\item The root number of $f^{0}$ and Fourier coefficients of $f$ and $f^{0}$ are contained in $\K$, where $f^{0}$ is the primitive form associated with $f$.
\item We have $k>\lfloor 2\alpha\rfloor+\lfloor \alpha\rfloor +2$.
\end{enumerate} 
Then, there exists a unique element $\mu_{(f,G)}\in \mathcal{D}_{\boldsymbol{h}}^{[\boldsymbol{d},\boldsymbol{e}]}(\Gamma_{1}\times \Gamma_{2},\K)$ which satisfies the 
following interpolation: 
\small 
\begin{multline*}
\int_{\Gamma_{1}\times \Gamma_{2}}\kappa\vert_{\Gamma_{1}\times \Gamma_{2}}\mu_{(f,G)}= \sqrt{-1}^{2w_{\kappa,1}+w_{\kappa,2} }
G(\phi_{\kappa,1})G(\omega^{-w_{\kappa,2}}\xi\phi_{\kappa,1}\phi_{\kappa,2})\\
\times E_{p,\phi_{\kappa,1}}(w_{\kappa,1}+w_{\kappa,2},f,\kappa\vert_{\mathcal{O}_{\K}[[\Gamma_{2}]]}(G))\frac{\Lambda\left(w_{\kappa,1}+w_{\kappa,2},f,
 \left(\kappa\vert_{\mathcal{O}_{\K}[[\Gamma_{2}]]}(G)\otimes\phi_{\kappa,1}\right)^{\rho}\right)}{\langle f^{0},f^{0}\rangle_{k,c_{f}}}
\end{multline*}
for every $\kappa\in \mathfrak{X}_{\mathcal{O}_{\K}[[\Gamma]]}^{[\boldsymbol{d},\boldsymbol{e}]}$ such that $w_{\kappa,1}+w_{\kappa,2}<k$ where $G(\phi_{\kappa,1})$ and $G(\omega^{-w_{\kappa,2}}\xi\phi_{\kappa,1}\phi_{\kappa,2})$ are Gauss sums of $\phi_{\kappa,1}$ and $\omega^{-w_{\kappa,2}}\xi\phi_{\kappa,1}\phi_{\kappa,2}$ respectively, $c_{f}$ is the conductor of $f$, $\rho$ is the complex conjugate and $\left(\kappa\vert_{\mathcal{O}_{\K}[[\Gamma_{2}]]}(G)\otimes\phi_{\kappa,1}\right)^{\rho}=\sum_{n=1}^{+\infty}\rho\left(\kappa\vert_{\mathcal{O}_{\K}[[\Gamma_{2}]]}(G)\phi_{\kappa,1}(n)\right)q^{n}$ and $E_{p,\phi_{\kappa,1}}(s,f,\kappa\vert_{\mathcal{O}_{\K}[[\Gamma_{2}]]}(G))$ is the $p$-th Euler factor which will be defined in \eqref{for comaptibilityof PC Euler factor}.
\end{thmA}
Theorem \ref{intro two variable p-aic} is a special case of Threorem \ref{two variable rankin selberg l series of hida family}. 
\par 
\begin{rem}
In Theorem \ref{intro two variable p-aic}, we constructed a two-variable $p$-adic $L$-function which is 
a $\K$-valued admissible distribution of growth $(2\alpha,\alpha)$ with $\alpha\in \mathbb{Q}_{>0}$ as an application of 
our Proposition \ref{multivariable Iii necesarry sufficient condition} and Theorem \ref{multi variabl admissible intro}. 
The use of these main results is quite essential in this result and Theorem \ref{intro two variable p-aic} is essentially the first example 
of a multi-variable non-ordinary $p$-adic $L$-function. 
\par 
It is true that there exists a few examples of non-ordinary multi-variable $p$-adic $L$-functions. For example, a two-variable $p$-adic $L$-function associated 
to the cyclotomic deformation of a Coleman family of slope $\alpha >0$ was constructed by 
Panchishkin \cite{Panchishkin03} without using our main results. 
In fact, the $p$-adic $L$-function of  \cite{Panchishkin03} is constructed as an element of the space 
$\mathcal{D}_{\alpha}(\Gamma_{1},\mathcal{O}_\mathcal{K} \langle X \rangle\otimes_{\mathcal{O}_{\K}}\K)$ of admissible distributions of growth $\alpha$ 
over a Tate algebra $\mathcal{O}_\mathcal{K} \langle X \rangle\otimes_{\mathcal{O}_{\K}}\K$.   
The Coleman family is often defined over a Tate algebra $\mathcal{O}_\mathcal{K} \langle X \rangle\otimes_{\mathcal{O}_{\K}}\K$ according to the original construction of Coleman \cite{coleman 1997}, 
But it can be defined as a family over an Iwasawa algebra $\mathcal{O}_{\K}[[X]] \otimes_{\mathcal{O}_{\K}}\K$ $($see \cite{NOR2023}$)$. 
Thus the above-mentioned work of Panchishkin give an example of non-ordinary multi-variable $p$-adic $L$-function in 
$\mathcal{D}_{\alpha}(\Gamma_{1},\mathcal{O}_{\K}[[X]] \otimes_{\mathcal{O}_{\K}}\K) $. Note that we have a non-canonical isomorphism 
$$
{\mathcal{D}_{\alpha}(\Gamma_{1},\mathcal{O}_\mathcal{K} [[X]] \otimes_{\mathcal{O}_{\K}}\K)\simeq 
\mathcal{D}_{\alpha}(\Gamma_{1},\mathcal{O}_\mathcal{K} [[\Gamma_{1}]] \otimes_{\mathcal{O}_{\K}}\K)
\simeq 
\mathcal{D}_{(\alpha,0)}(\Gamma_{1}\times\Gamma_{2},\K)}
$$
with certain $p$-adic group $\Gamma_{2}$ isomorphic to $1+p\mathbb{Z}_p$. {In addition to Panchishkin \cite{Panchishkin03}, Bella{\"\i}che \cite[Theorem V\hspace{-1.2pt}I\hspace{-1.2pt}I.3.1]{Bell2021} 
constructs a non-ordinary two-variable $p$-adic $L$-function. The growth of this $p$-adic $L$-function is also locally of type $(\alpha ,0)$ 
as in \cite{Panchishkin03}. 
}
\par 
As in these results, only in a special case where the growth of the corresponding multi-variable admissible distribution of the 
expected multi-variable $p$-adic $L$ function is of type $(\alpha , 0, \ldots ,0)$, the construction of the multi-variable $p$-adic $L$-function 
can be done without using our higher dimensional theories, but in an ad-hoc manner using the classical one-variable theory. 
However, in general multi-variable Galois deformation where the number of non-zero digit in the growth is greater than one (the growth is $(2\alpha ,\alpha )$ 
in Theorem \ref{intro two variable p-aic} for example), we can not construct a $p$-adic $L$-function and prove some fundamental properties 
in such an ad-hoc manner and the use of our theory given in this paper will be indispensable to construct multi-variable $p$-adic $L$-functions. 
\end{rem}
%
%


\section{Preparation on the precise notation}\label{preparation} 
In this section, we introduce some notation in order to state our results precisely. Let $R$ be a ring and $M$ an $R$-module. 
For any positive integer $k$, we put $M[[X_{1},\ldots,X_{k}]]=\prod_{\boldsymbol{n}\in \mathbb{Z}_{\geq 0}^{k}}M$. When $M=R$, each element $(a_{\boldsymbol{n}})_{\boldsymbol{n}}\in \prod_{\boldsymbol{n}\in \mathbb{Z}_{\geq 0}^{k}}R$ is identified with the power series $\sum_{\boldsymbol{n}\in \mathbb{Z}_{\geq 0}^{k}}^{+\infty}a_{\boldsymbol{n}}X^{\boldsymbol{n}}$ over $R$, where $X^{\boldsymbol{n}}=X_{1}^{n_{1}}\cdots X_{k}^{n_{k}}$ for each $\boldsymbol{n}\in \mathbb{Z}_{\geq 0}^{k}$. Thus, the notation $M[[X_{1},\ldots, X_{k}]]$ is justified for each $R$-module $M$. 
We regard the $R$-module $M[[X_{1},\ldots, X_{k}]]$ as an $R[[X_{1},\ldots, X_{k}]]$-module by the scalar multiplication defined by $f\cdot g=(\sum_{
\boldsymbol{l}_{1}+\boldsymbol{l}_{2}=\boldsymbol{n}, \ 
\boldsymbol{l}_{1},\boldsymbol{l}_{2}\in \mathbb{Z}_{\geq 0}^{k}
}a_{\boldsymbol{l}_{1}}m_{\boldsymbol{l}_{2}})_{\boldsymbol{n}\in \mathbb{Z}_{\geq 0}^{k}}$ for each $f=\sum_{\boldsymbol{n}\in \mathbb{Z}_{\geq 0}^{k}}a_{\boldsymbol{n}}X^{\boldsymbol{n}}\in R[[X_{1},\ldots, X_{k}]]$ and $g=(m_{\boldsymbol{n}})_{\boldsymbol{n}\in \mathbb{Z}_{\geq 0}^{k}}\in M[[X_{1},\ldots, X_{k}]]$. Further $M[X_{1},\ldots, X_{k}]=\oplus_{\boldsymbol{n}\in \mathbb{Z}_{\geq 0}^{k}}M\subset M[[X_{1},\ldots, X_{k}]]$ becomes an $R[X_{1},\ldots, X_{k}]$-submodule. We regard $M$ as an $R$-submodule of $M[X_{1},\ldots, X_{k}]$ naturally. Let $1\leq i\leq k$. We define the degree $\deg_{X_{i}} g$ of $g=(m_{\boldsymbol{n}})_{\boldsymbol{n}\in \mathbb{Z}_{\geq 0}^{k}}\in M[X_{1},\ldots, X_{k}]$ with respect to the variable $X_{i}$ to be
\begin{equation}
\deg_{X_{i}} g=\begin{cases}-\infty,\ &\mathrm{if}\ g=0,\\
\max\{n\in\mathbb{Z}_{\geq 0}\vert {}^{\exists}\boldsymbol{n}\in \mathbb{Z}_{\geq 0}^{k}\ \mathrm{s.t}\ n_{i}=n\ \mathrm{and}\ m_{\boldsymbol{n}}\neq 0\},\ &\mathrm{otherwise}.
\end{cases}
\end{equation}
{Let $\K$ be a complete subfield of $\mathbb{C}_{p}$.} Let us recall the definition of $\K$-Banch spaces. Let $M$ be a $\K$-vector space. A function $v_{M}:M\rightarrow \mathbb{R}\cup\{+\infty\}$ is called a valuation on $M$ if the following conditions are satisfied:
\begin{enumerate}
\item For $x\in M$, $v_{M}(x)=+\infty$ if and only if $x=0$.
\item For $x,y\in M$, $v_{M}(x+y)\geq \min\{v_{M}(x),v_{M}(y)\}$.
\item For $\lambda\in \K$ and $x\in M$, $v_{M}(\lambda x)=\mathrm{ord}_{p}(\lambda)+v_{M}(x)$.
\end{enumerate}
Let $v_{M}$ be a valuation on $M$. Then we say that the pair $(M,v_{M})$ is a $\K$-Banach space if $M$ is complete with respect to the topology defined by $v_{M}$. If there is no risk of confusion, we omit $v_{M}$ and call $M$ a Banach space. From now on, we fix a $\K$-Banach space $(M,v_{M})$. Let $\boldsymbol{h}\in \ord_{p}(\mathcal{O}_{\K}\backslash \{0\})^{k}$. We define 
 \begin{equation}\label{themultivariable Banach version_of_H_h}
\HH_{\boldsymbol{h}}(M) =\Big\{(m_{\boldsymbol{n}})_{\boldsymbol{n}\in \mathbb{Z}_{\geq 0}^{k}}\in M[[X_{1},\ldots, X_{k}]] 
\ \Big\vert \ 
\inf \big\{v_{M}(m_{\boldsymbol{n}} ) +\langle \boldsymbol{h},\ell(\boldsymbol{n})\rangle_{k} 
\big \}_{\boldsymbol{n}\in \mathbb{Z}_{\geq 0}^{k}}
  >-\infty\Big\}
\end{equation}
and 
\begin{equation}\label{multivariable banach version of Br}
B_{\boldsymbol{r}}(M) =\Big\{(m_{\boldsymbol{n}})_{\boldsymbol{n}\in \mathbb{Z}_{\geq 0}^{k}}\in M[[X_{1},\ldots, X_{k}]]  
\ \Big\vert \ 
\inf \big\{v_{M}(m_{\boldsymbol{n}} ) +\langle \boldsymbol{r},\boldsymbol{n}\rangle_{k} 
\big \}_{\boldsymbol{n}\in \mathbb{Z}_{\geq 0}^{k}}
  >-\infty\Big\}
\end{equation}
for each $\boldsymbol{r}\in \mathbb{Q}^{k}$. Note that $\HH_{\boldsymbol{h}}(M)$ and $B_{\boldsymbol{r}}(M)$ are $\mathcal{O}_{\K}[[X_{1},\ldots, X_{k}]]\otimes_{\mathcal{O}_{\K}}\K$-submodules of $M[[X_{1},\ldots,X_{k}]]$. We have $\HH_{\boldsymbol{h}}(M)\subset B_{\boldsymbol{r}}(M)$ for any $\boldsymbol{h}\in \ord_{p}(\mathcal{O}_{\K}\backslash \{0\})^{k}$ and $\boldsymbol{r}\in \mathbb{Q}_{>0}^{k}$ since $\displaystyle{\lim_{\boldsymbol{n}\rightarrow +\infty}}(\langle\boldsymbol{r},\boldsymbol{n}\rangle_{k}-\langle \boldsymbol{h},\ell(\boldsymbol{n})\rangle_{k})=+\infty$.

If $M=\K$, $\HH_{\boldsymbol{h}}(\K)$ is equal to the module 
$\HH_{\boldsymbol{h}/\K}$ defined in \eqref{equation:themultivariableversion_of_H_h}. For each $f=(m_{\boldsymbol{n}})_{\boldsymbol{n}\in \mathbb{Z}_{\geq 0}^{k}}\in \HH_{\boldsymbol{h}}(M)$, we put 
\begin{equation}
v_{\HH_{\boldsymbol{h}}}(f)=\inf \big\{v_{M}(m_{\boldsymbol{n}} ) +\langle \boldsymbol{h},\ell(\boldsymbol{n})\rangle_{k} 
\big \}_{\boldsymbol{n}\in \mathbb{Z}_{\geq 0}^{k}}.
\end{equation}
For each $f=(m_{\boldsymbol{n}})_{\boldsymbol{n}\in \mathbb{Z}_{\geq 0}^{k}}\in B_{\boldsymbol{r}}(M)$, we put 
\begin{equation}\label{equation:definition_of_v_r}
v_{\boldsymbol{r}}(f)=\inf \big\{v_{M}(m_{\boldsymbol{n}} ) +\langle \boldsymbol{r},\boldsymbol{n}\rangle_{k} 
\big \}_{\boldsymbol{n}\in \mathbb{Z}_{\geq 0}^{k}}. 
\end{equation}
Then, we have the following:
\begin{pro}\label{HHh is a k-banach space}
Let $\K$ be a complete subfield of $\mathbb{C}_{p}$ and let $M$ be a $\K$-Banach space. Then the pairs $(\HH_{\boldsymbol{h}}(M),v_{\HH_{\boldsymbol{h}}})$ and $(B_{\boldsymbol{r}}(M),v_{\boldsymbol{r}})$ are $\K$-Banach spaces.
\end{pro}
\begin{proof}
We prove that $(\HH_{\boldsymbol{h}}(M),v_{\HH_{\boldsymbol{h}}})$ is a $\K$-Banach space.
It is easy to see that $v_{\HH_{\boldsymbol{h}}}(f)=+\infty$ if and only if $f=0$ and we have $v_{\HH_{\boldsymbol{h}}}(\lambda f)=\ord_{p}(\lambda)+v_{\HH_{\boldsymbol{h}}}(f)$ for each $\lambda\in \K$ and $f\in \HH_{\boldsymbol{h}}(M)$. Since $v_{M}(m^{(1)}+m^{(2)})\geq \min\{v_{M}(m^{(1)}),v_{M}(m^{(2)})\}$ for each $m^{(1)},m^{(2)}\in M$, we can prove that $v_{\HH_{\boldsymbol{h}}}(f+g)\geq \min\{v_{\HH_{\boldsymbol{h}}}(f),v_{\HH_{\boldsymbol{h}}}(g)\}$ for each $f,g\in \HH_{\boldsymbol{h}}(M)$ easily. Then, $v_{\HH_{\boldsymbol{h}}}$ is a valuation on $\HH_{\boldsymbol{h}}(M)$. 

Next, we prove that $\HH_{\boldsymbol{h}}(M)$ is complete with respecet to the topology induced by $v_{\HH_{\boldsymbol{h}}}$. Let $(f_{l})_{l\in\mathbb{Z}_{\geq 0}}$ be a Cauchy sequence of $\HH_{\boldsymbol{h}}(M)$. We put $f_{l}=(m_{\boldsymbol{n}}^{(l)})_{\boldsymbol{n}\in \mathbb{Z}_{\geq 0}^{k}}$. Let $\boldsymbol{n}\in \mathbb{Z}_{\geq 0}^{k}$. Since $v_{M}(m_{\boldsymbol{n}}^{(l)}-m_{\boldsymbol{n}}^{(n)})\geq v_{\HH_{\boldsymbol{h}}}(f_{n}-f_{l})-\langle \boldsymbol{h},\ell(\boldsymbol{n})\rangle_{k}$ for each $l,n\in \mathbb{Z}_{\geq 0}$, $(m_{\boldsymbol{n}}^{(l)})_{l\in \mathbb{Z}_{\geq 0}}$ is a Cauchy sequence in $M$ and there exists a limit $m_{\boldsymbol{n}}=\displaystyle{\lim_{l\rightarrow +\infty}}m_{\boldsymbol{n}}^{(l)}\in M$. Further, we have $v_{M}(m_{\boldsymbol{n}})+\langle \boldsymbol{h},\ell(\boldsymbol{n})\rangle_{k}\geq \inf\{v_{\HH_{\boldsymbol{h}}}(f_{l})\}_{l\in\mathbb{Z}_{\geq 0}}$. Define $f=(m_{\boldsymbol{n}})_{\boldsymbol{n}}\in M[[X_{1},\ldots, X_{k}]]$. Since $v_{M}(m_{\boldsymbol{n}})+\langle \boldsymbol{h},\ell(\boldsymbol{n})\rangle_{k}\geq \inf\{v_{\HH_{\boldsymbol{h}}}(f_{l})\}_{l\in\mathbb{Z}_{\geq 0}}$ for each $\boldsymbol{n}\in \mathbb{Z}_{\geq 0}^{k}$, we see that $f\in \HH_{\boldsymbol{h}}(M)$.

We prove that $f=\displaystyle{{\lim}_{l\rightarrow +\infty}}f_{l}$. Let $A>0$. Since $(f_{l})_{l\in\mathbb{Z}_{\geq 0}}$ is a Cauchy sequence, there exists an $N\in \mathbb{Z}_{\geq 0}$ such that for each $l,n\geq N$, we have $v_{\HH_{\boldsymbol{h}}}(f_{l}-f_{n})\geq A$. Therefore, we have 
\small 
\begin{align*}
v_{M}(m_{\boldsymbol{n}}-m_{\boldsymbol{n}}^{(n)})+\langle \boldsymbol{h},\ell(\boldsymbol{n})\rangle_{k}&=\displaystyle{\lim_{l\rightarrow+\infty}}v_{M}(m_{\boldsymbol{n}}^{(l)}-m_{\boldsymbol{n}}^{(n)})+\langle \boldsymbol{h},\ell(\boldsymbol{n})\rangle_{k} \geq \inf\{v_{\HH_{\boldsymbol{h}}}(f_{l}-f_{n})\}_{l,n\geq N}\geq A
\end{align*}
\normalsize 
for each $n\geq N$ and $\boldsymbol{n}\in \mathbb{Z}_{\geq 0}^{k}$. Thus, $v_{\HH_{\boldsymbol{h}}}(f-f_{n})\geq A$ for each $n\geq N$ and we conclude that $f=\displaystyle{{\lim}_{l\rightarrow +\infty}}f_{l}$. In the same way, we can prove that $(B_{\boldsymbol{r}}(M),v_{\boldsymbol{r}})$ is a $\K$-Banach space.
\end{proof}
\begin{pro}\label{mult Br prodct equlity}
Let $f\in B_{\boldsymbol{r}}(\K)$ and $g\in B_{\boldsymbol{r}}(M)$ with $\boldsymbol{r}\in \mathbb{Q}^{k}$. Then, we have $fg\in B_{\boldsymbol{r}}(M)$ and $v_{\boldsymbol{r}}(fg)=v_{\boldsymbol{r}}(f)+v_{\boldsymbol{r}}(g)$.
\end{pro}
\begin{proof}
Put $f=\sum_{\boldsymbol{n}\in \mathbb{Z}_{\geq 0}^{k}}a_{\boldsymbol{n}}X^{\boldsymbol{n}}$ and $g=(m_{\boldsymbol{n}})_{\boldsymbol{n}\in \mathbb{Z}_{\geq 0}^{k}}$. We can assume that $f\neq 0$ and $g\neq 0$.  For each $\boldsymbol{l}_{1},\boldsymbol{l}_{2}\in \mathbb{Z}_{\geq 0}^{k}$, 
the equality $v_{M}(a_{\boldsymbol{l}_{1}}m_{\boldsymbol{l}_{2}})+\langle \boldsymbol{r},(\boldsymbol{l}_{1}+\boldsymbol{l}_{2})\rangle_{k}=(\ord_{p}(a_{\boldsymbol{l}_{1}})+\langle \boldsymbol{r},\boldsymbol{l}_{1}\rangle_{k})+(v_{M}(m_{\boldsymbol{l}_{2}})+\langle \boldsymbol{r},\boldsymbol{l}_{2}\rangle_{k})$ 
implies that $fg\in B_{\boldsymbol{r}}(M)$ and $v_{\boldsymbol{r}}(fg)\geq v_{\boldsymbol{r}}(f)+v_{\boldsymbol{r}}(g)$.
\par 
We assume that the set $S_{f,\boldsymbol{r}} =\{\boldsymbol{n}\in \mathbb{Z}_{\geq 0}^{k}\ \vert \ v_{\boldsymbol{r}}(f)=\ord_{p}(a_{\boldsymbol{n}})+\langle \boldsymbol{r},\boldsymbol{n}\rangle_{k}\}$ and the $S_{g ,\boldsymbol{r}} = \{\boldsymbol{n}\in \mathbb{Z}_{\geq 0}^{k}\ \vert \ v_{\boldsymbol{r}}(g)=v_{M}(m_{\boldsymbol{n}})+\langle \boldsymbol{r},\boldsymbol{n}\rangle_{k}\}$ are both non-empty. 
We take the minimum elements $\boldsymbol{n}_{f}$ and $\boldsymbol{n}_{g}$ of $S_{f ,\boldsymbol{r}}$ and $S_{g ,\boldsymbol{r}}$ respectively 
with respect to the lexicographic order. Then, we see that $v_{M}(a_{\boldsymbol{l}_{1}}m_{\boldsymbol{l}_{2}})+\langle \boldsymbol{r},\boldsymbol{n}_{f}+\boldsymbol{n}_{g}\rangle_{k}>v_{M}(a_{\boldsymbol{n}_{f}}m_{\boldsymbol{n}_{g}})+\langle \boldsymbol{r},\boldsymbol{n}_{f}+\boldsymbol{n}_{g}\rangle_{k}$ for each $\boldsymbol{l}_{1},\boldsymbol{l}_{2}\in \mathbb{Z}_{\geq 0}^{k}$ satisfying $\boldsymbol{l}_{1}+\boldsymbol{l}_{2}=\boldsymbol{n}_{f}+\boldsymbol{n}_{g}$ and $(\boldsymbol{l}_{1},\boldsymbol{l}_{2})\neq (\boldsymbol{n}_{f},\boldsymbol{n}_{g})$. Thus, we have
\begin{multline*}
v_{\boldsymbol{r}}(fg)\leq v_{M}\left(\sum_{\substack{\boldsymbol{l}_{1}+\boldsymbol{l}_{2}=\boldsymbol{n}_{f}+\boldsymbol{n}_{g}\\ \boldsymbol{l}_{1},\boldsymbol{l}_{2}\geq 0}}a_{\boldsymbol{l}_{1}}m_{\boldsymbol{l}_{2}}\right)+\langle \boldsymbol{r},\boldsymbol{n}_{f}+\boldsymbol{n}_{g}\rangle_{k}\\
=v_{M}(a_{\boldsymbol{n}_{f}}m_{\boldsymbol{n}_{g}})+\langle \boldsymbol{r},\boldsymbol{n}_{f}+\boldsymbol{n}_{g}\rangle_{k}=v_{\boldsymbol{r}}(f)+v_{\boldsymbol{r}}(g).
\end{multline*}
Therefore, we have $v_{\boldsymbol{r}}(fg)=v_{\boldsymbol{r}}(f)+v_{\boldsymbol{r}}(g)$. 

Next, we prove that $v_{\boldsymbol{r}}(fg)=v_{\boldsymbol{r}}(f)+v_{\boldsymbol{r}}(g)$ for general $f\in B_{\boldsymbol{r}}(\K)\backslash \{0\}$ and $g\in B_{\boldsymbol{r}}(M)\backslash\{0\}$. We have a natural inclusion $B_{\boldsymbol{r}}(M)\rightarrow B_{\boldsymbol{s}}(M)$ for each $\boldsymbol{s}\in \mathbb{Q}^{k}$ such that $\boldsymbol{s}\geq \boldsymbol{r}$. Further, we see that $S_{f,\boldsymbol{s}}\neq \emptyset$ and $S_{g,\boldsymbol{s}}\neq \emptyset$ for every $\boldsymbol{s}\in \mathbb{Q}^{k}$ such that $s_{i}>r_{i}$ with $1\leq i\leq k$. Then, we have
$$
v_{\boldsymbol{r}}(fg)=\lim_{\substack{\|\boldsymbol{s}-\boldsymbol{r}\|\rightarrow 0\\ \boldsymbol{s}\in \prod_{i=1}^{k}\mathbb{Q}_{>r_{i}}}}v_{\boldsymbol{s}}(fg)
=\lim_{\substack{\|\boldsymbol{s}-\boldsymbol{r}\|\rightarrow 0\\ \boldsymbol{s}\in \prod_{i=1}^{k}\mathbb{Q}_{>r_{i}}}}(v_{\boldsymbol{s}}(f)+v_{\boldsymbol{s}}(g)) 
=v_{\boldsymbol{r}}(f)+v_{\boldsymbol{r}}(g),
$$
where $\|\boldsymbol{s}-\boldsymbol{r}\|=\sqrt{\langle \boldsymbol{s}-\boldsymbol{r}, \boldsymbol{s}-\boldsymbol{r}\rangle_{k}}$. This completes the proof.
\end{proof}
Next, we recall the definition of complete tensor products on Banach spaces. 
Let $(M,v_{M})$ and $(N,v_{N})$ be $\K$-Banach spaces. For each $c\in M\otimes_{\K}N$, we define $v_{M,N}(c)$ to be the least upper bound of $\min\{v_{M}(m_{i})+v_{N}(n_{i})\}_{i}$ among all representations $c=\sum_{i} m_{i}\otimes n_{i}$. 
It is easy to see that $v_{M,N}(0)=+\infty$, $v_{M,N}(x+y)\geq\min\{v_{M,N}(x),v_{M,N}(y)\}$ and $v_{M,N}(\lambda x)=\ord_{p}(\lambda)+v_{M,N}(x)$ for each $x,y\in M\otimes_{\K}N$ and $\lambda\in \K$. 
Let $x\in M\otimes_{\K}N\backslash\{0\}$. We take finite dimensional $\K$-vector subspaces $M_{0}\subset M$ and $N_{0}\subset N$ such that $x\in M_{0}\otimes_{\K}N_{0}$ and  
{we put $v_{M_{0}}=v_{M}\vert_{M_{0}}$ and $v_{N_{0}}=v_{N}\vert_{N_{0}}$. In the same way as $v_{M,N}$, we define $v_{M_{0},N_{0}}: M_{0}\otimes_{\K}N_{0}\rightarrow \mathbb{R}\cup\{+\infty\}$.} We have $v_{M,N}(x)=v_{M_{0},N_{0}}(x)$ by \cite[Lemme 3.1]{Jerome2013}. 
Since $(M_{0}\otimes_{\K}N_{0},v_{M_{0},N_{0}})$ is a $\K$-Banach space, we see that $v_{M,N} (x) =v_{M_{0},N_{0}}(x)\neq +\infty$. Thus $v_{M,N}$ is a valuation on $M\otimes_{\K}N$. We denote by $M\widehat{\otimes}_{\K}N$ the completion of $(M\otimes_{\K}N,v_{M,N})$. We call $M\widehat{\otimes}_{\K}N$ the complete tensor product of $(M,v_{M})$ and $(N,v_{N})$. Let $i_{M,N}: M\otimes_{\K}N\rightarrow M\widehat{\otimes}_{\K}N$ be the natural map. We write $x\widehat{\otimes}_{\K}y$ for $i_{M,N}(x\otimes_{\K}y)$ where $x\in M$ and $y\in N$. For each closed intermediate field $\mathcal{L}$ of $\mathbb{C}_{p}\slash \K$, we put 
\begin{equation}\label{definition of ML}
M_{\mathcal{L}}=M\widehat{\otimes}_{\K}\mathcal{L}\end{equation}
and {we denote by $v_{M_{\mathcal{L}}}$ the valuation $v_{M,\mathcal{L}}$ on $M_{\mathcal{L}}$.} By \cite[Lemme 3.1]{Jerome2013}, we have $v_{M_{\mathcal{L}}}(x\widehat{\otimes}_{\K}1)=v_{M}(x)$ for every $x\in M$. Further, it is known that we have $M_{\mathcal{L}}=M\otimes_{\K}\mathcal{L}$ if $\mathcal{L}$ is a finite extension of $\K$.
Let $\boldsymbol{r}\in \mathbb{Q}^{k}$ and $\boldsymbol{b}=(b_{1},\ldots, b_{k})\in \overline{\K}^{k}$ such that $\ord_{p}(b_{i})>r_{i}$ for each $1\leq i\leq k$. For each $f=(m_{\boldsymbol{n}})_{\boldsymbol{n}\in \mathbb{Z}_{\geq 0}^{k}}\in B_{\boldsymbol{r}}(M)$, we define a substitution $f(\boldsymbol{b})\in M_{\K(b_{1},\ldots, b_{k})}$ to be
\begin{equation}\label{substitiution of banach Br(M)}
f(\boldsymbol{b})=\sum_{\boldsymbol{n}\in \mathbb{Z}_{\geq 0}^{k}}m_{\boldsymbol{n}}\otimes_{\K}\boldsymbol{b}^{\boldsymbol{n}},
\end{equation}
where $\boldsymbol{b}^{\boldsymbol{n}}=b_{1}^{n_{1}}\cdots b_{k}^{n_{k}}$ with $\boldsymbol{n}\in \mathbb{Z}_{\geq 0}^{k}$. 
\par 
{Let $\epsilon>0$ and $\mathcal{L}$ be a finite extension of $\K$. By \cite[Proposition 3 in \S2.6.2]{BGR1984}, there exists a basis $b_{1}\ldots, b_{d}$ of $ \mathcal{L}$ over $\K$ such that we have 
\begin{equation}\label{boldh H isometry completetensor eq prop3 bgr}
\min\{\ord_{p}(a_{i}b_{i})\}_{i=1}^{d}\geq \ord_{p}(b)-\epsilon
\end{equation}
 for every element $(a_{1},\ldots, a_{d})\in \K^{d}$, where $b=\sum_{i=1}^{d}a_{i}b_{i}\in \mathcal{L}$. We prove that
\begin{equation}\label{eq:boldh H isometry completetensor}
\min\{v_{M}(m_{i})+\ord_{p}(b_{i})\}_{i=1}^{d}\geq v_{M_{\mathcal{L}}}(m)-\epsilon
\end{equation}
for every $(m_{1},\ldots, m_{d})\in M^{d}$, where $m=\sum_{i=1}^{d}m_{i}\otimes_{\K}b_{i}\in M_{\mathcal{L}}$. Let $(m_{1},\ldots, m_{d})\in M^{d}$ and put $m=\sum_{i=1}^{d}m_{i}\otimes_{\K}b_{i}$. Assume that $m$ has a presentation $m=\sum_{j=1}^{n}m_{j}^{\prime}\otimes_{\K}b_{j}^{\prime}$ with $m_{j}^{\prime}\in M$ and $b_{j}^{\prime}\in \mathcal{L}$ where $n\in \mathbb{Z}_{\geq 1}$. Since $b_{1},\ldots, b_{d}$ is a basis of $\mathcal{L}$ over $\K$,  for each $1\leq j\leq n$, there exists a unique $d$-tuple $(a_{1,j},\ldots, a_{d,j})\in \K^{d}$ such that $b_{j}^{\prime}=\sum_{i=1}^{d}a_{i,j}b_{i}$. Thus, we have $m=\sum_{j=1}^{n}m_{j}^{\prime}\otimes_{\K}b_{j}^{\prime}=\sum_{i=1}^{d}(\sum_{j=1}^{n}a_{i,j}m_{j}^{\prime})\otimes_{\K}b_{i}$. Since $\sum_{i=1}^{d}m_{i}\otimes_{\K}b_{i}=\sum_{i=1}^{d}(\sum_{j=1}^{n}a_{i,j}m_{j}^{\prime})\otimes_{\K}b_{i}$, we have $m_{i}=\sum_{j=1}^{n}a_{i,j}m_{j}^{\prime}$ for each $1\leq i\leq d$. By \eqref{boldh H isometry completetensor eq prop3 bgr}, we see that
\begin{align*}
\min\{v_{M}(m_{j}^{\prime})+\ord_{p}(b_{j}^{\prime})\}_{j=1}^{n}-\epsilon&\leq \min\{v_{M}(m_{j}^{\prime})+\ord_{p}(a_{i,j}b_{i})\}_{\substack{1\leq i\leq d\\ 1\leq j\leq n}}\\
&=\min\{\min\{v_{M}(a_{i,j}m_{j}^{\prime})\}_{j=1}^{n}+\ord_{p}(b_{i})\}_{i=1}^{d}\\
&\leq \min\{v_{M}(m_{i})+\ord_{p}(b_{i})\}_{i=1}^{d}.
\end{align*}
Since the definition of $v_{M_{\mathcal{L}}}(m)$ is the least upper bound of $\min\{v_{M}(m_{j}^{\prime})+\ord_{p}(b_{j}^{\prime})\}_{j=1}^{n}$ among all representations $m=\sum_{j=1}^{n}m_{j}^{\prime}\otimes_{\K}b_{j}^{\prime}$, we see that $v_{M_{\mathcal{L}}}(m)-\epsilon \leq \min\{v_{M}(m_{i})+\ord_{p}(b_{i})\}_{i=1}^{d}$ and we have \eqref{eq:boldh H isometry completetensor}.}

We prepare some notation and recall some results on Banach spaces. For a reference, we mention \cite{BGR1984}. Let $(M,v_{M})$ and $(N,v_{N})$ be $\K$-Banach spaces. We define a valuation $v_{M\oplus N}$ on $M\oplus N$ to be $v_{M\oplus N}((m,n))=\min\{v_{M}(m),v_{N}(n)\}$ for each $m\in M$ and $n\in N$. Then it is easy to see that $(M\oplus N,v_{M\oplus N})$ is a $\K$-Banach space. We say that a $\K$-linear map $f:M\rightarrow N$ is bounded if the set $\{v_{N}(f(x))-v_{M}(x)\}_{x\in M\backslash \{0\}}$ is bounded below. In particular, $f$ is called an isometry, if $v_{N}(f(x))=v_{M}(x)$ for all $x\in M$.  As mentioned below \eqref{definition of ML}, the natural map $M\rightarrow M_{\mathcal{L}}$ is an isometry for each closed intermediate field $\mathcal{L}$ of $\mathbb{C}_{p}\slash \K$. We denote by $\mathfrak{L}(M,N)$ the $\K$-vector space of bounded $\K$-linear maps from $M$ to $N$. For each $f\in\mathfrak{L}(M,N)$, we put
\begin{equation}\label{valuation of bounded operator}
v_{\mathfrak{L}}(f)=\begin{cases} +\infty,\ &\mathrm{if}\ M=\{0\},\\
\inf\{v_{N}(f(x))-v_{M}(x)\}_{x\in M\backslash \{0\}},\ &\mathrm{if}\ M\neq \{0\}.\end{cases}
\end{equation}
It is known that $(\mathfrak{L}(M,N),v_{\mathfrak{L}})$ is a $\K$-Banach space ($cf$. \cite[Proposition 4 in \S2.1.6]{BGR1984}). If $f\in \mathfrak{L}(M,N)$ is bijective, we call $f$ a $\K$-Banach isomorphism from $M$ to $N$. By the open mapping theorem, if $f$ is a $\K$-Banach isomorphism, $f^{-1}$ is also a $\K$-Banach isomorphism. We say that $f\in \mathfrak{L}(M,N)$ is an isometric isomorphism if $f$ is a bijective isometry. To prove that a $\K$-Banach isomorphism $f:M\stackrel{\sim}{\rightarrow}N$ is an isometry, the following lemma is useful.
\begin{lem}\label{easy lemma on isometry}
Let $f:M\stackrel{\sim}{\rightarrow}N$ be a $\K$-Banach isomorphism. We assume that $v_{\mathfrak{L}}(f)\geq 0$ and $v_{\mathfrak{L}}(f^{-1})\geq 0$. Then $f$ is an isometric isomorphism. 
\end{lem} 
\begin{proof}
For each $x\in M$, we have $v_{N}(f(x))\geq v_{\mathfrak{L}}(f)+v_{M}(x)\geq v_{M}(x)$ and $v_{M}(x)=v_{M}(f^{-1}f(x))\geq v_{\mathfrak{L}}(f^{-1})+v_{N}(f(x))\geq v_{N}(f(x))$ 
by \eqref{valuation of bounded operator}. Hence $f$ is an isometry.
\end{proof}
Let $(\HH_{\boldsymbol{h}}(M),v_{\HH_{\boldsymbol{h}}})$ and $(B_{\boldsymbol{r}}(M),v_{\boldsymbol{r}})$ be the  $\K$-Banach spaces defined in \eqref{themultivariable Banach version_of_H_h} and \eqref{multivariable banach version of Br} for each $\boldsymbol{h}\in \ord_{p}(\mathcal{O}_{\K}\backslash \{0\})^{k}$ and $\boldsymbol{r}\in \mathbb{Q}^{k}$ with $k\in \mathbb{Z}_{\geq 1}$. We have the following:
\begin{pro}\label{isometry ofHh for induction}
\begin{enumerate}
\item
Let $\boldsymbol{h}\in \ord_{p}(\mathcal{O}_{\K}\backslash\{0\})^{k}$ and let $h\in \ord_{p}(\mathcal{O}_{\K}\backslash\{0\})$. We can define an isometric isomorphism
$$\varphi:\HH_{h}(\HH_{\boldsymbol{h}}(M))\stackrel{\sim}{\rightarrow}\HH_{(\boldsymbol{h},h)}(M)$$
by setting $\varphi ((f^{(n)})_{n=0}^{+\infty} )= (m_{\boldsymbol{n}}^{(n)})_{(\boldsymbol{n},n)\in \mathbb{Z}_{\geq 0}^{k+1}}$ where $f^{(n)}=(m_{\boldsymbol{n}}^{(n)})_{\boldsymbol{n}\in \mathbb{Z}_{\geq 0}^{k}}\in \HH_{\boldsymbol{h}}(M)$ for each $n\in \mathbb{Z}_{\geq 0}$.\label{isometry ofHh for induction1}
\item Let $\boldsymbol{r}\in \mathbb{Q}^{k}$ and $r\in \mathbb{Q}$. We can define an isometric isomorphism
$$\varphi:B_{r}(B_{\boldsymbol{r}}(M))\stackrel{\sim}{\rightarrow}B_{(\boldsymbol{r},r)}(M)$$
by setting $\varphi ((f^{(n)})_{n=0}^{+\infty} ) = (m_{\boldsymbol{n}}^{(n)})_{(\boldsymbol{n},n)\in \mathbb{Z}_{\geq 0}^{k+1}}$ where $f^{(n)}=(m_{\boldsymbol{n}}^{(n)})_{\boldsymbol{n}\in \mathbb{Z}_{\geq 0}^{k}}\in B_{\boldsymbol{r}}(M)$ for each $n\in \mathbb{Z}_{\geq 0}$. \label{isometry ofHh for induction2}
\end{enumerate}
\end{pro}
\begin{proof}
We prove \eqref{isometry ofHh for induction1}. Let $(f^{(n)})_{n\in \mathbb{Z}_{\geq 0}}\in \HH_{h}(\HH_{\boldsymbol{h}}(M))$ with $f^{(n)}=(m_{\boldsymbol{n}}^{(n)})_{\boldsymbol{n}\in \mathbb{Z}_{\geq 0}^{k}}\in \HH_{\boldsymbol{h}}(M)$ for each $n\in \mathbb{Z}_{\geq 0}$. By the inequality
\begin{align*}
v_{M}(m_{\boldsymbol{n}}^{(n)})+\langle (\boldsymbol{h},h),\ell((\boldsymbol{n},n))\rangle_{k+1}&=(v_{M}(m_{\boldsymbol{n}}^{(n)})+\langle \boldsymbol{h},\ell(\boldsymbol{n})\rangle_{k})+h\ell(n)\\
&\geq v_{\HH_{\boldsymbol{h}}}(f^{(n)})+h\ell(n)\\
&\geq v_{\HH_{h}(\HH_{\boldsymbol{h}}(M))}((f^{(n)})_{n=0}^{+\infty}),
\end{align*}
{we have 
\begin{equation}\label{isometry ofHh for induction varphi well bounded eq}
v_{\HH_{(\boldsymbol{h},h)}}((m_{(\boldsymbol{n},n)})_{(\boldsymbol{n},n)\in \mathbb{Z}_{\geq 0}})\geq v_{\HH_{h}(\HH_{\boldsymbol{h}}(M))}((f^{(n)})_{n=0}^{+\infty})>-\infty.\end{equation}
By regarding that $(m_{(\boldsymbol{n},n)})_{(\boldsymbol{n},n)\in \mathbb{Z}_{\geq 0}^{k}\times \mathbb{Z}_{\geq 0}}\in \HH_{(\boldsymbol{h},h)}(M)$, we define the $\K$-linear map $\varphi :\HH_{h}(\HH_{\boldsymbol{h}}(M))\rightarrow \HH_{(\boldsymbol{h},h)}(M)$ and we have $v_{\mathfrak{L}}(\varphi)\geq 0$ by \eqref{isometry ofHh for induction varphi well bounded eq}.}
\par 
Next, we prove that $\varphi$ has an inverse map $\varphi^{-1}$ with $v_{\mathfrak{L}}(\varphi^{-1})\geq 0$. Let $f=(m_{\boldsymbol{n}})_{\boldsymbol{n}\in \mathbb{Z}_{\geq 0}^{k+1}}\in \HH_{(\boldsymbol{h},h)}(M)$. Fix a non-negative integer $n$. We have
\begin{align*}
v_{M}(m_{(\boldsymbol{n},n)})+\langle \boldsymbol{h},\ell(\boldsymbol{n})\rangle_{k}&=(v_{M}(m_{(\boldsymbol{n},n)})+\langle (\boldsymbol{h},h),\ell((\boldsymbol{n},n))\rangle_{k+1})-h\ell(n)\\
&\geq v_{\HH_{(\boldsymbol{h},h)}}(f)-h\ell(n).
\end{align*}
for each $\boldsymbol{n}\in \mathbb{Z}_{\geq 0}^{k}$. Then, $(m_{(\boldsymbol{n},n)})_{\boldsymbol{n}\in \mathbb{Z}_{\geq 0}^{k}}$ is an element of $\HH_{\boldsymbol{h}}(M)$ which satisfies
$$v_{\HH_{\boldsymbol{h}}}((m_{(\boldsymbol{n},n)})_{\boldsymbol{n}\in \mathbb{Z}_{\geq 0}^{k}})\geq v_{\HH_{(\boldsymbol{h},h)}}(f)-h\ell(n).$$ 
Therefore, we can define a map $\psi: \HH_{(\boldsymbol{h},h)}(M)\rightarrow \HH_{h}(\HH_{\boldsymbol{h}}(M))$ by setting $\psi ((m_{\boldsymbol{n}})_{\boldsymbol{n}\in \mathbb{Z}_{\geq 0}^{k+1}})
= (f^{(n)})_{n=0}^{+\infty}$ with $f^{(n)}=(m_{(\boldsymbol{n},n)})_{\boldsymbol{n}\in \mathbb{Z}_{\geq 0}^{k}}$ for each $n\in \mathbb{Z}_{\geq 0}$. Further, we have $v_{\mathfrak{L}}(\psi)\geq 0$. It is easy to see that $\psi=\varphi^{-1}$. Then $\varphi$ is an isometric isomorphism by $\mathrm{Lemma\ \ref{easy lemma on isometry}}$. 
We can prove \eqref{isometry ofHh for induction2} in the same way as \eqref{isometry ofHh for induction1}.
\end{proof}
We have the following:
\begin{pro}\label{boldh H isometry completetensor}
Let $\mathcal{L}$ be a finite extension of $\mathcal{K}$ and let $k\in \mathbb{Z}_{\geq 1}$.
\begin{enumerate}
\item Let $\boldsymbol{h}\in \ord_{p}(\mathcal{O}_{\K}\backslash\{0\})^{k}$. Then, the natural map
$$\varphi:(\HH_{\boldsymbol{h}}(M))_{\mathcal{L}}\rightarrow\HH_{\boldsymbol{h}}(M_{\mathcal{L}}),$$
which is defined by setting $\varphi (f\otimes_{\K}a )= af$ for each $f\in \HH_{\boldsymbol{h}}(M)$ and for each $a\in \mathcal{L}$, is an isometric isomorphism.\label{boldh H isometry completetensor1}
\item Let $\boldsymbol{r}\in \mathbb{Q}^{k}$. Then, the natural map
$$\varphi:(B_{\boldsymbol{r}}(M))_{\mathcal{L}}\rightarrow B_{\boldsymbol{r}}(M_{\mathcal{L}}),$$
which is defined by setting $\varphi (f\otimes_{\K}a ) = af$ for each $f\in B_{\boldsymbol{r}}(M)$ and for each $a\in \mathcal{L}$, is an isometric isomorphism.\label{boldh H isometry completetensor2}
\end{enumerate}
\end{pro}
\begin{proof}
{We prove \eqref{boldh H isometry completetensor1}. First, we prove that $\varphi$ is well-defined. Let $f\in \HH_{\boldsymbol{h}}(M)_{\mathcal{L}}$. 
Let us express $f$ as a sum $f=\sum_{i=1}^{l}f^{(i)}\otimes_{\K}a_{i}$ where $f^{(i)}\in \HH_{\boldsymbol{h}}(M)$ and $a_{i}\in \mathcal{L}$ with $l\in \mathbb{Z}_{\geq 1}$. Put $f^{(i)}=(m_{\boldsymbol{n}}^{(i)})_{\boldsymbol{n}\in \mathbb{Z}_{\geq 0}^{k}}$. Then, we have $\sum_{i=1}^{l}a_{i}f^{(i)}=(\sum_{i=1}^{l}m^{(i)}_{\boldsymbol{n}}\otimes_{\K}a_{i})_{\boldsymbol{n}\in \mathbb{Z}_{\geq 0}^{k}}\in M_{\mathcal{L}}[[X]]$. We denote by $v_{M_{\mathcal{L}}}$ the valuation on $M_{\mathcal{L}}$ defined just after \eqref{definition of ML}. By the definition of $v_{M_{\mathcal{L}}}$, we have
\begin{equation}\label{boldh H isometry completetensor welldef 1}
v_{M_{\mathcal{L}}}\left(\sum_{i=1}^{l}m^{(i)}_{\boldsymbol{n}}\otimes_{\K}a_{i}\right)\geq \min\{v_{M}(m^{(i)}_{\boldsymbol{n}})+\ord_{p}(a_{i})\}_{i=1}^{l}
\end{equation}
for each $\boldsymbol{n}\in \mathbb{Z}_{\geq 0}^{k}$. Since $v_{\HH_{\boldsymbol{h}}}(f^{(i)})=\inf\{v_{M}(m^{(i)}_{\boldsymbol{n}})+\langle \boldsymbol{h},\ell(\boldsymbol{n})\rangle_{k}\}_{\boldsymbol{n}\in \mathbb{Z}_{\geq 0}^{k}}$ for each $1\leq i\leq l$, by \eqref{boldh H isometry completetensor welldef 1}, we have
\begin{align*}
v_{M_{\mathcal{L}}}\left(\sum_{i=1}^{l}m^{(i)}_{\boldsymbol{n}}\otimes_{\K}a_{i}\right)+\langle \boldsymbol{h},\ell(\boldsymbol{n})\rangle_{k}&\geq \min\{v_{M}(m^{(i)}_{\boldsymbol{n}})+\ord_{p}(a_{i})\}_{i=1}^{l}+\langle \boldsymbol{h},\ell(\boldsymbol{n})\rangle_{k}\\
&=\min\{(v_{M}(m^{(i)}_{\boldsymbol{n}})+\langle \boldsymbol{h},\ell(\boldsymbol{n})\rangle_{k})+\ord_{p}(a_{i})\}_{i=1}^{l}\\
&\geq \min\{v_{\HH_{\boldsymbol{h}}}(f^{(i)})+\ord_{p}(a_{i})\}_{i=1}^{l}
\end{align*}
for every $\boldsymbol{n}\in \mathbb{Z}_{\geq 0}^{k}$. Then, we have $\varphi(f)=\sum_{i=1}^{l}a_{i}f^{(i)}\in \HH_{\boldsymbol{h}}(M_{\mathcal{L}})$ and
\begin{align}\label{boldh H isometry completetensor vL(varphi)geq0}
\begin{split}
v_{\HH_{\boldsymbol{h}}}\left(\sum_{i=1}^{l}a_{i}f^{(i)}\right)&=\inf\left\{v_{M_{\mathcal{L}}}\left(\sum_{i=1}^{l}m^{(i)}_{\boldsymbol{n}}\otimes_{\K}a_{i}\right)+\langle \boldsymbol{h},\ell(\boldsymbol{n})\rangle_{k}\right\}_{\boldsymbol{n}\in \mathbb{Z}_{\geq 0}^{k}}\\
&\geq \min\{v_{\HH_{\boldsymbol{h}}}(f^{(i)})+\ord_{p}(a_{i})\}_{i=1}^{l}.
\end{split}
\end{align}
In particular, $\varphi$ is well-defined.}
\par
{Next, we prove that $v_{\mathfrak{L}}(\varphi)\geq 0$. We denote by $v_{\HH_{\boldsymbol{h}}(M)_{\mathcal{L}}}$ the valuation on $\HH_{\boldsymbol{h}}(M)_{\mathcal{L}}$ defined 
just after \eqref{definition of ML}. Let $f\in \HH_{\boldsymbol{h}}(M)_{\mathcal{L}}$. By \eqref{boldh H isometry completetensor vL(varphi)geq0}, we have 
\begin{align}\label{boldh H isometry completetensor vL(varphi)geq02}
\begin{split}
v_{\HH_{\boldsymbol{h}}}(\varphi(f))&=v_{\HH_{\boldsymbol{h}}}\left(\sum_{i=1}^{l}a_{i}f^{(i)}\right)\\
&\geq \min\{v_{\HH_{\boldsymbol{h}}}(f^{(i)})+\ord_{p}(a_{i})\}_{i=1}^{l}
\end{split}
\end{align}
for all representations $f=\sum_{i=1}^{l}f^{(i)}\otimes_{\K}a_{i}$. By the definition of $v_{\HH_{\boldsymbol{h}}(M)_{\mathcal{L}}}$, $v_{\HH_{\boldsymbol{h}}(M)_{\mathcal{L}}}(f)$ is the least upper bound of $\min\{v_{\HH_{\boldsymbol{h}}}(f^{(i)})+\ord_{p}(a_{i})\}_{i=1}^{l}$ among all representations $f=\sum_{i=1}^{l}f^{(i)}\otimes_{\K}a_{i}$. By \eqref{boldh H isometry completetensor vL(varphi)geq02}, we have $v_{\HH_{\boldsymbol{h}}}(\varphi(f))\geq v_{\HH_{\boldsymbol{h}}(M)_{\mathcal{L}}}(f)$. Thus, we have $v_{\mathfrak{L}}(\varphi)\geq 0$. }
\par 
{Next, we prove that $\varphi$ is injective. Let $b_{1},\ldots, b_{d}$ be a basis of $\mathcal{L}$ over $\K$. Let $f\in \HH_{\boldsymbol{h}}(M)_{\mathcal{L}}$ such that $\varphi(f)=0$. Since $b_{1},\ldots, b_{d}$ is a basis of $\mathcal{L}$ over $\K$, $f$ can be expressed as a sum $f=\sum_{i=1}^{d}f^{(i)}\otimes_{\K}b_{i}$ with $f^{(i)}\in \HH_{\boldsymbol{h}}(M)$ uniquely. Put $f^{(i)}=(m^{(i)}_{\boldsymbol{n}})_{\boldsymbol{n}\in \mathbb{Z}_{\geq 0}^{k}}$ with $1\leq i\leq d$. We have $\varphi(f)=\left(\sum_{i=1}^{d}m^{(i)}_{\boldsymbol{n}}\otimes_{\K}b_{i}\right)_{\boldsymbol{n}\in \mathbb{Z}_{\geq 0}^{k}}$. Since $\varphi(f)=\left(\sum_{i=1}^{d}m^{(i)}_{\boldsymbol{n}}\otimes_{\K}b_{i}\right)_{\boldsymbol{n}\in \mathbb{Z}_{\geq 0}^{k}}=0$, we see that $\sum_{i=1}^{d}m^{(i)}_{\boldsymbol{n}}\otimes_{\K}b_{i}=0$ for all $\boldsymbol{n}\in \mathbb{Z}_{\geq 0}^{k}$. Since $b_{1},\ldots, b_{d}$ is a basis of $\mathcal{L}$ over $\K$, for each $\boldsymbol{n}\in \mathbb{Z}_{\geq 0}^{k}$, the condition $\sum_{i=1}^{d}m^{(i)}_{\boldsymbol{n}}\otimes_{\K}b_{i}=0$ implies that $m_{\boldsymbol{n}}^{(i)}=0$ for every $1\leq i\leq d$. Therefore, we have $f^{(i)}=(m_{\boldsymbol{n}}^{(i)})_{\boldsymbol{n}\in \mathbb{Z}_{\geq 0}^{k}}=0$ for every $1\leq i\leq d$. Thus, we see that $f=\sum_{i=1}^{d}f^{(i)}\otimes_{\K}b_{i}=0$ and we conclude that $\varphi$ is injective.}
\par
{Next, we prove that $\varphi$ is surjective. Let $\epsilon>0$. By \cite[Proposition 3 in \S2.6.2]{BGR1984}, there exists a basis $b_{1}\ldots, b_{d}\in \mathcal{L}$ over $\K$ such that 
\begin{equation}\label{boldh H isometry completetensor surject bgr2.6.2basis}
\min\{\ord_{p}(a_{i}b_{i})\}_{i=1}^{d}\geq \ord_{p}(b)-\epsilon
\end{equation}
 for every element $(a_{1},\ldots, a_{d})\in \K^{d}$, where $b=\sum_{i=1}^{d}a_{i}b_{i}\in \mathcal{L}$.}
\par 
{Let $f=(m_{\boldsymbol{n}})_{\boldsymbol{n}\in \mathbb{Z}_{\geq 0}^{k}}\in \HH_{\boldsymbol{h}}(M_{\mathcal{L}})$ with $m_{\boldsymbol{n}}\in M_{\mathcal{L}}$. For each $\boldsymbol{n}\in \mathbb{Z}_{\geq 0}^{k}$, there exists a unique element $(m^{(1)}_{\boldsymbol{n}},\ldots, m^{(d)}_{\boldsymbol{n}})\in M^{d}$ such that 
\begin{equation}\label{boldh H isometry completetensor mboldn=sumboldsymboln(i)}
m_{\boldsymbol{n}}=\sum_{i=1}^{d}m^{(i)}_{\boldsymbol{n}}\otimes_{\K}b_{i}.
\end{equation}
 Put $f^{(i)}=(m^{(i)}_{\boldsymbol{n}})_{\boldsymbol{n}\in \mathbb{Z}_{\geq 0}^{k}}$ for each $1\leq i\leq d$. By \eqref{eq:boldh H isometry completetensor}, we see that
$$v_{M}(m^{(i)}_{\boldsymbol{n}})+\langle \boldsymbol{h},\ell(\boldsymbol{n})\rangle_{k}+\ord_{p}(b_{i})\geq v_{M_{\mathcal{L}}}(m_{\boldsymbol{n}})+\langle \boldsymbol{h},\ell(\boldsymbol{n})\rangle_{k}-\epsilon\geq v_{\HH_{\boldsymbol{h}}}(f)-\epsilon$$
for every $1\leq i\leq d$ and for every $\boldsymbol{n}\in \mathbb{Z}_{\geq 0}^{k}$. Therefore, we have $f^{(i)}\in \HH_{\boldsymbol{h}}(M)$ and 
\begin{align}\label{boldh H isometry completetensor vhhffigeevhhff-ep}
\begin{split}
v_{\HH_{\boldsymbol{h}}}(f^{(i)})+\ord_{p}(b_{i})&=\inf\{v_{M}(m^{(i)}_{\boldsymbol{n}})+\langle \boldsymbol{h},\ell(\boldsymbol{n})\rangle_{k}\}_{\boldsymbol{n}\in \mathbb{Z}_{\geq 0}^{k}}+\ord_{p}(b_{i})\\
&\geq v_{\HH_{\boldsymbol{h}}}(f)-\epsilon
\end{split}
\end{align}
for each $1\leq i\leq d$. By \eqref{boldh H isometry completetensor mboldn=sumboldsymboln(i)}, we see that 
\begin{equation}\label{boldh H isometry completetensor explict inverseimage surjective}
f=\varphi(\sum_{i=1}^{d}f^{(i)}\otimes_{\K}b_{i}).
\end{equation}
Therefore, we see that $\varphi$ is surjective.}

{Next, we prove that $v_{\mathfrak{L}}(\varphi^{-1})\geq 0$. Let $\epsilon >0$ and let $b_{1},\ldots, b_{d}$ be a basis of $\mathcal{L}$ over $\mathcal{K}$ 
which satisfies \eqref{boldh H isometry completetensor surject bgr2.6.2basis}. Let $f=(m_{\boldsymbol{n}})_{\boldsymbol{n}\in \mathbb{Z}_{\geq 0}^{k}}\in \HH_{\boldsymbol{h}}(M_{\mathcal{L}})$ with $m_{\boldsymbol{n}}\in M_{\mathcal{L}}$. For each $\boldsymbol{n}\in \mathbb{Z}_{\geq 0}^{k}$, let $(m_{\boldsymbol{n}}^{(1)},\ldots, m_{\boldsymbol{n}}^{(d)})\in M^{d}$ be the unique $d$-tuple which satisfies $m_{\boldsymbol{n}}=\sum_{i=1}^{d}m_{\boldsymbol{n}}^{(i)}\otimes_{\K}b_{i}$. By \eqref{boldh H isometry completetensor explict inverseimage surjective}, we have 
$$\varphi^{-1}(f)=\sum_{i=1}^{d}f^{(i)}\otimes_{\K}b_{i}$$
where $f^{(i)}=(m_{\boldsymbol{n}}^{(i)})_{\boldsymbol{n}\in \mathbb{Z}_{\geq 0}^{k}}\in \HH_{\boldsymbol{h}}(M)$ with $1\leq i\leq d$. By the definition of $v_{\HH_{\boldsymbol{h}}(M)_{\mathcal{L}}}$, we have  $v_{\HH_{\boldsymbol{h}}(M)_{\mathcal{L}}}(\varphi^{-1}(f))\geq \min\{v_{\HH_{\boldsymbol{h}}}(f^{(i)})+\ord_{p}(b_{i})\}_{i=1}^{d}$. By \eqref{boldh H isometry completetensor vhhffigeevhhff-ep}, we see that 
\begin{align*}
v_{\HH_{\boldsymbol{h}}(M)_{\mathcal{L}}}(\varphi^{-1}(f))&\geq \min\{v_{\HH_{\boldsymbol{h}}}(f^{(i)})+\ord_{p}(b_{i})\}_{i=1}^{d}\\
&\geq v_{\HH_{\boldsymbol{h}}}(f)-\epsilon.
\end{align*}
Thus, we have $v_{\mathfrak{L}}(\varphi^{-1})\geq -\epsilon$. Since $\epsilon$ is an arbitary positive real number, we have $v_{\mathfrak{L}}(\varphi^{-1})\geq 0$. }

By Lemma \ref{easy lemma on isometry}, we see that $\varphi$ is isometric. We can prove \eqref{boldh H isometry completetensor2} in the same way as \eqref{boldh H isometry completetensor1}.
\end{proof}
\begin{lem}\label{analytic funcuniqueof Brby Zp}
 Let $M$ be a $\K$-Banach space and $f\in B_{\boldsymbol{r}}(M)$ with $\boldsymbol{r}\in \mathbb{Q}^{k}$. If there exists an element $\boldsymbol{t}\in\mathbb{Q}^{k}$ such that $\boldsymbol{t}\geq \boldsymbol{r}$ and we have $f(\boldsymbol{x})=0$ for every $\boldsymbol{x}\in \mathbb{Z}_{p}^{k}$ with $\ord_{p}(x_{i})>t_{i}$ for each $1\leq i\leq k$, then we have $f=0$. 
 \end{lem}
 \begin{proof}
We prove this lemma by induction on $k$. Assume that $k=1$ and put $f=(m_{n})_{n\geq 0}$ with $m_{n}\in M$. If $f\neq 0$, there exists an $n_{0}\in \mathbb{Z}_{\geq 0}$ such that $m_{n_{0}}\neq 0$ and $m_{n}=0$ for every $n\in \mathbb{Z}_{\geq 0}$ such that $n<n_{0}$. Put $m^{\prime}_{n}=m_{n+n_{0}}$ for every $n\in \mathbb{Z}_{\geq 0}$ and $g=(m_{n}^{\prime})_{n\geq 0}$. Then, we see that $g\in B_{r}(M)$ and $f=X^{n_{0}}g$. Let $x\in \mathbb{Z}_{p}\backslash \{0\}$ such that $\ord_{p}(x)>t$. Since $f(x)=x^{n_{0}}g(x)=0$, we see that $g(x)=0$. Then, we see that $g(x)=0$ for every $x\in \mathbb{Z}_{p}\backslash \{0\}$ such that $\ord_{p}(x)>t$. Let $x_{n}\in \mathbb{Z}_{p}\backslash \{0\}$ be a sequence such that $\lim_{n\rightarrow +\infty}x_{n}=0$. Then, we see that $m_{n_{0}}=g(0)=\lim_{n\rightarrow +\infty}g(x_{n})=0$. This is a contradiction. Then, $f=0$. 

Next, we assume that $k\geq 2$. By Proposition \ref{isometry ofHh for induction}, we identify $B_{\boldsymbol{r}}(M)$ with $B_{r_{k}}(B_{\boldsymbol{r}^{\prime}}(M))$ where $\boldsymbol{r}^{\prime}=(r_{1},\ldots, r_{k-1})$ and put $f=(f_{n})_{n\in \mathbb{Z}_{\geq 0}}$ with $f_{n}\in B_{\boldsymbol{r}^{\prime}}(M)$. Let $\boldsymbol{x}^{\prime}\in \mathbb{Z}_{p}^{k-1}$ such that $\ord_{p}(x_{i}^{\prime})>t_{i}$ for each $1\leq i\leq k-1$. Put $f_{\boldsymbol{x}^{\prime}}=(f_{n}(x_{1}^{\prime},\ldots, x_{k-1}^{\prime}))\in B_{r_{k}}(M)$. Then, for each $x\in \mathbb{Z}_{p}$ such that $\ord_{p}(x)>t_{k}$, we have $f_{\boldsymbol{x}^{\prime}}(x)=f(x_{1}^{\prime},\ldots, x_{k-1}^{\prime},x)=0$. By applying the result in the case $k=1$ to $f_{\boldsymbol{x}^{\prime}}\in B_{r_{k}}(M)$, we see that $f_{\boldsymbol{x}^{\prime}}=0$. Thus, for each $n\in \mathbb{Z}_{\geq 0}$, we have $f_{n}(\boldsymbol{x}^{\prime})=0$. By induction on $k$, we have $f_{n}=0$ for every $n\in \mathbb{Z}_{\geq 0}$. Thus, we see that $f=(f_{n})_{n\in \mathbb{Z}_{\geq 0}}=0$.
 \end{proof}
\begin{pro}\label{Br slide prop}
Let $M$ be a $\K$-Banach space and let $\boldsymbol{r}\in \mathbb{Q}^{k}$. Let $f=(m_{\boldsymbol{n}})_{\boldsymbol{n}\in \mathbb{Z}_{\geq 0}^{k}}\in B_{\boldsymbol{r}}(M)$ and 
let $\boldsymbol{a}\in \K^{k}$ be an element satisfying $\ord_{p}(a_{i})>r_{i}$ for each $1\leq i\leq k$. 
\begin{enumerate}
\item For each $\boldsymbol{n}\in \mathbb{Z}_{\geq 0}^{k}$, we see that the series
$$\sum_{\substack{\boldsymbol{l}\in \mathbb{Z}_{\geq 0}^{k}\\  \boldsymbol{n}\leq \boldsymbol{l}}}\left(\prod_{i=1}^{k}\begin{pmatrix}l_{i}\\ n_{i}\end{pmatrix}a_{i}^{l_{i}-n_{i}}\right)m_{\boldsymbol{l}}$$
is convergent in $M$. Further, if we define an element $f_{+\boldsymbol{a}} \in M[[X_{1},\ldots, X_{k}]]$ 
to be $f_{+\boldsymbol{a}}=\left(\sum_{\boldsymbol{l}\in \mathbb{Z}_{\geq 0}^{k}, \boldsymbol{n}\leq \boldsymbol{l}}\left(\prod_{i=1}^{k}\begin{pmatrix}l_{i}\\ n_{i}\end{pmatrix}a_{i}^{l_{i}-n_{i}}\right)m_{\boldsymbol{l}}\right)_{\boldsymbol{n}\in \mathbb{Z}_{\geq 0}^{k}}$, we have $f_{+\boldsymbol{a}}\in B_{\boldsymbol{r}}(M)$ and $v_{\boldsymbol{r}}(f)=v_{\boldsymbol{r}}(f_{+\boldsymbol{a}})$. \label{Br slide prop1}
\item Let $f_{+\boldsymbol{a}}\in B_{\boldsymbol{r}}(M)$ be the element in \eqref{Br slide prop1}. Then, $f_{+\boldsymbol{a}}$ is the unique element which satisfies
$$f_{+\boldsymbol{a}}(\boldsymbol{b})=f(\boldsymbol{b}+\boldsymbol{a})$$
for every $\boldsymbol{b}\in \overline{\K}^{k}$ such that $\ord_{p}(b_{i})>r_{i}$ with $1\leq i\leq k$.\label{Br slide prop2}
\end{enumerate}
\end{pro}
\begin{proof}
First, we prove that $\sum_{\boldsymbol{l}\in \mathbb{Z}_{\geq 0}^{k}, \boldsymbol{n}\leq \boldsymbol{l}}\left(\prod_{i=1}^{k}\begin{pmatrix}l_{i}\\ n_{i}\end{pmatrix}a_{i}^{l_{i}-n_{i}}\right)m_{\boldsymbol{l}}$ is convergent in $M$ for each $\boldsymbol{n}\in \mathbb{Z}_{\geq 0}^{k}$. We have
\begin{align}\label{Br slide propproof1}
v_{M}\left(\left(\prod_{i=1}^{k}\begin{pmatrix}l_{i}\\ n_{i}\end{pmatrix}a_{i}^{l_{i}-n_{i}}\right)m_{\boldsymbol{l}}\right)\geq \left(\sum_{i=1}^{k}l_{i}\ord_{p}(a_{i})\right)+v_{M}(m_{\boldsymbol{l}})-\sum_{i=1}^{k}n_{i}\ord_{p}(a_{i})
\end{align}
for each $\boldsymbol{l}\in \mathbb{Z}_{\geq 0}^{k}$ such that $\boldsymbol{l}\geq \boldsymbol{n}$. Since $f=(m_{\boldsymbol{n}})_{\boldsymbol{n}\in \mathbb{Z}_{\geq 0}^{k}}\in B_{\boldsymbol{r}}(M)$, we see that $\lim_{\boldsymbol{l}\rightarrow +\infty}\left(\left(\sum_{i=1}^{k}l_{i}\ord_{p}(a_{i})\right)+v_{M}(m_{\boldsymbol{l}})\right)=+\infty$, which implies that 
$$\lim_{\boldsymbol{l}\rightarrow +\infty}v_{M}\left(\left(\prod_{i=1}^{k}\begin{pmatrix}l_{i}\\ n_{i}\end{pmatrix}a_{i}^{l_{i}-n_{i}}\right)m_{\boldsymbol{l}}\right)=+\infty.$$
Thus, $\sum_{\boldsymbol{l}\in \mathbb{Z}_{\geq 0}^{k}, \boldsymbol{n}\leq \boldsymbol{l}}\left(\prod_{i=1}^{k}\begin{pmatrix}l_{i}\\ n_{i}\end{pmatrix}a_{i}^{l_{i}-n_{i}}\right)m_{\boldsymbol{l}}$ is convergent in $M$.

Next, we prove that $f_{+\boldsymbol{a}}\in B_{\boldsymbol{r}}(M)$ and $v_{\boldsymbol{r}}(f_{+\boldsymbol{a}})\geq v_{\boldsymbol{r}}(f)$. By \eqref{Br slide propproof1}, we have
\begin{multline*}
v_{M}\left(\left(\prod_{i=1}^{k}\begin{pmatrix}l_{i}\\ n_{i}\end{pmatrix}a_{i}^{l_{i}-n_{i}}\right)m_{\boldsymbol{l}}\right)+\langle \boldsymbol{r},\boldsymbol{n}\rangle_{k}\\ \geq\left(\sum_{i=1}^{k}(l_{i}-n_{i})(\ord_{p}(a_{i})-r_{i})\right)+(v_{M}(m_{\boldsymbol{l}})+\langle \boldsymbol{r},\boldsymbol{l}\rangle_{k})\geq v_{\boldsymbol{r}}(f)
\end{multline*}
for each $\boldsymbol{n},\boldsymbol{l}\in \mathbb{Z}_{\geq 0}^{k}$ such that $\boldsymbol{n}\leq \boldsymbol{l}$. 
Hence, we have 
$$v_{M}\left(\sum_{\boldsymbol{l}\in \mathbb{Z}_{\geq 0}^{k}, \boldsymbol{n}\leq \boldsymbol{l}}\left(\prod_{i=1}^{k}\begin{pmatrix}l_{i}\\ n_{i}\end{pmatrix}a_{i}^{l_{i}-n_{i}}\right)m_{\boldsymbol{l}}\right)+\langle \boldsymbol{r},\boldsymbol{n}\rangle_{k}\geq v_{\boldsymbol{r}}(f)
$$ for every $\boldsymbol{n}\in \mathbb{Z}_{\geq 0}^{k}$, and we have $f_{+\boldsymbol{a}}\in B_{\boldsymbol{r}}(M)$ and
\begin{equation}\label{Br slide propproof2}
v_{\boldsymbol{r}}(f_{+\boldsymbol{a}})\geq v_{\boldsymbol{r}}(f).
\end{equation}
Next, we prove \eqref{Br slide prop2}. Let $\boldsymbol{b}\in \overline{\K}^{k}$ such that $\ord_{p}(b_{i})>r_{i}$ with $1\leq i\leq k$. For each $\boldsymbol{t}\in \mathbb{Z}_{\geq 0}^{k}$, we have
\begin{align*}
\sum_{\boldsymbol{n}\in  [\boldsymbol{0}_{k},\boldsymbol{t}]}m_{\boldsymbol{n}}(\boldsymbol{b}+\boldsymbol{a})^{\boldsymbol{n}}=\sum_{\boldsymbol{n}\in  [\boldsymbol{0}_{k},\boldsymbol{t}]}\sum_{\boldsymbol{l}\in  [\boldsymbol{n},\boldsymbol{t}]}\left(\left(\prod_{i=1}^{k}\begin{pmatrix}l_{i}\\ n_{i}\end{pmatrix}a_{i}^{l_{i}-n_{i}}\right)m_{\boldsymbol{l}}\right)\boldsymbol{b}^{\boldsymbol{n}}
\end{align*}
where $\boldsymbol{0}_{k}=(0,\ldots, 0)\in \mathbb{Z}^{k}$. Then, we see that
\begin{align}\label{Br slide propproof3}
\begin{split}
f_{+\boldsymbol{a}}(\boldsymbol{b})-\sum_{\boldsymbol{n}\in  [\boldsymbol{0}_{k},\boldsymbol{t}]}m_{\boldsymbol{n}}(\boldsymbol{b}+\boldsymbol{a})^{\boldsymbol{n}}&=\sum_{\boldsymbol{n}\in  [\boldsymbol{0}_{k},\boldsymbol{t}]}\left(\sum_{\substack{\boldsymbol{l}\in \mathbb{Z}_{\geq 0}^{k}\\  \boldsymbol{n}\leq \boldsymbol{l},\ \boldsymbol{l}\notin [\boldsymbol{n},\boldsymbol{t}]}}\left(\prod_{i=1}^{k}\begin{pmatrix}l_{i}\\ n_{i}\end{pmatrix}a_{i}^{l_{i}-n_{i}}\right)m_{\boldsymbol{l}}\right)\boldsymbol{b}^{\boldsymbol{n}}\\
&\ \ +\sum_{\substack{\boldsymbol{n}\in \mathbb{Z}_{\geq 0}^{k}\\ \boldsymbol{n}\notin [\boldsymbol{0}_{k},\boldsymbol{t}]}}\left(\sum_{\substack{\boldsymbol{l}\in \mathbb{Z}_{\geq 0}^{k}\\  \boldsymbol{n}\leq \boldsymbol{l}}}\left(\prod_{i=1}^{k}\begin{pmatrix}l_{i}\\ n_{i}\end{pmatrix}a_{i}^{l_{i}-n_{i}}\right)m_{\boldsymbol{l}}\right)\boldsymbol{b}^{\boldsymbol{n}}
\end{split}
\end{align}
for each $\boldsymbol{t}\in \mathbb{Z}_{\geq 0}^{k}$. By \eqref{Br slide propproof1},  we have
\begin{align*}
&v_{M}\left(\left(\prod_{i=1}^{k}\begin{pmatrix}l_{i}\\ n_{i}\end{pmatrix}a_{i}^{l_{i}-n_{i}}\right)m_{\boldsymbol{l}}\boldsymbol{b}^{\boldsymbol{n}}\right) \\
& \geq\sum_{i=1}^{k}(l_{i}-n_{i})\ord_{p}(a_{i})+v_{M}(m_{\boldsymbol{l}})+\sum_{i=1}^{k}n_{i}\ord_{p}(b_{i}) \notag \\
& =\sum_{i=1}^{k}(l_{i}-n_{i})(\ord_{p}(a_{i})-r_{i})+(v_{M}(m_{\boldsymbol{l}})+\langle \boldsymbol{r},\boldsymbol{l}\rangle_{k})+\sum_{i=1}^{k}n_{i}(\ord_{p}(b_{i})-r_{i}) \notag \\
& \geq v_{\boldsymbol{r}}(f)+\sum_{i=1}^{k}n_{i}(\ord_{p}(b_{i})-r_{i}) \notag 
\end{align*}
for every $\boldsymbol{n},\boldsymbol{l}\in \mathbb{Z}_{\geq 0}^{k}$ such that $\boldsymbol{n}\leq \boldsymbol{l}$. Thus, we see that
\begin{equation}\label{Br slide propproof4}
\lim_{\boldsymbol{t}\rightarrow +\infty}\sum_{\substack{\boldsymbol{n}\in \mathbb{Z}_{\geq 0}^{k}\\ \boldsymbol{n}\notin [\boldsymbol{0}_{k},\boldsymbol{t}]}}\left(\sum_{\substack{\boldsymbol{l}\in \mathbb{Z}_{\geq 0}^{k}\\  \boldsymbol{n}\leq \boldsymbol{l}}}\left(\prod_{i=1}^{k}\begin{pmatrix}l_{i}\\ n_{i}\end{pmatrix}a_{i}^{l_{i}-n_{i}}\right)m_{\boldsymbol{l}}\right)\boldsymbol{b}^{\boldsymbol{n}}=0.
\end{equation}
Since we have
\begin{align*}
&v_{M}\left(\sum_{\boldsymbol{n}\in  [\boldsymbol{0}_{k},\boldsymbol{t}]}\left(\sum_{\substack{\boldsymbol{l}\in \mathbb{Z}_{\geq 0}^{k}\\  \boldsymbol{n}\leq \boldsymbol{l},\ \boldsymbol{l}\notin [\boldsymbol{n},\boldsymbol{t}]}}\left(\prod_{i=1}^{k}\begin{pmatrix}l_{i}\\ n_{i}\end{pmatrix}a_{i}^{l_{i}-n_{i}}\right)m_{\boldsymbol{l}}\right)\boldsymbol{b}^{\boldsymbol{n}}\right)\\
&\geq\inf_{\substack{\boldsymbol{l},\boldsymbol{n}\in \mathbb{Z}_{\geq 0}^{k}\\ \boldsymbol{l}\notin [\boldsymbol{0}_{k},\boldsymbol{t}],\ \boldsymbol{n}\leq \boldsymbol{l}}}\left\{\sum_{i=1}^{k}\left((l_{i}-n_{i})\ord_{p}(a_{i})+n_{i}\ord_{p}(b_{i})\right)+v_{M}(m_{\boldsymbol{l}})\right\}\\
&\geq \inf_{\substack{\boldsymbol{l}\in \mathbb{Z}_{\geq 0}^{k}\\ \boldsymbol{l}\notin [\boldsymbol{0}_{k},\boldsymbol{t}]}}\left\{\sum_{i=1}^{k}l_{i}\min\{\ord_{p}(a_{i}),\ord_{p}(b_{i})\}+v_{M}(m_{\boldsymbol{l}})\right\}\\
&\geq \inf_{\substack{\boldsymbol{l}\in \mathbb{Z}_{\geq 0}^{k}\\ \boldsymbol{l}\notin [\boldsymbol{0}_{k},\boldsymbol{t}]}}\left\{\sum_{i=1}^{k}l_{i}(\min\{\ord_{p}(a_{i}),\ord_{p}(b_{i})\}-r_{i})\right\}+v_{\boldsymbol{r}}(f),
\end{align*}
we see that
\begin{equation}\label{Br slide propproof5}
\lim_{\boldsymbol{t}\rightarrow +\infty}\sum_{\boldsymbol{n}\in  [\boldsymbol{0}_{k},\boldsymbol{t}]}\left(\sum_{\substack{\boldsymbol{l}\in \mathbb{Z}_{\geq 0}^{k}\\  \boldsymbol{n}\leq \boldsymbol{l},\ \boldsymbol{l}\notin [\boldsymbol{n},\boldsymbol{t}]}}\left(\prod_{i=1}^{k}\begin{pmatrix}l_{i}\\ n_{i}\end{pmatrix}a_{i}^{l_{i}-n_{i}}\right)m_{\boldsymbol{l}}\right)\boldsymbol{b}^{\boldsymbol{n}}=0.
\end{equation}
By \eqref{Br slide propproof3}, \eqref{Br slide propproof4} and \eqref{Br slide propproof5}, we see that 
$$f_{+\boldsymbol{a}}(\boldsymbol{b})-f(\boldsymbol{b}+\boldsymbol{a})=f_{+\boldsymbol{a}}(\boldsymbol{b})-\lim_{\boldsymbol{t}\rightarrow +\infty}\sum_{\boldsymbol{n}\in  [\boldsymbol{0}_{k},\boldsymbol{t}]}m_{\boldsymbol{n}}(\boldsymbol{b}+\boldsymbol{a})^{\boldsymbol{n}}=0.$$
Thus, we have $f_{+\boldsymbol{a}}(\boldsymbol{b})=f(\boldsymbol{b}+\boldsymbol{a})$ for every $\boldsymbol{b}\in \overline{\K}^{k}$ such that $\ord_{p}(b_{i})>r_{i}$ with $1\leq i\leq k$. The uniqueness of $f_{+\boldsymbol{a}}$ follows from Lemma \ref{analytic funcuniqueof Brby Zp}. We complete the proof of \eqref{Br slide prop2}.

Finally, we prove that $v_{\boldsymbol{r}}(f)=v_{\boldsymbol{r}}(f_{+\boldsymbol{a}})$. By \eqref{Br slide propproof2}, we have $v_{\boldsymbol{r}}(f_{+\boldsymbol{a}})\geq v_{\boldsymbol{r}}(f)$. Further, by the uniqueness of \eqref{Br slide prop2}, we see that $(f_{+\boldsymbol{a}})_{+(-\boldsymbol{a})}=f$. Thus, by \eqref{Br slide propproof2}, we have $v_{\boldsymbol{r}}(f)=v_{\boldsymbol{r}}((f_{+\boldsymbol{a}})_{+(-\boldsymbol{a})})\geq v_{\boldsymbol{r}}(f_{+\boldsymbol{a}})$. Thus, we have $v_{\boldsymbol{r}}(f)=v_{\boldsymbol{r}}(f_{+\boldsymbol{a}})$.
\end{proof}

Let us fix $\boldsymbol{d},\boldsymbol{e}\in \mathbb{Z}^{k}$ satisfying $\boldsymbol{e}\geq \boldsymbol{d}$. For each $1\leq i\leq k$, 
we take a $p$-adic Lie group $\Gamma_{i}$ which is isomorphic to $1+2p\mathbb{Z}_p\subset \mathbb{Q}_{p}^{\times}$ via a continuous character $\chi_{i} : \Gamma_{i} \longrightarrow \mathbb{Q}_{p}^{\times}$. Fix a topological generator $\gamma_{i}\in \Gamma_{i}$ and put $u_{i}=\chi_{i}(\gamma_{i})$ for each $1\leq i\leq k$. We define $\Gamma=\Gamma_{1}\times \cdots \times\Gamma_{k}$. Let $\mathcal{O}_{\K}[[\Gamma]]$ be the $k$-variable Iwasawa algebra. We denote by $[\ ]: \Gamma\rightarrow \mathbb{Z}_{p}[[\Gamma]]^{\times}$ the tautological inclusion map. Let $M^{0}=\{m\in M\vert v_{M}(m)\geq 0\}$. We put
\begin{equation}\label{definition of M0[[Gamma]]}
\begin{split}
M^{0}[[\Gamma]]=\mathcal{O}_{\K}[[\Gamma]]\widehat{\otimes}_{\mathcal{O}_{\K}}M^{0} 
=\varprojlim_{U}\left(\mathcal{O}_{\K}[\Gamma\slash U]\otimes_{\mathcal{O}_{\K}}M^{0}\right),
\end{split}
\end{equation}
where $U$ runs over all open subgroups of $\Gamma$. By definition, $M^{0}[[\Gamma]]$ is an $\mathcal{O}_{\K}[[\Gamma]]$-module. 
For each $\boldsymbol{m}\in \mathbb{Z}_{\geq 0}^{k}$, we denote by $(\Omega_{\boldsymbol{m}}^{[\boldsymbol{d},\boldsymbol{e}]}(\gamma_{1},\ldots, \gamma_{k}))$ 
the ideal of $\mathcal{O}_{\K}[[\Gamma]]$ generated by $\Omega_{m_{1}}^{[d_{1},e_{1}]}(\gamma_{1}),\ldots, \Omega_{m_{k}}^{[d_{k},e_{k}]}(\gamma_{k})$, 
where $\Omega_{m_{i}}^{[d_{i},e_{i}]}(\gamma_{i})=\prod_{j=d_{i}}^{e_{i}}([\gamma_{i}]^{p^{m_{i}}}-u_{i}^{jp^{m_{i}}})\in \mathcal{O}_{\K}[[\Gamma_{i}]]$ for every 
$i$ satisfying $1\leq i\leq k$. We remark that the ideal $(\Omega_{\boldsymbol{m}}^{[\boldsymbol{d},\boldsymbol{e}]}(\gamma_{1},\ldots, \gamma_{k}))$ is independent of the choice of topological generators $\gamma_{i}\in \Gamma_{i}$ for each $1\leq i\leq k$. If there is no risk of confusion, we write $(\Omega_{\boldsymbol{m}}^{[\boldsymbol{d},\boldsymbol{e}]})$ for $(\Omega_{\boldsymbol{m}}^{[\boldsymbol{d},\boldsymbol{e}]}(\gamma_{1},\ldots, \gamma_{k}))$. We regard $\varprojlim_{\boldsymbol{m}\in\mathbb{Z}_{\geq 0}^{k}}\left(\frac{M^{0}[[\Gamma]]}{(\Omega_{\boldsymbol{m}}^{[\boldsymbol{d},\boldsymbol{e}]}(\gamma_{1},\ldots, \gamma_{k}))M^{0}[[\Gamma]]}\otimes_{\mathcal{O}_{\K}}\K\right)$ and $\left(\prod_{\boldsymbol{m}\in \mathbb{Z}_{\geq 0}^{k}}\frac{M^{0}[[\Gamma]]}{(\Omega_{\boldsymbol{m}}^{[\boldsymbol{d},\boldsymbol{e}]}(\gamma_{1},\ldots, \gamma_{k}))M^{0}[[\Gamma]]}\right)\otimes_{\mathcal{O}_{\K}}\K$ as submodules of $\prod_{\boldsymbol{m}\in \mathbb{Z}_{\geq 0}^{k}}\bigg(\frac{M^{0}[[\Gamma]]}{(\Omega_{\boldsymbol{m}}^{[\boldsymbol{d},\boldsymbol{e}]}(\gamma_{1},\ldots, \gamma_{k}))M^{0}[[\Gamma]]}\linebreak\otimes_{\mathcal{O}_{\K}}\K\bigg)$ and we define an $\mathcal{O}_{\K}[[\Gamma]]\otimes_{\mathcal{O}_{\K}}\K$-module $I_{\boldsymbol{h}}^{[\boldsymbol{d},\boldsymbol{e}]}(M)$ to be 
\begin{multline}\label{generalization of the project lim for deformation ring}
I_{\boldsymbol{h}}^{[\boldsymbol{d},\boldsymbol{e}]}(M)=
\left\{(s_{\boldsymbol{m}}^{[\boldsymbol{d},\boldsymbol{e}]})_{\boldsymbol{m}}\in \varprojlim_{\boldsymbol{m}\in\mathbb{Z}_{\geq 0}^{k}} \left(\frac{M^{0}[[\Gamma]]}{(\Omega_{\boldsymbol{m}}^{[\boldsymbol{d},\boldsymbol{e}]}(\gamma_{1},\ldots, \gamma_{k}))M^{0}[[\Gamma]]}\otimes_{\mathcal{O}_{\K}}\K\right) \right. 
\\
\Bigg\vert (p^{\langle \boldsymbol{h},\boldsymbol{m}\rangle_{k}}s_{\boldsymbol{m}}^{[\boldsymbol{d},\boldsymbol{e}]})_{\boldsymbol{m}}\in 
\left. 
\left(\prod_{\boldsymbol{m}\in \mathbb{Z}_{\geq 0}^{k}}\frac{M^{0}[[\Gamma]]}{(\Omega_{\boldsymbol{m}}^{[\boldsymbol{d},\boldsymbol{e}]}(\gamma_{1},\ldots, \gamma_{k}))M^{0}[[\Gamma]]}\right)\otimes_{\mathcal{O}_{\K}}\K\right\}.
\end{multline}
For each $\boldsymbol{m}\in \mathbb{Z}_{\geq 0}^{k}$, we denote by $(\Omega_{\boldsymbol{m}}^{[\boldsymbol{d},\boldsymbol{e}]}(X_{1},\ldots, X_{k}))$ 
the ideal of $\mathcal{O}_{\K}[[X_{1},\ldots, X_{k}]]$ generated by $\Omega_{m_{1}}^{[d_{1},e_{1}]}(X_{1}),\ldots, \Omega_{m_{k}}^{[d_{k},e_{k}]}(X_{k})$, where $\Omega_{m_{i}}^{[d_{i},e_{i}]}(X_{i})=\prod_{j=d_{i}}^{e_{i}}((1+X_{i})^{p^{m_{i}}}-u_{i}^{jp^{m_{i}}})\in \mathcal{O}_{\K}[[X_{i}]]$ for every 
$i$ satisfying $1\leq i\leq k$. We also define an $\mathcal{O}_{\K}[[X_{1},\ldots, X_{k}]]\otimes_{\mathcal{O}_{\K}}\K$-module $J_{\boldsymbol{h}}^{[\boldsymbol{d},\boldsymbol{e}]}(M)$ to be
\begin{align}\label{generalization of the project lim for deformation ring Jboldsymbol}
\begin{split}
J_{\boldsymbol{h}}^{[\boldsymbol{d},\boldsymbol{e}]}(M)=&
\left\{(s_{\boldsymbol{m}}^{[\boldsymbol{d},\boldsymbol{e}]})_{\boldsymbol{m}}\in \varprojlim_{\boldsymbol{m}\in\mathbb{Z}_{\geq 0}^{k}}
\left(\frac{M^{0}[[X_{1},\ldots, X_{k}]]}{(\Omega_{\boldsymbol{m}}^{[\boldsymbol{d},\boldsymbol{e}]}(X_{1},\ldots, X_{k}))M^{0}[[X_{1},\ldots, X_{k}]]}\otimes_{\mathcal{O}_{\K}}\K\right) \right.\\
&\Bigg\vert (p^{\langle \boldsymbol{h},\boldsymbol{m}\rangle_{k}}s_{\boldsymbol{m}}^{[\boldsymbol{d},\boldsymbol{e}]})_{\boldsymbol{m}}\in 
\left. \left(\prod_{\boldsymbol{m}\in \mathbb{Z}_{\geq 0}^{k}}\frac{M^{0}[[X_{1},\ldots, X_{k}]]}{(\Omega_{\boldsymbol{m}}^{[\boldsymbol{d},\boldsymbol{e}]}(X_{1},\ldots, X_{k}))M^{0}[[X_{1},\ldots, X_{k}]]}\right)\otimes_{\mathcal{O}_{\K}}\K\right\}.
\end{split}
\end{align}
We regard 
$$\varprojlim_{\boldsymbol{m}\in\mathbb{Z}_{\geq 0}^{k}}\left(\frac{M^{0}[[X_{1},\ldots,X_{k}]]}{(\Omega_{\boldsymbol{m}}^{[\boldsymbol{d},\boldsymbol{e}]}(X_{1},\ldots, X_{k}))M^{0}[[X_{1},\ldots, X_{k}]]}\otimes_{\mathcal{O}_{\K}}\K\right)$$ 
and $$\left(\prod_{\boldsymbol{m}\in \mathbb{Z}_{\geq 0}^{k}}\frac{M^{0}[[X_{1},\ldots, X_{k}]]}{(\Omega_{\boldsymbol{m}}^{[\boldsymbol{d},\boldsymbol{e}]}(X_{1},\ldots, X_{k}))M^{0}[[X_{1},\ldots, X_{k}]]}\right)\otimes_{\mathcal{O}_{\K}}\K$$ 
as submodules of $\prod_{\boldsymbol{m}\in \mathbb{Z}_{\geq 0}^{k}}\left(\frac{M^{0}[[X_{1},\ldots, X_{k}]]}{(\Omega_{\boldsymbol{m}}^{[\boldsymbol{d},\boldsymbol{e}]}(X_{1},\ldots, X_{k}))M^{0}[[X_{1},\ldots,X_{k}]]}\otimes_{\mathcal{O}_{\K}}\K\right)$. 
Let us consider the non-canonical continuous $\mathcal{O}_{\K}$-algebra isomorphism
\begin{equation}\label{iwasaawa noncaonnical multiisom}
\alpha^{(k)}:\mathcal{O}_{\K}[[\Gamma]]\stackrel{\sim}{\rightarrow}\mathcal{O}_{\K}[[X_{1},\ldots, X_{k}]]
\end{equation}
characterized by 
$
\alpha^{(k)} ([(\gamma_{1}^{n_{1}},\ldots, \gamma_{k}^{n_{k}})]) = \prod_{i=1}^{k}(1+X_{i})^{n_{i}}$ for each $\boldsymbol{n}\in \mathbb{Z}^{k}$. We note that $M^{0}[[X_{1},\ldots, X_{k}]]$ is isomorphic to 
$$\mathcal{O}_{\K}[[X_{1},\ldots, X_{k}]]\widehat{\otimes}_{\mathcal{O}_{\K}}M^{0}= \varprojlim_{\boldsymbol{m}\in \mathbb{Z}_{\geq 0}^{k}} (\mathcal{O}_{\K}[[X_{1},\ldots, X_{k}]]\slash (\Omega_{\boldsymbol{m}}^{[\boldsymbol{0}_{k},\boldsymbol{0}_{k}]}(X_{1},\ldots,X_{k}))\otimes_{\mathcal{O}_{\K}}M^{0}),
$$ 
where $\boldsymbol{0}_{k}=(0,\ldots, 0)\in \mathbb{Z}_{\geq 0}^{k}$. We can define a non-canonical $\mathcal{O}_{\K}$-module isomorphism 
\begin{equation}\label{non-canonical continuous isomorphihsm of iwasawa module of banach}
\alpha_{M}^{(k)}:M^{0}[[\Gamma]]\stackrel{\sim}{\rightarrow}M^{0}[[X_{1},\ldots,X_{k}]]
\end{equation}
to be $c\widehat{\otimes}_{\mathcal{O}_{\K}}m\mapsto \alpha^{(k)}(c)\widehat{\otimes}_{\mathcal{O}_{\K}}m$ for each $m\in M^{0}$ and $c\in \mathcal{O}_{\K}[[\Gamma]]$. Via $\alpha_{M}^{(k)}$, we have a non-canonical $\K$-linear isomorphism 
\begin{equation}\label{noncanonical between Ioldysmbolhde and J}
I_{\boldsymbol{h}}^{[\boldsymbol{d},\boldsymbol{e}]}(M)\simeq J_{\boldsymbol{h}}^{[\boldsymbol{d},\boldsymbol{e}]}(M).
\end{equation}
\par 
Next, we introduce $[\boldsymbol{d},\boldsymbol{e}]$-admissible distributions of growth $\boldsymbol{h}$. We denote by $\mathcal{O}_{\K}[X_{1},\ldots,\linebreak X_{k}]_{\leq \boldsymbol{n}}$ with $\boldsymbol{n}\in \mathbb{Z}_{\geq 0}^{k}$ the $\mathcal{O}_{\K}$-module of $k$-variable polynomials of $j$-th degree at most $n_{j}$ for each $1\leq j\leq k$. We say that a function $f:\Gamma\rightarrow \mathcal{O}_{\K}$ is a $k$-variable locally polynomial function on $\Gamma$ of degree at most $\boldsymbol{n}\in \mathbb{Z}_{\geq 0}^{k}$ if, for each $\boldsymbol{a}\in \Gamma$, there exists a neighborhood $U$ of $\boldsymbol{a}$ in $\Gamma$
and there exists a polynomial $p_{\boldsymbol{a}}\in \mathcal{O}_{\K}[X_{1},\ldots, X_{k}]_{\leq \boldsymbol{n}}$ such that we have $f(x_{1},\ldots, x_{k})=p_{\boldsymbol{a}}(\chi_{1}(x_{1}),\ldots, \chi_{k}(x_{k}))$ on $U$. We denote by $C^{[\boldsymbol{d},\boldsymbol{e}]}(\Gamma,\mathcal{O}_{\K})$ the $\mathcal{O}_{\K}$-module which consists of functions $f:\Gamma\rightarrow \mathcal{O}_{\K}$ such that $\left(\prod_{i=1}^{k}\chi_{i}(x_{i})^{-d_{i}}\right)f(x_{1},\ldots, x_{k})$ is a $k$-variable locally polynomial function of degree at most $\boldsymbol{e}-\boldsymbol{d}$. 
For any $\mu\in \mathrm{Hom}_{\mathcal{O}_{\mathcal{K}}}( C^{[\boldsymbol{d},\boldsymbol{e}]} (\Gamma , \mathcal{O}_\mathcal{K}) , M) $, 
we set 
\begin{multline}\label{admissible condition}
v_{\boldsymbol{h}}^{[\boldsymbol{d},\boldsymbol{e}]}(\mu)=
\inf_{\substack{\boldsymbol{a}\in \Gamma,\boldsymbol{m}\in 
\mathbb{Z}_{\geq 0}^{k}\\ \boldsymbol{i}\in [\boldsymbol{d},\boldsymbol{e}]}}
\Biggl\{v_{M}\left(\displaystyle{\int_{\boldsymbol{a}\Gamma^{p^{\boldsymbol{m}}}}}\prod_{j=1}^{k}\left((\chi_{j}(x_{j})-\chi_{j}(a_{j}))^{i_{j}-d_{j}}\chi_{j}(x_{j})^{d_{j}}\right)d\mu\right)
\\
+\langle \boldsymbol{h}-(\boldsymbol{i}-\boldsymbol{d}),\boldsymbol{m}\rangle_{k}\Bigg\},
\end{multline}
where $\boldsymbol{a}\Gamma^{p^{\boldsymbol{m}}}=\prod_{j=1}^{k}a_{j}\Gamma_{j}^{p^{m_{j}}}$. We define a $\K$-subspace $\mathcal{D}^{[\boldsymbol{d},\boldsymbol{e}]}_{\boldsymbol{h}}(\Gamma , M)$ 
of $\mathrm{Hom}_{\mathcal{O}_{\mathcal{K}}}( C^{[\boldsymbol{d},\boldsymbol{e}]} (\Gamma , \mathcal{O}_\mathcal{K}),\linebreak M)$ 
by 
\begin{equation}\label{defof admissible distri space}
\mathcal{D}^{[\boldsymbol{d},\boldsymbol{e}]}_{\boldsymbol{h}}(\Gamma , M) = \{ \mu \in  \mathrm{Hom}_{\mathcal{O}_{\mathcal{K}}}( C^{[\boldsymbol{d},\boldsymbol{e}]} (\Gamma , \mathcal{O}_\mathcal{K}) , M) \ \vert \  v_{\boldsymbol{h}}^{[\boldsymbol{d},\boldsymbol{e}]}(\mu) > -\infty   \}. 
\end{equation}
An element $\mu$ of $\mathcal{D}^{[\boldsymbol{d},\boldsymbol{e}]}_{\boldsymbol{h}}(\Gamma , M)$ is called 
a $[\boldsymbol{d},\boldsymbol{e}]$-admissible distribution of growth $\boldsymbol{h}$.

\begin{pro}\label{multi admissible banach}
The pair $(\mathcal{D}_{\boldsymbol{h}}^{[\boldsymbol{d},\boldsymbol{e}]}(\Gamma,M),v_{\boldsymbol{h}}^{[\boldsymbol{d},\boldsymbol{e}]})$ is a $\K$-Banach space.
\end{pro}
\begin{proof}
First, we will show that
\begin{equation}\label{multi admissible banach1}
v_{M}\left(\int_{\boldsymbol{a}\boldsymbol{\Gamma}^{p^{\boldsymbol{m}}}}\prod_{j=1}^{k}\chi_{j}(x_{j})^{i_{j}}d\mu\right)\geq -\langle\boldsymbol{h},\boldsymbol{m}\rangle_{k}+v_{\boldsymbol{h}}^{[\boldsymbol{d},\boldsymbol{e}]}(\mu)
\end{equation}
for every $\mu\in \mathcal{D}_{\boldsymbol{h}}^{[\boldsymbol{d},\boldsymbol{e}]}(\Gamma,\K)$, $\boldsymbol{a}\in \Gamma$, $\boldsymbol{m}\in \mathbb{Z}_{\geq 0}^{k}$ and 
$\boldsymbol{i} = 
(i_j)\in [\boldsymbol{d},\boldsymbol{e}]$. {To prove \eqref{multi admissible banach1}, it suffices to prove the followings:
\begin{enumerate}
\item The inequality \eqref{multi admissible banach1} holds when $\boldsymbol{i}=\boldsymbol{d}$.
\item Let $\boldsymbol{i}\in [\boldsymbol{d},\boldsymbol{e}]$ be an element satisfying $\boldsymbol{i}\neq \boldsymbol{d}$. 
Assume that the inequality \eqref{multi admissible banach1} holds for each $\boldsymbol{j}\in [\boldsymbol{d},\boldsymbol{i}]$ satisfying $\boldsymbol{j}\neq \boldsymbol{i}$. 
Then the inequality \eqref{multi admissible banach1} holds for $\boldsymbol{i}$.
\end{enumerate}}
{By the definition \eqref{admissible condition} of 
$v_{\boldsymbol{h}}^{[\boldsymbol{d},\boldsymbol{e}]}(\mu)$, we have 
\begin{align*}
v_{M}\left(\displaystyle{\int_{\boldsymbol{a}\Gamma^{p^{\boldsymbol{m}}}}}\prod_{j=1}^{k}\chi_{j}(x_{j})^{d_{j}}d\mu\right)
+\langle \boldsymbol{h},\boldsymbol{m}\rangle_{k}
\geq 
v_{\boldsymbol{h}}^{[\boldsymbol{d},\boldsymbol{e}]}(\mu)
\end{align*}
for every $\mu\in \mathcal{D}_{\boldsymbol{h}}^{[\boldsymbol{d},\boldsymbol{e}]}(\Gamma,\K)$, $\boldsymbol{a}\in \Gamma$, $\boldsymbol{m}\in \mathbb{Z}_{\geq 0}^{k}$.  
By moving $\langle \boldsymbol{h},\boldsymbol{m}\rangle_{k}$ to the right hand-side, we have the desired inequality \eqref{multi admissible banach1} when $\boldsymbol{i}=\boldsymbol{d}$.} 
{
Let us assume that $\boldsymbol{i}>\boldsymbol{d}$ and we have
\begin{equation}\label{multi admissible banach1 tver}
v_{M}\left(\int_{\boldsymbol{a}\boldsymbol{\Gamma}^{p^{\boldsymbol{m}}}}\prod_{j=1}^{k}\chi_{j}(x_{j})^{t_{j}}d\mu\right)\geq -\langle\boldsymbol{h},\boldsymbol{m}\rangle_{k}+v_{\boldsymbol{h}}^{[\boldsymbol{d},\boldsymbol{e}]}(\mu)
\end{equation}
for every $\mu\in \mathcal{D}_{\boldsymbol{h}}^{[\boldsymbol{d},\boldsymbol{e}]}(\Gamma,\K)$, $\boldsymbol{a}\in \Gamma$, $\boldsymbol{m}\in \mathbb{Z}_{\geq 0}^{k}$ and $\boldsymbol{t}\in [\boldsymbol{d},\boldsymbol{i}]$ such that $\boldsymbol{t}\neq \boldsymbol{i}$. We have 
\small 
\begin{align*}
\prod_{j=1}^{k}\chi_{j}(x_{j})^{i_{j}}
=\prod_{j=1}^{k}(\chi_{j}(x_{j})-\chi_{j}(a_{j}))^{i_{j}-d_{j}}\chi_{j}(x_{j})^{d_{j}}
-\sum_{\substack{\boldsymbol{t}\in [\boldsymbol{d},\boldsymbol{i}]
\\ 
\boldsymbol{t}\neq \boldsymbol{i}}}\left(\prod_{j=1}^{k}\begin{pmatrix}i_{j}-d_{j}\\ t_{j}-d_{j}\end{pmatrix}(-\chi_{j}(a_{j}))^{i_{j}-t_{j}}\chi_{j}(x_{j})^{t_{j}}\right).
\end{align*}
Hence we have 
\begin{multline}
\begin{split}
\label{multi admissible banach eq 1.5}
v_{M}\left(\int_{\boldsymbol{a}\boldsymbol{\Gamma}^{p^{\boldsymbol{m}}}}\prod_{j=1}^{k}\chi_{j}(x_{j})^{i_{j}}d\mu\right)\geq \min\Bigg\{v_{M}\left(\int_{\boldsymbol{a}\boldsymbol{\Gamma}^{p^{\boldsymbol{m}}}}\prod_{j=1}^{k}(\chi_{j}(x_{j})-\chi_{j}(a_{j}))^{i_{j}-d_{j}}\chi_{j}(x_{j})^{d_{j}}d\mu\right),\\
v_{M}\bigg(\sum_{\substack{\boldsymbol{t}\in [\boldsymbol{d},\boldsymbol{i}]
\\ 
\boldsymbol{t}\neq \boldsymbol{i}}}\prod_{j=1}^{k}\begin{pmatrix}i_{j}-d_{j}\\ t_{j}-d_{j}\end{pmatrix}(-\chi_{j}(a_{j}))^{i_{j}-t_{j}}\int_{\boldsymbol{a}\boldsymbol{\Gamma}^{p^{\boldsymbol{m}}}}\chi_{j}(x_{j})^{t_{j}}d\mu\bigg)\Bigg\}.
\end{split}
\end{multline}
\normalsize 
By the definition \eqref{admissible condition} of 
$v_{\boldsymbol{h}}^{[\boldsymbol{d},\boldsymbol{e}]}(\mu)$, we have 
\begin{align}\label{multi admissible banach eq 1.501}
v_{M}\left(\int_{\boldsymbol{a}\boldsymbol{\Gamma}^{p^{\boldsymbol{m}}}}\prod_{j=1}^{k}(\chi_{j}(x_{j})-\chi_{j}(a_{j}))^{i_{j}-d_{j}}\chi_{j}(x_{j})^{d_{j}}d\mu\right)
+\langle \boldsymbol{h}-(\boldsymbol{i}-\boldsymbol{d}),\boldsymbol{m}\rangle_{k}.
\geq 
v_{\boldsymbol{h}}^{[\boldsymbol{d},\boldsymbol{e}]}(\mu)
\end{align}
for every $\boldsymbol{a}\in \Gamma$, $\boldsymbol{m}\in \mathbb{Z}_{\geq 0}^{k}$ and $\boldsymbol{i}\in [\boldsymbol{d},\boldsymbol{e}]$. By 
moving $\langle \boldsymbol{h}-(\boldsymbol{i}-\boldsymbol{d}),\boldsymbol{m}\rangle_{k}$ in the inequality \eqref{multi admissible banach eq 1.501} to the right-hand side, 
we obtain 
\begin{align*}
v_{M}\left(\int_{\boldsymbol{a}\boldsymbol{\Gamma}^{p^{\boldsymbol{m}}}}\prod_{j=1}^{k}(\chi_{j}(x_{j})-\chi_{j}(a_{j}))^{i_{j}-d_{j}}\chi_{j}(x_{j})^{d_{j}}d\mu\right)&\geq-\langle \boldsymbol{h}-(\boldsymbol{i}-\boldsymbol{d}),\boldsymbol{m}\rangle_{k}+v_{\boldsymbol{h}}^{[\boldsymbol{d},\boldsymbol{e}]}(\mu)
\end{align*}
for every $\mu\in \mathcal{D}_{\boldsymbol{h}}^{[\boldsymbol{d},\boldsymbol{e}]}(\Gamma,\K)$, $\boldsymbol{a}\in \Gamma$, $\boldsymbol{m}\in \mathbb{Z}_{\geq 0}^{k}$ and 
$\boldsymbol{i} \in [\boldsymbol{d},\boldsymbol{e}]$. 
Since we have 
$$
\langle \boldsymbol{h}-(\boldsymbol{i}-\boldsymbol{d}),\boldsymbol{m}\rangle_{k}=\langle \boldsymbol{h},\boldsymbol{m}\rangle_{k}-\langle (\boldsymbol{i}-\boldsymbol{d}),
\boldsymbol{m}\rangle_{k}\leq \langle \boldsymbol{h},\boldsymbol{m}\rangle_{k}
$$ 
for every $\boldsymbol{i}\in [\boldsymbol{d},\boldsymbol{e}]$ and $\boldsymbol{m}\in \mathbb{Z}_{\geq 0}^{k}$, 
we have
\begin{equation}\label{multi admissible banach eq 1.51}
v_{M}\left(\int_{\boldsymbol{a}\boldsymbol{\Gamma}^{p^{\boldsymbol{m}}}}\prod_{j=1}^{k}(\chi_{j}(x_{j})-\chi_{j}(a_{j}))^{i_{j}-d_{j}}\chi_{j}(x_{j})^{d_{j}}d\mu\right)\geq -\langle \boldsymbol{h},\boldsymbol{m}\rangle_{k}+v_{\boldsymbol{h}}^{[\boldsymbol{d},\boldsymbol{e}]}(\mu)
\end{equation}
for every $\boldsymbol{a}\in \Gamma$, $\boldsymbol{m}\in \mathbb{Z}_{\geq 0}^{k}$ and $\boldsymbol{i} \in [\boldsymbol{d},\boldsymbol{e}]$. 
On the other hand, by the properties of valuations, we have
\begin{align*}
& v_{M}\left(\sum_{\substack{\boldsymbol{t}\in [\boldsymbol{d},\boldsymbol{i}]
\\ 
\boldsymbol{t}\neq \boldsymbol{i}}}\prod_{j=1}^{k}\begin{pmatrix}i_{j}-d_{j}\\ t_{j}-d_{j}\end{pmatrix}(-\chi_{j}(a_{j}))^{i_{j}-t_{j}}\int_{\boldsymbol{a}\boldsymbol{\Gamma}^{p^{\boldsymbol{m}}}}\chi_{j}(x_{j})^{t_{j}}d\mu\right)
\\
& \geq \min_{\substack{\boldsymbol{t}\in [\boldsymbol{d},\boldsymbol{i}]
\\ 
 \boldsymbol{t}\neq \boldsymbol{i}}}\left\{v_{M}\left(\prod_{j=1}^{k}\begin{pmatrix}i_{j}-d_{j}\\ t_{j}-d_{j}\end{pmatrix}(-\chi_{j}(a_{j}))^{i_{j}-t_{j}}\int_{\boldsymbol{a}\boldsymbol{\Gamma}^{p^{\boldsymbol{m}}}}\prod_{j=1}^{k}\chi_{j}(x_{j})^{t_{j}}d\mu\right)\right\}\\
& =\min_{\substack{\boldsymbol{t}\in [\boldsymbol{d},\boldsymbol{i}]
\\ 
\boldsymbol{t}\neq \boldsymbol{i}}}\left\{\ord_{p}\left(\prod_{j=1}^{k}\begin{pmatrix}i_{j}-d_{j}\\ t_{j}-d_{j}\end{pmatrix}(-\chi_{j}(a_{j}))^{i_{j}-t_{j}}\right)+v_{M}\left(\int_{\boldsymbol{a}\boldsymbol{\Gamma}^{p^{\boldsymbol{m}}}}\prod_{j=1}^{k}\chi_{j}(x_{j})^{t_{j}}d\mu\right)\right\} .
\end{align*}
Since $\ord_{p}\left(\prod_{j=1}^{k}\begin{pmatrix}i_{j}-d_{j}\\ t_{j}-d_{j}\end{pmatrix}(-\chi_{j}(a_{j}))^{i_{j}-t_{j}}\right)\geq 0$, we have 
\small  
\begin{multline*}
v_{M}\left(\sum_{\substack{\boldsymbol{t}\in [\boldsymbol{d},\boldsymbol{i}]
\\ 
\boldsymbol{t}\neq \boldsymbol{i}}}\prod_{j=1}^{k}\begin{pmatrix}i_{j}-d_{j}\\ t_{j}-d_{j}\end{pmatrix}(-\chi_{j}(a_{j}))^{i_{j}-t_{j}}\int_{\boldsymbol{a}\boldsymbol{\Gamma}^{p^{\boldsymbol{m}}}}\chi_{j}(x_{j})^{t_{j}}d\mu\right)\\
\geq \min_{\substack{\boldsymbol{t}\in [\boldsymbol{d},\boldsymbol{i}]
\\ 
\boldsymbol{t}\neq \boldsymbol{i}}}\left\{v_{M}\left(\int_{\boldsymbol{a}\boldsymbol{\Gamma}^{p^{\boldsymbol{m}}}}\prod_{j=1}^{k}\chi_{j}(x_{j})^{t_{j}}d\mu\right)\right\}.
\end{multline*}
\normalsize 
Since \eqref{multi admissible banach1 tver} holds for every $\boldsymbol{a}\in \Gamma$, $\boldsymbol{m}\in \mathbb{Z}_{\geq 0}^{k}$ and $\boldsymbol{t}\in [\boldsymbol{d},\boldsymbol{i}]$ with $\boldsymbol{t}\neq \boldsymbol{i}$, the above inequality implies 
\begin{equation}\label{multi admissible banach eq 1.52}
v_{M}\left(\sum_{\substack{\boldsymbol{t}\in [\boldsymbol{d},\boldsymbol{i}]
\\ 
\boldsymbol{t}\neq \boldsymbol{i}}}\prod_{j=1}^{k}\begin{pmatrix}i_{j}-d_{j}\\ t_{j}-d_{j}\end{pmatrix}(-\chi_{j}(a_{j}))^{i_{j}-t_{j}}\int_{\boldsymbol{a}\boldsymbol{\Gamma}^{p^{\boldsymbol{m}}}}\chi_{j}(x_{j})^{t_{j}}d\mu\right)\geq -\langle\boldsymbol{h},\boldsymbol{m}\rangle_{k}+v_{\boldsymbol{h}}^{[\boldsymbol{d},\boldsymbol{e}]}(\mu).
\end{equation}  
By \eqref{multi admissible banach eq 1.5},\eqref{multi admissible banach eq 1.51} and \eqref{multi admissible banach eq 1.52}, we deduce the desired inequality \eqref{multi admissible banach1}.}

By \eqref{multi admissible banach1}, we see that $v_{\boldsymbol{h}}^{[\boldsymbol{d},\boldsymbol{e}]}(\mu)=+\infty$ if and only if $\mu=0$. It is easy to check that 
\begin{align*}
& v_{\boldsymbol{h}}^{[\boldsymbol{d},\boldsymbol{e}]}(\mu+\nu)\geq \min\{v_{\boldsymbol{h}}^{[\boldsymbol{d},\boldsymbol{e}]}(\mu),v_{\boldsymbol{h}}^{[\boldsymbol{d},\boldsymbol{e}]}(\nu)\}, \\ 
& v_{\boldsymbol{h}}^{[\boldsymbol{d},\boldsymbol{e}]}(a\mu)=\ord_{p}(a)+v_{\boldsymbol{h}}^{[\boldsymbol{d},\boldsymbol{e}]}(\mu).
\end{align*}
Hence, $v_{\boldsymbol{h}}^{[\boldsymbol{d},\boldsymbol{e}]}$ is a valuation on $\mathcal{D}_{\boldsymbol{h}}^{[\boldsymbol{d},\boldsymbol{e}]}(\Gamma,M)$. 
\par 
Next, we prove that $\mathcal{D}_{\boldsymbol{h}}^{[\boldsymbol{d},\boldsymbol{e}]}(\Gamma,M)$ is complete with respect to $v_{\boldsymbol{h}}^{[\boldsymbol{d},\boldsymbol{e}]}$. Let $(\mu_{n})_{n\geq 0}$ be a Cauchy sequence of $\mathcal{D}_{\boldsymbol{h}}^{[\boldsymbol{d},\boldsymbol{e}]}(\Gamma,M)$. {By \eqref{multi admissible banach1}, we have 
\begin{multline*}
v_{M}\left(\int_{\boldsymbol{a}\boldsymbol{\Gamma}^{p^{\boldsymbol{m}}}}\prod_{j=1}^{k}\chi_{j}(x_{j})^{i_{j}}d\mu_{n_{1}}-\int_{\boldsymbol{a}\boldsymbol{\Gamma}^{p^{\boldsymbol{m}}}}\prod_{j=1}^{k}\chi_{j}(x_{j})^{i_{j}}d\mu_{n_{2}}\right)\\
=v_{M}\left(\int_{\boldsymbol{a}\boldsymbol{\Gamma}^{p^{\boldsymbol{m}}}}\prod_{j=1}^{k}\chi_{j}(x_{j})^{i_{j}}(d\mu_{n_{1}}-d\mu_{n_{2}})\right)\geq  -\langle\boldsymbol{h},\boldsymbol{m}\rangle_{k}+v_{\boldsymbol{h}}^{[\boldsymbol{d},\boldsymbol{e}]}(\mu_{n_{1}}-\mu_{n_{2}})
\end{multline*}
for every $n_{1},n_{2}\in \mathbb{Z}_{\geq 0}$, $\boldsymbol{a}\in \Gamma$, $\boldsymbol{m}\in \mathbb{Z}_{\geq 0}^{k}$ and $\boldsymbol{i}\in [\boldsymbol{d},\boldsymbol{e}]$, and $\left\{\int_{\boldsymbol{a}\Gamma^{p^{\boldsymbol{m}}}}\prod_{j=1}^{k}\chi_{j}(x_{j})^{i_{j}}d\mu_{n}\right\}_{n\geq 0}$ is a Cauchy sequence of $M$. For each $f\in C^{[\boldsymbol{d},\boldsymbol{e}]}(\Gamma, \mathcal{O}_{\K})$, there exists an 
element $\boldsymbol{m}\in \mathbb{Z}_{\geq 0}^{k}$ such that 
we have 
$$
f(x_{1},\ldots, x_{k})=\sum_{\boldsymbol{a}\in \Gamma\slash \Gamma^{p^{\boldsymbol{m}}}}1_{\boldsymbol{a}\Gamma^{p^{\boldsymbol{m}}}}(x_{1},\ldots, x_{k})\sum_{\boldsymbol{i}\in [\boldsymbol{d},\boldsymbol{e}]}c^{(\boldsymbol{a})}_{\boldsymbol{i}}\prod_{j=1}^{k}\chi_{j}(x_{j})^{i_{j}}
$$ 
with $c^{(\boldsymbol{a})}_{\boldsymbol{i}}\in \mathcal{O}_{\K}$ where $1_{\boldsymbol{a}\Gamma^{p^{\boldsymbol{m}}}}(x_{1},\ldots,x_{k})$ is the characteristic function of $\boldsymbol{a}\Gamma^{p^{\boldsymbol{m}}}$. Then, we see that $\left\{\int_{\Gamma}fd\mu_{n}\right\}_{n\geq 0}$ is also a Cauchy sequence of $M$. Since $M$ is complete, we have a limit $\displaystyle{\lim_{n\rightarrow +\infty}}\int_{\Gamma}fd\mu_{n}$ in $M$. By setting 
\begin{equation}\label{multi admissible banach eq in the proof2}
\int_{\Gamma}fd\mu^{\prime}=\displaystyle{\lim_{n\rightarrow +\infty}}\int_{\Gamma}fd\mu_{n}
\end{equation}
for each $f\in C^{[\boldsymbol{d},\boldsymbol{e}]}(\Gamma,\mathcal{O}_{\K})$, we have $\mu^{\prime} \in \mathrm{Hom}_{\mathcal{O}_{\K}}(C^{[\boldsymbol{d},\boldsymbol{e}]}(\Gamma,\mathcal{O}_{\K}),M)$. By \eqref{multi admissible banach eq in the proof2}, we have
\begin{align*}
& v_{M}\left(\displaystyle{\int_{\boldsymbol{a}\Gamma^{p^{\boldsymbol{m}}}}}\prod_{j=1}^{k}\left((\chi_{j}(x_{j})-\chi_{j}(a_{j}))^{i_{j}-d_{j}}\chi_{j}(x_{j})^{d_{j}}\right)d\mu^{\prime}\right)+\langle \boldsymbol{h}-(\boldsymbol{i}-\boldsymbol{d}),\boldsymbol{m}\rangle_{k}
\\ 
& =\lim_{n\rightarrow +\infty}v_{M}\left(\displaystyle{\int_{\boldsymbol{a}\Gamma^{p^{\boldsymbol{m}}}}}\prod_{j=1}^{k}\left((\chi_{j}(x_{j})-\chi_{j}(a_{j}))^{i_{j}-d_{j}}\chi_{j}(x_{j})^{d_{j}}\right)d\mu_{n}\right)+\langle \boldsymbol{h}-(\boldsymbol{i}-\boldsymbol{d}),\boldsymbol{m}\rangle_{k}\\
& \geq \inf\{v_{\boldsymbol{h}}^{[\boldsymbol{d},\boldsymbol{e}]}(\mu_{n})\}_{n\in \mathbb{Z}_{\geq 0}}>-\infty
\end{align*}
for every $\boldsymbol{m}\in \mathbb{Z}_{\geq 0}^{k}$, $\boldsymbol{a}\in \Gamma$ and $\boldsymbol{i}\in[\boldsymbol{d},\boldsymbol{e}]$. Thus, $\mu^{\prime}\in \mathcal{D}_{\boldsymbol{h}}^{[\boldsymbol{d},\boldsymbol{e}]}(\Gamma,M)$. We prove that $\mu^{\prime}={{\lim}_{n\rightarrow +\infty}}\mu_{n}$. Let $A>0$. There exists an integer $N\in \mathbb{Z}_{\geq 0}$ such that we have $v_{\boldsymbol{h}}^{[\boldsymbol{d},\boldsymbol{e}]}(\mu_{n_{1}}-\mu_{n_{2}})\geq A$ for every $n_{1},n_{2}\geq N$. Therefore, if $n_{1},n_{2}\geq N$, we have 
\begin{align*}
v_{M}\left(\displaystyle{\int_{\boldsymbol{a}\Gamma^{p^{\boldsymbol{m}}}}}\prod_{j=1}^{k}\left((\chi_{j}(x_{j})-\chi_{j}(a_{j}))^{i_{j}-d_{j}}\chi_{j}(x_{j})^{d_{j}}\right)d(\mu_{n_{1}}-\mu_{n_{2}})\right)+\langle \boldsymbol{h}-(\boldsymbol{i}-\boldsymbol{d}),\boldsymbol{m}\rangle_{k}\\
\geq v_{\boldsymbol{h}}^{[\boldsymbol{d},\boldsymbol{e}]}(\mu_{n_{1}}-\mu_{n_{2}})\geq A
\end{align*}
for every $\boldsymbol{a}\in \Gamma$, $\boldsymbol{m}\in \mathbb{Z}_{\geq 0}^{k}$ and $\boldsymbol{i}\in [\boldsymbol{d},\boldsymbol{e}]$. By \eqref{multi admissible banach eq in the proof2}, if $n_{2}\geq N$, we see that 
\begin{align}\label{multi admissible banach eq in the proof3}
& v_{M}\left(\displaystyle{\int_{\boldsymbol{a}\Gamma^{p^{\boldsymbol{m}}}}}\prod_{j=1}^{k}\left((\chi_{j}(x_{j})-\chi_{j}(a_{j}))^{i_{j}-d_{j}}\chi_{j}(x_{j})^{d_{j}}\right)d(\mu^{\prime}-\mu_{n_{2}})\right)+\langle \boldsymbol{h}-(\boldsymbol{i}-\boldsymbol{d}),\boldsymbol{m}\rangle_{k}\\
& =\lim_{n_{1}\rightarrow +\infty}v_{M}\left(\displaystyle{\int_{\boldsymbol{a}\Gamma^{p^{\boldsymbol{m}}}}}\prod_{j=1}^{k}\left((\chi_{j}(x_{j})-\chi_{j}(a_{j}))^{i_{j}-d_{j}}\chi_{j}(x_{j})^{d_{j}}\right)d(\mu_{n_{1}}-\mu_{n_{2}})\right)\notag \\ 
& +\langle \boldsymbol{h}-(\boldsymbol{i}-\boldsymbol{d}),\boldsymbol{m}\rangle_{k} \geq A \notag
\end{align}
for every $\boldsymbol{a}\in \Gamma$, $\boldsymbol{m}\in \mathbb{Z}_{\geq 0}^{k}$ and $\boldsymbol{i}\in [\boldsymbol{d},\boldsymbol{e}]$. By \eqref{multi admissible banach eq in the proof3}, we have $v_{\boldsymbol{h}}^{[\boldsymbol{d},\boldsymbol{e}]}(\mu^{\prime}-\mu_{n})\geq A$ for every $n\geq N$. Thus, we see that $\mu^{\prime}=\lim_{n\rightarrow +\infty}\mu_{n}$.}
\end{proof}
Let $\mu\in \mathrm{Hom}_{\mathcal{O}_{\K}}(C^{[\boldsymbol{d},\boldsymbol{e}]}(\Gamma,\mathcal{O}_{\K}),\K)$ and $\nu\in \mathrm{Hom}_{\mathcal{O}_{\K}}(C^{[\boldsymbol{d},\boldsymbol{e}]}(\Gamma,\mathcal{O}_{\K}),M)$. We can define a convolution $\mu*\nu\in \mathrm{Hom}_{\mathcal{O}_{\K}}(C^{[\boldsymbol{d},\boldsymbol{e}]}(\Gamma,\mathcal{O}_{\K}),M)$ to be 
\begin{equation}
\int_{\Gamma} f(\boldsymbol{x})d(\mu*\nu)=\int_{\Gamma}\left(\int_{\Gamma}f(\boldsymbol{x}\boldsymbol{y})d\mu(\boldsymbol{x})\right)d\nu(\boldsymbol{y})\end{equation}
for each $f\in C^{[\boldsymbol{d},\boldsymbol{e}]}(\Gamma,\mathcal{O}_{\K})$. 
{To verify that this product is well-defined, we will show that, for each $f\in C^{[\boldsymbol{d},\boldsymbol{e}]}(\Gamma,\mathcal{O}_{\K})$, 
the function $\boldsymbol{y}\mapsto \int_{\Gamma}f(\boldsymbol{x}\boldsymbol{y})d\mu(\boldsymbol{x})$ is in $C^{[\boldsymbol{d},\boldsymbol{e}]}(\Gamma,\mathcal{O}_{\K})\otimes_{\mathcal{O}_{\K}}\K$. 
If $f(\boldsymbol{x})=1_{\boldsymbol{a}\Gamma^{p^{\boldsymbol{m}}}}(\boldsymbol{x})\prod_{j=1}^{k}\chi_{j}(x_{j})^{i_{j}}$ for $\boldsymbol{i}\in [\boldsymbol{d},\boldsymbol{e}]$, 
where $1_{\boldsymbol{a}\Gamma^{p^{\boldsymbol{m}}}}(\boldsymbol{x})$ is the characteristic function on $\boldsymbol{a}\Gamma^{p^{\boldsymbol{m}}}$ with $\boldsymbol{a}\in \Gamma$, $\boldsymbol{m}\in \mathbb{Z}_{\geq 0}^{k}$, we have 
\begin{equation}\label{for convolution equation}
\small 
\int_{\Gamma}1_{\boldsymbol{a}\Gamma^{p^{\boldsymbol{m}}}}(\boldsymbol{x}\boldsymbol{y})\prod_{j=1}^{k}\chi_{j}(x_{j}y_{j})^{i_{j}}d\mu(\boldsymbol{x})
=\sum_{\boldsymbol{b}\in \Gamma\slash \Gamma^{p^{\boldsymbol{m}}}}1_{\boldsymbol{b}\Gamma^{p^{\boldsymbol{m}}}}(\boldsymbol{y})\prod_{j=1}^{k}\chi_{j}(y_{j})^{i_{j}}\int_{\boldsymbol{a}\boldsymbol{b}^{-1}\Gamma^{p^{\boldsymbol{m}}}}\prod_{j=1}^{k}\chi_{j}(x_{j})^{i_{j}}d\mu(\boldsymbol{x}).
\normalsize 
\end{equation}
In this situation, 
 the function $\boldsymbol{y}\mapsto \int_{\Gamma}f(\boldsymbol{x}\boldsymbol{y})d\mu(\boldsymbol{x})$ is in $C^{[\boldsymbol{d},\boldsymbol{e}]}(\Gamma,\mathcal{O}_{\K})\otimes_{\mathcal{O}_{\K}}\K$ 
by \eqref{for convolution equation}. Since every function $f\in C^{[\boldsymbol{d},\boldsymbol{e}]}(\Gamma,\mathcal{O}_{\K})$ is a linear combination of $1_{\boldsymbol{a}\Gamma^{p^{\boldsymbol{m}}}}(\boldsymbol{x})\prod_{j=1}^{k}\chi_{j}(x_{j})^{i_{j}}$ with $\boldsymbol{a}\in \Gamma$, $\boldsymbol{m}\in \mathbb{Z}_{\geq 0}^{k}$ and $\boldsymbol{i}\in [\boldsymbol{d},\boldsymbol{e}]$ over $\mathcal{O}_{\K}$, the function $\boldsymbol{y}\mapsto \int_{\Gamma}f(\boldsymbol{x}\boldsymbol{y})d\mu(\boldsymbol{x})$ is in $C^{[\boldsymbol{d},\boldsymbol{e}]}(\Gamma,\mathcal{O}_{\K})\otimes_{\mathcal{O}_{\K}}\K$ for any $f\in C^{[\boldsymbol{d},\boldsymbol{e}]}(\Gamma,\mathcal{O}_{\K})$.} Therefore, $\mathrm{Hom}_{\mathcal{O}_{\K}}(C^{[\boldsymbol{d},\boldsymbol{e}]}(\Gamma,\mathcal{O}_{\K}),\K)$ becomes a commutative $\K$-algebra and $\mathrm{Hom}_{\mathcal{O}_{\K}}( C^{[\boldsymbol{d},\boldsymbol{e}]}(\Gamma,\mathcal{O}_{\K}),M)$ becomes a $\mathrm{Hom}_{\mathcal{O}_{\K}}(C^{[\boldsymbol{d},\boldsymbol{e}]}(\Gamma,\mathcal{O}_{\K}),\K)$-module by the convolutions. 
\begin{lem}\label{admissibule convolution multivariable}
Let $\mu_{1}\in \mathcal{D}_{\boldsymbol{g}}^{[\boldsymbol{d},\boldsymbol{e}]}(\Gamma,\K)$ and $\mu_{2}\in \mathcal{D}_{\boldsymbol{h}}^{[\boldsymbol{d},\boldsymbol{e}]}(\Gamma,M)$, where $\boldsymbol{g},\boldsymbol{h}\in \ord_{p}(\mathcal{O}_{\K}\backslash\{0\})^{k}$. Then, we have $\mu_{1}*\mu_{2}\in \mathcal{D}_{\boldsymbol{g}+\boldsymbol{h}}^{[\boldsymbol{d},\boldsymbol{e}]}(\Gamma,M)$ and 
$v_{\boldsymbol{g}+\boldsymbol{h}}^{[\boldsymbol{d},\boldsymbol{e}]}
(\mu_{1}*\mu_{2})\geq v
_{\boldsymbol{g}}^{[\boldsymbol{d},\boldsymbol{e}]}
(\mu_{1})+v_{\boldsymbol{h}}^{[\boldsymbol{d},\boldsymbol{e}]}(\mu_{2})$.
\end{lem}
\begin{proof}
Let $\boldsymbol{a}\in \Gamma$ and $\boldsymbol{m}\in \mathbb{Z}_{\geq 0}^{k}$. 
Since $1_{\boldsymbol{a}\Gamma^{p^{\boldsymbol{m}}}}(\boldsymbol{x}\boldsymbol{y})=\sum_{\boldsymbol{b}\in \Gamma\slash \Gamma^{p^{\boldsymbol{m}}}}1_{\boldsymbol{b}\Gamma^{p^{\boldsymbol{m}}}}(\boldsymbol{x})1_{\boldsymbol{a}\boldsymbol{b}^{-1}\Gamma^{p^{\boldsymbol{m}}}}(\boldsymbol{y})$, we have
\begin{align}\label{admissibule convolution multivariable lemma eq1}
\begin{split}
&\int_{\boldsymbol{a}\Gamma^{p^{\boldsymbol{m}}}}\prod_{j=1}^{k}(\chi_{j}(x_{j})-\chi_{j}(a_{j}))^{i_{j}-d_{j}}\chi_{j}(x_{j})^{d_{j}}d(\mu_{1}*\mu_{2})\\
&=\sum_{\boldsymbol{b}\in \Gamma\slash \Gamma^{p^{\boldsymbol{m}}}}\int_{\boldsymbol{a}\boldsymbol{b}^{-1}\Gamma^{p^{\boldsymbol{m}}}}\int_{\boldsymbol{b}\Gamma^{p^{\boldsymbol{m}}}}\prod_{j=1}^{k}(\chi_{j}(x_{j}y_{j})-\chi_{j}(a_{j}))^{i_{j}-d_{j}}\chi_{j}(x_{j}y_{j})^{d_{j}}d\mu_{1}(\boldsymbol{x})d\mu_{2}(\boldsymbol{y})\\
&=\sum_{\boldsymbol{b}\in \Gamma\slash \Gamma^{p^{\boldsymbol{m}}}}\int_{\boldsymbol{a}\boldsymbol{b}^{-1}\Gamma^{p^{\boldsymbol{m}}}}\int_{\boldsymbol{b}\Gamma^{p^{\boldsymbol{m}}}}\Bigl(\sum_{\boldsymbol{j}\in [\boldsymbol{d},\boldsymbol{i}]}\prod_{r=1}^{k}\begin{pmatrix}i_{r}-d_{r}\\ j_{r}-d_{r}\end{pmatrix}\prod_{t=1}^{k}(\chi_{t}(x_{t}y_{t})-\chi_{t}(b_{t}y_{t}))^{j_{t}-d_{t}}\\
&\ \ \times \prod_{s=1}^{k}(\chi_{s}(b_{s}y_{s})-\chi_{s}(a_{s}))^{i_{s}-j_{s}}\Bigr)\chi_{j}(x_{j}y_{j})^{d_{j}}d\mu_{1}(\boldsymbol{x})d\mu_{2}(\boldsymbol{y})\\
&=\sum_{\boldsymbol{b}\in \Gamma\slash \Gamma^{p^{\boldsymbol{m}}}}\sum_{\boldsymbol{j}\in [\boldsymbol{d},\boldsymbol{i}]}\prod_{r=1}^{k}\begin{pmatrix}i_{r}-d_{r}\\ j_{r}-d_{r}\end{pmatrix}\chi_{r}(b_{r})^{i_{r}-j_{r}}\int_{\boldsymbol{a}\boldsymbol{b}^{-1}\Gamma^{p^{\boldsymbol{m}}}}\prod_{s=1}^{k}(\chi_{s}(y_{s})-\chi_{s}(a_{s}b_{s}^{-1}))^{i_{s}-j_{s}}\\
&\ \  \times \chi_{s}(y_{s})^{j_{s}}d\mu_{2}(\boldsymbol{y})\int_{\boldsymbol{b}\Gamma^{p^{\boldsymbol{m}}}}\prod_{t=1}^{k}(\chi_{t}(x_{t})-\chi_{t}(b_{t}))^{j_{t}-d_{t}}\chi_{t}(x_{t})^{d_{t}}d\mu_{1}(\boldsymbol{x})
\end{split}
\end{align}
for each $\boldsymbol{i}\in [\boldsymbol{d},\boldsymbol{e}]$. We have 
\begin{equation}\label{admissibule convolution multivariable lemma eq2}
\ord_{p}\left(\int_{\boldsymbol{b}\Gamma^{p^{\boldsymbol{m}}}}\prod_{t=1}^{k}(\chi_{t}(x_{t})-\chi_{t}(b_{t}))^{j_{t}-d_{t}}\chi_{t}(x_{t})^{d_{t}}d\mu_{1}(\boldsymbol{x})\right)\geq -\langle \boldsymbol{g}-(\boldsymbol{j}-\boldsymbol{d}),\boldsymbol{m}\rangle_{k}+v_{\boldsymbol{g}}^{[\boldsymbol{d},\boldsymbol{e}]}(\mu_{1})
\end{equation}
for each $\boldsymbol{j}\in [\boldsymbol{d},\boldsymbol{i}]$. We see that
\begin{align*}
&\int_{\boldsymbol{a}\boldsymbol{b}^{-1}\Gamma^{p^{\boldsymbol{m}}}}\prod_{s=1}^{k}(\chi_{s}(y_{s})-\chi_{s}(a_{s}b_{s}^{-1}))^{i_{s}-j_{s}}\chi_{s}(y_{s})^{j_{s}}d\mu_{2}(\boldsymbol{y})\\
&=\sum_{\boldsymbol{q}\in [\boldsymbol{d},\boldsymbol{j}]}\prod_{t=1}^{k}\begin{pmatrix}j_{t}-d_{t}\\ q_{t}-d_{t}\end{pmatrix}\chi_{t}(a_{t}b_{t}^{-1})^{j_{t}-q_{t}}\\
&\ \ \times \int_{\boldsymbol{a}\boldsymbol{b}^{-1}\Gamma^{p^{\boldsymbol{m}}}}\prod_{r=1}^{k}(\chi_{r}(y_{r})-\chi_{r}(a_{r}b_{r}^{-1}))^{i_{r}-j_{r}+q_{r}-d_{r}}\chi_{r}(y_{r})^{d_{r}}d\mu_{2}(\boldsymbol{y}).
\end{align*}
Then, we have
\begin{equation}\label{admissibule convolution multivariable lemma eq3}
v_{M}\left(\int_{\boldsymbol{a}\boldsymbol{b}^{-1}\Gamma^{p^{\boldsymbol{m}}}}\prod_{s=1}^{k}(\chi_{s}(y_{s})-\chi_{s}(a_{s}b_{s}^{-1}))^{i_{s}-j_{s}}\chi_{s}(y_{s})^{j_{s}}d\mu_{2}(\boldsymbol{y})\right)\geq -\langle \boldsymbol{h}-(\boldsymbol{i}-\boldsymbol{j}),\boldsymbol{m}\rangle_{k}+v_{\boldsymbol{h}}^{[\boldsymbol{d},\boldsymbol{e}]}(\mu_{2})
\end{equation}
for each $\boldsymbol{j}\in [\boldsymbol{d},\boldsymbol{i}]$.
By \eqref{admissibule convolution multivariable lemma eq1}, \eqref{admissibule convolution multivariable lemma eq2} and \eqref{admissibule convolution multivariable lemma eq3}, we see that
\begin{align*}
&v_{M}\left(\int_{\boldsymbol{a}\Gamma^{p^{\boldsymbol{m}}}}\prod_{j=1}^{k}(\chi_{j}(x_{j})-\chi_{j}(a_{j}))^{i_{j}-d_{j}}\chi_{j}(x_{j})^{d_{j}}d(\mu_{1}*\mu_{2})\right)\\
&\geq -\langle \boldsymbol{g}+\boldsymbol{h}-(\boldsymbol{i}-\boldsymbol{d}),\boldsymbol{m}\rangle_{k}+v_{\boldsymbol{g}}^{[\boldsymbol{d},\boldsymbol{e}]}
(\mu_{1})+v_{\boldsymbol{h}}^{[\boldsymbol{d},\boldsymbol{e}]}(\mu_{2}).
\end{align*}
Thus we obtain the desired inequality $v_{\boldsymbol{g}+\boldsymbol{h}}^{[\boldsymbol{d},\boldsymbol{e}]}
(\mu_{1}*\mu_{2})\geq 
v_{\boldsymbol{g}}^{[\boldsymbol{d},\boldsymbol{e}]}
(\mu_{1})+v_{\boldsymbol{h}}^{[\boldsymbol{d},\boldsymbol{e}]}(\mu_{2})$. 
\end{proof}
By Lemma\ \ref{admissibule convolution multivariable}, $\mathcal{D}_{\boldsymbol{0}_{k}}^{[\boldsymbol{d},\boldsymbol{e}]}(\Gamma,\K)$ becomes a $\K$-algebra and $\mathcal{D}_{\boldsymbol{h}}^{[\boldsymbol{d},\boldsymbol{e}]}(\Gamma,M)$ becomes a $\mathcal{D}_{\boldsymbol{0}_{k}}^{[\boldsymbol{d},\boldsymbol{e}]}(\Gamma,\K)$-module, where $\boldsymbol{0}_{k}=(0,\ldots, 0)\in \mathbb{Z}_{\geq 0}^{k}$. Let $C(\Gamma,\mathcal{O}_{\K})$ be the $\mathcal{O}_{\K}$-algebra of continuous functions $f: \Gamma\rightarrow \mathcal{O}_{\K}$. We note that $\mathrm{Hom}_{\mathcal{O}_{\K}}(C(\Gamma,\mathcal{O}_{\K}),\mathcal{O}_{\K})$ becomes an $\mathcal{O}_{\K}$-algebra by the natural convolution and we see that
\begin{equation}\label{continuous measure and spectral norm}
v_{M}\left(\int_{\Gamma}f(\boldsymbol{x})d\mu\right)\geq \inf\{\ord_{p}(f(\boldsymbol{x}))\}_{\boldsymbol{x}\in \Gamma}
\end{equation}
for every $\mu\in \mathrm{Hom}_{\mathcal{O}_{\K}}(C(\Gamma,\mathcal{O}_{\K}),M^{0})$ and for every $f\in C(\Gamma,\mathcal{O}_{\K})$ easily. 
\par 
Let $\mu\in \mathrm{Hom}_{\mathcal{O}_{\K}}(C(\Gamma,\mathcal{O}_{\K}),M^{0})$. 
Recall that we have $\mu\vert_{C^{[\boldsymbol{d},\boldsymbol{e}]}(\Gamma,\mathcal{O}_{\K})}\in \mathcal{D}_{\boldsymbol{0}_{k}}^{[\boldsymbol{d},\boldsymbol{e}]}(\Gamma,M)^{0}$ 
by \eqref{continuous measure and spectral norm}. 
Let $\varphi:\mathrm{Hom}_{\mathcal{O}_{\K}}(C(\Gamma,\mathcal{O}_{\K}),M^{0}){\rightarrow}\mathcal{D}_{\boldsymbol{0}_{k}}^{[\boldsymbol{d},\boldsymbol{e}]}(\Gamma,M)^{0}$ 
be an $\mathcal{O}_{\K}$-linear homomorphism defined by setting $\varphi (\mu ) =\mu\vert_{C^{[\boldsymbol{d},\boldsymbol{e}]}(\Gamma,\mathcal{O}_{\K})}$ for each $\mu\in \mathrm{Hom}_{\mathcal{O}_{\K}}(C(\Gamma,\mathcal{O}_{\K}),M^{0})$. 
\begin{pro}\label{isomorphism multi 0admissible continuous}
The $\mathcal{O}_{\K}$-linear homomorphism $\varphi:\mathrm{Hom}_{\mathcal{O}_{\K}}(C(\Gamma,\mathcal{O}_{\K}),M^{0}){\rightarrow}\mathcal{D}_{\boldsymbol{0}_{k}}^{[\boldsymbol{d},\boldsymbol{e}]}(\Gamma,\linebreak M)^{0}$ is an isomorphism. 
Further, if $M=\K$, $\varphi$ becomes an $\mathcal{O}_{\K}$-algebra isomorphism.
\end{pro}
\begin{proof}
Let $\mu\in \mathrm{Hom}_{\mathcal{O}_{\K}}(C(\Gamma,\mathcal{O}_{\K}),M^{0})$. 
Since $C^{[\boldsymbol{d}, \boldsymbol{e}]}(\Gamma,\mathcal{O}_{\K})$ is dense in $C(\Gamma,\mathcal{O}_{\K})$ with respect to the uniform norm, we have $\mu=0$ if $\mu\vert_{C^{[\boldsymbol{d},\boldsymbol{e}]}(\Gamma,\mathcal{O}_{\K})}=0$. Hence $\varphi$ is injective. 
\par 
In the rest of the proof, we prove that $\varphi$ is surjective. 
Let $f\in C^{[\boldsymbol{d},\boldsymbol{d}]}(\Gamma,\mathcal{O}_{\K})$ and let $\nu\in \mathcal{D}_{\boldsymbol{0}_{k}}^{[\boldsymbol{d},\boldsymbol{e}]}(\Gamma,M)^{0}$. 
Then there exist a sufficiently large $\boldsymbol{m}\in \mathbb{Z}_{\geq 0}^{k}$ and $c_{\boldsymbol{a}}\in \mathcal{O}_{\K}$ for each $\boldsymbol{a}\in \Gamma\slash \Gamma^{p^{\boldsymbol{m}}}$ such that $f(\boldsymbol{x})=\left(\sum_{\boldsymbol{a}\in \Gamma\slash \Gamma^{p^{\boldsymbol{m}}}}c_{\boldsymbol{a}}1_{\boldsymbol{a}\Gamma^{p^{\boldsymbol{m}}}}(\boldsymbol{x})\right)\prod_{t=1}^{k}\chi_{t}(x_{t})^{d_{t}}$ where $1_{\boldsymbol{a}\Gamma^{p^{\boldsymbol{m}}}}$ is the characteristic function on $\boldsymbol{a}\Gamma^{p^{\boldsymbol{m}}}$. 
Hence, we have 
\begin{align}\label{isomorphism multi 0admissible continuouseq1}
\begin{split}
v_{M}\left(\int_{\Gamma}f(\boldsymbol{x})d\nu\right)&\geq \min_{\boldsymbol{a}\in \Gamma\slash \Gamma^{p^{\boldsymbol{m}}}}\left\{\ord_{p}(c_{\boldsymbol{a}})+v_{M}\left(\int_{\boldsymbol{a}\Gamma^{p^{\boldsymbol{m}}}}\prod_{t=1}^{k}\chi_{t}(x_{t})^{d_{t}}d\nu\right)\right\}.
\end{split}
\end{align}
{By \eqref{admissible condition}, we have $v_{M}\left(\int_{\boldsymbol{a}\Gamma^{p^{\boldsymbol{m}}}}\prod_{t=1}^{k}\chi_{t}(x_{t})^{d_{t}}d\nu\right)\geq v_{\boldsymbol{0}_{k}}^{[\boldsymbol{d},\boldsymbol{e}]}(\nu)$ for each $\boldsymbol{a}\in \Gamma^{p^{\boldsymbol{m}}}$. 
Since $\nu\in \mathcal{D}_{\boldsymbol{0}_{k}}^{[\boldsymbol{d},\boldsymbol{e}]}(\Gamma,M)^{0}$, 
we have $v_{\boldsymbol{0}_{k}}^{[\boldsymbol{d},\boldsymbol{e}]}(\nu)\geq 0$. 
By $f(\boldsymbol{x})=\left(\sum_{\boldsymbol{a}\in \Gamma\slash \Gamma^{p^{\boldsymbol{m}}}}c_{\boldsymbol{a}}1_{\boldsymbol{a}\Gamma^{p^{\boldsymbol{m}}}}(\boldsymbol{x})\right)\prod_{t=1}^{k}\chi_{t}(x_{t})^{d_{t}}$, we have $\ord_{p}(c_{\boldsymbol{a}})=\inf_{\boldsymbol{x}\in \boldsymbol{a}\Gamma^{p^{\boldsymbol{m}}}}\{f(\boldsymbol{x})\}$. 
Combining these facts with \eqref{isomorphism multi 0admissible continuouseq1}, we have
\begin{align}\label{isomorphism multi 0admissible continuouseq1.1}
\begin{split}
v_{M}\left(\int_{\Gamma}f(\boldsymbol{x})d\nu\right)&\geq \min_{\boldsymbol{a}\in \Gamma\slash \Gamma^{p^{\boldsymbol{m}}}}\left\{\ord_{p}(c_{\boldsymbol{a}})+v_{M}\left(\int_{\boldsymbol{a}\Gamma^{p^{\boldsymbol{m}}}}\prod_{t=1}^{k}\chi_{t}(x_{t})^{d_{t}}d\nu\right)\right\}\\
&\geq \min_{\boldsymbol{a}\in \Gamma\slash \Gamma^{p^{\boldsymbol{m}}}}\{\ord_{p}(c_{\boldsymbol{a}})\}\\
&\geq \mathrm{inf}_{\boldsymbol{x}\in \Gamma}\{f(\boldsymbol{x})\}.
\end{split}
\end{align}}
By \eqref{isomorphism multi 0admissible continuouseq1.1}, if a sequence $\{f_{n}\}_{n\geq 1}$ of $C^{[\boldsymbol{d},\boldsymbol{d}]}(\Gamma, \mathcal{O}_{\K})$ converges to a
function $f\in C(\Gamma,\mathcal{O}_{\K})$ with respect to the uniform norm, there exists a limit $\lim_{n\rightarrow +\infty}\int_{\Gamma}f_{n}(\boldsymbol{x})d\nu\in M^{0}$. Since $C^{[\boldsymbol{d},\boldsymbol{d}]}(\Gamma,\mathcal{O}_{\K})$ is dense in $C(\Gamma,\mathcal{O}_{\K})$ with respect to the uniform norm, we define an element $\mu\in \mathrm{Hom}_{\mathcal{O}_{\K}}(C(\Gamma,\mathcal{O}_{\K}),M^{0})$ to be
$$\int_{\Gamma}f(\boldsymbol{x})d\mu=\lim_{n\rightarrow+\infty}\int_{\Gamma}f_{n}(\boldsymbol{x})d\nu$$
where $\{f_{n}\}_{n\geq 1}$ is a sequence of $C^{[\boldsymbol{d},\boldsymbol{d}]}(\Gamma,\mathcal{O}_{\K})$ which converges to $f$ with respect to the uniform norm. 
We want to prove that $\varphi(\mu)=\nu$. So we will prove that $\nu^{\prime}=0$ by putting $\nu^{\prime}=\varphi(\mu)-\nu$. 
{To prove $\nu^{\prime}=0$, it suffices to prove the followings:
\begin{enumerate}
\item We have $\nu^{\prime}\vert_{C^{[\boldsymbol{d},\boldsymbol{d}]}(\Gamma,\mathcal{O}_{\K})}=0$.
\item Let $\boldsymbol{i}\in [\boldsymbol{d},\boldsymbol{e}]$ be an element satisfying $\boldsymbol{i}\neq \boldsymbol{d}$. Assume that $\nu^{\prime}\vert_{C^{[\boldsymbol{d},\boldsymbol{j}]}(\Gamma,\mathcal{O}_{\K})}=0$ for each $\boldsymbol{j}\in [\boldsymbol{d},\boldsymbol{i}]$ such that $\boldsymbol{j}\neq \boldsymbol{i}$. Then we have  $\nu^{\prime}\vert_{C^{[\boldsymbol{d},\boldsymbol{i}]}(\Gamma,\mathcal{O}_{\K})}=0$.
\end{enumerate}}
By the definition of $\mu$, we see that $\int_{\Gamma}f(\boldsymbol{x})d\nu^{\prime}=0$ for each $f\in C^{[\boldsymbol{d},\boldsymbol{d}]}(\Gamma,\mathcal{O}_{\K})$. Let $\boldsymbol{i}\in [\boldsymbol{d},\boldsymbol{e}]$ such that $\boldsymbol{i}\neq \boldsymbol{d}$ and assume that $\nu^{\prime}\vert_{C^{[\boldsymbol{d},\boldsymbol{j}]}(\Gamma,\mathcal{O}_{\K})}=0$ for each $\boldsymbol{d}\leq \boldsymbol{j}<\boldsymbol{i}$. Let  $P_{\boldsymbol{i}}$ be the subset of $\{1,\ldots, k\}$ consisting of $t$ such that $d_{t}<i_{t}$. Put $\boldsymbol{i}_{t}^{\prime}=(i_{1},\ldots,i_{t}-1,\ldots,i_{k})$ for each  $t\in P_{\boldsymbol{i}}$. By definition, we see that $\boldsymbol{d}\leq \boldsymbol{i}_{t}^{\prime}<\boldsymbol{i}$ with $t\in P_{\boldsymbol{i}}$. Since $\left(\prod_{t=1}^{k}\chi_{t}(x_{t})^{i_{t}}\right)1_{\boldsymbol{a}\Gamma^{p^{\boldsymbol{m}}}}(\boldsymbol{x})-\left(\prod_{t=1}^{k}(\chi_{t}(x_{t})-\chi_{t}(a_{t})^{i_{t}-d_{t}})\chi_{t}(x_{t})^{d_{t}}\right)1_{\boldsymbol{a}\Gamma^{p^{\boldsymbol{m}}}}(\boldsymbol{x})\in \sum_{t\in P_{\boldsymbol{i}}}C^{[\boldsymbol{d},\boldsymbol{i}_{t}^{\prime}]}(\Gamma,\mathcal{O}_{\K})$ for each $\boldsymbol{a}\in \Gamma$ and $\boldsymbol{m}\in \mathbb{Z}_{\geq 0}^{k}$, we see that 
\begin{align}\label{isomorphism multi 0admissible continuouseq2}
\begin{split}
v_{M}\left(\int_{\boldsymbol{a}\Gamma^{p^{\boldsymbol{m}}}}\prod_{t=1}^{k}\chi_{t}(x_{t})^{i_{t}}d\nu^{\prime}\right)
& =v_{M}\left(\int_{\boldsymbol{a}\Gamma^{p^{\boldsymbol{m}}}}\prod_{t=1}^{k}(\chi_{t}(x_{t})-\chi_{t}(a_{t})^{i_{t}-d_{t}})\chi_{t}(x_{t})^{d_{t}}d\nu^{\prime}\right)\\
&\geq \sum_{t=1}^{k}(i_{t}-d_{t})m_{t}
\end{split}
\end{align}
where $1_{\boldsymbol{a}\Gamma^{p^{\boldsymbol{m}}}}$ is the characteristic function on $\boldsymbol{a}\Gamma^{p^{\boldsymbol{m}}}$. Let $\boldsymbol{a}\in \Gamma$ and $\boldsymbol{m}\in \mathbb{Z}_{\geq 0}^{k}$. Since $1_{\boldsymbol{a}\Gamma^{p^{\boldsymbol{m}}}}(\boldsymbol{x})=\sum_{\boldsymbol{b}\in \Gamma^{p^{\boldsymbol{m}}}\slash \Gamma^{p^{\boldsymbol{m}+\boldsymbol{n}}}}1_{\boldsymbol{a}\boldsymbol{b}\Gamma^{p^{\boldsymbol{m}+\boldsymbol{n}}}}(x)$ for each $\boldsymbol{n}\in \mathbb{Z}_{\geq 0}^{k}$, by \eqref{isomorphism multi 0admissible continuouseq2}, we have 
\begin{align*}
v_{M}\left(\int_{\boldsymbol{a}\Gamma^{p^{\boldsymbol{m}}}}\prod_{t=1}^{k}\chi_{t}(x_{t})^{i_{t}}d\nu^{\prime}\right)&\geq \lim_{\boldsymbol{n}\rightarrow +\infty}\min_{\boldsymbol{b}\in \Gamma^{p^{\boldsymbol{m}}}\slash \Gamma^{p^{\boldsymbol{m}+\boldsymbol{n}}}}\left\{v_{M}\left(\int_{\boldsymbol{ab}\Gamma^{p^{\boldsymbol{m}+\boldsymbol{n}}}}\prod_{t=1}^{k}\chi_{t}(x_{t})^{i_{t}}d\nu^{\prime}\right)\right\}\\
&\geq\lim_{\boldsymbol{n}\rightarrow +\infty}\sum_{t=1}^{k}(i_{t}-d_{t})(m_{t}+n_{t})=+\infty.
\end{align*}
Hence $\int_{\boldsymbol{a}\Gamma^{p^{\boldsymbol{m}}}}\prod_{t=1}^{k}\chi_{t}(x_{t})^{i_{t}}d\nu^{\prime}=0$. By assumption, we have $\int_{\boldsymbol{a}\Gamma^{p^{\boldsymbol{m}}}}\prod_{t=1}^{k}\chi_{t}(x_{t})^{j_{t}}d\nu^{\prime}=0$ for each $\boldsymbol{0}\leq \boldsymbol{j}<\boldsymbol{i}$. Since every $f\in C^{[\boldsymbol{d},\boldsymbol{i}]}(\Gamma,\mathcal{O}_{\K})$ can be written as a linear combination of $1_{\boldsymbol{a}\Gamma^{p^{\boldsymbol{m}}}}(\boldsymbol{x})\prod_{t=1}^{k}\chi_{t}(x_{t})^{j_{t}}$ with $\boldsymbol{a}\in \Gamma$, $\boldsymbol{m}\in \mathbb{Z}_{\geq 0}^{k}$ and $\boldsymbol{d}\leq \boldsymbol{j}\leq \boldsymbol{i}$, we  have $\nu^{\prime}\vert_{C^{[\boldsymbol{d},\boldsymbol{i}]}(\Gamma,\mathcal{O}_{\K})}=0$. Then we have $\nu^{\prime}=0$. Hence $\varphi(\mu)=\nu$. {Finally, it is easy to see that $\varphi$ is an $\mathcal{O}_{\K}$-algebra isomorphism if $M=\K$.}
\end{proof}
We recall the definition of arithmetic specializations. Let $\mathbf{I}$ be a finite free extension of $\mathcal{O}_{\K}[[\Gamma]]$. Assume that $\mathbf{I}$ is an integral domain. A continuous $\mathcal{O}_{\K}$-algebra homomorphism $\kappa:\mathbf{I}\rightarrow \overline{\K}$ is called  an arithmetic specialization of weight $\boldsymbol{w}_{\kappa}\in\mathbb{Z}^{k}$ and finite part $\boldsymbol{\phi}_{\kappa}=(\phi_{\kappa,1},\ldots, \phi_{\kappa,k})$ if $\kappa\vert_{\Gamma}:\Gamma\rightarrow \overline{\K}^{\times}$ is a continuous character given by $\kappa(\boldsymbol{x})=\prod_{i=1}^{k}(\chi_{i}^{w_{\kappa,i}}\phi_{\kappa,i})(x_{i})$ for each $\boldsymbol{x}\in \Gamma$, where $\phi_{\kappa,i}$ is a finite order character on $\Gamma_{i}$ with $1\leq i\leq k$. Let $\mathfrak{X}_{\mathbf{I}}$ be the set of arithmetic specializations on $\mathbf{I}$ and $\mathfrak{X}_{\mathbf{I}}^{[\boldsymbol{d},\boldsymbol{e}]}\subset \mathfrak{X}_{\mathbf{I}}$ the subset consisting of arithmetic specializations $\kappa$ with $\boldsymbol{w}_{\kappa}\in [\boldsymbol{d},\boldsymbol{e}]$. For each $\kappa\in \mathfrak{X}_{\mathbf{I}}$, we put $\boldsymbol{m}_{\kappa}=(m_{\kappa,1},\ldots, m_{\kappa,k})$, where $m_{\kappa,i}$ is the smallest integer $m$ such that $\phi_{\kappa,i}$ factors through $\Gamma_{i}\slash\Gamma_{i}^{p^m}$ with $1\leq i\leq k$. 
\par 
Let $\kappa\in \mathfrak{X}_{\mathcal{O}_{\K}[[\Gamma]]}$ be an arithmetic specialization. We define a map
\begin{equation}
\kappa: M^{0}[[\Gamma]]\rightarrow M_{\K(\phi_{\kappa,1},\ldots, \phi_{\kappa,k})}
\end{equation}
to be $\kappa(c\widehat{\otimes}_{\mathcal{O}_{\K}}m)=m\otimes_{\K}\kappa(c)$ for each $c\in \mathcal{O}_{\K}[[\Gamma]]$ and $m\in M^{0}$. We prove that we have an $\mathcal{O}_{\K}$-module isomorphism
\begin{equation}\label{hom banach iwasawa continuous}
\mathrm{Hom}_{\mathcal{O}_{\K}}(C(\Gamma,\mathcal{O}_{\K}),M^{0})\stackrel{\sim}{\rightarrow}M^{0}[[\Gamma]],\ \mu\mapsto h_{\mu},
\end{equation}
where $h_{\mu}$ is the unique element characterized by $\int_{\Gamma}\kappa\vert_{\Gamma}d\mu=\kappa(h_{\mu})$ for each $\kappa\in \mathfrak{X}_{\mathcal{O}_{\K}[[\Gamma]]}$. 

By Proposition \ref{isomorphism multi 0admissible continuous}, we have an isomorphism $\varphi:\mathrm{Hom}_{\mathcal{O}_{\K}}(C(\Gamma,\mathcal{O}_{\K}),M^{0})\stackrel{\sim}{\rightarrow}\mathrm{Meas}(\Gamma, M^{0})$, where $\mathrm{Meas}(\Gamma,M^{0})=\mathcal{D}_{\boldsymbol{0}_{k}}^{[\boldsymbol{0}_{k},\boldsymbol{0}_{k}]}(\Gamma,M)^{0}$. We denote by $LC(\Gamma\slash \Gamma^{p^{\boldsymbol{n}}},\mathcal{O}_{\K})$ the $\mathcal{O}_{\K}$-module of functions $f:\Gamma\slash \Gamma^{p^{\boldsymbol{n}}}\rightarrow \mathcal{O}_{\K}$ for each $\boldsymbol{n}\in \mathbb{Z}_{\geq 0}^{k}$. It is well-known that there exists a natural $\mathcal{O}_{\K}$-algebra isomorphism $\mathrm{Meas}(\Gamma\slash \Gamma^{p^{\boldsymbol{n}}},\mathcal{O}_{\K})=\mathrm{Hom}_{\mathcal{O}_{\K}}(LC(\Gamma\slash \Gamma^{p^{\boldsymbol{n}}},\mathcal{O}_{\K}),\mathcal{O}_{\K})\simeq \mathcal{O}_{\K}[\Gamma\slash \Gamma^{p^{\boldsymbol{n}}}]$ defined by $\mu_{\boldsymbol{a}}\mapsto [\boldsymbol{a}]$ for each $\boldsymbol{a}\in \Gamma\slash \Gamma^{p^{\boldsymbol{n}}}$, where $\mu_{\boldsymbol{a}}$ is the Dirac measure at $\boldsymbol{a}\in\Gamma \slash \Gamma^{p^{\boldsymbol{n}}}$. We remark that the natural maps $LC(\Gamma\slash \Gamma^{p^{\boldsymbol{n}}},\mathcal{O}_{\K})\rightarrow LC(\Gamma,\mathcal{O}_{\K})$ defined by $f\mapsto f\pi_{\boldsymbol{n}}$ with $\boldsymbol{n}\in \mathbb{Z}_{\geq 0}^{k}$ induce an $\mathcal{O}_{\K}$-module isomorphism $\displaystyle{{\varinjlim}_{\boldsymbol{n}\in\mathbb{Z}_{\geq 0}^{k}}}LC(\Gamma\slash \Gamma^{p^{\boldsymbol{n}}},\mathcal{O}_{\K})\stackrel{\sim}{\rightarrow}LC(\Gamma,\mathcal{O}_{\K})$, where $LC(\Gamma,\mathcal{O}_{\K})=C^{[\boldsymbol{0}_{k},\boldsymbol{0}_{k}]}(\Gamma,\mathcal{O}_{\K})$ and $\pi_{\boldsymbol{n}}:\Gamma\rightarrow \Gamma\slash \Gamma^{p^{\boldsymbol{n}}}$ is the projection. Then, we have a natural $\mathcal{O}_{\K}$-module isomorphism 
$$
\psi:\mathrm{Meas}(\Gamma,M^{0})\stackrel{\sim}{\rightarrow} \displaystyle{\varprojlim_{\boldsymbol{n}}}(\mathrm{Meas}(\Gamma\slash \Gamma^{p^{\boldsymbol{n}}},\mathcal{O}_{\K})\otimes_{\mathcal{O}_{\K}}M^{0})\stackrel{\sim}{\rightarrow} \displaystyle{\varprojlim_{\boldsymbol{n}}}(\mathcal{O}_{\K}[\Gamma\slash \Gamma^{p^{\boldsymbol{n}}}]\widehat{\otimes}_{\mathcal{O}_{\K}}M^{0})=M^{0}[[\Gamma]].
$$ Therefore, we have $\psi\circ\varphi:\mathrm{Hom}_{\mathcal{O}_{\K}}(C(\Gamma,\mathcal{O}_{\K}),M^{0})\stackrel{\sim}{\rightarrow}M^{0}[[\Gamma]]$. By definition, we see that $h_{\mu}=\psi\circ \varphi(\mu)=\displaystyle{\lim_{\boldsymbol{n}\rightarrow +\infty}}\sum_{\boldsymbol{a}\in \Gamma\slash \Gamma^{p^{\boldsymbol{n}}}}[\boldsymbol{a}]\widehat{\otimes}_{\mathcal{O}_{\K}}\int_{\boldsymbol{a}\Gamma^{p^{\boldsymbol{n}}}}d\mu$ and we have
$$
\kappa(h_{\mu})=\displaystyle{\lim_{\boldsymbol{n}\rightarrow +\infty}}\sum_{\boldsymbol{a}\in \Gamma\slash \Gamma^{p^{\boldsymbol{n}}}}\int_{\boldsymbol{a}\Gamma^{p^{\boldsymbol{n}}}}\kappa\vert_{\Gamma}(\boldsymbol{a})d\mu=\int_{\Gamma}\kappa\vert_{\Gamma}d\mu$$ 
for each $\kappa\in \mathfrak{X}_{\mathcal{O}_{\K}[[\Gamma]]}$. Thus, we have \eqref{hom banach iwasawa continuous}. If $M=\K$, the isomorphism of \eqref{hom banach iwasawa continuous} is an $\mathcal{O}_{\K}$-algebra isomorphism. By Proposition\ \ref{isomorphism multi 0admissible continuous}, we have an $\mathcal{O}_{\K}$-algebra isomorphism $\mathcal{O}_{\K}[[\Gamma]]\simeq \mathrm{Hom}_{\mathcal{O}_{\K}}(C(\Gamma,\mathcal{O}_{\K}),\mathcal{O}_{\K})\simeq \mathcal{D}_{\boldsymbol{0}_{k}}^{[\boldsymbol{d},\boldsymbol{e}]}(\Gamma,\K)^{0}$. By Lemma \ref{admissibule convolution multivariable}, $\mathcal{D}_{\boldsymbol{h}}^{[\boldsymbol{d},\boldsymbol{e}]}(\Gamma,M)$ becomes a $\mathcal{D}_{\boldsymbol{0}_{k}}^{[\boldsymbol{d},\boldsymbol{e}]}(\Gamma,\K)$-module. Thus, we can regard $\mathcal{D}_{\boldsymbol{h}}^{[\boldsymbol{d},\boldsymbol{e}]}(\Gamma,M)$ as an $\mathcal{O}_{\K}[[\Gamma]]\otimes_{\mathcal{O}_{\K}}\K$-module.

Let $\boldsymbol{d},\boldsymbol{e}\in \mathbb{Z}^{k}$ such that $\boldsymbol{e}\geq \boldsymbol{d}$ and $\boldsymbol{h}\in \ord_{p}(\mathcal{O}_{\K}\backslash \{0\})^{k}$. Assume that $k\geq 2$ and put $\boldsymbol{h}^{\prime}=(h_{1},\ldots,h_{k-1})$, $\boldsymbol{d}^{\prime}=(d_{1},\ldots, d_{k-1})$, $\boldsymbol{e}^{\prime}=(e_{1},\ldots, e_{k-1})$ and $\Gamma^{\prime}=\Gamma_{1}\times \cdots \times \Gamma_{k-1}$. Then, we have a natural $\mathcal{O}_{\K}$-module isomorphism
\begin{equation}\label{locally poly tensor naturalisom}
C^{[d_{k},e_{k}]}(\Gamma_{k},\mathcal{O}_{\K})\otimes_{\mathcal{O}_{\K}}C^{[\boldsymbol{d}^{\prime},\boldsymbol{e}^{\prime}]}(\Gamma^{\prime},\mathcal{O}_{\K})\stackrel{\sim}{\rightarrow}C^{[\boldsymbol{d},\boldsymbol{e}]}(\Gamma,\mathcal{O}_{\K}),\  f\otimes_{\mathcal{O}_{\K}}g\mapsto g\cdot f
\end{equation}
where $g\cdot f\in C^{[\boldsymbol{d},\boldsymbol{e}]}(\Gamma,\mathcal{O}_{\K})$ is the element defined by   $g\cdot f(\boldsymbol{x})=g(x_{1},\ldots,x_{k-1})f(x_{k})$ for each $\boldsymbol{x}\in \Gamma$. By the isomorphism \eqref{locally poly tensor naturalisom}, we have the following adjunction:
\begin{equation}\label{adjonction of homlocallly poly M}
\mathrm{Hom}_{\mathcal{O}_{\K}}(C^{[\boldsymbol{d},\boldsymbol{e}]}(\Gamma,\mathcal{O}_{\K}),M)\simeq \mathrm{Hom}_{\mathcal{O}_{\K}}(C^{[d_{k},e_{k}]}(\Gamma_{k},\mathcal{O}_{\K}),\mathrm{Hom}_{\mathcal{O}_{\K}}(C^{[\boldsymbol{d}^{\prime},\boldsymbol{e}^{\prime}]}(\Gamma^{\prime},\mathcal{O}_{\K}),M)).
\end{equation}
\begin{pro}\label{for induction admisible admissible}
Assume that $k\geq 2$.  Let $\boldsymbol{h}\in \ord_{p}(\mathcal{O}_{\K}\backslash \{0\})^{k}$ and $\boldsymbol{d},\boldsymbol{e}\in \mathbb{Z}^{k}$ such that $\boldsymbol{e}\geq \boldsymbol{d}$. Put $\boldsymbol{h}^{\prime}=(h_{1},\ldots,h_{k-1})$, $\boldsymbol{d}^{\prime}=(d_{1},\ldots, d_{k-1})$, $\boldsymbol{e}^{\prime}=(e_{1},\ldots, e_{k-1})$ and $\Gamma^{\prime}=\Gamma_{1}\times \cdots \times \Gamma_{k-1}$. The adjunction in \eqref{adjonction of homlocallly poly M} induces the following isometric isomorphism:
$$\varphi:\mathcal{D}_{\boldsymbol{h}}^{[\boldsymbol{d},\boldsymbol{e}]}(\Gamma,M)\stackrel{\sim}{\rightarrow}\mathcal{D}_{h_{k}}^{[d_{k},e_{k}]}(\Gamma_{k},\mathcal{D}_{\boldsymbol{h}^{\prime}}^{[\boldsymbol{d}^{\prime},\boldsymbol{e}^{\prime}]}(\Gamma^{\prime},M)).$$
\end{pro}
\begin{proof}
Let $\mu\in \mathcal{D}_{\boldsymbol{h}}^{[\boldsymbol{d},\boldsymbol{e}]}(\Gamma,M)$. We denote by $\mu^{\prime}\in \mathrm{Hom}_{\mathcal{O}_{\K}}(C^{[d_{k},e_{k}]}(\Gamma_{k},\mathcal{O}_{\K}),\mathrm{Hom}_{\mathcal{O}_{\K}}(C^{[\boldsymbol{d}^{\prime},\boldsymbol{e}^{\prime}]}(\linebreak\Gamma^{\prime},\mathcal{O}_{\K}),M))$  the image of $\mu$ by the adjunction in \eqref{adjonction of homlocallly poly M}. Let $a_{k}\in \Gamma_{k}$, $i_{k}\in [d_{k},e_{k}]$ and $m_{k}\in \mathbb{Z}_{\geq 0}$. We put 
$$\nu_{a_{k}\Gamma_{k}^{p^{m_{k}}}}^{(i_{k})}=\int_{a_{k}\Gamma_{k}^{p^{m_{k}}}}(\chi_{k}(x_{k})-\chi_{k}(a_{k}))^{i_{k}-d_{k}}\chi_{k}(x_{k})^{d_{k}}d\mu^{\prime}\in \mathrm{Hom}_{\mathcal{O}_{\K}}(C^{[\boldsymbol{d}^{\prime},\boldsymbol{e}^{\prime}]}(\Gamma^{\prime},\mathcal{O}_{\K}),M).$$
First, we prove that 
\begin{equation}\label{mufprime is a admissible distr}
\nu_{a_{k}\Gamma_{k}^{p^{m_{k}}}}^{(i_{k})}\in \mathcal{D}_{\boldsymbol{h}^{\prime}}^{[\boldsymbol{d}^{\prime},\boldsymbol{e}^{\prime}]}(\Gamma^{\prime},M).\end{equation}
For each $\boldsymbol{m}^{\prime}\in \mathbb{Z}_{\geq 0}^{k-1}$, $\boldsymbol{a}^{\prime}\in \Gamma^{k-1}$ and $\boldsymbol{i}^{\prime}\in [\boldsymbol{d}^{\prime},\boldsymbol{e}^{\prime}]$, we see that
\begin{align*} 
& v_{M}\left(\int_{\boldsymbol{a}^{\prime}{\Gamma^{\prime}}^{p^{\boldsymbol{m}^{\prime}}}}\prod_{j=1}^{k-1}(\chi_{j}(x_{j})-\chi_{j}(a_{j}))^{i^{\prime}_{j}-d_{j}}\chi_{j}(x_{j})^{d_{j}}d\nu_{a_{k}\Gamma_{k}^{p^{m_{k}}}}^{(i_{k})}\right)+\langle \boldsymbol{h}^{\prime}-(\boldsymbol{i}^{\prime}-\boldsymbol{d}^{\prime}),\boldsymbol{m}^{\prime}\rangle_{k-1}\\
&=v_{M}\Bigg(\int_{\boldsymbol{a}^{\prime}{\Gamma^{\prime}}^{p^{\boldsymbol{m}^{\prime}}}\times a_{k}\Gamma_{k}^{p^{m_{k}}}}\left(\prod_{j=1}^{k-1}(\chi_{j}(x_{j})-\chi_{j}(a_{j}))^{i^{\prime}_{j}-d_{j}}\chi_{j}(x_{j})^{d_{j}}\right)(\chi_{k}(x_{k})-\chi_{k}(a_{k}))^{i_{k}-d_{k}}\\ 
& \chi_{k}(x_{k})^{d_{k}}d\mu\Bigg) +\langle \boldsymbol{h}^{\prime}-(\boldsymbol{i}^{\prime}-\boldsymbol{d}^{\prime}),\boldsymbol{m}^{\prime}\rangle_{k-1} \\
& \geq v_{\boldsymbol{h}}^{[\boldsymbol{d},\boldsymbol{e}]}(\mu)-(h_{k}-(i_{k}-d_{k}))m_{k}.
\end{align*}
\normalsize 
Then, we have
\begin{equation}\label{admissible dieme distri 2}
v_{\boldsymbol{h}^{\prime}}^{[\boldsymbol{d}^{\prime},\boldsymbol{e}^{\prime}]}(\nu_{a_{k}\Gamma_{k}^{p^{m_{k}}}}^{(i_{k})})\geq v_{\boldsymbol{h}}^{[\boldsymbol{d},\boldsymbol{e}]}(\mu)-(h_{k}-(i_{k}-d_{k}))m_{k}.
\end{equation}
Thus, we have \eqref{mufprime is a admissible distr}. 

Next, we prove that $\mu^{\prime}(f)\in \mathcal{D}_{\boldsymbol{h}^{\prime}}^{[\boldsymbol{d}^{\prime},\boldsymbol{e}^{\prime}]}(\Gamma^{\prime},M)$ for each $f\in C^{[d_{k},e_{k}]}(\Gamma_{k},\mathcal{O}_{\K})$. For each $f\in C^{[d_{k},e_{k}]}(\Gamma_{k},\mathcal{O}_{\K})$, there exists an $m_{k}\in \mathbb{Z}_{\geq 0}$ such that we have
$$f(x_{k})=\sum_{a_{k}\in \Gamma_{k}\slash \Gamma_{k}^{p^{m_{k}}}}1_{a_{k}\Gamma_{k}^{p^{m_{k}}}}(x_{k})\sum_{i_{k}=d_{k}}^{e_{k}}c_{i_{k}}^{(a_{k})}(\chi_{k}(x_{k})-\chi_{k}(a_{k}))^{i_{k}-d_{k}}\chi_{k}(x_{k})^{d_{k}}$$
with $c_{i_{k}}^{(a_{k})}\in \mathcal{O}_{\K}$ where $1_{a_{k}\Gamma^{p^{m_{k}}}}(x_{k})$ is the characteristic function on $a_{k}\Gamma_{k}^{p^{m_{k}}}$. Therefore, we have $\mu^{\prime}(f)=\sum_{a_{k}\in \Gamma_{k}\slash \Gamma_{k}^{p^{m_{k}}}}\sum_{i_{k}=d_{k}}^{e_{k}}c_{i_{k}}^{(a_{k})}\nu_{a_{k}\Gamma_{k}^{p^{m_{k}}}}^{(i_{k})}\in \mathcal{D}_{\boldsymbol{h}^{\prime}}^{[\boldsymbol{d}^{\prime},\boldsymbol{e}^{\prime}]}(\Gamma^{\prime},M)$. Then, we see that $\mu^{\prime}(f)\in \mathcal{D}_{\boldsymbol{h}^{\prime}}^{[\boldsymbol{d}^{\prime},\boldsymbol{e}^{\prime}]}(\Gamma^{\prime},M)$ for each $f\in C^{[d_{k},e_{k}]}(\Gamma_{k},\mathcal{O}_{\K})$.

Next, we prove that $\mu^{\prime}\in \mathcal{D}_{h_{k}}^{[d_{k},e_{k}]}(\Gamma_{k},\mathcal{D}_{\boldsymbol{h}^{\prime}}^{[\boldsymbol{d}^{\prime},\boldsymbol{e}^{\prime}]}(\Gamma^{\prime},M))$ and 
\begin{equation}\label{admissible dieme distri 3}
v_{h_{k}}^{[d_{k},e_{k}]}(\mu^{\prime})\geq v_{\boldsymbol{h}}^{[\boldsymbol{d},\boldsymbol{e}]}(\mu).
\end{equation}
By \eqref{admissible dieme distri 2}, we have
\begin{multline*}
v_{\boldsymbol{h}^{\prime}}^{[\boldsymbol{d}^{\prime},\boldsymbol{e}^{\prime}]}\left(\int_{a_{k}\Gamma_{k}^{p^{m_{k}}}}(\chi_{k}(x_{k})-\chi_{k}(a_{k}))^{i_{k}-d_{k}}\chi_{k}(x_{k})^{d_{k}}d\mu^{\prime}\right)+(h_{k}-(i_{k}-d_{k}))m_{k}\\
=v_{\boldsymbol{h}^{\prime}}^{[\boldsymbol{d}^{\prime},\boldsymbol{e}^{\prime}]}(\nu_{a_{k}\Gamma_{k}^{p^{k}}}^{(i_{k})})+(h_{k}-(i_{k}-d_{k}))m_{k}\geq  v_{\boldsymbol{h}}^{[\boldsymbol{d},\boldsymbol{e}]}(\mu).
\end{multline*}
for every $m_{k}\in \mathbb{Z}_{\geq 0}$, $a_{k}\in \Gamma_{k}$ and $i_{k}\in [d_{k},e_{k}]$. Therefore, we have \eqref{admissible dieme distri 3} and $\mu^{\prime}\in  \mathcal{D}_{h_{k}}^{[d_{k},e_{k}]}(\Gamma_{k},\mathcal{D}_{\boldsymbol{h}^{\prime}}^{[\boldsymbol{d}^{\prime},\boldsymbol{e}^{\prime}]}(\Gamma^{\prime},M))$. Thus, $\varphi:\mathcal{D}_{\boldsymbol{h}}^{[\boldsymbol{d},\boldsymbol{e}]}(\Gamma,M)\rightarrow\mathcal{D}_{h_{k}}^{[d_{k},e_{k}]}(\Gamma_{k},\mathcal{D}_{\boldsymbol{h}^{\prime}}^{[\boldsymbol{d}^{\prime},\boldsymbol{e}^{\prime}]}(\Gamma^{\prime},M))$ is well-defined. Further, by \eqref{admissible dieme distri 3}, we have
\begin{equation}\label{admissible dieme distri 4}
v_{\mathfrak{L}}(\varphi)\geq 0.
\end{equation}
Next, we prove that the inverse $\varphi^{-1}$ of $\varphi$ is well-defined and 
\begin{equation}\label{admissible dieme distri 5}
v_{\mathfrak{L}}(\varphi^{-1})\geq 0.
\end{equation}
Let $\mu\in \mathcal{D}_{h_{k}}^{[d_{k},e_{k}]}(\Gamma_{k},\mathcal{D}_{\boldsymbol{h}^{\prime}}^{[\boldsymbol{d}^{\prime},\boldsymbol{e}^{\prime}]}(\Gamma^{\prime},M))$, we denote by $\mu^{\prime\prime}\in \mathrm{Hom}_{\mathcal{O}_{\K}}(C^{[\boldsymbol{d},\boldsymbol{e}]}(\Gamma,\mathcal{O}_{\K}),M)$ the inverse image of $\mu$ by the adjunction in \eqref{adjonction of homlocallly poly M}. Let $\boldsymbol{a}\in \Gamma$, $\boldsymbol{m}\in \mathbb{Z}_{\geq 0}^{k}$ and $\boldsymbol{i}\in [\boldsymbol{d},\boldsymbol{e}]$. We have
\begin{multline*}
\int_{\boldsymbol{a}\Gamma^{p^{\boldsymbol{m}}}}\prod_{j=1}^{k}(\chi_{j}(x_{j})-\chi_{j}(a_{j}))^{i_{j}-d_{j}}\chi_{j}(x_{j})^{d_{j}}d\mu^{\prime\prime}\\
=\int_{\boldsymbol{a}^{\prime}{\Gamma^{\prime}}^{p^{\boldsymbol{m}^{\prime}}}}\prod_{j=1}^{k-1}(\chi_{j}(x_{j})-\chi_{j}(a_{j}))^{i_{j}-d_{j}}\chi_{j}(x_{j})^{d_{j}}dw_{a_{k}\Gamma_{k}^{p^{m_{k}}}}^{(i_{k})}
\end{multline*}
where $\boldsymbol{a}^{\prime}=(a_{1},\ldots, a_{k-1})$, $\boldsymbol{m}^{\prime}=(m_{1},\ldots, m_{k-1})$ and $\boldsymbol{i}^{\prime}=(i_{1},\ldots, i_{k-1})$ and 
$$w_{a_{k}\Gamma_{k}^{p^{m_{k}}}}^{(i_{k})}=\int_{a_{k}\Gamma_{k}^{p^{m_{k}}}}(\chi_{k}(x_{k})-\chi_{k}(a_{k}))^{i_{k}-d_{k}}\chi_{k}(x_{k})^{d_{k}}d\mu\in \mathcal{D}_{\boldsymbol{h}^{\prime}}^{[\boldsymbol{d}^{\prime},\boldsymbol{e}^{\prime}]}(\Gamma^{\prime},M).$$
Then, we see that
\begin{multline*}
v_{M}\left(\int_{\boldsymbol{a}\Gamma^{p^{\boldsymbol{m}}}}\prod_{j=1}^{k}(\chi_{j}(x_{j})-\chi_{j}(a_{j}))^{i_{j}-d_{j}}\chi_{j}(x_{j})^{d_{j}}d\mu^{\prime\prime}\right)+\langle \boldsymbol{h}-(\boldsymbol{i}-\boldsymbol{d}),\boldsymbol{m}\rangle_{k}\\
\geq v_{\boldsymbol{h}^{\prime}}^{[\boldsymbol{d}^{\prime},\boldsymbol{e}^{\prime}]}(w_{a_{k}\Gamma_{k}^{p^{m_{k}}}}^{(i_{k})})+(h_{k}-(i_{k}-d_{k}))m_{k}\geq v_{h_{k}}^{[d_{k},e_{k}]}(\mu)
\end{multline*}
and we conclude that
\begin{equation}\label{admissible dieme distri 6}
v_{\boldsymbol{h}}^{[\boldsymbol{d},\boldsymbol{e}]}(\mu^{\prime\prime})\geq v_{h_{k}}^{[d_{k},e_{k}]}(\mu).
\end{equation}
Then, $\mu^{\prime\prime}\in \mathcal{D}_{\boldsymbol{h}}^{[\boldsymbol{d},\boldsymbol{e}]}(\Gamma,M)$ and we see that $\varphi^{-1}$ is well-defined. Further, by \eqref{admissible dieme distri 6}, we have \eqref{admissible dieme distri 5}.
Then, by Lemma \ref{easy lemma on isometry}, \eqref{admissible dieme distri 4} and  \eqref{admissible dieme distri 5}, we see that $\varphi$ is isometric.
\end{proof}
Let $\boldsymbol{d}^{(i)},\boldsymbol{e}^{(i)}\in \mathbb{Z}^{k}$ with $i=1,2$ such that $[\boldsymbol{d}^{(1)},\boldsymbol{e}^{(1)}]\subset [\boldsymbol{d}^{(2)},\boldsymbol{e}^{(2)}]$. We note that the natural restriction map $\mathrm{Hom}_{\mathcal{O}_{\K}}(C^{[\boldsymbol{d}^{(2)},\boldsymbol{e}^{(2)}]}(\Gamma,\mathcal{O}_{\K}),M)\rightarrow \mathrm{Hom}_{\mathcal{O}_{\K}}(C^{[\boldsymbol{d}^{(1)},\boldsymbol{e}^{(1)}]}(\Gamma,\mathcal{O}_{\K}),M)$, $\mu\mapsto \mu\vert_{C^{[\boldsymbol{d}^{(1)},\boldsymbol{e}^{(1)}]}(\Gamma,\mathcal{O}_{\K}),M)}$ induces the following $\mathcal{O}_{\K}[[\Gamma]]\otimes_{\mathcal{O}_{\K}}\K$-module homomorphism
\begin{equation}\label{proj admissibledistr defnotusis}
\mathcal{D}_{\boldsymbol{h}}^{[\boldsymbol{d}^{(2)},\boldsymbol{e}^{(2)}]}(\Gamma,M)\rightarrow \mathcal{D}_{\boldsymbol{h}}^{[\boldsymbol{d}^{(1)},\boldsymbol{e}^{(1)}]}(\Gamma,M)
\end{equation}
and 
\begin{equation}\label{projadmissibdistr notisom value}
v_{\boldsymbol{h}}^{[\boldsymbol{d}^{(1)},\boldsymbol{e}^{(1)}]}(\mu\vert_{C^{[\boldsymbol{d}^{(1)},\boldsymbol{e}^{(1)}]}(\Gamma,\mathcal{O}_{\K}),M)})\geq v_{\boldsymbol{h}}^{[\boldsymbol{d}^{(2)},\boldsymbol{e}^{(2)}]}(\mu)
\end{equation}
for every $\mu\in\mathcal{D}_{\boldsymbol{h}}^{[\boldsymbol{d}^{(2)},\boldsymbol{e}^{(2)}]}(\Gamma,M)$. Indeed, for each $\boldsymbol{a}\in \Gamma$ we see that
\begin{multline*}
\prod_{j=1}^{k}\chi_{j}(x_{j})^{d_{j}^{(1)}}\\
=\sum_{\boldsymbol{i}\in [\boldsymbol{0}_{k},\boldsymbol{d}^{(1)}-\boldsymbol{d}^{(2)}]}\left(\prod_{j=1}^{k}\begin{pmatrix}d_{j}^{(1)}-d_{j}^{(2)}\\ i_{j}\end{pmatrix}(\chi_{j}(x_{j})-\chi_{j}(a_{j}))^{d_{j}^{(1)}-d_{j}^{(2)}-i_{j}}\chi_{j}(a_{j})^{i_{j}}\right)\prod_{j=1}^{k}\chi_{j}(x_{j})^{d_{j}^{(2)}}.
\end{multline*}
Therefore, if $\mu\in \mathcal{D}_{\boldsymbol{h}}^{[\boldsymbol{d}^{(2)},\boldsymbol{e}^{(2)}]}(\Gamma,M)$, for each $\boldsymbol{m}\in \mathbb{Z}_{\geq 0}^{k}$, $\boldsymbol{a}\in \Gamma$ and $\boldsymbol{i}\in [\boldsymbol{d}^{(1)},\boldsymbol{e}^{(1)}]$, we have
\begin{align*}
&v_{M}\left(\int_{\boldsymbol{a}\boldsymbol{\Gamma}^{p^{\boldsymbol{m}}}}\prod_{j=1}^{k}(\chi_{j}(x_{j})-\chi_{j}(a_{j}))^{i_{j}-d_{j}^{(1)}}\chi_{j}(x_{j})^{d_{j}^{(1)}}d\mu\right)\\
&=v_{M}\Bigg(\sum_{\boldsymbol{t}\in [\boldsymbol{0}_{k},\boldsymbol{d}^{(1)}-\boldsymbol{d}^{(2)}]}\prod_{j=1}^{k}\begin{pmatrix}d_{j}^{(1)}-d_{j}^{(2)}\\ t_{j}\end{pmatrix}\chi_{j}(a_{j})^{t_{j}}\int_{\boldsymbol{a}\boldsymbol{\Gamma}^{p^{\boldsymbol{m}}}}(\chi_{j}(x_{j})-\chi_{j}(a_{j}))^{i_{j}-d_{j}^{(2)}-t_{j}}\\
&\prod_{j=1}^{k}\chi_{j}(x_{j})^{d_{j}^{(2)}}d\mu\Bigg)\\
&\geq \min_{\boldsymbol{t}\in [\boldsymbol{0}_{k},\boldsymbol{d}^{(1)}-\boldsymbol{d}^{(2)}]}\left\{v_{M}\left(\int_{\boldsymbol{a}\boldsymbol{\Gamma}^{p^{\boldsymbol{m}}}}(\chi_{j}(x_{j})-\chi_{j}(a_{j}))^{i_{j}-d_{j}^{(2)}-t_{j}}\prod_{j=1}^{k}\chi_{j}(x_{j})^{d_{j}^{(2)}}d\mu\right)\right\}\\
&\geq \min_{\boldsymbol{t}\in [\boldsymbol{0}_{k},\boldsymbol{d}^{(1)}-\boldsymbol{d}^{(2)}]}\left\{v_{\boldsymbol{h}}^{[\boldsymbol{d}^{(2)},\boldsymbol{e}^{(2)}]}(\mu)-\langle\boldsymbol{h}-(\boldsymbol{i}-\boldsymbol{d}^{(2)}-\boldsymbol{t}),\boldsymbol{m} \rangle_{k}\right\}\\
&\geq v_{\boldsymbol{h}}^{[\boldsymbol{d}^{(2)},\boldsymbol{e}^{(2)}]}(\mu)-\langle \boldsymbol{h}-(\boldsymbol{i}-\boldsymbol{d}^{(1)}),\boldsymbol{m}\rangle_{k}.
\end{align*}
 Then, we see that $v_{\boldsymbol{h}}^{[\boldsymbol{d}^{(1)},\boldsymbol{e}^{(1)}]}(\mu\vert_{C^{[\boldsymbol{d}^{(1)},\boldsymbol{e}^{(1)}]}(\Gamma,\mathcal{O}_{\K}),M)})\geq v_{\boldsymbol{h}}^{[\boldsymbol{d}^{(2)},\boldsymbol{e}^{(2)}]}(\mu)$. Therefore, we have \eqref{proj admissibledistr defnotusis} and \eqref{projadmissibdistr notisom value}.
 \begin{lem}\label{locally polynomial is locally analytic lemma}
Let $f\in C^{[\boldsymbol{d},\boldsymbol{e}]}(\Gamma,\mathcal{O}_{\K})$ with $\boldsymbol{d},\boldsymbol{e}\in \mathbb{Z}^{k}$ such that $\boldsymbol{d}\leq \boldsymbol{e}$. There exists an $\boldsymbol{m}\in \mathbb{Z}_{\geq 0}^{k}$ such that for each $\boldsymbol{a}\in \Gamma$, there exists a unique $g_{\boldsymbol{a}}\in B_{\boldsymbol{0}_{k}}(\K)^{0}$ which satisfies
 \begin{equation}\label{lemma for admsitrib localyanaly}
 f(\boldsymbol{x})=g_{\boldsymbol{a}}(\chi_{1}(x_{1})-\chi_{1}(a_{1}),\ldots, \chi_{k}(x_{k})-\chi_{k}(a_{k}))
 \end{equation}
 for every $\boldsymbol{x}\in \boldsymbol{a}\Gamma^{p^{\boldsymbol{m}}}$ where $\boldsymbol{0}_{k}=(0,\ldots,0)\in \mathbb{Z}^{k}$.
 \end{lem}
 \begin{proof}
 Let $f\in C^{[\boldsymbol{d},\boldsymbol{e}]}(\Gamma,\mathcal{O}_{\K})$. Then, there exists  an $\boldsymbol{m}\in  \mathbb{Z}_{\geq 0}^{k}$ such that for each $\boldsymbol{a}\in \Gamma$, there exists a $p_{\boldsymbol{a}}\in \mathcal{O}_{\K}[X_{1},\ldots, X_{k}]_{\leq \boldsymbol{e}-\boldsymbol{d}}$ which satisfies 
 $$f(\boldsymbol{x})=\left(\prod_{i=1}^{k}\chi_{i}(x_{i})^{d_{i}}\right)p_{\boldsymbol{a}}(\chi_{1}(x_{1}),\ldots, \chi_{k}(x_{k}))$$
for every $\boldsymbol{x}\in \boldsymbol{a}\Gamma^{p^{\boldsymbol{m}}}$. Put $q_{\boldsymbol{a}}=p_{\boldsymbol{a}}(X_{1}+\chi_{1}(a_{1}),\ldots, X_{k}+\chi_{k}(a_{k}))$. Then, we have $q_{\boldsymbol{a}}\in \mathcal{O}_{\K}[X_{1},\ldots, X_{k}]_{\leq \boldsymbol{e}-\boldsymbol{d}}\subset B_{\boldsymbol{0}_{k}}(\K)^{0}$. Further, if we put 
$$
r_{\boldsymbol{a}}(X_{1},\ldots, X_{k})=\sum_{\boldsymbol{n}\in \mathbb{Z}_{\geq 0}^{k}}\left(\prod_{i=1}^{k}\begin{pmatrix}d_{i}\\ n_{i}\end{pmatrix}\chi_{i}(a_{i})^{d_{i}-n_{i}}\right)X^{\boldsymbol{n}}\in B_{\boldsymbol{0}_{k}}(\K)^{0},
$$
we have $\prod_{i=1}^{k}\chi_{i}(x_{i})^{d_{i}}=r_{\boldsymbol{a}}(\chi_{1}(x_{1})-\chi_{1}(a_{1}),\ldots, \chi_{k}(x_{k})-\chi_{k}(a_{k}))$ for every $\boldsymbol{x}\in \boldsymbol{a}\Gamma^{p^{\boldsymbol{m}}}$ where 
$$ \begin{pmatrix}X\\ n\end{pmatrix}=\begin{cases}\frac{X(X-1)\cdots (X-n+1)}{n!}\ &\mathrm{if}\ n\geq 1\\1\ &\mathrm{if}\ n=0\end{cases}$$
for each $n\in \mathbb{Z}_{\geq 0}$. Put $g_{\boldsymbol{a}}=q_{\boldsymbol{a}}r_{\boldsymbol{a}}$. Then, we see that $g_{\boldsymbol{a}}\in B_{\boldsymbol{0}_{k}}(\K)^{0}$ and $f(\boldsymbol{x})=g_{\boldsymbol{a}}(\chi_{1}(x_{1})-\chi_{1}(a_{1}),\ldots, \chi_{k}(x_{k})-\chi_{k}(a_{k}))$ for every $\boldsymbol{x}\in \boldsymbol{a}\Gamma^{p^{\boldsymbol{m}}}$. The uniqueness of $g_{\boldsymbol{a}}$ follows from Lemma \ref{analytic funcuniqueof Brby Zp} immediately.
 \end{proof}

\begin{pro}\label{admissible proj is isom not hisom}
Let $\boldsymbol{b},\boldsymbol{c},\boldsymbol{d},\boldsymbol{e}\in \mathbb{Z}^{k}$ such that $\boldsymbol{c}-\boldsymbol{b}\geq \lfloor \boldsymbol{h}\rfloor$ and $[\boldsymbol{b},\boldsymbol{c}]\subset [\boldsymbol{d},\boldsymbol{e}]$. Then, the restriction map 
\begin{equation}\label{admissible proj is isom not hisomeqst}
\mathcal{D}_{\boldsymbol{h}}^{[\boldsymbol{d},\boldsymbol{e}]}(\Gamma,M)\rightarrow \mathcal{D}_{\boldsymbol{h}}^{[\boldsymbol{b},\boldsymbol{c}]}(\Gamma,M)
\end{equation}
defined in \eqref{proj admissibledistr defnotusis} is an $\mathcal{O}_{\K}[[\Gamma]]\otimes_{\mathcal{O}_{\K}}\K$-module isomorphism. Further, the restriction map in \eqref{admissible proj is isom not hisomeqst} is isometric.
\end{pro}
\begin{proof}
We prove this proposition by induction on $k$.

$\mathbf{Case}$\ $k=1$.

Assume that $k=1$. First, we prove the injectivity of \eqref{admissible proj is isom not hisomeqst}. Let $\mu\in \mathcal{D}_{h}^{[d,e]}(\Gamma,M)$ such that $\mu\vert_{C^{[b,c]}(\Gamma, \mathcal{O}_{\K})}=0$. Let $Z$ be the set of $[r,s]$ with $r,\in \mathbb{Z}$ such that $[b,c]\subset [r,s]\subset [d,e]$ and $\mu\vert_{C^{[r,s]}(\Gamma,\mathcal{O}_{\K})}\neq 0$. Assume that $Z$ is not empty. Let $[r,s]\in Z$ be a minimal element. Since $\mu\vert_{C^{[b,c]}(\Gamma, \mathcal{O}_{\K})}=0$, we have $[b,c]\neq [r,s]$. Then, $b\neq r$ or $c\neq s$. Assume that $c\neq s$. Then, we have $c<s$. Then, by the minimality of $[r,s]$, we have $[r,s-1]\notin Z$. Thus, 
\begin{equation}\label{admisible ristriction isoeq0.5}
\mu\vert_{C^{[r,s-1]}(\Gamma,\mathcal{O}_{\K})}=0.
\end{equation}
Since $\chi_{1}(x)^{s}1_{a\Gamma^{p^{m}}}(x)-(\chi_{1}(x)-\chi_{1}(a))^{s-r}\chi_{1}(x)^{r}1_{a\Gamma^{p^{m}}}(x)\in C^{[r,s-1]}(\Gamma,\mathcal{O}_{\K})$ for each $a\in \Gamma$ and $m\in \mathbb{Z}_{\geq 0}$, by \eqref{admisible ristriction isoeq0.5}, we see that 
\begin{align}\label{admisible ristriction isoeq1}
\begin{split}
v_{M}\left(\int_{a\Gamma^{p^{m}}}\chi_{1}(x)^{s}d\mu\right)
& =v_{M}\left(\int_{a\Gamma^{p^{m}}}(\chi_{1}(x)-\chi_{1}(a))^{s-r}\chi_{1}(x)^{r}d\mu\right)\\
&\geq((s-r)-h)m+v_{h}^{[r,s]}(\mu\vert_{C^{[r,s]}(\Gamma,\mathcal{O}_{\K})}).
\end{split}
\end{align}
We note that by \eqref{projadmissibdistr notisom value}, we have $v_{h}^{[r,s]}(\mu\vert_{C^{[r,s]}(\Gamma,\mathcal{O}_{\K})})>-\infty$. Further, since $c<s$, $r\leq b$ and $c-b\geq \lfloor h\rfloor$, we have
\begin{equation}\label{admisible ristriction isoeq2}
(s-r)-h>0.
\end{equation}
Let $a\in \Gamma$ and $m\in \mathbb{Z}_{\geq 0}$. Since $1_{a\Gamma^{p^{m}}}(x)=\sum_{w\in \Gamma^{p^{m}}\slash \Gamma^{p^{m+n}}}1_{aw\Gamma^{p^{m+n}}}(x)$ for each $n\in \mathbb{Z}_{\geq 0}$, by \eqref{admisible ristriction isoeq1} and \eqref{admisible ristriction isoeq2}, we have 
\begin{align*}
v_{M}\left(\int_{a\Gamma^{p^{m}}}\chi_{1}(x)^{s}d\mu\right)&\geq \lim_{n\rightarrow +\infty}\min_{w\in \Gamma^{p^{m}}\slash \Gamma^{p^{m+n}}}\left\{v_{M}\left(\int_{aw\Gamma^{p^{m+n}}}\chi_{1}(x)^{s}d\mu\right)\right\}\\
&\geq\lim_{n\rightarrow +\infty}((s-r)-h)(m+n)+v_{h}^{[r,s]}(\mu\vert_{C^{[r,s]}(\Gamma,\mathcal{O}_{\K})})=+\infty.
\end{align*}
Hence 
\begin{equation}\label{admissib proism 4.5}
\int_{a\Gamma^{p^{m}}}\chi_{1}(x)^{s}d\mu=0.
\end{equation}
 Since every $f\in C^{[r,s]}(\Gamma,\mathcal{O}_{\K})$ can be written as a linear combination of $1_{a\Gamma^{p^{m}}}(x)\chi_{1}(x)^{j}$ with $a\in \Gamma$, $m\in \mathbb{Z}_{\geq 0}$ and $r\leq j\leq s$, by \eqref{admisible ristriction isoeq0.5} and \eqref{admissib proism 4.5}, we have $\mu\vert_{C^{[r,s]}(\Gamma,\mathcal{O}_{\K})}=0$. This is a contradiction. Then, we have $c=s$. Since $[b,c]\neq [r,c]$, we have $b\neq r$. Then, $b>r$. By the minimality of $[r,c]$, we have $[r+1,c]\notin Z$ . Thus, 
\begin{equation}\label{admisible ristriction isoem3}
\mu\vert_{C^{[r+1,c]}(\Gamma,\mathcal{O}_{\K})}=0.
\end{equation}
Since 
$$
\chi_{1}(x)^{r}1_{a\Gamma^{p^{m}}}(x)-(-\chi_{1}(a))^{-(c-r)}(\chi_{1}(x)-\chi_{1}(a))^{c-r}\chi_{1}(x)^{r}1_{a\Gamma^{p^{m}}}(x)\in C^{[r+1,c]}(\Gamma,\mathcal{O}_{\K})$$
for each $a\in \Gamma$ and $m\in \mathbb{Z}_{\geq 0}$, by \eqref{admisible ristriction isoem3}, we see that 
\begin{align}\label{admisible ristriction isoem4}
\begin{split}
v_{M}\left(\int_{a\Gamma^{p^{m}}}\chi_{1}(x)^{r}d\mu\right)
& =v_{M}\left(\int_{a\Gamma^{p^{m}}}(\chi_{1}(x)-\chi_{1}(a))^{c-r}\chi_{1}(x)^{r}d\mu\right)\\
&\geq ((c-r)-h)m+v_{h}^{[r,c]}(\mu\vert_{C^{[r,c]}(\Gamma,\mathcal{O}_{\K})}).
\end{split}
\end{align}
We note that by \eqref{projadmissibdistr notisom value}, we have $v_{h}^{[r,c]}(\mu\vert_{C^{[r,c]}(\Gamma,\mathcal{O}_{\K})})>-\infty$. Further, since $r<b$ and $c-b\geq \lfloor h\rfloor$, we have
\begin{equation}\label{admisible ristriction isoem5}
(c-r)-h>0.
\end{equation}
Let $a\in \Gamma$ and $m\in \mathbb{Z}_{\geq 0}$. Since $1_{a\Gamma^{p^{m}}}(x)=\sum_{w\in \Gamma^{p^{m}}\slash \Gamma^{p^{m+n}}}1_{aw\Gamma^{p^{m+n}}}(x)$ for each $n\in \mathbb{Z}_{\geq 0}$, by \eqref{admisible ristriction isoem4} and \eqref{admisible ristriction isoem5}, we have 
\begin{align*}
v_{M}\left(\int_{a\Gamma^{p^{m}}}\chi_{1}(x)^{r}d\mu\right)&\geq \lim_{n\rightarrow +\infty}\min_{w\in \Gamma^{p^{m}}\slash \Gamma^{p^{m+n}}}\left\{v_{M}\left(\int_{aw\Gamma^{p^{m+n}}}\chi_{1}(x)^{r}d\mu\right)\right\}\\
&\geq\lim_{n\rightarrow +\infty}((c-r)-h)(m+n)+v_{h}^{[r,c]}(\mu\vert_{C^{[r,c]}(\Gamma,\mathcal{O}_{\K})})=+\infty.
\end{align*}
Hence we have 
\begin{equation}\label{admisible ristriction isoem5.5}
\int_{a\Gamma^{p^{m}}}\chi_{1}(x)^{r}d\mu=0.
\end{equation}
Since every $f\in C^{[r,c]}(\Gamma,\mathcal{O}_{\K})$ can be written as a linear combination of $1_{a\Gamma^{p^{m}}}(x)\chi_{1}(x)^{j}$ with $a\in \Gamma$, $m\in \mathbb{Z}_{\geq 0}$ and $r\leq j\leq c$, by \eqref{admisible ristriction isoem3} and \eqref{admisible ristriction isoem5.5} we have $\mu\vert_{C^{[r,c]}(\Gamma,\mathcal{O}_{\K})}=0$. This is a condtradiction. Then, the restriction map of \eqref{admissible proj is isom not hisomeqst} is injective.

Next, we prove the surjectivity of \eqref{admissible proj is isom not hisomeqst}. For each $m\in \mathbb{Z}_{\geq 0}$, let $R_{m}\subset \Gamma$ be a complete reprensetative set of $\Gamma\slash \Gamma^{p^{m}}$. Let $f\in C^{[d,e]}(\Gamma,\mathcal{O}_{\K})$. By Lemma \ref{locally polynomial is locally analytic lemma}, there exists an $m_{f}\in \mathbb{Z}_{\geq 0}$ such that for each $a\in \Gamma$ there exists a unique element $f^{\prime}_{a}\in B_{0}(\K)^{0}$ such that 
\begin{equation}\label{admisible ristriction isoem6}
f(x)\chi_{1}(x)^{-b}=f^{\prime}_{a}(\chi_{1}(x)-\chi_{1}(a))
\end{equation}
for every $x\in a\Gamma^{p^{m_{f}}}$. Let $y,w\in \Gamma$. By Proposition\ \ref{Br slide prop}, there exists a unique element $(f_{y}^{\prime})_{+(w-y)}\in B_{0}(\K)^{0}$ which satisfies 
$$(f_{y}^{\prime})_{+(w-y)}(z)=f_{y}^{\prime}(z+(\chi_{1}(w)-\chi_{1}(y)))$$
for every $z\in \overline{\K}$ such that $\ord_{p}(z)>0$. Then, we have
\begin{equation}\label{admisible ristriction isoem7}
f_{w}^{\prime}=(f_{y}^{\prime})_{+(w-y)}
\end{equation} 
if $y\Gamma^{p^{m_{f}}}=w\Gamma^{p^{m_{f}}}$. Indeed, by \eqref{admisible ristriction isoem6}, we have
\begin{align*}
&(f_{y}^{\prime})_{+(w-y)}(\chi_{1}(x)-\chi_{1}(w))\\
&=f_{y}^{\prime}(\chi_{1}(x)-\chi_{1}(y))=f(x)\chi_{1}(x)^{-b}\\
&=f_{w}^{\prime}(\chi_{1}(x)-\chi_{1}(w))
\end{align*}
for every $x\in y\Gamma^{p^{m_{f}}}$. Thus, by Lemma \ref{analytic funcuniqueof Brby Zp}, we have \eqref{admisible ristriction isoem7}.

For each $a\in \Gamma$, we put
$$f_{a}^{\prime}=\sum_{n=0}^{+\infty}a_{n,a}X^{n}$$
with $a_{n,a}\in \mathcal{O}_{\K}$. We define
$$S_{m}(f)=\sum_{a\in R_{m}}\int_{a\Gamma^{p^{m}}}\sum_{i=0}^{c-b}a_{i,a}(\chi_{1}(x)-\chi_{1}(a))^{i}\chi_{1}(x)^{b}d\mu.$$
{
If $f\in C^{[b,c]}(\Gamma, \mathcal{O}_{\K})$, $f(x)\chi_{1}(x)^{-b}$ is a locally polynomial function on $\Gamma$ of degree at most $c-b$. Thus, $f_{a}^{\prime}$ in \eqref{admisible ristriction isoem6} is a polynomial of degree at most $c-b$ for each $a\in \Gamma$, and we have 
$$\sum_{i=0}^{c-b}a_{i,a}(\chi_{1}(x)-\chi_{1}(a))^{i}\chi_{1}(x)^{b}=f_{a}^{\prime}(\chi_{1}(x)-\chi_{1}(a))\chi_{1}(a)^{b}=f(x)$$
for each $x\in a\Gamma^{p^{m_{f}}}$ where $m\in \mathbb{Z}_{\geq 0}$ with $m\geq m_{f}$. Therefore, if $f\in C^{[b,c]}(\Gamma, \mathcal{O}_{\K})$,
we have
\begin{equation}\label{admisible ristriction isoem7.5}
S_{m}(f)=\sum_{a\in R_{m}}\int_{a\Gamma^{p^{m}}}f(x)d\mu=\int_{\Gamma}f(x)d\mu
\end{equation}
for  each $m\in \mathbb{Z}_{\geq 0}$ such  that $m\geq m_{f}$ by the definition of $S_{m}(f)$.}

We prove that the sequence $\left(S_{m}(f)\right)_{m\in \mathbb{Z}_{\geq 0}}$ is convergent in $M$. Let $m,n\in \mathbb{Z}_{\geq 0}$ such that $m\geq n$. For each $a\in R_{n}$, we denote by $R_{m,n}^{(a)}$ the subset of $R_{m}$ consisting of $w\in R_{m}$ such that $w\Gamma^{p^{n}}=a\Gamma^{p^{n}}$. Thus, we have $a\Gamma^{p^{n}}=\coprod_{w\in R_{m,n}^{(a)}}w\Gamma^{p^{m}}$ for each $a\in R_{n}$. For each $m\in \mathbb{Z}_{\geq 0}$ such that $m\geq m_{f}$, we have
\begin{align}\label{admisible ristriction isoem8}
\begin{split}
& S_{m+1}(f)-S_{m}(f)  \\
& =\sum_{a\in R_{m}}\sum_{w\in R_{m+1,m}^{(a)}}
\int_{w\Gamma^{p^{m+1}}}\bigg(\sum_{i=0}^{c-b}a_{i,w}(\chi_{1}(x)-\chi_{1}(w))^{i}-\sum_{i=0}^{c-b}a_{i,a} 
(\chi_{1}(x)-\chi_{1}(a))^{i}\bigg)\chi_{1}(x)^{b}d\mu \\
& =\sum_{a\in R_{m}}
\sum_{w\in R_{m+1,m}^{(a)}}\int_{w\Gamma^{p^{m+1}}}\sum_{i=0}^{c-b}\left(a_{i,w}-\sum_{l=i}^{c-b}a_{l,a}
\begin{pmatrix}l\\ i\end{pmatrix}(\chi_{1}(w)-\chi_{1}(a))^{l-i}\right) 
(\chi_{1}(x)-\chi_{1}(w))^{i}\\
&\chi_{1}(x)^{b}d\mu. 
\end{split} 
\end{align}
By Proposition \ref{Br slide prop} and \eqref{admisible ristriction isoem7}, we have
$$a_{i,w}=\sum_{l=i}^{+\infty}a_{l,a}\begin{pmatrix}l\\i\end{pmatrix}(\chi_{1}(w)-\chi_{1}(a))^{l-i}$$
for every $i\in [0,c-b]$, $a\in R_{m}$ and $w\in R_{m+1,m}^{(a)}$. Then, by \eqref{admisible ristriction isoem8}, we see that 
\begin{multline*}
S_{m+1}(f)-S_{m}(f)\\
=\sum_{a\in R_{m}}\sum_{w\in R_{m+1,m}^{(a)}}\sum_{i=0}^{c-b}\int_{w\Gamma^{p^{m+1}}}\sum_{l=c-b+1}^{+\infty}a_{l,a}\begin{pmatrix}l\\i\end{pmatrix}(\chi_{1}(w)-\chi_{1}(a))^{l-i}\\
(\chi_{1}(x)-\chi_{1}(w))^{i}\chi_{1}(x)^{b}d\mu
\end{multline*}
if $m\geq m_{f}$. Since we have 
$$\ord_{p}\left(\sum_{l=c-b+1}^{+\infty}a_{l,a}\begin{pmatrix}l\\i\end{pmatrix}(\chi_{1}(w)-\chi_{1}(a))^{l-i}\right)\geq (c-b+1-i)(m+1)$$ and 
\begin{align*}
v_{M}\left(\int_{w\Gamma^{p^{m+1}}}(\chi_{1}(x)-\chi_{1}(w))^{i}\chi_{1}(x)^{b}d\mu\right)&\geq (i-h)(m+1)+v_{h}^{[b,c]}(\mu)
\end{align*}
 for each $i\in [0,c-b]$, we see that
 \begin{equation}\label{admisible ristriction isoem9}
 v_{M}(S_{m+1}(f)-S_{m}(f))\geq (c-b+1-h)(m+1)+v_{h}^{[b,c]}(\mu).
 \end{equation}
Since $c-b+1-h>0$, by \eqref{admisible ristriction isoem9}, we see that $\{S_{m}(f)\}_{m\in\mathbb{Z}_{\geq 0}}$ is a Cauchy sequence. Therefore, we have a limit $\lim_{m\rightarrow +\infty}S_{m}(f)\in M$. We put
$$\int_{\Gamma}f(x)d\nu=\lim_{m\rightarrow +\infty}S_{m}(f).$$
Then, $\nu$ is an element of $\mathrm{Hom}_{\mathcal{O}_{\K}}(C^{[d,e]}(\Gamma,\mathcal{O}_{\K}),M)$. 

Next, we prove that $\nu$ is in $\mathcal{D}_{h}^{[d,e]}(\Gamma,M)$. Let $a\in \Gamma$, $m\in \mathbb{Z}_{\geq 0}$ and $i\in [d,e]$. For each $w\in \Gamma$, we put 
$$r_{w}(X)=\sum_{n=0}^{+\infty}\begin{pmatrix}d-b\\ n\end{pmatrix}\chi_{1}(w)^{(d-b)-n}X^{n}\in B_{0}(\K)^{0}.
$$
and 
$$s_{w}(X)=\sum_{l=0}^{i-d}(\chi_{1}(w)-\chi_{1}(a))^{i-d-l}X^{l}\in\mathcal{O}_{\K}[X].$$
Then, we have 
\begin{equation}\label{admisible ristriction isoem10}
\chi_{1}(x)^{d-b}=r_{w}(\chi_{1}(x)-\chi_{1}(w))\ \mathrm{and}\ (\chi_{1}(x)-\chi_{1}(a))^{i-d}=s_{w}(\chi_{1}(x)-\chi_{1}(w))
\end{equation}
for every $x\in \Gamma$ where 
$$ \begin{pmatrix} d-b\\ n\end{pmatrix}=\begin{cases}\frac{(d-b)(d-b-1)\cdots (d-b-n+1)}{n!}\ &\mathrm{if}\ n\geq 1 ,\\
1\ &\mathrm{if}\ n=0 .
\end{cases}$$
Put
$$q_{w}(X)=r_{w}(X)s_{w}(X).$$
Then, we have
\begin{equation}\label{admisible ristriction isoem11}
q_{w}(X)=\sum_{n=0}^{+\infty}\sum_{l=0}^{\min\{n,i-d\}}\begin{pmatrix}d-b\\ n-l\end{pmatrix}\chi_{1}(w)^{(d-b)-(n-l)}(\chi_{1}(w)-\chi_{1}(a))^{i-d-l}X^{n}.
\end{equation}
By \eqref{admisible ristriction isoem10}, for each $w\in \Gamma$ such that $w\Gamma^{p^{m}}=a\Gamma^{p^{m}}$, we have
$$1_{a\Gamma^{p^{m}}}(x)(\chi_{1}(x)-\chi_{1}(a))^{i-d}\chi_{1}(x)^{d-b}=q_{w}(\chi_{1}(x)-\chi_{1}(w))$$
for every $x\in w\Gamma^{p^{m}}$. Then, by the definition of $S_{n}(1_{a\Gamma^{p^{m}}}(x)(\chi_{1}(x)-\chi_{1}(a))^{i-d}\chi_{1}(x)^{d})$ with $n\in \mathbb{Z}_{\geq 0}$ and \eqref{admisible ristriction isoem11}, we have
\begin{multline*}
S_{n}(1_{a\Gamma^{p^{m}}}(x)(\chi_{1}(x)-\chi_{1}(a))^{i-d}\chi_{1}(x)^{d})\\
=\sum_{w\in R_{n,m}^{(a)}}\int_{w\Gamma^{p^{n}}}\sum_{j=0}^{c-b}\sum_{l=0}^{\min\{j,i-d\}}\begin{pmatrix}d-b\\ j-l\end{pmatrix}\chi_{1}(w)^{(d-b)-(j-l)}(\chi_{1}(w)-\chi_{1}(a))^{i-d-l}\\
(\chi_{1}(x)-\chi_{1}(w))^{j}\chi_{1}(x)^{b}d\mu
\end{multline*}
for each $n\in \mathbb{Z}_{\geq 0}$ such  that $n\geq m$. Therefore, we see that
\begin{align*}
& v_{M}\left(S_{n}(1_{a\Gamma^{p^{m}}}(x)(\chi_{1}(x)-\chi_{1}(a))^{i-d}\chi_{1}(x)^{d})\right)\\
& \geq \inf_{\substack{0\leq j\leq c-b\\ 0\leq l\leq\min\{j,i-d\}}}\left\{(i-d-l)m+(j-h)m\right\}+v_{h}^{[b,c]}(\mu)\\
& \geq \inf_{0\leq j\leq c-b}\left\{(i-d-j)m+(j-h)m\right\}+v_{h}^{[b,c]}(\mu)\\
& =(i-d-h)m+v_{h}^{[b,c]}(\mu).
\end{align*}
Therefore, since we have
$$\int_{a\Gamma^{p^{m}}}(\chi_{1}(x)-\chi_{1}(a))^{i-d}\chi_{1}(x)^{d}d\nu=\lim_{n\rightarrow +\infty}S_{n}(1_{a\Gamma^{p^{m}}}(x)(\chi_{1}(x)-\chi_{1}(a))^{i-d}\chi_{1}(x)^{d}),$$
we see that
$$v_{M}\left(\int_{a\Gamma^{p^{m}}}(\chi_{1}(x)-\chi_{1}(a))^{i-d}\chi_{1}(x)^{d}d\nu\right)+(h-(i-d))m\geq v_{h}^{[b,c]}(\mu)$$
for every $m\in \mathbb{Z}_{\geq 0}$, $a\in \Gamma$ and $i\in [d,e]$. Thus, we have
\begin{equation}\label{admisible ristriction isoem12}
v_{h}^{[d,e]}(\nu)\geq v_{h}^{[b,c]}(\mu)
\end{equation}
and we see that $\nu\in D_{h}^{[d,e]}(\Gamma, M)$. Further, by \eqref{admisible ristriction isoem7.5}, we have $\nu\vert_{C^{[b,c]}(\Gamma,\mathcal{O}_{\K})}=\mu$. Then, the restriction map in \eqref{admissible proj is isom not hisomeqst} is surjective. Further, by \eqref{projadmissibdistr notisom value} and \eqref{admisible ristriction isoem12}, the restriction map in \eqref{admissible proj is isom not hisomeqst} is isometric.

$\mathbf{Case}$\ $k>1$. 

We assume that $k>1$. We denote by $\mathrm{res}^{[\boldsymbol{d},\boldsymbol{e}]}_{[\boldsymbol{b},\boldsymbol{c}]}: \mathcal{D}_{\boldsymbol{h}}^{[\boldsymbol{d},\boldsymbol{e}]}(\Gamma,M)\rightarrow \mathcal{D}_{\boldsymbol{h}}^{[\boldsymbol{b},\boldsymbol{c}]}(\Gamma,M)$ the restriction map in \eqref{proj admissibledistr defnotusis}. Put $\boldsymbol{b}^{\prime}=(b_{1},\ldots, b_{k-1})$, $\boldsymbol{c}^{\prime}=(c_{1},\ldots, c_{k-1})$, $\boldsymbol{d}^{\prime}=(d_{1},\ldots, d_{k-1})$, $\boldsymbol{e}^{\prime}=(e_{1},\ldots, e_{k-1})$, $\boldsymbol{h}^{\prime}=(h_{1},\ldots, h_{k-1})$ and $\Gamma^{\prime}=\Gamma_{1}\times\cdots\times \Gamma_{k-1}$. Then, by induction on $k$, the restriction map $\mathrm{res}^{[\boldsymbol{d}^{\prime},\boldsymbol{e}^{\prime}]}_{[\boldsymbol{b}^{\prime},\boldsymbol{c}^{\prime}]}: \mathcal{D}_{\boldsymbol{h}^{\prime}}^{[\boldsymbol{d}^{\prime},\boldsymbol{e}^{\prime}]}(\Gamma^{\prime},M)\rightarrow \mathcal{D}_{\boldsymbol{h}^{\prime}}^{[\boldsymbol{b}^{\prime},\boldsymbol{c}^{\prime}]}(\Gamma^{\prime},M)$ is an isometric isomorphism. Thus, we can define the following isometric isomorphism:
$$\psi: \mathcal{D}_{h_{k}}^{[d_{k},e_{k}]}(\Gamma_{k},\mathcal{D}_{\boldsymbol{h}^{\prime}}^{[\boldsymbol{d}^{\prime},\boldsymbol{e}^{\prime}]}(\Gamma^{\prime},M))\rightarrow \mathcal{D}_{h_{k}}^{[d_{k},e_{k}]}(\Gamma_{k},\mathcal{D}_{\boldsymbol{h}^{\prime}}^{[\boldsymbol{b}^{\prime},\boldsymbol{c}^{\prime}]}(\Gamma^{\prime},M)),\ \mu\mapsto \mathrm{res}^{[\boldsymbol{d}^{\prime},\boldsymbol{e}^{\prime}]}_{[\boldsymbol{b}^{\prime},\boldsymbol{c}^{\prime}]}\circ \mu.$$

Let 
$$\mathrm{res}^{\prime}: \mathcal{D}_{h_{k}}^{[d_{k},e_{k}]}(\Gamma_{k},\mathcal{D}_{\boldsymbol{h}^{\prime}}^{[\boldsymbol{b}^{\prime},\boldsymbol{c}^{\prime}]}(\Gamma^{\prime},M))\rightarrow \mathcal{D}_{h_{k}}^{[b_{k},c_{k}]}(\Gamma_{k},\mathcal{D}_{\boldsymbol{h}^{\prime}}^{[\boldsymbol{b}^{\prime},\boldsymbol{c}^{\prime}]}(\Gamma^{\prime},M)),\ \mu\mapsto \mu\vert_{C^{[b_{k},c_{k}]}(\Gamma_{k},\mathcal{O}_{\K})}$$
be the restriction map. By the result in the case $k=1$, $\mathrm{res}^{\prime}$ is an isometric isomorphism. We have the following commutative diagram:
$$
\xymatrix{
\mathcal{D}_{\boldsymbol{h}}^{[\boldsymbol{d},\boldsymbol{e}]}(\Gamma,M)\ar[d]^{\mathrm{res}^{[\boldsymbol{d},\boldsymbol{e}]}_{[\boldsymbol{b},\boldsymbol{c}]}}\ar[r]^{\simeq\ \ \ \ \ \ \ \ \ \ \ }&\mathcal{D}_{h_{k}}^{[d_{k},e_{k}]}(\Gamma_{k},\mathcal{D}_{\boldsymbol{h}^{\prime}}^{[\boldsymbol{d}^{\prime},\boldsymbol{e}^{\prime}]}(\Gamma^{\prime},M))\ar[d]^{\mathrm{res}^{\prime}\circ \psi}\\
\mathcal{D}_{\boldsymbol{h}}^{[\boldsymbol{b},\boldsymbol{c}]}(\Gamma,M)\ar[r]^{\simeq \ \ \ \ \ \ \ \ \ \ \ }&\mathcal{D}_{h_{k}}^{[b_{k},c_{k}]}(\Gamma_{k},\mathcal{D}_{\boldsymbol{h}^{\prime}}^{[\boldsymbol{b}^{\prime},\boldsymbol{c}^{\prime}]}(\Gamma^{\prime},M))
}
$$
where the two horizontal maps are isometric isomorphisms defined in Proposition \ref{for induction admisible admissible}. Since the two horizotal maps and $\mathrm{res}^{\prime}\circ \psi$ are isometric isomorphisms, we see that $\mathrm{res}^{[\boldsymbol{d},\boldsymbol{e}]}_{[\boldsymbol{b},\boldsymbol{c}]}$ is an isometric isomorphism.
\end{proof}

\section{One-variable power series of logarithmic order with values in Banach spaces}\label{section: one-variable} 
In this section, we generalize the classical theory of one-variable admissible distributions with values in a $p$-adic field obtained in \cite{amicevelu} and \cite{vishik1976} to the theory of one-variable admissible distributions with values in a Banach space. 
The results obtained in this section will be used to prove our main results in \S \ref{sc:ordinary}. In this subsection, we fix a $\K$-Banach space $(M,v_{M})$. Let $r\in \mathbb{Q}$. We define a subset $B_{r}^{\mathrm{md}}(M)\subset B_{r}(M)$ to be 
\begin{equation}\label{definition of Brmd(M)}
B_{r}^{\mathrm{md}}(M)=\left\{f=(m_{n})_{n=0}^{+\infty}\in B_{r}(M)\ \bigg\vert \  {}^{\exists}n_{0}\in \mathbb{Z}_{\geq 0}\ \mathrm{such\ that}\ v_{r}(f)=v_{M}(m_{n_{0}})+rn_{0}\right\}.
\end{equation}
We remark that $B_{r}^{\mathrm{md}}(M)=B_{r}(M)$ if and only if $x\notin \overline{(x,\infty)\cap v_{M}(M\backslash \{0\})}$ for every $x\in \mathbb{R}$. 
Especially, we have $B_{r}^{\mathrm{md}}(M)=B_{r}(M)$ for any  $r\in \mathbb{Q}$ if $v_{M}(M\backslash \{0\})$ is a discrete closed subset. 
As an example of $f\in B_{0}(\mathbb{C}_{p})\backslash B_{0}^{\mathrm{md}}(\mathbb{C}_{p})$, we can take $f=\sum_{n=1}^{+\infty}p^{\frac{1}{n}}X^{n}\in B_{0}(\mathbb{C}_{p})$. For each $f=(m_{n})_{n=0}^{+\infty}\in B_{r}^{\mathrm{md}}(M)$, we put 
\begin{equation}
d_{r}(f)=\begin{cases}\min\{n\in \mathbb{Z}_{\geq 0}\vert v_{r}(f)=v_{M}(m_{n})+rn\}\ &\mathrm{if}\ f\neq 0,\\
-\infty\ &\mathrm{if}\ f=0.
\end{cases}
\end{equation}
\begin{pro}\label{multiplication of B_{r}(K) and B_{r}(M)}
Let $r\in \mathbb{Q}$. If $f\in B_{r}^{\mathrm{md}}(\K)$ and $g\in B_{r}^{\mathrm{md}}(M)$, we see that $fg\in B_{r}^{\mathrm{md}}(M)$ and
$$
d_{r}(fg)=d_{r}(f)+d_{r}(g).
$$
 \end{pro}
 \begin{proof}
We may assume that $f\neq 0$ and $g\neq 0$. Put $f=(a_{n})_{n\in \mathbb{Z}_{\geq 0}}$, $g=(m_{n})_{n\in \mathbb{Z}_{\geq 0}}$ and $d=d_{r}(f)+d_{r}(g)$. We see that $v_{M}(a_{l_{1}}m_{l_{2}})+rd>v_{M}(a_{d_{r}(f)}m_{d_{r}(g)})+rd$ for every $(l_{1},l_{2})\in \mathbb{Z}_{\geq 0}^{2}$ such that $l_{l}+l_{2}=d$ and $(l_{1},l_{2})\neq (d_{r}(f),d_{r}(g))$. Then, we have
\begin{align*}
v_{M}(\sum_{\substack{l_{1}+l_{2}=d\\l_{1},l_{2}\geq 0}}a_{l_{1}}m_{l_{2}})+rd=v_{M}(a_{d_{r}(f)}m_{d_{r}(g)})+rd=v_{r}(f)+v_{r}(g).
\end{align*}
By Proposition \ref{mult Br prodct equlity}, we have $v_{r}(fg)=v_{r}(f)+v_{r}(g)$. Then, $fg\in B_{r}^{\mathrm{md}}(M)$. We see that $v_{M}(a_{l_{1}}m_{l_{2}})+r(l_{1}+l_{2})>v_{r}(f)+v_{r}(g)$ unless $l_{1}\geq d_{r}(f)$ and $l_{2}\geq d_{r}(g)$. Then, we have $d_{r}(fg)=d_{r}(f)+d_{r}(g)$. We complete the proof.
\end{proof}
We prove the Weiestrass division theorem on $B_{r}(M)$.
\begin{pro}\label{Weiestrass on Banach space}
Let $r\in \mathbb{Q}$ and $f\in B_{r}^{\mathrm{md}}(\K)\backslash\{0\}$ with $d_{r}(f)=s$. For each $g\in B_{r}(M)$, there exists a unique pair $(q,t)\in B_{r}(M)\times M[X]$ which satisfies $g=fq+t$ and $\deg t<s$. Further, we have
$$v_{r}(g)=\min\{v_{r}(f)+v_{r}(q),v_{r}(t)\}.$$
\end{pro}
\begin{proof}
First, we prove the uniqueness of $q$ and $t$. For this, it suffices to show that $q=t=0$ 
under the assumption that $fq+t=0$. By contradiction, we assume that $q\neq 0$. Then, we see that $fq=-t\in B_{r}^{\mathrm{md}}(M)\backslash \{0\}$ and $d_{r}(fq)<s$. We put $f=f_{1}+X^{s}f_{2}$ with $f_{1}\in \K[X]$ and $ f_{2}\in B_{r}(\K)$ such that $\deg f_{1}<s$. Since $v_{r}(f_{1})>v_{r}(f)$, we have $v_{r}(f_{1}q)>v_{r}(fq)$. If we put $fq=(m_{n})_{n\in \mathbb{Z}_{\geq 0}}$ and $f_{1}q=(m_{n}^{\prime})_{n\in\mathbb{Z}_{\geq 0}}$, we have $m_{n}=m_{n}^{\prime}$ for each $n\in \mathbb{Z}_{\geq 0}$ such that $n<s$. Therefore, by $v_{r}(f_{1}q)>v_{r}(fq)$, we see that 
$$v_{M}(m_{n})+rn\geq v_{r}(f_{1}q)>v_{r}(fq)$$
for each $n\in \mathbb{Z}_{\geq 0}$ such that $n<s$, which contradicts to $d_{r}(fq)<s$. Thus, $q=t=0$. 

Next we prove the existence of $q,t$ and the estimate $v_{r}(g)=\min\{v_{r}(f)+v_{r}(q),v_{r}(t)\}$. As a first step, we assume that $r\in \ord_{p}(\K^{\times})$. Then, without loss of generality, we can assume that $v_{r}(f)=0$. Let us define an operator 
$$\tau_{s}: B_{r}(M)\rightarrow B_{r}(M)$$
to be $\tau_{s}((m_{n})_{n\in\mathbb{Z}_{\geq 0}})=p^{rs}(m_{n+s})_{n\in \mathbb{Z}_{\geq 0}}$. It is easy to see that $\tau_{s}$ is well-defined and $v_{\mathfrak{L}}(\tau_{s})\geq 0$. Clearly $\tau_{s}$ satisfies
\begin{enumerate}
\item $\tau_{s}((p^{-r}X)^{s}l)=l$ for each $l\in B_{r}(M)$,
\item $\tau_{s}(l)=0\Leftrightarrow l\in M[X]$ with $\deg l<s$.
\end{enumerate}
 We can write $f=bh(p^{-r}X)+(p^{-r}X)^{s}u(p^{-r}X)$, where $b\in \mathcal{O}_{\K}$ such that $\ord_{p}(b)>0$, $h(Y)\in \mathcal{O}_{\K}[Y]$ with $\deg h(Y)<s$ and $u(Y)\in \mathcal{O}_{\K}[[Y]]^{\times}$. Let 
$$q=\frac{1}{u(p^{-r}X)}\sum_{j=0}^{+\infty}(-1)^{j}b^{j}\left(\tau_{s}\circ \frac{h(p^{-r}X)}{u(p^{-r}X)}\right)^{j}\circ \tau_{s}(g).$$
Here, for example, 
$$\left(\tau_{s}\circ \frac{h(p^{-r}X)}{u(p^{-r}X)}\right)^{2}\circ \tau_{s}(g)=\tau_{s}\left(\frac{h(p^{-r}X)}{u(p^{-r}X)}\tau_{s}\left(\frac{h(p^{-r}X)}{u(p^{-r}X)}\tau_{s}(g)\right)\right).$$
Then, the sum $q$ is well-defined in $B_{r}(M)$ and we have $v_{r}(q)\geq v_{r}(g)$. Since $fq=bh(p^{-r}X)q+(p^{-r}X)^{s}u(p^{-r}X)q$, we have
$$\tau_{s}(fq)=b\tau_{s}(h(p^{-r}X)q)+u(p^{-r}X)q.$$
But 
\begin{align*}
b\tau_{s}(h(p^{-r}X)q)&=b\left(\tau_{s}\circ \frac{h(p^{-r}X)}{u(p^{-r}X)}\circ \sum_{j=0}^{+\infty}(-1)^{j}b^{j}\left(\tau_{s}\circ \frac{h(p^{-r}X)}{u(p^{-r}X)}\right)^{j}\circ \tau_{s}(g)\right)\\
&=\tau_{s}(g)-u(p^{-r}X)q.
\end{align*} 
Therefore, $\tau_{s}(fq)=\tau_{s}(g)$. If we put $t=g-fq$, we have $t\in M[X]$ and $\deg t<s$. Since $v_{r}(q)\geq v_{r}(g)$ and $v_{r}(f)=0$, we see that $v_{r}(t)\geq v_{r}(g)$ and we have $\min\{v_{r}(q),v_{r}(t)\}\geq v_{r}(g)$. On the other hand, we have $v_{r}(g)\geq \min\{v_{r}(fq),v_{r}(t)\}=\min\{v_{r}(q),v_{r}(t)\}$. Then, we conclude that $v_{r}(g)=\min\{v_{r}(q),v_{r}(t)\}$.

As a second step, we take a general $r\in \mathbb{Q}$. Let $\mathcal{L}\slash \K$ be a finite Galois extension such that $r\in \ord_{p}(\mathcal{L}^{\times})$. By the result of the first step, there exists a unique pair $(q,t)\in B_{r}(M_{\mathcal{L}})\times M_{\mathcal{L}}[X]$ such that $g=fq+t$ and $\deg t<s$. In addition, we have $v_{r}(g)=\min\{v_{r}(f)+v_{r}(q),v_{r}(t)\}$. We denote by $G(\mathcal{L}\slash \K)$ the Galois group of $\mathcal{L}\slash \mathcal{K}$. We define an action of $G(\mathcal{L}\slash \K)$ on $M_{\mathcal{L}}$ to be $\sigma(x)=\sum_{i=1}^{d}m_{i}\otimes_{\K}\sigma(y_{i})$ for each $\sigma\in G(\mathcal{L}\slash \K)$ and $x=\sum_{i=1}^{d}m_{i}\otimes_{\K}y_{i}\in M_{\mathcal{L}}$. In addition, we put $\sigma(l)=(\sigma(m_{n}))_{n\in\mathbb{Z}_{\geq 0}}\in B_{r}(M_{\mathcal{L}})$ for each $l=(m_{n})_{n\in \mathbb{Z}_{\geq 0}}\in B_{r}(M_{\mathcal{L}})$. For each $\sigma\in G(\mathcal{L}\slash \K)$, we have
$$
g=\sigma(g)
=f\sigma(q)+\sigma(t).
$$
By the uniqueness of $q$ and $t$, we have $\sigma(q)=q$ and $\sigma(t)=t$. That is, $q\in B_{r}(M)$ and $t\in M[X]$. Since the natural map $M\rightarrow M_{\mathcal{L}}$ defined by $x\mapsto x\otimes_{\K}1$ for $x\in M$ is an isometry, we see that $v_{r}(g)=\min\{v_{r}(f)+v_{r}(q),v_{r}(t)\}$.
\end{proof}
Next, we prove the Weiestrass preparation theorem on $B_{r}(\K)$.
\begin{pro}\label{Weiestras preparation theorem}
Let $r\in \mathbb{Q}$ and $f\in B_{r}^{\mathrm{md}}(\K)\backslash\{0\}$ with $d_{r}(f)=s$. Then, $f$ can be written uniquely as $f=gu$ where $u\in B_{r}(\K)^{\times}$ with $u-1\in XB_{r}(\K)$ and $g\in \K[X]$ with $\deg g=d_{r}(g)=s$. In addition, we have $v_{r}(f)=v_{r}(g)$ and $v_{r}(u)=0$.
\end{pro}
\begin{proof}
First, we prove the uniqueness of $(g,u)$. We write $f=g_{i}u_{i}$, where $u_{i}\in B_{r}(\K)^{\times}$ with $u_{i}-1\in XB_{r}(\K)$ and $g_{i}\in \K[X]$ with $\deg g_{i}=d_{r}(g_{i})=s$ for $i=1,2$.  Put $g_{i}=b_{i}X^{s}-h_{i}$, where $h_{i}\in \K[X]$ with $\deg h_{i}<s$ and $b_{i}\in \K^{\times}$.  We have $X^{s}=b_{i}^{-1}(fu_{i}^{-1}+h_{i})$. The uniqueness of $\mathrm{Proposition\ \ref{Weiestrass on Banach space}}$ implies that $(b_{1}u_{1},b_{1}^{-1}h_{1})=(b_{2}u_{2},b_{2}^{-1}h_{2})$ .  Since $u_{i}-1\in XB_{r}(\K)$, we have $b_{1}=b_{2}$, $u_{1}=u_{2}$ and $h_{1}=h_{2}$. Thus, we see that $(g_{1},u_{1})=(g_{2},u_{2})$.

Next, we prove that $f$ can be written as $f=gu$ and we have $v_{r}(f)=v_{r}(g)$ and $v_{r}(u)=0$. As a first step, we assume that $r\in \ord_{p}(\K^{\times})$. Then, without loss of generality, we can assume that $v_{r}(f)=0$. By $\mathrm{Proposition\ \ref{Weiestrass on Banach space}}$, there exists a unique pair $(q,l)\in B_{r}(\K)\times \K[X]$ such that $(p^{-r}X)^{s}=fq+l$ and $\deg l<s$. In addition, we have $\min\{v_{r}(q),v_{r}(l)\}=0$. If $v_{r}(l)=0$, we see that $s=d_{r}((p^{-r}X)^{s})=d_{r}(fq+l)<s$. This is a contradiction. Then, we have $v_{r}(l)>0$ and $v_{r}(q)=0$. Let $q_{0}\in \mathcal{O}_{\K}$ be the constant term of $q$. Since $d_{r}(f)=s$, we have $q_{0}\in \mathcal{O}_{\K}^{\times}$. We put $B_{r}(\K)^{0}=\{t\in B_{r}(\K)\vert v_{r}(t)\geq 0\}$.  Then, $q$ is a unit in $B_{r}(\K)^{0}$. We put $u=q_{0}q^{-1}\in 1+XB_{r}(\K)$ and $g=q_{0}^{-1}((p^{-r}X)^{s}-l)\in \K[X]$. Then, we have $f=gu$ and $d_{r}(g)=\deg g=s$. Further, by construction, it is easy to see that $v_{r}(g)=v_{r}(u)=0$.

As a second step, we take a general $r\in \mathbb{Q}$.  Let $\mathcal{L}\slash \K$ be a finite Galois extension such that $r\in \ord_{p}(\mathcal{L})$. By the result of the first step, $f$ can be written in the form $f=gu$ uniquely, where $u\in B_{r}(\mathcal{L})^{\times}$ with $u-1\in XB_{r}(\mathcal{L})$ and $g\in \mathcal{L}[X]$ with $\deg g=d_{r}(g)=s$. In addition, we have $v_{r}(f)=v_{r}(g)$ and $v_{r}(u)=0$. We denote by $G(\mathcal{L}\slash \K)$ the Galois group of $\mathcal{L}\slash \mathcal{K}$. We define an action of $G(\mathcal{L}\slash \K)$ on $B_{r}(\mathcal{L})$ to be
$$\sigma(h)=\sum_{n=0}^{+\infty}\sigma(a_{n})X^{n}$$
for each $h=\sum_{n=0}^{+\infty}a_{n}X^{n}\in B_{r}(\mathcal{L})$. For each $\sigma\in G(\mathcal{L}\slash \K)$, we have $f=\sigma(g)\sigma(u)$. By the uniqueness of $(g,u)$, we have $g=\sigma(g)$ and $u=\sigma(u)$ for each $\sigma\in G(\mathcal{L}\slash \K)$. That is, $g\in \K[X]$ and $u\in B_{r}(\K)$. Since $\sigma(u^{-1})=\sigma(u)^{-1}=u^{-1}$ for each $\sigma\in G(\mathcal{L}\slash \K)$, we see that $u\in B_{r}(\K)^{\times}$. We complete the proof.
\end{proof}
\begin{cor}\label{leading degree and roots of poly}
Let $r\in \mathbb{Q}$ and $f\in B_{r}^{\mathrm{md}}(\K)\backslash \{0\}$. Then, $d_{r}(f)$ is equal to the number of roots of $f$ in the set 
$\{x\in\overline{\K}\ \vert \ \ord_{p}(x)> r\}$ with multiplicity.
\end{cor}
\begin{proof}
Put $s=d_{r}(f)$. By $\mathrm{Proposition\ \ref{Weiestras preparation theorem}}$, $f$ can be written in the form $f=gu$, where $u\in B_{r}(\K)^{\times}$ and $g\in \K[X]$ with $d_{r}(g)=\deg g=s$. By replacing $\K$ with a finite extension of $\K$, we can assume that we have a factorization $g=c(X-a_{1})\cdots (X-a_{s})$ with $a_{1},\ldots, a_{s}\in \K$ and $c\in \K^{\times}$. Then, we see that
$$s=d_{r}(g)=\sum_{i=1}^{s}d_{r}(X-a_{i}).$$
Since $d_{r}(X-a_{i})$ is equal to $1$ (resp. $0$) if $\ord_{p}(a_{i})> r$ (resp. otherwise), 
all $a_{i}$ ($i=1 ,\ldots ,s$) must satisfy $\ord_{p}(a_{i})>r$.
\end{proof}
\begin{cor}\label{cor of remainder theorem}
Let $r\in \mathbb{Q}$ and $f\in \K[X]\backslash \{0\}$ a non-zero separable polynomial with $d_{r}(f)=\deg f$. For each $g\in B_{r}(M)$, the following two conditions are equivalent:
\begin{enumerate}
\item There exists a unique $q\in B_{r}(M)$ such that $g=fq$.\label{cor of remainder theorem1}
\item For every root $a\in \overline{\K}$ of $f$, we have $g(a)=0$ in $M_{\K(a)}$.\label{cor of remainder theorem2}
\end{enumerate}
\end{cor}
\begin{proof}
We see that  \eqref{cor of remainder theorem1} implies \eqref{cor of remainder theorem2} easily. Then, we prove that \eqref{cor of remainder theorem2} implies \eqref{cor of remainder theorem1}. By $\mathrm{Proposition\ \ref{Weiestrass on Banach space}}$, there exists a unique pair $(q,t)\in B_{r}(M)\times M[X]$ such that $g=fq+t$ and $\deg t<\deg f$. Then it suffices to prove the following property:
\begin{enumerate}
\item[$(*)$] Let $t\in M[X]$ with $\deg t<\deg f$. If $t(a)=0$ in $M_{\K(a)}$ for all roots $a\in \overline{\K}$ of $f$, then $t=0$. 
\end{enumerate}
By replacing $\K$ with a finite extension of $\K$, we can assume that $\K$ contains all roots of $f$. Let $a_{1},\ldots, a_{s}\in \K$ be the roots of $f$ with $s=\deg f$. Put $t=(t_{n})_{n\in\mathbb{Z}_{\geq 0}}\in M[X]$. Since $\deg t<s$, we have $t_{n}=0$ if $n\geq s$. We define a square matrix $A=(a_{i,j})_{1\leq i,j\leq s}$ of order $s$ to be $a_{i,j}=a_{i}^{j-1}$ for each $1\leq i,j\leq s$. The matrix $A$ is invertible since $f$ is separable. By the assumption that $t(a_{i})=0$ for each $1\leq i\leq s$, we have $A{}^{t}(t_{0},\ldots, t_{s-1})={}^{t}(0,\ldots, 0)$. Then, $(t_{0},\ldots, t_{s-1})=(0,\ldots, 0)$ and we conclude that $t=0$.
\end{proof}
\begin{pro}\label{supectral norm gauss norm}
Let $r\in \mathbb{Q}$ and $f\in B_{r}(M)$. Then, we have 
\begin{align*}
v_{r}(f)&=\inf_{\substack{b\in \overline{\K}\\ \ord_{p}(b)>r}}\{v_{M_{\K(b)}}(f(b))\}.
\end{align*}
\end{pro}
\begin{proof}
Let $f=(m_{n})_{n\in\mathbb{Z}_{\geq 0}}\in B_{r}(M)$ with $m_{n}\in M$ for every non-negative integer $n$. 
By \eqref{equation:definition_of_v_r}, we have $v_{r}(f) = \inf\{v_{M}(m_{n})+rn\}_{n\in \mathbb{Z}_{\geq 0}}$. 
Hence, for every $b\in \overline{\K}$ such that $\ord_{p}(b)>r$, we have 
$$
v_{r}(f) \leq \inf_{n\in \mathbb{Z}_{\geq 0}}\{v_{M}(m_{n})+n\ord_{p}(b)\} \leq v_{M_{\K(b)}}\left(\sum_{n=0}^{+\infty}m_{n}\otimes_{\K}b^{n}\right)
= v_{M_{\K(b)}}(f(b)). 
$$ 
Thus we obtained the following inequality: 
\begin{equation}\label{onevariable eq spectraul norm gauss norm eqaf1}
v_{r}(f) \leq \inf_{\substack{b\in \overline{\K}\\ \ord_{p}(b)>r}}\{v_{M_{\K(b)}}(f(b))\}. 
\end{equation}
%
Let us prove the opposite inequality. 
We assume that $f\in B_{r}^{\mathrm{md}}(M)\backslash\{0\}$ and put $s=d_{r}(f)$.  There exists a real number $\delta>0$ such that for every $t\in (r,r+\delta)\cap \mathbb{Q}$, we have $v_{M_{\K(p^{t})}}(m_{n}p^{tn})>v_{M_{\K(p^{t})}}(m_{s}p^{ts})$ for every integer $n$ satisfying $0\leq n <s$. On the other hand, for every 
integer $n$ satisfying $s<n$ and for every $t\in (r,r+\delta)\cap \mathbb{Q}$, we have 
\begin{align*}
v_{M_{\K(p^{t})}}(m_{n}p^{tn})&=(v_{M}(m_{n})+rn)+(t-r)n\\
&\geq (v_{M}(m_{s})+rs)+(t-r)(s+1)\\
&=v_{M_{\K(p^{t})}}(m_{s}p^{ts})+(t-r).
\end{align*}
Therefore, we see that $v_{M_{\K(p^{t})}}(f(p^{t}))=v_{M_{\K(p^{t})}}(m_{s}p^{ts})=v_{M}(m_{s})+ts$ for every $t\in (r,r+\delta)\cap \mathbb{Q}$. Then, we have $\displaystyle{\inf_{\substack{b\in \overline{\K}\\ \ord_{p}(b)>r}}}\{v_{M_{\K(b)}}(f(b))\}\leq \inf_{t\in (r,r+\delta)\cap \mathbb{Q}}\{v_{M_{\K(p^{t})}}(f(p^{t}))\}=v_{M}(m_{s})+rs=v_{r}(f)$. By \eqref{onevariable eq spectraul norm gauss norm eqaf1}, we conclude that $v_{r}(f)=\displaystyle{\inf_{\substack{b\in \overline{\K}\\ \ord_{p}(b)>r}}}\{v_{M_{\K(b)}}(f(b))\}$.

Next we take a general $f\in B_{r}(M)$. We can assume that $f\neq 0$. For each $\epsilon>0$, there exists a $\delta>0$ such that $v_{r}(f)\leq v_{t}(f)<v_{r}(f)+\epsilon$ for every $t\in [r,r+\delta)\cap \mathbb{Q}$. Let $t\in (r,r+\delta)\cap \mathbb{Q}$. Since $f\in B_{t}^{\mathrm{md}}(M)$, by the result in the case of $B_{t}^{\mathrm{md}}(M)$, we see that 
$$
\inf_{\substack{b\in \overline{\K}\\ \ord_{p}(b)>r}}\{v_{M_{\K(b)}}(f(b))\} \leq \inf_{\substack{b\in \overline{\K}\\ \ord_{p}(b)>t}}\{v_{M_{\K(b)}}(f(b))\}
=v_{t}(f)<v_{r}(f)+\epsilon.
$$ 
Therefore, we have $\inf_{\substack{b\in \overline{\K}\\ \ord_{p}(b)>r}}\{v_{M_{\K(b)}}(f(b))\} \leq v_{r}(f)$. By \eqref{onevariable eq spectraul norm gauss norm eqaf1}, we conclude that $v_{r}(f)=\displaystyle{\inf_{\substack{b\in \overline{\K}\\ \ord_{p}(b)>r}}}\{v_{M_{\K(b)}}(f(b))\}$ for each general $f\in B_{r}(M)$.
\end{proof}
We have $B_{r}(M)\subset B_{r^{\prime}}(M)$ for each $r,r^{\prime}\in \mathbb{Q}$ such that $r<r^{\prime}$. We define $B_{+}(M)=\cap_{r\in \mathbb{Q}_{>0}}B_{r}(M)\subset B_{0}(M)$. Let $f=(m_{n})_{n\in\mathbb{Z}_{\geq 0}}\in B_{+}(M)\backslash \{0\}$. We define
\begin{align}
\begin{split}
m_{f}(t)&=v_{t}(f):\mathbb{R}_{>0}\rightarrow \mathbb{R},\\
n_{f}(t)&=d_{t}(f): \mathbb{R}_{>0}\rightarrow \mathbb{Z}_{\geq 0},
\end{split}
\end{align}
where $v_{t}(f)=\mathrm{inf}\{v_{M}(m_{n})+tn\}_{n\in\mathbb{Z}_{\geq0}}$ and $d_{t}(f)=\min\{n_{0}\in \mathbb{Z}_{\geq 0}\vert v_{t}(f)=v_{M}(m_{n_{0}})+tn_{0}\}$ for each $t\in \mathbb{R}_{>0}$. By definition, $m_{f}(t)$ is monotonically increasing and we have
\begin{equation}\label{trivial equation of M_{m}(t)}
m_{f}(t)=v_{M}(m_{n_{f}(t)})+tn_{f}(t).
\end{equation}
\begin{pro}\label{leading deree is left continuous}
Let $f\in B_{+}(M)\backslash \{0\}$. Then, the function $n_{f}(t)$ is monotonically decreasing and right continuous.
\end{pro}
\begin{proof}
Put $f=(m_{n})_{n\in\mathbb{Z}_{\geq 0}}$ and $n_{t}=d_{t}(f)$ for $t\in (0,+\infty)$. We first prove that the function $t\mapsto n_{t}$ with $t\in (0,+\infty)$ is monotonically decreasing. By contradiction, we suppose that there exist $t_{1},t_{2}\in(0,+\infty)$ such that $t_{1}<t_{2}$ and $n_{t_{1}}<n_{t_{2}}$. We put 
\begin{align*}
g(t)=v_{M}(m_{n_{t_{1}}})-v_{M}(m_{n_{t_{2}}})+t(n_{t_{1}}-n_{t_{2}})
\end{align*}
for $t\in (0,+\infty)$. Since $n_{t_{1}}<n_{t_{2}}$, $g(t)$ is monotonically decreasing. On the other hand, we have $v_{t_{1}}(f)=v_{M}(m_{n_{t_{1}}})+t_{1}n_{t_{1}}\leq v_{M}(m_{n_{t_{2}}})+t_{1}n_{t_{2}}$ and $v_{M}(m_{n_{t_{1}}})+t_{2}n_{t_{1}}>v_{t_{2}}(f)=v_{M}(m_{n_{t_{2}}})+t_{2}n_{t_{2}}$, which are equivalent to $g(t_{1})\leq 0$ and $g(t_{2})>0$. This is a contradiction.

Next we prove that $n_{f}(t)$ is right continuous at a $t_{0}\in (0,+\infty)$. There exists a small $\delta >0$ such that $v_{M}(m_{n})+tn>v_{M}(m_{n_{t_{0}}})+t_{0}n_{t_{0}}$ for every $t\in[t_{0},t_{0}+\delta)$ and $0\leq n<n_{t_{0}}$. Then, we have $n_{t_{0}}\leq n_{t}$ for every $t\in [t_{0},t_{0}+\delta)$. Since the function $t\mapsto n_{t}$ is monotonically decreasing, we have $n_{t}=n_{t_{0}}$ for every $t\in [t_{0}, t_{0}+\delta)$.
\end{proof}
Let $f\in B_{+}(\K)\backslash \{0\}$ and $g\in B_{+}(M)\backslash \{0\}$. We have
\begin{equation}\label{multiplicity of mfg and nfg real number}
m_{fg}(t)=m_{f}(t)+m_{g}(t),\ n_{fg}(t)=n_{f}(t)+n_{g}(t)
\end{equation}
for each $t\in \mathbb{R}$. Indeed, by Proposition \ref{mult Br prodct equlity} and Proposition \ref{multiplication of B_{r}(K) and B_{r}(M)}, we have \eqref{multiplicity of mfg and nfg real number} for each $t\in \mathbb{Q}$. Further, by \eqref{trivial equation of M_{m}(t)} and Proposition \ref{leading deree is left continuous}, we see that $m_{g}(t)$ and $n_{g}(t)$ is right continous for each $g\in B_{+}(M)\backslash \{0\}$. Then, we have \eqref{multiplicity of mfg and nfg real number} for each $t\in \mathbb{R}$. We call an  $r\in \mathbb{R}_{>0}$ a break-point of $g$ if the function $n_{g}(t)$ is not continuous at $r$. By $\mathrm{Proposition\ \ref{leading deree is left continuous}}$, the set of break-points of $g$ is a discrete subset. Further, by \eqref{trivial equation of M_{m}(t)}, $m_{g}(t)$ is differentiable except for break-points and satisfies 
\begin{equation}\label{mprimegt0ngt}
m_{g}^{\prime}(t)=n_{g}(t).
\end{equation}
\begin{pro}\label{break points and roots}
Let $f\in B_{+}(\K)\backslash\{0\}$. For each $r\in \mathbb{R}_{>0}$, $r$ is a break-point of $f$ if and only if there exists a root $x\in \overline{\K}$ of $f$ with $\ord_{p}(x)=r$.
\end{pro}
\begin{proof}
If there exists a root $x\in \overline{\K}$ of $f$ with $\ord_{p}(x)=r$, 
we have $d_{t}(f)>d_{r}(f)$ for each $t\in (0,r)\cap \mathbb{Q}$ by Corollary \ref{leading degree and roots of poly}. 
Thus, we conclude that $r$ is a break-point of $f$. 
On the other hand, if $r$ is a break-point of $f$, for each $t_{1},t_{2}\in \mathbb{Q}_{>0}$ with $t_{1}<r<t_{2}$, there exists a root of $f$ in the set $\{x\in \overline{\K}\ \vert \ t_{1}<\ord_{p}(x)\leq t_{2}\}$ by Corollary \ref{leading degree and roots of poly}. Thus, we see that there exists a root $x\in \overline{\K}$ of $f$ with $\ord_{p}(x)=r$.
\end{proof}
\begin{pro}\label{continuity ofv_{r}(m)}
Let $f\in B_{+}(M)\backslash\{0\}$. The function $m_{f}(t)$ is continuous.
\end{pro}
\begin{proof}
Put $f=(m_{n})_{n\in\mathbb{Z}_{\geq 0}}$. Let $x_{1},x_{2}\in \mathbb{R}_{>0}$ be break-points of $f$ such that $x_{1}<x_{2}$ and there exist no break-points in $(x_{1},x_{2})$. By \eqref{trivial equation of M_{m}(t)}, we have $m_{f}(t)=v_{M}(m_{d_{x_{1}}(f)})+td_{x_{1}}(f)$ on $t\in [x_{1},x_{2})$. Therefore, it suffices to prove that $m_{f}(t)$ is left continuous at the break-point $x_{2}$. Put $s=d_{x_{2}}(f)$ and $s_{0}=d_{x_{1}}(f)$. By the definition of $m_{f}$, we have $m_{f}(x_{2})\leq v_{M}(m_{s_{0}})+s_{0}x_{2}$. Further, we have $m_{f}(t)=v_{M}(m_{s_{0}})+s_{0}t\leq m_{f}(x_{2})$ for every $t\in [x_{1},x_{2})$. Thus, we see that $v_{M}(m_{s_{0}})+s_{0}x_{2}=m_{f}(x_{2})$ and $m_{f}(t)$ is left  continuous at $x_{2}$.
\end{proof}
Let $\log(1+X) \in B_{+}(\K)$ be the $p$-adic logarithm function defined by
\begin{equation}\label{definition of p-adic logarithmfunc}
\log(1+X)=\sum_{k=1}^{+\infty}(-1)^{k-1}\frac{X^{k}}{k}.
\end{equation}
We set $t_{n}=\frac{1}{p^{n}(p-1)}$ for each $n\in \mathbb{Z}_{\geq 0}$. 
The following proposition is stated in \cite[2.6. Example]{vishik1976}. We give a detail of the proof.
\begin{pro}\label{properties of log}
Let $\log (1+X)$ be the $p$-adic logarithm function. Then, the break-points of $\log(1+X)$ are $t_{n}$ with $n\geq 0$. In addiiton, we have 
$$d_{t_{n}}(\log(1+X))=p^{n},\  m_{\log(1+X)}(t_{n})=-n+\frac{1}{p-1}$$
for each $n\geq 0$.
\end{pro}
\begin{proof}
It is well-known that the roots of $\log(1+X)$ are $\epsilon-1$ with $\epsilon\in \mu_{p^{\infty}}$. Then, by $\mathrm{Proposition\ \ref{break points and roots}}$, the break-points of $\log(1+X)$ are $t_{n}$ with $n\geq 0$. Further, since $\log(1+X)^{\prime}=\frac{1}{1+X}$, $\log(1+X)$ has no multiple roots. Thus,  we see that $d_{t_{n}}(\log(1+X))=p^{n}$ for each $n\geq 0$. 

Next we prove that $m_{\log(1+X)}(t_{n})=-n+\frac{1}{p-1}$ for each $n\geq 0$. By $\mathrm{Proposition\ \ref{supectral norm gauss norm}}$, we get 
\begin{align*}
v_{t_{0}}(\log(1+X))=\inf_{\substack{a\in \mathbb{C}_{p}\\ \ord_{p}(a)>t_{0}}}\{\ord_{p}(\log(1+a))\}.
\end{align*}
It is known that $\displaystyle{{\inf}_{
a\in \mathbb{C}_{p}, 
\ord_{p}(a)>t_{0}
}
}
\{\ord_{p}(\log(1+a))\}=\tfrac{1}{p-1}$ ($cf$. \cite[Lemma 5.5]{washinton}). Then, we have $m_{\log(1+X)}(t_{0})=\frac{1}{p-1}$. Further, since the slope of $m_{\log(1+X)}(t)$ on $[t_{n+1},t_{n}]$ is $d_{t_{n+1}}(\log(1+X))=p^{n+1}$, we have 
\begin{align*}
m_{\log(1+X)}(t_{n+1})-m_{\log(1+X)}(t_{n})=p^{n+1}(t_{n+1}-t_{n})=-1.
\end{align*}
Thus, we conclude that $m_{\log(1+X)}(t_{n})=-n+\frac{1}{p-1}$ for each $n\geq 0$.
\end{proof}
We take a topological generator $u\in 1+2p\mathbb{Z}_{p}$. Let $d,e\in \mathbb{Z}$ be elements satisfying $e\geq d$. We define
\begin{equation}\label{definition of Omega}
\Omega_{m}^{[d,e]}(X)=\prod_{i=d}^{e}((1+X)^{p^{m}}-u^{ip^{m}}),
\end{equation}
for each $m\in \mathbb{Z}_{\geq 0}$.
\begin{lem}\label{Omega is separable}
Let $m\in\mathbb{Z}_{\geq 0}$. Then, $\Omega_{m}^{[d,e]}(X)$ is separable. 
\end{lem}
\begin{proof}
Put $\omega_{m,i}(X)=(1+X)^{p^{m}}-u^{ip^{n}}$. It is easy to see that $\omega_{m,i}(X)$ is separable for each $m\in \mathbb{Z}_{\geq 0}$ and $i\in [d,e]$. Then, it suffices to prove that $\omega_{m,i}$ and $\omega_{m,j}$ have no common roots for any two distinct elements $i$, $j$ in $[d,e]$. The roots of $\omega_{m,i}$ are given by $u^{i}\epsilon-1$ for $\epsilon\in \mu_{p^{m}}$.  Then, if $\omega_{m,i}$ and $\omega_{m,j}$ have a common root,  there are $\epsilon_{1},\epsilon_{2}\in \mu_{p^{m}}$ such that $u^{i}\epsilon_{1}=u^{j}\epsilon_{2}$. By raising the both sides to the $p^{m}$-th power, we get $u^{p^{m}i}=u^{p^{m}j}$, which is equivalent to $u^{p^{m}(j-i)}=1$. This contradicts to the assumption $i\not= j$ and this completes the proof.
\end{proof}
\begin{lem}\label{valuation of Omega}
Let $m\in\mathbb{Z}_{\geq 0}$. The break-points of $\Omega_{m}^{[d,e]}$ on $(0,t_{0}]$ are $t_{0},\ldots, t_{m-1}$ where $t_{n}=\frac{1}{p^{n}(p-1)}$ for $n\in\mathbb{Z}_{\geq 0}$.
Further, we have 
$$d_{t_{n}}(\Omega_{m}^{[d,e]})=(e-d+1)p^{n}$$
and
$$m_{\Omega_{m}^{[d,e]}}(t_{n})=(e-d+1)(m-n+t_{n}p^{n})$$
for every $n\leq m$.
\end{lem}
\begin{proof}
Let $n,m\in\mathbb{Z}_{\geq 0}$. 
{
For each $i\in [d,e]$ and for each primitive $p^{n}$-th power root of unity $\epsilon$, we have $\ord_{p}(u^{i}\epsilon-1)=\ord_{p}(u^{i}(\epsilon-1)+(u^{i}-1))$, which implies 
\begin{align}\label{valuation of Omega ordpuiepsilon uipeislon1 eq}
& \ord_{p}(u^{i}\epsilon-1)=\min\{\ord_{p}(u^{i}(\epsilon-1)),\ord_{p}(u^{i}-1)\}=t_{n-1},\ &\mathrm{if}\ n\geq 1,\\
& \ord_{p}(u^{i}\epsilon-1)=\ord_{p}(u^{i}-1)=\ord_{p}(2)+1+\ord_{p}(i)\ &\mathrm{if}\ n=0. \notag 
\end{align}}
The roots of $\Omega_{m}^{[d,e]}$ are $u^{i}\epsilon-1$ with $i\in [d,e]$ and $\epsilon\in \mu_{p^{m}}$. By Proposition \ref{break points and roots}, the break-points of $\Omega_{m}^{[d,e]}$ are given by $\ord_{p}(u^{i}\epsilon-1)$ with $i\in [d,e]$, $\epsilon\in \mu_{p^{m}}$. Therefore, by \eqref{valuation of Omega ordpuiepsilon uipeislon1 eq}, we see that $t_{0},\ldots, t_{m-1}$ are the break-points of $\Omega_{m}^{[d,e]}$ on $(0,t_{0}]$. 
Let $\omega_{m,i}(X)=(1+X)^{p^{m}}-u^{ip^{m}}$ and let $n$ be a non-negative integer satisfying $m\geq n$. By \eqref{valuation of Omega ordpuiepsilon uipeislon1 eq}, roots of $\omega_{m,i}$ on $\{x\in \overline{\K}\vert \ord_{p}(x)>t_{n}\}$ are given by $u^{i}\epsilon-1$ with $\epsilon\in \mu_{p^{n}}$. Then, by $\mathrm{Corollary\ \ref{leading degree and roots of poly}}$, we get 
\begin{equation}\label{valuation of Omega minomegadteq}
d_{t_{n}}(\omega_{m,i})=p^{n}.
\end{equation}
 Since $X^{p^{n}}$-th coefficient of $\omega_{m,i}(X)$ is $\begin{pmatrix}p^{m}\\ p^{n}\end{pmatrix}$, by \eqref{trivial equation of M_{m}(t)}, we have 
\begin{equation}\label{valuation of Omega vtnomegam,ival eq}
v_{t_{n}}(\omega_{m,i})=\ord_{p}\left(\begin{pmatrix}p^{m}\\p^{n}\end{pmatrix}\right)+t_{n}p^{n}=m-n+t_{n}p^{n}.
\end{equation}
By Proposition \ref{multiplication of B_{r}(K) and B_{r}(M)} and \eqref{valuation of Omega minomegadteq}, we conclude that
$$
d_{t_{n}}(\Omega_{m}^{[d,e]})=\sum_{i=d}^{e}d_{t_{n}}(\omega_{m,i})
=(e-d+1)p^{n}.
$$
Further, by Proposition \ref{mult Br prodct equlity} and \eqref{valuation of Omega vtnomegam,ival eq}, we have
$$ 
v_{t_{n}}(\Omega_{m}^{[d,e]})=\sum_{i=d}^{e}v_{t_{n}}(\omega_{m,i})
=(e-d+1)(m-n+t_{n}p^{n}).
$$ 
This completes the proof.
\end{proof}
We have $\HH_{h}(M)\subset B_{+}(M)$ since $\displaystyle{\lim_{n\rightarrow +\infty}}(rn-h\ell(n))=+\infty$ for every $r>0$. We define the map 
$v_{h}^{\prime} : B_{+}(M) \longrightarrow \mathbb{R} \cup \{ \pm\infty \}$
by setting 
\begin{equation}\label{definition of one variable hhhprimeor hhprime}
{v_{h}^{\prime}(f)}=\inf\{v_{t_{n}}(f)+hn\}_{n\geq 0}
\end{equation}
for each $f\in B_{+}(M)$ where $t_{n}=\frac{1}{p^{n}(p-1)}$ with $n\in\mathbb{Z}_{\geq 0}$. The following proposition is a generalization of \cite[Lemma I\hspace{-1.2pt}I.1.1]{colmez2010}.
\begin{pro}\label{another valuation on H_{h}}
For each $f\in B_{+}(M)$, we have $f\in \HH_{h}(M)$ if and only if ${v_{h}^{\prime}(f)}> -\infty$. In addition, 
${v_{h}^{\prime}\vert_{\HH_{h}(M)}}$ is a valuation on $\HH_{h}(M)$ which satisfies $v_{\HH_{h}}+\alpha_{h}\leq {v_{h}^{\prime}\vert_{\HH_{h}(M)}}\leq v_{\HH_{h}}+\beta_{h}$, where 
\begin{align*}
\alpha_{h}&=\begin{cases}-\max\{0, h-\frac{h}{\log p}(1+\log \frac{\log p}{(p-1)h})\}\ &\mathrm{if}\ h>0,\\
0\ &\mathrm{if}\ h=0,
\end{cases}\\
\beta_{h}&=\begin{cases}\max\{0,\frac{p}{p-1}-h\}\ \ \ \ \ \ \ \ \ \ \ \ \ \ \ \ \ \ \ \ \ \ \,&\mathrm{if}\ h>0,\\
0\ &\mathrm{if}\ h=0.
\end{cases}
\end{align*}
\end{pro}
In the case $M=\K$, the inequality $v_{\HH_{h}}+\alpha_{h}\leq {v_{h}^{\prime}\vert_{\HH_{h}(M)}}\leq v_{\HH_{h}}+\beta_{h}$ in $\mathrm{Proposition\ \ref{another valuation on H_{h}}}$ is given in the proof of \cite[Lemma I\hspace{-1.2pt}I.1.1]{colmez2010}. Further, it is easy to see that we can generalize the result \cite[Lemma I\hspace{-1.2pt}I.1.1]{colmez2010}  to a result on $\HH_{h}(M)$. 
Hence, we omit the proof of $\mathrm{Proposition\ \ref{another valuation on H_{h}}}$. By Proposition\ \ref{another valuation on H_{h}}, we see that $f_{1}\cdot f_{2}\in \HH_{g+h}(M)$ for each $f_{1}\in \HH_{g}(\K)$ and $f_{2}\in \HH_{h}(M)$ with $g,h\in \ord_{p}(\mathcal{O}_{\K}\backslash\{0\})$. {We define $v_{\HH_{h}}^{\prime}: \HH_{h}(M)\rightarrow \mathbb{R}\cup\{+\infty\}$ to be $v_{\HH_{h}}^{\prime}=v_{h}^{\prime}\vert_{\HH_{h}(M)}$, where $v_{h}^{\prime}$ is the map defined in  \eqref{definition of one variable hhhprimeor hhprime}.} 
The following proposition is a generalization of $\mathrm{Theorem\ \ref{one variable classical uniqueness}}$ on $\HH_{h}(\K )$ to a result on $\HH_{h}(M)$ with a $\K$-Banach space $M$.
\begin{pro}\label{generalization of uniqueness on M}
Let $f\in \HH_{h}(M)$. If there exists an integer $d\in \mathbb{Z}$ such that $f(u^{i}\epsilon-1)=0$ in $M_{\K(\epsilon)}$ for every $i\in [d,d+\lfloor h\rfloor]$ and for every $\epsilon\in \mu_{p^{\infty}}$, then we have $f=0$. 
\end{pro}
\begin{proof}
By contradiction, we assume that $f\neq 0$. We define $t_{m}=\frac{1}{p^{m}(p-1)}$ for each $m\geq 0$. Let $t\in [t_{m+1},t_{m})$. By $\mathrm{Lemma\ \ref{valuation of Omega}}$, we see that $d_{t}(\Omega_{m+1}^{[d,d+\lfloor h\rfloor]})=\deg\Omega_{m+1}^{[d,d+\lfloor h\rfloor]}$. Further, by $\mathrm{Corollary\ \ref{cor of remainder theorem}}$, we have $f\in \Omega_{m+1}^{[d,d+\lfloor h\rfloor]}B_{+}(M)$. Since $f\neq 0$, we can define $d_{t}(f)\in \mathbb{Z}_{\geq 0}$ and we have 
\begin{equation*}
d_{t}(f)\geq \deg \Omega_{m+1}^{[d,d+\lfloor h\rfloor]}=(\lfloor h\rfloor+1)p^{m+1}.
\end{equation*}
Thus, by $\mathrm{Proposition\ \ref{properties of log}}$, we have $d_{t}(f)\geq (\lfloor h\rfloor+1)d_{t}(\log(1+X))$ for each $t\in[t_{m+1},t_{m})$. Therefore, by \eqref{mprimegt0ngt}, we see that $m_{f}(t)-(\lfloor h\rfloor+1)m_{\log(1+X)}(t)$ is monotonically increasing on $t\in (0,t_{0}]$. In particular, by $\mathrm{Proposition\ \ref{properties of log}}$, $\mathrm{sup}\{v_{t_{n}}(f)+(\lfloor h\rfloor+1)n)\}_{n\geq 0}\neq +\infty$.

On the other hand, we have
\begin{align*}
v_{t_{n}}(f)+(\lfloor h\rfloor+1)n\geq v_{\HH_{h}}^{\prime}(f)+(\lfloor h\rfloor+1-h)n
\end{align*}
for each $n\geq 0$. By $\mathrm{Proposition\ \ref{another valuation on H_{h}}}$, we see that $\displaystyle{\lim_{n\rightarrow +\infty}}(v_{t_{n}}(f)+(\lfloor h\rfloor+1)n)\geq \displaystyle{\lim_{n\rightarrow+\infty}}(v_{\HH_{h}}^{\prime}(f)+(\lfloor h\rfloor+1-h)n)=+\infty$. This is a contradiction.
\end{proof}
Let $J_{h}^{[d,e]}(M)$ be the $\mathcal{O}_{\K}[[X]]\otimes_{\mathcal{O}_{\K}}\K$-module defined in \eqref{generalization of the project lim for deformation ring Jboldsymbol}. Let $(s_{m}^{[d,e]})_{m\in \mathbb{Z}_{\geq 0}}\in J_{h}^{[d,e]}(M)$. By Proposition \ref{Weiestrass on Banach space}, for each $m\in \mathbb{Z}_{\geq 0}$, there exists a unique element $r(s_{m}^{[d,e]})\in M^{0}[X]\otimes_{\mathcal{O}_{\K}}\K$ such that $s_{m}^{[d,e]}\equiv r(s_{m}^{[d,e]})\ \mathrm{mod}\ \Omega_{m}^{[d,e]}(X)$ and $\deg r(s_{m}^{[d,e]})<\deg \Omega_{m}^{[d,e]}$. We define a valuation on $v_{J_{h}}$ on $J_{h}^{[d,e]}(M)$ to be 
\begin{equation}
v_{J_{h}}((s_{m}^{[d,e]})_{m\in \mathbb{Z}_{\geq 0}})=\inf_{m\in \mathbb{Z}_{\geq 0}}\{v_{0}(r(s_{m}^{[d,e]}))+hm\}
\end{equation}
for each $(s_{m}^{[d,e]})_{m\in \mathbb{Z}_{\geq 0}}\in J_{h}^{[d,e]}(M)$ where $v_{0}$ is the valuation on $B_{0}(M)$. It is easy to see that $v_{J_{h}}$ is a valuation on $J_{h}^{[d,e]}(M)$. Further, we have the following:
\begin{pro}\label{one variable Jh is a K-Banach}
The pair $(J_{h}^{[d,e]}(M), v_{J_{h}})$ is a $\K$-Banach space.
\end{pro}
\begin{proof}
Let $(s^{[d,e]}_{(n)})_{n\geq 1}\subset J_{h}^{[d,e]}(M)$ be a Cauchy sequence. Put $s^{[d,e]}_{(n)}=(s^{[d,e]}_{(n),m})_{m\in \mathbb{Z}_{\geq 0}}$. By Proposition \ref{Weiestrass on Banach space}, for each $n\geq 1$ and $m\in \mathbb{Z}_{\geq 0}$, there exists a unique element $r(s^{[d,e]}_{(n),m})\in M^{0}[X]\otimes_{\mathcal{O}_{\K}}\K$ such that $s_{(n),m}^{[d,e]}\equiv r(s^{[d,e]}_{(n),m})\ \mathrm{mod}\ \Omega_{m}^{[d,e]}$ and $\deg r(s^{[d,e]}_{(n),m})<\deg \Omega_{m}^{[d,e]}$. By the definition of $v_{J_{h}}$, $(r(s^{[d,e]}_{(n),m}))_{n\geq 1}$ is a Cauchy sequence in $B_{0}(M)$ for each $m\in \mathbb{Z}_{\geq 0}$. Put $r_{m}^{[d,e]}=\displaystyle{\lim_{n\rightarrow +\infty}}r(s^{[d,e]}_{(n),m})$. It is easy to see that 
\begin{equation}\label{onvariable Jjkbnach v0boundedeq}
v_{0}(r_{m}^{[d,e]})+hm\geq \inf_{n\geq 1}\{v_{J_{h}}(s^{[d,e]}_{(n)})\}
\end{equation}
for every $m\in \mathbb{Z}_{\geq 0}$.  For every $m\in \mathbb{Z}_{\geq 0}$ and for every root $b\in \overline{\K}$ of $\Omega_{m}^{[d,e]}$, we see that
\begin{align*}
r^{[d,e]}_{m+1}(b)&=\lim_{n\rightarrow +\infty}r(s^{[d,e]}_{(n),m+1})(b)\\
&=\lim_{n\rightarrow +\infty}r(s^{[d,e]}_{(n),m})(b)=r^{[d,e]}_{m}(b).
\end{align*}
Thus, by Corollary \ref{cor of remainder theorem}, we have $s^{[d,e]}=([r_{m}^{[d,e]}])_{m\in \mathbb{Z}_{\geq 0}}\in \varprojlim_{m\in \mathbb{Z}_{\geq 0}}\left(\frac{M^{0}[[X]]}{ \Omega_{m}^{[d,e]}M^{0}[[X]]}\otimes_{\mathcal{O}_{\K}}\K\right)$, where $[r_{m}^{[d,e]}]
$ is the image of $r_{m}^{[d,e]}$ by the natural projection $M^{0}[[X]]\otimes_{\mathcal{O}_{\K}}\K\rightarrow \bigg(M^{0}[[X]]\slash \linebreak\Omega_{m}^{[d,e]}M^{0}[[X]]\bigg)\otimes_{\mathcal{O}_{\K}}\K$. Further, by \eqref{onvariable Jjkbnach v0boundedeq}, we see that $s^{[d,e]}\in J_{h}^{[d,e]}(M)$. Let $A>0$. There exists a positive integer $N$ such that $v_{J_{h}}(s^{[d,e]}_{(n_{1})}-s^{[d,e]}_{(n_{2})})\geq A$ for every $n_{1},n_{2}\geq N$. Thus, we have $v_{0}(r^{[d,e]}_{m}-r(s^{[d,e]}_{(n),m}))+hm=\lim_{n^{\prime}\rightarrow +\infty}(v_{0}(r(s^{[d,e]}_{(n^{\prime}),m})-r(s^{[d,e]}_{(n),m})))+hm\geq A$ for every $m\in \mathbb{Z}_{\geq 0}$ and $n\geq N$. Therefore, we have $v_{J_{h}}(s^{[d,e]}-s^{[d,e]}_{(n)})\geq A$ for every $n\geq N$. That is, we have $s^{[d,e]}=\displaystyle{{\lim}_{n\rightarrow +\infty}}s^{[d,e]}_{(n)}$.
\end{proof}
By definition, we have
\begin{multline}
J_{h}^{[d,e]}(M)^{0}=\Bigg\{(s_{m}^{[d,e]})_{m\in \mathbb{Z}_{\geq 0}}\in J_{h}^{[d,e]}(M)\Bigg\vert \\
(p^{hm}s_{m}^{[d,e]})_{m\in \mathbb{Z}_{\geq 0}}\in \prod_{m\in \mathbb{Z}_{\geq 0}}M^{0}[[X]]\slash \Omega_{m}^{[d,e]}(X)M^{0}[[X]]\Bigg\}.
\end{multline}
We generalize $\mathrm{Theorem\ \ref{onevariable projectresult}}$ to a result on a Banach space $(M,v_{M})$. 
\begin{pro}\label{isomorphism HH and project lim Banach}
Assume that $e-d\geq \lfloor h\rfloor$. For $s^{[d,e]}=(s_{m}^{[d,e]})_{m\in\mathbb{Z}_{\geq 0}}\in J_{h}^{[d,e]}(M)$, there exists a unique element $f_{s^{[d,e]}}\in \HH_{h}(M)$ such that 
$$f_{s^{[d,e]}}-\tilde{s}^{[d,e]}_{m}\in \Omega_{m}^{[d,e]}\HH_{h}(M)$$
for each $m\in\mathbb{Z}_{\geq 0}$, where $\tilde{s}^{[d,e]}_{m}\in M^{0}[[X]]\otimes_{\mathcal{O}_{\K}}\K$ is a lift of $s_{m}^{[d,e]}$. Further, the correspondence $s^{[d,e]}\mapsto f_{s^{[d,e]}}$ from $J_{h}^{[d,e]}(M)$ to $\HH_{h}(M)$ induces an $\mathcal{O}_\mathcal{K}[[X]]\otimes_{\mathcal{O}_{\mathcal{K}}}\mathcal{K}$-module isomorphism 
$$
J_{h}^{[d,e]}(M) \overset{\sim}{\longrightarrow} \HH_{h}(M)
$$ 
and, via the above isomorphism, we have
\begin{align*}
\{f\in \HH_{h}(M)\ \vert \ v_{\HH_{h}}(f)\geq \epsilon_{h}^{[d,e]}\}\subset J_{h}^{[d,e]}(M)^{0}\subset \{f\in \HH_{h}(M)\ \vert \ v_{\HH_{h}}(f)\geq \zeta_{h}\},
\end{align*}
where 
\begin{align*}
\epsilon_{h}^{[d,e]}&=\begin{cases}\lfloor\frac{(e-d+1)}{p-1}+\max\{0, h-\frac{h}{\log p}(1+\log \frac{\log p}{(p-1)h})\}\rfloor+1\ &\mathrm{if}\ h>0,\\ 0\ &\mathrm{if}\ h=0,\end{cases}\\
\zeta_{h}&=\begin{cases}-(\lfloor\max\{h,\frac{p}{p-1}\}\rfloor+1)\ \ \ \  \ \ \ \ \ \ \ \ \ \ \ \ \ \ \ \ \ \ \ \ \ \ \ \ \ \ \ \ \, &\mathrm{if}\ h>0,\\ 0\ &\mathrm{if}\ h=0.\end{cases}
\end{align*}
\end{pro}
\begin{proof}
Let $s^{[d,e]}=(s_{m}^{[d,e]})_{m\in\mathbb{Z}_{\geq 0}}\in J_{h}^{[d,e]}(M)$. First, we prove that there exists a unique element $f_{s^{[d,e]}}\in \HH_{h}(M)$ such that $f_{s^{[d,e]}}(u^{i}\epsilon-1)=\tilde{s}^{[d,e]}_{m}(u^{i}\epsilon-1)$ for each $i\in [d,e]$, non-negative integer $m$ and $\epsilon\in \mu_{p^{m}}$. The uniqueness of $f_{s^{[d,e]}}$ follows from $\mathrm{Proposition\ \ref{generalization of uniqueness on M}}$. Then it suffices to prove the existence of $f_{s^{[d,e]}}$. We can assume that $s\in J_{h}^{[d,e]}(M)^{0}$ and $\tilde{s}^{[d,e]}_{m}\in M^{0}[[X]]\otimes_{\mathcal{O}_{\K}}p^{-hm}\mathcal{O}_{\K}$ for every $m\in \mathbb{Z}_{\geq 0}$. By $\mathrm{Corollary\ \ref{cor of remainder theorem}}$, there exists a $q_{m}\in p^{-h(m+1)}M^{0}[[X]]$ which satisfies $\tilde{s}^{[d,e]}_{m+1}-\tilde{s}^{[d,e]}_{m}=\Omega_{m}^{[d,e]}q_{m}$ for each $m\in \mathbb{Z}_{\geq 0}$. We fix a non-negative integer $n$ and put $t_{n}=\frac{1}{p^{n}(p-1)}$. By $\mathrm{Lemma\ \ref{valuation of Omega}}$, we see that
\begin{align*}
v_{t_{n}}(\tilde{s}^{[d,e]}_{m+1}-\tilde{s}^{[d,e]}_{m})&=v_{t_{n}}(\Omega_{m}^{[d,e]})+v_{t_{n}}(q_{m})\\
&\geq  (e-d+1)(m-n+\frac{1}{p-1})-h(m+1)\\
&=(e-d+1-h)m+(e-d+1)(\frac{1}{p-1}-n)-h
\end{align*}
for each $m\geq n$. Thus the sequence $(\tilde{s}^{[d,e]}_{m})_{m\geq 0}$ converges in $B_{t_{n}}(M)$ and there exists a unique element $f_{s^{[d,e]}}\in B_{+}(M)$ such that $\displaystyle{\lim_{m\rightarrow +\infty}}v_{t_{n}}(f_{s^{[d,e]}}-\tilde{s}^{[d,e]}_{m})=+\infty$ for all $n\in \mathbb{Z}_{\geq 0}$. We have $f_{s^{[d,e]}}=\tilde{s}^{[d,e]}_{n}+\sum_{m=n}^{+\infty}(\tilde{s}^{[d,e]}_{m+1}-\tilde{s}^{[d,e]}_{m})$ in $B_{t_{n}}(M)$ and then
\begin{align*}
v_{t_{n}}(f_{s^{[d,e]}})&\geq \min\{v_{t_{n}}(\tilde{s}^{[d,e]}_{n}),\inf\{v_{t_{n}}(\tilde{s}^{[d,e]}_{m+1}-\tilde{s}^{[d,e]}_{m})\}_{m\geq n}\}\\
&\geq -hn+\min\{0,(e-d+1)\frac{1}{p-1}-h\}\\
&\geq -hn-h.
\end{align*}
By $\mathrm{Proposition\ \ref{another valuation on H_{h}}}$, $f_{s^{[d,e]}}$ is an element of $\HH_{h}(M)$ and satisfies $v_{\HH_{h}}(f_{s^{[d,e]}})\geq \zeta_{h}$. By construction, $f_{s^{[d,e]}}$ satisfies $f_{s^{[d,e]}}(u^{i}\epsilon-1)=\tilde{s}^{[d,e]}_{m}(u^{i}\epsilon-1)$ for each $i\in [d,e]$, non-negative integer $m$ and $\epsilon\in \mu_{p^{m}}$. 

Next, we prove that $f_{s^{[d,e]}}-\tilde{s}_{m}^{[d,e]}\in \Omega_{m}^{[d,e]}\HH_{h}(M)$ for each $m\in \mathbb{Z}_{\geq 0}$. There exists a $q_{m}^{[d,e]}\in B_{+}(M)$ such that $f_{s^{[d,e]}}-\tilde{s}^{[d,e]}_{m}=\Omega_{m}^{[d,e]}q_{m}^{[d,e]}$ by $\mathrm{Corollary\ \ref{cor of remainder theorem}}$. Then, for each $n\in \mathbb{Z}_{\geq 0}$, we see that
\begin{align*}
v_{t_{n}}(q_{m}^{[d,e]})+hn&=v_{t_{n}}(f_{s^{[d,e]}}-\tilde{s}^{[d,e]}_{m})-v_{t_{n}}(\Omega_{m}^{[d,e]})+hn\\
&\geq \min\{v_{\HH_{h}}^{\prime}(f_{s^{[d,e]}}),-hm\}-v_{t_{0}}(\Omega_{m}^{[d,e]}),
\end{align*}
where $v_{\HH_{h}}^{\prime}$ is the valuation defined in $\mathrm{Proposition\ \ref{another valuation on H_{h}}}$. Therefore, we have $v_{\HH_{h}}^{\prime}(q_{m}^{[d,e]})\neq -\infty$, which is equivalent to $q_{m}^{[d,e]}\in \HH_{h}(M)$. We conclude that $f_{s^{[d,e]}}-\tilde{s}^{[d,e]}_{m}\in \Omega_{m}^{[d,e]}\HH_{h}(M)$ for each $m\in\mathbb{Z}_{\geq 0}$.

By Corollary\ \ref{cor of remainder theorem}, we see that the correspondence $s^{[d,e]}\mapsto f_{s^{[d,e]}}$ induces an injective $\mathcal{O}_{\mathcal{K}}[[X]]\otimes_{\mathcal{O}_{\K}}\K$-module homomorphism
$$J_{h}^{[d,e]}(M)\rightarrow \HH_{h}(M).$$
Further, as mentioned above, we have $J_{h}^{[d,e]}(M)^{0}\subset \{f\in \HH_{h}(M)\vert v_{\HH_{h}}(f)\geq \zeta_{h}\}$. Then if we prove $\{f\in \HH_{h}(M)\vert v_{\HH_{h}}(f)\geq \epsilon_{h}^{[d,e]}\}\subset J_{h}^{[d,e]}(M)^{0}$, we complete the proof.

Let $f\in \HH_{h}(M)$ with $v_{\HH_{h}}(f)\geq \epsilon_{h}^{[d,e]}$. We take an $m\in \mathbb{Z}_{\geq 0}$. If $h=0$, by $\mathrm{Proposition\ \ref{Weiestrass on Banach space}}$, there exists a unique pair $(q_{m}^{[d,e]},r_{m}^{[d,e]})\in B_{0}(M)\times M[X]$ such that $f=\Omega_{m}^{[d,e]}q_{m}^{[d,e]}+r_{m}^{[d,e]}$ and $\deg r_{m}^{[d,e]}<(e-d+1)p^{m}$. In addition, we have $v_{0}(f)=\inf\{v_{0}(\Omega_{m}^{[d,e]})+v_{0}(q_{m}^{[d,e]}), v_{0}(r_{m}^{[d,e]})\}$. Since $v_{0}(f)=v_{\HH_{0}}(f)\geq \epsilon_{0}^{[d,e]}=0$, we see that $r_{m}^{[d,e]}\in M^{0}[[X]]$. We denote by $[r_{m}^{[d,e]}]\in M^{0}[[X]]\slash \Omega_{m}^{[d,e]}M^{0}[[X]]$ the image of $r_{m}^{[d,e]}$ by the natural projection $M^{0}[[X]]\rightarrow M^{0}[[X]]\slash \Omega_{m}^{[d,e]}M^{0}[[X]]$ for each $m\in \mathbb{Z}_{\geq 0}$. Then, we see that $s^{[d,e]}=([r_{m}^{[d,e]}])_{m\in \mathbb{Z}_{\geq 0}}\in J_{0}^{[d,e]}(M)^{0}$ and $f_{s^{[d,e]}}=f$. We conclude that $\{f\in \HH_{0}(M)\vert v_{\HH_{0}}(f)\geq 0\}\subset J_{0}^{[d,e]}(M)^{0}$.

Therefore, we can assume that $h>0$. By $\mathrm{Proposition\ \ref{Weiestrass on Banach space}}$, there exists a unique pair $(q_{m}^{[d,e]},r_{m}^{[d,e]})\in B_{t_{m}}(M)\times M[X]$ such that $f=\Omega_{m}^{[d,e]}q_{m}^{[d,e]}+r_{m}^{[d,e]}$ and $\deg r_{m}^{[d,e]}<(e-d+1)p^{m}$. In addition, we have $v_{t_{m}}(f)=\inf\{v_{t_{m}}(\Omega_{m}^{[d,e]})+v_{t_{m}}(q_{m}^{[d,e]}), v_{t_{m}}(r_{m}^{[d,e]})\}$. Since $\deg r_{m}^{[d,e]}<(e-d+1)p^{m}$, we see that $v_{0}(r_{m}^{[d,e]})+t_{m}((e-d+1)p^{m}-1)\geq v_{t_{m}}(r_{m}^{[d,e]})\geq v_{t_{m}}(f)$. Then, by $\mathrm{Proposition\ \ref{another valuation on H_{h}}}$, we have
\begin{align*}
v_{0}(r_{m}^{[d,e]})+hm&\geq-t_{m}((e-d+1)p^{m}-1)+(v_{t_{m}}(f)+hm)\\
&\geq-t_{m}((e-d+1)p^{m}-1)+v_{\HH_{h}}(f)+\alpha_{h}\\
&\geq \frac{-1}{p-1}(e-d+1)+v_{\HH_{h}}(f)+\alpha_{h}\\
&\geq \frac{-1}{p-1}(e-d+1)+\epsilon_{h}^{[d,e]}+\alpha_{h}\geq 0,
\end{align*}
where $\alpha_{h}=-\max\{0, h-\frac{h}{\log p}(1+\log \frac{\log p}{(p-1)h})\}$. Then, we see that $v_{0}(r_{m}^{[d,e]})\geq -hm$ and $s^{[d,e]}=([r_{m}^{[d,e]}])_{m\in \mathbb{Z}_{\geq 0}}\in J_{h}^{[d,e]}(M)^{0}$. By Proposition\ \ref{generalization of uniqueness on M}, we see that $f_{s^{[d,e]}}=f$. Then, we conclude that $\{f\in \HH_{h}(M)\vert v_{\HH_{h}}(f)\geq \epsilon_{h}^{[d,e]}\}\subset J_{h}^{[d,e]}(M)^{0}$. We complete the proof.
\end{proof}

Let $\Gamma$ be a $p$-adic Lie group which is isomorphic to $1+2p\mathbb{Z}_p\subset \mathbb{Q}_{p}^{\times}$ via a continuous character $\chi : \Gamma \longrightarrow \mathbb{Q}_{p}^{\times}$. Fix a topological  generator $\gamma\in \Gamma$ such that $\chi(\gamma)=u$. Let $\mathfrak{X}_{\mathcal{O}_{\K}[[\Gamma]]}^{[d,e]}$ be the set of arithmetic specializations $\kappa$ such that $w_{\kappa}\in [d,e]$ defined in \S\ref{preparation}. Put $\Omega_{m}^{[d,e]}(\gamma)=\prod_{j=d}^{e}([\gamma]^{p^{m}}-u^{jp^{m}})\in \mathcal{O}_{\K}[[\Gamma]]$ for each $m\in \mathbb{Z}_{\geq 0}$. Let $M^{0}[[\Gamma]]$ be the $\mathcal{O}_{\K}[[\Gamma]]$-module defined in \eqref{definition of M0[[Gamma]]}. Let $s\in M^{0}[[\Gamma]]\otimes_{\mathcal{O}_{\K}}\K$ and $m\in \mathbb{Z}_{\geq 0}$. Via the non-canonical isomorphism $M^{0}[[\Gamma]]\simeq M^{0}[[X]]$ in \eqref{non-canonical continuous isomorphihsm of iwasawa module of banach}, by Corollary \ref{cor of remainder theorem}, we see that $s\in \Omega_{m}^{[d,e]}(\gamma)(M^{0}[[\Gamma]]\otimes_{\mathcal{O}_{\K}}\K)$ if and only if 
\begin{equation}\label{onevariable iwasawa modoomega speq}
\kappa(s)=0\ \mathrm{for\ every}\ \kappa\in \mathfrak{X}_{\mathcal{O}_{\K}[[\Gamma]]}^{[d,e]}\ \mathrm{with}\ m_{\kappa}\leq m.
\end{equation}
\begin{lem}\label{onevariable litfing prop}
Let $m\in \mathbb{Z}_{\geq 0}$. 
Let $s_{m}^{[i]}\in \frac{M^{0}[[\Gamma]]}{\Omega_{m}^{[i]}(\gamma)M^{0}[[\Gamma]]} 
\otimes_{\mathcal{O}_{\K}}\K$ and $\widetilde{s}_{m}^{[i]} \in  M^{0}[[\Gamma]]\otimes_{\mathcal{O}_{\K}}\K$  a lift of  $s_{m}^{[i]}$ for each $i\in [d,e]$. For each $j\in [d,e]$, we define $\theta_{j}\in M^{0}[[\Gamma]]\otimes_{\mathcal{O}_{\K}}\K$ by 
\begin{equation}\label{generalization of prop 1 on Banach1}
\theta_{j}=\displaystyle{\sum_{i=d}^{j}}\begin{pmatrix}j-d\\i-d\end{pmatrix}(-1)^{j-i}\widetilde{s}_{m}^{[i]}.
\end{equation}
If $\theta_{j}$ is contained in $p^{m(j-d)}M^{0}[[\Gamma]]\subset M^{0}[[\Gamma]]\otimes_{\mathcal{O}_{\K}}\K$ for every $j\in [d,e]$, there exists a unique element $s_m^{[d,e]}\in \frac{M^{0}[[\Gamma]]}{\Omega_{m}^{[d,e]}(\gamma)M^{0}[[\Gamma]]}\otimes_{\mathcal{O}_{\K}}p^{-c^{[d,e]}}\mathcal{O}_{\K}$ such that the image of $s_m^{[d,e]}$ by the natural projection 
$$\frac{M^{0}[[\Gamma]]}{\Omega_{m}^{[d,e]}(\gamma)M^{0}[[\Gamma]]}\otimes_{\mathcal{O}_{\K}}\K\rightarrow \frac{M^{0}[[\Gamma]]}{\Omega_{m}^{[i]}(\gamma)M^{0}[[\Gamma]]}\otimes_{\mathcal{O}_{\K}}\K$$
 is equal to $s_{m}^{[i]}\in \frac{M^{0}[[\Gamma]]}{\Omega_{m}^{[i]}(\gamma)M^{0}[[\Gamma]]}\otimes_{\mathcal{O}_{\K}}\K$ 
%
 for each $i\in [d,e]$, where 
\begin{equation}\label{generalization of porp 1 on Banach2}
c^{[d,e]}=\begin{cases}\ord_{p}((e-d)!)+2(e-d)+\lfloor\frac{e-d+1}{p-1}\rfloor+1\ &\mathrm{if}\ d<e,\\
0\ &\mathrm{if}\ d=e.\end{cases}
\end{equation}
\end{lem}
\begin{proof}

By identifying $M^{0}[[\Gamma ]]\otimes_{\mathcal{O}_{\K}}\K$ 
with $M^{0}[[X]]\otimes_{\mathcal{O}_{\K}}\K$ by the isomorphism $\alpha_{M}=\alpha_{M}^{(1)}$ 
defined in \eqref{non-canonical continuous isomorphihsm of iwasawa module of banach}, we regard 
$s_{m}^{[i]}$ as an element in $\frac{M^{0}[[X ]]}{\Omega_{m}^{[i]}M^{0}[[X]]} 
\otimes_{\mathcal{O}_{\K}}\K$ for each $i\in [d,e]$. Further, we regard $\tilde{s}_{m}^{[i]}$ and $\theta_{j}$ as elements of $M^{0}[[X]]\otimes_{\mathcal{O}_{\K}}\K$ for each $i,j\in [d,e]$. We will show that there exists a unique element 
$s_m^{[d,e]} \in 
\frac{M^{0}[[X]]}{\Omega_{m}^{[d,e]}M^{0}[[X]]}
\otimes_{\mathcal{O}_{\K}}p^{-c^{[d,e]}}\mathcal{O}_{\K}$ which satisfies
$$\tilde{s}_m^{[d,e]}(u^{i}\epsilon-1)=\widetilde{s}_{m}^{[i]}(u^{i}\epsilon-1)$$
for every $i\in [d,e]$ and for every $\epsilon\in \mu_{p^{m}}$ where $\tilde{s}_{m}^{[d,e]}\in M^{0}[[X]]\otimes_{\mathcal{O}_{\K}}p^{-c^{[d,e]}}\mathcal{O}_{\K}$ is a lift of $s_{m}^{[d,e]}$ and
$\Omega_{m}^{[d,e]} = \Omega_{m}^{[d,e]} (X)$ is 
the polynomial in $\mathcal{O}_{\K}[[X]]$ defined in \eqref{definition of Omega}. 
%
%
%
%
If $d=e$, the existence and the uniqueness of the desired element $s_m^{[d,e]}$ is trivial. Let us assume that $d<e$. The uniqueness of $s_m^{[d,e]}$ follows from $\mathrm{Corollary\ \ref{cor of remainder theorem}}$. We put $s(X,Y)=\sum_{i=0}^{e-d}\begin{pmatrix}Y-d\\ i\end{pmatrix}\theta_{i+d}(X)\in (M^{0}[[X]][Y])\otimes_{\mathcal{O}_{\K}}\K$, where
$$\begin{pmatrix}Y\\d\end{pmatrix}=\begin{cases}\frac{Y(Y-1)\cdots (Y-d+1)}{d!}\ &\mathrm{if}\ d\geq 1,\\
1\ &\mathrm{if}\ d=0.\end{cases}$$
Since $\theta_{i+d}(X)=\sum_{j=0}^{i}\begin{pmatrix}i\\j\end{pmatrix}(-1)^{i-j}\tilde{s}_{m}^{[j+d]}$, we have
\begin{align}\label{generalization of prop 1 on Banach1new eq1}
\begin{split}
s(X,i)&=\sum_{l=0}^{i-d}\begin{pmatrix}i-d\\ l\end{pmatrix}\sum_{j=0}^{l}\begin{pmatrix}l\\j\end{pmatrix}(-1)^{l-j}\tilde{s}_{m}^{[j+d]}\\
&=\sum_{j=0}^{i-d}\left(\sum_{l=j}^{i-d}(-1)^{l-j}\begin{pmatrix}i-d-j\\l-j\end{pmatrix}\right)\begin{pmatrix}i-d\\j\end{pmatrix}\tilde{s}_{m}^{[j+d]}\\
&=\sum_{j=0}^{i-d}\left(\sum_{l=0}^{i-d-j}(-1)^{l}\begin{pmatrix}i-d-j\\l\end{pmatrix}\right)\begin{pmatrix}i-d\\j\end{pmatrix}\tilde{s}_{m}^{[j+d]}\\
&=\tilde{s}_{m}^{[i]}
\end{split}
\end{align}
for each $i\in [d,e]$. Put $w=\log(1+(u-1))$. By the natural inclusion $M^{0}[[X]]\subset B_{+}(M)$, we regard $\tilde{s}_{m}^{[i]}$ with $i\in [d,e]$ as an element of $B_{+}(M)$ and 
we define $t(X) \in B_{+}(M)$ to be 
\begin{align*}
t(X)&=s(X,\log(1+X)\slash w)\\
&=\sum_{l=0}^{e-d}\begin{pmatrix}(\log(1+X)\slash w)-d\\ l\end{pmatrix}\theta_{l+d}(X).
\end{align*}
By \eqref{generalization of prop 1 on Banach1new eq1}, we have 
$t(u^{i}\epsilon-1)=s(u^{i}\epsilon-1,i)=\tilde{s}_{m}^{[i]}(u^{i}\epsilon-1) $ 
for each $i\in [d,e]$ and $\epsilon\in \mu_{p^{m}}$. We put $t_{m}=\frac{1}{p^{m}(p-1)}$. By $\mathrm{Proposition\ \ref{Weiestrass on Banach space}}$, there exists a unique pair $(g,r)\in B_{t_{m}}(M)\times M[X]$ such that $t=\Omega_{m}^{[d,e]}g+r$ and $\deg r<(e-d+1)p^{m}$. In addition, we have $v_{t_{m}}(t)=\min\{v_{t_{m}}(\Omega_{m}^{[d,e]})+v_{t_{m}}(g), v_{t_{m}}(r)\}$. By definition, $r$ satisfies 
\begin{equation}\label{generalization of prop 1 on Banach1new eq3}
r(u^{i}\epsilon-1)=\widetilde{s}_{m}^{[i]}(u^{i}\epsilon-1)
\end{equation}
for every $i\in [d,e]$ and for every $\epsilon\in \mu_{p^{m}}$. Next we prove that $r\in p^{-c^{[d,e]}}M^{0}[X]$. Since $\deg r<(e-d+1)p^{m}$, we see that $v_{0}(r)+t_{m}((e-d+1)p^{m}-1)\geq v_{t_{m}}(r)\geq v_{t_{m}}(t)$. Further, since $t_{m}((e-d+1)p^{m}-1)\leq \lfloor\frac{e-d+1}{p-1}\rfloor+1$, we have $ v_{0}(r)\geq v_{t_{m}}(t)-(\lfloor\frac{e-d+1}{p-1}\rfloor+1)$. Therefore, it suffices to prove that $v_{t_{m}}(t)\geq -c^{[d,e]}+(\lfloor\frac{e-d+1}{p-1}\rfloor+1)$. We have $v_{t_{m}}(t)=
\inf_{b\in \overline{\K}, \ord_{p}(b)>t_{m}}
\{v_{M_{\K(b)}}(t(b))\}$ 
by $\mathrm{Proposition\ \ref{supectral norm gauss norm}}$. Further, we see that $\displaystyle{\inf_{\substack{b\in \overline{\K}\\ \ord_{p}(b)>t_{m}}}}\{\ord_{p}(\log(1+b))\}=v_{t_{m}}(\log(1+X))> -m$ by $\mathrm{Proposition\ \ref{properties of log}}$. Thus, we have
\begin{align}\label{generalization of prop 1 on Banach1new eq2}
\begin{split}
v_{t_{m}}(t)&=\inf_{\substack{b\in \overline{\K}\\ \ord_{p}(b)>t_{m}}}\{v_{M_{\K(b)}}(s(b,\log(1+b)\slash w))\}\\
&\geq \inf_{\substack{(b,c)\in\overline{\K}^{2}\\ \ord_{p}(b)>t_{m}, \ord_{p}(c)>-(m+2)}}\{v_{M_{\K(b,c)}}(s(b,c))\}\\
&=\inf_{\substack{(b,c)\in\overline{\K}^{2}\\ \ord_{p}(b)>0, \ord_{p}(c)>0}}\{v_{M_{\K(b,c)}}(s(b,c\slash p^{m+2}))\}.
\end{split}
\end{align}
Since $\begin{pmatrix}Y\slash p^{m+2}\\ l\end{pmatrix}\in \frac{1}{(e-d)!p^{(m+2)l}}\mathcal{O}_{\K}[Y]$ and $\theta_{l+d}(X)\in p^{lm}M^{0}[[X]]$ for each $0\leq l\leq e-d$, we see that $s(X,Y\slash p^{m+2})$ is in $\frac{1}{(e-d)!p^{2(e-d)}}M^{0}[[X]][Y]$. It is easy to see that we have $\inf_{\substack{(b,c)\in\overline{\K}^{2}\\ \ord_{p}(b)>0, \ord_{p}(c)>0}}\{v_{M_{\K(b,c)}}(s(b,c\slash p^{m+2}))\}\geq v_{(0,0)}(s(X,\frac{Y}{p^{m+2}}))$ where $v_{(0,0)}$ is the valuation on $B_{(0,0)}(M)$. Then, by \eqref{generalization of prop 1 on Banach1new eq2}, we have
$$v_{t_{m}}(t)\geq  v_{(0,0)}(s(X,Y\slash p^{m+2}))\geq -\ord_{p}((e-d)!)-2(e-d).$$
Thus, $r\in M^{0}[[X]]\otimes_{\mathcal{O}_{\K}}p^{-c^{[d,e]}}\mathcal{O}_{\K}$. Put $s_{m}^{[d,e]}=[r]\in \frac{M^{0}[[X]]}{\Omega_{m}^{[d,e]}M^{0}[[X]]}\otimes_{\mathcal{O}_{\K}}p^{-c^{[d,e]}}\mathcal{O}_{\K}$ where $[r]$ is the class of $r$. Then, by \eqref{generalization of prop 1 on Banach1new eq3}, we see that $s_{m}^{[d,e]}$ satisfies the desired property. We complete the proof.
\end{proof}

Let $I_{h}^{[d,e]}(M)$ be the $\mathcal{O}_{\K}[[\Gamma]]\otimes_{\mathcal{O}_{\K}}\K$-module defined in \eqref{generalization of the project lim for deformation ring}. We put 
$$
I_{h}^{[d,e]}(M)^{0}=\left\{(s_{m})_{m}\in I_{h}^{[d,e]}(M)
\left\vert (p^{hm}s_{m})_{m}\in \prod_{m\in \mathbb{Z}_{\geq 0}}\frac{M^{0}[[\Gamma]]}{\Omega_{m}^{[d,e]}(\gamma)M^{0}[[X]]}\right. \right\}.
$$ 
Via $\alpha_{M}=\alpha_{M}^{(1)}$ in \eqref{non-canonical continuous isomorphihsm of iwasawa module of banach}, we can define a non-canonical $\mathcal{O}_{\K}$-module isomophirsm $I_{h}^{[d,e]}(M)^{0}\stackrel{\sim}{\rightarrow}J_{h}^{[d,e]}(M)^{0}$. By Lemma \ref{onevariable litfing prop}, we can generalize $\mathrm{Proposition\ \ref{prop1}}$ to a result on a $\K$-Banach space $M$.
\begin{pro}\label{one variable I(i) banach sufficient cond}
Let $s^{[i]}=(s^{[i]}_{m})_{m\in \mathbb{Z}_{\geq 0}}\in I_{h}^{[i]}$ and $\tilde{s}^{[i]}_{m}\in \mathcal{O}_{\K}[[\Gamma]]\otimes_{\mathcal{O}_{\K}}\K$ a lift of $s^{[i]}_{m}$ for each $m\in \mathbb{Z}_{\geq 0}$ and $i\in [d,e]$.  If there exists a non-negative integer $n$ which satisfies 
$$p^{m(h-(j-d))}\displaystyle{\sum_{i=d}^{j}}\begin{pmatrix}j-d\\i-d\end{pmatrix}(-1)^{j-i}\tilde{s}^{[i]}_{m}\in M^{0}[[\Gamma]]\otimes_{\mathcal{O}_{\K}}p^{-n}\mathcal{O}_{\K}$$
for each $m\in \mathbb{Z}_{\geq 0}$ and $j\in [d,e]$, we have a unique element $s^{[d,e]}\in I_{h}^{[d,e]}(M)^{0}\otimes_{\mathcal{O}_{\K}}p^{-c^{[d,e]}-n}\mathcal{O}_{\K}$ such that the image of $s^{[d,e]}$ by the natural projection $I_{h}^{[d,e]}(M)\rightarrow I_{h}^{[i]}(M)$ is $s^{[i]}$ for each $i\in [d,e]$, where $c^{[d,e]}$ is the constant defined in Lemma \ref{onevariable litfing prop}.
\end{pro}
\begin{proof}
For each $m\in \mathbb{Z}_{\geq 0}$, by Lemma \ref{onevariable litfing prop}, there exists a  unique element $s_{m}^{[d,e]}\in \linebreak
\frac{M^{0}[[\Gamma]]}{\Omega_{m}^{[d,e]}(\gamma)M^{0}[[\Gamma]]} 
\otimes_{\mathcal{O}_{\K}}p^{-hm-c^{[d,e]}-n}\mathcal{O}_{\K}$ such that the image of $s_{m}^{[d,e]}$ by the natural projection $
\frac{M^{0}[[\Gamma]]}{\Omega_{m}^{[d,e]}(\gamma)M^{0}[[\Gamma]]} 
\otimes_{\mathcal{O}_{\K}}\K\rightarrow 
\frac{M^{0}[[\Gamma]]}{\Omega_{m}^{[i]}(\gamma)M^{0}[[\Gamma]]}
\otimes_{\mathcal{O}_{\K}}\K$ is $s_{m}^{[i]}$ for each $i\in [d,e]$. Then, we see that$(p^{hm}s_{m}^{[d,e]})_{m\in \mathbb{Z}_{\geq 0}}\in\left(\prod_{m\in \mathbb{Z}_{\geq 0}}\frac{M^{0}[[\Gamma]]}{\Omega_{m}^{[d,e]}M^{0}[[\Gamma]]}\right)\otimes_{\mathcal{O}_{\K}}p^{-c^{[d,e]}-n}\mathcal{O}_{\K}$. Let $\tilde{s}_{m}^{[d,e]}$ be a lift of $s_{m}^{[d,e]}$. Since $s^{[i]}\in I_{h}^{[i]}$, we see that 
$$
\kappa(\tilde{s}_{m+1}^{[d,e]})=\kappa(\tilde{s}_{m+1}^{[w_{\kappa}]})=\kappa(\tilde{s}_{m}^{[w_{\kappa}]})=\kappa(\tilde{s}_{m}^{[d,e]})$$
for every $m\in \mathbb{Z}_{\geq 0}$ and for every $\kappa\in \mathfrak{X}_{\mathcal{O}_{\K}[[\Gamma]]}^{[d,e]}$. Therefore, by \eqref{onevariable iwasawa modoomega speq}, we see that $s_{m+1}^{[d,e]}\equiv s_{m}^{[d,e]}\ \mathrm{mod}\ \Omega_{m}^{[d,e]}$ for every $m\in \mathbb{Z}_{\geq 0}$ and we have $(s_{m}^{[d,e]})_{m\in \mathbb{Z}_{\geq 0}}\in \varprojlim_{m\in \mathbb{Z}_{\geq 0}}\left(\frac{M^{0}[[\Gamma]]}{\Omega_{m}^{[d,e]}M^{0}[[\Gamma]]}\otimes_{\mathcal{O}_{\K}}\K\right)$. Then, we have $(s_{m}^{[d,e]})_{m\in \mathbb{Z}_{\geq 0}}\in I_{h}^{[d,e]}(M)^{0}\otimes_{\mathcal{O}_{\K}}p^{-c^{[d,e]}-n}\mathcal{O}_{\K}$ and the image of $(s_{m}^{[d,e]})_{m\in \mathbb{Z}_{\geq 0}}$ by the natural projection $I_{h}^{[d,e]}(M)\rightarrow I_{h}^{[i]}(M)$ is $s^{[i]}$ for each $i\in [d,e]$.
\end{proof}
Let $\mathcal{D}_{h}^{[d,e]}(\Gamma,M)$ be the $\K$-Banach space of admissible distributions defined in \S\ref{preparation}. As seen in \S\ref{preparation}, $\mathcal{D}_{h}^{[d,e]}(\Gamma,M)$ is an $\mathcal{O}_{\K}[[\Gamma]]\otimes_{\mathcal{O}_{\K}}\K$-module. \begin{pro}\label{onevariable isom Ih from Dh for banach}
We have an $\mathcal{O}_{\K}[[\Gamma]]\otimes_{\mathcal{O}_{\K}}\K$-module isomorphism 
\begin{equation}\label{equation:mapfromDtoH banach}
\Psi:I_{h}^{[d,e]}(M)\stackrel{\sim}{\rightarrow} \mathcal{D}^{[d,e]}_h (\Gamma , M)
\end{equation}
such that the image $\mu_{s^{[d,e]}}\in \mathcal{D}^{[d,e]}_h (\Gamma , M)$ of each element $s^{[d,e]}=(s_{m}^{[d,e]})_{m\in \mathbb{Z}_{\geq 0}}\in I_{h}^{[d,e]}(M)$ is characterized 
by  
\begin{equation}\label{equation:mapfromDtoH banach onevariable interolationf}
\kappa(\tilde{s}_{m_{\kappa}}^{[d,e]})= \int_\Gamma \chi^{w_{\kappa}} \phi_{\kappa} d\mu_{s^{[d,e]}}\in M_{\K(\phi_{\kappa})}
\end{equation}
for every $\kappa\in\mathfrak{X}_{\mathcal{O}_{\K}[[\Gamma]]}^{[d,e]}$ where $\tilde{s}_{m_{\kappa}}^{[d,e]}$ is a lift of $s^{[d,e]}$.  Further, via the isomorphism in \eqref{equation:mapfromDtoH banach}, we have
\begin{align*}
\left\{\mu\in \mathcal{D}^{[d,e]}_h (\Gamma , M)\Big\vert v_{h}^{[d,e]}(\mu)\geq c^{[d,e]}\right\}&\subset I_{h}^{[d,e]}(M)^{0}\\
&\subset \left\{\mu\in \mathcal{D}^{[d,e]}_h (\Gamma , M)\Big\vert v_{h}^{[d,e]}(\mu)\geq 0\right\},
\end{align*}
where $c^{[d,e]}$ is the constant defined in \eqref{generalization of porp 1 on Banach2}.
\end{pro} 
\begin{proof}
To define a map from $I_{h}^{[d,e]}(M)$ into $\mathcal{D}^{[d,e]}_h (\Gamma , M)$, we prove that,
 for each $s^{[d,e]}\in I_{h}^{[d,e]}(M)$, there exists a unique element $\mu_{s^{[d,e]}}\in \mathcal{D}^{[d,e]}_h (\Gamma , M)$ which satisfies 
the condition  \eqref{equation:mapfromDtoH banach onevariable interolationf}. 
{Since each $\mu\in \mathcal{D}_{h}^{[d,e]}(\Gamma,M)$ is characterized by the specializations $\int_\Gamma \chi^{w_{\kappa}} \phi_{\kappa} d\mu$ for every $\kappa\in \mathfrak{X}_{\mathcal{O}_{\K}[[\Gamma]]}^{[d,e]}$, we see that $\mu_{s^{[d,e]}}$ which satisfies \eqref{equation:mapfromDtoH banach onevariable interolationf} is unique. 
The desired map $\Psi $ is defined if we prove the existence of $\mu_{s^{[d,e]}}$. 
\par 
First, we will prove the existence of the desired element $\mu_{s^{[d,e]}} \in \mathrm{Hom}_{\mathcal{O}_{\K}}(C^{[d,e]}(\Gamma,\mathcal{O}_{\K}),M)$ 
satisfying the condition \eqref{equation:mapfromDtoH banach onevariable interolationf}. 
Let $s^{[d,e]}=(s_{m}^{[d,e]})_{m\in \mathbb{Z}_{\geq 0}}\in I_{h}^{[d,e]}(M)$. 
We can assume that $s^{[d,e]}\in I_{h}^{[d,e]}(M)^{0}$. For each $m\in \mathbb{Z}_{\geq 0}$, we denote by $C_{m}^{[i]}(\Gamma,\mathcal{O}_{\K})$ the free $\mathcal{O}_{\K}$-submodule of $C^{[i]}(\Gamma,\mathcal{O}_{\K})$ generated by $\chi^{i}(x)1_{a\Gamma^{p^{m}}}(x)$ with $a\in \Gamma\slash \Gamma^{p^{m}}$. Here $1_{a\Gamma^{p^{m}}}(x):\Gamma\rightarrow \mathcal{O}_{\K}$ is the characteristic function of the open subset $a\Gamma^{p^{m}}$ of $\Gamma$. We note that $\mathrm{Hom}_{\mathcal{O}_{\K}}(C^{[i]}_{m}(\Gamma,\mathcal{O}_{\K}),\mathcal{O}_{\K})$ is an $\mathcal{O}_{\K}$-algebra by the natural convolution.  We can define an $\mathcal{O}_{\K}$-algebra isomorphism 
\begin{equation}\label{onevariable isom Ih from Dh eq finite group ring Omegaisom}
\mathcal{O}_{\K}[\Gamma\slash \Gamma^{p^{m}}]\stackrel{\sim}{\rightarrow}\mathcal{O}_{\K}[[\Gamma]]\slash (\Omega_{m}^{[i]}(\gamma))
\end{equation}
to be $\sum_{a\in \Gamma\slash \Gamma^{p^{m}}}c_{a}[a]\mapsto \sum_{a\in \Gamma\slash \Gamma^{p^{m}}}c_{a}\chi^{-i}(a)[a]$ with $c_{a}\in \mathcal{O}_{\K}$ and an $\mathcal{O}_{\K}$-algebra isomorphism
\begin{equation}\label{onevariable isom Ih from Dh eq finite group ring muiisom}
\mathcal{O}_{\K}[\Gamma\slash \Gamma^{p^{m}}]\stackrel{\sim}{\rightarrow}\mathrm{Hom}_{\mathcal{O}_{\K}}(C_{m}^{[i]}(\Gamma,\mathcal{O}_{\K}),\mathcal{O}_{\K})
\end{equation}
to be $\sum_{a\in \Gamma\slash \Gamma^{p^{m}}}c_{a}[a]\mapsto \sum_{a\in \Gamma\slash \Gamma^{p^{m}}}c_{a}\mu_{a}^{(i)}$ with $c_{a}\in \mathcal{O}_{\K}$ where $\mu_{a}^{(i)}$ is the mesure defined by $\mu_{a}^{(i)}(\chi(x)^{i}1_{a\Gamma^{p^{m}}}(x))=1$ and $\mu_{a}^{(i)}(\chi(x)^{i}1_{b\Gamma^{p^{m}}}(x))=0$  for every $b\in \Gamma\slash \Gamma^{p^{m}}$ such that $b\neq a$. By the isomorphisms \eqref{onevariable isom Ih from Dh eq finite group ring Omegaisom} and \eqref{onevariable isom Ih from Dh eq finite group ring muiisom}, we have an $\mathcal{O}_{\K}$-algebra isomorphism
\begin{equation}\label{onevariable isom Ih from Dh eq Omega ring muiisom}
\mathcal{O}_{\K}[[\Gamma]]\slash (\Omega_{m}^{[i]}(\gamma))\stackrel{\sim}{\rightarrow}\mathrm{Hom}_{\mathcal{O}_{\K}}(C_{m}^{[i]}(\Gamma,\mathcal{O}_{\K}),\mathcal{O}_{\K}).
\end{equation}
Since $\frac{M^{0}[[\Gamma]]}{\Omega_{m}^{[i]}(\gamma)M^{0}[[\Gamma]]}\otimes_{\mathcal{O}_{\K}}\K$ is isomorphic to $\frac{\mathcal{O}_{\K}[[\Gamma]]}{(\Omega_{m}^{[i]}(\gamma))}\otimes_{\mathcal{O}_{\K}}M$ and $\mathrm{Hom}_{\mathcal{O}_{\K}}(C_{m}^{[i]}(\Gamma,\mathcal{O}_{\K}),M)$ is isomorphic to $\mathrm{Hom}_{\mathcal{O}_{\K}}(C_{m}^{[i]}(\Gamma,\mathcal{O}_{\K}),\mathcal{O}_{\K})\otimes_{\mathcal{O}_{\K}}M$, the isomorphism \eqref{onevariable isom Ih from Dh eq Omega ring muiisom} induces a $\K$-linear isomorphism
$$\frac{M^{0}[[\Gamma]]}{\Omega_{m}^{[i]}(\gamma)M^{0}[[\Gamma]]}\otimes_{\mathcal{O}_{\K}}\K\stackrel{\sim}{\rightarrow}\mathrm{Hom}_{\mathcal{O}_{\K}}(C_{m}^{[i]}(\Gamma,\mathcal{O}_{\K}),M)$$
naturally. Since we have a natural isomorphism $\varinjlim_{m\in \mathbb{Z}_{\geq 0}}C_{m}^{[i]}(\Gamma,\mathcal{O}_{\K})\stackrel{\sim}{\rightarrow}C^{[i]}(\Gamma,\mathcal{O}_{\K})$, we see that 
\begin{align*}
\varprojlim_{m\in \mathbb{Z}_{\geq 0}}\left(
\frac{M^{0}[[\Gamma]]}{ \Omega_{m}^{[i]}(\gamma)M^{0}[[\Gamma]]}
\otimes_{\mathcal{O}_{\K}}\K
\right)
& 
\stackrel{\sim}{\rightarrow}\varprojlim_{m\in\mathbb{Z}_{\geq 0}}\mathrm{Hom}_{\mathcal{O}_{\K}}(C_{m}^{[i]}(\Gamma,\mathcal{O}_{\K}),M)
\\ 
& \simeq \mathrm{Hom}_{\mathcal{O}_{\K}}(C^{[i]}(\Gamma,\mathcal{O}_{\K}),M). 
\end{align*}
Since $I_{h}^{[d,e]}(M)$ is a $\K$-linear subspace of $\varprojlim_{m\in \mathbb{Z}_{\geq 0}}\left(
\frac{M^{0}[[\Gamma]]}{\Omega_{m}^{[d,e]}(\gamma)M^{0}[[\Gamma]]}\otimes_{\mathcal{O}_{\K}}\K\right)$ and there exists a natural injective map $\varprojlim_{m\in \mathbb{Z}_{\geq 0}}\left(
\frac{M^{0}[[\Gamma]]}{\Omega_{m}^{[d,e]}(\gamma)M^{0}[[\Gamma]]}\otimes_{\mathcal{O}_{\K}}\K\right)\hookrightarrow \prod_{i=d}^{e}\varprojlim_{m\in\mathbb{Z}_{\geq0}}\bigg(
\frac{M^{0}[[\Gamma]]}{\Omega_{m}^{[i]}(\gamma)M^{0}[[\Gamma]]}\linebreak\otimes_{\mathcal{O}_{\K}}\K\bigg)$, we have an injective map 
\begin{align}\label{onevariable isom Ih from Dh eq Ihinject prod hom}
\begin{split}
I_{h}^{[d,e]}(M)&\hookrightarrow \prod_{i=d}^{e}\varprojlim_{m\in\mathbb{Z}_{\geq0}}\left(
\frac{M^{0}[[\Gamma]]}{\Omega_{m}^{[i]}(\gamma)M^{0}[[\Gamma]]}\otimes_{\mathcal{O}_{\K}}\K\right)\\
&\stackrel{\sim}{\rightarrow}\prod_{i=d}^{e}\mathrm{Hom}_{\mathcal{O}_{\K}}(C^{[i]}(\Gamma,\mathcal{O}_{\K}),M).
\end{split}
\end{align}
We remark that we have a natural $\K$-linear isomorphism 
\begin{equation}\label{onevariable isom Ih from Dh eq  prod hom and homde}
\mathrm{Hom}_{\mathcal{O}_{\K}}(C^{[d,e]}(\Gamma,\mathcal{O}_{\K}),M)\stackrel{\sim}{\rightarrow}\prod_{i=d}^{e}\mathrm{Hom}_{\mathcal{O}_{\K}}(C^{[i]}(\Gamma,\mathcal{O}_{\K}),M)\end{equation}
defined by $\mu\mapsto (\mu\vert_{C^{[i]}(\Gamma,\mathcal{O}_{\K})})_{i=d}^{e}$. By \eqref{onevariable isom Ih from Dh eq Ihinject prod hom} and \eqref{onevariable isom Ih from Dh eq  prod hom and homde}, we have a $\K$-linear injective map
\begin{equation}\label{onevariable isom Ih from Dh eq inj ih[d,e] into homde}
I_{h}^{[d,e]}(M) \hookrightarrow \mathrm{Hom}_{\mathcal{O}_{\K}}(C^{[d,e]}(\Gamma,\mathcal{O}_{\K}),M).
\end{equation}
For each $s^{[d,e]}\in I_{h}^{[d,e]}(M)$, we denote by $\mu_{s^{[d,e]}}\in \mathrm{Hom}_{\mathcal{O}_{\K}}(C^{[d,e]}(\Gamma,\mathcal{O}_{\K}),M)$ the image of $s^{[d,e]}$ by \eqref{onevariable isom Ih from Dh eq inj ih[d,e] into homde}. By the construction of \eqref{onevariable isom Ih from Dh eq inj ih[d,e] into homde}, we see that $\mu_{s^{[d,e]}}$ satisfies the condition  \eqref{equation:mapfromDtoH banach onevariable interolationf} for every $\kappa\in\mathfrak{X}_{\mathcal{O}_{\K}[[\Gamma]]}^{[d,e]}$. }
\par 
Next, we will prove that $\mu_{s^{[d,e]}}\in \mathcal{D}_{h}^{[d,e]}(\Gamma,M)$ and $v_{h}^{[d,e]}(\mu_{s^{[d,e]}})\geq 0$ for each $s^{[d,e]}\in I_{h}^{[d,e]}(M)^{0}$. Let $\tilde{s}_{m}^{[d,e]}\in p^{-hm}M^{0}[[\Gamma]]$ be a lift of $s_{m}^{[d,e]}$ for each $m\in \mathbb{Z}_{\geq 0}$. Let $m\in \mathbb{Z}_{\geq 0}$ and $\nu_{m}\in p^{-hm}\mathrm{Hom}_{\mathcal{O}_{\K}}(C(\Gamma,\mathcal{O}_{\K}),M^{0})$ the inverse image of $\tilde{s}_{m}^{[d,e]}$ by the isomorphism \eqref{hom banach iwasawa continuous}. Then, we have 
\begin{equation}\label{onevariable isom Ih from Dh eqint musdeintgammanum}
\int_{\Gamma}\kappa\vert_{\Gamma}d\mu_{s^{[d,e]}}=\kappa(\tilde{s}_{m}^{[d,e]})=\int_{\Gamma}\kappa\vert_{\Gamma}d\nu_{m}
\end{equation}
for each $\kappa\in \mathfrak{X}_{\mathcal{O}_{\K}[[\Gamma]]}^{[d,e]}$ with $m_{\kappa}\leq m$. For each $a\in \Gamma$ and $i\in [d,e]$, we have
$$
1_{a\Gamma^{p^{m}}}(x)\chi(x)^{i}=\frac{1}{p^{m}}\sum_{\substack{\kappa\in \mathfrak{X}_{\mathcal{O}_{\K}[[\Gamma]]}^{[i]}\\ m_{\kappa}\leq m}}\phi_{\kappa}(a)^{-1}\kappa\vert_{\Gamma}(x)$$
by the inverse Fourier transform. Then, by \eqref{onevariable isom Ih from Dh eqint musdeintgammanum}, we have
\begin{equation}\label{onevariable isom Ih from Dh eqint hiadmmusadchinum}
\int_{a\Gamma^{p^{m}}}(\chi(x)-\chi(a))^{i-d}\chi(x)^{d}d\mu_{s^{[d,e]}}=\int_{a\Gamma^{p^{m}}}(\chi(x)-\chi(a))^{i-d}\chi(x)^{d}d\nu_{m}
\end{equation}
for each $a\in \Gamma$ and $i\in [d,e]$.
Since $\nu_{m}\in p^{-hm}\mathrm{Hom}_{\mathcal{O}_{\K}}(C(\Gamma,\mathcal{O}_{\K}),M^{0})$, by \eqref{continuous measure and spectral norm}, we see that $v_{M}\left(\int_{\Gamma}f(x)d\nu_{m}\right)\geq \inf\{\ord_{p}(f(x))\}_{x\in \Gamma}-hm$ for each $f\in C(\Gamma,\mathcal{O}_{\K})$. In particular, we have
\begin{align}\label{onevariable isom Ih from Dhvm(chi-chi)chinumgeq --dim}
\begin{split}
v_{M}\left(\int_{a\Gamma^{p^{m}}}(\chi(x)-\chi(a))^{i-d}\chi(x)^{d}d\nu_{m}\right)&\geq\inf\left\{\ord_{p}\left((\chi(x)-\chi(a))^{i-d}\chi(x)^{d}1_{a\Gamma^{p^{m}}}(x)\right)\right\}_{x\in \Gamma}\\
&-hm\geq -(h-(i-d))m.
\end{split}
\end{align}
By \eqref{onevariable isom Ih from Dh eqint hiadmmusadchinum} and \eqref{onevariable isom Ih from Dhvm(chi-chi)chinumgeq --dim}, we have
\begin{align*}
v_{M}\left(\int_{a\Gamma^{p^{m}}}(\chi(x)-\chi(a))^{i-d}\chi(x)^{d}d\mu_{s^{[d,e]}}\right)\geq -(h-(i-d))m.
\end{align*}
Thus, we have $\mu_{s^{[d,e]}}\in \mathcal{D}_{h}^{[d,e]}(\Gamma,M)$ and 
\begin{equation}\label{onevariable isom Ih from Dhvdhde,musdegeq0}
v_{h}^{[d,e]}(\mu_{s^{[d,e]}})\geq 0
\end{equation}
for each $s^{[d,e]}\in I_{h}^{[d,e]}(M)^{0}$. Therefore, we have defined  the desired map \eqref{equation:mapfromDtoH banach} from $I_{h}^{[d,e]}(M)$ into $D_{h}^{[d,e]}(\Gamma,M)$. 
\par 
Up to now, we have defined the map $\Psi$. We will prove that $\Psi$ is an isomorphism in the rest of the proof.
\par 
We  prove the injectivity of the map $\Psi$. Let $s^{[d,e]}=(s_{m}^{[d,e]})_{m\in \mathbb{Z}_{\geq 0}}\in I_{h}^{[d,e]}(M)$ such that $\Psi(s^{[d,e]})=0$. Since $\Psi(s^{[d,e]})=0$, we have 
$$
\kappa(\tilde{s}_{m_{\kappa}}^{[d,e]})=0$$
for every $\kappa\in \mathfrak{X}_{\mathcal{O}_{\K}[[\Gamma]]}^{[d,e]}$ where $\tilde{s}_{m_{\kappa}}^{[d,e]}$ is a lift of $s_{m_{\kappa}}^{[d,e]}$. Thus, by \eqref{onevariable iwasawa modoomega speq}, we see that $\tilde{s}_{m}^{[d,e]}\in \Omega_{m}^{[d,e]}M^{0}[[\Gamma]]\otimes_{\mathcal{O}_{\K}}\K$ for every $m\in \mathbb{Z}_{\geq 0}$ and we have $s^{[d,e]}=0$. Therefore, the map of \eqref{equation:mapfromDtoH banach} is injective.
\par 
By the injectivity of the map $\Psi$, we  can regard $I_{h}^{[d,e]}(M)$ as an $\mathcal{O}_{\K}[[\Gamma]]\otimes_{\mathcal{O}_{\K}}\K$-module subspace of $\mathcal{D}_{h}^{[d,e]}(\Gamma,M)$. Further, by \eqref{onevariable isom Ih from Dhvdhde,musdegeq0}, we have $I_{h}^{[d,e]}(M)^{0}\subset \left\{\mathcal{D}^{[d,e]}_h (\Gamma , M)\Big\vert v_{h}^{[d,e]}(\mu)\geq 0\right\}$. If we have $\Big\{\mu\in \mathcal{D}^{[d,e]}_h (\Gamma , M)\Big\vert v_{h}^{[d,e]}(\mu)\geq c^{[d,e]}\Big\}\subset I_{h}^{[d,e]}(M)^{0}$, we see that the map is surjective easily. 
\par 
To complete the proof, it suffices to prove that 
\begin{equation}\label{onevariable isom Ih from Dhvdh suffcondi surjective}
\Big\{\mu\in \mathcal{D}^{[d,e]}_h (\Gamma , M)\Big\vert v_{h}^{[d,e]}(\mu)\geq c^{[d,e]}\Big\}\subset I_{h}^{[d,e]}(M)^{0}.
\end{equation}
Let $\mu\in \mathcal{D}_{h}^{[d,e]}(\Gamma,M)$ with $v_{h}^{[d,e]}(\mu)\geq c^{[d,e]}$. Let $m\in \mathbb{Z}_{\geq 0}$ and $i\in [d,e]$. We define 
$$
r_{m}^{[i]}=\sum_{l=0}^{p^{m}-1}
\int_{\gamma^{l}\Gamma^{p^{m}}}\chi(x)^{i}d\mu(u^{-i}[\gamma])^{l}\in M^{0}[[\Gamma]]\otimes_{\mathcal{O}_{\K}}\K.
$$
We note that $r_{m}^{[i]}$ satisfies 
\begin{equation}\label{onevariable isom Ih from Dhvdh sp of rmi}
\kappa(r_{m}^{[i]})=\sum_{l=0}^{p^{m}-1}\int_{\gamma^{l}\Gamma^{p^{m}}}\chi(x)^{i}\phi_{\kappa}(\gamma^{l})d\mu =\int_{\Gamma}\kappa\vert_{\Gamma}\mu
\end{equation}
for every $\kappa\in \mathfrak{X}^{[i]}_{\mathcal{O}_{\K}[[\Gamma]]}$ such that $m_{\kappa}\leq m$. Thus, we see that
\begin{equation}\label{onevariable isom Ih from Dhvminj rm+1rmeqiv}
\kappa(r_{m+1}^{[i]})=\kappa(r_{m}^{[i]})
\end{equation}
for every $m\in \mathbb{Z}_{\geq 0}$ and for every $\kappa\in \mathfrak{X}_{\mathcal{O}_{\K}[[\Gamma]]}^{[i]}$ with $m_{\kappa}\leq m$.  By the definition of $r_{m}^{[i]}$, we have
\begin{align*}
&\sum_{i\in [d,j]}\begin{pmatrix}j-d\\i-d\end{pmatrix}
(-1)^{j-i}r_{m}^{[i]}\\
&=\sum_{l=0}^{p^{m}-1}u^{-lj}\int_{\gamma^{l}\Gamma^{p^{m}}}\sum_{i\in[d,j]}\begin{pmatrix}j-d\\i-d\end{pmatrix}(-u^{l})^{j-i}\chi(x)^{i} d\mu[\gamma]^{l}\\
&=\sum_{l=0}^{p^{m}-1}u^{-lj}\int_{\gamma^{l}\Gamma^{p^{m}}}(\chi(x)-u^{l})^{j-d}\chi(x)^{d} d\mu[\gamma]^{l}\in p^{c^{[d,e]}+(j-d-h)m}M^{0}[[\Gamma]].
\end{align*}
Therefore, by Lemma \ref{onevariable litfing prop}, there exists a unique element $s_{m}^{[d,e]}\in \frac{M^{0}[[\Gamma]]}{\Omega_{m}^{[d,e]}M^{0}[[\Gamma]]}\otimes_{\mathcal{O}_{\K}}p^{-hm}\mathcal{O}_{\K}$ such that the image of $s^{[d,e]}$ by the natural projection $\frac{M^{0}[[\Gamma]]}{\Omega_{m}^{[d,e]}M^{0}[[\Gamma]]}\otimes_{\mathcal{O}_{\K}}\K\rightarrow \frac{M^{0}[[\Gamma]]}{\Omega_{m}^{[i]}M^{0}[[\Gamma]]}\otimes_{\mathcal{O}_{\K}}\K$ is $[r_{m}^{[i]}]$ for every $i\in [d,e]$. By \eqref{onevariable isom Ih from Dhvminj rm+1rmeqiv}, we have
$$\kappa(\tilde{s}^{[d,e]}_{m+1})=\kappa(r_{m+1}^{[w_{\kappa}]})=\kappa(r_{m}^{[w_{\kappa}]})=\kappa(\tilde{s}_{m}^{[d,e]})$$
for every $m\in \mathbb{Z}_{\geq 0}$ and for every $\kappa\in \mathfrak{X}_{\mathcal{O}_{\K}[[\Gamma]]}^{[d,e]}$ with $m_{\kappa}\leq m$. By \eqref{onevariable iwasawa modoomega speq}, we have $s_{m+1}^{[d,e]}\equiv s_{m}^{[d,e]}\ \mathrm{mod}\ \Omega_{m}^{[d,e]}$ for every $m\in \mathbb{Z}_{\geq 0}$. Thus, we see that $(s_{m}^{[d,e]})_{m\in \mathbb{Z}_{\geq 0}}\in \varprojlim_{m\in \mathbb{Z}_{\geq 0}}\bigg(\frac{M^{0}[[\Gamma]]}{\Omega_{m}^{[d,e]}M^{0}[[\Gamma]]}\linebreak\otimes_{\mathcal{O}_{\K}}\K\bigg)$. Since $s_{m}^{[d,e]}\in \frac{M^{0}[[\Gamma]]}{\Omega_{m}^{[d,e]}M^{0}[[\Gamma]]}\otimes_{\mathcal{O}_{\K}}p^{-hm}\mathcal{O}_{\K}$ for every $m\in \mathbb{Z}_{\geq 0}$, $(s_{m}^{[d,e]})_{m\in \mathbb{Z}_{\geq 0}}\in I_{h}^{[d,e]}(M)^{0}$. By \eqref{onevariable isom Ih from Dhvdh sp of rmi}, we see that $\Psi((s_{m}^{[d,e]})_{m\in \mathbb{Z}_{\geq 0}})=\mu$. Therefore, we have \eqref{onevariable isom Ih from Dhvdh suffcondi surjective}.
\end{proof}

\section{Proof of the main result for the case of the multi-variable Iwasawa algebra} \label{sc:ordinary}
In this section, we prove main results for the case of the multi-variable Iwasawa algebra. 
Let $k$ be a positive integer. 
We put $\boldsymbol{0}_{k}=(0,\ldots,0)\in \mathbb{Z}_{\geq 0}^{k}$. 
For each element $\boldsymbol{a}\in \mathbb{R}^{k}$ with $k\geq 2$, we put $\boldsymbol{a}^{\prime}=(a_{1},\ldots, a_{k-1})\in \mathbb{R}^{k-1}$. 
For each integer $i$ satisfying $1\leq i\leq k$, we set $\Gamma_{i}$ to be a $p$-adic Lie group which is isomorphic to $1+2p\mathbb{Z}_p\subset \mathbb{Q}_{p}^{\times}$ via a continuous character $\chi_{i} : \Gamma_{i} \longrightarrow \mathbb{Q}_{p}^{\times}$. 
For each $i$, we choose and fix a topological generator $\gamma_{i}\in \Gamma_{i}$ and we put $u_{i}=\chi_{i}(\gamma_{i})$.
We define $\Gamma=\Gamma_{1}\times \cdots \times\Gamma_{k}$. Put $\Gamma^{\prime}=\Gamma_{1}\times\cdots \times \Gamma_{k-1}$ if $k\geq 2$. In this section, we fix a $\K$-Banach space $(M,v_{M})$. 
\begin{thm}\label{main theorem 1 and proof}
Let $k$ be a positive integer. Let $\boldsymbol{h}\in \ord_{p}(\mathcal{O}_{\K}\backslash \{0\})^{k}$ and let $\boldsymbol{d}\in\mathbb{Z}^{k}$. 
 If $f\in \HH_{\boldsymbol{h}}(M)$ satisfies $f(u_{1}^{i_{1}}\epsilon_{1}-1,\ldots,u_{k}^{i_{k}}\epsilon_{k}-1)=0$ for every $k$-tuple $\boldsymbol{i}\in
 [\boldsymbol{d},\boldsymbol{d}+\lfloor\boldsymbol{h}\rfloor]$ and for every $(\epsilon_{1},\ldots, \epsilon_{k})\in \mu_{p^{\infty}}^{k}$, then $f$ is zero.
\end{thm}
\begin{proof}
We prove this theorem by induction on $k$. When $k=1$, the theorem is already proved in $\mathrm{Proposition\ \ref{generalization of uniqueness on M}}$. 
In the rest of the proof, we will prove the desired statement for general $k\geq 2$ assuming that it is already proved up to $k-1$. 
%
%
{For each $\boldsymbol{i}^{\prime}\in[\boldsymbol{d}^{\prime},\boldsymbol{d}^{\prime}+\lfloor\boldsymbol{h}^{\prime}\rfloor]$ and for 
each $(\epsilon_{1},\ldots, \epsilon_{k-1})\in \mu_{p^{\infty}}^{k-1}$, we define 
a $\K$-Banach homomorphism $$\phi_{\boldsymbol{i}^{\prime},(\epsilon_{1},\ldots, \epsilon_{k-1})}: \HH_{h_{k}}(\HH_{\boldsymbol{h}^{\prime}}(M))\rightarrow\HH_{h_{k}}(M_{\K(\epsilon_{1},\ldots,\epsilon_{k-1})})$$
by setting $\phi_{\boldsymbol{i}^{\prime},(\epsilon_{1},\ldots, \epsilon_{k-1})}((f_{n})_{n=0}^{+\infty})=(f_{n}(u_{1}^{i_{1}^{\prime}}\epsilon_{1}-1,\ldots, u_{k-1}^{i_{k-1}^{\prime}}\epsilon_{k-1}-1))_{n=0}^{+\infty}$ for each $(f_{n})_{n=0}^{+\infty}\in \HH_{h_{k}}(\HH_{\boldsymbol{h}^{\prime}}(M))$ and we define 
a map 
$$\phi:\HH_{h_{k}}(\HH_{\boldsymbol{h}^{\prime}}(M))\rightarrow \prod_{\substack{{\boldsymbol{d}^{\prime}}\leq {\boldsymbol{i}^{\prime}}\leq {\boldsymbol{d}^{\prime}}+\lfloor\boldsymbol{h}^{\prime}\rfloor\\ 
(\epsilon_{1},\ldots, \epsilon_{k-1})\in \mu_{p^{\infty}}^{k-1}}}\HH_{h_{k}}(M_{\K(\epsilon_{1},\ldots, \epsilon_{k-1})}),\ 
$$
by setting 
$ \phi (f ) = 
\prod_{\substack{{\boldsymbol{d}^{\prime}}\leq {\boldsymbol{i}^{\prime}}\leq {\boldsymbol{d}^{\prime}}+\lfloor\boldsymbol{h}^{\prime}\rfloor , 
(\epsilon_{1},\ldots, \epsilon_{k-1})\in \mu_{p^{\infty}}^{k-1}}}
 (\phi_{\boldsymbol{i}^{\prime},(\epsilon_{1},\ldots \epsilon_{k-1})}(f))
$.  
By the induction hypothesis, we see that $\phi$ is injective. By applying the result in the case $k=1$, for every $(\epsilon_{1},\ldots, \epsilon_{k-1})\in \mu_{p^{\infty}}^{k-1}$, we have an injective $\K(\epsilon_{1},\ldots, \epsilon_{k-1})$-linear map:
\begin{align*}
\psi_{(\epsilon_{1},\ldots,\epsilon_{k-1})}:\HH_{h_{k}}(M_{\K(\epsilon_{1},\ldots, \epsilon_{k-1})})&\hookrightarrow \prod_{\substack{d_{k}\leq i_{k}\leq d_{k}+\lfloor h_{k}\rfloor\\\epsilon_{k}\in\mu_{p}^{\infty}}}M_{\K(\epsilon_{1},\ldots, \epsilon_{k})}\\
f&\mapsto (f(u_{k}^{i_{k}}\epsilon_{k}-1))_{\substack{d_{k}\leq i_{k}\leq d_{k}+\lfloor h_{k}\rfloor\\ \epsilon_{k}\in\mu_{p^{\infty}}}}.
\end{align*}
Then, we have the following injective $\K$-linear map:
\begin{align*}
\psi: \prod_{\substack{{\boldsymbol{d}^{\prime}}\leq {\boldsymbol{i}^{\prime}}\leq {\boldsymbol{d}^{\prime}}+\lfloor\boldsymbol{h}^{\prime}\rfloor\\ 
(\epsilon_{1},\ldots, \epsilon_{k-1})\in \mu_{p^{\infty}}^{k-1}}}\HH_{h_{k}}(M_{\K(\epsilon_{1},\ldots, \epsilon_{k-1})})&\hookrightarrow \prod_{\substack{{\boldsymbol{d}^{\prime}}\leq {\boldsymbol{i}^{\prime}}\leq {\boldsymbol{d}^{\prime}}+\lfloor\boldsymbol{h}^{\prime}\rfloor\\ 
(\epsilon_{1},\ldots, \epsilon_{k-1})\in \mu_{p^{\infty}}^{k-1}}}\prod_{\substack{d_{k}\leq i_{k}\leq d_{k}+\lfloor h_{k}\rfloor\\\epsilon_{k}\in\mu_{p}^{\infty}}}M_{\K(\epsilon_{1},\ldots, \epsilon_{k})}\\
(f_{\boldsymbol{i}^{\prime},(\epsilon_{1},\ldots, \epsilon_{k-1})})_{\substack{{\boldsymbol{d}^{\prime}}\leq {\boldsymbol{i}^{\prime}}\leq {\boldsymbol{d}^{\prime}}+\lfloor\boldsymbol{h}^{\prime}\rfloor\\ 
(\epsilon_{1},\ldots, \epsilon_{k-1})\in \mu_{p^{\infty}}^{k-1}}}&\mapsto (\psi_{(\epsilon_{1},\ldots,\epsilon_{k-1})}(f_{\boldsymbol{i}^{\prime},(\epsilon_{1},\ldots, \epsilon_{k-1})}))_{\substack{{\boldsymbol{d}^{\prime}}\leq {\boldsymbol{i}^{\prime}}\leq {\boldsymbol{d}^{\prime}}+\lfloor\boldsymbol{h}^{\prime}\rfloor\\ 
(\epsilon_{1},\ldots, \epsilon_{k-1})\in \mu_{p^{\infty}}^{k-1}}}.
\end{align*}
The injective maps $\phi$ and $\psi$ and the isometric isomorphism $\HH_{\boldsymbol{h}}(M)\simeq \HH_{h_{k}}(\HH_{\boldsymbol{h}^{\prime}}(M))$ of $\mathrm{Proposition\ \ref{isometry ofHh for induction}}$ induce the following injective $\K$-linear map: 
\begin{align*}
\HH_{\boldsymbol{h}}(M)&\simeq \HH_{h_{k}}(\HH_{\boldsymbol{h}^{\prime}}(M))\stackrel{\phi}{  \hookrightarrow}\prod_{\substack{{\boldsymbol{d}^{\prime}}\leq {\boldsymbol{i}^{\prime}}\leq {\boldsymbol{d}^{\prime}}+\lfloor\boldsymbol{h}^{\prime}\rfloor\\ 
(\epsilon_{1},\ldots, \epsilon_{k-1})\in \mu_{p^{\infty}}^{k-1}}}\HH_{h_{k}}(M_{\K(\epsilon_{1},\ldots, \epsilon_{k-1})})\\
&\stackrel{\psi}{\hookrightarrow}\prod_{\substack{{\boldsymbol{d}^{\prime}}\leq {\boldsymbol{i}^{\prime}}\leq {\boldsymbol{d}^{\prime}}+\lfloor\boldsymbol{h}^{\prime}\rfloor\\ 
(\epsilon_{1},\ldots, \epsilon_{k-1})\in \mu_{p^{\infty}}^{k-1}}}\prod_{\substack{d_{k}\leq i_{k}\leq d_{k}+\lfloor h_{k}\rfloor\\\epsilon_{k}\in\mu_{p}^{\infty}}}M_{\K(\epsilon_{1},\ldots, \epsilon_{k})}\simeq \prod_{\substack{\boldsymbol{d}\leq \boldsymbol{i}\leq \boldsymbol{d}+\lfloor\boldsymbol{h}\rfloor\\ (\epsilon_{1},\ldots, \epsilon_{k})\in \mu_{p^{\infty}}^{k}}}M_{\K(\epsilon_{1},\ldots, \epsilon_{k})}.
\end{align*}
We note that the composite of the above injective maps is equal to the map sending 
$f$ to $(f(u_{1}^{i_{1}}\epsilon_{1}-1,\ldots, u_{k}^{i_{k}}\epsilon_{k}-1))
_{{\boldsymbol{d}}\leq {\boldsymbol{i}}\leq {\boldsymbol{d}+\lfloor\boldsymbol{h}\rfloor,  (\epsilon_{1},\ldots, \epsilon_{k})\in \mu_{p^{\infty}}^{k}}}$.} 
The desired conclusion of the theorem follows by the injectivity of this composite map.
\end{proof}
Let $\boldsymbol{r} =(r_i)_{1\leq i\leq k}\in \mathbb{Q}^{k}$. In the following Proposition \ref{multivariable weiestrass polynomial pro} and Corollary \ref{lemma for main theorem two}, we regard $B_{(r_{1},\ldots, r_{i})}(M)$ as a $\K$-subspace of $M[[X_{1},\ldots, X_{i}]]$ for each $1\leq i\leq k$.
\begin{pro}\label{multivariable weiestrass polynomial pro}
Let $\boldsymbol{r} =(r_i)_{1\leq i\leq k}\in \mathbb{Q}^{k}$. For each $1\leq i\leq k$, 
we choose $f_{i}\in B_{r_{i}}^{\mathrm{md}}(\K)\backslash \{0\}$ and set $s_i = d_{r_{i}}(f_{i})$. Then for each $f\in B_{\boldsymbol{r}}(M)$, there exist a unique  $q_{i}\in B_{(r_{1},\ldots, r_{i})}(M)[X_{i+1},\ldots, X_{k}]$ for each $1\leq i\leq k$ and a unique $t\in M[X_{1},\ldots, X_{k}]$ which satisfy the following conditions:
\begin{enumerate}
\item We have $f=f_{1}(X_{1})q_{1}+\cdots +f_{k}(X_{k})q_{k}+t$.\label{multivariable weiestrass polynomial pro1}
\item We have $\deg_{X_{i}} t<s_{i}$ for each $1\leq i\leq k$.\label{multivariable weiestrass polynomial pro2}
\item For each $1\leq i<k$, $q_{i}\in B_{(r_{1},\ldots, r_{i})}(M)[X_{i+1},\ldots, X_{k}]$ satisfies $\deg_{X_{j}}q_{i}<s_{j}$ for each $i+1\leq j\leq k$.\label{multivariable weiestrass polynomial pro3}
\end{enumerate}
In addition, we have 
\begin{equation}\label{multivariable weiestrass polynomial pro vboldsymbolrf=pro}
v_{\boldsymbol{r}}(f)=\min\{v_{r_{1}}(f_{1})+v_{\boldsymbol{r}}(q_{1}),\ldots, v_{r_{k}}(f_{k})+v_{\boldsymbol{r}}(q_{k}),v_{\boldsymbol{r}}(t)\}.
\end{equation}
\end{pro}
\begin{proof}
Let $f\in B_{\boldsymbol{r}}(M)$. When $k=1$, Proposition \ref{multivariable weiestrass polynomial pro} is already proved in $\mathrm{Proposition\ \ref{Weiestrass on Banach space}}$ (note that the condition (3) is an empty condition when $k=1$). We assume that $k\geq 2$ and assume that the proposition is already proved for $k-1$. First, we prove the uniqueness of $q_{1},\ldots, q_{k}$ and $t$, which reduces to showing that 
$f_{1}(X_{1})q_{1}+\cdots +f_{k}(X_{k})q_{k}+t=0$ implies $q_{1}=\cdots =q_{k}=t=0$. Put $h=f_{1}(X_{1})q_{1}+\ldots+f_{k-1}(X_{k-1})q_{k-1}+t$. Via the isomorphism of Proposition \ref{isometry ofHh for induction}, we identify $B_{\boldsymbol{r}}(M)$ with $B_{r_{k}}(B_{\boldsymbol{r}^{\prime}}(M))$. Then, we have $f_{k}(X_{k})q_{k}+h=0$ in $B_{r_{k}}(B_{\boldsymbol{r}^{\prime}}(M))$. Further, since $q_{1},\ldots, q_{k-1}$ and $t$ satisfy the conditions \eqref{multivariable weiestrass polynomial pro2} and \eqref{multivariable weiestrass polynomial pro3}, we see that $h\in B_{\boldsymbol{r}^{\prime}}(M)[X_{k}]$ and $\deg_{X_{k}}h<s_{k}$. Therefore, by applying the result in the case $k=1$, we have $h=q_{k}=0$. Put $q_{i}=\sum_{l=0}^{s_{k}-1}X_{k}^{l}q_{i}^{(l)}$ for each $1\leq i\leq k-1$ and $t=\sum_{l=0}^{s_{k}-1}X_{k}^{l}t^{(l)}$, where $q_{i}^{(l)}\in B_{(r_{1},\ldots, r_{i})}(M)[X_{i+1},\ldots, X_{k-1}]$ and $t^{(l)}\in M[X_{1},\ldots, X_{k-1}]$. Since $h=\sum_{l=0}^{s_{k}-1}X_{k}^{l}(f_{1}(X_{1})q_{1}^{(l)}+\cdots f_{k-1}(X_{k-1})q_{k-1}^{(l)}+t^{(l)})=0$, we see that $f_{1}(X_{1})q_{1}^{(l)}+\cdots+f_{k-1}(X_{k-1})q_{k-1}^{(l)}+t^{(l)}=0$ for each $0\leq l<s_{k}$. Let $0\leq l<s_{k}$. By the condition \eqref{multivariable weiestrass polynomial pro2}, we have $\deg_{X_{i}}t^{(l)}<s_{i}$ for each $1\leq i\leq k-1$. Further, by the condition \eqref{multivariable weiestrass polynomial pro3}, for each $1\leq i<k-1$, we see that $\deg_{X_{j}}q_{i}^{(l)}<s_{j}$ for each $i+1\leq j\leq k-1$. Therefore, by induction on $k$, we have $q_{i}^{(l)}=0$ for each $0\leq i\leq k-1$ and $t^{(l)}=0$. Thus, $q_{i}=\sum_{l=0}^{s_{k}-1}X_{k}^{l}q_{i}^{(l)}=0$ for each $0\leq i\leq k-1$ and $t=\sum_{l=0}^{s_{k}-1}X_{k}^{l}t^{(l)}=0$. We get the uniqueness.

Next, we prove the existence $q_{1},\ldots, q_{k}$ and $t$. We also prove the estimate $v_{\boldsymbol{r}}(f)=\min\{v_{r_{1}}(f_{1})+v_{\boldsymbol{r}}(q_{1}),\ldots, v_{r_{k}}(f_{k})+v_{\boldsymbol{r}}(q_{k}),v_{\boldsymbol{r}}(t)\}$ simultaneously. We regard $f$ as an element of $B_{r_{k}}(B_{\boldsymbol{r}^{\prime}}(M))$. Since the isomorphism from $B_{\boldsymbol{r}}(M)$ into $B_{r_{k}}(B_{\boldsymbol{r}^{\prime}}(M))$ in Proposition \ref{isometry ofHh for induction} is isometric, we identify $v_{\boldsymbol{r}}$ with the valuation on $B_{r_{k}}(B_{\boldsymbol{r}^{\prime}}(M))$. By the result in the case $k=1$, we have the following unique expression: 
$$f=f_{k}(X_{k})q_{k}+u,$$
where $q_{k}\in B_{r_{k}}(B_{\boldsymbol{r}^{\prime}}(M))$ and $u\in B_{\boldsymbol{r}^{\prime}}(M)[X_{k}]$ with $\deg_{X_{k}}u<s_{k}$. In addition, we get $v_{\boldsymbol{r}}(f)=\min\{v_{r_{k}}(f_{k})+v_{\boldsymbol{r}}(q_{k}),v_{\boldsymbol{r}}(u)\}$. Put $u=\sum_{l=0}^{s_{k}-1}X_{k}^{l}u^{(l)}$ with $u^{(l)}\in B_{{\boldsymbol{r}^{\prime}}}(M)$ for $0\leq l<s_{k}$. By the definition of the valuation on $B_{r_{k}}(B_{\boldsymbol{r}^{\prime}}(M))$, we have $v_{\boldsymbol{r}}(u)=\min\{v_{{\boldsymbol{r}^{\prime}}}(u^{(l)})+r_{k}l\}_{l=0}^{s_{k}-1}$. Therefore, we get 
$$v_{\boldsymbol{r}}(f)=\min\{v_{r_{k}}(f_{k})+v_{\boldsymbol{r}}(q_{k}),\min\{v_{{\boldsymbol{r}^{\prime}}}(u^{(l)})+r_{k}l\}_{l=0}^{s_{k}-1}\}.$$
Let $0\leq l<s_{k}$. By induction on $k$, there exists a unique $q_{i}^{(l)}\in B_{(r_{1},\ldots, r_{i})}(M)[X_{i+1},\ldots, X_{k-1}]$ for each $1\leq i\leq k-1$ and a unique $t^{(l)}\in M[X_{1},\ldots, X_{k-1}]$ which satisfy the followings:
\begin{enumerate}
\item[$(a)$] We have $u^{(l)}=f(X_{1})q_{1}^{(l)}+\cdots +f(X_{k-1})q_{k-1}^{(l)}+t^{(l)}$.
\item[$(b)$]  We have $\deg_{X_{i}}t^{(l)}<s_{i}$ for each $1\leq i\leq k-1$.
\item[$(c)$] For each $1\leq i< k-1$, $q_{i}^{(l)}$ satisfies $\deg_{X_{j}}q_{i}^{(l)}<s_{j}$ for each $i+1\leq j\leq k-1$.
\end{enumerate} 
Further, we have  $v_{{\boldsymbol{r}^{\prime}}}(u^{(l)})=\min\{v_{r_{1}}(f_{1})+v_{{\boldsymbol{r}^{\prime}}}(q_{1}^{(l)}),\ldots,v_{r_{k-1}}(f_{k-1})+v_{{\boldsymbol{r}^{\prime}}}(q_{k-1}^{(l)}),v_{{\boldsymbol{r}^{\prime}}}(t^{(l)})\}$. We put $q_{i}=\sum_{l=0}^{s_{k}-1}X_{k}^{l}q_{i}^{(l)}$ for each $1\leq i\leq k-1$ and $t=\sum_{l=0}^{s_{k}-1}X_{k}^{l}t^{(l)}$ Then $q_{1},\ldots, q_{k}$ and $t$ satisfy the conditions from \eqref{multivariable weiestrass polynomial pro1} to \eqref{multivariable weiestrass polynomial pro3} and $v_{\boldsymbol{r}}(q_{i})=\min\{v_{\boldsymbol{r}^{\prime}}(q_{i}^{(l)})+r_{k}l\}_{l=0}^{s_{k}-1}$ for each $1\leq i\leq k-1$ and $v_{\boldsymbol{r}}(t)=\min\{v_{{\boldsymbol{r}^{\prime}}}(t^{(l)})+r_{k}l\}_{l=0}^{s_{k}-1}$. Therefore, we see that
\begin{multline*}
\min\{v_{{\boldsymbol{r}^{\prime}}}(u^{(l)})+r_{k}l\}_{l=0}^{s_{k}-1}=\min\{v_{r_{1}}(f_{1})+\min\{v_{{\boldsymbol{r}^{\prime}}}(q_{1}^{(l)})+r_{k}l\}_{l=0}^{s_{k}-1},\\
\ldots,v_{r_{k-1}}(f_{k-1})+\min\{v_{{\boldsymbol{r}^{\prime}}}(q_{k-1}^{(l)})+r_{k}l\}_{l=0}^{s_{k}-1} ,\min\{v_{{\boldsymbol{r}^{\prime}}}(t^{(l)})+r_{k}l\}_{l=0}^{s_{k}-1}\}\\
=\min\{v_{r_{1}}(f_{1})+v_{\boldsymbol{r}}(q_{1}),\ldots, v_{r_{k-1}}(f_{k-1})+v_{\boldsymbol{r}}(q_{k-1}),v_{\boldsymbol{r}}(t)\}
\end{multline*}
and we have
\begin{align*}
v_{\boldsymbol{r}}(f)&=\min\{v_{r_{k}}(f_{k})+v_{\boldsymbol{r}}(q_{k}),\min\{v_{{\boldsymbol{r}^{\prime}}}(u^{(l)})+r_{k}l\}_{l=0}^{s_{k}-1}\}\\
&=\min\{v_{r_{1}}(f_{1})+v_{\boldsymbol{r}}(q_{1}),\ldots, v_{r_{k}}(f_{k})+v_{\boldsymbol{r}}(q_{k}),v_{\boldsymbol{r}}(t)\}.
\end{align*}
We complete the proof.
\end{proof}
\begin{cor}\label{lemma for main theorem two}
Let $\boldsymbol{r}\in \mathbb{Q}^{k}$ and $f_{i}\in  \K[X]$ be a non-zero separable polynomial such that $d_{r_{i}}(f_{i})=\deg f_{i}$ with $1\leq i\leq k$. If $f\in B_{\boldsymbol{r}}(M)$ satisfies $f(a_{1},\ldots, a_{k})=0$ for every root $a_{i}\in \overline{\K}$ of $f_{i}$ with $1\leq i\leq k$, there exists a unique $q_{i}\in B_{(r_{1},\ldots, r_{i})}(M)[X_{i+1},\ldots, X_{k}]$ for each $1\leq i\leq k$ which satisfy the following:
\begin{enumerate}
\item We have $f=f_{1}(X_{1})q_{1}+\cdots f_{k}(X_{k})q_{k}$.
\item For each $1\leq i<k$, $q_{i}\in B_{\boldsymbol{r}_{i}}(M)[X_{i+1},\ldots, X_{k}]$ satisfies $\deg_{X_{j}}q_{i}<\deg f_{j}$ for each $i+1\leq j\leq k$.
\end{enumerate}
In addition, we have $v_{\boldsymbol{r}}(f)=\min\{v_{r_{1}}(f_{1})+v_{\boldsymbol{r}}(q_{1}),\ldots, v_{r_{k}}(f_{k})+v_{\boldsymbol{r}}(q_{k})\}$.
\end{cor}
\begin{proof}
By $\mathrm{Proposition\ \ref{multivariable weiestrass polynomial pro}}$, it suffices to prove the following statement:
\begin{enumerate}
\item[$(*)$] Let $r\in M[X_{1},\ldots, X_{k}]$ with $\deg_{X_{i}} r<\deg f_{i}$ for each $1\leq i\leq k$. If $r(a_{1},\ldots,  a_{k})=0$  for all roots $a_{i}\in \overline{\K}$ of $f_{i}$ with $1\leq i\leq k$, then $r=0$. 
\end{enumerate}
If $k=1$, by Corollary \ref{cor of remainder theorem}, there exists a unique $q\in B_{r}(M)$ such that $r=f_{1}q$. Since $\deg r<\deg f_{1}$, we see that $r=0$. Then, we assume that $k\geq 2$ and the corollary is already proved for $k-1$. By induction on $k$, we see that $r(X_{1},\ldots, X_{k-1},a_{k})=0$ for every root $a_{k}\in \overline{\K}$ of $f_{k}$. We regard $r$ as an element of $B_{r_{k}}(B_{\boldsymbol{r}^{\prime}}(M))$ via the isometric isomorphism of $\mathrm{Proposition\ \ref{isometry ofHh for induction}}$. By applying $\mathrm{Corollary\ \ref{cor of remainder theorem}}$ to $r\in B_{r_{k}}(B_{\boldsymbol{r}^{\prime}}(M))$, there exists a unique $q\in B_{r_{k}}(B_{\boldsymbol{r}^{\prime}}(M))$ such that $r=f_{k}(X_{k})q$. Since $\deg_{X_{k}}r<\deg f_{k}$, we see that $r=0$.
\end{proof}
\begin{pro}\label{multivariable verof vr(f)=infbinkoerlinevm(f(b))}
Let $\boldsymbol{r}\in \mathbb{Q}^{k}$ and $f\in B_{\boldsymbol{r}}(M)$. Then, we have 
\begin{align*}
v_{\boldsymbol{r}}(f)&=\inf_{\substack{\boldsymbol{b}\in \overline{\K}^{k}\\ \ord_{p}(b_{i})>r_{i},\ 1\leq i\leq k}}\{v_{M_{\K(b_{1},\ldots, b_{k})}}(f(b_{1},\ldots, b_{k}))\}.
\end{align*}
\end{pro}
\begin{proof}
When $k=1$, Proposition \ref{multivariable verof vr(f)=infbinkoerlinevm(f(b))} is proved in Proposition \ref{supectral norm gauss norm}. Then, we assume that $k\geq 2$ and Proposition \ref{multivariable verof vr(f)=infbinkoerlinevm(f(b))} is already proved up to $k-1$. By the isometric isomorphism of  \ref{isometry ofHh for induction}, we can regard $f$ as an element of $B_{r_{k}}(B_{\boldsymbol{r}^{\prime}}(M))$. Then, by applying the result in the case $k=1$ to $f\in B_{r_{k}}(B_{\boldsymbol{r}^{\prime}}(M))$, we see that
$$v_{\boldsymbol{r}}(f)=\inf_{b_{k}\in \overline{\K},\ \ord_{p}(b_{k})>r_{k}}\{v_{\boldsymbol{r}^{\prime}}(f(X_{1},\ldots, X_{k-1},b_{k}))\}$$
where $f(X_{1},\ldots, X_{k-1},b_{k})\in B_{\boldsymbol{r}^{\prime}}(M_{\K(b_{k})})$ for each $b_{k}\in \K$ with $\ord_{p}(b_{k})>r_{k}$. By induction on $k$, we have
$$v_{\boldsymbol{r}^{\prime}}(f(X_{1},\ldots, X_{k-1},b_{k}))=\inf_{\substack{\boldsymbol{b}^{\prime}\in \overline{\K}^{k-1}\\ \ord_{p}(b_{i}^{\prime})>r_{i},\ 1\leq i\leq k-1}}\{v_{M_{\K(b_{1}^{\prime},\ldots, b_{k-1}^{\prime},b_{k})}}(f(b_{1}^{\prime},\ldots, b_{k-1}^{\prime},b_{k}))\}$$
for each $b_{k}\in \K$ with $\ord_{p}(b_{k})>r_{k}$. Then, we have
\begin{align*}
v_{\boldsymbol{r}}(f)&=\inf_{b_{k}\in \overline{\K},\ \ord_{p}(b_{k})>r_{k}}\left\{\inf_{\substack{\boldsymbol{b}^{\prime}\in \overline{\K}^{k-1}\\ \ord_{p}(b_{i}^{\prime})>r_{i},\ 1\leq i\leq k-1}}\{v_{M_{\K(b_{1}^{\prime},\ldots, b_{k-1}^{\prime},b_{k})}}(f(b_{1}^{\prime},\ldots, b_{k-1}^{\prime},b_{k}))\}\right\}\\
&=\inf_{\substack{\boldsymbol{b}\in \overline{\K}^{k}\\ \ord_{p}(b_{i})>r_{i},\ 1\leq i\leq k}}\{v_{M_{\K(b_{1},\ldots, b_{k})}}(f(b_{1},\ldots, b_{k}))\}.
\end{align*}
\end{proof}
We put $B_{+}^{(k)}(M)=\cap_{\boldsymbol{r}\in \mathbb{Q}_{>0}^{k}}B_{\boldsymbol{r}}(M)\subset B_{\boldsymbol{0}_{k}}(M)$. Let $\boldsymbol{t}_{\boldsymbol{n}}=(t_{n_{1}},\ldots, t_{n_{k}})$ for each  $\boldsymbol{n} =(n_1 ,\ldots ,n_k )\in \mathbb{Z}_{\geq 0}^{k}$, where $t_{n_{i}}=\frac{1}{p^{n_{i}}(p-1)}$ with $1\leq i\leq k$. We define the map $v_{\boldsymbol{h}}^{\prime} : B_{+}^{(k)}(M) \longrightarrow \mathbb{R} \cup \{\pm\infty\}$ by setting 
\begin{equation}\label{definition of multivariable vHprime}
v_{\boldsymbol{h}}^{\prime}(f)=\inf\{v_{\boldsymbol{t}_{\boldsymbol{n}}}(f)+\langle \boldsymbol{h},\boldsymbol{n}\rangle_{k}\}_{\boldsymbol{n}\in\mathbb{Z}_{\geq 0}^{k}}
\end{equation}
for each $f\in B_{+}^{(k)}(M)$. We note that we have $\HH_{\boldsymbol{h}}(M)\subset B_{+}^{(k)}(M)$.
\begin{pro}\label{another valuation on multivariable log order}
 For each $f\in B_{+}^{(k)}(M)$, we have $f\in \mathcal{H}_{\boldsymbol{h}}(M)$ if and only if $v_{\boldsymbol{h}}^{\prime}(f)>-\infty$. In addition, $v_{\boldsymbol{h}}^{\prime}\vert_{\mathcal{H}_{\boldsymbol{h}\slash \mathcal{K}}}$ is a valuation on $\mathcal{H}_{\boldsymbol{h}\slash \mathcal{K}}$ which satisfies $v_{\mathcal{H}_{\boldsymbol{h}}}(f)+\alpha_{\boldsymbol{h}}\leq v_{\boldsymbol{h}}^{\prime}\vert_{\mathcal{H}_{\boldsymbol{h}\slash \mathcal{K}}}(f)\leq v_{\mathcal{H}_{\boldsymbol{h}}}(f)+\beta_{\boldsymbol{h}}$ for every $f\in \mathcal{H}_{\boldsymbol{h}\slash \mathcal{K}}$, where $\alpha_{\boldsymbol{h}}=\sum_{i=1}^{k}\alpha_{h_{i}}$ and $\beta_{\boldsymbol{h}}=\sum_{i=1}^{k}\beta_{h_{i}}$ with
\begin{align*}
\alpha_{h_{i}}&=\begin{cases}-\max\{0, h_{i}-\frac{h_{i}}{\log p}(1+\log \frac{\log p}{(p-1)h_{i}})\}\ &\mathrm{if}\ h_{i}>0,\\
0\ &\mathrm{if}\ h_{i}=0,
\end{cases}\\
\beta_{h_{i}}&=\begin{cases}\max\{0,\frac{p}{p-1}-h_{i}\}\ \ \ \ \ \ \ \ \ \ \ \ \ \ \ \ \ \ \ \ \ \ \ \,&\mathrm{if}\ h_{i}>0,\\
0\ &\mathrm{if}\ h_{i}=0.
\end{cases}
\end{align*}
\end{pro}
\begin{proof}
The proposition for the case $k=1$ is proved in Proposition\ \ref{another valuation on H_{h}}. By induction on $k$, we assume that $k\geq 2$ 
and assume that the proposition is valid up to $k-1$. 
Let $f\in B_{+}^{(k)}(M)$. If $f\notin \mathcal{H}_{\boldsymbol{h}}(M)$, we set $v_{\mathcal{H}_{\boldsymbol{h}}}(f)=-\infty$. 

 First, we will show that we have $v_{\mathcal{H}_{\boldsymbol{h}}}(f)+\alpha_{\boldsymbol{h}}\leq v_{\boldsymbol{h}}^{\prime}(f)\leq v_{\mathcal{H}_{\boldsymbol{h}}}(f)+\beta_{\boldsymbol{h}}$ for every $f\in B_{+}^{(k)}(M)$. By the isometric isomorphism of $\mathrm{Proposition\ \ref{isometry ofHh for induction}}$, we identify $B_{\boldsymbol{t}_{\boldsymbol{n}}}(M)$ with $B_{t_{n_{k}}}(B_{\boldsymbol{t}_{\boldsymbol{n}^{\prime}}}(M))$ and we identify the valuation $v_{\boldsymbol{t}_{\boldsymbol{n}}}$ on $B_{\boldsymbol{t}_{\boldsymbol{n}}}(M)$ with the valuation on $B_{t_{n_{k}}}(B_{\boldsymbol{t}_{\boldsymbol{n}^{\prime}}}(M))$. Therefore, we have 
 $$v_{\boldsymbol{t}_{\boldsymbol{n}}}(g)=\inf\{v_{\boldsymbol{t}_{\boldsymbol{n}^{\prime}}}(g_{n})+t_{n_{k}}n\}_{n\in \mathbb{Z}_{\geq 0}}$$
for each $g=(g_{n})_{n\in \mathbb{Z}_{\geq 0}}\in B_{t_{n_{k}}}(B_{\boldsymbol{t}_{\boldsymbol{n}^{\prime}}}(M))$ with $g_{n}\in B_{\boldsymbol{t}_{\boldsymbol{n}^{\prime}}}(M)$. 
By \eqref{definition of multivariable vHprime}, we have
\begin{align}\label{eq another valuation on multivariable log order first}
v_{\boldsymbol{h}}^{\prime}(f)
& =\inf\{v_{\boldsymbol{t}_{\boldsymbol{n}}}(f)+\langle \boldsymbol{h},\boldsymbol{n}\rangle_{k}\}_{\boldsymbol{n}\in\mathbb{Z}_{\geq 0}^{k}} \notag \\
& =\inf\{\inf\{v_{\boldsymbol{t}_{(\boldsymbol{n}^{\prime},n_{k})}}(f)+h_{k}n_{k}\}_{n_{k}\in\mathbb{Z}_{\geq 0}}+\langle \boldsymbol{h}^{\prime},\boldsymbol{n}^{\prime}\rangle_{k-1}\} _{\boldsymbol{n}^{\prime}\in \mathbb{Z}_{\geq 0}^{k-1}}.
\end{align}
 Let $v_{\mathcal{H}_{h_{k}}}^{(\boldsymbol{t}_{\boldsymbol{n}^{\prime}})}$ be the valuation on $\mathcal{H}_{h_{k}}(B_{\boldsymbol{t}_{\boldsymbol{n}^{\prime}}}(M))$ defined by $v_{\mathcal{H}_{h_{k}}}^{(\boldsymbol{t}_{\boldsymbol{n}^{\prime}})}(g)=\inf\{v_{\boldsymbol{t}_{\boldsymbol{n}^{\prime}}}(g_{n_{k}})+h_{k}\ell(n_{k})\linebreak\}_{n_{k}\in \mathbb{Z}_{\geq 0}}$ for each $\boldsymbol{n}^{\prime}\in \mathbb{Z}_{\geq 0}^{k-1}$ and for each $g=(g_{n_{k}})_{n_{k}\in \mathbb{Z}_{\geq 0}}\in \mathcal{H}_{h_{k}}(B_{\boldsymbol{t}_{\boldsymbol{n}^{\prime}}}(M))$ with $g_{n_{k}}\in B_{\boldsymbol{t}_{\boldsymbol{n}^{\prime}}}(M)$. For each $\boldsymbol{n}^{\prime}\in \mathbb{Z}_{\geq 0}^{k-1}$, we can regard $f$ as an element of $B_{+}(B_{\boldsymbol{t}_{\boldsymbol{n}^{\prime}}}(M))=\cap_{n_{k}\in\mathbb{Z}_{\geq 0}}B_{t_{n_{k}}}(B_{\boldsymbol{t}_{\boldsymbol{n}^{\prime}}}(M))$. By applying the result in the case $k=1$ to $f\in B_{+}(B_{\boldsymbol{t}_{\boldsymbol{n}^{\prime}}}(M))$, we have
$$v_{\mathcal{H}_{h_{k}}}^{(\boldsymbol{t}_{\boldsymbol{n}^{\prime}})}(f)+\alpha_{h_{k}}\leq \inf\{v_{\boldsymbol{t}_{(\boldsymbol{n}^{\prime},n_{k})}}(f)+h_{k}n_{k}\}_{n_{k}\in \mathbb{Z}_{\geq 0}}\leq v_{\mathcal{H}_{h_{k}}}^{(\boldsymbol{t}_{\boldsymbol{n}^{\prime}})}(f)+\beta_{k}$$
for every $\boldsymbol{n}^{\prime}\in \mathbb{Z}_{\geq 0}^{k-1}$. Therefore, by \eqref{eq another valuation on multivariable log order first}, we have
\begin{multline}\label{eq another valuation on multivariable log order second}
\inf\{v_{\mathcal{H}_{h_{k}}}^{(\boldsymbol{t}_{\boldsymbol{n}^{\prime}})}(f)+\langle \boldsymbol{h}^{\prime},\boldsymbol{n}^{\prime}\rangle_{k-1}\}_{\boldsymbol{n}^{\prime}\in \mathbb{Z}_{\geq 0}^{k-1}}+\alpha_{h_{k}}\leq v_{\boldsymbol{h}}^{\prime}(f)\\
\leq \inf\{v_{\mathcal{H}_{h_{k}}}^{(\boldsymbol{t}_{\boldsymbol{n}^{\prime}})}(f)+\langle \boldsymbol{h}^{\prime},\boldsymbol{n}^{\prime}\rangle_{k-1}\}_{\boldsymbol{n}^{\prime}\in \mathbb{Z}_{\geq 0}^{k-1}}+\beta_{h_{k}}.
\end{multline}
Let us set $f=(f_{n_{k}})_{n_{k}=0}^{+\infty}$, with $f_{n_{k}}\in B_{+}^{(k-1)}(\K)$. By the definitions of $v_{\mathcal{H}_{h_{k}}}^{(\boldsymbol{t}_{\boldsymbol{n}^{\prime}})}$ and $v_{\boldsymbol{h}^{\prime}}^{\prime}$, we have
\begin{align}\label{eq another valuation on multivariable log order second half}
\begin{split}
&\inf\{v_{\mathcal{H}_{h_{k}}}^{(\boldsymbol{t}_{\boldsymbol{n}^{\prime}})}(f)+\langle \boldsymbol{h}^{\prime},\boldsymbol{n}^{\prime}\rangle_{k-1}\}_{\boldsymbol{n}^{\prime}\in \mathbb{Z}_{\geq 0}^{k-1}}\\
&=\inf\{\inf\{v_{\boldsymbol{t}_{\boldsymbol{n}^{\prime}}}(f_{n_{k}})+h_{k}\ell(n_{k})\}_{n_{k}\in\mathbb{Z}_{\geq 0}}+\langle \boldsymbol{h}^{\prime},\boldsymbol{n}^{\prime}\rangle_{k-1}\}_{\boldsymbol{n}^{\prime}\in \mathbb{Z}_{\geq 0}^{k-1}}\\
&=\inf\{\inf\{v_{\boldsymbol{t}_{\boldsymbol{n}^{\prime}}}(f_{n_{k}})+\langle \boldsymbol{h}^{\prime},\boldsymbol{n}^{\prime}\rangle_{k-1}\}_{\boldsymbol{n}^{\prime}\in \mathbb{Z}_{\geq 0}^{k-1}}+h_{k}\ell(n_{k})\}_{n_{k}\in\mathbb{Z}_{\geq 0}}\\
&=\inf\{v_{\boldsymbol{h}^{\prime}}^{\prime}(f_{n_{k}})+h_{k}\ell(n_{k})\}_{n_{k}\in\mathbb{Z}_{\geq 0}}.
\end{split}
\end{align}
By \eqref{eq another valuation on multivariable log order second} and \eqref{eq another valuation on multivariable log order second half}, we have
\begin{multline}\label{eq another valuation on multivariable log order third}
\inf\{v_{\boldsymbol{h}^{\prime}}^{\prime}(f_{n_{k}})
 +h_{k}\ell(n_{k})\}_{n_{k}\in\mathbb{Z}_{\geq 0}}+\alpha_{h_{k}}\leq v_{\boldsymbol{h}}^{\prime}(f)
 \leq \inf\{v_{\boldsymbol{h}^{\prime}}^{\prime}(f_{n_{k}})+h_{k}\ell(n_{k})\}_{n_{k}\in\mathbb{Z}_{\geq 0}}+\beta_{h_{k}}.
\end{multline}
By $\mathrm{Proposition\ \ref{isometry ofHh for induction}}$, we have $v_{\HH_{\boldsymbol{h}}}(f)=\inf\{v_{\HH_{\boldsymbol{h}^{\prime}}}(f_{n_{k}})+h_{k}\ell(n_{k})\}_{n_{k}\in\mathbb{Z}_{\geq 0}}$. By the assumption of our induction argument on $k$, we have $\alpha_{\boldsymbol{h}^{\prime}}+v_{\HH_{\boldsymbol{h}^{\prime}}}\leq v_{\boldsymbol{h}^{\prime}}^{\prime}\leq \beta_{\boldsymbol{h}^{\prime}}+v_{\HH_{\boldsymbol{h}^{\prime}}}$. Therefore, we have
\begin{align}\label{eq another valuation on multivariable log order fourth}
\begin{split}
v_{\mathcal{H}_{\boldsymbol{h}}}(f)+\alpha_{\boldsymbol{h}^{\prime}}
&=\inf\{v_{\mathcal{H}_{\boldsymbol{h}^{\prime}}}(f_{n_{k}})+h_{k}\ell(n_{k})\}_{n_{k}\in\mathbb{Z}_{\geq 0}}+\alpha_{\boldsymbol{h}^{\prime}}\\
&\leq \inf\{v_{\boldsymbol{h}^{\prime}}^{\prime}(f_{n_{k}})+h_{k}\ell(n_{k})\}_{n_{k}\in\mathbb{Z}_{\geq 0}}\\
&\leq \inf\{v_{\mathcal{H}_{\boldsymbol{h}^{\prime}}}(f_{n_{k}})+h_{k}\ell(n_{k})\}_{n_{k}\in\mathbb{Z}_{\geq 0}}+\beta_{\boldsymbol{h}^{\prime}}\\
&=v_{\mathcal{H}_{\boldsymbol{h}}}(f)+\beta_{\boldsymbol{h}^{\prime}}.
\end{split}
\end{align}
Therefore, by \eqref{eq another valuation on multivariable log order third} and \eqref{eq another valuation on multivariable log order fourth}, we have 
\begin{align*}
v_{\mathcal{H}_{\boldsymbol{h}}}(f)+\alpha_{\boldsymbol{h}}&\leq (v_{\mathcal{H}_{\boldsymbol{h}}}(f)+\alpha_{\boldsymbol{h}^{\prime}})+\alpha_{h_{k}}\leq \inf\{v_{\boldsymbol{h}^{\prime}}^{\prime}(f_{n_{k}})
 +h_{k}\ell(n_{k})\}_{n_{k}\in\mathbb{Z}_{\geq 0}}+\alpha_{h_{k}}\\
 &\leq v_{\boldsymbol{h}}^{\prime}(f)\leq \inf\{v_{\boldsymbol{h}^{\prime}}^{\prime}(f_{n_{k}})+h_{k}\ell(n_{k})\}_{n_{k}\in\mathbb{Z}_{\geq 0}}+\beta_{h_{k}}\\
&\leq (v_{\mathcal{H}_{\boldsymbol{h}}}(f)+\beta_{\boldsymbol{h}^{\prime}})+\beta_{h_{k}}=v_{\mathcal{H}_{\boldsymbol{h}}}(f)+\beta_{\boldsymbol{h}}.
\end{align*}
 Since $v_{\mathcal{H}_{\boldsymbol{h}}}(f)+\alpha_{\boldsymbol{h}}\leq v_{\boldsymbol{h}}^{\prime}(f)\leq v_{\mathcal{H}_{\boldsymbol{h}}}(f)+\beta_{\boldsymbol{h}}$, we see that $f\in \HH_{\boldsymbol{h}}(M)$ if and only $v_{\boldsymbol{h}}^{\prime}(f)>-\infty$. It is easy to check that $v_{\boldsymbol{h}}^{\prime}\vert_{\HH_{\boldsymbol{h}}}$ is a valuation on $\mathcal{H}_{\boldsymbol{h}}(M)$.
\end{proof}
Let $f_{1}\in \HH_{\boldsymbol{g}}(\K)$ and $f_{2}\in \HH_{\boldsymbol{h}}(M)$ with $\boldsymbol{g},\boldsymbol{h}\in \ord_{p}(\mathcal{O}_{\K}\backslash \{0\})^{k}$. By Proposition \ref{another valuation on multivariable log order}, we see that $f_{1}f_{2}\in \HH_{\boldsymbol{g}+\boldsymbol{h}}(M)$ easily. For each $\boldsymbol{m}\in \mathbb{Z}_{\geq 0}^{k}$, we denote by $(\Omega_{\boldsymbol{m}}^{[\boldsymbol{d},\boldsymbol{e}]})=(\Omega_{\boldsymbol{m}}^{[\boldsymbol{d},\boldsymbol{e}]}(X_{1},\ldots, X_{k}))$ the ideal of $\mathcal{O}_{\K}[[X_{1},\ldots, X_{k}]]$ generated by $\Omega_{m_{1}}^{[d_{1},e_{1}]}(X_{1}),\ldots,\Omega_{m_{k}}^{[d_{k},e_{k}]}(X_{k})$, where $\Omega_{m_{i}}^{[d_{i},e_{i}]}(X_{i})=\prod_{j=d_{i}}^{e_{i}}((1+X_{i})^{p^{m_{j}}}-u_{i}^{jp^{m_{i}}})$. Let $J_{\boldsymbol{h}}^{[\boldsymbol{d},\boldsymbol{e}]}(M)$ be the $\mathcal{O}_{\K}[[X_{1},\ldots, X_{k}]]\otimes_{\mathcal{O}_{\K}}\K$-module defined in \eqref{generalization of the project lim for deformation ring Jboldsymbol}. Let $(s_{\boldsymbol{m}}^{[\boldsymbol{d},\boldsymbol{e}]})_{\boldsymbol{m}\in \mathbb{Z}_{\geq 0}^{k}}\in J_{\boldsymbol{h}}^{[\boldsymbol{d},\boldsymbol{e}]}(M)$. By Proposition \ref{multivariable weiestrass polynomial pro}, for each $\boldsymbol{m}\in \mathbb{Z}_{\geq 0}^{k}$, there exists a unique element $r(s_{\boldsymbol{m}}^{[\boldsymbol{d},\boldsymbol{e}]})\in M^{0}[X_{1},\ldots, X_{k}]\otimes_{\mathcal{O}_{\K}}\K$ such that $s_{\boldsymbol{m}}^{[\boldsymbol{d},\boldsymbol{e}]}\equiv r(s_{\boldsymbol{m}}^{[\boldsymbol{d},\boldsymbol{e}]})\ \mathrm{mod}\ (\Omega_{\boldsymbol{m}}^{[\boldsymbol{d},\boldsymbol{e}]})$ and $\deg r(s_{\boldsymbol{m}}^{[\boldsymbol{d},\boldsymbol{e}]})<\deg_{X_{i}} \Omega_{m_{i}}^{[d_{i},e_{i}]}$ for each $1\leq i\leq k$. We define a valuation on $v_{J_{\boldsymbol{h}}}$ on $J_{\boldsymbol{h}}^{[\boldsymbol{d},\boldsymbol{e}]}(M)$ by setting 
\begin{equation}\label{multivariable Jboldysmbolh valuation}
v_{J_{\boldsymbol{h}}}((s_{\boldsymbol{m}}^{[\boldsymbol{d},\boldsymbol{e}]})_{\boldsymbol{m}\in \mathbb{Z}_{\geq 0}^{k}})=\inf_{\boldsymbol{m}\in \mathbb{Z}_{\geq 0}^{k}}\{v_{\boldsymbol{0}_{k}}(r(s_{\boldsymbol{m}}^{[\boldsymbol{d},\boldsymbol{e}]}))+\langle \boldsymbol{h},\boldsymbol{m}\rangle_{k}\}
\end{equation}
for each $(s_{\boldsymbol{m}}^{[\boldsymbol{d},\boldsymbol{e}]})_{\boldsymbol{m}\in \mathbb{Z}_{\geq 0}^{k}}\in J_{\boldsymbol{h}}^{[\boldsymbol{d},\boldsymbol{e}]}(M)$ where $v_{\boldsymbol{0}_{k}}$ is the valuation on $B_{\boldsymbol{0}_{k}}(M)$. It is easy to see that $v_{J_{\boldsymbol{h}}}$ is a valuation on $J_{\boldsymbol{h}}^{[\boldsymbol{d},\boldsymbol{e}]}(M)$. {For each $\boldsymbol{m}\in \mathbb{Z}_{\geq 0}^{k}$, let $M[X_{1},\ldots, X_{k}]_{<\deg (\Omega_{\boldsymbol{m}}^{[\boldsymbol{d},\boldsymbol{e}]})}$ be the finite dimensional $\K$-Banach submodule of $B_{\boldsymbol{0}_{k}}(M)$ consisting of $f\in M[X_{1},\ldots, X_{k}]$ with $\deg_{X_{i}}f<\deg\Omega_{m_{i}}^{[d_{i},e_{i}]}$ for every $1\leq i\leq k$. By Proposition \ref{multivariable weiestrass polynomial pro}, we have the following natural $\K$-linear isomorphism 
\begin{equation}\label{multi J induction pro finitek-1deg poly isom proomega}
M[X_{1},\ldots, X_{k}]_{<\deg (\Omega_{\boldsymbol{m}}^{[\boldsymbol{d},\boldsymbol{e}]})}\stackrel{\sim}{\rightarrow}\frac{M^{0}[[X_{1},\ldots, X_{k}]]}{(\Omega_{\boldsymbol{m}}^{[\boldsymbol{d},\boldsymbol{e}]})M^{0}[[X_{1},\ldots, X_{k}]]}\otimes_{\mathcal{O}_{\K}}\K,\ f\mapsto [f].
\end{equation}
Via the isomorphism \eqref{multi J induction pro finitek-1deg poly isom proomega}, we regard $\frac{M^{0}[[X_{1},\ldots, X_{k}]]}{(\Omega_{\boldsymbol{m}}^{[\boldsymbol{d},\boldsymbol{e}]})M^{0}[[X_{1},\ldots, X_{k}]]}\otimes_{\mathcal{O}_{\K}}\K$ as a $\K$-Banach space. By the definition of $v_{J_{\boldsymbol{h}}}$, the natural projection 
\begin{equation}\label{multi Jproj is bounded}
J_{\boldsymbol{h}}^{[\boldsymbol{d},\boldsymbol{e}]}(M)\rightarrow \frac{M^{0}[[X_{1},\ldots, X_{k}]]}{(\Omega_{\boldsymbol{m}}^{[\boldsymbol{d},\boldsymbol{e}]})M^{0}[[X_{1},\ldots, X_{k}]]}\otimes_{\mathcal{O}_{\K}}\K
\end{equation}
is a bounded $\K$-linear homomorphism for each $\boldsymbol{m}\in \mathbb{Z}_{\geq 0}^{k}$.}
\begin{pro}
$(J_{\boldsymbol{h}}^{[\boldsymbol{d},\boldsymbol{e}]}(M),v_{J_{\boldsymbol{h}}})$ is a $\K$-Banach space.
\end{pro}
The above proposition is proved in the same way as Proposition \ref{one variable Jh is a K-Banach}. Hence, we omit the proof of 
the above proposition. By definition, we have
\begin{align}
\begin{split}
J_{\boldsymbol{h}}^{[\boldsymbol{d},\boldsymbol{e}]}(M)^{0}&=\bigg\{(s_{\boldsymbol{m}}^{[\boldsymbol{d},\boldsymbol{e}]})_{m\in \mathbb{Z}_{\geq 0}^{k}}\in J_{\boldsymbol{h}}^{[\boldsymbol{d},\boldsymbol{e}]}(M)\bigg\vert\\
& (p^{\langle \boldsymbol{h},\boldsymbol{m}\rangle_{k}}s_{\boldsymbol{m}}^{[\boldsymbol{d},\boldsymbol{e}]})_{\boldsymbol{m}\in \mathbb{Z}_{\geq 0}^{k}}\in \prod_{\boldsymbol{m}\in \mathbb{Z}_{\geq 0}^{k}}M^{0}[[X_{1},\ldots, X_{k}]]\slash (\Omega_{\boldsymbol{m}}^{[\boldsymbol{d},\boldsymbol{e}]})M^{0}[[X_{1},\ldots,X_{k}]]\bigg\}.
\end{split}
\end{align}
We have the following:
\begin{pro}\label{Jboldsymbol(M)Lidetify Jboldsymbol(ML)prop}
Let $\mathcal{L}$ be a finite extension of $\K$. Then, we have an isometric isomorphism 
$$\varphi:J_{\boldsymbol{h}}^{[\boldsymbol{d},\boldsymbol{e}]}(M)_{\mathcal{L}}\rightarrow J_{\boldsymbol{h}}^{[\boldsymbol{d},\boldsymbol{e}]}(M_{\mathcal{L}}),$$
defined by $(s^{[\boldsymbol{d},\boldsymbol{e}]}\otimes_{\K}a)\mapsto as^{[\boldsymbol{d},\boldsymbol{e}]}$ for each $s^{[\boldsymbol{d},\boldsymbol{e}]}\in J_{\boldsymbol{h}}^{[\boldsymbol{d},\boldsymbol{e}]}(M)$ and for each $a\in \mathcal{L}$.
\end{pro}.
\begin{proof}
{First, we prove that $\varphi$ is well-defined. Let $s^{[\boldsymbol{d},\boldsymbol{e}]}\in J_{\boldsymbol{h}}^{[\boldsymbol{d},\boldsymbol{e}]}(M)_{\mathcal{L}}$. Assume that $s^{[\boldsymbol{d},\boldsymbol{e}]}$ is expressed as a sum $s^{[\boldsymbol{d},\boldsymbol{e}]}=\sum_{i=1}^{l}s^{(i)}\otimes_{\K}a_{i}$ where $s^{(i)}\in J_{\boldsymbol{h}}^{[\boldsymbol{d},\boldsymbol{e}]}(M)$ and $a_{i}\in \mathcal{L}$ with $l\in \mathbb{Z}_{\geq 1}$. We see that $\sum_{i=1}^{l}a_{i}s^{(i)}$ is in $\prod_{\boldsymbol{m}\in \mathbb{Z}_{\geq 0}^{k}}\left(\frac{M_{\mathcal{L}}^{0}[[X_{1},\ldots, X_{k}]]}{(\Omega_{\boldsymbol{m}}^{[\boldsymbol{d},\boldsymbol{e}]})M_{\mathcal{L}}^{0}[[X_{1},\ldots, X_{k}]]}\otimes_{\mathcal{O}_{\mathcal{L}}}\mathcal{L}\right)$. To prove that $\varphi$ is well-defined, it suffices to prove that $\sum_{i=1}^{l}a_{i}s^{(i)}\in J_{\boldsymbol{h}}^{[\boldsymbol{d},\boldsymbol{e}]}(M_{\mathcal{L}})$. Since $s^{(i)}\in \varprojlim_{\boldsymbol{m}\in \mathbb{Z}_{\geq 0}^{k}}\left(\frac{M^{0}[[X_{1},\ldots, X_{k}]]}{(\Omega_{\boldsymbol{m}}^{[\boldsymbol{d},\boldsymbol{e}]})M^{0}[[X_{1},\ldots, X_{k}]]}\otimes_{\mathcal{O}_{\K}}\K\right)$ for every $1\leq i\leq l$, we have 
$$
\sum_{i=1}^{l}a_{i}s^{(i)}\in \varprojlim_{\boldsymbol{m}\in \mathbb{Z}_{\geq 0}^{k}}\left(\frac{M_{\mathcal{L}}^{0}[[X_{1},\ldots, X_{k}]]}{(\Omega_{\boldsymbol{m}}^{[\boldsymbol{d},\boldsymbol{e}]})M_{\mathcal{L}}^{0}[[X_{1},\ldots, X_{k}]]}\otimes_{\mathcal{O}_{\mathcal{L}}}\mathcal{L}\right).
$$
Put $s^{(i)}=(s^{(i)}_{\boldsymbol{m}})_{\boldsymbol{m}\in \mathbb{Z}_{\geq 0}^{k}}$. We have $\sum_{i=1}^{l}a_{i}s^{(i)}=(\sum_{i=1}^{l}a_{i}s^{(i)}_{\boldsymbol{m}})_{\boldsymbol{m}\in\mathbb{Z}_{\geq 0}^{k}}$. Since $(p^{\langle\boldsymbol{h},\boldsymbol{m}\rangle_{k}}s^{(i)}_{\boldsymbol{m}})_{\boldsymbol{m}\in\mathbb{Z}_{\geq 0}^{k}}\linebreak\in \left(\prod_{\boldsymbol{m}\in \mathbb{Z}_{\geq 0}^{k}}\frac{M^{0}[[X_{1},\ldots, X_{k}]]}{(\Omega_{\boldsymbol{m}}^{[\boldsymbol{d},\boldsymbol{e}]})M^{0}[[X_{1},\ldots, X_{k}]]}\right)\otimes_{\mathcal{O}_{\K}}\K$ for every $1\leq i\leq l$, we have
$$(p^{\langle\boldsymbol{h},\boldsymbol{m}\rangle_{k}}\sum_{i=1}^{l}a_{i}s^{(i)}_{\boldsymbol{m}})_{\boldsymbol{m}\in\mathbb{Z}_{\geq 0}^{k}}\in \left(\prod_{\boldsymbol{m}\in \mathbb{Z}_{\geq 0}^{k}}\frac{M_{\mathcal{L}}^{0}[[X_{1},\ldots, X_{k}]]}{(\Omega_{\boldsymbol{m}}^{[\boldsymbol{d},\boldsymbol{e}]})M_{\mathcal{L}}^{0}[[X_{1},\ldots, X_{k}]]}\right)\otimes_{\mathcal{O}_{\mathcal{L}}}\mathcal{L}.$$
Hence, we have $\sum_{i=1}^{l}a_{i}s^{(i)}\in J_{\boldsymbol{h}}^{[\boldsymbol{d},\boldsymbol{e}]}(M_{\mathcal{L}})$ and we conclude that $\varphi$ is well-defined.}
\par 
{Next, we prove that $v_{\mathfrak{L}}(\varphi)\geq 0$. Let $s^{[\boldsymbol{d},\boldsymbol{e}]}\in J_{\boldsymbol{h}}^{[\boldsymbol{d},\boldsymbol{e}]}(M)_{\mathcal{L}}$. Assume that $s^{[\boldsymbol{d},\boldsymbol{e}]}$ is expressed as a sum $s^{[\boldsymbol{d},\boldsymbol{e}]}=\sum_{i=1}^{l}s^{(i)}\otimes_{\K}a_{i}$ where $s^{(i)}\in J_{\boldsymbol{h}}^{[\boldsymbol{d},\boldsymbol{e}]}(M)$ and $a_{i}\in \mathcal{L}$ with $l\in \mathbb{Z}_{\geq 1}$. By the definition of $\varphi$, we have $\varphi(s^{[\boldsymbol{d},\boldsymbol{e}]})=\sum_{i=1}^{l}a_{i}s^{(i)}$. Put $s^{(i)}=(s^{(i)}_{\boldsymbol{m}})_{\boldsymbol{m}\in \mathbb{Z}_{\geq 0}^{k}}$. Proposition \ref{multivariable weiestrass polynomial pro} implies that, for each $\boldsymbol{m}\in\mathbb{Z}_{\geq 0}^{k}$ and $0\leq i\leq l$, there exists a unique $r^{(i)}_{\boldsymbol{m}}\in M[X_{1},\ldots, X_{k}]$ such that $\deg_{X_{j}}r^{(i)}_{\boldsymbol{m}}<\deg\Omega_{m_{j}}^{[d_{j},e_{j}]}$ for every $1\leq j\leq k$ and $r^{(i)}_{\boldsymbol{m}}\equiv s^{(i)}_{\boldsymbol{m}}\ \mathrm{mod}\ (\Omega_{\boldsymbol{m}}^{[\boldsymbol{d},\boldsymbol{e}]})$. By the definition of $v_{J_{\boldsymbol{h}}}$, we see that
\begin{equation}
v_{J_{\boldsymbol{h}}}(s^{(i)})=\inf\{v_{\boldsymbol{0}_{k}}(r^{(i)}_{\boldsymbol{m}})+\langle \boldsymbol{h},\boldsymbol{m}\rangle_{k}\}_{\boldsymbol{m}\in \mathbb{Z}_{\geq 0}^{k}}
\end{equation}
for each $1\leq i\leq l$. Hence, we have 
\begin{align}\label{Jboldsymbol(M)Lidetify Jboldsymbol(ML)prop welldefined eq1}
\begin{split}
v_{\boldsymbol{0}_{k}}\left(\sum_{i=1}^{l}a_{i}r^{(i)}_{\boldsymbol{m}}\right)+\langle \boldsymbol{h},\boldsymbol{m}\rangle_{k}&\geq \min\{(v_{\boldsymbol{0}_{k}}(r_{\boldsymbol{m}}^{(i)})+\langle \boldsymbol{h},\boldsymbol{m}\rangle_{k})+\ord_{p}(a_{i})\}_{i=1}^{l}\\
&\geq\min\{v_{J_{\boldsymbol{h}}}(s^{(i)})+\ord_{p}(a_{i})\}_{i=1}^{l}
\end{split}
\end{align}
for every $\boldsymbol{m}\in \mathbb{Z}_{\geq 0}^{k}$. On the other hand, we have $\varphi(s^{[\boldsymbol{d},\boldsymbol{e}]})=([\sum_{i=1}^{l}a_{i}r^{(i)}_{\boldsymbol{m}}])_{\boldsymbol{m}\in \mathbb{Z}_{\geq 0}^{k}}$ where $[\sum_{i=1}^{l}a_{i}r^{(i)}_{\boldsymbol{m}}]  \in \frac{M_{\mathcal{L}}^{0}[[X_{1},\ldots, X_{k}]]}{(\Omega_{\boldsymbol{m}}^{[\boldsymbol{d},\boldsymbol{e}]})M_{\mathcal{L}}^{0}[[X_{1},\ldots, X_{k}]]}\otimes_{\mathcal{O}_{\mathcal{L}}}\mathcal{L}$ is the class of $\sum_{i=1}^{l}a_{i}r^{(i)}_{\boldsymbol{m}}\in M_{\mathcal{L}}[X_{1},\ldots, X_{k}]$. We have $\deg_{X_{j}}\left(\sum_{i=1}^{l}a_{i}r^{(i)}_{\boldsymbol{m}}\right)<\deg\Omega_{m_{j}}^{[d_{j},e_{j}]}$ for every $1\leq j\leq k$ and for every $\boldsymbol{m}\in \mathbb{Z}_{\geq 0}^{k}$. By the definition of $v_{J_{\boldsymbol{h}}}$, we have
$$
v_{J_{\boldsymbol{h}}}(\varphi(s^{[\boldsymbol{d},\boldsymbol{e}]}))=\inf\left\{v_{\boldsymbol{0}_{k}}\left(\sum_{i=1}^{l}a_{i}r_{\boldsymbol{m}}^{(i)}\right)+\langle \boldsymbol{h},\boldsymbol{m}\rangle_{k}\right\}_{\boldsymbol{m}\in \mathbb{Z}_{\geq 0}^{k}}.$$
By \eqref{Jboldsymbol(M)Lidetify Jboldsymbol(ML)prop welldefined eq1}, we have
\begin{equation}\label{Jboldsymbol(M)Lidetify Jboldsymbol(ML)prop vL(varphi)geq 0eq2}
v_{J_{\boldsymbol{h}}}(\varphi(s^{[\boldsymbol{d},\boldsymbol{e}]}))\geq \min\{v_{J_{\boldsymbol{h}}}(s^{(i)})+\ord_{p}(a_{i})\}_{i=1}^{l}.
\end{equation}
Let $v_{J_{\boldsymbol{h}}^{[\boldsymbol{d},\boldsymbol{e}]}(M)_{\mathcal{L}}}$ be the valuation on $J_{\boldsymbol{h}}^{[\boldsymbol{d},\boldsymbol{e}]}(M)_{\mathcal{L}}$ defined below \eqref{definition of ML}. By the definition of $v_{J_{\boldsymbol{h}}^{[\boldsymbol{d},\boldsymbol{e}]}(M)_{\mathcal{L}}}$, $v_{J_{\boldsymbol{h}}^{[\boldsymbol{d},\boldsymbol{e}]}(M)_{\mathcal{L}}}(s^{[\boldsymbol{d},\boldsymbol{e}]})$  is the least upper bound of $\min\{v_{J_{\boldsymbol{h}}}(s^{(i)})+\ord_{p}(a_{i})\}_{i=1}^{l}$ among all representations $s^{[\boldsymbol{d},\boldsymbol{e}]}=\sum_{i=1}^{l}s^{(i)}\otimes_{\K}a_{i}$ where $s^{(i)}\in J_{\boldsymbol{h}}^{[\boldsymbol{d},\boldsymbol{e}]}(M)$ and $a_{i}\in \mathcal{L}$. By \eqref{Jboldsymbol(M)Lidetify Jboldsymbol(ML)prop vL(varphi)geq 0eq2}, we have
$$v_{J_{\boldsymbol{h}}}(\varphi(s^{[\boldsymbol{d},\boldsymbol{e}]}))\geq v_{J_{\boldsymbol{h}}^{[\boldsymbol{d},\boldsymbol{e}]}(M)_{\mathcal{L}}}(s^{[\boldsymbol{d},\boldsymbol{e}]})$$
and we conclude that $v_{\mathfrak{L}}(\varphi)\geq 0$.}
\par 
{Next, we prove that $\varphi$ is injective. We have the following diagram:
\begin{equation}\label{Jboldsymbol(M)Lidetify Jboldsymbol(ML)prop inj pic}
\xymatrix{
J_{\boldsymbol{h}}^{[\boldsymbol{d},\boldsymbol{e}]}(M)_{\mathcal{L}}\ar[r]^{\varphi}\ar@{^{(}-_{>}}[d]&J_{\boldsymbol{h}}^{[\boldsymbol{d},\boldsymbol{e}]}(M_{\mathcal{L}})\ar@{^{(}-_{>}}[d]\\
\prod_{\boldsymbol{m}\in \mathbb{Z}_{\geq 0}^{k}}\left(\frac{M^{0}[[X_{1},\ldots,X_{k}]]}{(\Omega_{\boldsymbol{m}}^{[\boldsymbol{d},\boldsymbol{e}]})M^{0}[[X_{1},\ldots, X_{k}]]}\otimes_{\mathcal{O}_{\K}}\K\right)\otimes_{\K}\mathcal{L}\ar[r]&\prod_{\boldsymbol{m}\in \mathbb{Z}_{\geq 0}^{k}}\left(\frac{M_{\mathcal{L}}^{0}[[X_{1},\ldots,X_{k}]]}{(\Omega_{\boldsymbol{m}}^{[\boldsymbol{d},\boldsymbol{e}]})M_{\mathcal{L}}^{0}[[X_{1},\ldots, X_{k}]]}\otimes_{\mathcal{O}_{\mathcal{L}}}\mathcal{L}\right).
}
\end{equation}
The two vertical maps of \eqref{Jboldsymbol(M)Lidetify Jboldsymbol(ML)prop inj pic} are the natural inclusions and the bottom map is defined by $(s^{[\boldsymbol{d},\boldsymbol{e}]}\otimes_{\K}a)\mapsto as^{[\boldsymbol{d},\boldsymbol{e}]}$ for each $s^{[\boldsymbol{d},\boldsymbol{e}]}\in \prod_{\boldsymbol{m}\in \mathbb{Z}_{\geq 0}^{k}}\frac{M^{0}[[X_{1},\ldots,X_{k}]]}{(\Omega_{\boldsymbol{m}}^{[\boldsymbol{d},\boldsymbol{e}]})M^{0}[[X_{1},\ldots, X_{k}]]}\otimes_{\mathcal{O}_{\K}}\K$ and for each $a\in \mathcal{L}$. For each $\boldsymbol{m}\in \mathbb{Z}_{\geq 0}^{k}$, let $M[X_{1},\ldots, X_{k}]_{<\deg (\Omega_{\boldsymbol{m}}^{[\boldsymbol{d},\boldsymbol{e}]})}$ be the finite dimensional $\K$-Banach submodule of $B_{\boldsymbol{0}_{k}}(M)$ consisting of $f\in M[X_{1},\ldots, X_{k}]$ with $\deg_{X_{i}}f<\Omega_{m_{i}}^{[d_{i},e_{i}]}$ for every $1\leq i\leq k$. By \eqref{multi J induction pro finitek-1deg poly isom proomega}, we have the natural isomorphism
$$M[X_{1},\ldots, X_{k}]_{\leq \deg (\Omega_{\boldsymbol{m}}^{[\boldsymbol{d},\boldsymbol{e}]})}\simeq \frac{M^{0}[[X_{1},\ldots,X_{k}]]}{(\Omega_{\boldsymbol{m}}^{[\boldsymbol{d},\boldsymbol{e}]})M^{0}[[X_{1},\ldots, X_{k}]]}\otimes_{\mathcal{O}_{\K}}\K.$$
By using this isomorphism, we see that  
\small 
\begin{align*}
\prod_{\boldsymbol{m}\in \mathbb{Z}_{\geq 0}^{k}}\left(\frac{M^{0}[[X_{1},\ldots,X_{k}]]}{(\Omega_{\boldsymbol{m}}^{[\boldsymbol{d},\boldsymbol{e}]})M^{0}[[X_{1},\ldots, X_{k}]]}\otimes_{\mathcal{O}_{\K}}\K\right)\otimes_{\K}\mathcal{L}&\simeq \left(\prod_{\boldsymbol{m}\in \mathbb{Z}_{\geq 0}^{k}}M[X_{1},\ldots, X_{k}]_{<\deg (\Omega_{\boldsymbol{m}^{\prime}}^{[\boldsymbol{d},\boldsymbol{e}]})}\right)\otimes_{\K}\mathcal{L}\\
&\simeq \prod_{\boldsymbol{m}\in \mathbb{Z}_{\geq 0}^{k}}M_{\mathcal{L}}[X_{1},\ldots, X_{k}]_{<\deg (\Omega_{\boldsymbol{m}^{\prime}}^{[\boldsymbol{d},\boldsymbol{e}]})}\\
&\simeq \prod_{\boldsymbol{m}\in \mathbb{Z}_{\geq 0}^{k}}\left(\frac{M_{\mathcal{L}}^{0}[[X_{1},\ldots,X_{k}]]}{(\Omega_{\boldsymbol{m}}^{[\boldsymbol{d},\boldsymbol{e}]})M_{\mathcal{L}}^{0}[[X_{1},\ldots, X_{k}]]}\otimes_{\mathcal{O}_{\mathcal{L}}}\mathcal{L}\right).
\end{align*}
\normalsize 
Therefore, we see that the bottom map of \eqref{Jboldsymbol(M)Lidetify Jboldsymbol(ML)prop inj pic} is an isomorphism. Since the vertical  maps of \eqref{Jboldsymbol(M)Lidetify Jboldsymbol(ML)prop inj pic} are injectives, $\varphi$ is injective.}
\par 
Next, we prove that $\varphi$ is surjective. Let $\epsilon>0$. By \cite[Proposition 3 in \S2.6.2]{BGR1984}, there exists a $\K$-basis $b_{1}\ldots, b_{d}$ of $\mathcal{L}$ 
depending on $\epsilon$ such that,  
for every $(a_1 ,\ldots , a_d ) \in \K^d$, the inequality $\min\{\ord_{p}(a_{i}b_{i})\}_{i=1}^{d}\geq \ord_{p}(b)-\epsilon$ holds 
where $b=\sum_{i=1}^{d}a_{i}b_{i}$. By the isometric isomorphism in Proposition \ref{boldh H isometry completetensor}, we identify $B_{\boldsymbol{0}_{k}}(M_{\mathcal{L}})$ with $B_{\boldsymbol{0}_{k}}(M)_{\mathcal{L}}$. For each $s\in B_{\boldsymbol{0}_{k}}(M_{\mathcal{L}})$, we can express $s$ as a sum $s=\sum_{i=1}^{d}s^{(i)}\otimes_{\K}b_{i}$ with $s^{(i)}\in B_{\boldsymbol{0}_{k}}(M)$ uniquely. Further, by \eqref{eq:boldh H isometry completetensor}, we see that 
\begin{equation}\label{multvariable jboldysmbolh finite eq mL eq}
\min\{v_{\boldsymbol{0}_{k}}(s^{(i)})+\ord_{p}(b_{i}))\}_{i=1}^{d}\geq v_{\boldsymbol{0}_{k}}(s)-\epsilon
\end{equation}
for each $s\in B_{\boldsymbol{0}_{k}}(M_{\mathcal{L}})$. Let $s^{[\boldsymbol{d},\boldsymbol{e}]}=(s_{\boldsymbol{m}}^{[\boldsymbol{d},\boldsymbol{e}]})_{\boldsymbol{m}\in \mathbb{Z}_{\geq 0}^{k}}\in J_{\boldsymbol{h}}^{[\boldsymbol{d},\boldsymbol{e}]}(M_{\mathcal{L}})$. By Proposition \ref{multivariable weiestrass polynomial pro}, for each $\boldsymbol{m}\in \mathbb{Z}_{\geq 0}^{k}$, there exists a unique element $r(s_{\boldsymbol{m}}^{[\boldsymbol{d},\boldsymbol{e}]})\in M_{\mathcal{L}}[X_{1},\ldots,X_{k}]$ such that $s_{\boldsymbol{m}}^{[\boldsymbol{d},\boldsymbol{e}]}\equiv r(s_{\boldsymbol{m}}^{[\boldsymbol{d},\boldsymbol{e}]})\ \mathrm{mod}\ (\Omega_{\boldsymbol{m}}^{[\boldsymbol{d},\boldsymbol{e}]})$ and $\deg_{X_{j}}r_{\boldsymbol{m}}<\deg \Omega_{m_{j}}^{[d_{j},e_{j}]}$ for each $1\leq j\leq k$. For each $\boldsymbol{m}\in \mathbb{Z}_{\geq 0}^{k}$, we can expres $r(s_{\boldsymbol{m}}^{[\boldsymbol{d},\boldsymbol{e}]})$ as a sum
\begin{equation}\label{multvariable jboldysmbolh finite eq rsboldsymbolm=sumeq}
r(s_{\boldsymbol{m}}^{[\boldsymbol{d},\boldsymbol{e}]})=\sum_{i=1}^{d}b_{i}r(s_{\boldsymbol{m}}^{[\boldsymbol{d},\boldsymbol{e}]})^{(i)}
\end{equation}
uniquely where $r(s_{\boldsymbol{m}}^{[\boldsymbol{d},\boldsymbol{e}]})^{(i)}\in M[X_{1},\ldots, X_{k}]$ with $1\leq i\leq d$.
Since $r(s_{\boldsymbol{n}}^{[\boldsymbol{d},\boldsymbol{e}]})-r(s_{\boldsymbol{m}}^{[\boldsymbol{d},\boldsymbol{e}]})$ is contained in $(\Omega_{\boldsymbol{m}}^{[\boldsymbol{d},\boldsymbol{e}]})M_{\mathcal{L}}^{0}[[X_{1},\ldots, X_{k}]]\otimes_{\mathcal{O}_{\mathcal{L}}}\mathcal{L}$ for each $\boldsymbol{m},\boldsymbol{n}\in \mathbb{Z}_{\geq 0}^{k}$ with $\boldsymbol{n}\geq \boldsymbol{m}$, we have $r(s_{\boldsymbol{n}}^{[\boldsymbol{d},\boldsymbol{e}]})^{(i)}-r(s_{\boldsymbol{m}}^{[\boldsymbol{d},\boldsymbol{e}]})^{(i)}\in(\Omega_{\boldsymbol{m}}^{[\boldsymbol{d},\boldsymbol{e}]})M^{0}[[X_{1},\ldots,X_{k}]]\otimes_{\mathcal{O}_{\K}}\K$ for each $i$ satisfying $1\leq i\leq d$. Therefore, we have 
\begin{equation}\label{multvariable jboldysmbolh finite eq rmproject}
([r(s_{\boldsymbol{m}}^{[\boldsymbol{d},\boldsymbol{e}]})^{(i)}])_{\boldsymbol{m}\in \mathbb{Z}_{\geq 0}^{k}}\in \varprojlim_{\boldsymbol{m}\in \mathbb{Z}_{\geq 0}^{k}}\left(\frac{M^{0}[[X_{1},\ldots, X_{k}]]}{(\Omega_{\boldsymbol{m}}^{[\boldsymbol{d},\boldsymbol{e}]})M^{0}[[X_{1},\ldots, X_{k}]]}\otimes_{\mathcal{O}_{\K}}\K\right)\end{equation}
for each $i$ satisfying $1\leq i\leq d$. By \eqref{multivariable Jboldysmbolh valuation} and \eqref{multvariable jboldysmbolh finite eq mL eq}, we see that
\begin{equation}\label{multvariable jboldysmbolh finite eq rsboldysmbolmigeq roldsymbolm}
v_{\boldsymbol{0}_{k}}(r(s_{\boldsymbol{m}}^{[\boldsymbol{d},\boldsymbol{e}]})^{(i)})+\langle \boldsymbol{h},\boldsymbol{m}\rangle_{k}+\ord_{p}(b_{i})\geq v_{\boldsymbol{0}_{k}}(r(s_{\boldsymbol{m}}^{[\boldsymbol{d},\boldsymbol{e}]}))+\langle \boldsymbol{h},\boldsymbol{m}\rangle_{k}-\epsilon\geq v_{J_{\boldsymbol{h}}}(s)-\epsilon
\end{equation}
for each $1\leq i\leq d$ and each $\boldsymbol{m}\in \mathbb{Z}_{\geq 0}^{k}$. Then, we have 
\begin{equation}\label{multvariable jboldysmbolh finite eq rboldsmbolm bounded}
(p^{\langle \boldsymbol{h},\boldsymbol{m}\rangle_{k}}[r(s_{\boldsymbol{m}}^{[\boldsymbol{d},\boldsymbol{e}]})^{(i)}])_{\boldsymbol{m}\in \mathbb{Z}_{\geq 0}^{k}}\in \left(\prod_{\boldsymbol{m}\in \mathbb{Z}_{\geq 0}^{k}}\frac{M^{0}[[X_{1},\ldots, X_{k}]]}{(\Omega_{\boldsymbol{m}}^{[\boldsymbol{d},\boldsymbol{e}]})M^{0}[[X_{1},\ldots, X_{k}]]}\right)\otimes_{\mathcal{O}_{\K}}\K\end{equation}
for each $1\leq i\leq d$. By \eqref{multvariable jboldysmbolh finite eq rmproject} and \eqref{multvariable jboldysmbolh finite eq rboldsmbolm bounded}, we see that $([r(s_{\boldsymbol{m}}^{[\boldsymbol{d},\boldsymbol{e}]})^{(i)}])_{\boldsymbol{m}\in \mathbb{Z}_{\geq 0}^{k}}\in J_{\boldsymbol{h}}^{[\boldsymbol{d},\boldsymbol{e}]}(M)$ for each $1\leq i\leq d$.  By \eqref{multvariable jboldysmbolh finite eq rsboldsymbolm=sumeq}, we have 
\begin{equation}\label{multvariable jboldysmbolh finite vL(varphi-1)geq 0}
\varphi(\sum_{i=1}^{d}([r(s_{\boldsymbol{m}}^{[\boldsymbol{d},\boldsymbol{e}]})^{(i)}])_{\boldsymbol{m}\in \mathbb{Z}_{\geq 0}^{k}}\otimes_{\K}b_{i})=([\sum_{i=1}^{d}b_{i}r(s_{\boldsymbol{m}}^{[\boldsymbol{d},\boldsymbol{e}]})^{(i)}])_{\boldsymbol{m}\in \mathbb{Z}_{\geq 0}^{k}}=([r(s_{\boldsymbol{m}}^{[\boldsymbol{d},\boldsymbol{e}]})])_{\boldsymbol{m}\in \mathbb{Z}_{\geq 0}^{k}}=s^{[\boldsymbol{d},\boldsymbol{e}]}.
\end{equation}
Then, $\varphi$ is surjective. 
\par 
Next, we prove that $v_{\mathfrak{L}}(\varphi^{-1})\geq 0$.  Let $\epsilon>0$ and $b_{1}\ldots, b_{d}\in \mathcal{L}$ a basis over $\K$ such that,  
for every $(a_1 ,\ldots , a_d ) \in \K^d$, the inequality $\min\{\ord_{p}(a_{i}b_{i})\}_{i=1}^{d}\geq \ord_{p}(b)-\epsilon$ holds 
where $b=\sum_{i=1}^{d}a_{i}b_{i}$. Let $s^{[\boldsymbol{d},\boldsymbol{e}]}=(s_{\boldsymbol{m}}^{[\boldsymbol{d},\boldsymbol{e}]})_{\boldsymbol{m}\in \mathbb{Z}_{\geq 0}^{k}}\in J_{\boldsymbol{h}}^{[\boldsymbol{d},\boldsymbol{e}]}(M_{\mathcal{L}})$ and $r(s_{\boldsymbol{m}}^{[\boldsymbol{d},\boldsymbol{e}]})\in M_{\mathcal{L}}[X_{1},\ldots,X_{k}]$ the unique element such that $s_{\boldsymbol{m}}^{[\boldsymbol{d},\boldsymbol{e}]}\equiv r(s_{\boldsymbol{m}}^{[\boldsymbol{d},\boldsymbol{e}]})\ \mathrm{mod}\ (\Omega_{\boldsymbol{m}}^{[\boldsymbol{d},\boldsymbol{e}]})$ and $\deg_{X_{j}}r_{\boldsymbol{m}}<\deg \Omega_{m_{j}}^{[d_{j},e_{j}]}$ for each $1\leq j\leq k$. For each $\boldsymbol{m}\in \mathbb{Z}_{\geq 0}^{k}$, we put $r(s_{\boldsymbol{m}}^{[\boldsymbol{d},\boldsymbol{e}]})=\sum_{i=1}^{d}b_{i}r(s_{\boldsymbol{m}}^{[\boldsymbol{d},\boldsymbol{e}]})^{(i)}$ with $r(s_{\boldsymbol{m}}^{[\boldsymbol{d},\boldsymbol{e}]})^{(i)}\in M[X_{1},\ldots, X_{k}]$. By \eqref{multvariable jboldysmbolh finite vL(varphi-1)geq 0}, we have $\varphi^{-1}(s^{[\boldsymbol{d},\boldsymbol{e}]})=\sum_{i=1}^{d}([r(s_{\boldsymbol{m}}^{[\boldsymbol{d},\boldsymbol{e}]})^{(i)}])_{\boldsymbol{m}\in \mathbb{Z}_{\geq 0}^{k}}\otimes_{\K}b_{i}$. By the definition of $v_{J_{\boldsymbol{h}}}$, we have
$$v_{J_{\boldsymbol{h}}}(([r(s_{\boldsymbol{m}}^{[\boldsymbol{d},\boldsymbol{e}]})^{(i)}])_{\boldsymbol{m}\in \mathbb{Z}_{\geq 0}^{k}})=\inf\{v_{\boldsymbol{0}_{k}}(r(s_{\boldsymbol{m}}^{[\boldsymbol{d},\boldsymbol{e}]})^{(i)})+\langle \boldsymbol{h},\boldsymbol{m}\rangle_{k}\}_{\boldsymbol{m}\in \mathbb{Z}_{\geq 0}^{k}}.$$
Then, by \eqref{multvariable jboldysmbolh finite eq rsboldysmbolmigeq roldsymbolm}, we have
\begin{align}\label{multvariable jboldysmbolh finite eq vJboldsymbolrigeq vJboldysmbols-epsilon}
\begin{split}
v_{J_{\boldsymbol{h}}}(([r(s_{\boldsymbol{m}}^{[\boldsymbol{d},\boldsymbol{e}]})^{(i)}])_{\boldsymbol{m}\in \mathbb{Z}_{\geq 0}^{k}})+\ord_{p}(b_{i})&=\inf\{v_{\boldsymbol{0}_{k}}(r(s_{\boldsymbol{m}}^{[\boldsymbol{d},\boldsymbol{e}]})^{(i)})+\langle \boldsymbol{h},\boldsymbol{m}\rangle_{k}\}_{\boldsymbol{m}\in \mathbb{Z}_{\geq 0}^{k}}+\ord_{p}(b_{i})\\
&\geq v_{J_{\boldsymbol{h}}}(s)-\epsilon
\end{split}
\end{align}
for each $1\leq i\leq d$.  By the definition of $v_{J_{\boldsymbol{h}}^{[\boldsymbol{d},\boldsymbol{e}]}(M)_{\mathcal{L}}}$, we have
\begin{align*}
v_{J_{\boldsymbol{h}}^{[\boldsymbol{d},\boldsymbol{e}]}(M)_{\mathcal{L}}}(\varphi^{-1}(s^{[\boldsymbol{d},\boldsymbol{e}]}))&=v_{J_{\boldsymbol{h}}^{[\boldsymbol{d},\boldsymbol{e}]}(M)_{\mathcal{L}}}\left(\sum_{i=1}^{d}([r(s_{\boldsymbol{m}}^{[\boldsymbol{d},\boldsymbol{e}]})^{(i)}])_{\boldsymbol{m}\in \mathbb{Z}_{\geq 0}^{k}}\otimes_{\K}b_{i}\right)\\
&\geq \min_{1\leq i\leq d}\{v_{J_{\boldsymbol{h}}}(([r(s_{\boldsymbol{m}}^{[\boldsymbol{d},\boldsymbol{e}]})^{(i)}])_{\boldsymbol{m}\in \mathbb{Z}_{\geq 0}^{k}})+\ord_{p}(b_{i})\}.
\end{align*}
Then, by \eqref{multvariable jboldysmbolh finite eq vJboldsymbolrigeq vJboldysmbols-epsilon}, we have $v_{J_{\boldsymbol{h}}^{[\boldsymbol{d},\boldsymbol{e}]}(M)_{\mathcal{L}}}(\varphi^{-1}(s^{[\boldsymbol{d},\boldsymbol{e}]}))\geq v_{J_{\boldsymbol{h}}}(s^{[\boldsymbol{d},\boldsymbol{e}]})-\epsilon$. Thus, we have $v_{\mathfrak{L}}(\varphi^{-1})\geq -\epsilon$. Since $\epsilon$ is arbitrary positive real number, we have $v_{\mathfrak{L}}(\varphi^{-1})\geq 0$. 
\par 
By Lemma \ref{easy lemma on isometry}, we see that $\varphi$ is isometric. We complete the proof.
\end{proof}
{For each root $b\in \overline{\K}$ of $\Omega_{m_{k}}^{[d_{k},e_{k}]}(X_{k})$ with $m_{k}\in \mathbb{Z}_{\geq 0}$, we have the following two $\K$-Banach homomorphisms
\begin{align}\label{maps for multi J induction pro}
\begin{split}
&\varphi_{b,m_{k}}:J_{\boldsymbol{h}}^{[\boldsymbol{d},\boldsymbol{e}]}(M)\rightarrow J_{\boldsymbol{h}^{\prime}}^{[\boldsymbol{d}^{\prime},\boldsymbol{e}^{\prime}]}(M_{\K(b)}), (s^{[\boldsymbol{d},\boldsymbol{e}]}_{\boldsymbol{m}})_{\boldsymbol{m}\in \mathbb{Z}_{\geq 0}^{k}}\mapsto (\tilde{s}_{(\boldsymbol{m}^{\prime},m_{k})}^{[\boldsymbol{d},\boldsymbol{e}]}(X_{1},\ldots, X_{k-1},b))_{\boldsymbol{m}^{\prime}\in\mathbb{Z}_{\geq 0}^{k-1}},\\
&\psi_{b,m_{k}}:J_{h_{k}}^{[d_{k},e_{k}]}(J_{\boldsymbol{h}^{\prime}}^{[\boldsymbol{d}^{\prime},\boldsymbol{e}^{\prime}]}(M))\rightarrow J_{\boldsymbol{h}^{\prime}}^{[\boldsymbol{d}^{\prime},\boldsymbol{e}^{\prime}]}(M)_{\K(b)}, (s_{m}^{[d_{k},e_{k}]})_{m\in\mathbb{Z}_{\geq 0}}\mapsto \tilde{s}_{m_{k}}^{[d_{k},e_{k}]}(b)
\end{split}
\end{align}
where $\tilde{s}_{(\boldsymbol{m}^{\prime},m_{k})}^{[\boldsymbol{d},\boldsymbol{e}]}\in M^{0}[[X_{1},\ldots, X_{k}]]\otimes_{\mathcal{O}_{\K}}\K$ is a lift of $s_{(\boldsymbol{m}^{\prime},m_{k})}^{[\boldsymbol{d},\boldsymbol{e}]}$ and $\tilde{s}_{m_{k}}^{[d_{k},e_{k}]}\in J_{\boldsymbol{h}^{\prime}}^{[\boldsymbol{d}^{\prime},\boldsymbol{e}^{\prime}]}(M)^{0}[[X_{k}]]\linebreak\otimes_{\mathcal{O}_{\K}}\K$ is a lift of $s_{m_{k}}^{[d_{k},e_{k}]}$. When $k=1$, we define $J_{\boldsymbol{h}^{\prime}}^{[\boldsymbol{d}^{\prime},\boldsymbol{e}^{\prime}]}(M)$ and $J_{\boldsymbol{h}^{\prime}}^{[\boldsymbol{d}^{\prime},\boldsymbol{e}^{\prime}]}(M_{\K(b)})$ to be $M$ and $M_{\K(b)}$ respectively and define $\varphi_{b,m_{1}}$ to be $\psi_{b,m_{1}}$.}

 In the following proposition, we identify $J_{\boldsymbol{h}^{\prime}}^{[\boldsymbol{d}^{\prime},\boldsymbol{e}^{\prime}]}(M)_{\mathcal{L}}$ with $J_{\boldsymbol{h}^{\prime}}^{[\boldsymbol{d}^{\prime},\boldsymbol{e}^{\prime}]}(M_{\mathcal{L}})$ for each finite extension $\mathcal{L}$ of $\K$ by the isometric isomorphism in Proposition \ref{Jboldsymbol(M)Lidetify Jboldsymbol(ML)prop}.
\begin{pro}\label{multi J induction pro}
{There exists a unique $\mathcal{O}_{\K}[[X_{1},\ldots, X_{k}]]$-module isomorphism: 
$$
\varphi_{k}:J_{h_{k}}^{[d_{k},e_{k}]}(J_{\boldsymbol{h}^{\prime}}^{[\boldsymbol{d}^{\prime},\boldsymbol{e}^{\prime}]}(M))^{0}\stackrel{\sim}{\rightarrow}J_{\boldsymbol{h}}^{[\boldsymbol{d},\boldsymbol{e}]}(M)^{0}
$$
which satisfies $\varphi_{b,m_{k}} \circ \varphi_{k}=\psi_{b,m_{k}}$ for every $m_{k}\in \mathbb{Z}_{\geq 0}$ and for every root $b\in \overline{\K}$ of $\Omega_{m_{k}}^{[d_{k},e_{k}]}(X_{k})$ where $\varphi_{b,m_{k}}$ and $\psi_{b,m_{k}}$ are the $\K$-Banach homomorphisms defined in  \eqref{maps for multi J induction pro}.}
\end{pro}
\begin{proof}
As explained above, we define $J_{\boldsymbol{h}^{\prime}}^{[\boldsymbol{d}^{\prime},\boldsymbol{e}^{\prime}]}(M)$ to be $M$ when $k=1$. 
If $k=1$, Proposition \ref{multi J induction pro} is trivially true. In the rest of the proof, we assume that $k\geq 2$. To define the map $\varphi_{k}$, we need to prove that, for each $s^{[d_{k},e_{k}]}\in J_{h_{k}}^{[d_{k},e_{k}]}(J_{\boldsymbol{h}^{\prime}}^{[\boldsymbol{d}^{\prime},\boldsymbol{e}^{\prime}]}(M))^{0}$, 
there exists a unique element $s^{[\boldsymbol{d},\boldsymbol{e}]}\in J_{\boldsymbol{h}}^{[\boldsymbol{d},\boldsymbol{e}]}(M)^{0}$ which satisfies 
\begin{equation}\label{multi J induction pro well defeq}
\varphi_{b,m_{k}}(s^{[\boldsymbol{d},\boldsymbol{e}]})=\psi_{b,m_{k}}(s^{[d_{k},e_{k}]})
\end{equation}
 for every $m_{k}\in \mathbb{Z}_{\geq 0}$ and for every root $b\in \overline{\K}$ of $\Omega_{m_{k}}^{[d_{k},e_{k}]}$. Let $s^{[d_{k},e_{k}]}\in J_{h_{k}}^{[d_{k},e_{k}]}(J_{\boldsymbol{h}^{\prime}}^{[\boldsymbol{d}^{\prime},\boldsymbol{e}^{\prime}]}(M))^{0}$. 
\par 
{First, we prove the uniqueness of $s^{[\boldsymbol{d},\boldsymbol{e}]}$ which satisfies \eqref{multi J induction pro well defeq}. It suffices to prove that, if $s^{[\boldsymbol{d},\boldsymbol{e}]}$ satisfies $\varphi_{b,m_{k}}(s^{[\boldsymbol{d},\boldsymbol{e}]})=0$ for every $m_{k}\in \mathbb{Z}_{\geq 0}$ and for every root $b\in \overline{\K}$ of $\Omega_{m_{k}}^{[d_{k},e_{k}]}$, we have $s^{[\boldsymbol{d},\boldsymbol{e}]}=0$. Put $s^{[\boldsymbol{d},\boldsymbol{e}]}=(s_{\boldsymbol{m}}^{[\boldsymbol{d},\boldsymbol{e}]})_{\boldsymbol{m}\in \mathbb{Z}_{\geq 0}^{k}}$. Since $\varphi_{b,m_{k}}(s^{[\boldsymbol{d},\boldsymbol{e}]})=0$ for every $m_{k}\in \mathbb{Z}_{\geq 0}$ and for every root $b\in \overline{\K}$ of $\Omega_{m_{k}}^{[d_{k},e_{k}]}$, we have $\tilde{s}_{\boldsymbol{m}}^{[\boldsymbol{d},\boldsymbol{e}]}(b_{1},\ldots, b_{k})=0$ for every $\boldsymbol{m}\in \mathbb{Z}_{\geq 0}^{k}$ and for every root $b_{i}\in \overline{\K}$ of $\Omega_{m_{i}}^{[d_{i},e_{i}]}$ with $1\leq i\leq k$ where $\tilde{\boldsymbol{s}}_{\boldsymbol{m}}^{[\boldsymbol{d},\boldsymbol{e}]}$ is a lift of $\boldsymbol{s}_{\boldsymbol{m}}^{[\boldsymbol{d},\boldsymbol{e}]}$. By Corollary \ref{lemma for main theorem two}, we have $\tilde{s}_{\boldsymbol{m}}^{[\boldsymbol{d},\boldsymbol{e}]}\in (\Omega_{\boldsymbol{m}}^{[\boldsymbol{d},\boldsymbol{e}]})M^{0}[[X_{1},\ldots, X_{k}]]\otimes_{\mathcal{O}_{\K}}\K$, which implies that $s^{[\boldsymbol{d},\boldsymbol{e}]}=(s_{\boldsymbol{m}}^{[\boldsymbol{d},\boldsymbol{e}]})_{\boldsymbol{m}\in \mathbb{Z}_{\geq 0}^{k}}=0$. Therefore, we have the uniqueness of $s^{[\boldsymbol{d},\boldsymbol{e}]}$ which satisfies \eqref{multi J induction pro well defeq}.}
\par
{Next, we will prove the existence of $s^{[\boldsymbol{d},\boldsymbol{e}]}\in J_{\boldsymbol{h}}^{[\boldsymbol{d},\boldsymbol{e}]}(M)^{0}$ which satisfies \eqref{multi J induction pro well defeq}. Let $\tilde{s}_{m_{k}}^{[d_{k},e_{k}]}\in J_{\boldsymbol{h}^{\prime}}^{[\boldsymbol{d}^{\prime},\boldsymbol{e}^{\prime}]}(M)^{0}[[X_{k}]]\otimes_{\mathcal{O}_{\K}}p^{-h_{k}m_{k}}\mathcal{O}_{\K}$ be a lift of $s_{m_{k}}^{[d_{k},e_{k}]}$ for each $m_{k}\in \mathbb{Z}_{\geq 0}$. Put $\tilde{s}_{m_{k}}^{[d_{k},e_{k}]}=(s_{(n),m_{k}}^{[d_{k},e_{k}]})_{n\in\mathbb{Z}_{\geq 0}}$ where $s_{(n),m_{k}}^{[d_{k},e_{k}]}\in J_{\boldsymbol{h}^{\prime}}^{[\boldsymbol{d}^{\prime},\boldsymbol{e}^{\prime}]}(M)^{0}\otimes_{\mathcal{O}_{\K}}p^{-h_{k}m_{k}}\mathcal{O}_{\K}$. We regard $J_{\boldsymbol{h}^{\prime}}^{[\boldsymbol{d}^{\prime},\boldsymbol{e}^{\prime}]}(M)^{0}\otimes_{\mathcal{O}_{\K}}p^{-h_{k}m_{k}}\mathcal{O}_{\K}$ as an $\mathcal{O}_{\K}$-submodule of $\prod_{\boldsymbol{m}^{\prime}\in \mathbb{Z}_{\geq 0}^{k-1}}\frac{M^{0}[[X_{1},\ldots, X_{k-1}]]}{(\Omega_{\boldsymbol{m}^{\prime}}^{[\boldsymbol{d}^{\prime},\boldsymbol{e}^{\prime}]})M^{0}[[X_{1},\ldots, X_{k-1}]]}\otimes_{\mathcal{O}_{\K}}p^{-\langle \boldsymbol{h},(\boldsymbol{m}^{\prime},m_{k})\rangle_{k}}\mathcal{O}_{\K}$ naturally and put 
$s_{(n),m_{k}}^{[d_{k},e_{k}]}=(s_{(n),(\boldsymbol{m}^{\prime},m_{k})}^{[d_{k},e_{k}]})_{\boldsymbol{m}^{\prime}
\in \mathbb{Z}_{\geq 0}^{k-1}}$ for each $n\in \mathbb{Z}_{\geq 0}$ where $s_{(n),(\boldsymbol{m}^{\prime},m_{k})}^{[\boldsymbol{d},\boldsymbol{e}]}\in \frac{M^{0}[[X_{1},\ldots, X_{k-1}]]}{(\Omega_{\boldsymbol{m}^{\prime}}^{[\boldsymbol{d}^{\prime},\boldsymbol{e}^{\prime}]})M^{0}[[X_{1},\ldots, X_{k-1}]]}\otimes_{\mathcal{O}_{\K}}p^{-\langle \boldsymbol{h},(\boldsymbol{m}^{\prime},m_{k})\rangle_{k}}\mathcal{O}_{\K}$. Let $n\in \mathbb{Z}_{\geq 0}$, $\boldsymbol{m}^{\prime}\in \mathbb{Z}_{\geq 0}^{k-1}$ and $m_{k}\in \mathbb{Z}_{\geq 0}$. By Proposition \ref{multivariable weiestrass polynomial pro}, there exists a unique element $r(s_{(n),(\boldsymbol{m}^{\prime},m_{k})}^{[d_{k},e_{k}]})\in M^{0}[X_{1},\ldots, X_{k-1}]\otimes_{\mathcal{O}_{\K}}p^{-\langle \boldsymbol{h},(\boldsymbol{m}^{\prime},m_{k})\rangle_{k}}\mathcal{O}_{\K}$ such that $s_{(n),(\boldsymbol{m}^{\prime},m_{k})}^{[d_{k},e_{k}]}\equiv r(s_{(n),(\boldsymbol{m}^{\prime},m_{k})}^{[d_{k},e_{k}]})\ \mathrm{mod}\ (\Omega_{\boldsymbol{m}^{\prime}}^{[\boldsymbol{d}^{\prime},\boldsymbol{e}^{\prime}]})$ and $\deg_{X_{i}}\linebreak r(s_{(n),(\boldsymbol{m}^{\prime},m_{k})}^{[d_{k},e_{k}]})<\deg \Omega_{m_{i}^{\prime}}^{[d_{i},e_{i}]}$ for each $1\leq i\leq k-1$. We put $r(s_{(\boldsymbol{m}^{\prime},m_{k})}^{[d_{k},e_{k}]}) 
\in \linebreak M^{0} [X_{1},\ldots, X_{k-1}][[X_{k}]]\otimes_{\mathcal{O}_{\K}}p^{-\langle \boldsymbol{h},(\boldsymbol{m}^{\prime},m_{k})\rangle_{k}}\mathcal{O}_{\K} $ 
to be $r(s_{(\boldsymbol{m}^{\prime},m_{k})}^{[d_{k},e_{k}]})=(r(s_{(n),(\boldsymbol{m}^{\prime},m_{k})}^{[d_{k},e_{k}]}))_{n=0}^{+\infty}$ for each $(\boldsymbol{m}^{\prime},m_{k})\in \mathbb{Z}_{\geq 0}^{k-1}\times \mathbb{Z}_{\geq 0}$. We regard $r(s_{(\boldsymbol{m}^{\prime},m_{k})}^{[d_{k},e_{k}]})$ as an element of $M^{0} [[X_{1},\ldots, X_{k}]]\otimes_{\mathcal{O}_{\K}}p^{-\langle \boldsymbol{h},(\boldsymbol{m}^{\prime},m_{k})\rangle_{k}}\mathcal{O}_{\K}$ naturally. Put $r(s^{[d_{k},e_{k}]})=([r(s_{\boldsymbol{m}}^{[d_{k},e_{k}]})])_{\boldsymbol{m}\in \mathbb{Z}_{\geq 0}^{k}}\in \prod_{\boldsymbol{m}\in \mathbb{Z}_{\geq 0}^{k}}\linebreak\frac{M^{0}[[X_{1},\ldots, X_{k}]]}{(\Omega_{\boldsymbol{m}}^{[\boldsymbol{d},\boldsymbol{e}]})M^{0}[[X_{1},\ldots, X_{k}]]}\otimes_{\mathcal{O}_{\K}}p^{-\langle \boldsymbol{h},\boldsymbol{m}\rangle_{k}}\mathcal{O}_{\K}$ where $[r(s_{\boldsymbol{m}}^{[d_{k},e_{k}]})]\in \prod_{\boldsymbol{m}\in \mathbb{Z}_{\geq 0}^{k}}\frac{M^{0}[[X_{1},\ldots, X_{k}]]}{(\Omega_{\boldsymbol{m}}^{[\boldsymbol{d},\boldsymbol{e}]})M^{0}[[X_{1},\ldots, X_{k}]]}\otimes_{\mathcal{O}_{\K}}p^{-\langle \boldsymbol{h},\boldsymbol{m}\rangle_{k}}\mathcal{O}_{\K}$ is the class of $r(s_{\boldsymbol{m}}^{[d_{k},e_{k}]})$. For each $m_{k}\in \mathbb{Z}_{\geq 0}$ and for each root $b\in \overline{\K}$ of $\Omega_{m_{k}}^{[d_{k},e_{k}]}$, let $\tilde{s}_{m_{k}}^{[d_{k},e_{k}]}(b)\in J_{\boldsymbol{h}^{\prime}}^{[\boldsymbol{d}^{\prime},\boldsymbol{e}^{\prime}]}(M_{\K(b)})$ be the specialization $\tilde{s}_{m_{k}}^{[d_{k},e_{k}]}$ at $b$. By the definition of the specialization, we have $\tilde{s}_{m_{k}}^{[d_{k},e_{k}]}(b)=\sum_{n=0}^{+\infty}s_{(n),m_{k}}^{[d_{k},e_{k}]}b^{n}$. Since we have $s_{(n),m_{k}}^{[d_{k},e_{k}]}=([r(s_{(n),(\boldsymbol{m}^{\prime},m_{k})}^{[d_{k},e_{k}]})])_{\boldsymbol{m}^{\prime}\in \mathbb{Z}_{\geq 0}^{k-1}}$ and the projection of \eqref{multi Jproj is bounded} is bounded, we see that 
$$\tilde{s}_{m_{k}}^{[d_{k},e_{k}]}(b)=\left(\left[\sum_{n=0}^{+\infty}r(s_{(n),(\boldsymbol{m}^{\prime},m_{k})}^{[d_{k},e_{k}]})b^{n}\right]\right)_{\boldsymbol{m}^{\prime}\in \mathbb{Z}_{\geq 0}^{k-1}}.$$
On the other hand, we have $r(s_{(\boldsymbol{m}^{\prime},m_{k})}^{[d_{k},e_{k}]})(X_{1},\ldots, X_{k-1},b)=\sum_{n=0}^{+\infty}r(s_{(n),(\boldsymbol{m}^{\prime},m_{k})}^{[d_{k},e_{k}]})b^{n}$. Therefore, we see that
\begin{equation}\label{multi J induction pro tildesmk []r(sm,mkbnexpre}
\tilde{s}_{m_{k}}^{[d_{k},e_{k}]}(b)=([r(s_{(\boldsymbol{m}^{\prime},m_{k})}^{[d_{k},e_{k}]})(X_{1},\ldots, X_{k-1},b)])_{\boldsymbol{m}^{\prime}\in \mathbb{Z}_{\geq 0}^{k-1}}
\end{equation}
for every $m_{k}\in \mathbb{Z}_{\geq 0}$ and for every root $b\in \overline{\K}$ of $\Omega_{m_{k}}^{[d_{k},e_{k}]}$. Since we have $\tilde{s}_{n_{k}}^{[d_{k},e_{k}]}(b)=\tilde{s}_{m_{k}}^{[d_{k},e_{k}]}(b)$ for every $m_{k},n_{k}\in \mathbb{Z}_{\geq 0}$ with $n_{k}\geq m_{k}$ and for every root $b\in \overline{\K}$ of $\Omega_{m_{k}}^{[d_{k},e_{k}]}$, by \eqref{multi J induction pro tildesmk []r(sm,mkbnexpre}, we have
\begin{equation}\label{multi J induction provarphi surj rmk rnk b eq}
([r(s_{(\boldsymbol{m}^{\prime},n_{k})}^{[d_{k},e_{k}]})(X_{1},\ldots, X_{k-1},b)])_{\boldsymbol{m}^{\prime}\in \mathbb{Z}_{\geq 0}^{k-1}}=([r(s_{(\boldsymbol{m}^{\prime},m_{k})}^{[d_{k},e_{k}]})(X_{1},\ldots, X_{k-1},b)])_{\boldsymbol{m}^{\prime}\in \mathbb{Z}_{\geq 0}^{k-1}}.
\end{equation}
By \eqref{multi J induction provarphi surj rmk rnk b eq}, we see that $r(s_{\boldsymbol{n}}^{[d_{k},e_{k}]})(b_{1},\ldots, b_{k})=r(s_{\boldsymbol{m}}^{[d_{k},e_{k}]})(b_{1},\ldots, b_{k})$ for every $\boldsymbol{m},\boldsymbol{n}\in \mathbb{Z}_{\geq 0}^{k}$ with $\boldsymbol{n}\geq \boldsymbol{m}$ and for every root $b_{i}\in \overline{\K}$ of $\Omega_{m_{i}}^{[d_{i},e_{i}]}$ with $1\leq i\leq k$. By Corollary \ref{lemma for main theorem two}, we have $r(s_{\boldsymbol{n}}^{[d_{k},e_{k}]})\equiv r(s_{\boldsymbol{m}}^{[d_{k},e_{k}]})\ \mathrm{mod}\ (\Omega_{\boldsymbol{m}}^{[\boldsymbol{d},\boldsymbol{e}]})$ for every $\boldsymbol{m},\boldsymbol{n}\in \mathbb{Z}_{\geq 0}^{k}$ with $\boldsymbol{n}\geq \boldsymbol{m}$. Then, we have $r(s^{[d_{k},e_{k}]})\in \varprojlim_{\boldsymbol{m}\in \mathbb{Z}_{\geq 0}^{k}}\left(\frac{M^{0}[[X_{1},\ldots, X_{k}]]}{(\Omega_{\boldsymbol{m}}^{[\boldsymbol{d},\boldsymbol{e}]})M^{0}[[X_{1},\ldots, X_{k}]]}\otimes_{\mathcal{O}_{\K}}\K\right)$. Since $r(s^{[d_{k},e_{k}]})$ is in $\prod_{\boldsymbol{m}\in \mathbb{Z}_{\geq 0}^{k}}\frac{M^{0}[[X_{1},\ldots, X_{k}]]}{(\Omega_{\boldsymbol{m}}^{[\boldsymbol{d},\boldsymbol{e}]})M^{0}[[X_{1},\ldots, X_{k}]]}\otimes_{\mathcal{O}_{\K}}p^{-\langle \boldsymbol{h},\boldsymbol{m}\rangle_{k}}\mathcal{O}_{\K}$, we have  $r(s^{[d_{k},e_{k}]})\in J_{\boldsymbol{h}}^{[\boldsymbol{d},\boldsymbol{e}]}(M)^{0}.$ By \eqref{multi J induction pro tildesmk []r(sm,mkbnexpre}, we have $\varphi_{m_{k},b}(r(s^{[d_{k},e_{k}]}))=\psi_{m_{k},b}(s^{[d_{k},e_{k}]})$ for every $m_{k}\in \mathbb{Z}_{\geq 0}$ and for every root $b\in \overline{\K}$ of $\Omega_{m_{k}}^{[d_{k},e_{k}]}$. Therefore, $\varphi_{k}$ is well-defined.}
 \par 
{Next, we prove that $\varphi_{k}$ is injective. Let $s^{[d_{k},e_{k}]}\in J_{h_{k}}^{[d_{k},e_{k}]}(J_{\boldsymbol{h}^{\prime}}^{[\boldsymbol{d}^{\prime},\boldsymbol{e}^{\prime}]}(M))^{0}$ such that \linebreak$\varphi_{k}(s^{[d_{k},e_{k}]})=0$. Put $s^{[d_{k},e_{k}]}=(s_{m_{k}}^{[d_{k},e_{k}]})_{m_{k}\in \mathbb{Z}_{\geq 0}}$. Since $\psi_{b,m_{k}}(s^{[d_{k},e_{k}]})=\varphi_{m_{k},b}\varphi_{k}(s^{[d_{k},e_{k}]})=0$ for every $m_{k}\in \mathbb{Z}_{\geq 0}$ and for every root $b \in \overline{\K}$ of $\Omega_{m_{k}}^{[d_{k},e_{k}]}$, we have $\tilde{s}_{m_{k}}^{[d_{k},e_{k}]}(b)=0$ where $\tilde{s}_{m_{k}}^{[d_{k},e_{k}]}\in J_{\boldsymbol{h}^{\prime}}^{[\boldsymbol{d}^{\prime},\boldsymbol{e}^{\prime}]}(M)^{0}[[X]]\otimes_{\mathcal{O}_{\K}}p^{-h_{k}m_{k}}\mathcal{O}_{\K}$ is a lift of $s_{m_{k}}^{[d_{k},e_{k}]}$. By Corollary\ \ref{cor of remainder theorem}, we have $\tilde{s}_{m_{k}}^{[d_{k},e_{k}]}\in \Omega_{m_{k}}^{[d_{k},e_{k}]}J_{\boldsymbol{h}^{\prime}}^{[\boldsymbol{d}^{\prime},\boldsymbol{e}^{\prime}]}(M)^{0}[[X]]\otimes_{\mathcal{O}_{\K}}p^{-h_{k}m_{k}}\mathcal{O}_{\K}$ for every $m_{k}\in \mathbb{Z}_{\geq 0}$, which implies that $s^{[d_{k},e_{k}]}=(s_{m_{k}}^{[d_{k},e_{k}]})_{m_{k}\in \mathbb{Z}_{\geq 0}}=0$. Thus, $\varphi_{k}$ is injective.}
\par 
{Finally, we prove that $\varphi_{k}$ is surjective. Let $s^{[\boldsymbol{d},\boldsymbol{e}]}=(s_{\boldsymbol{m}}^{[\boldsymbol{d},\boldsymbol{e}]})_{\boldsymbol{m}\in\mathbb{Z}_{\geq 0}^{k}}\in J_{\boldsymbol{h}}^{[\boldsymbol{d},\boldsymbol{e}]}(M)^{0}$. We fix an integer $m_{k}\in \mathbb{Z}_{\geq 0}$. By Proposition \ref{multivariable weiestrass polynomial pro}, for each $\boldsymbol{m}^{\prime}\in \mathbb{Z}_{\geq 0}^{k-1}$, there exists a unique element $r(s_{(\boldsymbol{m^{\prime}},m_{k})}^{[\boldsymbol{d},\boldsymbol{e}]})\in M^{0}[X_{1},\ldots, X_{k}]\otimes_{\mathcal{O}_{\K}}p^{-\langle \boldsymbol{h},(\boldsymbol{m}^{\prime},m_{k})\rangle_{k}}\mathcal{O}_{\K}$ satisfying 
the congruence $s_{(\boldsymbol{m^{\prime}},m_{k})}^{[\boldsymbol{d},\boldsymbol{e}]}\equiv r(s_{(\boldsymbol{m^{\prime}},m_{k})}^{[\boldsymbol{d},\boldsymbol{e}]})\ \mathrm{mod}\ (\Omega_{(\boldsymbol{m}^{\prime},m_{k})}^{[\boldsymbol{d},\boldsymbol{e}]})$ and the inequality $\deg_{X_{i}} r(s_{(\boldsymbol{m^{\prime}},m_{k})}^{[\boldsymbol{d},\boldsymbol{e}]})<\deg \Omega_{m_{i}^{\prime}}^{[d_{i},e_{i}]}$ for each $1\leq i\leq k-1$, as well as the inequality $\deg_{X_{k}} r(s_{(\boldsymbol{m^{\prime}},m_{k})}^{[\boldsymbol{d},\boldsymbol{e}]})<\deg \Omega_{m_{k}}^{[d_{k},e_{k}]}$. Put $r(s_{(\boldsymbol{m}^{\prime},m_{k})}^{[\boldsymbol{d},\boldsymbol{e}]})=\sum_{j=0}^{(\deg \Omega_{m_{k}}^{[d_{k},e_{k}]})-1}X_{k}^{j}r^{(j)}(s_{(\boldsymbol{m}^{\prime},m_{k})}^{[\boldsymbol{d}, \boldsymbol{e}]})$ with $r^{(j)}(s_{(\boldsymbol{m}^{\prime},m_{k})}^{[\boldsymbol{d},\boldsymbol{e}]})\in M^{0}[X_{1},\ldots, X_{k-1}]\otimes_{\mathcal{O}_{\K}}p^{-\langle \boldsymbol{h},(\boldsymbol{m}^{\prime},m_{k})\rangle_{k}}\mathcal{O}_{\K}$. Let $\boldsymbol{m}^{\prime},\boldsymbol{n}^{\prime}\in \mathbb{Z}_{\geq 0}^{k-1}$ with $\boldsymbol{n}^{\prime}\geq \boldsymbol{m}^{\prime}$. 
Since the congruence $r(s_{(\boldsymbol{n}^{\prime},m_{k})}^{[\boldsymbol{d},\boldsymbol{e}]})\equiv r(s_{(\boldsymbol{m}^{\prime},m_{k})}^{[\boldsymbol{d},\boldsymbol{e}]})\ \mathrm{mod}\ (\Omega_{(\boldsymbol{m}^{\prime},m_{k})}^{[\boldsymbol{d},\boldsymbol{e}]})$ holds, we have 
\begin{equation}\label{multi J induction provarphi varphiksurj rs implies=0 eq}
(r(s_{(\boldsymbol{n}^{\prime},m_{k})}^{[\boldsymbol{d},\boldsymbol{e}]})-r(s_{(\boldsymbol{m}^{\prime},m_{k})}^{[\boldsymbol{d},\boldsymbol{e}]}))(b_{1},\ldots, b_{k-1},X_{k})\in p^{-\langle \boldsymbol{h},(\boldsymbol{n}^{\prime},m_{k})\rangle_{k}}\Omega_{m_{k}}^{[d_{k},e_{k}]}(X_{k})M^{0}_{\K(b_{1},\ldots, b_{k-1})}[X_{k}]
\end{equation}
 for every root $b_{i}\in \overline{\K}$ of $\Omega_{m_{i}}^{[d_{i},e_{i}]}$ with $1\leq i\leq k-1$. Since we have the inequality $\deg_{X_{k}}(r(s_{(\boldsymbol{n}^{\prime},m_{k})}^{[\boldsymbol{d},\boldsymbol{e}]})-r(s_{(\boldsymbol{m}^{\prime},m_{k})}^{[\boldsymbol{d},\boldsymbol{e}]}))(b_{1},\ldots, b_{k-1},X_{k})<\deg \Omega_{m_{k}}^{[d_{k},e_{k}]}$, by  Proposition \ref{Weiestrass on Banach space}, \eqref{multi J induction provarphi varphiksurj rs implies=0 eq} implies that 
\begin{equation}
(r(s_{(\boldsymbol{n}^{\prime},m_{k})}^{[\boldsymbol{d},\boldsymbol{e}]})-r(s_{(\boldsymbol{m}^{\prime},m_{k})}^{[\boldsymbol{d},\boldsymbol{e}]}))(b_{1},\ldots, b_{k-1},X_{k})=0.
\end{equation}
Since we have 
\begin{multline*}(r(s_{(\boldsymbol{n}^{\prime},m_{k})}^{[\boldsymbol{d},\boldsymbol{e}]})-r(s_{(\boldsymbol{m}^{\prime},m_{k})}^{[\boldsymbol{d},\boldsymbol{e}]}))(b_{1},\ldots, b_{k-1},X_{k})
\\ 
=\sum_{j=0}^{(\deg \Omega_{m_{k}}^{[d_{k},e_{k}]})-1}(r^{(j)}(s_{(\boldsymbol{n}^{\prime},m_{k})}^{[\boldsymbol{d},\boldsymbol{e}]})-r^{(j)}(s_{(\boldsymbol{m}^{\prime},m_{k})}^{[\boldsymbol{d},\boldsymbol{e}]}))(b_{1},\ldots, b_{k-1})X_{k}^{j},
\end{multline*} 
we have
$$(r^{(j)}(s_{(\boldsymbol{n}^{\prime},m_{k})}^{[\boldsymbol{d},\boldsymbol{e}]})-r^{(j)}(s_{(\boldsymbol{m}^{\prime},m_{k})}^{[\boldsymbol{d},\boldsymbol{e}]}))(b_{1},\ldots, b_{k-1})=0$$
for each $0\leq j<\deg \Omega_{m_{k}}^{[d_{k},e_{k}]}$. By Corollary\ \ref{lemma for main theorem two}, we see that 
\begin{equation}\label{multi J induction provarphi varphiksurj berfer rjmk eq}
r^{(j)}(s_{(\boldsymbol{n}^{\prime},m_{k})}^{[\boldsymbol{d},\boldsymbol{e}]})-r^{(j)}(s_{(\boldsymbol{m}^{\prime},m_{k})}^{[\boldsymbol{d},\boldsymbol{e}]})\in (\Omega_{\boldsymbol{m}^{\prime}}^{[\boldsymbol{d}^{\prime},\boldsymbol{e}^{\prime}]})M_{\K}^{0}[[X_{1},\ldots, X_{k-1}]]\otimes_{\mathcal{O}_{\K}}\K
\end{equation}
for every $\boldsymbol{m}^{\prime},\boldsymbol{n}^{\prime}\in \mathbb{Z}_{\geq 0}^{k-1}$ with $\boldsymbol{n}^{\prime}\geq \boldsymbol{m}^{\prime}$ and for every $m_{k}\in \mathbb{Z}_{\geq 0}$ and for every $0\leq j< \deg \Omega_{m_{k}}^{[d_{k},e_{k}]}$. By \eqref{multi J induction provarphi varphiksurj berfer rjmk eq}, we have $([r^{(j)}(s_{(\boldsymbol{m}^{\prime},m_{k})}^{[\boldsymbol{d},\boldsymbol{e}]})])_{\boldsymbol{m}^{\prime}\in \mathbb{Z}_{\geq 0}^{k-1}}\in \varprojlim_{\boldsymbol{m}^{\prime}\in \mathbb{Z}_{\geq 0}^{k-1}}\linebreak\left(\frac{M^{0}[[X_{1},\ldots, X_{k-1}]]}{(\Omega_{\boldsymbol{m}^{\prime}}^{[\boldsymbol{d}^{\prime},\boldsymbol{e}^{\prime}]}M^{0}[[X_{1},\ldots, X_{k-1}]])}\otimes_{\mathcal{O}_{\K}}\K\right)$ for every $m_{k}\in \mathbb{Z}_{\geq 0}$ and for every $0\leq j<\deg \Omega_{m_{k}}^{[d_{k},e_{k}]}$. Put $r^{(j)}(s^{[\boldsymbol{d},\boldsymbol{e}]}_{m_{k}})=([r^{(j)}(s_{(\boldsymbol{m}^{\prime},m_{k})}^{[\boldsymbol{d},\boldsymbol{e}]})])_{\boldsymbol{m}^{\prime}\in \mathbb{Z}_{\geq 0}^{k-1}}$. Since $r^{(j)}(s_{(\boldsymbol{m}^{\prime},m_{k})}^{[\boldsymbol{d},\boldsymbol{e}]})$ is in $M^{0}[X_{1},\ldots, X_{k}]\otimes_{\mathcal{O}_{\K}}p^{-\langle \boldsymbol{h},(\boldsymbol{m}^{\prime},m_{k})\rangle_{k}}\mathcal{O}_{\K}$ for every $\boldsymbol{m}^{\prime}\in \mathbb{Z}_{\geq 0}^{k-1}$, $m_{k}\in \mathbb{Z}_{\geq 0}$ and $0\leq j<\deg\Omega_{m_{k}}^{[d_{k},e_{k}]}$, we see that $r^{(j)}(s^{[\boldsymbol{d},\boldsymbol{e}]}_{m_{k}})\in J_{\boldsymbol{h}^{\prime}}^{[\boldsymbol{d}^{\prime},\boldsymbol{e}^{\prime}]}(M)^{0}\otimes_{\mathcal{O}_{\K}}p^{-h_{k}m_{k}}\mathcal{O}_{\K}$. Put $r_{m_{k}}^{[d_{k},e_{k}]}=\sum_{j=0}^{(\deg\Omega_{m_{k}}^{[d_{k},e_{k}]})-1}X_{k}^{j}r^{(j)}(s^{[\boldsymbol{d},\boldsymbol{e}]}_{m_{k}})\in J_{\boldsymbol{h}^{\prime}}^{[\boldsymbol{d}^{\prime},\boldsymbol{e}^{\prime}]}(M)^{0}[X_{k}]\otimes_{\mathcal{O}_{\K}}p^{-h_{k}m_{k}}\mathcal{O}_{\K}$. 
By definition, for each root of $b\in \overline{\K}$ of $\Omega_{m_{k}}^{[d_{k},e_{k}]}$, we have 
\begin{align}\label{multi J induction provarphi varphiksurj rmk rboldsymboldeboldm eq}
\begin{split}
r_{m_{k}}^{[d_{k},e_{k}]}(b)&=\sum_{j=0}^{(\deg\Omega_{m_{k}}^{[d_{k},e_{k}]})-1}r^{(j)}(s^{[\boldsymbol{d},\boldsymbol{e}]}_{m_{k}})b^{j}\\
&=\left(\left[\sum_{j=0}^{(\deg\Omega_{m_{k}}^{[d_{k},e_{k}]})-1}r^{(j)}(s^{[\boldsymbol{d},\boldsymbol{e}]}_{(\boldsymbol{m}^{\prime},m_{k})})b^{j}\right]\right)_{\boldsymbol{m}^{\prime}\in \mathbb{Z}_{\geq 0}^{k-1}}\\ 
&=\left(\left[r(s^{[\boldsymbol{d},\boldsymbol{e}]}_{(\boldsymbol{m}^{\prime},m_{k})})(b)\right]\right)_{\boldsymbol{m}^{\prime}\in \mathbb{Z}_{\geq 0}^{k-1}}.
\end{split}
\end{align}
Since $r(s^{[\boldsymbol{d},\boldsymbol{e}]}_{\boldsymbol{n}})\equiv r(s^{[\boldsymbol{d},\boldsymbol{e}]}_{\boldsymbol{m}})\ \mathrm{mod}\ (\Omega_{\boldsymbol{m}}^{[\boldsymbol{d},\boldsymbol{e}]})$ for every $\boldsymbol{m},\boldsymbol{n}\in\mathbb{Z}_{\geq 0}^{k}$ with $\boldsymbol{n}\geq \boldsymbol{m}$, we have $r(s^{[\boldsymbol{d},\boldsymbol{e}]}_{(\boldsymbol{m}^{\prime},m_{k}+1)})(b)\equiv r(s^{[\boldsymbol{d},\boldsymbol{e}]}_{(\boldsymbol{m}^{\prime},m_{k})})(b)\ \mathrm{mod}\ (\Omega_{\boldsymbol{m}^{\prime}}^{[\boldsymbol{d}^{\prime},\boldsymbol{e}^{\prime}]})$ for every $\boldsymbol{m}^{\prime}\in \mathbb{Z}_{\geq 0}^{k-1}$, $m_{k}\in \mathbb{Z}_{\geq 0}$ and for every root $b\in \overline{\K}$ of $\Omega_{m_{k}}^{[d_{k},e_{k}]}$. By \eqref{multi J induction provarphi varphiksurj rmk rboldsymboldeboldm eq}, we have $r_{m_{k}+1}^{[d_{k},e_{k}]}(b)=r_{m_{k}}^{[d_{k},e_{k}]}(b)$ for every $m_{k}\in \mathbb{Z}_{\geq 0}$ and for every root $b\in \overline{\K}$ of $\Omega_{m_{k}}^{[d_{k},e_{k}]}$. By Corollary\ \ref{cor of remainder theorem}, we see that $([r_{m_{k}}^{[d_{k},e_{k}]}])_{m_{k}\in \mathbb{Z}_{\geq 0}}\in \varprojlim_{m_{k}\in \mathbb{Z}_{\geq 0}}\left(\frac{J_{\boldsymbol{h}^{\prime}}^{[\boldsymbol{d}^{\prime},\boldsymbol{e}^{\prime}]}(M)^{0}[[X_{k}]]}{\Omega_{m_{k}}^{[d_{k},e_{k}]}J_{\boldsymbol{h}^{\prime}}^{[\boldsymbol{d}^{\prime},\boldsymbol{e}^{\prime}]}(M)^{0}[[X_{k}]]}\otimes_{\mathcal{O}_{\K}}\K\right)$. Put $s^{[d_{k},e_{k}]}=([r_{m_{k}}^{[d_{k},e_{k}]}])_{m_{k}\in \mathbb{Z}_{\geq 0}}$. Since $r_{m_{k}}^{[d_{k},e_{k}]}\in J_{\boldsymbol{h}^{\prime}}^{[\boldsymbol{d}^{\prime},\boldsymbol{e}^{\prime}]}(M)^{0}[X_{k}]\otimes_{\mathcal{O}_{\K}}p^{-h_{k}m_{k}}\mathcal{O}_{\K}$ for every $m_{k}\in \mathbb{Z}_{\geq 0}$, we have $s^{[d_{k},e_{k}]}\in J_{h_{k}}^{[d_{k},e_{k}]}(J_{\boldsymbol{h}^{\prime}}^{[\boldsymbol{d}^{\prime},\boldsymbol{e}^{\prime}]}(M))^{0}$. By \eqref{multi J induction provarphi varphiksurj rmk rboldsymboldeboldm eq}, we have 
$\psi_{b,m_{k}}(s^{[d_{k},e_{k}]})=\left(\left[r(s^{[\boldsymbol{d},\boldsymbol{e}]}_{(\boldsymbol{m}^{\prime},m_{k})})(b)\right]\right)_{\boldsymbol{m}^{\prime}\in \mathbb{Z}_{\geq 0}^{k-1}} =\varphi_{b,m_{k}}(s^{[\boldsymbol{d},\boldsymbol{e}]})$ for every $m_{k}\in \mathbb{Z}_{\geq 0}$ and for every root $b\in \overline{\K}$ of $\Omega_{m_{k}}^{[d_{k},e_{k}]}$. Therefore, we have $\varphi_{k}(s^{[d_{e},e_{k}]})=s^{[\boldsymbol{d},\boldsymbol{e}]}$ and we conclude that $\varphi_{k}$ is surjective.}
\end{proof}
\begin{thm}\label{main thm 2 and proof}
Assume that $\boldsymbol{e}-\boldsymbol{d}\geq \lfloor \boldsymbol{h}\rfloor$. For $s^{[\boldsymbol{d},\boldsymbol{e}]}=(s_{\boldsymbol{m}}^{[\boldsymbol{d},\boldsymbol{e}]})_{\boldsymbol{m}\in \mathbb{Z}_{\geq 0}^{k}}\in J_{\boldsymbol{h}}^{[\boldsymbol{d},\boldsymbol{e}]}(M)$, there exists a unique element $f_{s^{[\boldsymbol{d},\boldsymbol{e}]}}\in \HH_{\boldsymbol{h}}(M)$ such that 
$$f_{s^{[\boldsymbol{d},\boldsymbol{e}]}}-\tilde{s}_{\boldsymbol{m}}^{[\boldsymbol{d},\boldsymbol{e}]}\in (\Omega_{\boldsymbol{m}}^{[\boldsymbol{d},\boldsymbol{e}]})\HH_{\boldsymbol{h}}(M)$$
for each $\boldsymbol{m}\in \mathbb{Z}_{\geq 0}^{k}$, where $\tilde{s}_{\boldsymbol{m}}^{[\boldsymbol{d},\boldsymbol{e}]}\in M^{0}[[X_{1},\ldots, X_{k}]]\otimes_{\mathcal{O}_{\K}}\K$ is a lift of $s_{\boldsymbol{m}}^{[\boldsymbol{d},\boldsymbol{e}]}$. Further, the correspondence $s^{[\boldsymbol{d},\boldsymbol{e}]}\mapsto f_{s^{[\boldsymbol{d},\boldsymbol{e}]}}$ from $J_{\boldsymbol{h}}^{[\boldsymbol{d},\boldsymbol{e}]}(M)$ to $\HH_{\boldsymbol{h}}(M)$ induces an $\mathcal{O}_\mathcal{K}[ [X_{1},\ldots, X_{k}] ]\otimes_{\mathcal{O}_{\K}}\K$-module isomorphism 
$$
J_{\boldsymbol{h}}^{[\boldsymbol{d},\boldsymbol{e}]}(M) \overset{\sim}{\longrightarrow} \HH_{\boldsymbol{h}}(M)
$$ 
and, via the above isomorphism, we have
$$
\{f\in \HH_{\boldsymbol{h}}(M)\vert v_{\HH_{\boldsymbol{h}}}(f)\geq \alpha_{\boldsymbol{h}}^{[\boldsymbol{d},\boldsymbol{e}]}\}\subset J_{\boldsymbol{h}}^{[\boldsymbol{d},\boldsymbol{e}]}(M)^{0}\subset \{f\in \HH_{\boldsymbol{h}}(M)\vert v_{\HH_{\boldsymbol{h}}}(f)\geq \beta_{\boldsymbol{h}}\},$$
where $\alpha_{\boldsymbol{h}}^{[\boldsymbol{d},\boldsymbol{e}]}=\sum_{i=1}^{k}\alpha_{h_{i}}^{[d_{i},e_{i}]}$ and $\beta_{\boldsymbol{h}}=\sum_{i=1}^{k}\beta_{h_{i}}$ with
\begin{align*}
\alpha_{h_{i}}^{[d_{i},e_{i}]}&=\begin{cases}\lfloor\frac{(e_{i}-d_{i}+1)}{p-1}+\max\{0, h_{i}-\frac{h_{i}}{\log p}(1+\log \frac{\log p}{(p-1)h_{i}})\}\rfloor+1\ &\mathrm{if}\ h_{i}>0,\\ 0\ &\mathrm{if}\ h_{i}=0,\end{cases}\\
\beta_{h_{i}}&=\begin{cases}-\lfloor\max\{h_{i},\frac{p}{p-1}\}\rfloor-1\ \ \ \  \ \ \ \ \ \ \ \ \ \ \ \ \ \ \ \ \ \ \ \ \ \ \ \ \ \ \ \ \ \ \ \ \ &\mathrm{if}\ h_{i}>0,\\ 0\ &\mathrm{if}\ h_{i}=0.\end{cases}
\end{align*}
\end{thm}
\begin{proof}
We prove this theorem by induction on $k$. When $k=1$, the desired statement is already proved in $\mathrm{Proposition\ \ref{isomorphism HH and project lim Banach}}$. Let us assume that $k\geq 2$. By the induction argument with respect to $k$, we have $J_{\boldsymbol{h}^{\prime}}^{[\boldsymbol{d}^{\prime},\boldsymbol{e}^{\prime}]}(M)\simeq \HH_{\boldsymbol{h}^{\prime}}(M)$ and
\begin{equation}\label{main thm 2 and proofeq 0}
\{f\in \HH_{\boldsymbol{h}^{\prime}}(M)\vert v_{\HH_{\boldsymbol{h}}^{\prime}}(f)\geq \alpha_{\boldsymbol{h}^{\prime}}^{[\boldsymbol{d}^{\prime},\boldsymbol{e}^{\prime}]}\}\subset J_{\boldsymbol{h}^{\prime}}^{[\boldsymbol{d}^{\prime},\boldsymbol{e}^{\prime}]}(M)^{0}\subset \{f\in \HH_{\boldsymbol{h}^{\prime}}(M)\vert v_{\HH_{\boldsymbol{h}^{\prime}}}(f)\geq \beta_{\boldsymbol{h}^{\prime}}\}.
\end{equation}
By \eqref{main thm 2 and proofeq 0}, we can show that we have $J_{h_{k}}^{[d_{k},e_{k}]}(J_{\boldsymbol{h}^{\prime}}^{[\boldsymbol{d}^{\prime},\boldsymbol{e}^{\prime}]}(M))\simeq J_{h_{k}}^{[d_{k},e_{k}]}(\HH_{\boldsymbol{h}^{\prime}}(M))$ and
 \begin{equation}\label{main thm 2 and proofeq 1}
 p^{\alpha_{\boldsymbol{h}^{\prime}}^{[\boldsymbol{d}^{\prime},\boldsymbol{e}^{\prime}]}}J_{h_{k}}^{[d_{k},e_{k}]}(\HH_{\boldsymbol{h}^{\prime}}(M))^{0}\subset J_{h_{k}}^{[d_{k},e_{k}]}(J_{\boldsymbol{h}^{\prime}}^{[\boldsymbol{d}^{\prime},\boldsymbol{e}^{\prime}]}(M))^{0}\subset p^{\beta_{\boldsymbol{h}^{\prime}}}J_{h_{k}}^{[d_{k},e_{k}]}(\HH_{\boldsymbol{h}^{\prime}}(M))^{0}.
 \end{equation}
 On the other hand, by the result in the case $k=1$, we see that $J_{h_{k}}^{[d_{k},e_{k}]}(\HH_{\boldsymbol{h}^{\prime}}(M))\simeq \HH_{h_{k}}(\HH_{\boldsymbol{h}^{\prime}}(M))$ and
\begin{multline}\label{main thm 2 and proofeq 2}
\left\{f\in \HH_{h_{k}}(\HH_{\boldsymbol{h}^{\prime}}(M))\big\vert v_{\HH_{h_{k}}(\HH_{\boldsymbol{h}^{\prime}}(M))}(f)\geq \alpha_{h_{k}}^{[d_{k},e_{k}]}\right\}\subset J_{h_{k}}^{[d_{k},e_{k}]}(\HH_{\boldsymbol{h}^{\prime}}(M))^{0}\\
\subset \left\{f\in \HH_{h_{k}}(\HH_{\boldsymbol{h}^{\prime}}(M))\big\vert v_{\HH_{h_{k}}(\HH_{\boldsymbol{h}^{\prime}}(M))}(f)\geq \beta_{h_{k}}\right\}
\end{multline}
where $v_{\HH_{h_{k}}(\HH_{\boldsymbol{h}^{\prime}}(M))}$ is the valuation on $\HH_{h_{k}}(\HH_{\boldsymbol{h}^{\prime}}(M))$. Therefore, by \eqref{main thm 2 and proofeq 1} and \eqref{main thm 2 and proofeq 2}, we have $J_{h_{k}}^{[d_{k},e_{k}]}(J_{\boldsymbol{h}^{\prime}}^{[\boldsymbol{d}^{\prime},\boldsymbol{e}^{\prime}]}(M))\simeq \HH_{h_{k}}(\HH_{\boldsymbol{h}^{\prime}}(M))$ and
\begin{multline}\label{main thm 2 and proofeq 3}
\{f\in \HH_{h_{k}}(\HH_{\boldsymbol{h}^{\prime}}(M))\vert v_{\HH_{h_{k}}(\HH_{\boldsymbol{h}^{\prime}}(M))}(f)\geq \alpha_{\boldsymbol{h}}^{[\boldsymbol{d},\boldsymbol{e}]}\}\subset J_{h_{k}}^{[d_{k},e_{k}]}(J_{\boldsymbol{h}^{\prime}}^{[\boldsymbol{d}^{\prime},\boldsymbol{e}^{\prime}]}(M))^{0}\\
\subset \{f\in  \HH_{h_{k}}(\HH_{\boldsymbol{h}^{\prime}}(M))\vert v_{\HH_{h_{k}}(\HH_{\boldsymbol{h}^{\prime}}(M))}(f)\geq\beta_{\boldsymbol{h}}\}.
\end{multline}
By Proposition \ref{isometry ofHh for induction}, we have an isometric isomorphism $\HH_{h_{k}}(\HH_{\boldsymbol{h}^{\prime}}(M))\simeq \HH_{\boldsymbol{h}}(M)$. Further, by Proposition \ref{multi J induction pro}, we have an $\mathcal{O}_{\K}[[[X_{1},\ldots, X_{k}]]\otimes_{\mathcal{O}_{\K}}\K$-module isomorphism $J_{\boldsymbol{h}}^{[\boldsymbol{d},\boldsymbol{e}]}(M)\simeq J_{h_{k}}^{[d_{k},e_{k}]}(J_{\boldsymbol{h}^{\prime}}^{[\boldsymbol{d}^{\prime},\boldsymbol{e}^{\prime}]}(M))$ induced by an $\mathcal{O}_{\K}[[X_{1},\ldots, X_{k}]]$-module isomorphism $J_{\boldsymbol{h}}^{[\boldsymbol{d},\boldsymbol{e}]}(M)^{0}\simeq J_{h_{k}}^{[d_{k},e_{k}]}(J_{\boldsymbol{h}^{\prime}}^{[\boldsymbol{d}^{\prime},\boldsymbol{e}^{\prime}]}(M))^{0}$. Therefore, by \eqref{main thm 2 and proofeq 3}, we have $J_{\boldsymbol{h}}^{[\boldsymbol{d},\boldsymbol{e}]}(M)\simeq \HH_{\boldsymbol{h}}(M)$ and
$$
\{f\in \HH_{\boldsymbol{h}}(M)\vert v_{\HH_{\boldsymbol{h}}}(f)\geq \alpha_{\boldsymbol{h}}^{[\boldsymbol{d},\boldsymbol{e}]}\}\subset J_{\boldsymbol{h}}^{[\boldsymbol{d},\boldsymbol{e}]}(M)^{0}\subset \{f\in \HH_{\boldsymbol{h}}(M)\vert v_{\HH_{\boldsymbol{h}}}(f)\geq \beta_{\boldsymbol{h}}\}.$$
\end{proof}
\begin{rem}\label{remark of uniqueness of jh}
Assume that $\boldsymbol{e}-\boldsymbol{d}\geq \lfloor \boldsymbol{h}\rfloor$. We regard $M^{0}[[X_{1},\ldots, X_{k}]]\otimes_{\mathcal{O}_{\K}}\K$ as an $\mathcal{O}_{\K}[[X_{1},\ldots, X_{k}]]\otimes_{\mathcal{O}_{\K}}\K$-submodule of $J_{\boldsymbol{h}}^{[\boldsymbol{d},\boldsymbol{e}]}(M)$ and $\HH_{\boldsymbol{h}}(M)$ naturally and denote by $i: M^{0}[[X_{1},\ldots, X_{k}]]\otimes_{\mathcal{O}_{\K}}\K\rightarrow J_{\boldsymbol{h}}^{[\boldsymbol{d},\boldsymbol{e}]}(M) $ and $j:M^{0}[[X_{1},\ldots, X_{k}]]\otimes_{\mathcal{O}_{\K}}\K\rightarrow \HH_{\boldsymbol{h}}(M)$ the natural inclusion maps respectively. We denot by $\varphi: J_{\boldsymbol{h}}^{[\boldsymbol{d},\boldsymbol{e}]}(M) \overset{\sim}{\longrightarrow} \HH_{\boldsymbol{h}}(M)$ the $\mathcal{O}_{\K}[[X_{1},\ldots, X_{k}]]\otimes_{\mathcal{O}_{\K}}\K$-module isomorphism defined in Theorem \ref{main thm 2 and proof}. We remark that $\varphi$ is the unique \linebreak$\mathcal{O}_{\K}[[X_{1},\ldots, X_{k}]]\otimes_{\mathcal{O}_{\K}}\K$-module isomorphism from $J_{\boldsymbol{h}}^{[\boldsymbol{d},\boldsymbol{e}]}(M)$ into $\HH_{\boldsymbol{h}}(M)$ which satisfies $\varphi i=j$.

Indeed, let $\alpha:J_{\boldsymbol{h}}^{[\boldsymbol{d},\boldsymbol{e}]}(M) \overset{\sim}{\longrightarrow} \HH_{\boldsymbol{h}}(M)$ be another  $\mathcal{O}_{\K}[[X_{1},\ldots, X_{k}]]\otimes_{\mathcal{O}_{\K}}\K$-module isomorphism which satisfies $\alpha i=j$. Let $s^{[\boldsymbol{d},\boldsymbol{e}]}=(s_{\boldsymbol{m}}^{[\boldsymbol{d},\boldsymbol{e}]})_{\boldsymbol{m}\in \mathbb{Z}_{\geq 0}^{k}}\in J_{\boldsymbol{h}}^{[\boldsymbol{d},\boldsymbol{e}]}(M)$. By Theorem \ref{main thm 2 and proof}, we have 
$$\varphi(s^{[\boldsymbol{d},\boldsymbol{e}]})-j(\tilde{s}_{\boldsymbol{m}}^{[\boldsymbol{d},\boldsymbol{e}]})\in (\Omega_{\boldsymbol{m}}^{[\boldsymbol{d},\boldsymbol{e}]})\HH_{\boldsymbol{h}}(M)$$
for each $\boldsymbol{m}\in \mathbb{Z}_{\geq 0}^{k}$, where $\tilde{s}_{\boldsymbol{m}}^{[\boldsymbol{d},\boldsymbol{e}]}\in M^{0}[[X_{1},\ldots, X_{k}]]\otimes_{\mathcal{O}_{\K}}\K$ is a lift of $s_{\boldsymbol{m}}^{[\boldsymbol{d},\boldsymbol{e}]}$. Therefore, we have
\begin{equation}\label{remark of uniqueness of jh 1}
\alpha^{-1}\varphi(s^{[\boldsymbol{d},\boldsymbol{e}]})-i(\tilde{s}_{\boldsymbol{m}}^{[\boldsymbol{d},\boldsymbol{e}]})\in (\Omega_{\boldsymbol{m}}^{[\boldsymbol{d},\boldsymbol{e}]})J_{\boldsymbol{h}}^{[\boldsymbol{d},\boldsymbol{e}]}(M)
\end{equation}
for every $\boldsymbol{m}\in \mathbb{Z}_{\geq 0}^{k}$. Put $\alpha^{-1}\varphi(s^{[\boldsymbol{d},\boldsymbol{e}]})=(w^{[\boldsymbol{d},\boldsymbol{e}]}_{\boldsymbol{m}})_{\boldsymbol{m}\in \mathbb{Z}_{\geq 0}^{k}}$.  By \eqref{remark of uniqueness of jh 1}, we see that $w^{[\boldsymbol{d},\boldsymbol{e}]}_{\boldsymbol{m}}=s^{[\boldsymbol{d},\boldsymbol{e}]}_{\boldsymbol{m}}$ for every $\boldsymbol{m}\in \mathbb{Z}_{\geq 0}^{k}$. Then, we have $\alpha^{-1}\varphi(s^{[\boldsymbol{d},\boldsymbol{e}]})=s^{[\boldsymbol{d},\boldsymbol{e}]}$, which is equivalent to $\varphi(s^{[\boldsymbol{d},\boldsymbol{e}]})=\alpha(s^{[\boldsymbol{d},\boldsymbol{e}]})$. Thus, we conclude that $\varphi=\alpha$.
\end{rem}
\begin{lem}\label{lemma for main prop 3}
Let $\boldsymbol{n}\in\mathbb{Z}_{\geq 0}^{k}$, $1\leq l\leq k$ and $s^{[\boldsymbol{i}]}\in M^{0}[[X_{1},\ldots, X_{k}]]\otimes_{\mathcal{O}_{\K}}\K$ where $\boldsymbol{i}\in [\boldsymbol{d},\boldsymbol{e}]$. We assume that
\begin{equation}\label{lemma for main prop 3 eq 1}
\theta^{(\boldsymbol{j})}=\displaystyle{\sum_{\boldsymbol{i}\in [\boldsymbol{d},\boldsymbol{j}]}}\left(\prod_{t=1}^{k}\begin{pmatrix}j_{t}-d_{t}\\i_{t}-d_{t}\end{pmatrix}\right)(-1)^{\sum_{t=1}^{k}(j_{t}-i_{t})}s^{[\boldsymbol{i}]}\in p^{\langle \boldsymbol{n},\boldsymbol{j}-\boldsymbol{d}\rangle_{k}}M^{0}[[X_{1},\ldots, X_{k}]]
\end{equation}
for each $\boldsymbol{j}\in[\boldsymbol{d},\boldsymbol{e}]$. Then, we have
$$\displaystyle{\sum_{\boldsymbol{i}\in [\boldsymbol{d}_{(l)},\boldsymbol{j}_{(l)}]}}\left(\prod_{t=1}^{l}\begin{pmatrix}j_{t}-d_{t}\\ i_{t}-d_{t}\end{pmatrix}\right)(-1)^{\sum_{t=1}^{l}(j_{t}-i_{t})}s^{[(\boldsymbol{i},\boldsymbol{j}^{(l)})]}\in p^{\langle \boldsymbol{n}_{(l)},\boldsymbol{j}_{(l)}-\boldsymbol{d}_{(l)}\rangle_{l}}M^{0}[[X_{1},\ldots, X_{k}]]$$
for each $\boldsymbol{j}\in [\boldsymbol{d},\boldsymbol{e}]$, where $\boldsymbol{j}_{(l)}=(j_{1},\ldots, j_{l})$ and $\boldsymbol{j}^{(l)}=(j_{l+1},\ldots, j_{k})$. If $l=k$, we define $(\boldsymbol{i},\boldsymbol{j}^{(l)})$ to be $\boldsymbol{i}$.
\end{lem}
\begin{proof}
Put $\theta^{(\boldsymbol{j})}_{l}=\displaystyle{\sum_{\boldsymbol{i}\in [\boldsymbol{d}_{(l)},\boldsymbol{j}_{(l)}]}}\left(\prod_{t=1}^{l}\begin{pmatrix}j_{t}-d_{t}\\ i_{t}-d_{t}\end{pmatrix}\right)(-1)^{\sum_{t=1}^{l}(j_{t}-i_{t})}s^{[(\boldsymbol{i},\boldsymbol{j}^{(l)})]}$, where $\boldsymbol{j}\in [\boldsymbol{d},\boldsymbol{e}]$. Let $\boldsymbol{j}\in [\boldsymbol{d},\boldsymbol{e}]$. 
{It suffices to prove the followings:
\begin{enumerate}
\item If $\boldsymbol{j}^{(l)}=\boldsymbol{d}^{(l)}$, we have $\theta^{(\boldsymbol{j})}_{l}\in p^{\langle \boldsymbol{n}_{(l)},\boldsymbol{j}_{(l)}-\boldsymbol{d}_{(l)}\rangle_{l}}M^{0}[[X_{1},\ldots, X_{k}]]$.
\item Assume that $\boldsymbol{d}^{(l)}<\boldsymbol{j}^{(l)}$. If $\theta^{(\boldsymbol{j}_{(l)},\boldsymbol{i})}_{l}$ is contained in $p^{\langle \boldsymbol{n}_{(l)},\boldsymbol{j}_{(l)}-\boldsymbol{d}_{(l)}\rangle_{l}}M^{0}[[X_{1},\ldots, X_{k}]]$ for each $\boldsymbol{d}^{(l)}\leq \boldsymbol{i}<\boldsymbol{j}^{(l)}$, we have $\theta^{(\boldsymbol{j})}_{l}\in p^{\langle \boldsymbol{n}_{(l)},\boldsymbol{j}_{(l)}-\boldsymbol{d}_{(l)}\rangle_{l}}M^{0}[[X_{1},\ldots, X_{k}]]$.
\end{enumerate}}
If $\boldsymbol{j}^{(l)}=\boldsymbol{d}^{(l)}$, we have $\theta^{(\boldsymbol{j})}_{l}=\theta^{(\boldsymbol{j})}$. Then, by the assumption \eqref{lemma for main prop 3 eq 1}, we have
$$\theta^{(\boldsymbol{j})}_{l}\in p^{\langle \boldsymbol{n}_{(l)},\boldsymbol{j}_{(l)}-\boldsymbol{d}_{(l)}\rangle_{l}}M^{0}[[X_{1},\ldots, X_{k}]].$$
Next, we assume that $\boldsymbol{d}^{(l)}<\boldsymbol{j}^{(l)}$ and assume that 
$\theta^{(\boldsymbol{j}_{(l)},\boldsymbol{i})}_{l}$ is contained in $p^{\langle \boldsymbol{n}_{(l)},\boldsymbol{j}_{(l)}-\boldsymbol{d}_{(l)}\rangle_{l}}M^{0}[[X_{1},\ldots, X_{k}]]$ for each $\boldsymbol{d}^{(l)}\leq \boldsymbol{i}<\boldsymbol{j}^{(l)}$. By definition, we see that
$$\theta^{(\boldsymbol{j})}=\sum_{\boldsymbol{i}\in [\boldsymbol{d}^{(l)},\boldsymbol{j}^{(l)}]}\left(\prod_{t=l+1}^{k}\begin{pmatrix}j_{t}-d_{t}\\i_{t-l}-d_{t}\end{pmatrix}\right)(-1)^{\sum_{t=l+1}^{k}(j_{t}-i_{t-l})}\theta^{(\boldsymbol{j}_{(l)},\boldsymbol{i})}_{l}.$$
Therefore, $\theta^{(\boldsymbol{j})}_{l}=\theta^{(\boldsymbol{j})}-\sum_{\boldsymbol{d}^{(l)}\leq \boldsymbol{i}<\boldsymbol{j}^{(l)}}\left(\prod_{t=l+1}^{k}\begin{pmatrix}j_{t}-d_{t}\\i_{t-l}-d_{t}\end{pmatrix}\right)(-1)^{\sum_{t=l+1}^{k}(j_{t}-i_{t-l})}\theta^{(\boldsymbol{j}_{(l)},\boldsymbol{i})}_{l}$ is contained in $ p^{\langle \boldsymbol{n}_{(l)},\boldsymbol{j}_{(l)}-\boldsymbol{d}_{(l)}\rangle_{l}}M^{0}[[X_{1},\ldots, X_{k}]]$.
\end{proof}
Let $(\Omega_{\boldsymbol{m}}^{[\boldsymbol{d},\boldsymbol{e}]}(\gamma_{1},\ldots, \gamma_{k}))$ be the ideal of $\mathcal{O}_{\K}[[\Gamma]]$ generated by $\Omega_{m_{1}}^{[d_{1},e_{1}]}(\gamma_{1}),\ldots,\Omega_{m_{k}}^{[d_{k},e_{k}]}(\gamma_{k})$ with $\Omega_{m_{i}}^{[d_{i},e_{i}]}(\gamma_{i})=\prod_{j=d_{i}}^{e_{i}}([\gamma_{i}]^{p^{m_{i}}}-u_{i}^{jp^{m_{i}}})$ for each $\boldsymbol{m}\in \mathbb{Z}_{\geq 0}^{k}$. Let $s\in M^{0}[[\Gamma]]\otimes_{\mathcal{O}_{\K}}\K$ and $\boldsymbol{m}\in \mathbb{Z}_{\geq 0}^{k}$. Via the non-canonical isomorphism $M^{0}[[\Gamma]]\simeq M^{0}[[X_{1},\ldots, X_{k}]]$ in \eqref{non-canonical continuous isomorphihsm of iwasawa module of banach}, by Corollary \ref{lemma for main theorem two}, we see that $s\in \Omega_{\boldsymbol{m}}^{[\boldsymbol{d},\boldsymbol{e}]}(\gamma_{1},\ldots, \gamma_{k})(M^{0}[[\Gamma]]\otimes_{\mathcal{O}_{\K}}\K)$ if and only if 
\begin{equation}\label{multivariable iwasawa modoomega speq}
\kappa(s)=0\ \mathrm{for\ every}\ \kappa\in \mathfrak{X}_{\mathcal{O}_{\K}[[\Gamma]]}^{[\boldsymbol{d},\boldsymbol{e}]}\ \mathrm{with}\ \boldsymbol{m}_{\kappa}\leq \boldsymbol{m}.
\end{equation}
\begin{lem}\label{multivariable litfing prop}
Let $s^{[\boldsymbol{i}]}\in M^{0}[[\Gamma]]\otimes_{\mathcal{O}_{\K}}\K$ for each $\boldsymbol{i}\in [\boldsymbol{d},\boldsymbol{e}]$ and 
we define $\theta^{(\boldsymbol{j})} \in  M^{0}[[\Gamma]]\otimes_{\mathcal{O}_{\K}}\K$ by 
$$\theta^{(\boldsymbol{j})}=\displaystyle{\sum_{\boldsymbol{i}\in [\boldsymbol{d},\boldsymbol{j}]}}\left(\prod_{t=1}^{k}\begin{pmatrix}j_{t}-d_{t}\\i_{t}-d_{t}\end{pmatrix}\right)(-1)^{\sum_{t=1}^{k}(j_{t}-i_{t})}s^{[\boldsymbol{i}]}
$$
for each $\boldsymbol{j}\in [\boldsymbol{d},\boldsymbol{e}]$. 
Let $\boldsymbol{m}\in\mathbb{Z}_{\geq 0}^{k}$. 
Assume that $\theta^{(\boldsymbol{j})}$ is contained in $p^{\langle\boldsymbol{m},(\boldsymbol{j}-\boldsymbol{d})\rangle_{k}}M^{0}[[\Gamma]] \subset 
 M^{0}[[\Gamma]]\otimes_{\mathcal{O}_{\K}}\K$ 
for every $\boldsymbol{j}\in [\boldsymbol{d},\boldsymbol{e}]$. \par 
Then, there exists a unique element $s^{[\boldsymbol{d},\boldsymbol{e}]}\in \frac{M^{0}[[\Gamma]]}{(\Omega_{\boldsymbol{m}}^{[\boldsymbol{d},\boldsymbol{e}]}(\gamma_{1},\ldots,\gamma_{k}))M^{0}[[\Gamma]]}\otimes_{\mathcal{O}_{\K}}p^{-c^{[\boldsymbol{d},\boldsymbol{e}]}}\mathcal{O}_{\K}$ such that the image of $s^{[\boldsymbol{d},\boldsymbol{e}]}$ by the natural projection 
$$
\frac{M^{0}[[\Gamma]]}{(\Omega_{\boldsymbol{m}}^{[\boldsymbol{d},\boldsymbol{e}]}(\gamma_{1},\ldots,\gamma_{k}))M^{0}[[\Gamma]]}\otimes_{\mathcal{O}_{\K}}\K \longrightarrow \frac{M^{0}[[\Gamma]]}{(\Omega_{\boldsymbol{m}}^{[\boldsymbol{i}]}(\gamma_{1},\ldots,\gamma_{k}))M^{0}[[\Gamma]]}\otimes_{\mathcal{O}_{\K}}\K
$$
is equal to the class $[s^{[\boldsymbol{i}]}]_{\boldsymbol{m}} 
\in \tfrac{M^{0}[[\Gamma]]}{(\Omega_{\boldsymbol{m}}^{[\boldsymbol{i}]}(\gamma_{1},\ldots,\gamma_{k}))M^{0}[[\Gamma]]}\otimes_{\mathcal{O}_{\K}}\K$ of $s^{[\boldsymbol{i}]}\in M^{0}[[\Gamma]]\otimes_{\mathcal{O}_{\K}}\K$ for each $\boldsymbol{i}\in [\boldsymbol{d},\boldsymbol{e}]$, 
where we define $c^{[\boldsymbol{d},\boldsymbol{e}]}$ by $c^{[\boldsymbol{d},\boldsymbol{e}]}=\sum_{i=1}^{k}c^{[d_{i},e_{i}]}$ with
\begin{equation}\label{constant for the admissible}
c^{[d_{i},e_{i}]}=\begin{cases}\ord_{p}((e_{i}-d_{i})!)+2(e_{i}-d_{i})+\lfloor\frac{e_{i}-d_{i}+1}{p-1}\rfloor+1\ &\mathrm{if}\ d_{i}<e_{i},\\
0\ &\mathrm{if}\ d_{i}=e_{i}.\end{cases}
\end{equation}
\end{lem}
\begin{proof}
Let $\alpha_{M}^{(k)}$ be the $\mathcal{O}_{\K}$-module isomorphism defined in \eqref{non-canonical continuous isomorphihsm of iwasawa module of banach}. By replacing $s^{[\boldsymbol{i}]}$ with $\alpha_{M}^{(k)}(s^{[\boldsymbol{i}]})\in M^{0}[[X_{1},\ldots, X_{k}]]\otimes_{\mathcal{O}_{\K}}\K$, it suffices to prove that there exists a unique $s^{[\boldsymbol{d},\boldsymbol{e}]}\in \frac{M^{0}[[X_{1},\ldots, X_{k}]]}{(\Omega_{\boldsymbol{m}}^{[\boldsymbol{d},\boldsymbol{e}]}(X_{1},\ldots, X_{k}))M^{0}[[X_{1},\ldots, X_{k}]]}\otimes_{\mathcal{O}_{\K}}p^{-c^{[\boldsymbol{d},\boldsymbol{e}]}}\mathcal{O}_{\K}$ which satisfies
\begin{align}\label{multivariable litfing prop desired interpolation pro}
\tilde{s}^{[\boldsymbol{d},\boldsymbol{e}]}(u^{i_{1}}\epsilon_{1}-1,\ldots, u^{i_{k}}\epsilon_{k}-1)=s^{[\boldsymbol{i}]}(u^{i_{1}}\epsilon_{1}-1,\ldots, u^{i_{k}}\epsilon_{k}-1)
\end{align}
for every $\boldsymbol{i}\in [\boldsymbol{d},\boldsymbol{e}]$ and for every $\boldsymbol{\epsilon}\in \prod_{i=1}^{k}\mu_{p^{m_{i}}}$ where $\tilde{s}^{[\boldsymbol{d},\boldsymbol{e}]}$ is a lift of $s^{[\boldsymbol{d},\boldsymbol{e}]}$. Once we prove the existence of an element $s^{[\boldsymbol{d},\boldsymbol{e}]}$, the uniqueness of $s^{[\boldsymbol{d},\boldsymbol{e}]}$ follows from $\mathrm{Corollary\ \ref{lemma for main theorem two}}$. In the rest of the proof, 
we prove the existence of $s^{[\boldsymbol{d},\boldsymbol{e}]}$ satisfying \eqref{multivariable litfing prop desired interpolation pro}. 
\par If $k=1$, it is proved in $\mathrm{Lemma\ \ref{onevariable litfing prop}}$. From now on, we assume that $k\geq 2$. We replace $\theta^{(\boldsymbol{j})}$ with $\alpha_{M}^{(k)}(\theta^{(\boldsymbol{j})})$. We put
$$s^{[\boldsymbol{i}]}=\sum_{l_{k}=0}^{+\infty}X_{k}^{l_{k}}s^{[\boldsymbol{i}]}_{l_{k}}, \theta^{(\boldsymbol{j})}=\sum_{l_{k}=0}^{+\infty}X_{k}^{l_{k}}\theta^{(\boldsymbol{j})}_{l_{k}}$$
where $s^{[\boldsymbol{i}]}_{l_{k}},\theta^{(\boldsymbol{j})}_{l_{k}}\in M^{0}[[X_{1},\ldots, X_{k-1}]]\otimes_{\mathcal{O}_{\K}}\K$. Let $x\in [d_{k},e_{k}]$. 
{By $\mathrm{Lemma\ \ref{lemma for main prop 3}}$, we have
\begin{equation}\label{multivariable litfing prop first from lemmeq}
\displaystyle{\sum_{\boldsymbol{i}^{\prime}\in [\boldsymbol{d}^{\prime},{\boldsymbol{j}^{\prime}}]}}\left(\prod_{t=1}^{k-1}\begin{pmatrix}j^{\prime}_{t}-d_{t}\\ i^{\prime}_{t}-d_{t}\end{pmatrix}\right)(-1)^{\sum_{t=1}^{k-1}(j^{\prime}_{t}-i^{\prime}_{t})}s^{[(\boldsymbol{i}^{\prime},x)]}\in p^{\langle \boldsymbol{m}^{\prime},\boldsymbol{j}^{\prime}-\boldsymbol{d}^{\prime}\rangle_{k-1}}M^{0}[[X_{1},\ldots, X_{k}]]
\end{equation}
for each ${\boldsymbol{j}^{\prime}}\in [\boldsymbol{d}^{\prime},\boldsymbol{e}^{\prime}]$, and we have
\begin{multline}\label{multivariable litfing prop second fromdef}
\displaystyle{\sum_{\boldsymbol{i}^{\prime}\in [\boldsymbol{d}^{\prime},{\boldsymbol{j}^{\prime}}]}}\left(\prod_{t=1}^{k-1}\begin{pmatrix}j^{\prime}_{t}-d_{t}\\ i^{\prime}_{t}-d_{t}\end{pmatrix}\right)(-1)^{\sum_{t=1}^{k-1}(j^{\prime}_{t}-i^{\prime}_{t})}s^{[(\boldsymbol{i}^{\prime},x)]}\\
=\sum_{l_{k}=0}^{+\infty}X_{k}^{l_{k}}\left(\displaystyle{\sum_{\boldsymbol{i}^{\prime}\in [\boldsymbol{d}^{\prime},{\boldsymbol{j}^{\prime}}]}}\left(\prod_{t=1}^{k-1}\begin{pmatrix}j^{\prime}_{t}-d_{t}\\ i^{\prime}_{t}-d_{t}\end{pmatrix}\right)(-1)^{\sum_{t=1}^{k-1}(j^{\prime}_{t}-i^{\prime}_{t})}s^{[(\boldsymbol{i}^{\prime},x)]}_{l_{k}}\right).
\end{multline}
Let $l_{k}\in \mathbb{Z}_{\geq 0}$. By \eqref{multivariable litfing prop first from lemmeq} and \eqref{multivariable litfing prop second fromdef}, we have 
$$
\displaystyle{\sum_{\boldsymbol{i}^{\prime}\in [\boldsymbol{d}^{\prime},{\boldsymbol{j}^{\prime}}]}}\left(\prod_{t=1}^{k-1}\begin{pmatrix}j^{\prime}_{t}-d_{t}\\ i^{\prime}_{t}-d_{t}\end{pmatrix}\right)(-1)^{\sum_{t=1}^{k-1}(j^{\prime}_{t}-i^{\prime}_{t})}s^{[(\boldsymbol{i}^{\prime},x)]}_{l_{k}}\in p^{\langle \boldsymbol{m}^{\prime},\boldsymbol{j}^{\prime}-\boldsymbol{d}^{\prime}\rangle_{k-1}}M^{0}[[X_{1},\ldots, X_{k-1}]]
$$ 
for each $\boldsymbol{j}^{\prime}\in [\boldsymbol{d}^{\prime},\boldsymbol{e}^{\prime}]$. By the induction argument with respect to $k$, we can show that there exists an element $s^{[\boldsymbol{d}^{\prime},\boldsymbol{e}^{\prime}]}_{(x),l_{k}}\in p^{-c^{[\boldsymbol{d}^{\prime},\boldsymbol{e}^{\prime}]}}M^{0}[[X_{1},\ldots, X_{k-1}]]$ such that
\begin{equation}\label{multivariable litfing props dprime eprime (x),lkfist induc}
s^{[\boldsymbol{d}^{\prime},\boldsymbol{e}^{\prime}]}_{(x),l_{k}}(u_{1}^{i^{\prime}_{1}}\epsilon_{1}-1,\ldots, u_{k-1}^{i^{\prime}_{k-1}}\epsilon_{k-1}-1)=s^{[({\boldsymbol{i}^{\prime}},x)]}_{l_{k}}(u_{1}^{i^{\prime}_{1}}\epsilon_{1}-1,\ldots, u_{k-1}^{i^{\prime}_{k-1}}\epsilon_{k-1}-1)
\end{equation}
for each ${\boldsymbol{i}^{\prime}}\in [\boldsymbol{d}^{\prime},\boldsymbol{e}^{\prime}]$ and $(\epsilon_{1},\ldots, \epsilon_{k-1})\in \prod_{t=1}^{k-1}\mu_{p^{m_{t}}}$. By $\mathrm{Proposition\ \ref{multivariable weiestrass polynomial pro}}$, there exists a unique element $r^{[\boldsymbol{d}^{\prime},\boldsymbol{e}^{\prime}]}_{(x),l_{k}}\in p^{-c^{[\boldsymbol{d}^{\prime},\boldsymbol{e}^{\prime}]}}M^{0}[X_{1},\ldots, X_{k-1}]$ which satisfies $s^{[\boldsymbol{d}^{\prime},\boldsymbol{e}^{\prime}]}_{(x),l_{k}}\equiv r^{[\boldsymbol{d}^{\prime},\boldsymbol{e}^{\prime}]}_{(x),l_{k}}\ \mathrm{mod}\ \linebreak(\Omega_{\boldsymbol{m}^{\prime}}^{[\boldsymbol{d}^{\prime},\boldsymbol{e}^{\prime}]})$ and $\deg_{X_{t}}r^{[\boldsymbol{d}^{\prime},\boldsymbol{e}^{\prime}]}_{(x),l_{k}}<\deg \Omega_{m_{t}}^{[d_{t},e_{t}]}$ for each $1\leq i\leq k-1$. By replacing $s^{[\boldsymbol{d}^{\prime},\boldsymbol{e}^{\prime}]}_{(x),l_{k}}$ with $r^{[\boldsymbol{d}^{\prime},\boldsymbol{e}^{\prime}]}_{(x),l_{k}}$, we can assume that $s^{[\boldsymbol{d}^{\prime},\boldsymbol{e}^{\prime}]}_{(x),l_{k}}$ is in $p^{-c^{[\boldsymbol{d}^{\prime},\boldsymbol{e}^{\prime}]}}M^{0}[X_{1},\ldots, X_{k-1}]$ and $s^{[\boldsymbol{d}^{\prime},\boldsymbol{e}^{\prime}]}_{(x),l_{k}}$ satisfies
\begin{equation}\label{multivariable litfing props deg Xtsdprime erime <degOmega}
\deg_{X_{t}}s^{[\boldsymbol{d}^{\prime},\boldsymbol{e}^{\prime}]}_{(x),l_{k}}<\deg \Omega_{m_{t}}^{[d_{t},e_{t}]}
\end{equation}
for each $1\leq t\leq k-1$. Since we have
\begin{align*}
\theta^{(\boldsymbol{j})}&=\displaystyle{\sum_{\boldsymbol{i}\in [\boldsymbol{d},\boldsymbol{j}]}}\left(\prod_{t=1}^{k}\begin{pmatrix}j_{t}-d_{t}\\i_{t}-d_{t}\end{pmatrix}\right)(-1)^{\sum_{t=1}^{k}(j_{t}-i_{t})}s^{[\boldsymbol{i}]}\\
&=\displaystyle{\sum_{\boldsymbol{i}\in [\boldsymbol{d},\boldsymbol{j}]}}\left(\prod_{t=1}^{k}\begin{pmatrix}j_{t}-d_{t}\\i_{t}-d_{t}\end{pmatrix}\right)(-1)^{\sum_{t=1}^{k}(j_{t}-i_{t})}\sum_{l_{k}=0}^{+\infty}X_{k}^{l_{k}}s^{[\boldsymbol{i}]}_{l_{k}}\\
&=\sum_{l_{k}=0}^{+\infty}X_{k}^{l_{k}}\displaystyle{\sum_{\boldsymbol{i}\in [\boldsymbol{d},\boldsymbol{j}]}}\left(\prod_{t=1}^{k}\begin{pmatrix}j_{t}-d_{t}\\i_{t}-d_{t}\end{pmatrix}\right)(-1)^{\sum_{t=1}^{k}(j_{t}-i_{t})}s^{[\boldsymbol{i}]}_{l_{k}},
\end{align*}
we see that
$$\theta^{(\boldsymbol{j})}_{l_{k}}=\displaystyle{\sum_{\boldsymbol{i}\in [\boldsymbol{d},\boldsymbol{j}]}}\left(\prod_{t=1}^{k}\begin{pmatrix}j_{t}-d_{t}\\i_{t}-d_{t}\end{pmatrix}\right)(-1)^{\sum_{t=1}^{k}(j_{t}-i_{t})}s^{[\boldsymbol{i}]}_{l_{k}}$$
for every $\boldsymbol{j}\in [\boldsymbol{d},\boldsymbol{e}]$. Hence, we see that
\begin{equation}\label{multivariable litfing prop first from lemm thetaj,xlk=}
\theta^{(({\boldsymbol{j}^{\prime}},x))}_{l_{k}}=\sum_{\boldsymbol{i}^{\prime}\in [\boldsymbol{d}^{\prime},{\boldsymbol{j}^{\prime}}]}\left(\prod_{t=1}^{k-1}\begin{pmatrix}j^{\prime}_{t}-d_{t}\\i^{\prime}_{t}-d_{t}\end{pmatrix}\right)(-1)^{\sum_{t=1}^{k-1}(j^{\prime}_{t}-i^{\prime}_{t})}\sum_{i_{k}\in [d_{k},x]}\begin{pmatrix}x-d_{k}\\ i_{k}-d_{k}\end{pmatrix}(-1)^{x-i_{k}}s^{[(\boldsymbol{i}^{\prime},i_{k})]}_{l_{k}}.
\end{equation}
Since $\theta^{(({\boldsymbol{j}^{\prime}},x))}$ is in $p^{\langle \boldsymbol{m}^{\prime},\boldsymbol{j}^{\prime}-\boldsymbol{d}^{\prime}\rangle+m_{k}(x-d_{k})}M^{0}[[X_{1},\ldots, X_{k}]]$, we have
\begin{equation}\label{multivariable litfing prop first from lemm thetaj,xlkin m0kakk0}
\theta^{(({\boldsymbol{j}^{\prime}},x))}_{l_{k}}\in p^{\langle \boldsymbol{m}^{\prime},\boldsymbol{j}^{\prime}-\boldsymbol{d}^{\prime}\rangle+m_{k}(x-d_{k})}M^{0}[[X_{1},\ldots, X_{k-1}]]
\end{equation}
for every ${\boldsymbol{j}^{\prime}}\in [\boldsymbol{d}^{\prime},\boldsymbol{e}^{\prime}]$. Put $b_{(x),l_{k}}^{[\boldsymbol{i}^{\prime}]}=\sum_{i_{k}\in [d_{k},x]}\begin{pmatrix}x-d_{k}\\ i_{k}-d_{k}\end{pmatrix}(-1)^{x-i_{k}}s^{[(\boldsymbol{i}^{\prime},i_{k})]}_{l_{k}}\in \linebreak M^{0}[[X_{1},\ldots, X_{k-1}]]\otimes_{\mathcal{O}_{\K}}\K$ for each $\boldsymbol{i}^{\prime}\in [\boldsymbol{d}^{\prime},\boldsymbol{e}^{\prime}]$. By \eqref{multivariable litfing prop first from lemm thetaj,xlk=} and \eqref{multivariable litfing prop first from lemm thetaj,xlkin m0kakk0}, we have  
$$\sum_{\boldsymbol{i}^{\prime}\in [\boldsymbol{d}^{\prime},{\boldsymbol{j}^{\prime}}]}\left(\prod_{t=1}^{k-1}\begin{pmatrix}j^{\prime}_{t}-d_{t}\\i^{\prime}_{t}-d_{t}\end{pmatrix}\right)(-1)^{\sum_{t=1}^{k-1}(j^{\prime}_{t}-i^{\prime}_{t})}b_{(x),l_{k}}^{[\boldsymbol{i}^{\prime}]}\in p^{\langle \boldsymbol{m}^{\prime},\boldsymbol{j}^{\prime}-\boldsymbol{d}^{\prime}\rangle+m_{k}(x-d_{k})}M^{0}[[X_{1},\ldots, X_{k-1}]]$$
for every $\boldsymbol{j}\in [\boldsymbol{d}^{\prime},\boldsymbol{e}^{\prime}]$. Therefore, we can apply the induction argument on $k$ to $b_{(x),l_{k}}^{[\boldsymbol{i}^{\prime}]}$ for each $\boldsymbol{i}^{\prime}\in [\boldsymbol{d}^{\prime},\boldsymbol{e}^{\prime}]$ and we see that there exists a power series $t_{(x),l_{k}}\in p^{m_{k}(x-d_{k})-c^{[\boldsymbol{d}^{\prime},\boldsymbol{e}^{\prime}]}}\linebreak M^{0}[[X_{1},\ldots, X_{k-1}]]$ such that 
\begin{align}\label{multivariable litfing props deft(x),lk}
\begin{split}
t_{(x),l_{k}}(u_{1}^{i^{\prime}_{1}}\epsilon_{1}-1,\ldots, u_{k-1}^{i^{\prime}_{k-1}}\epsilon_{k-1}-1)&=b_{(x),l_{k}}^{[\boldsymbol{i}^{\prime}]}(u_{1}^{i^{\prime}_{1}}\epsilon_{1}-1,\ldots, u_{k-1}^{i^{\prime}_{k-1}}\epsilon_{k-1}-1)\\
&=\sum_{i_{k}\in [d_{k},x]}(-1)^{x-i_{k}}s^{[({\boldsymbol{i}^{\prime}},i_{k})]}_{l_{k}}(u_{1}^{i^{\prime}_{1}}\epsilon_{1}-1,\ldots, u_{k-1}^{i^{\prime}_{k-1}}\epsilon_{k-1}-1)
\end{split}
\end{align}
for every ${\boldsymbol{i}^{\prime}}\in [\boldsymbol{d}^{\prime},\boldsymbol{e}^{\prime}]$ and for every $(\epsilon_{1},\ldots, \epsilon_{k-1})\in \prod_{t=1}^{k-1}\mu_{p^{m_{t}}}$. By \eqref{multivariable litfing props dprime eprime (x),lkfist induc} and \eqref{multivariable litfing props deft(x),lk}, we have
\begin{equation}
t_{(x),l_{k}}(u_{1}^{i^{\prime}_{1}}\epsilon_{1}-1,\ldots, u_{k-1}^{i^{\prime}_{k-1}}\epsilon_{k-1}-1)=\sum_{i_{k}\in [d_{k},x]}(-1)^{x-i_{k}}s^{[\boldsymbol{d}^{\prime},\boldsymbol{e}^{\prime}]}_{(i_{k}),l_{k}}(u_{1}^{i^{\prime}_{1}}\epsilon_{1}-1,\ldots, u_{k-1}^{i^{\prime}_{k-1}}\epsilon_{k-1}-1)
\end{equation}
for every ${\boldsymbol{i}^{\prime}}\in [\boldsymbol{d}^{\prime},\boldsymbol{e}^{\prime}]$ and for every $(\epsilon_{1},\ldots, \epsilon_{k-1})\in \prod_{t=1}^{k-1}\mu_{p^{m_{t}}}$.
By Corollary \ref{lemma for main theorem two}, we see that 
\begin{equation}\label{multivariable litfing props t(x),l,-summoomega}
t_{(x),l_{k}}- \sum_{i_{k}\in [d_{k},x]}(-1)^{x-i_{k}}s^{[\boldsymbol{d}^{\prime},\boldsymbol{e}^{\prime}]}_{(i_{k}),l_{k}}\in (\Omega_{\boldsymbol{m}^{\prime}}^{[\boldsymbol{d}^{\prime},\boldsymbol{e}^{\prime}]})M^{0}[[X_{1},\ldots, X_{k-1}]]\otimes_{\mathcal{O}_{\K}}\K.
\end{equation}
By \eqref{multivariable litfing props deg Xtsdprime erime <degOmega}, we have 
\begin{equation}\label{multivariable litfing props degxtsumik(-1)sdprime}
\deg_{X_{t}} \sum_{i_{k}\in [d_{k},x]}(-1)^{x-i_{k}}s^{[\boldsymbol{d}^{\prime},\boldsymbol{e}^{\prime}]}_{(i_{k}),l_{k}}<\deg\Omega_{m_{t}}^{[d_{t},e_{t}]}\end{equation}
 for every $1\leq t\leq k-1$. We note that $t_{(x),l_{k}}$ is an element of $B_{\boldsymbol{0}_{k-1}}(M)$ and $\sum_{i_{k}\in [d_{k},x]}(-1)^{x-i_{k}}\linebreak s^{[\boldsymbol{d}^{\prime},\boldsymbol{e}^{\prime}]}_{(i_{k}),l_{k}}$ is a unique element of $M[X_{1},\ldots, X_{k-1}]$ which satisfies \eqref{multivariable litfing props t(x),l,-summoomega} and \eqref{multivariable litfing props degxtsumik(-1)sdprime}. By \eqref{multivariable weiestrass polynomial pro vboldsymbolrf=pro} in $\mathrm{Proposition\ \ref{multivariable weiestrass polynomial pro}}$, we see that
 $$v_{\boldsymbol{0}_{k-1}}\left(\sum_{i_{k}\in [d_{k},x]}\begin{pmatrix}x-d_{k}\\i_{k}-d_{k}\end{pmatrix}
(-1)^{x-i_{k}}s^{[\boldsymbol{d}^{\prime},\boldsymbol{e}^{\prime}]}_{(i_{k}),l_{k}}\right)\geq v_{\boldsymbol{0}_{k-1}}(t_{(x),l_{k}}).$$
Since  $t_{(x),l_{k}}\in p^{m_{k}(x-d_{k})-c^{[\boldsymbol{d}^{\prime},\boldsymbol{e}^{\prime}]}}M^{0}[[X_{1},\ldots, X_{k-1}]]$, we have
\begin{equation}\label{multivariable litfing prop equation ineeq}
v_{\boldsymbol{0}_{k-1}}\left(\sum_{i_{k}\in [d_{k},x]}\begin{pmatrix}x-d_{k}\\i_{k}-d_{k}\end{pmatrix}
(-1)^{x-i_{k}}s^{[\boldsymbol{d}^{\prime},\boldsymbol{e}^{\prime}]}_{(i_{k}),l_{k}}\right)\geq m_{k}(x-d_{k})-c^{[\boldsymbol{d}^{\prime},\boldsymbol{e}^{\prime}]}.
\end{equation}
We define $s_{i_{k}} \in p^{-c^{[\boldsymbol{d}^{\prime},\boldsymbol{e}^{\prime}]}}B_{\boldsymbol{0}_{k-1}}(M)^{0}[[X_{k}]]$ to be 
$s_{i_{k}}=(s^{[\boldsymbol{d}^{\prime},\boldsymbol{e}^{\prime}]}_{(i_{k}),l_{k}})_{l_{k}\in \mathbb{Z}_{\geq 0}}$ with $d_{k}\leq i_{k}\leq e_{k}$. 
By \eqref{multivariable litfing prop equation ineeq}, $s_{i_{k}}$ satisfies 
$$\sum_{i_{k}\in [d_{k},j_{k}]}\begin{pmatrix}j_{k}-d_{k}\\
i_{k}-d_{k}
\end{pmatrix}(-1)^{j_{k}-i_{k}}s_{i_{k}}\in p^{m_{k}(j_{k}-d_{k})-c^{[\boldsymbol{d}^{\prime},\boldsymbol{e}^{\prime}]}}B_{\boldsymbol{0}_{k-1}}(M)^{0}[[X_{k}]]$$
with $d_{k}\leq j_{k}\leq e_{k}$.  
By the result of the case $k=1$, there exists an element $r\in p^{-c^{[\boldsymbol{d},\boldsymbol{e}]}} B_{\boldsymbol{0}_{k-1}}(M)^{0}[[X_{k}]]$ which satisfies 
\begin{equation}\label{multivariable litfing prop equationdef of rstisXk=uk}
r\vert_{X_{k}=u_{k}^{i_{k}}\epsilon_{k}-1}=s_{i_{k}}\vert_{X_{k}=u_{k}^{i_{k}}\epsilon_{k}-1}
\end{equation}
for every $i_{k}\in [d_{k},e_{k}]$ and for every $\epsilon_{k}\in \mu_{p^{m_{k}}}$. Via the isometry $B_{\boldsymbol{0}_{k-1}}(M)^{0}[[X_{k}]]\simeq M^{0}[[X_{1},\ldots, X_{k}]]$ of Proposition \ref{isometry ofHh for induction}, we regard $r$ as an element of $p^{-c^{[\boldsymbol{d},\boldsymbol{e}]}}M^{0}[[X_{1},\ldots, X_{k}]]$. By \eqref{multivariable litfing prop equationdef of rstisXk=uk}, we have
\begin{align*}
r(u_{1}^{i_{1}}\epsilon_{1}-1,\ldots,  u_{k}^{i_{k}}\epsilon_{k}-1)&=s_{i_{k}}(u_{1}^{i_{1}}\epsilon_{1}-1,\ldots, u_{k}^{i_{k}}\epsilon_{k}-1)\\
&=\sum_{l_{k}=0}^{+\infty}s^{[\boldsymbol{d}^{\prime},\boldsymbol{e}^{\prime}]}_{(i_{k}),l_{k}}(u_{1}^{i_{1}}\epsilon_{1}-1,\ldots,  u_{k-1}^{i_{k-1}}\epsilon_{k-1}-1)(u_{k}^{i_{k}}\epsilon_{k}-1)^{l_{k}}
\end{align*}
for every $\boldsymbol{i}\in [\boldsymbol{d},\boldsymbol{e}]$ and for every $(\epsilon_{1},\ldots, \epsilon_{k})\in \prod_{t=1}^{k}\mu_{p^{m_{t}}}$. By \eqref{multivariable litfing props dprime eprime (x),lkfist induc}, we have $s^{[\boldsymbol{d}^{\prime},\boldsymbol{e}^{\prime}]}_{(i_{k}),l_{k}}(u_{1}^{i_{1}}\epsilon_{1}-1,\ldots,u_{k-1}^{i_{k-1}}\epsilon_{k-1}-1)=s_{l_{k}}^{[\boldsymbol{i}]}(u_{1}^{i_{1}}\epsilon_{1}-1,\ldots,u_{k-1}^{i_{k-1}}\epsilon_{k-1}-1)$. Therefore, we have
\begin{align*}
r(u_{1}^{i_{1}}\epsilon_{1}-1,\ldots,  u_{k}^{i_{k}}\epsilon_{k}-1)&=\sum_{l_{k}=0}^{+\infty}s_{l_{k}}^{[\boldsymbol{i}]}(u_{1}^{i_{1}}\epsilon_{1}-1,\ldots,u_{k-1}^{i_{k-1}}\epsilon_{k-1}-1)(u_{k}^{i_{k}}\epsilon_{k}-1)^{l_{k}}\\
&=s^{[\boldsymbol{i}]}(u_{1}^{i_{1}}\epsilon_{1}-1,\ldots,  u_{k}^{i_{k}}\epsilon_{k}-1)
\end{align*}
for every $\boldsymbol{i}\in [\boldsymbol{d},\boldsymbol{e}]$ and for every $(\epsilon_{1},\ldots, \epsilon_{k})\in \prod_{t=1}^{k}\mu_{p^{m_{t}}}$. Thus, 
$s^{[\boldsymbol{d},\boldsymbol{e}]}=[r]_{\boldsymbol{m}}$ satisfies \eqref{multivariable litfing prop desired interpolation pro} 
for every $\boldsymbol{i}\in [\boldsymbol{d},\boldsymbol{e}]$ and for every $\boldsymbol{\epsilon}\in \prod_{i=1}^{k}\mu_{p^{m_{i}}}$.}
\end{proof}

Let $\mathcal{D}^{[\boldsymbol{d},\boldsymbol{e}]}_{\boldsymbol{h}} (\Gamma, M)$ be the space of $[\boldsymbol{d},\boldsymbol{e}]$-admissible distributions of growth $\boldsymbol{h}$ and $I_{\boldsymbol{h}}^{[\boldsymbol{d},\boldsymbol{e}]}(M)$  the module defined in \S\ref{preparation}. Put
\begin{align}
\begin{split}
I_{\boldsymbol{h}}^{[\boldsymbol{d},\boldsymbol{e}]}(M)^{0}&=\bigg\{(s_{\boldsymbol{m}})_{\boldsymbol{m}\in \mathbb{Z}_{\geq 0}^{k}}\in I_{\boldsymbol{h}}^{[\boldsymbol{d},\boldsymbol{e}]}(M)\bigg\vert \\
&(p^{\langle \boldsymbol{h},\boldsymbol{m}\rangle_{k}}s_{\boldsymbol{m}})_{\boldsymbol{m}\in \mathbb{Z}_{\geq 0}^{k}}\in \prod_{\boldsymbol{m}\in \mathbb{Z}_{\geq 0}^{k}}M^{0}[[\Gamma]]\slash (\Omega_{\boldsymbol{m}}^{[\boldsymbol{d},\boldsymbol{e}]}(\gamma_{1},\ldots, \gamma_{k}))M^{0}[[\Gamma]]\bigg\}.
\end{split}
\end{align}
By Lemma \ref{multivariable litfing prop}, we have the following:
\begin{pro}\label{multivariable iwasawa I(ii) sufficient}
Let $s^{[\boldsymbol{i}]}=(s^{[\boldsymbol{i}]}_{\boldsymbol{m}})_{\boldsymbol{m}\in \mathbb{Z}_{\geq 0}^{k}}\in I_{\boldsymbol{h}}^{[\boldsymbol{i}]}(M)$ and $\tilde{s}_{\boldsymbol{m}}^{[\boldsymbol{i}]}$ a lift of $s_{\boldsymbol{m}}^{[\boldsymbol{i}]}$ for each $\boldsymbol{m}\in \mathbb{Z}_{\geq 0}^{k}$ and $\boldsymbol{i}\in [\boldsymbol{d},\boldsymbol{e}]$. If there exists a non-negative integer $n$ which satisfies
$$p^{\langle \boldsymbol{m},\boldsymbol{h}-(\boldsymbol{j}-\boldsymbol{d})\rangle_{k}}\displaystyle{\sum_{\boldsymbol{i}\in [\boldsymbol{d},\boldsymbol{j}]}}\left(\prod_{t=1}^{k}\begin{pmatrix}j_{t}-d_{t}\\i_{t}-d_{t}\end{pmatrix}\right)(-1)^{\sum_{t=1}^{k}(j_{t}-i_{t})}\tilde{s}_{\boldsymbol{m}}^{[\boldsymbol{i}]}\in M^{0}[[\Gamma]]\otimes_{\mathcal{O}_{\K}}p^{-n}\mathcal{O}_{\K}$$
for every $\boldsymbol{m}\in \mathbb{Z}_{\geq 0}^{k}$ and $\boldsymbol{j}\in [\boldsymbol{d},\boldsymbol{e}]$, we have a unique element $s^{[\boldsymbol{d},\boldsymbol{e}]}\in I_{\boldsymbol{h}}^{[\boldsymbol{d},\boldsymbol{e}]}(M)^{0}\otimes_{\mathcal{O}_{\K}}p^{-c^{[\boldsymbol{d},\boldsymbol{e}]}-n}\mathcal{O}_{\K}$ such that the image of $s^{[\boldsymbol{d},\boldsymbol{e}]}$ by the natural projection $I_{\boldsymbol{h}}^{[\boldsymbol{d},\boldsymbol{e}]}(M)\rightarrow I_{\boldsymbol{h}}^{[\boldsymbol{i}]}(M)$ is $s^{[\boldsymbol{i}]}$ for each $\boldsymbol{i}\in [\boldsymbol{d},\boldsymbol{e}]$, where $c^{[\boldsymbol{d},\boldsymbol{e}]}$ is the constant defined in Lemma \ref{multivariable litfing prop}.
\end{pro}
\begin{proof}
{
For each $\boldsymbol{m}\in \mathbb{Z}_{\geq 0}^{k}$, there exists a unique element $s_{\boldsymbol{m}}^{[\boldsymbol{d},\boldsymbol{e}]}\in \frac{M^{0}[[\Gamma]]}{(\Omega_{\boldsymbol{m}}^{[\boldsymbol{d},\boldsymbol{e}]}(\gamma_{1},\ldots,\gamma_{k}))M^{0}[[\Gamma]]}\otimes_{\mathcal{O}_{\K}}p^{-\langle \boldsymbol{h},\boldsymbol{m}\rangle_{k}-c^{[\boldsymbol{d},\boldsymbol{e}]}-n}\mathcal{O}_{\K}$ such that the image of $s_{\boldsymbol{m}}^{[\boldsymbol{d},\boldsymbol{e}]}$ by the natural projection \linebreak$\frac{M^{0}[[\Gamma]]}{(\Omega_{\boldsymbol{m}}^{[\boldsymbol{d},\boldsymbol{e}]}(\gamma_{1},\ldots,\gamma_{k}))M^{0}[[\Gamma]]}\otimes_{\mathcal{O}_{\K}}\K\rightarrow \frac{M^{0}[[\Gamma]]}{(\Omega_{\boldsymbol{m}}^{[\boldsymbol{i}]}(\gamma_{1},\ldots,\gamma_{k}))M^{0}[[\Gamma]]}\otimes_{\mathcal{O}_{\K}}\K$ is $s_{\boldsymbol{m}}^{[\boldsymbol{i}]}$ for every $\boldsymbol{i}\in [\boldsymbol{d},\boldsymbol{e}]$ by Lemma \ref{multivariable litfing prop}. 
Since this construction is compatible with the projective systems of $s_{\boldsymbol{m}}^{[\boldsymbol{d},\boldsymbol{e}]}$ and $s_{\boldsymbol{m}}^{[\boldsymbol{i}]}$ with respect to $\boldsymbol{m}$, $s^{[\boldsymbol{d},\boldsymbol{e}]}=(s_{\boldsymbol{m}}^{[\boldsymbol{d},\boldsymbol{e}]})_{\boldsymbol{m}\in \mathbb{Z}_{\geq 0}^{k}}\in I_{\boldsymbol{h}}^{[\boldsymbol{d},\boldsymbol{e}]}(M)^{0}\otimes_{\mathcal{O}_{\K}}p^{-c^{[\boldsymbol{d},\boldsymbol{e}]}-n}\mathcal{O}_{\K}$ such that the image of $s^{[\boldsymbol{d},\boldsymbol{e}]}$ by the natural projection $I_{\boldsymbol{h}}^{[\boldsymbol{d},\boldsymbol{e}]}(M)\rightarrow I_{\boldsymbol{h}}^{[\boldsymbol{i}]}(M)$ is $s^{[\boldsymbol{i}]}$ for every $\boldsymbol{i}\in [\boldsymbol{d},\boldsymbol{e}]$.}
\end{proof}
\begin{thm}\label{multi-variable results on admissible distributions}
We have a unique $\mathcal{O}_{\K}[[\Gamma]]\otimes_{\mathcal{O}_{\K}}\K$-module isomorphism 
\begin{equation}\label{equation:multi-variable results on admissible distributions}
\Psi:I_{\boldsymbol{h}}^{[\boldsymbol{d},\boldsymbol{e}]}(M) \stackrel{\sim}{\rightarrow} \mathcal{D}^{[\boldsymbol{d},\boldsymbol{e}]}_{\boldsymbol{h}} (\Gamma, M)
\end{equation}
such that the image $\mu_{s^{[\boldsymbol{d},\boldsymbol{e}]}}\in \mathcal{D}^{[\boldsymbol{d},\boldsymbol{e}]}_{\boldsymbol{h}} (\Gamma, M)$ of each element $s^{[\boldsymbol{d},\boldsymbol{e}]}=(s_{\boldsymbol{m}}^{[\boldsymbol{d},\boldsymbol{e}]})_{\boldsymbol{m}\in \mathbb{Z}_{\geq 0}^{k}} \in I_{\boldsymbol{h}}^{[\boldsymbol{d},\boldsymbol{e}]}(M)$ is characterized by the interpolation property  
\begin{equation}\label{admissible interpolation formula}
\kappa(\tilde{s}^{[\boldsymbol{d},\boldsymbol{e}]}_{\boldsymbol{m}_{\kappa}})= \int_{\Gamma} \prod_{j=1}^{k}(\chi_{j}^{w_{\kappa,j}}\phi_{\kappa,j})(x_{j})d\mu_{s^{[\boldsymbol{d},\boldsymbol{e}]}} \end{equation}
 for each $\kappa\in \mathfrak{X}_{\mathcal{O}_{\K}[[\Gamma]]}^{[\boldsymbol{d},\boldsymbol{e}]}$, where $\tilde{s}^{[\boldsymbol{d},\boldsymbol{e}]}_{\boldsymbol{m}_{\kappa}}$ is a lift of $s_{\boldsymbol{m}_{\kappa}}^{[\boldsymbol{d},\boldsymbol{e}]}$. In addition, if we regard $ I_{\boldsymbol{h}}^{[\boldsymbol{d},\boldsymbol{e}]}(M)^{0}$ 
 as a submodule of $\mathcal{D}^{[\boldsymbol{d},\boldsymbol{e}]}_{\boldsymbol{h}} (\Gamma, M)$ via the isomorphism \eqref{equation:multi-variable results on admissible distributions}, we have 
\begin{equation*}
\{\mu\in \mathcal{D}^{[\boldsymbol{d},\boldsymbol{e}]}_{\boldsymbol{h}} (\Gamma ,M)\ \vert \ v_{\boldsymbol{h}}^{[\boldsymbol{d},\boldsymbol{e}]}(\mu)\geq c^{[\boldsymbol{d},\boldsymbol{e}]}\}\subset I_{\boldsymbol{h}}^{[\boldsymbol{d},\boldsymbol{e}]}(M)^{0}
\subset \{\mu\in \mathcal{D}^{[\boldsymbol{d},\boldsymbol{e}]}_{\boldsymbol{h}} (\Gamma ,M)\ \vert \ v_{\boldsymbol{h}}^{[\boldsymbol{d},\boldsymbol{e}]}(\mu)\geq 0\},
\end{equation*}
where $c^{[\boldsymbol{d},\boldsymbol{e}]}=\sum_{i=1}^{k}c^{[d_{i},e_{i}]}$ is the constant defined in \eqref{constant for the admissible}. 
\end{thm}
\begin{proof}
We prove this theorem by induction on $k$. When $k=1$, the desired statement is already proved in Proposition \ref{onevariable isom Ih from Dh for banach}. Let us assume that $k\geq 2$. By the induction argument with respect to $k$, we have $I_{\boldsymbol{h}^{\prime}}^{[\boldsymbol{d}^{\prime},\boldsymbol{e}^{\prime}]}(M)\simeq \mathcal{D}^{[\boldsymbol{d}^{\prime},\boldsymbol{e}^{\prime}]}_{\boldsymbol{h}^{\prime}} (\Gamma^{\prime}, M)$ and
\begin{equation}\label{admissible interpolation formulaeq0}
\{\mu\in\mathcal{D}^{[\boldsymbol{d}^{\prime},\boldsymbol{e}^{\prime}]}_{\boldsymbol{h}^{\prime}} (\Gamma^{\prime}, M)\vert v_{\boldsymbol{h}}^{[\boldsymbol{d}^{\prime},\boldsymbol{e}^{\prime}]}(\mu)\geq c^{[\boldsymbol{d}^{\prime},\boldsymbol{e}^{\prime}]}\}\subset I_{\boldsymbol{h}^{\prime}}^{[\boldsymbol{d}^{\prime},\boldsymbol{e}^{\prime}]}(M)^{0}\subset \{\mu\in\mathcal{D}^{[\boldsymbol{d}^{\prime},\boldsymbol{e}^{\prime}]}_{\boldsymbol{h}^{\prime}} (\Gamma^{\prime}, M)\vert v_{\boldsymbol{h}^{\prime}}^{[\boldsymbol{d}^{\prime},\boldsymbol{e}^{\prime}]}(\mu)\geq 0\}.
\end{equation}
By \eqref{admissible interpolation formulaeq0}, we can show that we have $I_{h_{k}}^{[d_{k},e_{k}]}(I_{\boldsymbol{h}^{\prime}}^{[\boldsymbol{d}^{\prime},\boldsymbol{e}^{\prime}]}(M))\simeq I_{h_{k}}^{[d_{k},e_{k}]}(\mathcal{D}^{[\boldsymbol{d}^{\prime},\boldsymbol{e}^{\prime}]}_{\boldsymbol{h}^{\prime}} (\Gamma^{\prime}, M))$ and
 \begin{equation}\label{admissible interpolation formulaeq1}
 p^{c^{[\boldsymbol{d}^{\prime},\boldsymbol{e}^{\prime}]}}I_{h_{k}}^{[d_{k},e_{k}]}(\mathcal{D}^{[\boldsymbol{d}^{\prime},\boldsymbol{e}^{\prime}]}_{\boldsymbol{h}^{\prime}} (\Gamma^{\prime}, M))^{0}\subset I_{h_{k}}^{[d_{k},e_{k}]}(I_{\boldsymbol{h}^{\prime}}^{[\boldsymbol{d}^{\prime},\boldsymbol{e}^{\prime}]}(M))^{0}\subset I_{h_{k}}^{[d_{k},e_{k}]}(\mathcal{D}^{[\boldsymbol{d}^{\prime},\boldsymbol{e}^{\prime}]}_{\boldsymbol{h}^{\prime}} (\Gamma^{\prime}, M))^{0}.
 \end{equation}
 On the other hand, by the result in the case $k=1$, we see that $I_{h_{k}}^{[d_{k},e_{k}]}(\mathcal{D}^{[\boldsymbol{d}^{\prime},\boldsymbol{e}^{\prime}]}_{\boldsymbol{h}^{\prime}} (\Gamma^{\prime}, M))\simeq \mathcal{D}_{h_{k}}^{[d_{k},e_{k}]}(\Gamma_{k},\mathcal{D}^{[\boldsymbol{d}^{\prime},\boldsymbol{e}^{\prime}]}_{\boldsymbol{h}^{\prime}} (\Gamma^{\prime}, M))$ and
\begin{multline}\label{admissible interpolation formulaeq2}
\left\{\mu\in\mathcal{D}_{h_{k}}^{[d_{k},e_{k}]}(\Gamma_{k},\mathcal{D}^{[\boldsymbol{d}^{\prime},\boldsymbol{e}^{\prime}]}_{\boldsymbol{h}^{\prime}} (\Gamma^{\prime}, M))\big\vert v_{\mathcal{D}_{h_{k}}^{[d_{k},e_{k}]}(\Gamma_{k},\mathcal{D}^{[\boldsymbol{d}^{\prime},\boldsymbol{e}^{\prime}]}_{\boldsymbol{h}^{\prime}} (\Gamma^{\prime}, M))}(\mu)\geq c^{[d_{k},e_{k}]}\right\}\\
\subset I_{h_{k}}^{[d_{k},e_{k}]}(\mathcal{D}^{[\boldsymbol{d}^{\prime},\boldsymbol{e}^{\prime}]}_{\boldsymbol{h}^{\prime}} (\Gamma^{\prime}, M))^{0}\\
\subset \left\{\mu\in\mathcal{D}_{h_{k}}^{[d_{k},e_{k}]}(\Gamma_{k},\mathcal{D}^{[\boldsymbol{d}^{\prime},\boldsymbol{e}^{\prime}]}_{\boldsymbol{h}^{\prime}} (\Gamma^{\prime}, M))\big\vert v_{\mathcal{D}_{h_{k}}^{[d_{k},e_{k}]}(\Gamma_{k},\mathcal{D}^{[\boldsymbol{d}^{\prime},\boldsymbol{e}^{\prime}]}_{\boldsymbol{h}^{\prime}} (\Gamma^{\prime}, M))}(\mu)\geq 0\right\}
\end{multline}
where $v_{\mathcal{D}_{h_{k}}^{[d_{k},e_{k}]}(\Gamma_{k},\mathcal{D}^{[\boldsymbol{d}^{\prime},\boldsymbol{e}^{\prime}]}_{\boldsymbol{h}^{\prime}} (\Gamma^{\prime}, M))}$ is the valuation on $\mathcal{D}_{h_{k}}^{[d_{k},e_{k}]}(\Gamma_{k},\mathcal{D}^{[\boldsymbol{d}^{\prime},\boldsymbol{e}^{\prime}]}_{\boldsymbol{h}^{\prime}} (\Gamma^{\prime}, M))$. Therefore, by \eqref{admissible interpolation formulaeq1} and \eqref{admissible interpolation formulaeq2}, we have $I_{h_{k}}^{[d_{k},e_{k}]}(I_{\boldsymbol{h}^{\prime}}^{[\boldsymbol{d}^{\prime},\boldsymbol{e}^{\prime}]}(M))\simeq \mathcal{D}_{h_{k}}^{[d_{k},e_{k}]}(\Gamma_{k},\mathcal{D}^{[\boldsymbol{d}^{\prime},\boldsymbol{e}^{\prime}]}_{\boldsymbol{h}^{\prime}} (\Gamma^{\prime}, M))$ and
\begin{multline}\label{admissible interpolation formulaeq3}
\{\mu\in\mathcal{D}_{h_{k}}^{[d_{k},e_{k}]}(\Gamma_{k},\mathcal{D}^{[\boldsymbol{d}^{\prime},\boldsymbol{e}^{\prime}]}_{\boldsymbol{h}^{\prime}} (\Gamma^{\prime}, M))\vert v_{\mathcal{D}_{h_{k}}^{[d_{k},e_{k}]}(\Gamma_{k},\mathcal{D}^{[\boldsymbol{d}^{\prime},\boldsymbol{e}^{\prime}]}_{\boldsymbol{h}^{\prime}} (\Gamma^{\prime}, M))}(\mu)\geq c^{[\boldsymbol{d},\boldsymbol{e}]}\}\subset I_{h_{k}}^{[d_{k},e_{k}]}(I_{\boldsymbol{h}^{\prime}}^{[\boldsymbol{d}^{\prime},\boldsymbol{e}^{\prime}]}(M))^{0}\\
\subset \{\mu\in \mathcal{D}_{h_{k}}^{[d_{k},e_{k}]}(\Gamma_{k},\mathcal{D}^{[\boldsymbol{d}^{\prime},\boldsymbol{e}^{\prime}]}_{\boldsymbol{h}^{\prime}} (\Gamma^{\prime}, M))\vert v_{\mathcal{D}_{h_{k}}^{[d_{k},e_{k}]}(\Gamma_{k},\mathcal{D}^{[\boldsymbol{d}^{\prime},\boldsymbol{e}^{\prime}]}_{\boldsymbol{h}^{\prime}} (\Gamma^{\prime}, M))}(\mu)\geq0\}.
\end{multline}
By Proposition \ref{for induction admisible admissible}, we have an isometric isomorphism $\mathcal{D}_{h_{k}}^{[d_{k},e_{k}]}(\Gamma_{k},\mathcal{D}^{[\boldsymbol{d}^{\prime},\boldsymbol{e}^{\prime}]}_{\boldsymbol{h}^{\prime}} (\Gamma^{\prime}, M))\simeq \mathcal{D}_{\boldsymbol{h}}^{[\boldsymbol{d},\boldsymbol{e}]}(\Gamma,\linebreak M)$. Further, by Proposition \ref{multi J induction pro}, we have an $\mathcal{O}_{\K}[[X_{1},\ldots, X_{k}]]\otimes_{\mathcal{O}_{\K}}\K$-module isomorphism 
\begin{equation}\label{multadifor jinduction eq}
J_{\boldsymbol{h}}^{[\boldsymbol{d},\boldsymbol{e}]}(M)\simeq J_{h_{k}}^{[d_{k},e_{k}]}(J_{\boldsymbol{h}^{\prime}}^{[\boldsymbol{d}^{\prime},\boldsymbol{e}^{\prime}]}(M))
\end{equation}
which is induced by an $\mathcal{O}_{\K}[[X_{1},\ldots, X_{k}]]$-module isomorphism $J_{\boldsymbol{h}}^{[\boldsymbol{d},\boldsymbol{e}]}(M)^{0}\simeq J_{h_{k}}^{[d_{k},e_{k}]}(\linebreak J_{\boldsymbol{h}^{\prime}}^{[\boldsymbol{d}^{\prime},\boldsymbol{e}^{\prime}]}(M))^{0}$. By \eqref{noncanonical between Ioldysmbolhde and J}, we have non-canonical isomorphisms $I_{\boldsymbol{h}}^{[\boldsymbol{d},\boldsymbol{e}]}(M)^{0}\simeq J_{\boldsymbol{h}}^{[\boldsymbol{d},\boldsymbol{e}]}(M)^{0}$ and $I_{h_{k}}^{[d_{k},e_{k}]}(I_{\boldsymbol{h}^{\prime}}^{[\boldsymbol{d}^{\prime},\boldsymbol{e}^{\prime}]}(M))^{0}\simeq J_{h_{k}}^{[d_{k},e_{k}]}(J_{\boldsymbol{h}^{\prime}}^{[\boldsymbol{d}^{\prime},\boldsymbol{e}^{\prime}]}(M))^{0}$ which depend on topological generators on $\Gamma_{i}$ for each $1\leq i\leq k$. Then, we have isomorphisms
\begin{equation}\label{multadifor jinduction ene9151}
I_{\boldsymbol{h}}^{[\boldsymbol{d},\boldsymbol{e}]}(M)\simeq J_{\boldsymbol{h}}^{[\boldsymbol{d},\boldsymbol{e}]}(M)\simeq J_{h_{k}}^{[d_{k},e_{k}]}(J_{\boldsymbol{h}^{\prime}}^{[\boldsymbol{d}^{\prime},\boldsymbol{e}^{\prime}]}(M))\simeq I_{h_{k}}^{[d_{k},e_{k}]}(I_{\boldsymbol{h}^{\prime}}^{[\boldsymbol{d}^{\prime},\boldsymbol{e}^{\prime}]}(M))
\end{equation}
and $I_{\boldsymbol{h}}^{[\boldsymbol{d},\boldsymbol{e}]}(M)^{0}\simeq I_{h_{k}}^{[d_{k},e_{k}]}(I_{\boldsymbol{h}^{\prime}}^{[\boldsymbol{d}^{\prime},\boldsymbol{e}^{\prime}]}(M))^{0}$. Therefore, by \eqref{admissible interpolation formulaeq3} and \eqref{multadifor jinduction ene9151}, we have $I_{\boldsymbol{h}}^{[\boldsymbol{d},\boldsymbol{e}]}(M)\simeq \mathcal{D}_{\boldsymbol{h}}^{[\boldsymbol{d},\boldsymbol{e}]}(\Gamma,M)$ and
$$
\{\mu\in\mathcal{D}_{\boldsymbol{h}}^{[\boldsymbol{d},\boldsymbol{e}]}(\Gamma,M)\vert v_{\boldsymbol{h}}^{[\boldsymbol{d},\boldsymbol{e}]}(\mu)\geq c^{[\boldsymbol{d},\boldsymbol{e}]}\}\subset I_{\boldsymbol{h}}^{[\boldsymbol{d},\boldsymbol{e}]}(M)^{0}\subset \{\mu\in\mathcal{D}_{\boldsymbol{h}}^{[\boldsymbol{d},\boldsymbol{e}]}(\Gamma,M)\vert v_{\boldsymbol{h}}^{[\boldsymbol{d},\boldsymbol{e}]}(\mu)\geq0\}.$$
\end{proof}
Let $\boldsymbol{d}^{(i)},\boldsymbol{e}^{(i)}\in \mathbb{Z}^{k}$ such that $\boldsymbol{d}^{(i)}\leq \boldsymbol{e}^{(i)}$ with $i=1,2$. Assume that $[\boldsymbol{d}^{(1)},\boldsymbol{e}^{(1)}]\subset [\boldsymbol{d}^{(2)},\boldsymbol{e}^{(2)}]$. By Proposition \ref{admissible proj is isom not hisom} and Theorem \ref{multi-variable results on admissible distributions}, if $\boldsymbol{e}^{(1)}-\boldsymbol{d}^{(1)}\geq \lfloor \boldsymbol{h}\rfloor$, the natural projection map
\begin{equation}\label{projection I is isom if e-dgeq h}
I_{\boldsymbol{h}}^{[\boldsymbol{d}^{(2)},\boldsymbol{e}^{(2)}]}(M)\rightarrow I_{\boldsymbol{h}}^{[\boldsymbol{d}^{(1)},\boldsymbol{e}^{(1)}]}(M)
\end{equation}
is an $\mathcal{O}_{\K}[[\Gamma]]\otimes_{\mathcal{O}_{\K}}\K$-module isomorphism.

\section{Proof of the main result for the case of the deformation space}\label{sc:ordinary deformation}
In this section, we prove main results for the case of deformation spaces. Let $\boldsymbol{h}\in \ord_{p}(\mathcal{O}_{\K}\backslash \{0\})^{k}$ and $\boldsymbol{d},\boldsymbol{e}\in \mathbb{Z}^{k}$ such that $\boldsymbol{e}\geq\boldsymbol{d}$ with a positive integer $k$. Let $\Gamma_{i}$ be a $p$-adic Lie group which is isomorphic to $1+2p\mathbb{Z}_p\subset \mathbb{Q}_{p}^{\times}$ via a continuous character $\chi_{i} : \Gamma_{i} \longrightarrow \mathbb{Q}_{p}^{\times}$ for each $1\leq i\leq k$. We define $\Gamma=\Gamma_{1}\times \cdots \times\Gamma_{k}$. We take a topological generator $\gamma_{i}\in \Gamma_{i}$ and put $u_{i}=\chi_{i}(\gamma_{i})$ with $1\leq i\leq k$. In this section, we fix a $\K$-Banach space $(M,v_{M})$. Let $\mathbf{J}$ be a finite extension on $\mathcal{O}_{\K}[[X_{1},\ldots, X_{k}]]$ such that $\mathbf{J}$ is an integral domain. We denote by $\mathfrak{X}_{\mathbf{J}}$ the set of continuous $\mathcal{O}_{\K}$-algebra homomorphism $\kappa: \mathbf{J}\rightarrow \overline{\K}$ which satisfies $\kappa(X_{i})=u_{i}^{w_{\kappa,i}}\epsilon_{\kappa,i}-1$ for each $1\leq i\leq k$, where $w_{\kappa,i}\in \mathbb{Z}$ and $\epsilon_{\kappa,i}\in\mu_{p^{\infty}}$. For each $\kappa\in \mathfrak{X}_{\mathbf{J}}$, we put $\boldsymbol{w}_{\kappa}=(w_{\kappa,1},\ldots, w_{\kappa,k})$ and $\boldsymbol{\epsilon}_{\kappa}=(\epsilon_{\kappa,1},\ldots, \epsilon_{\kappa,k})$. Let $f=\sum_{j=1}^{n}f_{j}\otimes c_{j}\in \HH_{\boldsymbol{h}}(M)\otimes_{\mathcal{O}_{\K}[[X_{1},\ldots,X_{k}]]}\mathbf{J}$. For each $\kappa\in \mathfrak{X}_{\mathbf{J}}$, we define a specialization $\kappa(f)\in M_{\K_{\kappa}}$ to be
\begin{equation}\label{finite free specialization log}
\kappa(f)=\sum_{j=1}^{n}f_{j}(u_{1}^{w_{\kappa,1}}\epsilon_{\kappa,1}-1,\ldots, u_{k}^{w_{\kappa,k}}\epsilon_{\kappa,k}-1)\kappa(c_{j}),
\end{equation}
where $\K_{\kappa}=\K(\kappa(\mathbf{J}))$. Let $\mathfrak{X}_{\mathbf{J}}^{[\boldsymbol{d},\boldsymbol{e}]}$ be a subset  of $\mathfrak{X}_{\mathbf{J}}$ consisting of $\kappa\in \mathfrak{X}_{\mathbf{J}}$ with $\boldsymbol{w}_{\kappa}\in [\boldsymbol{d},\boldsymbol{e}]$. Hereafter, we assume that $\mathbf{J}$ is a finite free extension of $\mathcal{O}_{\K}[[X_{1},\ldots,X_{k}]]$.
\begin{thm}\label{main theorem 1 for deformation space}
 If $f\in \HH_{\boldsymbol{h}}(M)\otimes_{\mathcal{O}_{\K}[[X_{1},\ldots, X_{k}]]}\mathbf{J}$ satisfies $\kappa(f)=0$ for each $\kappa\in
 \mathfrak{X}_{\mathbf{J}}^{[\boldsymbol{d},\boldsymbol{d}+\lfloor\boldsymbol{h}\rfloor]}$, then $f$ is zero.
\end{thm}
\begin{proof}
By contradiction, we suppose that $f\neq 0$. We take a basis $\alpha_{1},\ldots, \alpha_{n}\in \mathbf{J}$ over $\mathcal{O}_{\K}[[X_{1},\ldots, X_{k}]]$. We write $f=\sum_{j=1}^{n}f_{j}\otimes\alpha_{j}$ with $f_{j}\in \HH_{\boldsymbol{h}}(M)$. We denote by $K$ and $L$ the fraction fields of  $\mathcal{O}_{\K}[[X_{1},\ldots, X_{k}]]$ and $\mathbf{J}$ respectively. Let $\alpha_{1}^{*},\ldots,\alpha_{n}^{*}\in L$ be the dual basis of $\alpha_{1},\ldots, \alpha_{n}$ with respect to the trace map $\mathrm{Tr}_{L\slash K}:L\rightarrow K$. We define 
$$\mathrm{Tr}: \HH_{\boldsymbol{h}}(M)\otimes_{\mathcal{O}_{\K}[[X_{1},\ldots,X_{k}]]}L\rightarrow \HH_{\boldsymbol{h}}(M)\otimes_{\mathcal{O}_{\K}[[X_{1},\ldots,X_{k}]]}K$$
to be $\sum_{j=1}^{m}g_{j}\otimes c_{j}\mapsto \sum_{j=1}^{m}g_{j}\otimes \mathrm{Tr}_{L\slash K}(c_{j})$. By definition, we have $f_{j}=\mathrm{Tr}(f\alpha_{j}^{*})$ for each $1\leq j\leq n$. Let $d=d(\alpha_{1},\ldots, \alpha_{n})\in \mathcal{O}_{\K}[[X_{1},\ldots,X_{k}]]\backslash \{0\}$ be the discriminant of the basis $\alpha_{1},\ldots, \alpha_{n}$. It is well-known that $d\alpha_{i}^{*}\in \mathbf{J}$ with $1\leq i\leq n$. By replacing $f$ with $df$, we can assume that $f\alpha_{j}^{*}\in \HH_{\boldsymbol{h}}(M)\otimes_{\mathcal{O}_{\K}[[X_{1},\ldots,X_{k}]]}\mathbf{J}$ and
\begin{equation}\label{main theorem 1 for deformation spaceeq1}
\kappa(f\alpha_{j}^{*})=0
\end{equation} 
for every $1\leq j\leq n$ and for every $\kappa\in \mathfrak{X}_{\mathbf{J}}^{[\boldsymbol{d},\boldsymbol{d}+\lfloor\boldsymbol{h}\rfloor]}$.

Let $W$ be the Galois closure of $L\slash K$ and $\mathbf{T}$ the integral closure of $\mathcal{O}_{\K}[[X_{1},\ldots,X_{k}]]$ in $W$. For each $K$-embedding $\sigma: L\rightarrow W$ and $g=\sum_{j=1}^{m}g_{j}\otimes c_{j}\in \HH_{\boldsymbol{h}}(M)\otimes_{\mathcal{O}_{\K}[[X_{1},\ldots,X_{k}]]}\mathbf{J}$, we write $\sigma(g)=\sum_{j=1}^{m}g_{j}\otimes \sigma(c_{j})\in \HH_{\boldsymbol{h}}(M)\otimes_{\mathcal{O}_{\K}[[X_{1},\ldots,X_{k}]]}\mathbf{T}$. By the definition of the trace map, we have
\begin{align*}
f_{j}=\mathrm{Tr}(f\alpha_{j}^{*})=\sum_{\sigma}\sigma(f\alpha_{j}^{*})\ \ \mathrm{in}\ \HH_{\boldsymbol{h}}(M)\otimes_{\mathcal{O}_{\K}[[X_{1},\ldots,X_{k}]]}\mathbf{T},
\end{align*}
where the sum $\sum_{\sigma}$ runs over all $K$-embeddings $\sigma:L\rightarrow W$. For each $\kappa\in \mathfrak{X}_{\mathbf{T}}^{[\boldsymbol{d},\boldsymbol{d}+\lfloor \boldsymbol{h}\rfloor]}$ and $K$-enbedding $\sigma:L\rightarrow W$, we have $\kappa\circ \sigma\in \mathfrak{X}_{\mathbf{J}}^{[\boldsymbol{d},\boldsymbol{d}+\lfloor \boldsymbol{h}\rfloor]}$. By \eqref{main theorem 1 for deformation spaceeq1}, we see that 
\begin{equation}\label{main theorem 1 for deformation spaceeq2}
\kappa(f_{j})=\sum_{\sigma}\kappa\circ\sigma(f\alpha_{j}^{*})=0
\end{equation}
for every $\kappa\in \mathfrak{X}_{\mathbf{T}}^{[\boldsymbol{d},\boldsymbol{d}+\lfloor\boldsymbol{h}\rfloor]}$. Since $\mathbf{T}$ is integral over $\mathcal{O}_{\K}[[X_{1},\ldots, X_{k}]]$, we see that the restrection map $\mathfrak{X}_{\mathbf{T}}^{[\boldsymbol{d},\boldsymbol{d}+\lfloor\boldsymbol{h}\rfloor]}\rightarrow \mathfrak{X}_{\mathcal{O}_{\K}[[X_{1},\ldots, X_{k}]]}^{[\boldsymbol{d},\boldsymbol{d}+\lfloor\boldsymbol{h}\rfloor]}$ is surjective. Then, by \eqref{main theorem 1 for deformation spaceeq2}, we see that 
$$\kappa(f_{j})=0$$
for every $1\leq j\leq n$ and for every $\kappa\in \mathfrak{X}_{\mathcal{O}_{\K}[[X_{1},\ldots, X_{k}]]}^{[\boldsymbol{d},\boldsymbol{d}+\lfloor\boldsymbol{h}\rfloor]}$. By $\mathrm{Theorem\ \ref{main theorem 1 and proof}}$, we conclude that $f_{j}=0$ for every $1\leq j\leq n$, which is equivalent to $f=0$. This is a contradiction.
\end{proof}
Let $\alpha_{1},\ldots, \alpha_{n}$ be a basis of $\mathbf{J}$ over $\mathcal{O}_{\K}[[X_{1},\ldots,X_{k}]]$. Through the $\K$-vector isomorphism $\oplus_{i=1}^{n}\HH_{\boldsymbol{h}}(M)\stackrel{\sim}{\rightarrow}\HH_{\boldsymbol{h}}(M)\otimes_{\mathcal{O}_{\K}[[X_{1},\ldots, X_{k}]]}\mathbf{J}$ defined by $(f_{i})_{i=1}^{n}\mapsto \sum_{i=1}^{n}f_{i}\alpha_{i}$, we regard $\HH_{\boldsymbol{h}}(M)\otimes_{\mathcal{O}_{\K}[[X_{1},\ldots,X_{k}]]}\mathbf{J}$ as a $\K$-Banach space and denote by $v_{\mathcal{H}_{\boldsymbol{h}},\mathbf{J}}$ the valuation on $\HH_{\boldsymbol{h}}(M)\otimes_{\mathcal{O}_{\K}[[X_{1},\ldots,X_{k}]]}\mathbf{J}$. That is, $v_{\mathcal{H}_{\boldsymbol{h}},\mathbf{J}}(f)=\min_{1\leq i\leq n}\{v_{\HH_{\boldsymbol{h}}}(f_{i})\}$ for each $f=\sum_{i=1}^{n}f_{i}\alpha_{i}$ with $f_{i}\in \HH_{\boldsymbol{h}}(M)$. We remark that the valuation $v_{\HH_{\boldsymbol{h}},\mathbf{J}}$ does not depend on the basis $\alpha_{1},\ldots, \alpha_{n}$.

Let $J_{\boldsymbol{h}}^{[\boldsymbol{d},\boldsymbol{e}]}(M)$ be the $\mathcal{O}_{\K}[[X_{1},\ldots, X_{k}]]\otimes_{\mathcal{O}_{\K}}\K$-module defined in \eqref{generalization of the project lim for deformation ring Jboldsymbol}. Put $M^{0}(\mathbf{J})=M^{0}[[X_{1},\ldots, X_{k}]]\otimes_{\mathcal{O}_{\K}[[X_{1},\ldots,X_{k}]]}\mathbf{J}$ and $(\Omega_{\boldsymbol{m}}^{[\boldsymbol{d},\boldsymbol{e}]})=(\Omega_{\boldsymbol{m}}^{[\boldsymbol{d},\boldsymbol{e}]}(X_{1},\ldots, X_{k}))$. 
We regard the modules $\varprojlim_{\boldsymbol{m}\in \mathbb{Z}_{\geq 0}^{k}}\left(\frac{M^{0}(\mathbf{J})}{(\Omega_{\boldsymbol{m}}^{[\boldsymbol{d},\boldsymbol{e}]})M^{0}(\mathbf{J})}\otimes_{\mathcal{O}_{\K}}\K\right)$ and $\left(\prod_{\boldsymbol{m}\in \mathbb{Z}_{\geq 0}^{k}}\frac{M^{0}(\mathbf{J})}{(\Omega_{\boldsymbol{m}}^{[\boldsymbol{d},\boldsymbol{e}]})M^{0}(\mathbf{J})}\right)\otimes_{\mathcal{O}_{\K}}\K$ as submodules of $\prod_{\boldsymbol{m}\in \mathbb{Z}_{\geq 0}^{k}}\left(\frac{M^{0}(\mathbf{J})}{(\Omega_{\boldsymbol{m}}^{[\boldsymbol{d},\boldsymbol{e}]})M^{0}(\mathbf{J})}\otimes_{\mathcal{O}_{\K}}\K\right)$. Then, we see that $J_{\boldsymbol{h}}^{[\boldsymbol{d},\boldsymbol{e}]}(M)\otimes_{\mathcal{O}_{\K}[[X_{1},\ldots, X_{k}]]}\mathbf{J}$ is isomorphic to the following $\mathbf{J}\otimes_{\mathcal{O}_{\K}}\K$-module: 
\begin{align}\label{Jbold(I)definition}
\small 
\Bigg\{(s_{\boldsymbol{m}})_{\boldsymbol{m}}\in \varprojlim_{\boldsymbol{m}\in\mathbb{Z}_{\geq 0}^{k}}\left(\frac{M^{0}(\mathbf{J})}{(\Omega_{\boldsymbol{m}}^{[\boldsymbol{d},\boldsymbol{e}]})M^{0}(\mathbf{J})}\otimes_{\mathcal{O}_{\K}}\K\right)\Bigg\vert (p^{\langle \boldsymbol{h},\boldsymbol{m}\rangle_{k}}s_{\boldsymbol{m}})_{\boldsymbol{m}}\in \left(\prod_{\boldsymbol{m}\in \mathbb{Z}_{\geq 0}^{k}}\frac{M^{0}(\mathbf{J})}{(\Omega_{\boldsymbol{m}}^{[\boldsymbol{d},\boldsymbol{e}]})M^{0}(\mathbf{J})}\right)\otimes_{\mathcal{O}_{\K}}\K\Bigg\}.
\normalsize
\end{align}
 Throughout this section, we identify $J_{\boldsymbol{h}}^{[\boldsymbol{d},\boldsymbol{e}]}(M)\otimes_{\mathcal{O}_{\K}[[X_{1},\ldots, X_{k}]]}\mathbf{J}$ with the module given by  \eqref{Jbold(I)definition}. Let $s\in J_{\boldsymbol{h}}^{[\boldsymbol{d},\boldsymbol{e}]}(M)\otimes_{\mathcal{O}_{\K}[[X_{1},\ldots, X_{k}]]}\mathbf{J}$. Whenever we write $s=(s_{\boldsymbol{m}})_{\boldsymbol{m}\in \mathbb{Z}_{\geq 0}^{k}}$, $(s_{\boldsymbol{m}})_{\boldsymbol{m}\in \mathbb{Z}_{\geq 0}^{k}}$ is an element of \eqref{Jbold(I)definition}. We have the following theorem:
\begin{thm}\label{deformtion Jhtheorem multivariable}
Assume that $\boldsymbol{e}-\boldsymbol{d}\geq \lfloor \boldsymbol{h}\rfloor$. For $s^{[\boldsymbol{d},\boldsymbol{e}]}=(s_{\boldsymbol{m}}^{[\boldsymbol{d},\boldsymbol{e}]})_{\boldsymbol{m}\in \mathbb{Z}_{\geq 0}^{k}}\in J_{\boldsymbol{h}}^{[\boldsymbol{d},\boldsymbol{e}]}(M)\linebreak\otimes_{\mathcal{O}_{\K}[[X_{1},\ldots,X_{k}]]}\mathbf{J}$, there exists a unique element $f_{s^{[\boldsymbol{d},\boldsymbol{e}]}}\in \HH_{\boldsymbol{h}}(M)\otimes_{\mathcal{O}_{\K}[[X_{1},\ldots, X_{k}]]}\mathbf{J}$ such that 
\begin{equation}\label{deformtion Jhtheorem multivariable1}
f_{s^{[\boldsymbol{d},\boldsymbol{e}]}}-\tilde{s}_{\boldsymbol{m}}^{[\boldsymbol{d},\boldsymbol{e}]}\in (\Omega_{\boldsymbol{m}}^{[\boldsymbol{d},\boldsymbol{e}]})\HH_{\boldsymbol{h}}(M)\otimes_{\mathcal{O}_{\K}[[X_{1},\ldots, X_{k}]]}\mathbf{J}
\end{equation}
for each $\boldsymbol{m}\in \mathbb{Z}_{\geq 0}^{k}$, where $\tilde{s}_{\boldsymbol{m}}^{[\boldsymbol{d},\boldsymbol{e}]}\in M^{0}(\mathbf{J})\otimes_{\mathcal{O}_{\K}}\K$ is a lift of $s_{\boldsymbol{m}}^{[\boldsymbol{d},\boldsymbol{e}]}$. Further, the correspondence $s^{[\boldsymbol{d},\boldsymbol{e}]}\mapsto f_{s^{[\boldsymbol{d},\boldsymbol{e}]}}$ from $J_{\boldsymbol{h}}^{[\boldsymbol{d},\boldsymbol{e}]}(M)\otimes_{\mathcal{O}_{\K}[[X_{1},\ldots, X_{k}]]}\mathbf{J}$ to $\HH_{\boldsymbol{h}}(M)\otimes_{\mathcal{O}_{\K}[[X_{1},\ldots, X_{k}]]}\mathbf{J}$ induces an $\mathbf{J}\otimes_{\mathcal{O}_{\K}}\K$-module isomorphism \
$$
J_{\boldsymbol{h}}^{[\boldsymbol{d},\boldsymbol{e}]}(M)\otimes_{\mathcal{O}_{\K}[[X_{1},\ldots, X_{k}]]}\mathbf{J} \overset{\sim}{\longrightarrow} \HH_{\boldsymbol{h}}(M)\otimes_{\mathcal{O}_{\K}[[X_{1},\ldots,X_{k}]]}\mathbf{J}$$
and, via the above isomorphism, we have
\begin{align}\label{deformtion Jhtheorem multivariable1 equa 1}
\begin{split}
&\{f\in \HH_{\boldsymbol{h}}(M)\otimes_{\mathcal{O}_{\K}[[X_{1},\ldots, X_{k}]]}\mathbf{J}\vert v_{\HH_{\boldsymbol{h}},\mathbf{J}}(f)\geq \alpha_{\boldsymbol{h}}^{[\boldsymbol{d},\boldsymbol{e}]}\}\subset J_{\boldsymbol{h}}^{[\boldsymbol{d},\boldsymbol{e}]}(M)^{0}\otimes_{\mathcal{O}_{\K}[[X_{1},\ldots, X_{k}]]}\mathbf{J}\\
&\subset \{f\in \HH_{\boldsymbol{h}}(M)\otimes_{\mathcal{O}_{\K}[[X_{1},\ldots, X_{k}]]}\mathbf{J}\vert v_{\HH_{\boldsymbol{h}},\mathbf{J}}(f)\geq \beta_{\boldsymbol{h}}\},
\end{split}
\end{align}
where $\alpha_{\boldsymbol{h}}^{[\boldsymbol{d},\boldsymbol{e}]}=\sum_{i=1}^{k}\alpha_{h_{i}}^{[d_{i},e_{i}]}$ and $\beta_{\boldsymbol{h}}=\sum_{i=1}^{k}\beta_{h_{i}}$ with
\begin{align*}
\alpha_{h_{i}}^{[d_{i},e_{i}]}&=\begin{cases}\lfloor\frac{(e_{i}-d_{i}+1)}{p-1}+\max\{0, h_{i}-\frac{h_{i}}{\log p}(1+\log \frac{\log p}{(p-1)h_{i}})\}\rfloor+1\ &\mathrm{if}\ h_{i}>0,\\ 0\ &\mathrm{if}\ h_{i}=0,\end{cases}\\
\beta_{h_{i}}&=\begin{cases}-\lfloor\max\{h_{i},\frac{p}{p-1}\}\rfloor-1\ \ \ \  \ \ \ \ \ \ \ \ \ \ \ \ \ \ \ \ \ \ \ \ \ \ \ \ \ \ \ \ \ \ \ \ \ &\mathrm{if}\ h_{i}>0,\\ 0\ &\mathrm{if}\ h_{i}=0.\end{cases}
\end{align*}
\end{thm}
\begin{proof}
Let $s^{[\boldsymbol{d},\boldsymbol{e}]}\in J_{\boldsymbol{h}}^{[\boldsymbol{d},\boldsymbol{e}]}(M)\otimes_{\mathcal{O}_{\K}[[X_{1},\ldots,X_{k}]]}\mathbf{J}$. We prove that there exists a unique element $f_{s^{[\boldsymbol{d},\boldsymbol{e}]}}\in \HH_{\boldsymbol{h}}(M)\otimes_{\mathcal{O}_{\K}[[X_{1},\ldots, X_{k}]]}\mathbf{J}$ which satisfies \eqref{deformtion Jhtheorem multivariable1}. The uniqueness of $f_{s^{[\boldsymbol{d},\boldsymbol{e}]}}$ follows from Theorem\ \ref{main theorem 1 for deformation space} immediately. Let $\Psi_{\boldsymbol{J}}:J_{\boldsymbol{h}}^{[\boldsymbol{d},\boldsymbol{e}]}(M)\otimes_{\mathcal{O}_{\K}[[X_{1},\ldots, X_{k}]]}\mathbf{J} \overset{\sim}{\longrightarrow} \HH_{\boldsymbol{h}}(M)\otimes_{\mathcal{O}_{\K}[[X_{1},\ldots,X_{k}]]}\mathbf{J}$ be the $\mathbf{J}\otimes_{\mathcal{O}_{\K}}\K$-module isomorphism induced by the isomorphism $J_{\boldsymbol{h}}^{[\boldsymbol{d},\boldsymbol{e}]}(M)\stackrel{\sim}{\rightarrow}\HH_{\boldsymbol{h}}(M)$ defined in Theorem \ref{main thm 2 and proof}. By the definition of $\Psi_{\boldsymbol{J}}$, we see that $\Psi_{\boldsymbol{J}}(s^{[\boldsymbol{d},\boldsymbol{e}]})$ satisfies \eqref{deformtion Jhtheorem multivariable1}. Then, $\Psi_{\boldsymbol{J}}(s^{[\boldsymbol{d},\boldsymbol{e}]})$ is the unique element which satisifies \eqref{deformtion Jhtheorem multivariable1}.

Since $\Psi_{\boldsymbol{J}}$ is an isomorphism, the correspondence $s^{[\boldsymbol{d},\boldsymbol{e}]}\mapsto f_{s^{[\boldsymbol{d},\boldsymbol{e}]}}$ from $J_{\boldsymbol{h}}^{[\boldsymbol{d},\boldsymbol{e}]}(M)\linebreak\otimes_{\mathcal{O}_{\K}[[X_{1},\ldots, X_{k}]]}\mathbf{J}$ to $\HH_{\boldsymbol{h}}(M)\otimes_{\mathcal{O}_{\K}[[X_{1},\ldots, X_{k}]]}\mathbf{J}$ is an isomorphism. Further, we have \eqref{deformtion Jhtheorem multivariable1 equa 1} by Theorem \ref{main thm 2 and proof}. 
\end{proof}
\begin{rem}
Assume that $\boldsymbol{e}-\boldsymbol{d}\geq \lfloor \boldsymbol{h}\rfloor$. We regard $M^{0}(\mathbf{J})\otimes_{\mathcal{O}_{\K}}\K$ as a $\mathbf{J}$-submodule of $J_{\boldsymbol{h}}^{[\boldsymbol{d},\boldsymbol{e}]}(M)\otimes_{\mathcal{O}_{\K}[[X_{1},\ldots, X_{k}]]}\mathbf{J}$ and $\HH_{\boldsymbol{h}}(M)\otimes_{\mathcal{O}_{\K}[[X_{1},\ldots, X_{k}]]}\mathbf{J}$ naturally and denote by $i: M^{0}(\mathbf{J})\otimes_{\mathcal{O}_{\K}}\K\rightarrow J_{\boldsymbol{h}}^{[\boldsymbol{d},\boldsymbol{e}]}(M)\otimes_{\mathcal{O}_{\K}[[X_{1},\ldots, X_{k}]]}\mathbf{J}$ and $j:M^{0}(\mathbf{J})\otimes_{\mathcal{O}_{\K}}\K\rightarrow \HH_{\boldsymbol{h}}(M)\otimes_{\mathcal{O}_{\K}[[X_{1},\ldots, X_{k}]]}\mathbf{J}$ the natural inclusion maps respectively. We denote by $\varphi: J_{\boldsymbol{h}}^{[\boldsymbol{d},\boldsymbol{e}]}(M)\otimes_{\mathcal{O}_{\K}[[X_{1},\ldots, X_{k}]]}\mathbf{J} \overset{\sim}{\longrightarrow} \HH_{\boldsymbol{h}}(M)\otimes_{\mathcal{O}_{\K}[[X_{1},\ldots, X_{k}]]}\mathbf{J}$ the $\mathbf{J}\otimes_{\mathcal{O}_{\K}}\K$-module isomorphism defined in Theorem \ref{deformtion Jhtheorem multivariable}. In the same way as Remark \ref{remark of uniqueness of jh}, we see that $\varphi$ is the unique $\mathbf{J}\otimes_{\mathcal{O}_{\K}}\K$-module isomorphism which satisfies $\varphi i=j$.
\end{rem}
We fix a finite free extension $\mathbf{I}$ of $\mathcal{O}_{\K}[[\Gamma]]$ such that $\mathbf{I}$ is an integral domain. Let $\mathfrak{X}_{\mathbf{I}}$ be the set of arithmetic specializations on $\mathbf{I}$ and $\mathfrak{X}_{\mathbf{I}}^{[\boldsymbol{d},\boldsymbol{e}]}\subset \mathfrak{X}_{\mathbf{I}}$ a subset consisting of $\kappa\in \mathfrak{X}_{\mathbf{I}}$ with $\boldsymbol{w}_{\kappa}\in [\boldsymbol{d},\boldsymbol{e}]$. Put $M^{0}(\mathbf{I})=M^{0}[[\Gamma]]\otimes_{\mathcal{O}_{\K}[[\Gamma]]}\mathbf{I}$. Let $\kappa\in \mathfrak{X}_{\mathbf{I}}$ and $f=\sum_{i=0}^{n}f_{i}\otimes_{\mathcal{O}_{\K}[[\Gamma]]}c_{i}\in M^{0}(\mathbf{I})$, where $f_{i}\in M^{0}[[\Gamma]]$ and $c_{i}\in \mathbf{I}$ for each $1\leq i\leq n$. We define a substitution $\kappa(f)\in  M_{\K_{\kappa}}$ to be $\kappa(f)=\sum_{i=1}^{n}\kappa(f_{i})\kappa(c_{i})$, where $\K_{\kappa}=\K(\kappa(\mathbf{I}))$. Let $\mathbf{I}_{\boldsymbol{h}}^{[\boldsymbol{d},\boldsymbol{e}]}(M)$ be the $\mathcal{O}_{\K}[[\Gamma]]\otimes_{\mathcal{O}_{\K}}\K$-module defined in \eqref{generalization of the project lim for deformation ring}. In the same way as \eqref{Jbold(I)definition}, we can identify $I_{\boldsymbol{h}}^{[\boldsymbol{d},\boldsymbol{e}]}(M)\otimes_{\mathcal{O}_{\K}[[\Gamma]]}\mathbf{I}$ with the following $\mathbf{I}\otimes_{\mathcal{O}_{\K}}\K$-module: 
\begin{align}\label{Ibold(I)definition}
\begin{split}
\Bigg\{(s_{\boldsymbol{m}}^{[\boldsymbol{d},\boldsymbol{e}]})_{\boldsymbol{m}}\in &\varprojlim_{\boldsymbol{m}\in\mathbb{Z}_{\geq 0}^{k}}\left(\frac{M^{0}(\mathbf{I})}{(\Omega_{\boldsymbol{m}}^{[\boldsymbol{d},\boldsymbol{e}]}(\gamma_{1},\ldots, \gamma_{k}))M^{0}(\mathbf{I})}\otimes_{\mathcal{O}_{\K}}\K\right)\Bigg\vert \\
&(p^{\langle \boldsymbol{h},\boldsymbol{m}\rangle_{k}}s_{\boldsymbol{m}}^{[\boldsymbol{d},\boldsymbol{e}]})_{\boldsymbol{m}}\in \left(\prod_{\boldsymbol{m}\in \mathbb{Z}_{\geq 0}^{k}}\frac{M^{0}(\mathbf{I})}{(\Omega_{\boldsymbol{m}}^{[\boldsymbol{d},\boldsymbol{e}]}(\gamma_{1},\ldots, \gamma_{k}))M^{0}(\mathbf{I})}\right)\otimes_{\mathcal{O}_{\K}}\K\Bigg\}.
\end{split}
\end{align}
 Throughout this section, we identify $I_{\boldsymbol{h}}^{[\boldsymbol{d},\boldsymbol{e}]}(M)\otimes_{\mathcal{O}_{\K}[[\Gamma]]}\mathbf{I}$ with the module given by  \eqref{Ibold(I)definition}. Let $s^{[\boldsymbol{d},\boldsymbol{e}]}\in I_{\boldsymbol{h}}^{[\boldsymbol{d},\boldsymbol{e}]}(M)\otimes_{\mathcal{O}_{\K}[[\Gamma]]}\mathbf{I}$. Whenever we write $s^{[\boldsymbol{d},\boldsymbol{e}]}=(s_{\boldsymbol{m}}^{[\boldsymbol{d},\boldsymbol{e}]})_{\boldsymbol{m}\in \mathbb{Z}_{\geq 0}^{k}}$, $(s_{\boldsymbol{m}}^{[\boldsymbol{d},\boldsymbol{e}]})_{\boldsymbol{m}\in \mathbb{Z}_{\geq 0}^{k}}$ is an element of \eqref{Ibold(I)definition}.
 
Let $\alpha_{1},\ldots, \alpha_{n}$ be a basis of $\mathbf{I}$ over $\mathcal{O}_{\K}[[\Gamma]]$. We regard $\mathcal{D}_{\boldsymbol{h}}^{[\boldsymbol{d},\boldsymbol{e}]}(\Gamma,M)\otimes_{\mathcal{O}_{\K}[[\Gamma]]}\mathbf{I}$ as a $\K$-Banach space 
through the $\K$-linear isomorphism $\oplus_{i=1}^{n}\mathcal{D}_{\boldsymbol{h}}^{[\boldsymbol{d},\boldsymbol{e}]}(\Gamma,M)\stackrel{\sim}{\rightarrow}\mathcal{D}_{\boldsymbol{h}}^{[\boldsymbol{d},\boldsymbol{e}]}(\Gamma,M)\otimes_{\mathcal{O}_{\K}[[\Gamma]]}\mathbf{I}$ defined by $(\mu_{i})_{i=1}^{n}\mapsto \sum_{i=1}^{n}\mu_{i}\alpha_{i}$. 
We denote by $v_{\mathcal{D}_{\boldsymbol{h}}^{[\boldsymbol{d},\boldsymbol{e}]},\mathbf{I}}$ the valuation on $\mathcal{D}_{\boldsymbol{h}}^{[\boldsymbol{d},\boldsymbol{e}]}(\Gamma,M)\otimes_{\mathcal{O}_{\K}[[\Gamma]]}\mathbf{I}$. That is, $v_{\mathcal{D}_{\boldsymbol{h}}^{[\boldsymbol{d},\boldsymbol{e}]},\mathbf{I}}(\mu)=\min_{1\leq i\leq n}\{v_{\boldsymbol{h}}^{[\boldsymbol{d},\boldsymbol{e}]}(\mu_{i})\}$ for each $\mu=\sum_{i=1}^{n}\mu_{i}\otimes_{\mathcal{O}_{\K}[[\Gamma]]}\alpha_{i}$ with $\mu_{i}\in \mathcal{D}_{\boldsymbol{h}}^{[\boldsymbol{d},\boldsymbol{e}]}(\Gamma,\K)$. Let $\mu=\sum_{i=1}^{m}\mu_{i}\otimes_{\mathcal{O}_{\K}[[\Gamma]]}a_{i}\in \mathcal{D}_{\boldsymbol{h}}^{[\boldsymbol{d},\boldsymbol{e}]}(\Gamma, M)\otimes_{\mathcal{O}_{\K}[[\Gamma]]}\mathbf{I}$ with $\mu_{i}\in \mathcal{D}_{\boldsymbol{h}}^{[\boldsymbol{d},\boldsymbol{e}]}(\Gamma, M)$ and $a_{i}\in \mathbf{I}$. For each $\kappa\in \mathfrak{X}_{\mathbf{I}}^{[\boldsymbol{d},\boldsymbol{e}]}$, we define a specialization $\kappa(\mu)\in M_{\K_{\kappa}}$ to be
\begin{align}\label{specalization of admissible deformation}
\kappa(\mu)=\sum_{i=1}^{m}\int_{\Gamma}\kappa\vert_{\Gamma}d\mu_{i}\kappa(a_{i}).
\end{align}
By the following proposition, an element $\mu\in \mathcal{D}_{\boldsymbol{h}}^{[\boldsymbol{d},\boldsymbol{e}]}(\Gamma,M)\otimes_{\mathcal{O}_{\K}[[\Gamma]]}\mathbf{I}$ is characterized by the specializations \eqref{specalization of admissible deformation} with sufficiently many $\kappa$.

\begin{pro}\label{admissible distribution characterization prop}
Let $\boldsymbol{d}\in \mathbb{Z}^{k}$ and $\mu\in \mathcal{D}_{\boldsymbol{h}}^{[\boldsymbol{d},\boldsymbol{d}+\lfloor \boldsymbol{h}\rfloor]}(\Gamma,M)\otimes_{\mathcal{O}_{\K}[[\Gamma]]}\mathbf{I}$. If $\mu$ satisfies $\kappa(\mu)=0$ for every $\kappa\in \mathfrak{X}_{\mathbf{I}}^{[\boldsymbol{d},\boldsymbol{d}+\lfloor \boldsymbol{h}\rfloor]}$, 
then we have $\mu=0$.
\end{pro}
\begin{proof}
Via the non-canonical $\mathcal{O}_{\K}$-algebra isomorphism $\mathcal{O}_{\K}[[\Gamma]]\simeq \mathcal{O}_{\K}[[X_{1},\ldots, X_{k}]]$ in \eqref{iwasaawa noncaonnical multiisom}, we can regard $\mathbf{I}$ as an $\mathcal{O}_{\K}[[X_{1},\ldots, X_{k}]]$-algebra. We denote by $\mathbf{I}^{\prime}$ the $\mathcal{O}_{\K}[[X_{1},\ldots, X_{k}]]$-algebra $\mathbf{I}$. Then, by Theorem \ref{main thm 2 and proof} and Theorem \ref{multi-variable results on admissible distributions}, we have a non-canonical $\K$-Banach isomorphism 
$$\mathcal{D}_{\boldsymbol{h}}^{[\boldsymbol{d},\boldsymbol{d}+\lfloor \boldsymbol{h}\rfloor]}(\Gamma,M)\otimes_{\mathcal{O}_{\K}[[\Gamma]]}\mathbf{I}\stackrel{\sim}{\rightarrow} \HH_{\boldsymbol{h}}(M)\otimes_{\mathcal{O}_{\K}[[X_{1},\ldots, X_{k}]]}\mathbf{I}^{\prime},\ \mu\mapsto f_{\mu}$$ 
such that $\kappa(\mu)=\kappa(f_{\mu})$ for each $\mu\in \mathcal{D}_{\boldsymbol{h}}^{[\boldsymbol{d},\boldsymbol{d}+\lfloor \boldsymbol{h}\rfloor]}(\Gamma,M)\otimes_{\mathcal{O}_{\K}[[\Gamma]]}\mathbf{I}$ and $\kappa\in \mathfrak{X}_{\mathbf{I}}^{[\boldsymbol{d},\boldsymbol{d}+\lfloor \boldsymbol{h}\rfloor]}$. Then, this proposition follows from Theorem \ref{main theorem 1 for deformation space}.
\end{proof}
{The following lemma is a generalization of Lemma \ref{multivariable litfing prop} to the setting of deformation spaces.}
\begin{lem}\label{multivariable litfing prop deformation}
Let $s^{[\boldsymbol{i}]}\in M^{0}(\mathbf{I})\otimes_{\mathcal{O}_{\K}}\K$ for each $\boldsymbol{i}\in [\boldsymbol{d},\boldsymbol{e}]$ and 
we define $\theta^{(\boldsymbol{j})} \in  M^{0}(\mathbf{I})\otimes_{\mathcal{O}_{\K}}\K$ by 
$$\theta^{(\boldsymbol{j})}=\displaystyle{\sum_{\boldsymbol{i}\in [\boldsymbol{d},\boldsymbol{j}]}}\left(\prod_{t=1}^{k}\begin{pmatrix}j_{t}-d_{t}\\i_{t}-d_{t}\end{pmatrix}\right)(-1)^{\sum_{t=1}^{k}(j_{t}-i_{t})}s^{[\boldsymbol{i}]}
$$
for each $\boldsymbol{j}\in [\boldsymbol{d},\boldsymbol{e}]$. 
Let $\boldsymbol{m}\in\mathbb{Z}_{\geq 0}^{k}$. 
If $\theta^{(\boldsymbol{j})}$ is contained in $p^{\langle\boldsymbol{m},(\boldsymbol{j}-\boldsymbol{d})\rangle_{k}}M^{0}(\mathbf{I})\subset 
 M^{0}(\mathbf{I})\otimes_{\mathcal{O}_{\K}}\K$ 
for every $\boldsymbol{j}\in [\boldsymbol{d},\boldsymbol{e}]$, there exists a unique element $s^{[\boldsymbol{d},\boldsymbol{e}]}\in \frac{M^{0}(\mathbf{I})}{(\Omega_{\boldsymbol{m}}^{[\boldsymbol{d},\boldsymbol{e}]}(\gamma_{1},\ldots,\gamma_{k}))M^{0}(\mathbf{I})}\otimes_{\mathcal{O}_{\K}}p^{-c^{[\boldsymbol{d},\boldsymbol{e}]}}\mathcal{O}_{\K}$ such that the image of $s^{[\boldsymbol{d},\boldsymbol{e}]}$ by the natural projection 
$$
\frac{M^{0}(\mathbf{I})}{(\Omega_{\boldsymbol{m}}^{[\boldsymbol{d},\boldsymbol{e}]}(\gamma_{1},\ldots,\gamma_{k}))M^{0}(\mathbf{I})}\otimes_{\mathcal{O}_{\K}}\K \longrightarrow \frac{M^{0}(\mathbf{I})}{(\Omega_{\boldsymbol{m}}^{[\boldsymbol{i}]}(\gamma_{1},\ldots,\gamma_{k}))M^{0}(\mathbf{I})}\otimes_{\mathcal{O}_{\K}}\K
$$
is equal to the class $[s^{[\boldsymbol{i}]}]_{\boldsymbol{m}} 
\in \tfrac{M^{0}(\mathbf{I})}{(\Omega_{\boldsymbol{m}}^{[\boldsymbol{i}]}(\gamma_{1},\ldots,\gamma_{k}))M^{0}(\mathbf{I})}\otimes_{\mathcal{O}_{\K}}\K$ of $s^{[\boldsymbol{i}]}\in M^{0}(\mathbf{I})\otimes_{\mathcal{O}_{\K}}\K$ for each $\boldsymbol{i}\in [\boldsymbol{d},\boldsymbol{e}]$,where $c^{[\boldsymbol{d},\boldsymbol{e}]}$ is the constant defined in \eqref{constant for the admissible}.
\end{lem}
\begin{proof}
Let $\alpha_{1},\ldots,\alpha_{n}$ be a basis of $\mathbf{I}$ over $\mathcal{O}_{\K}[[\Gamma]]$. Put $s^{[\boldsymbol{i}]}=\sum_{v=1}^{n}s^{[\boldsymbol{i}]}_{v}\otimes_{\mathcal{O}_{\K}[[\Gamma]]}\alpha_{v}$ with $s^{[\boldsymbol{i}]}_{v}\in M^{0}[[\Gamma]]\otimes_{\mathcal{O}_{\K}}\K$ for each $\boldsymbol{i}\in [\boldsymbol{d},\boldsymbol{e}]$. Since
$$\theta^{(\boldsymbol{j})}=\sum_{v=1}^{n}\left(\displaystyle{\sum_{\boldsymbol{i}\in [\boldsymbol{d},\boldsymbol{j}]}}\left(\prod_{t=1}^{k}\begin{pmatrix}j_{t}-d_{t}\\i_{t}-d_{t}\end{pmatrix}\right)(-1)^{\sum_{t=1}^{k}(j_{t}-i_{t})}s_{v}^{[\boldsymbol{i}]}\right)\alpha_{v}$$
for every $\boldsymbol{j}\in [\boldsymbol{d},\boldsymbol{e}]$, we have
$$\displaystyle{\sum_{\boldsymbol{i}\in [\boldsymbol{d},\boldsymbol{j}]}}\left(\prod_{t=1}^{k}\begin{pmatrix}j_{t}-d_{t}\\i_{t}-d_{t}\end{pmatrix}\right)(-1)^{\sum_{t=1}^{k}(j_{t}-i_{t})}s^{[\boldsymbol{i}]}_{v}\in p^{\langle\boldsymbol{m},(\boldsymbol{j}-\boldsymbol{d})\rangle_{k}}M^{0}[[\Gamma]]$$
for every $\boldsymbol{j}\in [\boldsymbol{d},\boldsymbol{e}]$ and $1\leq v\leq n$. Then, by Lemma \ref{multivariable litfing prop}, there exists a unique element $s_{v}^{[\boldsymbol{d},\boldsymbol{e}]}\in\frac{M^{0}[[\Gamma]]}{(\Omega_{\boldsymbol{m}}^{[\boldsymbol{d},\boldsymbol{e}]}(\gamma_{1},\ldots,\gamma_{k}))M^{0}[[\Gamma]]}\otimes_{\mathcal{O}_{\K}}p^{-c^{[\boldsymbol{d},\boldsymbol{e}]}}\mathcal{O}_{\K}$ such that the image of $s_{v}^{[\boldsymbol{d},\boldsymbol{e}]}$ by the natural projection 
$$
\frac{M^{0}[[\Gamma]]}{(\Omega_{\boldsymbol{m}}^{[\boldsymbol{d},\boldsymbol{e}]}(\gamma_{1},\ldots,\gamma_{k}))M^{0}[[\Gamma]]}\otimes_{\mathcal{O}_{\K}}\K \longrightarrow \frac{M^{0}[[\Gamma]]}{(\Omega_{\boldsymbol{m}}^{[\boldsymbol{i}]}(\gamma_{1},\ldots,\gamma_{k}))M^{0}[[\Gamma]]}\otimes_{\mathcal{O}_{\K}}\K
$$
is equal to the class $[s^{[\boldsymbol{i}]}_{v}]_{\boldsymbol{m}} 
\in \tfrac{M^{0}[[\Gamma]]}{(\Omega_{\boldsymbol{m}}^{[\boldsymbol{i}]}(\gamma_{1},\ldots,\gamma_{k}))M^{0}[[\Gamma]]}\otimes_{\mathcal{O}_{\K}}\K$ of $s^{[\boldsymbol{i}]}_{v}\in M^{0}[[\Gamma]]\otimes_{\mathcal{O}_{\K}}\K$ for each $\boldsymbol{i}\in [\boldsymbol{d},\boldsymbol{e}]$. Put $s^{[\boldsymbol{d},\boldsymbol{e}]}=\sum_{v=1}^{n}s_{v}^{[\boldsymbol{d},\boldsymbol{e}]}\otimes_{\mathcal{O}_{\K}[[\Gamma]]}\alpha_{v}\in \frac{M^{0}(\mathbf{I})}{(\Omega_{\boldsymbol{m}}^{[\boldsymbol{d},\boldsymbol{e}]}(\gamma_{1},\ldots,\gamma_{k}))M^{0}(\mathbf{I})}\otimes_{\mathcal{O}_{\K}}p^{-c^{[\boldsymbol{d},\boldsymbol{e}]}}\mathcal{O}_{\K}$. By the definition of $s_{v}^{[\boldsymbol{d},\boldsymbol{e}]}$, we see that $s^{[\boldsymbol{d},\boldsymbol{e}]}$ is the unique element such that the image of $s^{[\boldsymbol{d},\boldsymbol{e}]}$ by the natural projection $
\frac{M^{0}(\mathbf{I})}{(\Omega_{\boldsymbol{m}}^{[\boldsymbol{d},\boldsymbol{e}]}(\gamma_{1},\ldots,\gamma_{k}))M^{0}(\mathbf{I})}\otimes_{\mathcal{O}_{\K}}\K \longrightarrow \frac{M^{0}(\mathbf{I})}{(\Omega_{\boldsymbol{m}}^{[\boldsymbol{i}]}(\gamma_{1},\ldots,\gamma_{k}))M^{0}(\mathbf{I})}\otimes_{\mathcal{O}_{\K}}\K
$ is equal to the class $[s^{[\boldsymbol{i}]}]_{\boldsymbol{m}} 
\in \tfrac{M^{0}(\mathbf{I})}{(\Omega_{\boldsymbol{m}}^{[\boldsymbol{i}]}(\gamma_{1},\ldots,\gamma_{k}))M^{0}(\mathbf{I})}\otimes_{\mathcal{O}_{\K}}\K$ for each $\boldsymbol{i}\in [\boldsymbol{d},\boldsymbol{e}]$.
\end{proof}

{The following proposition is a generalization of Proposition \ref{multivariable iwasawa I(ii) sufficient} to the setting of deformation spaces.}
\begin{pro}\label{deformation Ih[d,e] lift from Ihi,i}
Let $s^{[\boldsymbol{i}]}=(s^{[\boldsymbol{i}]}_{\boldsymbol{m}})_{\boldsymbol{m}\in \mathbb{Z}_{\geq 0}^{k}}\in I_{\boldsymbol{h}}^{[\boldsymbol{i}]}(M)\otimes_{\mathcal{O}_{\K}[[\Gamma]]}\mathbf{I}$ and $\tilde{s}_{\boldsymbol{m}}^{[\boldsymbol{i}]}\in M^{0}(\mathbf{I})\otimes_{\mathcal{O}_{\K}}\K$ a lift of $s_{\boldsymbol{m}}^{[\boldsymbol{i}]}$ for each $\boldsymbol{m}\in \mathbb{Z}_{\geq 0}^{k}$ and for each $\boldsymbol{i}\in [\boldsymbol{d},\boldsymbol{e}]$. If there exists a non-negative integer $n$ which satisfies
$$p^{\langle \boldsymbol{m},\boldsymbol{h}-(\boldsymbol{j}-\boldsymbol{d})\rangle_{k}}\displaystyle{\sum_{\boldsymbol{i}\in [\boldsymbol{d},\boldsymbol{j}]}}\left(\prod_{t=1}^{k}\begin{pmatrix}j_{t}-d_{t}\\i_{t}-d_{t}\end{pmatrix}\right)(-1)^{\sum_{t=1}^{k}(j_{t}-i_{t})}\tilde{s}_{\boldsymbol{m}}^{[\boldsymbol{i}]}\in M^{0}(\mathbf{I})\otimes_{\mathcal{O}_{\K}}p^{-n}\mathcal{O}_{\K}$$
for each $\boldsymbol{m}\in \mathbb{Z}_{\geq 0}^{k}$ and for each $\boldsymbol{j}\in [\boldsymbol{d},\boldsymbol{e}]$, we have a unique element $s^{[\boldsymbol{d},\boldsymbol{e}]}\in I_{\boldsymbol{h}}^{[\boldsymbol{d},\boldsymbol{e}]}(M)^{0}\linebreak\otimes_{\mathcal{O}_{\K}[[\Gamma]]}\mathbf{I}\otimes_{\mathcal{O}_{\K}}p^{-c^{[\boldsymbol{d},\boldsymbol{e}]}-n}\mathcal{O}_{\K}$ such that the image of $s^{[\boldsymbol{d},\boldsymbol{e}]}$ by the natural projection $I_{\boldsymbol{h}}^{[\boldsymbol{d},\boldsymbol{e}]}(M)\linebreak\otimes_{\mathcal{O}_{\K}[[\Gamma]]}\mathbf{I}\rightarrow I_{\boldsymbol{h}}^{[\boldsymbol{i}]}(M)\otimes_{\mathcal{O}_{\K}[[\Gamma]]}\mathbf{I}$ is $s^{[\boldsymbol{i}]}$ for every $\boldsymbol{i}\in [\boldsymbol{d},\boldsymbol{e}]$, where $c^{[\boldsymbol{d},\boldsymbol{e}]}$ is the constant defined in \eqref{constant for the admissible}.
\end{pro}
\begin{proof}
For each $\boldsymbol{m}\in \mathbb{Z}_{\geq 0}^{k}$, there exists a unique elemet $s_{\boldsymbol{m}}^{[\boldsymbol{d},\boldsymbol{e}]}\in \frac{M^{0}(\mathbf{I})}{(\Omega_{\boldsymbol{m}}^{[\boldsymbol{d},\boldsymbol{e}]}(\gamma_{1},\ldots,\gamma_{k}))M^{0}(\mathbf{I})}\otimes_{\mathcal{O}_{\K}}p^{-\langle \boldsymbol{h},\boldsymbol{m}\rangle_{k}-c^{[\boldsymbol{d},\boldsymbol{e}]}-n}\mathcal{O}_{\K}$ such that the image of $s_{\boldsymbol{m}}^{[\boldsymbol{d},\boldsymbol{e}]}$ by the natural projection \linebreak$\frac{M^{0}(\mathbf{I})}{(\Omega_{\boldsymbol{m}}^{[\boldsymbol{d},\boldsymbol{e}]}(\gamma_{1},\ldots,\gamma_{k}))M^{0}(\mathbf{I})}\otimes_{\mathcal{O}_{\K}}\K\rightarrow \frac{M^{0}(\mathbf{I})}{(\Omega_{\boldsymbol{m}}^{[\boldsymbol{i}]}(\gamma_{1},\ldots,\gamma_{k}))M^{0}(\mathbf{I})}\otimes_{\mathcal{O}_{\K}}\K$ is $s_{\boldsymbol{m}}^{[\boldsymbol{i}]}$ for every $\boldsymbol{i}\in [\boldsymbol{d},\boldsymbol{e}]$ by Lemma \ref{multivariable litfing prop deformation}. 
Since this construction is compatible with the projective systems of $s_{\boldsymbol{m}}^{[\boldsymbol{d},\boldsymbol{e}]}$ and $s_{\boldsymbol{m}}^{[\boldsymbol{i}]}$ with respect to $\boldsymbol{m}$, $s^{[\boldsymbol{d},\boldsymbol{e}]}=(s_{\boldsymbol{m}}^{[\boldsymbol{d},\boldsymbol{e}]})_{\boldsymbol{m}\in \mathbb{Z}_{\geq 0}^{k}}\in I_{\boldsymbol{h}}^{[\boldsymbol{d},\boldsymbol{e}]}(M)^{0}\otimes_{\mathcal{O}_{\K}}p^{-c^{[\boldsymbol{d},\boldsymbol{e}]}-n}\mathcal{O}_{\K}$ such that the image of $s^{[\boldsymbol{d},\boldsymbol{e}]}$ by the natural projection $I_{\boldsymbol{h}}^{[\boldsymbol{d},\boldsymbol{e}]}(M)\rightarrow I_{\boldsymbol{h}}^{[\boldsymbol{i}]}(M)$ is $s^{[\boldsymbol{i}]}$ for every $\boldsymbol{i}\in [\boldsymbol{d},\boldsymbol{e}]$.
\end{proof}
We remark that we do not require the condition $\boldsymbol{e}-\boldsymbol{d}\geq\lfloor \boldsymbol{h}\rfloor$ in Lemma \ref{multivariable litfing prop deformation} and Proposition \ref{deformation Ih[d,e] lift from Ihi,i}.
{The following theorem is a generalization of Theorem \ref{multi-variable results on admissible distributions} to the setting of  deformation spaces.}
\begin{thm}\label{multi-variable results on admissible distributions deformation ver}
Assume that $\boldsymbol{e}-\boldsymbol{d}\geq \lfloor \boldsymbol{h}\rfloor$. We have a unique $\mathbf{I}\otimes_{\mathcal{O}_{\K}}\K$-module isomorphism 
\begin{equation}\label{equation:multi-variable results on admissible distributions deformation ver}
\Psi:I_{\boldsymbol{h}}^{[\boldsymbol{d},\boldsymbol{e}]}(M)\otimes_{\mathcal{O}_{\K}[[\Gamma]]}\mathbf{I} \stackrel{\sim}{\rightarrow} \mathcal{D}^{[\boldsymbol{d},\boldsymbol{e}]}_{\boldsymbol{h}} (\Gamma, M)\otimes_{\mathcal{O}_{\K}[[\Gamma]]}\mathbf{I}
\end{equation}
such that the image $\mu_{s^{[\boldsymbol{d},\boldsymbol{e}]}}\in \mathcal{D}^{[\boldsymbol{d},\boldsymbol{e}]}_{\boldsymbol{h}} (\Gamma, M)\otimes_{\mathcal{O}_{\K}[[\Gamma]]}\mathbf{I}$ of each element $s^{[\boldsymbol{d},\boldsymbol{e}]}=(s_{\boldsymbol{m}}^{[\boldsymbol{d},\boldsymbol{e}]})_{\boldsymbol{m}\in \mathbb{Z}_{\geq 0}^{k}} \in I_{\boldsymbol{h}}^{[\boldsymbol{d},\boldsymbol{e}]}(M)\otimes_{\mathcal{O}_{\K}[[\Gamma]]}\mathbf{I}$ is characterized by the interpolation property  
\begin{align}\label{admissible interpolation formula deformation}
\kappa(\tilde{s}_{\boldsymbol{m}_{\kappa}}^{[\boldsymbol{d},\boldsymbol{e}]})=\kappa(\mu_{s^{[\boldsymbol{d},\boldsymbol{e}]}}) \end{align}
 for each $\kappa\in \mathfrak{X}_{\mathcal{O}_{\K}[[\Gamma]]}^{[\boldsymbol{d},\boldsymbol{e}]}$, where $\tilde{s}_{\boldsymbol{m}_{\kappa}}^{[\boldsymbol{d},\boldsymbol{e}]}\in M^{0}(\mathbf{I})\otimes_{\mathcal{O}_{\K}}\K$ is a lift of $s_{\boldsymbol{m}_{\kappa}}^{[\boldsymbol{d},\boldsymbol{e}]}$. In addition, via the isomorphism \eqref{equation:multi-variable results on admissible distributions deformation ver}, we have 
\begin{multline}\label{admissible interpolation formula deformation sub int}
\{\mu\in \mathcal{D}^{[\boldsymbol{d},\boldsymbol{e}]}_{\boldsymbol{h}} (\Gamma ,M)\otimes_{\mathcal{O}_{\K}[[\Gamma]]}\mathbf{I}\vert v_{\mathcal{D}_{\boldsymbol{h}}^{[\boldsymbol{d},\boldsymbol{e}]},\mathbf{I}}(\mu)\geq c^{[\boldsymbol{d},\boldsymbol{e}]}\}\subset I_{\boldsymbol{h}}^{[\boldsymbol{d},\boldsymbol{e}]}(M)^{0}\otimes_{\mathcal{O}_{\K}[[\Gamma]]}\mathbf{I}\\
\subset \{\mu\in \mathcal{D}^{[\boldsymbol{d},\boldsymbol{e}]}_{\boldsymbol{h}} (\Gamma ,M)\otimes_{\mathcal{O}_{\K}[[\Gamma]]}\mathbf{I}\vert v_{\mathcal{D}_{\boldsymbol{h}}^{[\boldsymbol{d},\boldsymbol{e}]},\mathbf{I}}(\mu)\geq 0\},
\end{multline}
where $c^{[\boldsymbol{d},\boldsymbol{e}]}=\sum_{i=1}^{k}c^{[d_{i},e_{i}]}$ is the constant defined in \eqref{constant for the admissible}.
\end{thm}
\begin{proof}
Let $s^{[\boldsymbol{d},\boldsymbol{e}]}\in I_{\boldsymbol{h}}^{[\boldsymbol{d},\boldsymbol{e}]}(M)\otimes_{\mathcal{O}_{\K}[[\Gamma]]}\mathbf{I}$. We prove that there exists a unique element $\mu_{s^{[\boldsymbol{d},\boldsymbol{e}]}}\in \mathcal{D}^{[\boldsymbol{d},\boldsymbol{e}]}_{\boldsymbol{h}} (\Gamma, M)\otimes_{\mathcal{O}_{\K}[[\Gamma]]}\mathbf{I}$ which satisfies \eqref{admissible interpolation formula deformation}. The uniqueness of $\mu_{s^{[\boldsymbol{d},\boldsymbol{e}]}}$ follows from Proposition \ref{admissible distribution characterization prop}. Let $\Psi:I_{\boldsymbol{h}}^{[\boldsymbol{d},\boldsymbol{e}]}(M)\otimes_{\mathcal{O}_{\K}[[\Gamma]]}\mathbf{I} \stackrel{\sim}{\rightarrow} \mathcal{D}^{[\boldsymbol{d},\boldsymbol{e}]}_{\boldsymbol{h}} (\Gamma, M)\otimes_{\mathcal{O}_{\K}[[\Gamma]]}\mathbf{I}$ be the isomorphism induced by the isomorphism $I_{\boldsymbol{h}}^{[\boldsymbol{d},\boldsymbol{e}]}(M)\stackrel{\sim}{\rightarrow} \mathcal{D}^{[\boldsymbol{d},\boldsymbol{e}]}_{\boldsymbol{h}} (\Gamma, M)$ defined in Theorem \ref{multi-variable results on admissible distributions}. By the definition of $\Psi$, we see that $\Psi(s^{[\boldsymbol{d},\boldsymbol{e}]})$ satisfies \eqref{admissible interpolation formula deformation}. Therefore, $\Psi$ is the unique isomorphism which satisfies \eqref{admissible interpolation formula deformation}.  By Theorem \ref{multi-variable results on admissible distributions}, we have \eqref{admissible interpolation formula deformation sub int}
\end{proof}
The family $\left(\mathcal{D}_{\boldsymbol{h}}^{[\boldsymbol{a},\boldsymbol{b}]}(\Gamma,M)\right)_{\substack{\boldsymbol{a},\boldsymbol{b}\in \mathbb{Z}^{k}\\ \boldsymbol{b}\geq \boldsymbol{a}}}$ becomes a projecitve system by the natural restriction map $\mathcal{D}_{\boldsymbol{h}}^{[\boldsymbol{a}^{(2)},\boldsymbol{b}^{(2)}]}(\Gamma,M)\rightarrow \mathcal{D}_{\boldsymbol{h}}^{[\boldsymbol{a}^{(1)},\boldsymbol{b}^{(1)}]}(\Gamma,M)$, $\mu\mapsto \mu\vert_{C^{[\boldsymbol{a}^{(1)},\boldsymbol{b}^{(1)}]}(\Gamma,\mathcal{O}_{\K})}$ for every $\boldsymbol{a}^{(i)},\boldsymbol{b}^{(i)}\in \mathbb{Z}^{k}$ such that $\boldsymbol{b}^{(i)}\geq \boldsymbol{a}^{(i)}$ and $[\boldsymbol{a}^{(1)},\boldsymbol{b}^{(1)}]\subset [\boldsymbol{a}^{(2)},\boldsymbol{b}^{(2)}]$ with $i=1,2$. Then, we can define a projecitve limit $\mathcal{D}_{\boldsymbol{h}}(\Gamma,M)=\varprojlim_{\substack{\boldsymbol{a},\boldsymbol{b}\in \mathbb{Z}^{k}\\ \boldsymbol{b}\geq \boldsymbol{a}}}\mathcal{D}_{\boldsymbol{h}}^{[\boldsymbol{a},\boldsymbol{b}]}(\Gamma,M)$. Since $\mathbf{I}$ is a finite free $\mathcal{O}_{\K}[[\Gamma]]$-module, we have a natural isomorphism
\begin{equation}\label{not-bounded admissible distributions}
\mathcal{D}_{\boldsymbol{h}}(\Gamma,M)\otimes_{\mathcal{O}_{\K}[[\Gamma]]}\mathbf{I}\simeq \varprojlim_{\substack{\boldsymbol{a},\boldsymbol{b}\in \mathbb{Z}^{k}\\ \boldsymbol{b}\geq \boldsymbol{a}}}\left(\mathcal{D}_{\boldsymbol{h}}^{[\boldsymbol{a},\boldsymbol{b}]}(\Gamma,M)\otimes_{\mathcal{O}_{\K}[[\Gamma]]}\mathbf{I}\right).
\end{equation}
We denote by $p^{[\boldsymbol{a},\boldsymbol{b}]}: \mathcal{D}_{\boldsymbol{h}}(\Gamma,M)\otimes_{\mathcal{O}_{\K}[[\Gamma]]}\mathbf{I}\rightarrow \mathcal{D}_{\boldsymbol{h}}^{[\boldsymbol{a},\boldsymbol{b}]}(\Gamma,M)\otimes_{\mathcal{O}_{\K}[[\Gamma]]}\mathbf{I}$ the projection for each $\boldsymbol{a},\boldsymbol{b}\in \mathbb{Z}^{k}$ such that $\boldsymbol{b}\geq \boldsymbol{a}$. If $\boldsymbol{e}-\boldsymbol{d}\geq \lfloor \boldsymbol{h}\rfloor$, by Proposition \ref{admissible proj is isom not hisom}, the restriction map $\mathcal{D}_{\boldsymbol{h}}^{[\boldsymbol{a},\boldsymbol{b}]}(\Gamma,M)\otimes_{\mathcal{O}_{\K}[[\Gamma]]}\mathbf{I}\rightarrow \mathcal{D}_{\boldsymbol{h}}^{[\boldsymbol{d},\boldsymbol{e}]}(\Gamma,M)\otimes_{\mathcal{O}_{\K}[[\Gamma]]}\mathbf{I}$ is an isomorphism for every $\boldsymbol{a},\boldsymbol{b}\in \mathbb{Z}^{k}$ such that $[\boldsymbol{d},\boldsymbol{e}]\subset [\boldsymbol{a},\boldsymbol{b}]$. Then, if $\boldsymbol{e}-\boldsymbol{d}\geq \lfloor \boldsymbol{h}\rfloor$, we see that $p^{[\boldsymbol{d},\boldsymbol{e}]}$ is an isomorphism.

Let $\mu\in \mathcal{D}_{\boldsymbol{h}}(\Gamma,M)\otimes_{\mathcal{O}_{\K}[[\Gamma]]}\mathbf{I}$. For each $\kappa\in \mathfrak{X}_{\mathbf{I}}$ we can define a specialization of $\mu$ by $\kappa$ to be
\begin{equation}\label{specialization of notbounded Dh}
\kappa(\mu)=\kappa (p^{[\boldsymbol{w}_{\kappa},\boldsymbol{w}_{\kappa}]}(\mu))\in M_{\K_{\kappa}}.
\end{equation}

\section{Applications} \label{sc:application}
In this section, we construct a two-variable $p$-adic Rankin Selberg $L$-series (see Theorem \ref{two variable rankin selberg l series of hida family}) by applying 
the theory developed in this paper. 
For each Dirichlet character $\psi$ modulo $N\in \mathbb{Z}_{\geq 1}$, we denote by $\psi_{0}$ and $c_{\psi}$ the primitive Dirichlet character associated to $\psi$ and the conductor of $\psi$. For a subring $R$ of $\mathbb{C}$, we denote by $M_{2}(R)$ the set of square matrices of order $2$ with coefficients in $R$. We assume that $p\geq 5$ and $\K$ is a finite extension of $\mathbb{Q}_{p}$.
\subsection{Review of modular forms}
In this subsection, we introduce nearly holomorphic modular forms, Rankin-selberg $L$-series and Hida families. Let $N$ be a positive integer and $k$ a non-negative integer. We define a congruence subgroup $\Gamma_{0}(N)$ of $SL_{2}(\mathbb{Z})$ to be
\begin{equation}
\Gamma_{0}(N)=\left\{ \left. \begin{pmatrix}a&b\\ c&d\end{pmatrix}\in SL_{2}(\mathbb{Z}) \ \right\vert \ c\in N\mathbb{Z}\right\}.
\end{equation}
Let $\mathfrak{H}=\{z\in\mathbb{C}\vert y>0\}$ be the upper half plane. We define an action of $GL_{2}^{+}(\mathbb{R})=\{\alpha\in GL_{2}(\mathbb{R})\vert \det \alpha>0\}$ on the space of functions $f:\mathfrak{H}\rightarrow \mathbb{C}$ to be
\begin{equation}
(f\vert_{k}\gamma)(z)=(\det \alpha)^{k\slash 2}(cz+d)^{-k}f(\gamma z),
\end{equation}
where $\gamma z=\frac{az+b}{cz+d}$ with $\gamma=\begin{pmatrix}a&b\\c&d\end{pmatrix}\in GL_{2}^{+}(\mathbb{R})$. Let $\psi$ be a Dirichlet character modulo $N$. We put $\psi(\gamma)=\overline{\psi(a)}$ for each $\gamma=\begin{pmatrix}a&b\\ c&d\end{pmatrix}\in M_{2}(\mathbb{Z})$ with $c\equiv 0\ \mathrm{mod}\ N$, $\mathrm{gcd}(a,N)=1$ and $\det\gamma >0$. We denote by $C_{k}^{\infty}(N,\psi)$ the $\mathbb{C}$-vector space of infinitely differentiable functions $f: \mathfrak{H}\rightarrow \mathbb{C}$ such that $f\vert_{k}\gamma=\psi(\gamma)f$ for each $\gamma\in \Gamma_{0}(N)$. Let $r\in \mathbb{Z}_{\geq 0}$. We denote by $\mathbb{C}[X]_{\leq r}$ the $\mathbb{C}$-vector space of polynomials over $\mathbb{C}$ with degree $\leq r$. We say that a function $f\in C^{\infty}_{k}(N,\psi)$ is a nearly holomorphic modular form of weight $k$, level $N$, character $\psi$ and order $\leq r$ if we have $(f\vert_{k}\gamma)(z)=\sum_{n=0}^{+\infty}a_{n}^{(\gamma)}(f,\frac{-1}{4\pi y})e^{2\pi\sqrt{-1}nz\slash N}$ for every $\gamma\in SL_{2}(\mathbb{Z})$, where $a_{n}^{(\gamma)}(f,X)\in \mathbb{C}[X]_{\leq r}$ with $n\in\mathbb{Z}_{\geq0}$. We denote by $N^{\leq r}_{k}(N,\psi)$ the space of nearly holomorphic modular forms of weight $k$, level $N$, character $\psi$ and order $\leq r$. For each $f\in N_{k}^{\leq r}(N,\psi)$, we write $a_{n}(f,X)=a_{Nn}^{(I_{2})}(f,X)$ with $n\in\mathbb{Z}_{\geq0}$, where $I_{2}=\begin{pmatrix}1&0\\ 0&1\end{pmatrix}$. Then, we have the Fourier expansion $f=\sum_{n=0}^{+\infty}a_{n}\left(f,\frac{-1}{4\pi y}\right)e^{2\pi\sqrt{-1}nz}$. We say that a nearly holomorphic modular form $f\in N_{k}^{\leq r}(N,\psi)$ is cuspidal if $a_{0,f}^{(\gamma)}(X)=0$ for every $\gamma\in SL_{2}(\mathbb{Z})$. We denote by $N_{k}^{\leq r,\mathrm{cusp}}(N,\psi)$ the space of nearly holomorphic cusp forms of weight $k$, level $N$, character $\psi$ and order $\leq r$. We put $M_{k}(N,\psi)=N^{\leq 0}_{k}(N,\psi)$ and $S_{k}(N,\psi)=N_{k}^{\leq 0,\mathrm{cusp}}(N,\psi)$. We call an element $f\in M_{k}(N,\psi)$ (resp$.\ S_{k}(N,\psi)$) a modular form (resp$.\ \mathrm{a\ cusp\ form}$) of weight $k$, level $N$ and character $\psi$. Let $\xi$ be a Dirichlet character modulo $M$, where $M\in \mathbb{Z}_{\geq 1}$. For each $f\in N_{k}^{\leq r}(N,\chi)$, we define the twist $f\otimes \xi$ to be 
\begin{equation}
f\otimes\xi=\sum_{n=0}^{+\infty}a_{n}\left(f,\frac{-1}{4\pi y}\right)\xi(n)e^{2\pi\sqrt{-1}nz}\in N_{k}^{\leq r}(L,\chi\xi^{2}),\end{equation}
where $L$ is the least common multiple of $N$ and $M^{2}$. For each $f\in S_{k}(N,\psi)$, we denote by 
\begin{equation}\label{definition of frho}
f^{\rho}=\sum_{n=1}^{+\infty}\overline{a_{n}(f)}e^{2\pi\sqrt{-1}nz}\in S_{k}(N,\overline{\psi}).
\end{equation}
 Let $f,g\in N_{k}^{\leq r}(N,\psi)$ such that $fg\in N_{2k}^{\leq 2r,\mathrm{cusp}}(N,\psi^{2})$. We define the Petersson inner product $\langle f,g\rangle_{k,N}$ to be
\begin{equation}
\langle f,g\rangle_{k,N}=\int_{\Gamma_{0}(N)\backslash \mathfrak{H}}\overline{f}gy^{k -2}{dxdy}.
\end{equation} 
For each integer $k$ and for each non-negative integer $r$, we define the differential operators $\delta_{k}$, $\delta_{k}^{(r)}$ and $\epsilon$ by
\begin{align}\label{shimura operator}
\begin{split}
&\delta_{k}=\frac{1}{2\pi\sqrt{-1}}\left(\frac{k}{2\sqrt{-1}y}+\frac{\partial}{\partial z}\right), \ \delta_{k}^{(r)}=\delta_{k+2r-2}\cdots\delta_{k+2}\delta_{k},\\
&\epsilon=(-8\sqrt{-1}\pi)y^{-2}\frac{\partial}{\partial\overline{z}}
\end{split}
\end{align}
where $\frac{\partial}{\partial z}=\frac{1}{2}\left(\frac{\partial}{\partial x}-\sqrt{-1}\frac{\partial}{\partial y}\right)$ and $\frac{\partial}{\partial\overline{z}}=\frac{1}{2}\left(\frac{\partial}{\partial x}+\sqrt{-1}\frac{\partial}{\partial y}\right)$.
We remark that we understand that $\delta_{k}^{(0)}=1$ is the identity operator. By \cite[p35]{Shimura2012}, we have
\begin{equation}\label{shimura operators and vertkgamma}
\delta_{k}(f\vert_{k}\gamma)=\delta_{k}(f)\vert_{k+2}\gamma,\ \epsilon(f\vert_{k}\gamma)=\epsilon(f)\vert_{k-2}\gamma
\end{equation}
for each $\gamma\in GL_{2}^{+}(\mathbb{R})$. By \eqref{shimura operators and vertkgamma}, we see that $\delta_{k}(f)\in N_{k+2}^{\leq r+1}(N,\psi)$ (resp. $N_{k+2}^{\leq r+1,\mathrm{cusp}}(N,\psi)$) and $\epsilon(f)\in N_{k-2}^{\leq r-1}(N,\psi)$ (resp. $N_{k-2}^{\leq r-1,\mathrm{cusp}}(N,\psi)$) for each $f\in N_{k}^{\leq r}(N,\psi)$ (resp. $N_{k}^{\leq r,\mathrm{cusp}}(\linebreak N,\psi)$) where $N_{k-2}^{\leq -1}(N,\psi)=N_{k-2}^{\leq -1,\mathrm{cusp}}(N,\psi)=0$. We prove a lemma.
\begin{lem}\label{nearlyholormorphic infty easy lemma}
{Let $f\in N_{k}^{\leq r}(N,\psi)$ where $k,r\in \mathbb{Z}_{\geq 0}$, and let $\psi$ be a Dirichlet character modulo $N$ with $N\in \mathbb{Z}_{\geq 1}$. 
Let $m$ be a non-negative integer satisfying $m\leq r$. If we have $a_{n}(f,X)\in \mathbb{C}[X]_{\leq m}$ for every $n\in \mathbb{Z}_{\geq 0}$, we have $f\in N_{k}^{\leq m}(N,\psi)$.}
\end{lem}
\begin{proof}
{By a simple calculation, we see that 
$$
\epsilon((-4\pi y)^{-n})=n(-4\pi y)^{-(n-1)}$$
for each $n\in \mathbb{Z}_{\geq 0}$. Hence, for each $a(X)\in \mathbb{C}[X]$, we see that 
\begin{equation}\label{eepsilon poyleq mtreq}
\epsilon^{m+1}\left(a\left(\frac{-1}{4\pi y}\right)\right)=0\ \mathrm{if\ and \ only\ if}\ a(X)\in \mathbb{C}[X]_{\leq m}.
\end{equation} 
For each $\gamma\in SL_{2}(\mathbb{Z})$, we have
\begin{align}\label{eepsilon poyleq mtreq2}
\begin{split}
\epsilon^{m+1}(f\vert_{k}\gamma)&=\sum_{n=0}^{+\infty}\epsilon^{m+1}\left(a_{n}^{(\gamma)}\left(f,\frac{-1}{4\pi y}\right)e^{2\pi\sqrt{-1}nz\slash N}\right)\\
&=\sum_{n=0}^{+\infty}\epsilon^{m+1}\left(a_{n}^{(\gamma)}\left(f,\frac{-1}{4\pi y}\right)\right)e^{2\pi\sqrt{-1}nz\slash N}.
\end{split}
\end{align}
Since we have $a_{n}(f,X)\in \mathbb{C}[X]_{\leq m}$ for every $n\in \mathbb{Z}_{\geq 0}$ by \eqref{eepsilon poyleq mtreq} and \eqref{eepsilon poyleq mtreq2}, we have $\epsilon^{m+1}(f)=0$. Let $\gamma\in SL_{2}(\mathbb{Z})$. By \eqref{shimura operators and vertkgamma}, we have $\epsilon^{m+1} (f\vert_{k}\gamma)=\epsilon^{m+1}(f)\vert_{k-2(m+1)}\gamma=0$ for each element $\gamma\in SL_{2}(\mathbb{Z})$. By \eqref{eepsilon poyleq mtreq} and \eqref{eepsilon poyleq mtreq2}, $\epsilon^{m+1} (f\vert_{k}\gamma)=0$ implies that $a_{n}^{(\gamma)}(f,X)\in \mathbb{C}[X]_{\leq m}$ for every $n\in \mathbb{Z}_{\geq 0}$. Therefore, we see that $f\in N_{k}^{\leq m}(N,\psi)$. }
\end{proof}

By \cite[Lemma 7]{Shimura1976}, we have the following:
\begin{pro}\label{prop:nearlyholomorphic}
We assume that $k>2r$. Then, each $f\in N_{k}^{\leq r}(N,\psi)$ has an expression
$$f=\sum_{j=0}^{r}\delta_{k-2j}^{(j)}(f_{j})$$
with $f_{j}\in M_{k-2j}(N,\psi)$ which are uniquelly determined by $f$. Moreover, if $f\in N_{k}^{\leq r,\mathrm{cusp}}(N,\psi)$, $f_{j}\in S_{k-2j}(N,\psi)$ for every $j$ satisfying $0\leq j\leq r$.
\end{pro}
For each $f=\sum_{j=0}^{r}\delta_{k-2j}^{(j)}(f_{j})\in N_{k}^{\leq r}(N,\psi)$ with $f_{j}\in M_{k-2j}(N,\psi)$, we call $f_{0}$ a holomorphic projection of $f$.

Let $l$ be a prime and  $\{\alpha_{1},\ldots, \alpha_{v}\}$ a subset of $\Gamma_{0}(N)\begin{pmatrix}1&0\\ 0&l\end{pmatrix}\Gamma_{0}(N)$ which is a complete representative set for $\Gamma_{0}(N)\backslash \Gamma_{0}(N)\begin{pmatrix}1&0\\ 0&l\end{pmatrix}\Gamma_{0}(N)$. We define the $l$-th Hecke operator $T_{l}: N_{k}^{\leq r}(N,\psi)\rightarrow N_{k}^{\leq r}(N,\psi)$ to be 
\begin{equation}\label{Hecke operator at l}
T_{l}(f)=\det (\alpha)^{\frac{k}{2}-1}\sum_{i=1}^{v}\overline{\psi(\alpha_{i})}f\vert_{k}\alpha_{i}\end{equation}
for each $f\in N_{k}^{\leq r}(N,\psi)$. It is known that $T_{l}(f)=\sum_{n=0}^{+\infty}a_{ln,f}(\frac{-l}{4\pi y})e^{2\pi\sqrt{-1}nz}$ for each prime $l$ such that $l\vert N$. If  a prime $l$ satisfies $l\vert N$, we have $\Gamma_{0}(Nl)\begin{pmatrix}1&0\\ 0&l\end{pmatrix}\Gamma_{0}(Nl)=\Gamma_{0}(Nl)\begin{pmatrix}1&0\\0&l\end{pmatrix}\Gamma_{0}(N)$. Then, we see that $T_{l}$ induces the following homomorphism: 
\begin{equation}
T_{l}:N_{k}^{\leq r}(Nl,\psi)\rightarrow N_{k}^{\leq r}(N,\psi)
\end{equation}
for each prime $l$ such that $l\vert N$. We have $\Gamma_{0}(N)\begin{pmatrix}l&0\\0&1\end{pmatrix}\Gamma_{0}(Nl)=\Gamma_{0}(N)\begin{pmatrix}l&0\\0&1\end{pmatrix}$ for each prime $l$ such that $l\vert N$. Then, by \cite[(3.4.5)]{Shimura1971}, we have
\begin{equation}\label{relaTl anddiagl1}
\langle f,T_{l}(g)\rangle_{k,N}=l^{\frac{k}{2}-1}\left\langle f\vert_{k}\begin{pmatrix}l&0\\0&1\end{pmatrix},g\right\rangle_{k,Nl}
\end{equation}
for each prime $l$ such that $l\vert N$ and each $f\in N_{k}^{\leq r}(N,\psi)$ and $g\in N_{k}^{\leq r}(Nl,\psi)$ such that $fg\in N_{k}^{\leq 2r,\mathrm{cusp}}(Nl,\psi^{2})$. Let $L$ be a positive integer such that $N\vert L$. We define a trace operator 
\begin{equation}\label{definition of trace operator}
\mathrm{Tr}_{L\slash N}: N_{k}^{\leq r}(L,\psi)\rightarrow N_{k}^{\leq r}(N,\psi)
\end{equation} 
to be $\mathrm{Tr}_{L\slash N}(f)=(L\slash N)^{k\slash 2-1}\sum_{\gamma\in R}\overline{\psi}(\gamma)f\vert_{k}\gamma$ for each $f\in N_{k}^{\leq r}(L,\psi)$, where $R$ is a complete representative set for $\Gamma_{0}(L)\backslash \Gamma_{0}(L)\begin{pmatrix}1&0\\0&L\slash N\end{pmatrix}\Gamma_{0}(N)$. By \cite[(3.4.5)]{Shimura1971}, we see that 
\begin{equation}\label{trace operator}
\langle f,\mathrm{Tr}_{L\slash N}(g)\rangle_{k,N}=(L\slash N)^{\frac{k}{2}-1}\bigg\langle f\vert_{k}\begin{pmatrix}L\slash N&0\\0&1\end{pmatrix},g\bigg\rangle_{k,L}
\end{equation}
for each $f\in N_{k}^{\leq r}(N,\psi)$ and $g\in N_{k}^{\leq r}(L,\psi)$ such that $fg\in N_{2k}^{\leq 2r,\mathrm{cusp}}(L,\psi^{2})$. Let $A$ be a subring of $\mathbb{C}$. We define $A$-modules
\begin{align*}
N_{k}^{\leq r}(N,\psi;A)&=\{f\in N_{k}^{\leq r}(N,\psi)\ \vert \ a_{n}(f,X)\in A[X]\ \text{ for any } n\in \mathbb{Z}_{\geq 0}\},\\
N_{k}^{\leq r,\mathrm{cusp}}(N,\psi;A)&=\{f\in N_{k}^{\leq r,\mathrm{cusp}}(N,\psi)\ \vert \ a_{n}(f,X)\in A[X]\  \text{ for any }  n\in \mathbb{Z}_{\geq 1}\}.
\end{align*}
When $K$ is a subfield of $\overline{\mathbb{Q}}$, we put
\begin{align*}
N_{k}^{\leq r}(N,\psi;\K)&=N_{k}^{\leq r}(N,\psi;K)\otimes_{K}\K,\\
N_{k}^{\leq r}(N,\psi;\mathcal{O}_{\K})&=N_{k}^{\leq r}(N,\psi;\mathcal{O}_{K})\otimes_{\mathcal{O}_{K}}\mathcal{O}_{\K},\\
N_{k}^{\leq r,\mathrm{cusp}}(N,\psi;\K)&=N_{k}^{\leq r,\mathrm{cusp}}(N,\psi;K)\otimes_{K}\K,\\
N_{k}^{\leq r,\mathrm{cusp}}(N,\psi;\mathcal{O}_{\K})&=N_{k}^{\leq r,\mathrm{cusp}}(N,\psi;\mathcal{O}_{K})\otimes_{\mathcal{O}_{K}}\mathcal{O}_{\K},
\end{align*}
where $\mathcal{O}_{K}$ is the ring of integers of $K$ and $\K$ is the completion of $K$ in $\mathbb{C}_{p}$. We can regard $N_{k}^{\leq r}(N,\psi;\K)$ as a $\K$-Banach space by the valuation $v_{N_{k}^{\leq r}(N,\psi)}$ defined by
\begin{equation}\label{valuation of nearly holomorphic modular forms}
v_{N_{k}^{\leq r}(N,\psi)}(f)=\inf_{n\geq 0}\{v_{0}(a_{n}(f,X))\}
\end{equation}
for each $f\in N_{k}^{\leq r}(N,\psi;\K)$, where $v_{0}$ is the valuation on $\mathcal{O}_{\K}[[X]]\otimes_{\mathcal{O}_{\K}}\K$ defined by $v_{0}(\sum_{n=0}^{+\infty}a_{n}X^{n})=\inf_{n\in \mathbb{Z}_{\geq 0}}\{\ord_{p}(a_{n})\}$ with $a_{n}\in \K$. We see that $N_{k}^{\leq r,\mathrm{cusp}}(N,\psi;\K)$ is a $\K$-Banach subspace of $N_{k}^{\leq r}(N,\psi;\K)$.

Let $f\in S_{k}(N,\psi)$ be a normalized cuspidal Hecke eigenform. We denote by $c_{f}$ and $f^{0}$ the conductor of $f$ and the primitive form associated with $f$ respectively. For each $M\in \mathbb{Z}_{\geq 1}$, we put 
\begin{equation}\label{definition of tauM}
\tau_{M}=\begin{pmatrix}0&-1\\ M&0\end{pmatrix}.\end{equation}
\begin{pro}\label{peterson quatient algebraic}
Let $K$ be a finite extension of $\mathbb{Q}$. Assume that $(p,N)=1$. Let $f\in S_{k}(Np^{m(f)},\psi;K)$ be a normalized cuspidal Hecke eigenform which is new away from $p$ with $m(f)\in \mathbb{Z}_{\geq 1}$. Here $\psi$ is a Dirichlet character modulo $Np^{m(f)}$. Assume that $a_{p}(f)\neq 0$, $f^{0}\in S_{k}(c_{f},\psi;K)$ and $m(f)$ is the smallest positive integer $m$ such that $f\in S_{k}(Np^{m},\psi)$. Further, if $f$ is not a primitive form, we assume that $a_{p}(f)^{2}\neq \psi_{0}(p)p^{k-1}$ where $\psi_{0}$ is the primitive character attached to $\psi$. Then, for each $g\in S_{k}(Np^{m(f)},\psi;K)$, we have 
$\frac{\langle f^{\rho}\vert_{k}\tau_{Np^{m(f)}},g\rangle_{k,Np^{m(f)}}}{\langle f^{\rho}\vert_{k}\tau_{Np^{m(f)}},f\rangle_{k,Np^{m(f)}}}\in K$, where $f^{\rho}$ is the cusp form defined in \eqref{definition of frho}.
\end{pro}
\begin{proof}
It suffices to prove that 
$$\frac{\langle f^{\rho}\vert_{k}\tau_{Np^{m(f)}},g^{\sigma}\rangle_{k,Np^{m(f)}}}{\langle f^{\rho}\vert_{k}\tau_{Np^{m(f)}},f\rangle_{k,Np^{m(f)}}}=\sigma\left(\frac{\langle f^{\rho}\vert_{k}\tau_{Np^{m(f)}},g\rangle_{k,Np^{m(f)}}}{\langle f^{\rho}\vert_{k}\tau_{Np^{m(f)}},f\rangle_{k,Np^{m(f)}}}\right)$$
for any $g\in S_{k}(Np^{m(f)},\psi)$ and for any $\sigma\in \mathrm{Aut}(\mathbb{C}\slash K)$ where $g^{\sigma}=\sum_{n=1}^{+\infty}\sigma(a_{n}(g))e^{2\pi\sqrt{-1}nz}$.
\par 
Let $P$ be the set of primitive forms $h\in S_{k}(c_{h},\psi)$ such that $c_{h}\vert Np^{m(f)}$. For each $h\in P$, we define a $\mathbb{C}$-vector space 
$U(h,Np^{m(f)})$ by 
$$U(h,Np^{m(f)})=\{g\in S_{k}(Np^{m(f)},\psi)\ \vert \ T_{l}(g)=a_{l}(h)g\ \mathrm{except\ for\ finitely\ many\ primes}\ l\}.$$
Then, it is well-known that we have the following orthonormal decomposition with respect to the Petersson inner product: 
$$S_{k}(Np^{m(f)},\psi)=\oplus_{h\in P}U(h,Np^{m(f)})$$
and the space $U(h,Np^{m(f)})$ is spanned by $\{h(tz)\}_{t\vert\frac{Np^{m(f)}}{c_{h}}}$ for each $h\in P$ (see \cite[Lemma 4.6.9]{Miyake89}). For each $\sigma\in \mathrm{Aut}(\mathbb{C}\slash K)$, 
we can define a bijection $P\stackrel{\sim}{\rightarrow}P$ to be $h\mapsto h^{\sigma}$ and we have a $\mathbb{C}$-linear isomorphism $U(h,Np^{m(f)})\stackrel{\sim}{\rightarrow}U(h^{\sigma},Np^{m(f)})$, $g\mapsto g^{\sigma}$ for each $h\in P$. Then, $U(f^{0},Np^{m(f)})$ and $\oplus_{h\in P\backslash \{f^{0}\}}U(h,Np^{m(f)})$ are stable under the action of $\mathrm{Aut}(\mathbb{C}\slash K)$. We remark that $f^{\rho}\vert_{k}\tau_{Np^{m(f)}}\in U(f^{0},Np^{m(f)})$. Thus, it suffices to prove that $\frac{\langle f^{\rho}\vert_{k}\tau_{Np^{m(f)}},g^{\sigma}\rangle_{k,Np^{m(f)}}}{\langle f^{\rho}\vert_{k}\tau_{Np^{m(f)}},f\rangle_{k,Np^{m(f)}}}=\sigma\left(\frac{\langle f^{\rho}\vert_{k}\tau_{Np^{m(f)}},g\rangle_{k,Np^{m(f)}}}{\langle f^{\rho}\vert_{k}\tau_{Np^{m(f)}},f\rangle_{k,Np^{m(f)}}}\right)$ for any $g\in U(f^{0},Np^{m(f)})$ and for any $\sigma\in \mathrm{Aut}(\mathbb{C}\slash K)$.

If $f$ is primitive, since $f$ is a basis of $U(f^{0},Np^{m(f)})$, we have $g^{\sigma}=\sigma(a_{1}(g))f^{\sigma}=\sigma(a_{1}(g))f$ and $\frac{\langle f^{\rho}\vert_{k}\tau_{Np^{m(f)}},g^{\sigma}\rangle_{k,Np^{m(f)}}}{\langle f^{\rho}\vert_{k}\tau_{Np^{m(f)}},f\rangle_{k,Np^{m(f)}}}=\sigma(a_{1}(g))=\sigma\left(\frac{\langle f^{\rho}\vert_{k}\tau_{Np^{m(f)}},g\rangle_{k,Np^{m(f)}}}{\langle f^{\rho}\vert_{k}\tau_{Np^{m(f)}},f\rangle_{k,Np^{m(f)}}}\right)$ for any $g\in U(f^{0},Np^{m(f)})$ and for any $\sigma\in \mathrm{Aut}(\mathbb{C}\slash K)$. In the rest of the proof, we assume that $f\neq f^{0}$. We note that $m(f)=1$ and $c_{f}=N$. There exists a unique element $\alpha\in K$ such that $f=f^{0}-\alpha f^{0}(pz)$. Let $T_{p}$ be the $p$-the Hecke operator on $S_{k}(Np,\psi)$. It is well-kwnon that $T_{p}(f^{0})=a_{p}(f^{0})f^{0}-\psi_{0}(p)p^{k-1}f^{0}(pz)$ and $T_{p}(f^{0}(pz))=f^{0}$. Therefore, we see that $a_{p}(f)$ and $\alpha$ are roots of the Hecke polynomial $X^{2}-a_{p}(f^{0})X+\psi_{0}(p)p^{k-1}$ where $\psi_{0}$ is the primitive Dirichlet character associated with $\psi$. Since $a_{p}(f)^{2}\neq \psi_{0}(p)p^{k-1}$, we see that $\alpha\neq a_{p}(f)$. We put $f_{1}=f^{0}-a_{p}(f)f^{0}(pz)\in U(f^{0},Np)$. Then, $T_{p}(f_{1})=\alpha f_{1}$ and $\{f,f_{1}\}$ is a basis of $U(f^{0},Np)$. Let $T_{p}^{*}$ be the adjoint operator of $T_{p}$ with respect to the Petersson inner product. Then, by \cite[Theorem 4.5.5]{Miyake89}, we see that
\begin{align*}
\alpha\langle f^{\rho}\vert_{k}\tau_{Np},f_{1}\rangle_{k,Np}&=\langle f^{\rho}\vert_{k}\tau_{Np},T_{p}(f_{1})\rangle_{k,Np}\\
&=\langle T_{p}^{*}(f^{\rho}\vert_{k}\tau_{Np}),f_{1}\rangle_{k,Np}\\
&=a_{p}(f)\langle f^{\rho}\vert_{k}\tau_{Np},f_{1}\rangle_{k,Np}.
\end{align*}
Therefore, we have $\langle f^{\rho}\vert_{k}\tau_{Np},f_{1}\rangle_{k,Np}=0$. Let $g\in U(f^{0},Np)$. There exites a unique pair $(a,b)\in \mathbb{C}^{2}$ such that $g=af+bf_{1}$. Since $f$ and $f_{1}$ are in $S_{k}(Np,\psi;K)$, we hava $g^{\sigma}=\sigma(a)f+\sigma(b)f_{1}$ and $\frac{\langle f^{\rho}\vert_{k}\tau_{Np},g^{\sigma}\rangle_{k,Np}}{\langle f^{\rho}\vert_{k}\tau_{Np},f\rangle_{k,Np}}=\sigma(a)=\sigma\left(\frac{\langle f^{\rho}\vert_{k}\tau_{Np},g\rangle_{k,Np}}{\langle f^{\rho}\vert_{k}\tau_{Np},f\rangle_{k,Np}}\right)$ for any $\sigma\in \mathrm{Aut}(\mathbb{C}\slash K)$.
\end{proof}

Assume that $(p,N)=1$. Let $K$ be a finite extension of $\mathbb{Q}$. Let $f\in S_{k}(Np^{m(f)},\psi;K)$ be a normalized cuspidal Hecke eigenform which is new away from $p$ with $m(f)\in \mathbb{Z}_{\geq 1}$. Here $\psi$ is a Dirichlet character modulo $Np^{m(f)}$. Assume that $a_{p}(f)\neq 0$, $f^{0}\in S_{k}(c_{f},\psi;K)$ and $m(f)$ is the smallest positive integer $m$ such that $f\in S_{k}(Np^{m},\psi)$. Further, if $f$ is not a primitive form, we assume that $a_{p}(f)^{2}\neq \psi_{0}(p)p^{k-1}$. We denote by $\K$ the completion of $K$ in $\mathbb{C}_{p}$. Let $M$ be a positive integer such that $(M,p)=1$ and $N\vert M$. We assume that $K$ contains a primitive $M$-th root of unity. Then, it is known that we have $\mathrm{Tr}_{Mp^{m(f)}\slash Np^{m(f)}}(S_{k}(Mp^{m(f)},\psi;K))\subset S_{k}(Np^{m(f)},\psi;K)$ where $\mathrm{Tr}_{Mp^{m(f)}\slash Np^{m(f)}}$ is the trace map defined in \eqref{definition of trace operator}. Further, it is known that the holomorphic projection of an element in $N_{k}^{\leq r}(Np^{m(f)},\psi;K)$ with $k>2r$ is contained in $M_{k}(Np^{m(f)},\psi;K)$. Then, for each positive integer $m$ such that $m\geq m(f)$ and each non-negative ineger $r$ satisfying $k>2r$, by Proposition \ref{peterson quatient algebraic}, there exists a unique $\K$-linear map 
\begin{equation}\label{definition of lf,L(n)}
l_{f,M}^{(m)}: N_{k}^{\leq r,\mathrm{cusp}}(Mp^{m},\psi;\K)\rightarrow \K
\end{equation}
such that $l_{f,M}^{(m)}(g)=a_{p}(f)^{-(m-m(f))}\frac{\langle f^{\rho}\vert_{k}\tau_{Np^{m(f)}},\mathrm{Tr}_{Mp^{m(f)}\slash Np^{m(f)}}(T_{p}^{m-m(f)}(g)_{0})\rangle_{k,Np^{m(f)}}}{\langle f^{\rho}\vert_{k}\tau_{Np^{m(f)}},f\rangle_{k,Np^{m(f)}}}$ for each $g\in N_{k}^{\leq r,\mathrm{cusp}}(Mp^{m},\psi;K)$ with $n\in \mathbb{Z}_{\geq 1}$, where $T_{p}^{m-m(f)}(g)_{0}\in S_{k}(Mp^{m(f)},\psi,K)$ is the holomorphic projection of $T_{p}^{m-m(f)}(g)$. Let $i_{m}: N_{k}^{\leq r,\mathrm{cusp}}(Mp^{m},\psi;\K)\rightarrow N_{k}^{\leq r,\mathrm{cusp}}(Mp^{m+1},\linebreak\psi;\K)$ be the natural inclusion map for each positive integer $m$ such that $m\geq m(f)$. We prove that 
\begin{equation}\label{lfL(n)is compatible with inlusion}
l_{f,M}^{(m+1)}i_{m}=l_{f,M}^{(m)}.
\end{equation}
For each positive integer $m$ such that $m\geq m(f)$ and $g\in N_{k}^{\leq r,\mathrm{cusp}}(Mp^{m},\psi;\K)$, by \eqref{trace operator}, we see that 
\begin{multline}\label{hecke opeator trace Mp}
\langle f^{\rho}\vert_{k}\tau_{Np^{m(f)}},\mathrm{Tr}_{Mp^{m(f)}\slash Np^{m(f)}}(T_{p}^{m-m(f)}(g)_{0})\rangle_{k,Np^{m(f)}}\\
=(M\slash N)^{\frac{k}{2}-1}\langle f^{\rho}\vert_{k}\tau_{Mp^{m(f)}},T_{p}^{m-m(f)}(g)_{0}\rangle_{k,Mp^{m(f)}}
\end{multline}
and $T_{p}^{m+1-m(f)}\iota_{m}(g)=T_{p}^{m+1-m(f)}(g)$ in $N_{k}^{\leq r,\mathrm{cusp}}(Mp^{m(f)},\psi;\K)$. By \cite[Theorem 4.5.5]{Miyake89}, we have
\begin{align*}
\langle f^{\rho}\vert_{k}\tau_{Mp^{m(f)}},T_{p}^{m+1-m(f)}(i_{m}(g))_{0}\rangle_{k,Mp^{m(f)}}&=\langle f^{\rho}\vert_{k}\tau_{Mp^{m(f)}},T_{p}^{m+1-m(f)}(g)_{0}\rangle_{k,Mp^{m(f)}}\\
&=\langle T_{p}(f^{\rho})\vert_{k}\tau_{Mp^{m(f)}},T_{p}^{m-m(f)}(g)_{0}\rangle_{k,Mp^{m(f)}}\\
&=a_{p}(f)\langle f^{\rho}\vert_{k}\tau_{Mp^{m(f)}},T_{p}^{m-m(f)}(g)_{0}\rangle_{k,Mp^{m(f)}}.
\end{align*}
and
\begin{align*}
&l_{f,M}^{(m+1)}(i_{m}(g))\\
&=a_{p}(f)^{-(m+1-m(f))}\frac{\langle f^{\rho}\vert_{k}\tau_{Np^{m(f)}},\mathrm{Tr}_{Mp^{m(f)}\slash Np^{m(f)}}(T_{p}^{m+1-m(f)}(i_{m}(g))_{0})\rangle_{k,Np^{m(f)}}}{\langle f^{\rho}\vert_{k}\tau_{Np^{m(f)}},f\rangle_{k,Np^{m(f)}}}\\
&=a_{p}(f)^{-(m-m(f))}(M\slash N)^{\frac{k}{2}-1}\frac{\langle f^{\rho}\vert_{k}\tau_{Mp^{m(f)}},T_{p}^{m-m(f)}(g)_{0}\rangle_{k,Mp^{m(f)}}}{\langle f^{\rho}\vert_{k}\tau_{Np^{m(f)}},f\rangle_{k,Np^{m(f)}}}\\
&=a_{p}(f)^{-(m-m(f))}\frac{\langle f^{\rho}\vert_{k}\tau_{Np^{m(f)}},\mathrm{Tr}_{Mp^{m(f)}\slash Np^{m(f)}}(T_{p}^{m-m(f)}(g)_{0})\rangle_{k,Np^{m(f)}}}{\langle f^{\rho}\vert_{k}\tau_{Np^{m(f)}},f\rangle_{k,Np^{m(f)}}}\\
&=l_{f,M}^{(m)}(g).
\end{align*}
for each $g\in N_{k}^{\leq r,\mathrm{cusp}}(Mp^{m},\psi;\K)$ with $m\in \mathbb{Z}_{\geq 1}$ such that $m\geq m(f)$. By \eqref{lfL(n)is compatible with inlusion}, there exists a unique $\K$-linear homomorphism 
\begin{equation}\label{classical lf map}
l_{f,M}: \cup_{m=m(f)}^{+\infty}N_{k}^{\leq r,\mathrm{cusp}}(Mp^{m},\psi;\K)\rightarrow \K
\end{equation}
which satisfies $l_{f,M}(g)=l_{f,M}^{(m)}(g)$ for every $g\in N_{k}^{\leq r,\mathrm{cusp}}(Mp^{m},\psi;\K)$ and $m\in \mathbb{Z}_{\geq 1}$ such that $m\geq m(f)$.

Next, we introduce the Rankin-Selberg $L$-series. As a refference, see \cite{Shimura1976}. Let $k,l$ be non-negative integers such that $k\geq l$. Let $N\in \mathbb{Z}_{\geq 1}$ and $\psi,\xi$ Dirichlet characters modulo $N$. For a couple $(f,g)\in S_{k}(N,\psi)\times M_{l}(N,\xi)$, we define the Rankin-Selberg $L$-series to be
\begin{equation}\label{Rankin-Selberg L-series}
D(s,f,g)=\sum_{n=1}^{+\infty}a_{n}(f)a_{n}(g)n^{-s}.
\end{equation}
The Dirichlet series \eqref{Rankin-Selberg L-series} is absolutely convergent for $\mathrm{Re}(s)>\frac{k+1}{2}+l$. Further, if 
$g$ is in $S_{l}(N,\xi)$, the series \eqref{Rankin-Selberg L-series} is absolutely convergent for $\mathrm{Re}(s)>\frac{k+l}{2}$. 
We set 
\begin{align}\label{Lplus Drankinselberg}
\begin{split}
\mathscr{D}_{N}(s,f,g)&=L_{N}(2s+2-k-l,\psi\xi)D\left(s,f,g\right),\\
\Lambda_{N}(s,f,g)&=\Gamma_{\mathbb{C}}\left(s-l+1\right)\Gamma_{\mathbb{C}}\left(s\right)\mathscr{D}_{N}\left(s,f,g\right), 
\end{split}
\end{align}
where $L_{N}(s,\psi)=\sum_{n=1}^{+\infty}\psi(n)n^{-s}$ and $\Gamma_{\mathbb{C}}(s)=2(2\pi)^{-s}\Gamma(s)$. It is well-known that $\mathscr{D}_{N}(s,f,g)$ has a meromorphic continuation for all $s\in \mathbb{C}$. If $k>l$, $\mathscr{D}_{N}(s,f,g)$ is holomorphic on the whole $\mathbb{C}$-plane. If $k=l$, we have
\begin{equation}\label{residue of rankin selberg}
\mathrm{Res}_{s=k}D(s,f^{\rho},g)=(4\pi)^{k}\Gamma(k)^{-1}\mathrm{Vol}(\Gamma_{0}(N)\backslash \mathfrak{H})^{-1}\langle f,g\rangle_{k,N},
\end{equation}
where $\mathrm{Vol}(\Gamma_{0}(N)\backslash \mathfrak{H})$ is the volume of $\Gamma_{0}(N)\backslash \mathfrak{H}$ with respect to 
the measure $\frac{dxdy}{y^{2}}$ (see \cite[(2.5)]{Shimura1976}).  Assume that $f$ and $g$ are cuspidal normalized Hecke eigenforms and denote by $f^{0}$ and $g^{0}$ the primitive forms associated with $f$ and $g$ respectively. We set
\begin{equation}\label{rankin attached primitives}
\Lambda(s,f,g)=\Lambda_{M}(s,f^{0},g^{0})
\end{equation}
where $M$ is the least common multiple of the conductor of $f$ and the conductor of $g$.

Let $r$ be a non-negative integer. We denote by
\begin{equation}\label{definition of iota}
\iota: N_{k}^{\leq r}(N,\psi)\rightarrow \mathbb{C}[[q]]
\end{equation}
the composition of the map $N_{k}^{\leq r}(N,\psi)\rightarrow \mathbb{C}[X][[q]]$ defined by $f\mapsto \sum_{n=0}^{+\infty}a_{n}(f,X)q^{n}$ and the map $\mathbb{C}[X][[q]]\rightarrow \mathbb{C}[[q]]$ defined by $\sum_{n=0}^{+\infty}a_{n}(X)q^{n}\mapsto \sum_{n=0}^{+\infty}a_{n}(0)q^{n}$ with $a_{n}(X)\in \mathbb{C}[X]$. We define $d: \mathbb{C}[[q]]\rightarrow \mathbb{C}[[q]]$ by
$d=q\frac{d}{dq}$
and we define $T_{l}:\mathbb{C}[[q]]\rightarrow \mathbb{C}[[q]]$ by $T_{l}\left(\sum_{n=0}^{+\infty}a_{n}q^{n}\right)=\sum_{n=0}^{+\infty}a_{ln}q^{n}$ for each prime $l$ with $l\vert N$. Then, we have the following commutative diagrams:
\begin{equation}\label{commutative delta and d and Tp and Up}
\xymatrix{
N_{k}^{\leq r}(N,\psi)\ar[d]^{\delta_{k}}\ar[r]^{\iota}&\mathbb{C}[[q]]\ar[d]^{d}\\
N_{k+2}^{\leq r+1}(N,\psi)\ar[r]^{\iota}&\mathbb{C}[[q]],}\ \ \xymatrix{
N_{k}^{\leq r}(N,\psi)\ar[d]^{T_{l}}\ar[r]^{\iota}&\mathbb{C}[[q]]\ar[d]^{T_{l}}\\
N_{k}^{\leq r}(N,\psi)\ar[r]^{\iota}&\mathbb{C}[[q]].}
\end{equation}
The following proposition is a consequence of \cite[Proposition 3.2.4]{Urban2014} proved by Urban. 
In \cite[Proposition 3.2.4]{Urban2014}, Urban proves that a map from the space of overconvergent nearly holomorphic modular forms to the space of $p$-adic modular forms is injective 
using the theory of $p$-adic modular forms and the technique of algebraic geometry. The following proposition is obtained as a corollary 
of \cite[Proposition 3.2.4]{Urban2014} by restricting this injective map to the space of classical nearly holomorphic modular forms. 
Below, we prove the following proposition in a much more elementary manner by using the theory of Rankin-Selberg $L$-series.
\begin{pro}\label{classical ver of urban}
The map $\iota: N_{k}^{\leq r}(N,\psi)\rightarrow \mathbb{C}[[q]]$ defined in \eqref{definition of iota} is injective.  
\end{pro}
\begin{proof}
If $r=0$, it is clear that $\iota$ is injective since $M_{k}(N,\psi)=N^{\leq 0}_{k}(N,\psi)$. From now on, we assume that $r\geq 1$. By induction on $r$, we assume that the map $\iota: N_{k}^{\leq r^{\prime}}(N,\psi)\rightarrow \mathbb{C}[[q]]$ is injective for each $0\leq r^{\prime}\leq r-1$ and $k\in \mathbb{Z}$. For each non-zero cusp form $h\in S_{m}(SL_{2}(\mathbb{Z}))\backslash \{0\}$ of level $1$ with $m>2r-k$, we have the following commutative diagram:
\begin{align*}
\xymatrix{
N_{k}^{\leq r}(N,\psi)\ar@{^{(}-_{>}}[d]^{\times h}\ar[r]^{\iota}&\mathbb{C}[[q]]\ar@{^{(}-_{>}}[d]^{\times \iota (h)}\\
N_{k+m}^{\leq r,\mathrm{cusp}}(N,\psi)\ar[r]^{\iota}&\mathbb{C}[[q]],
}
\end{align*}
where the vertical maps are defined by the multiplication by $h$ and $\iota(h)$ respectively. 
Since $\mathbb{C}[[q]]$ is an integral domain, 
the right vertical map of the diagram is injective. 
Let $f\in N_{k}^{\leq r}(N,\psi )$ be a non-zero nearly holomorphic modular form. 
Let $n_{0}$ be the smallest integer $m$ such that $a_{m}(f,X)\neq 0$ and $n_{1}$ the smallest integer $m$ such that $a_{m}(h)\neq 0$. 
Then $a_{n_{0}+n_{1}}(fh,X)=a_{n_{0}}(f,X)a_{n_{1}}(h)\neq 0$. Especially we have $fh\neq 0$. 
Thus, the vertical map on the left-hand side is also injective. 
Then, by replacing $k$ with $k+m$, it suffices to prove that the map $\iota :N_{k}^{\leq r,\mathrm{cusp}}(N,\psi )\rightarrow \mathbb{C}[[q]]$ is injective 
with $k>2r$. Let $f\in N_{k}^{\leq r,\mathrm{cusp}}(N,\psi)$. 
Recall that we have an expression $f=\sum_{j=0}^{r}\delta_{k-2j}^{(j)}(f_{j})$ with $f_{j}\in S_{k-2j}(N,\psi)$ for $0\leq j\leq r$ (see Proposition 
\ref{prop:nearlyholomorphic}). By \eqref{commutative delta and d and Tp and Up}, we have 
$\iota(f)=\sum_{j=0}^{r}\iota(\delta_{k-2j}^{(j)}(f_{j}))=\sum_{j=0}^{r}d^{j}(f_{j}(q))$. 
We assume that $\iota(f)=0$, hence $\sum_{j=0}^{r}d^{j}(f_{j}(q))=0$. 
Since we have $\sum_{j=0}^{r}n^{j}a_{n}(f_{j})=0$ for each $n\in \mathbb{Z}_{\geq 1}$, we have  
$$\sum_{j=0}^{r}D(s-j,f_{0}^{\rho},f_{j})=\sum_{n=1}^{+\infty}\overline{a_{n}(f_{0})}\left(\sum_{j=0}^{r}n^{j}a_{n}(f_{j})\right)n^{-s}=0.$$
Since $D(s-j,f_{0}^{\rho},f_{j})$ is holomorphic at $s=k$ for each $1\leq j\leq r$, we have $\mathrm{Res}_{s=k}D(s,f_{0}^{\rho},f_{0})=\mathrm{Res}_{s=k}\left(\sum_{j=0}^{r}D(s-j,f_{0}^{\rho},f_{j})\right)=0$. On the other hand, by \eqref{residue of rankin selberg}, we see that $\mathrm{Res}_{s=k} D(s,\linebreak f_{0}^{\rho},f_{0})\in \langle f_{0},f_{0}\rangle_{k,N}\mathbb{C}^{\times}$. Thus, we have $f_{0}=0$, which implies that $d\left(\sum_{j=1}^{r}d^{j-1}(f_{j}(q))\right)=0$. {We put $\sum_{j=1}^{r}d^{j-1}(f_{j}(q))=\sum_{n=0}^{+\infty}b_{n}q^{n}$ with $b_{n}\in \mathbb{C}$. Since $f_{j}$ with $1\leq j\leq r$ are cusp forms, we have $b_{0}=0$. Since $d\left(\sum_{j=1}^{r}d^{j-1}(f_{j}(q))\right)=\sum_{n=1}^{+\infty}nb_{n}q^{n}=0$, we have $b_{n}=0$ for every $n\geq 1$. Thus, we have $\sum_{j=1}^{r}d^{j-1}(f_{j}(q))=0$. We have $\iota(\sum_{j=1}^{r}\delta_{k-2j}^{(j-1)}(f_{j}))=\sum_{j=1}^{r}d^{j-1}(f_{j}(q))=0$ and $\sum_{j=1}^{r}\delta_{k-2j}^{(j-1)}(f_{j})\in N_{k-2}^{\leq r-1,\mathrm{cusp}}(N,\psi)$. By induction on $r$, we have $\sum_{j=1}^{r}\delta_{k-2j}^{(j-1)}(f_{j})=0$ and $f=f_{0}+\delta_{k-2}\left(\sum_{j=1}^{r}\delta_{k-2j}^{(j-1)}(f_{j})\right)=0$.} 
This completes the proof. 
\end{proof}
By Proposition \ref{classical ver of urban}, $\iota$ in \eqref{definition of iota} induces an injective $\K$-linear map
\begin{equation}\label{k-lineara injective neary}
\iota: N_{k}^{\leq r}(N,\psi;\K)\rightarrow \K[[q]].
\end{equation} 
Let $\chi$ be a Dirichlet character modulo $N$ with $N\in \mathbb{Z}_{\geq 1}$. We define the Gauss sum of $\chi$ to be
\begin{equation}\label{definition of gauss sum}
G(\chi)=\sum_{a=1}^{c_{\chi}}\chi_{0}(a)e^{2\pi\sqrt{-1}a\slash c_{\chi}}
\end{equation} 
where $\chi_{0}$ is the primitive Dirichlet character associated with $\chi$ and $c_{\chi}$ the conductor of $\chi$. 

For each Dirichlet character $\psi_{1}$ (resp. $\psi_{2}$) modulo $N_{1}$ (resp. $N_{2}$) with $N_{1},N_{2}\in \mathbb{Z}_{\geq 1}$ and for each 
$k\in \mathbb{Z}_{\geq 1}$ such that $\psi_{1}(-1)\psi_{2}(-1)=(-1)^{k}$, we define an Eisenstein series 
$E_{k}(z,s;\psi_{1},\psi_{2})$ by 
\begin{equation}\label{definition of nonholo eisenpsi12}
E_{k}(z,s;\psi_{1},\psi_{2})=y^{s}\sum_{(m,n)\in \mathbb{Z}^{2}\backslash \{(0,0)\}}\psi_{1}(m)\overline{\psi}_{2}(n)(mN_{2}z+n)^{-k}\vert mN_{2}z+n\vert^{-2s}.
\end{equation}
The series in the right-hand side is uniformly absolutely convergent on the region $\{ s \in \mathbb{C} \ \vert \ k+2\mathrm{Re}(s)>2\}$. By \cite[Corollary 7.2.11]{Miyake89}, $\Gamma(k+s)E_{k}(z,s;\psi_{1},\psi_{2})$ is continued holomorphically to the whole $\mathbb{C}$-plane. By \cite[(7.2.2)]{Miyake89}, we see that $E_{k}(z,s;\psi_{1},\psi_{2})\in C_{k}^{\infty}(N_{1}N_{2},\psi_{1}\psi_{2})$. Let $r\in \mathbb{Z}_{\geq 0}$ such that $0\leq r<k$. We define $\epsilon_{k,2r+2}(\psi_{1},\psi_{2})$ to be $1$ (resp. $0$) when $k=2r+2$ and $\psi_{1}$ and $\psi_{2}$ are trivial characters modulo $N_{1}$ and $N_{2}$ respectively (resp. otherwise). By \cite[Theorem 7.2.9]{Miyake89}, we have
\begin{equation}\label{fourier expan eisen at infty}
E_{k}(z,-r;\psi_{1},\psi_{2})=\sum_{n=0}^{+\infty}a_{n}\left(E_{k}(z,-r;\psi_{1},\psi_{2}),\tfrac{-1}{4\pi y}\right)e^{2\pi\sqrt{-1}nz}
\end{equation}
where $a_{n}\left(E_{k}(z,-r;\psi_{1},\psi_{2}),X\right)\in \mathbb{C}[X]_{\leq r+\epsilon_{k,2r+2}(\psi_{1},\psi_{2})}$ with $n\in \mathbb{Z}_{\geq 0}$. By \cite[(7.2.56) and Theorem 7.2.15]{Miyake89}, 
$E_{k}(z,-r;\psi_{1},\psi_{2})\vert_{k}\gamma$ has the following expression for each $\gamma\in SL_{2}(\mathbb{Z})$:  
\begin{equation}\label{eisensteinvertgamma linear combination}
E_{k}(z,-r;\psi_{1},\psi_{2})\vert_{k}\gamma=\sum_{i=1}^{m}a_{i}E_{k}\left(\tfrac{u_{i}}{v_{i}}z,-r;\psi_{1}^{(i)},\psi_{2}^{(i)}\right)
\end{equation}
where $u_{i},v_{i}$ and $m$ are positive integers, $a_{i}\in \mathbb{C}$ and $\psi_{1}^{(i)}$ (resp. $\psi_{2}^{(i)}$) is a Dirichlet character modulo $N_{1}^{(i)}$ (resp. $N_{2}^{(i)}$) such that $\psi_{1}^{(i)}\psi_{2}^{(i)}(-1)=(-1)^{k}$. By \eqref{fourier expan eisen at infty} and \eqref{eisensteinvertgamma linear combination}, we see that there exists a positive integer $m$ such that we have 
\begin{equation}\label{eisensteinvertgamma other cusp fourier}
E_{k}(z,-r;\psi_{1},\psi_{2})\vert_{k}\gamma=\sum_{n=0}^{+\infty}a_{n}^{(\gamma)}\left(E_{k}(z,-r;\psi_{1},\psi_{2}),\frac{-1}{4\pi y}\right)e^{2\pi\sqrt{-1}nz\slash m}
\end{equation}
for every $\gamma\in SL_{2}(\mathbb{Z})$ where $a_{n}^{(\gamma)}\left(E_{k}(z,-r;\psi_{1},\psi_{2}),X\right)\in \mathbb{C}[X]_{\leq r+1}$ with $n\in \mathbb{Z}_{\geq 0}$. 
Therefore, we see that $E_{k}(z,-r;\psi_{1},\psi_{2})\in N_{k}^{\leq r+1}(N_{1}N_{2},\psi_{1}\psi_{2})$. Further, by Lemma \ref{nearlyholormorphic infty easy lemma} and \eqref{fourier expan eisen at infty}, we have
\begin{equation}\label{level of eisenstein series phi1phi2}
E_{k}(z,-r;\psi_{1},\psi_{2})\in N_{k}^{\leq r+\epsilon_{k,2r+2}(\psi_{1},\psi_{2})}(N_{1}N_{2},\psi_{1}\psi_{2}).
\end{equation}}
We define
\begin{multline}\label{phi1phi2mueisensteinseries}
F_{k}(z,s;\psi_{1},\psi_{2})=2^{-k-1}\pi^{-(k+s)}\sqrt{-1}^{k}\Gamma(k+s)G(\psi_{2})(\psi_{2})_{0}(-1)c_{\psi_{2}}^{s}\\
\times\sum_{0<t \left\vert \frac{N_{2}}{c_{\psi_{2}}} \right.}\mu(t)(\psi_{2})_{0}(t)(c_{\psi_{2}}t)^{k+s-1}E_{k}(tz,s;\psi_{1},(\psi_{2})_{0})
\end{multline}
where $\mu$ is the M\"{o}bius function. Put
\begin{equation}
F_{k}(z;\psi_{1},\psi_{2})=F_{k}(z,0;\psi_{1},\psi_{2}).
\end{equation}
Let $k,N$ be  positive integers such that $N>1$ and let $\psi$ be a Dirichlet character modulo $N$ such that $\psi(-1)=(-1)^{k}$. It is known that the Fourier expansions of $F_{k}(z;\boldsymbol{1},\psi)$ and $F_{k}(z;\psi,\boldsymbol{1})$ are given by
\begin{equation}\label{definition of eisenstein seriesnonap}
F_{k}(z;\boldsymbol{1},\psi)=\tfrac{1}{2}L_{N}(1-k,\psi)+\sum_{n=1}^{+\infty}\left(\sum_{0<d\vert n}\psi(d)d^{k-1}\right)e^{2\pi\sqrt{-1}nz}
\end{equation}
and
\begin{equation}\label{another Eisenstein seriesnonap}
F_{k}(z;\psi,\boldsymbol{1})=\epsilon_{k,2}(\psi,\boldsymbol{1})\frac{\varphi(N)}{8\pi Ny}+\sum_{n=1}^{+\infty}\left(\sum_{0<d\vert n}\psi(d)\left(\frac{n}{d}\right)^{k-1}\right)e^{2\pi\sqrt{-1}nz},
\end{equation}
respectively, where $\boldsymbol{1}$ is the trivial character modulo $1$ and $\varphi(N)$ is the Euler function (see \cite[(3.4)]{Shimura1976} and \cite[Corollary 6.2 and Corollary 6.3]{Hida1988}). Then, we see that 
\begin{equation}\label{level of eisenstein series left1nap}
F_{k}(z;\boldsymbol{1},\psi)\in M_{k}(N,\psi,\mathbb{Q}(\psi))
\end{equation}
and
\begin{equation}\label{level of eisenstein series right 1nap}
F_{k}(z;\psi,\boldsymbol{1})\in N_{k}^{\leq \epsilon_{k,2}(\psi,\boldsymbol{1})}(N,\psi;\mathbb{Q}(\psi)).
\end{equation}
 The following lemma is proved in \cite[Theorem 6.6]{Hida1988}.
\begin{lem}\label{lemma for interpolation formula}
Let $f\in S_{k}(Np^{\beta},\psi)$ and $g\in S_{l}(Np^{\beta},\xi)$, where $\beta\in \mathbb{Z}_{\geq 1}$. Assume that $(N,p)=1$ and $k>l$. 
\begin{enumerate}
\item
For each $0\leq m<\frac{k-l}{2}$, we have
$$
\Lambda_{Np^{\beta}}(m+l,f,g)=t_{m} \left\langle 
f^{\rho}\vert_{k}\tau_{Np^{\beta}},g\vert_{l}\tau_{Np^{\beta}}\delta_{k-l-2m}^{(m)}(F_{k-l-2m}(z;\boldsymbol{1},\psi\xi)) 
\right\rangle_{k,Np^{\beta}}.
$$
\item For each $\frac{1}{2}(k-l)\leq m<k-l$, we have
$$\Lambda_{Np^{\beta}}(m+l,f,g)=t_{m} 
\left\langle 
f^{\rho}\vert_{k}\tau_{Np^{\beta}},g\vert_{l}\tau_{Np^{\beta}}\delta_{l-k+2m+2}^{(k-l-m-1)}(F_{l-k+2m+2}(z;\psi\xi,\boldsymbol{1}))\right\rangle_{k,Np^{\beta}}$$
\end{enumerate}
where 
$$t_{m} =2^{k+1}(Np^{\beta})^{\frac{1}{2}(k-l-2m-2)}(\sqrt{-1})^{l-k}.$$
\end{lem}
Let $f\in S_{k}(N,\psi)$ be a primitive form. It is classically known that 
we have 
\begin{equation}
f\vert_{k}\tau_{N}=w(f)f^{\rho}
\end{equation}
where $w(f)$ is a complex number such that $\vert w(f)\vert=1$ ($cf$. \cite[Theorem 4.6.15]{Miyake89}). Let $\pi_{f}$ be the automorphic representation of $GL_{2}(\mathbb{A})$ attached to the primitive form $f$, where $\mathbb{A}$ is the adele of $\mathbb{Q}$. We factorize $\pi_{f}$ into the tensor product of locall representations 
\begin{equation}\label{factorization of aut rep}
\pi_{f}=\otimes_{q}\pi_{f,q}
\end{equation}
over all the places $q$ of $\mathbb{Q}$. By \cite[page 38]{Hida1988}, we have 
\begin{equation}
w(f)=\prod_{l<\infty}\epsilon(1\slash2,\pi_{f,l})
\end{equation}
where $\epsilon(s,\pi_{f,l})$ is the $\epsilon$-factor attached to $\pi_{f,l}$ defined in \cite[Theorem 2.18]{JaLa1970}. Put
\begin{equation}\label{notp part of rootnumber}
w^{\prime}(f)=\prod_{\substack{l<\infty\\ l\neq p}}\epsilon(1\slash 2,\pi_{f,l}).
\end{equation}

We have the following lemma.
\begin{lem}\label{lemma for modefied euler factor}
Let $N$ be a positive integer which is prime to $p$. Let $f\in S_{k}(Np^{m(f)},\psi)$ be a normalized Hecke eigenform which is new away from $p$ with $m(f)\in \mathbb{Z}_{\geq 1}$. Assume that $a_{p}(f)\neq 0$ and $m(f)$ is the smallest positive integer $m$ such that $f\in S_{k}(Np^{m},\psi)$. Then, we have 
\begin{multline*}
\frac{\langle f^{\rho}\vert_{k}\tau_{Np^{m(f)}},f\rangle_{k,Np^{m(f)}}}{\langle f^{0},f^{0}\rangle_{k,c_{f}}}\\
=\begin{cases}(-1)^{k}w(f^{0})\ &\mathrm{if}\ f=f^{0},\\
p^{-\frac{k}{2}+1}
a_{p}(f)\left(1-\frac{\psi_{0}(p)p^{k-1}}{a_{p}(f)^{2}}\right)\left(1-\frac{\psi_{0}(p)p^{k-2}}{a_{p}(f)^{2}}\right)
(-1)^{k}w(f^{0})
\ &\mathrm{if}\ f\neq f^{0},\end{cases}
\end{multline*}
where $\psi_{0}$ is the primitive Dirichlet character associated with $\psi$.
\end{lem}
\begin{proof}
First, we assume that $f = f^{0}$. Then $f^\rho$ is also a primitive form with conductor $Np^{m(f)}$.  Since $(-1)^{k}f=f\vert_{k}\tau_{Np^{m(f)}}^{2}=w(f)(f^{\rho})\vert_{k}\tau_{Np^{m(f)}}=w(f)w(f^{\rho})f$ and $\vert w(f)\vert =1$, we have $\overline{w(f^{\rho})}=(-1)^{k}w(f)$. Hence we have 
$$
\frac{\langle f^{\rho}\vert_{k}\tau_{Np^{m(f)}},f\rangle_{k,Np^{m(f)}}}{\langle f^{0},f^{0}\rangle_{k,m_f}}
=\frac{\langle w(f^\rho ) f,f\rangle_{k,Np^{m(f)}}}{\langle f,f\rangle_{k,Np^{m(f)}}}
= \overline{w(f^\rho)} =  (-1)^{k}w(f^{0}).
$$ 
In the rest of the proof, we assume that $f\neq f^{0}$. Note that we have $m(f)=1$ and $c_{f}=N$ in this case. By the proof of Proposition \ref{peterson quatient algebraic}, 
we have $f=f^{0}-\frac{\psi_{0}(p)p^{k-1}}{a_{p}(f)}f^{0}(pz)$. We have 
$$
(f^{0})^{\rho}\vert_{k}\tau_{Np}=(f^{0})^{\rho}\vert_{k}\tau_{N}\begin{pmatrix}p&0\\0&1\end{pmatrix}
=(-1)^{k}\overline{w(f^{0})}(f^{0})\vert_{k}\begin{pmatrix}p&0\\0&1\end{pmatrix}=(-1)^{k}\overline{w(f^{0})}p^{\frac{k}{2}}f^{0}(pz)
$$ 
and
$$
(f^{0}(pz))^{\rho}\vert_{k}\tau_{Np}=p^{-\frac{k}{2}}(f^{0})^{\rho}\vert_{k}\begin{pmatrix}p&0\\0&1\end{pmatrix}\tau_{Np}
=p^{-\frac{k}{2}}(f^{0})^{\rho}\vert_{k}\tau_{N}=(-1)^{k}\overline{w(f^{0})}p^{-\frac{k}{2}}f^{0}.
$$
By \cite[(3.2)]{Shimura1976},  for each $t_{1},t_{2}\in \{0,1\}$, we have
$$\frac{\langle f^{0}(p^{t_{1}}z),f^{0}(p^{t_{2}z})\rangle_{k,Np}}{\langle f^{0},f^{0}\rangle_{k,N}}=\begin{cases}1+p\ &\mathrm{if}\ (t_{1},t_{2})=(0,0),\\
(1+p)p^{-k}\ &\mathrm{if}\ (t_{1},t_{2})=(1,1),\\
p^{-k+1}a_{p}(f^{0})\overline{\psi_{0}}(p)\ &\mathrm{if}\ (t_{1},t_{2})=(0,1),\\
p^{-k+1}a_{p}(f^{0})\ &\mathrm{if}\ (t_{1},t_{2})=(1,0).
\end{cases}$$
Then, since $a_{p}(f^{0})=a_{p}(f)+\frac{\psi_{0}(p)p^{k-1}}{a_{p}(f)}$we see that
\begin{align*}
&\langle f^{\rho}\vert_{k}\tau_{Np},f\rangle_{k,Np}\\
&=(-1)^{k}w(f^{0})p^{\frac{k}{2}}\left(p^{-k+1}a_{p}(f^{0})-2(1+p)p^{-1}\frac{\psi_{0}(p)}{a_{p}(f)}+\frac{\psi_{0}(p)a_{p}(f^{0})}{pa_{p}(f)^{2}}\right)\langle f^{0},f^{0}\rangle_{k,N}\\
&=(-1)^{k}w(f^{0})p^{-\frac{k}{2}+1}a_{p}(f)\left(1-\frac{\psi_{0}(p)p^{k-1}}{a_{p}(f)^{2}}\right)\left(1-\frac{\psi_{0}(p)p^{k-2}}{a_{p}(f)^{2}}\right)\langle f^{0},f^{0}\rangle_{k,N}.
\end{align*}
We complete the proof.
\end{proof}
Let $A$ be a ring. For each $m\in \mathbb{Z}$, we define
\begin{align}\label{formal hecke operator and delta operator}
T_{p}:A[[X]][[q]]\rightarrow A[[X]][[q]],\ \delta_{m}:A[[X]][[q]]\rightarrow A[[X]][[q]]
\end{align}
to be $T_{p}(h)=\sum_{n=0}^{+\infty}a_{pn}(pX)q^{n}$ and $\delta_{m}(h)=\sum_{n=0}^{+\infty}\left((n+mX)a_{n}(X)-X^{2}\frac{\partial a_{n}(X)}{\partial X}\right)q^{n}$ for each $h=\sum_{n=0}^{+\infty}a_{n}(X)q^{n}$, where $a_{n}(X)\in A[[X]]$. Put $\delta_{m}^{(r)}=\delta_{m+2r-2}\cdots\delta_{m+2}\delta_{m}$ for each non-negative integer $r$. We remark that we understand that $\delta_{m}^{(0)}=1$ is the identity operator. Let 
\begin{equation}\label{definition of formal iota and d}
\iota: A[[X]][[q]]\rightarrow A[[q]],\ d:A[[q]]\rightarrow A[[q]]
\end{equation}
be the operators defined by $\iota(h)=\sum_{n=0}^{+\infty}a_{n}(0)q^{n}$ for each $h=\sum_{n=0}^{+\infty}a_{n}(X)q^{n}\in A[[X]][[q]]$, where $a_{n}(X)\in A[[X]]$ and $d=\frac{d}{dq}$. In the same way as \eqref{commutative delta and d and Tp and Up}, the following diagrams are commutative: 
\begin{equation}\label{formal commutative delta and d and Tp and Up}
\xymatrix{
A[[X]][[q]]\ar[d]^{\delta_{m}}\ar[r]^{\iota}&A[[q]]\ar[d]^{d}\\
A[[X]][[q]]\ar[r]^{\iota}&A[[q]],}\ \ \xymatrix{
A[[X]][[q]]\ar[d]^{T_{p}}\ar[r]^{\iota}&A[[q]]\ar[d]^{T_{p}}\\
A[[X]][[q]]\ar[r]^{\iota}&A[[q]]}
\end{equation}
for each $m\in \mathbb{Z}$. For each $g=\sum_{n=0}^{+\infty}a_{n}q^{n}\in  A[[q]]$, $N\in \mathbb{Z}_{\geq 1}$ and $a\in \mathbb{Z}\slash m\mathbb{Z}$ with $m\in \mathbb{Z}_{\geq 1}$, we put
\begin{equation}\label{power series and qn mapst qNn}
g\vert_{[N]}=\sum_{n=0}^{+\infty}a_{n}q^{Nn}
\end{equation}
and 
\begin{equation}\label{porwer series and mod series}
g_{\equiv a (m)}=\sum_{n\equiv a\ \mathrm{mod}\ m}a_{n}q^{n}.
\end{equation}
\begin{lem}\label{lemma 1 of shimura rankin selberg}
Let $M$ be a positive integer such that $p\vert M$. Let $f\in S_{k}(M,\psi)$ be a normalized cuspidal Hecke eigenform and $g\in S_{l}(M,\xi)$ a cusp form where $k$ and $l$ are positive integers and $\psi$ and $\xi$ are Dirichlet characters modulo $M$. We have
$$D(s,f,g\vert_{[p^{n}]})=a_{p}(f)^{n}p^{-ns}D(s,f,g)$$
for each non-negative integer $n$ where $D(s,f,g)$ is the Rankin-Selberg $L$-series defined in \eqref{Rankin-Selberg L-series}.
\end{lem}
\begin{proof}
Let $P$ be the set of primitive forms $h\in S_{k}(c_{h},\xi)$ such that $c_{h}\vert M$ where $c_{h}$ is the conductor of $h$. By \cite[Lemma 4.6.9]{Miyake89}, we see that $\{h\vert_{[t]}\}_{\substack{h\in P\\ 0<t\vert \frac{M}{c_{h}}}}$ is a basis of $S_{k}(M,\xi)$. By \cite[Lemma 1]{Shimura1976}, we have $D(s,f,h\vert_{[tp^{n}]})=a_{p}(f)^{n}p^{-ns}D(s,f,h\vert_{[t]})$ for each $h\in P$ and $0<t\vert\frac{M}{c_{h}}$. Since $g$ is a linear combination of $h\vert_{[t]}$, $h\in P$ and $0<t\vert\frac{M}{c_{h}}$, we have $D(s,f,g\vert_{[p^{n}]})=a_{p}(f)^{n}p^{-ns}D(s,f,g)$.
\end{proof}
\par 
Let $f\in S_{k}(Np^{r},\psi)$ and $g\in S_{l}(N^{\prime}p^{t},\xi)$ be normalized cuspidal Hecke eigenforms which are new away from $p$ and $a_{p}(f)\neq 0$ and $a_{p}(g)\neq 0$. Here $N,N^{\prime},r,t$ are positive integers such that $N$ and $N^{\prime}$ are prime to $p$. We denote by $\pi_{f,p}$ and $\pi_{g,p}$ the automorphic representations of $GL_{2}(\mathbb{Q}_{p})$ attached to $f$ and $g$ respectively. Let $\alpha(f^0)$ and $\alpha^{\prime}(f^0 )$ (resp. $\alpha(g^0)$ and $\alpha^{\prime} (g^0 )$) be the algebraic numbers which satisfy $[(1-\alpha(f^{0})p^{-s})(1-\alpha^{\prime}(f^{0})p^{-s})]^{-1}=\sum_{n=0}^{+\infty}a_{p^{n}}(f^{0})p^{-ns}$ (resp. $[(1-\alpha(g^{0})p^{-s})(1-\alpha^{\prime}(g^{0})p^{-s})]^{-1}=\sum_{n=0}^{+\infty}a_{p^{n}}(g^{0})p^{-ns}$) where $f^{0}$ and $g^{0}$ are the primitive forms associated with $f$ and $g$ respectively. Assume that $\alpha(f^{0})=a_{p}(f)$ and $\alpha(g^{0})=a_{p}(g)$. Let $\xi=\xi^{\prime}\xi_{(p)}$ be the decomposition of $\xi$ where $\xi^{\prime}$ and $\xi_{(p)}$ are Dirichlet characters modulo $N^{\prime}$ and $p^{t}$ respectively. Put $\beta(g^{0})=\frac{p^{l-1}\xi^{\prime}(p)}{\alpha(g^{0})}$.
Let $\phi$ be a Dirichlet character modulo $p^{n}$ with $n\in \mathbb{Z}_{\geq 1}$. 
{We define 
\begin{equation}\label{for comaptibilityof PC Euler factor}
E_{p,\phi}(s,f,g)=E_{1}(s)E_{2}(s)E_{3}(s)
\end{equation}
with
\begin{align*}
&E_{1}(s)\\
&=\begin{cases}\left(\frac{p^{s-1}}{\alpha(g^{0})^{\rho}\alpha(f^{0})}\right)^{\ord_{p}(c_{\phi})}\left(\frac{p^{s-1}}{\beta(g_{0})^{\rho}\alpha(f^{0})}\right)^{\ord_{p}(c_{\xi_{(p)}\phi})}\ &\mathrm{if}\ \pi_{g,p}\ \mathrm{is\ principal\ or}\ \ord_{p}(c_{\phi})>0, \\
-\left(\frac{p^{s-1}}{\alpha(g^{0})^{\rho}\alpha(f^{0})}\right)\ &\mathrm{if}\ \pi_{g,p}\ \mathrm{is\ special\ and}\ \ord_{p}(c_{\phi})=0 ,
\end{cases}\\
&E_{2}(s)=\begin{cases}\left(1-\frac{\phi_{0}(p)p^{s-1}}{\alpha(g^{0})^{\rho}\alpha(f^{0})}\right)\left(1-\frac{(\xi_{(p)}\phi)_{0}(p)p^{s-1}}{\beta(g_{0})^{\rho}\alpha(f^{0})}\right)\ &\mathrm{if}\ \pi_{g,p}\ \mathrm{is\ principal\ or}\ \ord_{p}(c_{\phi})>0, \\
\left(1-\frac{p^{s-1}}{\beta(g_{0})^{\rho}\alpha(f^{0})}\right)\ &\mathrm{if}\ \pi_{g,p}\ \mathrm{is\ special\ and}\ \ord_{p}(c_{\phi})=0 , 
\end{cases}\\
&E_{3}(s)=(1-\phi_{0}(p)\alpha^{\prime}(f^{0})\alpha(g^{0})^{\rho}p^{-s})(1-(\xi_{(p)}\phi)_{0}(p)\alpha^{\prime}(f^{0})\alpha^{\prime}(g^{0})^{\rho}p^{-s}), 
\end{align*}
where $\phi_{0}$ and $(\xi_{(p)}\phi)_{0}$ are primitive Dirichlet characters modulo $c_{\phi}$ and $c_{\xi_{(p)}\phi}$ attached to $\phi$ and $\xi_{(p)}\phi$ respectively and $c^{\rho}$ is the complex conjugate of $c\in \mathbb{C}$. }
\begin{lem}\label{for comaptibilityof PC}
Let $f\in S_{k}(Np^{r},\psi)$ and $g\in S_{l}(N^{\prime}p^{t},\xi)$ be normalized cuspidal Hecke eigenforms which are new away from $p$ with $k>l$. Here $N,N^{\prime},r,t$ are positive integers such that $N$ and $N^{\prime}$ are prime to $p$. Let $\xi_{(p)}$ be the restriction of $\xi$ on $(\mathbb{Z}\slash p^{t}\mathbb{Z})^{\times}$. We denote by $M$ the least common multiple of $N$ and $N^{\prime}$. Assume that $a_{p}(f)\neq 0$ and $a_{p}(g)\neq 0$.  Let $\phi$ be a Dirichlet character modulo $p^{n}$ with $n\in \mathbb{Z}_{\geq 1}$ and $E_{p,\phi}(s,f,g)$ the $p$-th Euler factor defined in \eqref{for comaptibilityof PC Euler factor}. We denote by  $\beta$ the smallest positive integer $m$ such that $g\otimes\phi\in S_{l}(N^{\prime}p^{\beta},\xi\phi^{2})$. Then, we have
\begin{multline}
p^{\beta(2m+l)\slash 2}a_{p}(f)^{-\beta}\Lambda_{Mp^{\max\{r,\beta\}}}(m+l,f,(g\otimes\phi)\vert_{l}\tau_{N^{\prime}p^{\beta}})\\
=w^{\prime}(g^{0})G(\phi)G(\xi_{(p)}\phi)\phi(N^{\prime})E_{p,\phi}(m+l,f,g)\Lambda(m+l,f,(g\otimes\phi)^{\rho})
\end{multline}
for each integer $m$ with $l\leq m<k$ where $\Lambda(s,f,(g\otimes\phi)^{\rho})$ is the Rankin-Selberg $L$-series defined in \eqref{rankin attached primitives}, $w^{\prime}(g^{0})$ is the constant defined in \eqref{notp part of rootnumber} and $G(\phi)$ and $G(\xi_{(p)}\phi)$ are the Gauss sums defined in \eqref{definition of gauss sum}.
 \end{lem}
In the case of $\ord_{p}(a_{p}(f))=0$, Lemma \ref{for comaptibilityof PC} is proved in \cite[Lemma 5.2]{Hida1988}. We can prove Lemma \ref{for comaptibilityof PC} for any $f$ with $a_{p}(f)\neq 0$ in the same way as \cite[Lemma 5.2]{Hida1988}. Then, we omit the proof of Lemma \ref{for comaptibilityof PC}.
\par 
In the end of this subsection, we recall the definition of Hida families. Let $\Gamma$ be a $p$-adic Lie group which is isomorphic to $1+p\mathbb{Z}_p\subset \mathbb{Q}_{p}^{\times}$ via a continuous character $\chi : \Gamma \longrightarrow \mathbb{Q}_{p}^{\times}$. Let $\xi$ be a Dirichlet character modulo $Np$ and $\mathbf{I}$ an integral domain which is a finite free extension of $\mathcal{O}_{\K}[[\Gamma]]$, where $N$ is a positive integer such that $(N,p)=1$. Let $\omega$ be the Teichm\"{u}ller character modulo $p$. 
Recall that an $\mathbf{I}$-adic cusp form of tame level $N$ and character 
$\xi$ is a formal power series $G(q)=\sum_{n=1}^{+\infty}a_{n}(G)q^{n}\in \mathbf{I}[[q]]$ such that for each arithmetic specialization $\kappa\in \mathfrak{X}_{\mathbf{I}}$ with $w_{\kappa}\geq 2$, the specialization $\kappa(G)=\sum_{n=1}^{+\infty}\kappa(a_{n}(G))q^{n}$ is in $S_{w_{\kappa}}(Np^{m_{\kappa}+1},\xi\omega^{-k}\phi_{\kappa})$. 
We denote by $S(Np,\xi;\mathbf{I})$ the space of $\mathbf{I}$-adic cusp forms of tame level $N$ and character $\xi$. 
The operator $T_{p}: \mathbf{I}[[q]]\rightarrow \mathbf{I}[[q]]$ defined by $\sum_{n=0}^{+\infty}a_{n}q^{n}\mapsto \sum_{n=0}^{+\infty}a_{pn}q^{n}$ induces an $\mathbf{I}$-module homomorphism $T_{p}: S(Np,\xi;\mathbf{I})\rightarrow S(Np,\xi;\mathbf{I})$. Let $e$ be the ordinary projection on $S(Np,\xi;\mathbf{I})$ defined by
$$e=\lim_{n\rightarrow +\infty}T_{p}^{n!}.$$
The space $eS(Np,\xi;\mathbf{I})$ is called the space of ordinary $\mathbf{I}$-adic cusp forms. 
Let $\alpha_{1},\ldots, \alpha_{n}$ be a basis of $\mathbf{I}$ over $\mathcal{O}_{\K}[[\Gamma]]$. Then, we have an $\mathcal{O}_{\K}[[\Gamma]]$-module isomorphism
\begin{equation}\label{eSIisom eSLambda}
\oplus_{i=1}^{n}eS(Np,\xi;\mathcal{O}_{\K}[[\Gamma]])\stackrel{\sim}{\rightarrow}eS(Np,\xi;\mathbf{I}),\ (G_{i})_{i=1}^{n}\mapsto \sum_{i=1}^{n}G_{i}\alpha_{i}.
\end{equation}
{We say that $G\in eS(Np,\xi;\mathbf{I})$ is a primitive Hida family, if $\kappa(G)$ is a normalized Hecke eigenform which is new away from $p$ for any $\kappa\in \mathfrak{X}_{\mathbf{I}}$ with $w_{\kappa}\geq 2$.
\subsection{A two-variable admissible distribution for the case of $\Lambda$-adic cusp forms}\label{ssc const. ad}
{Let $\K$ be a finite extension of $\mathbb{Q}_{p}$. In this subsection, we regard nearly holomorphic modular forms over $\K$ as elements of $\K[X][[q]]$ via the $q$-expansions.}
{Let $\Gamma_{1}$ and $\Gamma_{2}$ be $p$-adic Lie groups which are isomorphic to $1+p\mathbb{Z}_{p}$. 
We set $\Delta_{L}=(\mathbb{Z}\slash Lp\mathbb{Z})^{\times}$ for each positive integer $L$ which is prime to $p$ and we denote 
$\Delta_{1}$ by $\Delta$. We fix continuous characters $\chi_{1}: \Delta\times \Gamma_{1}\rightarrow \mathbb{Q}_{p}^{\times}$ and $\chi_{2}: \Gamma_{2}\rightarrow \mathbb{Q}_{p}^{\times}$ which induce $\chi_{1}: \Delta\times \Gamma_{1}\stackrel{\sim}{\rightarrow}\mathbb{Z}_{p}^{\times}$ and $\chi_{i}:\Gamma_{i}\stackrel{\sim}{\rightarrow}1+p\mathbb{Z}_{p}$ for $i=1,2$.} We fix positive integers $N$ and $N^{\prime}$ which are prime to $p$. Let $f\in S_{k}(Np^{m(f)},\psi;\K)$ be a normalized cuspidal Hecke eigenform which is new away from $p$ with $m(f)\in \mathbb{Z}_{\geq 1}$ and $G\in S(N^{\prime}p,\xi;\mathcal{O}_{\K}[[\Gamma_{2}]])$. Assume that $m(f)$ is the smallest positive integer $m$ such that $f\in S_{k}(Np^{m},\psi)$. Put $\boldsymbol{h}=(2\alpha,\alpha)$ with $\alpha=\ord_{p}(a_{p}(f))$. Let $M$ be the least common multiple of $N$ and $N^{\prime}$. We assume the following conditions:
\begin{enumerate}
\item We have $k>\lfloor 2\alpha\rfloor+\lfloor \alpha\rfloor+2$.
\item All $M$-th roots of unity and Fourier coefficients of $f^{0}$ are contained in $\K$, where $f^{0}$ is the primitive form associated with $f$.
\end{enumerate}
Let $L$ be a positive integer which is prime to $p$. {There exists a natural isomorphism 
\begin{equation}\label{isomorphism of zslahs lpztimeand gammaslash gama}
(\mathbb{Z}\slash Lp\mathbb{Z})^{\times}\times (1+p\mathbb{Z}_{p})\slash (1+p\mathbb{Z}_{p})^{p^{m}}\simeq (\mathbb{Z}\slash Lp^{m+1}\mathbb{Z})^{\times}
\end{equation}
 for each $m\in \mathbb{Z}_{\geq 0}$. 
Hence the isomorphism $\chi_{1}: \Delta\times \Gamma_{1}\stackrel{\sim}{\rightarrow}\mathbb{Z}_{p}^{\times}$ makes us  
identify $\Delta_{L}\times (\Gamma_{1}\slash \Gamma_{1}^{p^{m}})$ with $(\mathbb{Z}\slash Lp^{m+1}\mathbb{Z})^{\times}$ 
for each $m\in \mathbb{Z}_{\geq 0}$. 
\par 
Let $\boldsymbol{d}=(0,2)$, $\boldsymbol{e}=(k-3, k-1)$. In this subsection, we construct a two-variable admissible distribution $s_{(f,G)}\in I_{\boldsymbol{h}}^{[\boldsymbol{d},\boldsymbol{e}]}\otimes_{\mathcal{O}_{\K}[[\Gamma_{1}\times \Gamma_{2}]]}\mathcal{O}_{\K}[[(\Delta\times \Gamma_{1})\times \Gamma_{2}]]$ which satisfies an interpolation property (see \eqref{definition of inverse q[d,ealpha]} and Lemma \ref{proposition of interpolation of L(f,G)}). Here, $\mathcal{O}_{\K}[[(\Delta\times \Gamma_{1})\times \Gamma_{2}]]=\varprojlim_{U}\mathcal{O}_{\K}[((\Delta\times \Gamma_{1})\times \Gamma_{2})\slash U]$, where $U$ runs through all open subgroups of $(\Delta\times \Gamma_{1})\times \Gamma_{2}$.
\subsubsection*{{\bf Outline of \S6.2}} 
It is difficult to construct the element $s_{(f,G)}$ in $I_{\boldsymbol{h}}^{[\boldsymbol{d},\boldsymbol{e}]}\otimes_{\mathcal{O}_{\K}[[\Gamma_{1}\times \Gamma_{2}]]}\mathcal{O}_{\K}[[(\Delta\times \Gamma_{1})\times \Gamma_{2}]]$ directly, since $[\boldsymbol{d},\boldsymbol{e}]$ contains a point which is not a critical point of the two-variable Rankin--Selberg $L$-series attached to $f$ and $G$ (see the illustrations of \eqref{illustration critical range [d,e],[d,ealpha]}).

Let $\boldsymbol{e}_{\alpha}=(\lfloor 2\alpha\rfloor, \lfloor \alpha\rfloor+2)$. As mentioned in \eqref{projection I is isom if e-dgeq h}, the projection $I_{\boldsymbol{h}}^{[\boldsymbol{d},\boldsymbol{e}]}\otimes_{\mathcal{O}_{\K}[[\Gamma_{1}\times \Gamma_{2}]]}\mathcal{O}_{\K}[[(\Delta\times \Gamma_{1})\times \Gamma_{2}]]\rightarrow I_{\boldsymbol{h}}^{[\boldsymbol{d},\boldsymbol{e}_{\alpha}]}\otimes_{\mathcal{O}_{\K}[[\Gamma_{1}\times \Gamma_{2}]]}\mathcal{O}_{\K}[[(\Delta\times \Gamma_{1})\times \Gamma_{2}]]$ is an isomorphism. Then, we construct the desired element $s_{(f,G)}\in  I_{\boldsymbol{h}}^{[\boldsymbol{d},\boldsymbol{e}]}\otimes_{\mathcal{O}_{\K}[[\Gamma_{1}\times \Gamma_{2}]]}\mathcal{O}_{\K}[[(\Delta\times \Gamma_{1})\times \Gamma_{2}]]$ as the inverse image by the projection of a similar element 
$s^{[\boldsymbol{d},\boldsymbol{e}_{\alpha}]}\in I_{\boldsymbol{h}}^{[\boldsymbol{d},\boldsymbol{e}_{\alpha}]}\otimes_{\mathcal{O}_{\K}[[\Gamma_{1}\times \Gamma_{2}]]}\mathcal{O}_{\K}[[(\Delta\times \Gamma_{1})\times \Gamma_{2}]]$ having a smaller range of interpolation. 
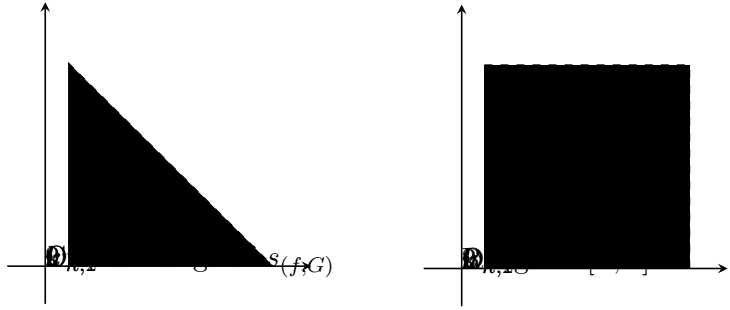
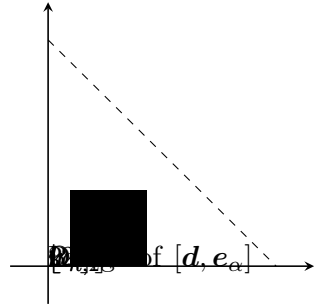
\begin{align}\label{illustration critical range [d,e],[d,ealpha]}
\begin{split}
\begin{tikzpicture}
 \draw[->,>=stealth,semithick] (-0.5,0)--(3.5,0)node[above]{$w_{\kappa,2}$}; 
 \draw[->,>=stealth,semithick] (0,-0.5)--(0,3.5)node[right]{$w_{\kappa,1}$}; 
 \draw (0,0)node[above  left]{O}; 
 \draw (0.3,0)node[below]{2};
 \draw (3,0)node[below]{$k$}; 
 \draw(0,2.7)node[left]{$k-2$};
  \fill[lightgray] (0.3,0)--(0.3,2.7)--(3,0);
 \draw[dashed,domain=0.3:3] plot(\x,3-\x);
   \draw (1.65,-0.5)node[below]{$\mathrm{Critical\ range\ of}\ s_{(f,G)}$};
\end{tikzpicture},&\ \ \begin{tikzpicture}
 \draw[->,>=stealth,semithick] (-0.5,0)--(3.5,0)node[above]{$w_{\kappa,2}$}; 
 \draw[->,>=stealth,semithick] (0,-0.5)--(0,3.5)node[right]{$w_{\kappa,1}$}; 
 \draw (0,0)node[above  left]{O}; 
 \draw (0.3,0)node[below]{2};
 \draw(0,2.7)node[left]{$k-2$};
  \draw (3,0)node[below]{$k$}; 
  \fill[lightgray] (0.3,0)--(0.3,2.7)--(3,2.7)--(3,0);
 \draw[dashed,domain=0.3:3] plot(\x,2.7);
 \draw[dashed, domain=0:2.7] plot(3,\x);
 \draw (1.65,-0.5)node[below]{$\mathrm{Range\ of}\ [\boldsymbol{d},\boldsymbol{e}]$};
\end{tikzpicture},\\ &\begin{tikzpicture}
 \draw[->,>=stealth,semithick] (-0.5,0)--(3.5,0)node[above]{$w_{\kappa,2}$}; 
 \draw[->,>=stealth,semithick] (0,-0.5)--(0,3.5)node[right]{$w_{\kappa,1}$}; 
 \draw (0,0)node[above  left]{O}; 
 \draw(1.3,0)node[below]{$\lfloor\alpha\rfloor+2$};
 \draw(0,1)node[left]{$\lfloor 2\alpha\rfloor$};
 \draw (0.3,0)node[below]{2};
 \draw (3,0)node[below]{$k$}; 
 \draw(0,3)node[left]{$k$};
  \fill[lightgray] (0.3,0)--(0.3,1)--(1.3,1)--(1.3,0);
 \draw[dashed,domain=0:3] plot(\x,3-\x);
  \draw (1.65,-0.5)node[below]{$\mathrm{Range\ of}\ [\boldsymbol{d},\boldsymbol{e}_{\alpha}]$};
\end{tikzpicture}.
\end{split}
\end{align}
{The construction of $s^{[\boldsymbol{d},\boldsymbol{e}_{\alpha}]}\in I_{\boldsymbol{h}}^{[\boldsymbol{d},\boldsymbol{e}_{\alpha}]}\otimes_{\mathcal{O}_{\K}[[\Gamma_{1}\times \Gamma_{2}]]}\mathcal{O}_{\K}[[(\Delta\times \Gamma_{1})\times \Gamma_{2}]]$ proceeds in three-step arguments. First, we construct a candidate of the element $
s^{[\boldsymbol{i}]}_{\boldsymbol{m}} 
\in \frac{\mathcal{O}_{\K} [[(\Delta\times \Gamma_{1})\times \Gamma_{2}]]}{(\Omega_{\boldsymbol{m}}^{[\boldsymbol{i}]})\mathcal{O}_{\K} [[(\Delta\times \Gamma_{1})\times \Gamma_{2}]]} 
\otimes_{\mathcal{O}_{\K}} \K 
$ with the desired interpolation property for every $\boldsymbol{i}\in [\boldsymbol{d},\boldsymbol{e}_{\alpha}]$ and for every $\boldsymbol{m}\in \mathbb{Z}_{\geq 0}^{2}$ (see \eqref{definition of qm1m2i1i1}). Second, for each $\boldsymbol{i}\in [\boldsymbol{d},\boldsymbol{e}_{\alpha}]$, 
we show that $\{
s^{[\boldsymbol{i}]}_{\boldsymbol{m}}\}_{\boldsymbol{m}\in \mathbb{Z}_{\geq 0}^{2}} $ satisfies the distribution property when $\boldsymbol{m}$ varies (see Proposition \ref{distribution of mu[boldsymbolr,boldsymbols]}). Third, for each $\boldsymbol{m}\in \mathbb{Z}_{\geq 0}^{2}$, we show that the elements 
$s^{[\boldsymbol{i}]}_{\boldsymbol{m}} 
\in \frac{\mathcal{O}_{\K} [[(\Delta\times \Gamma_{1})\times \Gamma_{2}]]}{(\Omega_{\boldsymbol{m}}^{[\boldsymbol{i}]})\mathcal{O}_{\K} [[(\Delta\times \Gamma_{1})\times \Gamma_{2}]]}
\otimes_{\mathcal{O}_{\K}} \K$ with $\boldsymbol{i} \in [\boldsymbol{d},\boldsymbol{e}_{\alpha}]$ 
satisfy the $\boldsymbol{h}$-admissible condition (see Proposition \ref{admissible condition of p-adic l}) and this allows us to have a non-negative integer $n\in \mathbb{Z}_{\geq 0}$ which satisfies 
$$p^{\langle \boldsymbol{m},\boldsymbol{h}-(\boldsymbol{j}-\boldsymbol{d})\rangle_{2}}\displaystyle{\sum_{\boldsymbol{i}\in [\boldsymbol{d},\boldsymbol{j}]}}\left(\prod_{t=1}^{2}\begin{pmatrix}j_{t}-d_{t}\\i_{t}-d_{t}\end{pmatrix}\right)(-1)^{\sum_{t=1}^{2}(j_{t}-i_{t})}\tilde{s}_{\boldsymbol{m}}^{[\boldsymbol{i}]}\in \mathcal{O}_{\K}[[(\Delta\times\Gamma_{1})\times\Gamma_{2}]]\otimes_{\mathcal{O}_{\K}}p^{-n}\mathcal{O}_{\K}$$
for every $\boldsymbol{m}\in \mathbb{Z}_{\geq 0}^{k}$ and $\boldsymbol{j}\in [\boldsymbol{d},\boldsymbol{e}_{\alpha}]$ where $\tilde{s}_{\boldsymbol{m}}$ is the lift of $s_{\boldsymbol{m}}^{[\boldsymbol{i}]}$ defined in \eqref{lift of qboldsymbolmboldsymboli}. By Lemma \ref{multivariable litfing prop deformation}, for each $\boldsymbol{m}\in \mathbb{Z}_{\geq 0}^{2}$, we have an 
element 
$$
s_{\boldsymbol{m}}^{[\boldsymbol{d},\boldsymbol{e}_{\alpha}]} \in \left(
\frac{\mathcal{O}_{\K} [[(\Delta\times \Gamma_{1})\times \Gamma_{2}]]}{(\Omega_{\boldsymbol{m}}^{[\boldsymbol{d},\boldsymbol{e}_{\alpha}]})\mathcal{O}_{\K} [[(\Delta\times \Gamma_{1})\times \Gamma_{2}]]} \otimes_{\mathcal{O}_{\K}} p^{-\langle \boldsymbol{h},\boldsymbol{m}\rangle_{2}}\mathcal{O}_{\K}\right)\otimes_{\mathcal{O}_{\K}}p^{-n-c^{[\boldsymbol{d},\boldsymbol{e}_{\alpha}]}} 
\mathcal{O}_{\K}
$$ 
projected to $s^{[\boldsymbol{i}]}_{\boldsymbol{m}}$ for every $\boldsymbol{i}$ in $[\boldsymbol{d},\boldsymbol{e}_{\alpha}]$ where $c^{[\boldsymbol{d},\boldsymbol{e}_{\alpha}]}$ is the constant defined in \eqref{constant for the admissible}. Then we obtain $s^{[\boldsymbol{d},\boldsymbol{e}_{\alpha}]} \in I_{\boldsymbol{h}}^{[\boldsymbol{d},\boldsymbol{e}_{\alpha}]}\otimes_{\mathcal{O}_{\K}[[\Gamma_{1}\times \Gamma_{2}]]}\mathcal{O}_{\K} [[(\Delta\times \Gamma_{1})\times \Gamma_{2}]]$ by taking the projective limit of $s_{\boldsymbol{m}}^{[\boldsymbol{d},\boldsymbol{e}_{\alpha}]}$.}
 \subsubsection*{{\bf{Construction of a candidate of}\ $s_{\boldsymbol{m}}^{[\boldsymbol{i}]}$}} Each arithmetic specialization $\kappa\in \mathfrak{X}_{\mathcal{O}_{\K}[[\Gamma_{2}]]}$ induces a continuous $\mathcal{O}_{\K}$-algebra homomorphism
\begin{equation}
\kappa\widehat{\otimes}_{\mathcal{O}_{\K}}\mathrm{id}_{\mathcal{O}_{\K}[[X]][[q]]}: \mathcal{O}_{\K}[[\Gamma_{2}]][[X]][[q]]\rightarrow \mathcal{O}_{\K}[\phi_{\kappa}][[X]][[q]],
\end{equation}
sending $c\widehat{\otimes}_{\mathcal{O}_{\K}}h$ to $ \kappa(c)\widehat{\otimes}_{\mathcal{O}_{\K}}h$ for each $c\in \mathcal{O}_{\K}[[\Gamma_{2}]]$ and $h\in \mathcal{O}_{\K}[[X]][[q]]$. If there is no risk of confusion, we denote $\kappa\widehat{\otimes}_{\mathcal{O}_{\K}}\mathrm{id}_{\mathcal{O}_{\K}[[X]][[q]]}$ by $\kappa$ by abuse of notation. We define
\begin{equation}
\langle\ \rangle: \mathbb{Z}_{p}^{\times}\rightarrow \mathbb{Z}_{p}[[\Gamma_{2}]]^{\times}
\end{equation}
to be $\langle z\rangle=[\chi_{2}^{-1}(z\omega^{-1}(z))]$ for $z\in \mathbb{Z}_{p}^{\times}$, where $[\ ]: \Gamma_{2}\rightarrow \mathbb{Z}_{p}[[\Gamma_{2}]]^{\times}$ is the tautological inclusion map and $\omega$ is the Teichm\"{u}ller character modulo $p$. By definition, we see that $\kappa(\langle z\rangle)=(z\omega^{-1}(z))^{w_{\kappa}}\phi_{\kappa}(z)$ for each $z\in \mathbb{Z}_{p}^{\times}$ and $\kappa\in \mathfrak{X}_{\mathbb{Z}_{p}[[\Gamma_{2}]]}$. Let $\psi$ be a character on $\Delta_{N}\times \left(\Gamma_{1}\slash \Gamma_{1}^{p^{m(\psi)}}\right)$ with $m(\psi)\in \mathbb{Z}_{\geq 0}$ where $N$ is a positive integer such that $(N,p)=1$, $G\in S(N^{\prime}p,\xi;\mathcal{O}_{\K}[[\Gamma_{2}]])$ where $N^{\prime}$ is a positive integer such that $(N^{\prime},p)=1$ and $\xi$ is a Dirichlet character modulo $N^{\prime}p$. 
{
For each $(i_{1},i_{2},i_{3})\in \mathbb{Z}_{\geq 0}^{3}$ and for each $a\in \Delta_{M}\times (\Gamma_{1}\slash  \Gamma_{1}^{p^{m}})$ with $m\in \mathbb{Z}_{\geq 0}$, we define 
an element $F_{(i_{1},i_{2},i_{3}),\Gamma_{2}}(a;Mp^{m+1}) \in \mathcal{O}_{\K}[[\Gamma_{2}]][[q]]$ by 
\begin{equation}\label{Definition of Fm,a}
F_{(i_{1},i_{2},i_{3}),\Gamma_{2}}(a;Mp^{m+1})=\sum_{n\in \mathbb{Z}_{\geq 1}}\sum_{\substack{0<d\vert n\\ d\equiv a\ \mathrm{mod}\ Mp^{m+1}}}\left(\frac{n}{d}\right)^{i_{1}}d^{i_{2}}n^{i_{3}}\langle d\rangle^{-1}q^{n}
\end{equation}
 where $M$ is the least common multiple of $N$ and $N^{\prime}$.} Let $\delta_{m}^{(r)}$ be the operator defined in \eqref{formal hecke operator and delta operator} for each $m\in \mathbb{Z}$ and for each $r\in \mathbb{Z}_{\geq 0}$. We put 
\begin{equation}\label{definition of Phii1ii2(a1;G)}
\small 
\Phi^{(\boldsymbol{i})}(a;\psi,G)=
\hspace*{-5mm}
\sum_{b\in \Delta\times (\Gamma_{1}\slash  \Gamma_{1}^{p^{m}})}G_{\equiv ab^{2} (p^{m+1})}
\hspace*{-5mm}
\sum_{\substack{c\in \Delta_{M}\times \left(\Gamma_{1}\slash  \Gamma_{1}^{p^{\max\{m,m(\psi)\}}}\right)\\ p_{m}^{(M)}(c)=b}} (\psi\xi^{-1})(c)H_{c}^{(\boldsymbol{i})}\in \mathcal{O}_{\K}[[\Gamma_{2}]][[X]][[q]]
\normalsize 
\end{equation}
for each $\boldsymbol{i}\in \mathbb{Z}_{\geq 0}^{2}$ satisfying $i_{2}\geq 2$ and $i_{1}+i_{2}<k$ and for each $a\in \Delta\times (\Gamma_{1}\slash  \Gamma_{1}^{p^{m}})$ with $m\in \mathbb{Z}_{\geq 0}$. Here, $p_{m}^{(M)}: \Delta_{M}\times \left(\Gamma_{1}\slash  \Gamma_{1}^{p^{\max\{m,m(\psi)\}}}\right)\rightarrow \Delta\times (\Gamma_{1}\slash  \Gamma_{1}^{p^{m}})$ is the natural projection,
\begin{equation}\label{definition of Wb(i1,i2)}
H_{c}^{(\boldsymbol{i})}=\begin{cases}\delta_{k-2i_{1}-i_{2}}^{(i_{1})}\left(F_{(0,k-2i_{1}-1,0),\Gamma_{2}}(c;Mp^{\max\{m,m(\psi)\}+1})\right)\ &\mathrm{if}\ 0\leq i_{1}<\frac{1}{2}(k-i_{2}),\\
\delta_{i_{2}-k+2i_{1}+2}^{(k-i_{1}-i_{2}-1)}\left(F_{(-k+2i_{1}+1,0,i_{2}),\Gamma_{2}}(c;Mp^{\max\{m,m(\psi)\}+1})\right)\ &\mathrm{if}\ \frac{1}{2}(k-i_{2})\leq i_{1}<k-i_{2}, \end{cases}
\end{equation}
and $G_{\equiv ab^{2} (p^{m+1})}\in \mathcal{O}_{\K}[[\Gamma_{2}]][[q]]$ is the power series defined in \eqref{porwer series and mod series}. Let $T_{p}$ be the operator defined in \eqref{formal hecke operator and delta operator}. 

For each $\boldsymbol{i}\in \mathbb{Z}_{\geq 0}^{2}$ such that $i_{2}\geq 2$ and $i_{1}+i_{2}<k$ and $(a_{1},a_{2})\in (\Delta\times \Gamma_{1}\slash  \Gamma_{1}^{p^{m_{1}}})\times \Gamma_{2}\slash \Gamma_{2}^{p^{m_{2}}}$ with $\boldsymbol{m}\in \mathbb{Z}_{\geq 0}^{2}$, we define an element 
$\phi^{(\boldsymbol{i})}((a_{1},a_{2});\psi,G) \in  \mathcal{O}_{\K}[[X]][[q]]$ by 
\begin{align}\label{definition Psii1i2a1a2}
\phi^{(\boldsymbol{i})}((a_{1},a_{2});\psi,G)=(-1)^{i_{1}}T_{p}\left(\int_{a_{2}\Gamma_{2}^{p^{m_{2}}}}\chi_{2}(x)^{i_{2}}d\mu_{\Phi^{(\boldsymbol{i})}(a_{1};\psi,G)}\right)
\end{align}
where $\mu_{\Phi^{(\boldsymbol{i})}(a_{1};\psi,G)}\in \mathrm{Hom}_{\mathcal{O}_{\K}}(C(\Gamma_{2},\mathcal{O}_{\K}),\mathcal{O}_{\K}[[X]][[q]])$ is the inverse image of $\Phi^{(\boldsymbol{i})}(a_{1};\psi,G)$ by the isomorphism $\mathrm{Hom}_{\mathcal{O}_{K}}(C(\Gamma_{2},\mathcal{O}_{\K}),\mathcal{O}_{\K}[[X]][[q]])\stackrel{\sim}{\rightarrow}\mathcal{O}_{\K}[[\Gamma_{2}]][[X]][[q]]$ defined in \eqref{hom banach iwasawa continuous}.

\begin{pro}\label{proposition for interpolation of mu[r,s]}
Let $N$ and $N^{\prime}$ be positive integers which are prime to $p$. We denote by $M$ the least common multiple of $N$ and $N^{\prime}$. Take a character $\psi$ on $\Delta_{N}\times \left(\Gamma_{1}\slash \Gamma_{1}^{p^{m(\psi)}}\right)$ with $m(\psi)\in \mathbb{Z}_{\geq 0}$ and $G\in S(N^{\prime}p,\xi,\mathcal{O}_{\K}[[\Gamma_{2}]])$. Let $\phi_{1}$ be a character on $\Delta\times \Gamma_{1}\slash  \Gamma_{1}^{p^{m_{1}}}$ and $\phi_{2}$ a character on $\Gamma_{2}\slash \Gamma_{2}^{p^{m_{2}}}$ with $\boldsymbol{m}\in \mathbb{Z}_{\geq 0}^{2}$. Let $\boldsymbol{i}\in \mathbb{Z}_{\geq 0}^{2}$ be an element satisfying $i_{2}\geq 2$ and $i_{1}+i_{2}<k$. 
We denote by $\kappa_{i_{2},\phi_{2}}\in \mathfrak{X}_{\mathcal{O}_{\K}[[\Gamma_{2}]]}$ the arithmetic specialization induced by 
the arithmetic character $\chi_2^{i_{2}}\phi_{2} : \Gamma_2 \longrightarrow \overline{\mathbb{Q}}^\times_p$. 
Then the following statements hold true. 
\begin{enumerate}
\item If $0\leq i_{1}<\frac{1}{2}(k-i_{2})$, we have
\begin{multline}\label{proposition for interpolation of mu[r,s] equality}
\displaystyle{\sum_{(a_{1},a_{2})\in (\Delta\times \Gamma_{1}\slash  \Gamma_{1}^{p^{m_{1}}})\times \Gamma_{2}\slash \Gamma_{2}^{p^{m_{2}}}}}\phi_{1}(a_{1})\phi_{2}(a_{2})\phi^{(\boldsymbol{i})}((a_{1},a_{2});\psi,G) \\ =(-1)^{i_{1}}
T_{p}
\left(\kappa_{i_{2},\phi_{2}}(G)\otimes\phi_{1}\delta^{(i_{1})}_{k-2i_{1}-i_{2}}\left(F_{k-2i_{1}-i_{2}}(\boldsymbol{1},\psi\xi^{-1}\phi_{1}^{-2}\omega^{i_{2}}\phi_{2}^{-1})\right)\right),
\end{multline}
\item If $\frac{1}{2}(k-i_{2})\leq i_{1}<k-i_{2}$, we have
\begin{multline}\label{proposition for interpolation of mu[r,s] equality 2}
\displaystyle{\sum_{(a_{1},a_{2})\in (\Delta\times \Gamma_{1}\slash  \Gamma_{1}^{p^{m_{1}}})\times \Gamma_{2}\slash 
\Gamma_{2}^{p^{m_{2}}}}}\phi_{1}(a_{1})\phi_{2}(a_{2})\phi^{(\boldsymbol{i})}((a_{1},a_{2});\psi,G) \\  =(-1)^{i_{1}}
T_{p}
\left(\kappa_{i_{2},\phi_{2}}(G)\otimes\phi_{1}\delta^{(k-i_{1}-i_{2}-1)}_{i_{2}-k+2i_{1}+2}\left(F_{i_{2}-k+2i_{1}+2}(\psi\xi^{-1}\phi_{1}^{-2}\omega^{i_{2}}\phi_{2}^{-1},\boldsymbol{1})\right)\right).
\end{multline}
\end{enumerate}
Here, $\boldsymbol{1}$ is the trivial character modulo $1$, the elements $F_{k-2i_{1}-i_{2}}(\boldsymbol{1},\psi\xi^{-1}\phi_{1}^{-2}\omega^{i_{2}}\phi_{2}^{-1})$ and 
$F_{i_{2}-k+2i_{1}+2}(\psi\xi^{-1}\phi_{1}^{-2}\omega^{i_{2}}\phi_{2}^{-1},\boldsymbol{1})$ 
are the $q$-expansions of the Eisenstein series defined in \eqref{definition of eisenstein seriesnonap} and \eqref{another Eisenstein seriesnonap}, 
$\delta^{(r)}_{k}$ is the differential operator defined in \eqref{shimura operator} for each integer $k$ and for each non-negative integer $r$. 
Further, we have 
\begin{equation}\label{proposition for interpolation of mu[r,s]3}
\phi^{(\boldsymbol{i})}((a_{1},a_{2});\psi,G)\in N_{k}^{\leq \lfloor\frac{k-1}{2}\rfloor,\mathrm{cusp}}(Mp^{m_{\psi}(\boldsymbol{m})},\psi;\mathcal{O}_{\K})
\end{equation}
for each $(a_{1},a_{2})\in (\Delta\times \Gamma_{1}\slash  \Gamma_{1}^{p^{m_{1}}})\times \Gamma_{2}\slash \Gamma_{2}^{p^{m_{2}}}$ with $\boldsymbol{m}\in \mathbb{Z}_{\geq 0}^{2}$ where $m_{\psi}(\boldsymbol{m})=\max\{2m_{1}+1,m_{2},m(\psi)+1\}$.
\end{pro}
\begin{proof}
For each $\boldsymbol{i}\in \mathbb{Z}_{\geq 0}^{2}$ such that $i_{2}\geq 2$ and $i_{1}+i_{2}<k$ and for each $(a_{1},a_{2})\in (\Delta\times \Gamma_{1}\slash  \Gamma_{1}^{p^{m_{1}}})\times \Gamma_{2}\slash \Gamma_{2}^{p^{m_{2}}}$ with $\boldsymbol{m}\in \mathbb{Z}_{\geq 0}^{2}$, we have
\begin{align}\label{proposition forminterpolation fo mu[r,s] fourier inverse}
\begin{split}
\phi^{(\boldsymbol{i})}((a_{1},a_{2});\psi,G)=&\frac{1}{
\# C_{1}\# C_{2}
}\sum_{(\phi_{1},\phi_{2})\in C_{1}\times C_{2}}\phi_{1}^{-1}(a_{1})\phi_{2}^{-1}(a_{2})\\
&\displaystyle{\sum_{(b_{1},b_{2})\in (\Delta\times \Gamma_{1}\slash  \Gamma_{1}^{p^{m_{1}}})\times \Gamma_{2}\slash \Gamma_{2}^{p^{m_{2}}}}}\phi_{1}(b_{1})\phi_{2}(b_{2})\phi^{(\boldsymbol{i})}((b_{1},b_{2})_{m_{1},m_{2}};\psi,G)
\end{split}
\end{align}
by the inverse Fourier transform where $C_{1}$ and $C_{2}$ are the groups of characters on $\Delta\times \Gamma_{1}\slash  \Gamma_{1}^{p^{m_{1}}}$ and $\Gamma_{2}\slash \Gamma_{2}^{p^{m_{2}}}$ respectively. By \eqref{level of eisenstein series left1nap} and \eqref{level of eisenstein series right 1nap}, the right-hand sides of \eqref{proposition for interpolation of mu[r,s] equality} and \eqref{proposition for interpolation of mu[r,s] equality 2} are in $N_{k}^{\leq \lfloor\frac{k-1}{2}\rfloor,\mathrm{cusp}}(Mp^{m_{\psi}(\boldsymbol{m})},\psi;\K(\phi_{1},\phi_{2}))$. Then, if we have \eqref{proposition for interpolation of mu[r,s] equality} and \eqref{proposition for interpolation of mu[r,s] equality 2},  by \eqref{proposition forminterpolation fo mu[r,s] fourier inverse}, we have \eqref{proposition for interpolation of mu[r,s]3}. Therefore, it suffices to prove \eqref{proposition for interpolation of mu[r,s] equality} and \eqref{proposition for interpolation of mu[r,s] equality 2}. By definition, we see that 
\begin{align}\label{proposition for interpolation of mu[r,s] of Phiphi1=}
\displaystyle{\sum_{(a_{1},a_{2})\in (\Delta\times \Gamma_{1}\slash  \Gamma_{1}^{p^{m_{1}}})\times \Gamma_{2}\slash \Gamma_{2}^{p^{m_{2}}}}}\phi_{1}(a_{1})\phi_{2}(a_{2})\phi^{(\boldsymbol{i})}((a_{1},a_{2});\psi,G)=(-1)^{i_{1}}T_{p}(\kappa_{i_{2},\phi_{2}}(\Phi_{\phi_{1}})),
\end{align}
where 
\begin{align*}
\Phi_{\phi_{1}}=\sum_{a_{1}\in \Delta\times \Gamma_{1}\slash  \Gamma_{1}^{p^{m_{1}}}}\phi_{1}(a_{1})\Phi^{(\boldsymbol{i})}(a_{1};\psi,G).
\end{align*}
We have 
\small 
\begin{align*}
&\Phi_{\phi_{1}}\\
&=
%
%
%
%
%
\begin{cases}
(G\otimes\phi_{1}) 
\hspace*{-1.1cm}
\displaystyle{\sum_{b\in \Delta_{M}\times (\Gamma_{1}\slash  \Gamma_{1}^{p^{\max\{m_{1},m(\psi)\}}})}} 
\hspace*{-1.1cm} 
(\psi\xi^{-1}\phi_{1}^{-2})(b)\delta^{(i_{1})}_{k-2i_{1}-i_{2}}
(F_{(0,k-2i_{1}-1,0)}(q))
 & \hspace*{-0.2cm}
 \mathrm{if}\ 0\leq i_{1}<\frac{1}{2}(k-i_{2}),\\
(G\otimes\phi_{1}) 
\hspace*{-1.1cm}
\displaystyle{\sum_{b\in \Delta_{M}\times (\Gamma_{1}\slash  \Gamma_{1}^{p^{\max\{m_{1},m(\psi)\}}})}} 
\hspace*{-1.1cm}
(\psi\xi^{-1}\phi_{1}^{-2})(b)\delta_{i_{2}-k+2i_{1}+2}^{(k-i_{1}-i_{2}-1)}(F_{(-k+2i_{1}+1,0,i_{2})}(q))
 & \hspace*{-0.2cm} \mathrm{if}\ \frac{1}{2}(k-i_{2})\leq i_{1}<k-i_{2},
\end{cases}
\end{align*}
\normalsize 
where $G\otimes\phi_{1}=\sum_{n=1}^{+\infty}a_{n}(G)\phi_{1}(n)q^{n}$ 
and we denote $F_{(n_1,n_2,n_3 ),\Gamma_2}(b;Mp^{\max\{m_{1},m(\psi)\}+1})$ by $F_{(n_1 ,n_2,n_3)} (q)$ for short. 
\\ 
$\mathbf{Proof\ of\ (1)}$. Assume that $0\leq i_{1}<\frac{1}{2}(k-i_{2})$.
 By \eqref{definition of eisenstein seriesnonap}, we have 
\begin{multline*}
F_{k-2i_{1}-i_{2}}(\boldsymbol{1},\psi\xi^{-1}\phi_{1}^{-2}\omega^{i_{2}}\phi_{2}^{-1})\\
=C+\kappa_{i_{2},\phi_{2}}\left(\sum_{b\in \Delta_{M}\times (\Gamma_{1}\slash  \Gamma_{1}^{p^{\max\{m_{1},m(\psi)\}}})}(\psi\xi^{-1}\phi_{1}^{-2})(b)F_{(1,k-2i_{1},0),\Gamma_{2}}(b;Mp^{\max\{m_{1},m(\psi)\}+1})\right)
\end{multline*}
where $C=\frac{1}{2}L_{Mp^{\max\{m_{1},m_{2},m(\psi)\}+1}}(1-(k-2i_{1}-i_{2}),\psi\xi^{-1}\phi_{1}^{-2}\omega^{i_{2}}\phi_{2}^{-1})\in \K(\phi_{1},\phi_{2})$. We put  
\begin{align*}
\delta_{k-2i_{1}-i_{2}}^{(i_{1})}\left(F_{k-2i_{1}-i_{2}}(\boldsymbol{1},\psi\xi^{-1}\phi_{1}^{-2}\omega^{i_{2}}\phi_{2}^{-1})\right)&=\sum_{n=0}^{+\infty}c_{n}(X)q^{n}
\end{align*}
with $c_{n}(X)\in \K(\phi_{1},\phi_{2})[X]$ for every $n\geq 0$.  We see that $c_{n}(X)\in \mathcal{O}_{\K}[\phi_{1},\phi_{2}][X]$ for every $n\geq 1$. We have 
$$\kappa_{i_{2},\phi_{2}}(\Phi_{\phi_{1}})=(\kappa_{i_{2},\phi_{2}}(G)\otimes\phi_{1})\sum_{n=1}^{+\infty}c_{n}(X)q^{n}.$$
We put $\kappa_{i_{2},\phi_{2}}(G)\otimes\phi_{1}=\sum_{n=1}^{+\infty}b_{n}q^{n}$, where $b_{n}\in \mathcal{O}_{\K}[\phi_{1},\phi_{2}]$. We have
\begin{align*}
&T_{p}\left((\kappa_{i_{2},\phi_{2}}(G)\otimes\phi_{1})\delta_{k-2i_{1}-i_{2}}^{(i_{1})}\left(F_{k-2i_{1}-i_{2}}(\boldsymbol{1},\psi\xi^{-1}\phi_{1}^{-2}\omega^{i_{2}}\phi_{2}^{-1})\right)\right)\\
&=\sum_{n=1}^{+\infty}\left(\sum_{\substack{(n_{1},n_{2})\in \mathbb{Z}_{\geq 1}\times \mathbb{Z}_{\geq 0}\\
n_{1}+n_{2}=pn}}b_{n_{1}}c_{n_{2}}(pX)\right)q^{n}
\end{align*}
and
\begin{align*}
T_{p}(\kappa_{i_{2},\phi_{2}}(\Phi_{\phi_{1}}))
& =T_{p}\left(\kappa_{i_{2},\phi_{2}}(G)\otimes\phi_{1}\sum_{n=1}^{+\infty}c_{n}(X)q^{n}\right) 
\\
&=\sum_{n=1}^{+\infty}\left(\sum_{\substack{(n_{1},n_{2})\in \mathbb{Z}_{\geq 1}^{2}\\
n_{1}+n_{2}=pn}}b_{n_{1}}c_{n_{2}}(pX)\right)q^{n}.
\end{align*}
Since we have $b_{n}=0$ for every $n\in \mathbb{Z}_{\geq 1}$ satisfying $p\vert n$, we have 
$$ 
\sum_{\substack{(n_{1},n_{2})\in \mathbb{Z}_{\geq 1}\times \mathbb{Z}_{\geq 0}\\
n_{1}+n_{2}=pn}}b_{n_{1}}c_{n_{2}}(pX)=\sum_{\substack{(n_{1},n_{2})\in \mathbb{Z}_{\geq 1}^{2}\\
n_{1}+n_{2}=pn}}b_{n_{1}}c_{n_{2}}(pX)
$$  
for every $n\in \mathbb{Z}_{\geq 1}$.
Thus, we have
\begin{align*}
&T_{p}\left(\kappa_{i_{2},\phi_{2}}(\Phi_{\phi_{1}})\right)\\
&=T_{p}\left(\kappa_{i_{2},\phi_{2}}(G)\otimes\phi_{1}\delta_{k-2i_{1}-i_{2}}^{(i_{1})}\left(F_{k-2i_{1}-i_{2}}(\boldsymbol{1},\psi\xi^{-1}\phi_{1}^{-2}\omega^{i_{2}}\phi_{2}^{-1})\right)\right).
\end{align*}
By \eqref{proposition for interpolation of mu[r,s] of Phiphi1=}, we have
\begin{align*}
&\displaystyle{\sum_{(a_{1},a_{2})\in (\Delta\times \Gamma_{1}\slash  \Gamma_{1}^{p^{m_{1}}})\times \Gamma_{2}\slash \Gamma_{2}^{p^{m_{2}}}}}\phi_{1}(a_{1})\phi_{2}(a_{2})\phi^{(\boldsymbol{i})}((a_{1},a_{2});\psi,G) \\ 
&=(-1)^{i_{1}}T_{p}(\kappa_{i_{2},\phi_{2}}(\Phi_{\phi_{1}}))\\
&=(-1)^{i_{1}}T_{p}\left(\kappa_{i_{2},\phi_{2}}(G)\otimes\phi_{1}\delta^{(i_{1})}_{k-2i_{1}-i_{2}}\left(F_{k-2i_{1}-i_{2}}(\boldsymbol{1},\psi\xi^{-1}\phi_{1}^{-2}\omega^{i_{2}}\phi_{2}^{-1})\right)\right).
\end{align*}
$\mathbf{Proof\ of\ (2)}$. Assume that $\frac{1}{2}(k-i_{2})\leq i_{1}<k-i_{2}$. By \eqref{another Eisenstein seriesnonap}, 
we have 
\begin{multline*}
F_{i_{2}-k+2i_{1}+2}(\psi\xi^{-1}\phi_{1}^{-2}\omega^{i_{2}}\phi_{2}^{-1},\boldsymbol{1})=-\delta_{(\boldsymbol{i})}(\psi\xi^{-1}\phi_{1}^{-2}\omega^{i_{2}}\phi_{2}^{-1})\frac{\varphi(Mp^{\max\{m_{1},m_{2},m(\psi)\}+1})}{2Mp^{\max\{m_{1},m_{2},m(\psi)\}+1}}X\\
+\kappa_{i_{2},\phi_{2}}\left(\sum_{b\in \Delta_{M}\times \Gamma_{1}\slash  \Gamma_{1}^{p^{m_{1}}}}(\psi\xi^{-1}\phi_{1}^{-2})(b)F_{(-k+2i_{1}+1,0,i_{2}),\Gamma_{2}}(b;Mp^{\max\{m_{1},m(\psi)\}+1})\right),
\end{multline*}
where $\varphi$ is the Euler function and $\delta_{(\boldsymbol{i})}(\psi\xi^{-1}\phi_{1}^{-2}\omega^{i_{2}}\phi_{2}^{-1})$ is defined to be $1$ (resp. $0$) when $i_{1}=\frac{1}{2}(k-i_{2})$ and $\psi\xi^{-1}\phi_{1}^{-2}\omega^{i_{2}}\phi_{2}^{-1}$ is trivial character (resp. otherwise). We put
\begin{align*}
\delta_{i_{2}-k+2i_{1}+2}^{(k-i_{1}-i_{2}-1)}\left(F_{i_{2}-k+2i_{1}+2}(\psi\xi^{-1}\phi_{1}^{-2}\omega^{i_{2}}\phi_{2}^{-1},\boldsymbol{1})\right)&=\sum_{n=0}^{+\infty}c_{n}^{\prime}(X)q^{n}
\end{align*}
with $c_{n}^{\prime}(X)\in \K(\phi_{1},\phi_{2})[X]$ for each $n\geq 0$.  We see that $c_{n}^{\prime}(X)\in \mathcal{O}_{\K}[\phi_{1},\phi_{2}][X]$ for every $n\geq 1$. We have 
$$\kappa_{i_{2},\phi_{2}}(\Phi_{\phi_{1}})=(\kappa_{i_{2},\phi_{2}}(G)\otimes\phi_{1})\sum_{n=1}^{+\infty}c_{n}^{\prime}(X)q^{n}.$$
Let $\kappa_{i_{2},\phi_{2}}(G)\otimes\phi_{1}=\sum_{n=1}^{+\infty}b_{n}q^{n}$ with $b_{n}\in \mathcal{O}_{\K}[\phi_{1},\phi_{2}]$. We have
\begin{multline*}
T_{p}\left((\kappa_{i_{2},\phi_{2}}(G)\otimes\phi_{1})\delta_{i_{2}-k+2i_{1}+2}^{(k-i_{1}-i_{2}-1)}\left(F_{i_{2}-k+2i_{1}+2}(\psi\xi^{-1}\phi_{1}^{-2}\omega^{i_{2}}\phi_{2}^{-1},\boldsymbol{1})\right)\right)\\
=\sum_{n=1}^{+\infty}\left(\sum_{\substack{(n_{1},n_{2})\in \mathbb{Z}_{\geq 1}\times \mathbb{Z}_{\geq 0}\\
n_{1}+n_{2}=pn}}b_{n_{1}}c_{n_{2}}^{\prime}(pX)\right)q^{n}
\end{multline*}
and
\begin{align*}
T_{p}(\kappa_{i_{2},\phi_{2}}(\Phi_{\phi_{1}}))
& =T_{p}\left(\kappa_{i_{2},\phi_{2}}(G)\otimes\phi_{1}\sum_{n=1}^{+\infty}c_{n}^{\prime}(X)q^{n}\right)\\
&=\sum_{n=1}^{+\infty}\left(\sum_{\substack{(n_{1},n_{2})\in \mathbb{Z}_{\geq 1}^{2}\\
n_{1}+n_{2}=pn}}b_{n_{1}}c_{n_{2}}^{\prime}(pX)\right)q^{n}.
\end{align*}
Since we have $b_{n}=0$ for every $n\in \mathbb{Z}_{\geq 1}$ satisfying $p \vert n$, we obtain  
$$
\sum_{\substack{(n_{1},n_{2})\in \mathbb{Z}_{\geq 1}\times \mathbb{Z}_{\geq 0}\\
n_{1}+n_{2}=pn}}b_{n_{1}}c_{n_{2}}^{\prime}(pX)=\sum_{\substack{(n_{1},n_{2})\in \mathbb{Z}_{\geq 1}^{2}\\
n_{1}+n_{2}=pn}}b_{n_{1}}c_{n_{2}}^{\prime}(pX)
$$ 
for every $n\in \mathbb{Z}_{\geq 1}$. Thus, we have
\begin{align*}
&T_{p}\left(\kappa_{i_{2},\phi_{2}}(\Phi_{\phi_{1}})\right)\\
&=T_{p}\left((\kappa_{i_{2},\phi_{2}}(G)\otimes\phi_{1})\delta_{i_{2}-k+2i_{1}+2}^{(k-i_{1}-i_{2}-1)}\left(F_{i_{2}-k+2i_{1}+2}(\psi\xi^{-1}\phi_{1}^{-2}\omega^{i_{2}}\phi_{2}^{-1},\boldsymbol{1})\right)\right).
\end{align*}
By \eqref{proposition for interpolation of mu[r,s] of Phiphi1=}, we have 
\begin{align*}
&\displaystyle{\sum_{(a_{1},a_{2})\in (\Delta\times \Gamma_{1}\slash  \Gamma_{1}^{p^{m_{1}}})\times \Gamma_{2}\slash \Gamma_{2}^{p^{m_{2}}}}}\phi_{1}(a_{1})\phi_{2}(a_{2})\phi^{(\boldsymbol{i})}((a_{1},a_{2});\psi,G) \\ 
&=(-1)^{i_{1}}T_{p}(\kappa_{i_{2},\phi_{2}}(\Phi_{\phi_{1}}))\\
&=(-1)^{i_{1}}T_{p}\left((\kappa_{i_{2},\phi_{2}}(G)\otimes\phi_{1})\delta_{i_{2}-k+2i_{1}+2}^{(k-i_{1}-i_{2}-1)}\left(F_{i_{2}-k+2i_{1}+2}(\psi\xi^{-1}\phi_{1}^{-2}\omega^{i_{2}}\phi_{2}^{-1},\boldsymbol{1})\right)\right).
\end{align*}
\end{proof}

For each $\boldsymbol{i}\in \mathbb{Z}^{2}$, we define a continuous group homomorphism
\begin{equation}\label{definition of rboldsymboli group hom}
r^{(\boldsymbol{i})}: (\Delta\times \Gamma_{1})\times \Gamma_{2}\rightarrow \mathcal{O}_{\K}[[(\Delta\times \Gamma_{1})\times \Gamma_{2}]]^{\times}
\end{equation}
to be $r^{(\boldsymbol{i})}((x_{1},x_{2}))=\chi_{1}(x_{1})^{-i_{1}}\chi_{2}(x_{2})^{-i_{2}}[x_{1},x_{2}]$ for each $(x_{1},x_{2})\in (\Delta\times \Gamma_{1})\times \Gamma_{2}$, where $[x_{1},x_{2}]\in \mathcal{O}_{\K}[[(\Delta\times \Gamma_{1})\times \Gamma_{2}]]^{\times}$ is the class of $(x_{1},x_{2})\in (\Delta\times \Gamma_{1})\times \Gamma_{2}$. Then, the above group homomorphism $r^{(\boldsymbol{i})}$ induces a $\K$-algebra isomorphism
\begin{equation}\label{definition of isom rmi}
r_{\boldsymbol{m}}^{(\boldsymbol{i})}: \K[\Delta\times \Gamma_{1}\slash  \Gamma_{1}^{p^{m_{1}}}\times \Gamma_{2}\slash \Gamma_{2}^{p^{m_{2}}}]
\overset{\sim}{\longrightarrow}\frac{\mathcal{O}_{\K}[[(\Delta\times \Gamma_{1})\times \Gamma_{2}]]}{(\Omega_{\boldsymbol{m}}^{[\boldsymbol{i}]})\mathcal{O}_{\K}[[(\Delta\times \Gamma_{1})\times \Gamma_{2}]]}\otimes_{\mathcal{O}_{\K}}\K\end{equation}
for each $\boldsymbol{m}\in \mathbb{Z}_{\geq 0}^{2}$. Let $f\in S_{k}(Np^{m(f)},\psi;\K)$ be a normalized cuspidal Hecke eigenform which is new away from $p$ with $m(f)\in \mathbb{Z}_{\geq 1}$ and $G\in S(N^{\prime}p,\xi;\mathcal{O}_{\K}[[\Gamma_{2}]])$ where $N$ and $N^{\prime}$ are positive integers which are prime to $p$. Assume that $m(f)$ is the smallest positive integer $m$ such that $f\in S_{k}(Np^{m},\psi)$ and $\ord_{p}(a_{p}(f))<\frac{k-1}{2}$. We denote by $M$ the least common multiple of $N$ and $N^{\prime}$. Let $\vert_{[M\slash N^{\prime}]}$ be the operator defined in \eqref{power series and qn mapst qNn}. It is easy to see that $G\vert_{[M\slash N^{\prime}]}\in S(Mp,\xi;\mathcal{O}_{\K}[[\Gamma_{2}]])$. By replacing $G$ with $G\vert_{[M\slash N^{\prime}]}$ in \eqref{definition Psii1i2a1a2}, we can define $\phi^{(\boldsymbol{i})}((a_{1},a_{2});\psi,G\vert_{[M\slash N^{\prime}]})\in \mathcal{O}_{\K}[[X]][[q]]$ for each $(a_{1},a_{2})\in (\Delta\times \Gamma_{1}\slash  \Gamma_{1}^{p^{m_{1}}})\times \Gamma_{2}\slash \Gamma_{2}^{p^{m_{2}}}$ with $\boldsymbol{m}\in \mathbb{Z}_{\geq 0}^{2}$ and we see that $\phi^{(\boldsymbol{i})}((a_{1},a_{2});\psi,G\vert_{[M\slash N^{\prime}]})\in N_{k}^{\leq \lfloor \frac{k-1}{2}\rfloor,\mathrm{cusp}}(Mp^{m_{f}(\boldsymbol{m})},\psi;\mathcal{O}_{\K})$ with $m_{f}(\boldsymbol{m})=\max\{2m_{1}+1,m_{2},m(f)\}$ by Proposition \ref{proposition for interpolation of mu[r,s]}.

We assume that all $M$-th roots of unity and Fourier coefficients of $f^{0}$ are contained in $\K$, where $f^{0}$ is the primitive form associated with $f$. Since  $\ord_{p}(a_{p}(f))<\frac{k-1}{2}$, we see that $a_{p}(f)^{2}\neq \psi_{0}(p)p^{k-1}$ easily. Let $l_{f,M}: \cup_{m=m(f)}^{+\infty}N^{\leq \lfloor \frac{k-1}{2}\rfloor,\mathrm{cusp}}_{k}(Mp^{m},\psi;\K)\rightarrow \K$ be the $\K$-linear map defined in \eqref{classical lf map}. For each $\boldsymbol{m}\in \mathbb{Z}_{\geq 0}^{2}$ and $\boldsymbol{i}\in \mathbb{Z}_{\geq 0}^{2}$ such that $i_{2}\geq 2$ and $i_{1}+i_{2}<k$, 
we define an element $s_{\boldsymbol{m}}^{[\boldsymbol{i}]} \in \frac{\mathcal{O}_{\K}[[(\Delta\times \Gamma_{1})\times \Gamma_{2}]]}{(\Omega_{\boldsymbol{m}}^{[\boldsymbol{i}]})\mathcal{O}_{\K}[[(\Delta\times \Gamma_{1})\times \Gamma_{2}]]}\otimes_{\mathcal{O}_{\K}}\K$ to be 
\begin{equation}\label{definition of qm1m2i1i1}
s_{\boldsymbol{m}}^{[\boldsymbol{i}]}=\sum_{(a_{1},a_{2})\in (\Delta\times \Gamma_{1}\slash  \Gamma_{1}^{p^{m_{1}}})\times \Gamma_{2}\slash \Gamma_{2}^{p^{m_{2}}}}s^{[\boldsymbol{i}]}(a_{1},a_{2})r_{\boldsymbol{m}}^{(\boldsymbol{i})}([a_{1},a_{2}]), 
\end{equation}
where $s^{[\boldsymbol{i}]}(a_{1},a_{2}) \in \K$ is defined by 
\begin{equation}\label{definition of qboldysmboli a12}
s^{[\boldsymbol{i}]}(a_{1},a_{2})=l_{f,M}(\phi^{(\boldsymbol{i})}((a_{1},a_{2});\psi,G\vert_{[M\slash N^{\prime}]})) 
\end{equation}
and $[a_{1},a_{2}]\in \K[\Delta\times \Gamma_{1}\slash  \Gamma_{1}^{p^{m_{1}}}\times \Gamma_{2}\slash \Gamma_{2}^{p^{m_{2}}}]$ is the class of $(a_{1},a_{2})\in (\Delta\times \Gamma_{1}\slash  \Gamma_{1}^{p^{m_{1}}})\times \Gamma_{2}\slash \Gamma_{2}^{p^{m_{2}}}$. 
\subsubsection*{{\bf{Verification of the distribution property of}\ $s_{\boldsymbol{m}}^{[\boldsymbol{i}]}$}}
Let $f\in S_{k}(Np^{m(f)},\psi;\K)$ be a normalized Hecke eigenform which is new away from $p$ with $m(f)\in \mathbb{Z}_{\geq 1}$ and $G\in S(N^{\prime}p,\xi;\linebreak\mathcal{O}_{\K}[[\Gamma_{2}]])$ where $N$ and $N^{\prime}$ are positive integers which are prime to $p$ and $\psi$ and $\xi$ are Dirichlet characters modulo $Np^{m(f)}$ and $N^{\prime}p$ respectively. Assume that $m(f)$ is the smallest positive integer $m$ such that $f\in S_{k}(Np^{m},\psi)$ and $\ord_{p}(a_{p}(f))<\frac{k-1}{2}$. Let $M$ be the common multiple of $N$ and $N^{\prime}$. We assume that all $M$-th roots of unity and Fourier coefficients of $f^{0}$ are contained in $\K$ where $f^{0}$ is the primitive form associated with $f$. We prove the following:
\begin{pro}\label{distribution of mu[boldsymbolr,boldsymbols]}
Let $\boldsymbol{i}\in \mathbb{Z}_{\geq 0}^{2}$ be an element satisfying $i_{2}\geq 2$ and $i_{1}+i_{2}<k$ and 
let $s_{\boldsymbol{m}}^{[\boldsymbol{i}]}   \in \tfrac{\mathcal{O}_{\K}[[(\Delta\times \Gamma_{1})\times \Gamma_{2}]]}{(\Omega_{\boldsymbol{m}}^{[\boldsymbol{i}]})\mathcal{O}_{\K}[[(\Delta\times \Gamma_{1})\times \Gamma_{2}]]}\otimes_{\mathcal{O}_{\K}}\K $ be the element defined in \eqref{definition of qm1m2i1i1} 
for each $\boldsymbol{m}\in \mathbb{Z}_{\geq 0}^{2}$. 
Then, we have 
$(s_{\boldsymbol{m}}^{[\boldsymbol{i}]})_{\boldsymbol{m}\in \mathbb{Z}_{\geq 0}^{2}}\in 
\varprojlim_{\boldsymbol{m}\in \mathbb{Z}_{\geq 0}^{2}}\left(
\tfrac{\mathcal{O}_{\K}[[(\Delta\times \Gamma_{1})\times \Gamma_{2}]]}{ (\Omega_{\boldsymbol{m}}^{[\boldsymbol{i}]})
\mathcal{O}_{\K}[[(\Delta\times \Gamma_{1})\times \Gamma_{2}]]
}
\otimes_{\mathcal{O}_{\K}}\K\right)$. That is, the elements $s_{\boldsymbol{m}}^{[\boldsymbol{i}]}$ form a projective system 
with respect to the index set $ \mathbb{Z}_{\geq 0}^{2}$. 
\end{pro}
\begin{proof}
By \eqref{definition of qm1m2i1i1}, we have
$$s_{\boldsymbol{m}}^{[\boldsymbol{i}]}=r^{(\boldsymbol{i})}_{\boldsymbol{m}}\left(\sum_{(a_{1},a_{2})\in (\Delta\times \Gamma_{1}\slash  \Gamma_{1}^{p^{m_{1}}})\times \Gamma_{2}\slash \Gamma_{2}^{p^{m_{2}}}}s^{[\boldsymbol{i}]}(a_{1},a_{2})[a_{1},a_{2}]\right)$$
for each $\boldsymbol{m}\in \mathbb{Z}_{\geq 0}^{2}$, where $r_{\boldsymbol{m}}^{(\boldsymbol{i})}$ is the isomorphism defined in \eqref{definition of isom rmi}, $s^{[\boldsymbol{i}]}(a_{1},a_{2})\in \K$ is the element defined in \eqref{definition of qboldysmboli a12} and $[a_{1},a_{2}]\in \mathcal{O}_{\K}[(\Delta\times \Gamma_{1}\slash  \Gamma_{1}^{p^{m_{1}}})\times \Gamma_{2}\slash \Gamma_{2}^{p^{m_{2}}}]$ is the class of $(a_{1},a_{2})\in \Delta\times \Gamma_{1}\slash  \Gamma_{1}^{p^{m_{1}}}\times \Gamma_{2}\slash \Gamma_{2}^{p^{m_{2}}}$. If we have
\begin{multline}\label{distribution of mu[boldsymbolr,boldsymbols] eq1}
\left(\sum_{(a_{1},a_{2})\in (\Delta\times \Gamma_{1}\slash  \Gamma_{1}^{p^{m_{1}}})\times \Gamma_{2}\slash \Gamma_{2}^{p^{m_{2}}}}s^{[\boldsymbol{i}]}(a_{1},a_{2})[a_{1},a_{2}]\right)_{\boldsymbol{m}\in \mathbb{Z}_{\geq 0}^{2}}\\
\in \varprojlim_{\boldsymbol{m}\in \mathbb{Z}_{\geq 0}^{2}}\K[(\Delta\times \Gamma_{1}\slash  \Gamma_{1}^{p^{m_{1}}})\times \Gamma_{2}\slash \Gamma_{2}^{p^{m_{2}}}],
\end{multline}
we have $(s_{\boldsymbol{m}}^{[\boldsymbol{i}]})_{\boldsymbol{m}\in \mathbb{Z}_{\geq 0}^{2}}\in 
\varprojlim_{\boldsymbol{m}\in \mathbb{Z}_{\geq 0}^{2}}\left(
\frac{\mathcal{O}_{\K}[[(\Delta\times \Gamma_{1})\times \Gamma_{2}]]}{ (\Omega_{\boldsymbol{m}}^{[\boldsymbol{i}]})
\mathcal{O}_{\K}[[(\Delta\times \Gamma_{1})\times \Gamma_{2}]]
}
\otimes_{\mathcal{O}_{\K}}\K\right)$. 
Let $p_{1,m_{1}}: \Delta\times \Gamma_{1}\slash  \Gamma_{1}^{p^{m_{1}+1}}\rightarrow \Delta\times \Gamma_{1}\slash  \Gamma_{1}^{p^{m_{1}}}$ and $p_{2,m_{2}}: \Gamma_{2}\slash \Gamma_{2}^{p^{m_{2}+1}}\rightarrow \Gamma_{2}\slash \Gamma_{2}^{p^{m_{2}}}$ be the natural projections. 
Then, to prove \eqref{distribution of mu[boldsymbolr,boldsymbols] eq1}, 
it suffices to prove the following equalities: 
\begin{align}\label{distribution of mu[boldsymbolr,boldsymbols]eq 2}
\begin{split}
\sum_{b_{1}\in p_{1,m_{1}}^{-1}(a_{1})}s^{[\boldsymbol{i}]}(b_{1},a_{2})=s^{[\boldsymbol{i}]}(a_{1},a_{2}),\\
\sum_{b_{2}\in p_{2,m_{2}}^{-1}(a_{2})}s^{[\boldsymbol{i}]}(a_{1},b_{2})=s^{[\boldsymbol{i}]}(a_{1},a_{2})
\end{split}
\end{align} 
for each $(a_{1},a_{2})\in (\Delta\times \Gamma_{1}\slash  \Gamma_{1}^{p^{m_{1}}})\times\Gamma_{2}\slash \Gamma_{2}^{p^{m_{2}}}$ with $\boldsymbol{m}\in \mathbb{Z}_{\geq 0}^{2}$.
Let $(a_{1},a_{2})\in (\Delta\times \Gamma_{1}\slash  \Gamma_{1}^{p^{m_{1}}})\times \Gamma_{2}\slash \Gamma_{2}^{p^{m_{2}}}$ with $\boldsymbol{m}\in \mathbb{Z}_{\geq 0}^{2}$. We denote by $M$ the least common multiple of $N$ and $N^{\prime}$. First we prove that $\sum_{b_{2}\in p_{2,m_{2}}^{-1}(a_{2})}s^{[\boldsymbol{i}]}(a_{1},b_{2})=s^{[\boldsymbol{i}]}(a_{1},a_{2})$. By \eqref{definition Psii1i2a1a2} and \eqref{definition of qboldysmboli a12}, we have
$$s^{[\boldsymbol{i}]}(a_{1},a_{2})=(-1)^{i_{1}}l_{f,M}\circ T_{p}\left(\int_{a_{2}\Gamma_{2}^{p^{m_{2}}}}\chi_{2}(x)^{i_{2}}d\mu_{\Phi^{(\boldsymbol{i})}(a_{1};\psi,G\vert_{[M\slash N^{\prime}]})}\right),$$
where $l_{f,M}: \cup_{m=m(f)}^{+\infty}N^{\leq \lfloor \frac{k-1}{2}\rfloor,\mathrm{cusp}}_{k}(Mp^{m},\psi;\K)\rightarrow \K$ is the $\K$-linear map defined in \eqref{classical lf map} and $\mu_{\Phi^{(\boldsymbol{i})}(a_{1};\psi,G\vert_{[M\slash N^{\prime}]})}\in \mathrm{Hom}_{\mathcal{O}_{\K}}(C(\Gamma_{2},\mathcal{O}_{\K}),\mathcal{O}_{\K}[[X]][[q]])$ is the inverse image of $\Phi^{(\boldsymbol{i})}(a_{1};\psi,\linebreak G\vert_{[M\slash N^{\prime}]})$ in \eqref{definition of Phii1ii2(a1;G)} by the isomorphism $\mathrm{Hom}_{\mathcal{O}_{K}}(C(\Gamma_{2},\mathcal{O}_{\K}),\mathcal{O}_{\K}[[X]][[q]])\stackrel{\sim}{\rightarrow}\mathcal{O}_{\K}[[\Gamma_{2}]][[X]][[q]]$ defined in \eqref{hom banach iwasawa continuous}. We have
\begin{align*}
&\sum_{b_{2}\in p_{2,m_{2}}^{-1}(a_{2})}s^{[\boldsymbol{i}]}(a_{1},b_{2})\\
&=(-1)^{i_{1}}l_{f,M}\circ T_{p}\left(\sum_{b_{2}\in p_{2,m_{2}}^{-1}(a_{2})}\int_{b_{2}\Gamma_{2}^{p^{m_{2}+1}}}\chi_{2}(x)^{i_{2}}d\mu_{\Phi^{(\boldsymbol{i})}(a_{1};\psi,G\vert_{[M\slash N^{\prime}]})}\right)\\
&=(-1)^{i_{1}}l_{f,M}\circ T_{p}\left(\int_{b_{2}\Gamma_{2}^{p^{m_{2}}}}\chi_{2}(x)^{i_{2}}d\mu_{\Phi^{(\boldsymbol{i})}(a_{1};\psi,G\vert_{[M\slash N^{\prime}]})}\right)\\
&=s^{[\boldsymbol{i}]}(a_{1},a_{2}).
\end{align*}
Next, we prove that
$\sum_{b_{1}\in p_{1,m_{1}}^{-1}(a_{1})}s^{[\boldsymbol{i}]}(b_{1},a_{2})=s^{[\boldsymbol{i}]}(a_{1},a_{2})$. By \eqref{definition of Phii1ii2(a1;G)}, we have
\begin{align*}
&\sum_{b_{1}\in p_{1,m_{1}}^{-1}(a_{1})}\Phi^{(\boldsymbol{i})}(b_{1};\psi,G\vert_{[M\slash N^{\prime}]}) \\
& =\sum_{b\in \Delta\times (\Gamma_{1}\slash  \Gamma_{1}^{p^{m_{1}+1}})}\sum_{\substack{c\in \Delta_{M}\times (\Gamma_{1}\slash  \Gamma_{1}^{p^{\max\{m_{1}+1,m(f)-1\}}})\\ p_{m_{1}+1}^{(M)}(c)=b}} (\psi\xi^{-1})(c)H_{c}^{(\boldsymbol{i})}\\
&\sum_{b_{1}\in p_{1,m_{1}}^{-1}(a_{1})} (G\vert_{[M\slash N^{\prime}]})_{\equiv b_{1}b^{2} (p^{m_{1}+2})}\\
& =\sum_{b\in \Delta\times \Gamma_{1}\slash  \Gamma_{1}^{p^{m_{1}}}}(G\vert_{[M\slash N^{\prime}]})_{\equiv a_{1}b^{2} (p^{m_{1}+1})}\sum_{\substack{c\in \Delta_{M}\times (\Gamma_{1}\slash  \Gamma_{1}^{p^{\max\{m_{1}+1,m(f)-1\}}})\\ p_{1,m_{1}}p_{m_{1}+1}^{(M)}(c)=b}}(\psi\xi^{-1})(c)H_{c}^{(\boldsymbol{i})}\\
&=\sum_{b\in \Delta\times \Gamma_{1}\slash  \Gamma_{1}^{p^{m_{1}}}}(G\vert_{[M\slash N^{\prime}]})_{\equiv a_{1}b^{2} (p^{m_{1}+1})}\sum_{\substack{c\in \Delta_{M}\times (\Gamma_{1}\slash  \Gamma_{1}^{p^{\max\{m_{1},m(f)-1\}}})\\ p_{m_{1}}^{(M)}(c)=b}}(\psi\xi^{-1})(c)H_{c}^{(\boldsymbol{i})}\\ 
& =\Phi^{(\boldsymbol{i})}(a_{1};\psi,G\vert_{[M\slash N^{\prime}]})
\end{align*}
where $p_{m}^{(M)}: \Delta_{M}\times \Gamma_{1}\slash  \Gamma_{1}^{p^{\max\{m,m(f)-1\}}}\rightarrow \Delta\times \Gamma_{1}\slash  \Gamma_{1}^{p^{m}}$ is the natural projection for each $m\in \mathbb{Z}_{\geq 0}$ and $H_{c}^{(\boldsymbol{i})}$ is the element defined in \eqref{definition of Wb(i1,i2)}. We see that 
\begin{align*}
&\sum_{b_{1}\in p_{1,m_{1}}^{-1}(a_{1})}s^{[\boldsymbol{i}]}(b_{1},a_{2}) \\ 
&=(-1)^{i_{1}}l_{f,M}\circ T_{p}\left(\int_{a_{2}\Gamma_{2}^{p^{m_{2}}}}\chi_{2}(x)^{i_{2}}\sum_{b_{1}\in p_{1,m_{1}}^{-1}(a_{1})}d\mu_{\Phi^{(\boldsymbol{i})}(b_{1};\psi,G\vert_{[M\slash N^{\prime}]})}\right)\\
&=(-1)^{i_{1}}l_{f,M}\circ T_{p}\left(\int_{a_{2}\Gamma_{2}^{p^{m_{2}}}}\chi_{2}(x)^{i_{2}}d\mu_{\Phi^{(\boldsymbol{i})}(a_{1};\psi,G\vert_{[M\slash N^{\prime}]})}\right)\\
&=s^{[\boldsymbol{i}]}(a_{1},a_{2}).
\end{align*}
We have proved \eqref{distribution of mu[boldsymbolr,boldsymbols]eq 2}, which completes the proof of the proposition. 
\end{proof}

\subsubsection*{{\bf{Verification of the admissible condition of}\ $s_{\boldsymbol{m}}^{[\boldsymbol{i}]}$}}

Let $\boldsymbol{r},\boldsymbol{s}\in \mathbb{Z}^{2}$ be elements satisfying $\boldsymbol{s}\geq \boldsymbol{r}$, $[\boldsymbol{r},\boldsymbol{s}]\subset [\boldsymbol{d},\boldsymbol{e}]$ and $s_{1}+s_{2}<k$, where $\boldsymbol{d}=(0,2)$ and $\boldsymbol{e}=(k-3,k-1)$. Let $f\in S_{k}(Np^{m(f)},\psi;\K)$ be a normalized Hecke eigenform which is new away from $p$ with $m(f)\in \mathbb{Z}_{\geq 1}$ and let $G\in S(N^{\prime}p,\xi;\mathcal{O}_{\K}[[\Gamma_{2}]])$ where $N$ and $N^{\prime}$ are positive integers which are prime to $p$ and $\psi$ and $\xi$ are Dirichlet characters modulo $Np^{m(f)}$ and $N^{\prime}p$ respectively. Assume that $m(f)$ is the smallest positive integer $m$ such that $f\in S_{k}(Np^{m},\psi)$. Let $M$ be the least common multiple of $N$ and $N^{\prime}$. We assume that $\alpha=\ord_{p}(a_{p}(f))<\frac{k-1}{2}$ and all $M$-th roots of unity and Fourier coefficients of $f^{0}$ are contained in $\K$ where $f^{0}$ is the primitive form associated with $f$. Let $R_{1,m_{1}}\subset \Delta\times \Gamma_{1}$ (resp. $R_{2,m_{2}}\subset \Gamma_{2}$) be a complete set of representatives of 
$\Delta\times \Gamma_{1}\slash  \Gamma_{1}^{p^{m_{1}}}$ (resp. $\Gamma_{2}\slash \Gamma_{2}^{p^{m_{2}}}$) 
for each $\boldsymbol{m}\in \mathbb{Z}_{\geq 0}^{2}$. For each $\boldsymbol{m}\in \mathbb{Z}_{\geq 0}^{2}$ and for each $\boldsymbol{i}\in [\boldsymbol{r},\boldsymbol{s}]$, we define  $\tilde{s}_{\boldsymbol{m}}^{[\boldsymbol{i}]}\in \mathcal{O}_{\K}[[(\Delta\times \Gamma_{1})\times \Gamma_{2}]]\otimes_{\mathcal{O}_{\K}}\K$ to be
\begin{equation}\label{lift of qboldsymbolmboldsymboli}
\tilde{s}_{\boldsymbol{m}}^{[\boldsymbol{i}]}=\sum_{(a_{1},a_{2})\in R_{1,m_{1}}\times R_{2,m_{2}}}s^{[\boldsymbol{i}]}([a_{1}]_{m_{1}},[a_{2}]_{m_{2}})r^{(\boldsymbol{i})}((a_{1},a_{2}))
\end{equation}
where $[a_{1}]_{m_{1}}\in \Delta\times \Gamma_{1}\slash  \Gamma_{1}^{p^{m_{1}}}$ (resp. $[a_{2}]_{m_{2}}\in \Gamma_{2}\slash \Gamma_{2}^{p^{m_{2}}}$) is the class of $a_{1}\in R_{1,m_{1}}$ (resp. $a_{2}\in R_{2,m_{2}}$), $s^{[\boldsymbol{i}]}([a_{1}]_{m_{1}},[a_{2}]_{m_{2}})\in \K$ is the element defined in \eqref{definition of qboldysmboli a12}, $r^{(\boldsymbol{i})}$ is the group homomorphism defined in \eqref{definition of rboldsymboli group hom}. By definition, $\tilde{s}_{\boldsymbol{m}}^{[\boldsymbol{i}]}$ is a lift of $s_{\boldsymbol{m}}^{[\boldsymbol{i}]}$ in \eqref{definition of qm1m2i1i1}. Put $\boldsymbol{h}=(2\alpha,\alpha)$.
\begin{pro}\label{admissible condition of p-adic l}
Let $f\in S_{k}(Np^{m(f)},\psi;\K)$ be a normalized cuspidal Hecke eigenform which is new away from $p$ with $m(f)\in \mathbb{Z}_{\geq 1}$. We assume that $\alpha=\ord_{p}(a_{p}(f))<\frac{k-1}{2}$ and all $M$-th roots of unity and Fourier coefficients of $f^{0}$ are contained in $\K$ where $f^{0}$ is the primitive form associated with $f$.
Let $\boldsymbol{r},\boldsymbol{s}\in \mathbb{Z}^{2}$ be elements satisfying $\boldsymbol{s}\geq \boldsymbol{r}$, $[\boldsymbol{r},\boldsymbol{s}]\subset [\boldsymbol{d},\boldsymbol{e}]$ and $s_{1}+s_{2}<k$. There exists a non-negative integer $n^{[\boldsymbol{r},\boldsymbol{s}]}(f)$ depending only on $f$ and $[\boldsymbol{r}, \boldsymbol{s}]$ 
which satisfies
\begin{multline}\label{admissible condition of p-adic l statement equation}
p^{\langle \boldsymbol{m},\boldsymbol{h}-(\boldsymbol{j}-\boldsymbol{r})\rangle_{2}}\displaystyle{\sum_{\boldsymbol{i}\in [\boldsymbol{r},\boldsymbol{j}]}}\left(\prod_{t=1}^{2}\begin{pmatrix}j_{t}-r_{t}\\i_{t}-r_{t}\end{pmatrix}\right)(-1)^{\sum_{t=1}^{2}(j_{t}-i_{t})}\tilde{s}_{\boldsymbol{m}}^{[\boldsymbol{i}]}\\
\in \mathcal{O}_{\K}[[(\Delta\times \Gamma_{1})\times \Gamma_{2}]]\otimes_{\mathcal{O}_{\K}}p^{-n^{[\boldsymbol{r},\boldsymbol{s}]}(f)}\mathcal{O}_{\K}
\end{multline}
for every $\boldsymbol{m}\in \mathbb{Z}_{\geq 0}^{2}$ and for every $\boldsymbol{j}\in [\boldsymbol{r},\boldsymbol{s}]$ where $\tilde{s}_{\boldsymbol{m}}^{[\boldsymbol{i}]}$ is the element defined in \eqref{lift of qboldsymbolmboldsymboli}.
\end{pro}
\begin{proof}
Assume that $m(f)$ is the smallest positive integer $m$ such that $f\in S_{k}(Np^{m},\psi)$. Denote by $M$ the least common multiple of $N$ and $N^{\prime}$. For each $(a_{1},a_{2})\in R_{1,m_{1}}\times R_{2,m_{2}}$ with $\boldsymbol{m}\in \mathbb{Z}_{\geq 0}^{2}$ and for each $\boldsymbol{j}\in [\boldsymbol{r},\boldsymbol{s}]$, we put
\begin{multline*}
\theta^{(\boldsymbol{j})}_{\boldsymbol{m}}(a_{1},a_{2})=\displaystyle{\sum_{\boldsymbol{i}\in [\boldsymbol{r},\boldsymbol{j}]}}\left(\prod_{t=1}^{2}\begin{pmatrix}j_{t}-r_{t}\\i_{t}-r_{t}\end{pmatrix}\right)(-1)^{i_{1}}\int_{a_{2}\Gamma_{2}^{p^{m_{2}}}}\chi_{2}(x)^{i_{2}}d\mu_{\Phi^{(\boldsymbol{i})}(a_{1};\psi,G\vert_{[M\slash N^{\prime}]})}\\
(-\chi_{1}(a_{1}))^{j_{1}-i_{1}}(-\chi_{2}(a_{2}))^{j_{2}-i_{2}}
\end{multline*}
where $\Phi^{(\boldsymbol{i})}(a_{1};\psi,G\vert_{[M\slash N^{\prime}]})$ is the element defined in \eqref{definition of Phii1ii2(a1;G)} and the measure 
$\mu_{\Phi^{(\boldsymbol{i})}(a_{1};\psi,G\vert_{[M\slash N^{\prime}]})}\linebreak\in \mathrm{Hom}_{\mathcal{O}_{\K}}(C(\Gamma_{2},\mathcal{O}_{\K}),\mathcal{O}_{\K}[[X]][[q]])$ is the inverse image of the element $\Phi^{(\boldsymbol{i})}(a_{1};\psi,G\vert_{[M\slash N^{\prime}]})\in \mathcal{O}_{\K}[[\Gamma_{2}]][[X]][[q]]$ by the isomorphism $\mathrm{Hom}_{\mathcal{O}_{\K}}(C(\Gamma_{2},\mathcal{O}_{\K}),\mathcal{O}_{\K}[[X]][[q]])\stackrel{\sim}{\rightarrow}\mathcal{O}_{\K}[[\Gamma_{2}]][[X]][[q]]$ defined in \eqref{hom banach iwasawa continuous}. Let $T_{p}$ be the operator defined in \eqref{formal hecke operator and delta operator}. By the definition of $\theta^{(\boldsymbol{j})}_{\boldsymbol{m}}(a_{1},a_{2})$, we have
\begin{multline*}T_{p}(\theta^{(\boldsymbol{j})}_{\boldsymbol{m}}(a_{1},a_{2}))\\
=\displaystyle{\sum_{\boldsymbol{i}\in [\boldsymbol{r},\boldsymbol{j}]}}\left(\prod_{t=1}^{2}\begin{pmatrix}j_{t}-r_{t}\\i_{t}-r_{t}\end{pmatrix}\right)\phi^{(\boldsymbol{i})}((a_{1},a_{2});\psi,G\vert_{[M\slash N^{\prime}]})(-\chi_{1}(a_{1}))^{j_{1}-i_{1}}(-\chi_{2}(a_{2}))^{j_{2}-i_{2}}
\end{multline*}
where $\phi^{(\boldsymbol{i})}((a_{1},a_{2});\psi,G\vert_{[M\slash N^{\prime}]})$ is the element defined in \eqref{definition Psii1i2a1a2}. By Proposition \ref{proposition for interpolation of mu[r,s]}, we have $T_{p}(\theta^{(\boldsymbol{j})}_{\boldsymbol{m}}(a_{1},a_{2})) \in N_{k}^{\leq \lfloor\frac{k-1}{2}\rfloor,\mathrm{cusp}}(Mp^{m_{f}(\boldsymbol{m})},\psi;\mathcal{O}_{\K})$ with $m_{f}(\boldsymbol{m})=\max\{2m_{1}+1,m_{2},m(f)\}$. By \eqref{definition of qboldysmboli a12}, we have
\begin{multline}\label{admissible prop display prodt=12j1 sum R1m1 R2m2}
\displaystyle{\sum_{\boldsymbol{i}\in [\boldsymbol{r},\boldsymbol{j}]}}\left(\prod_{t=1}^{2}\begin{pmatrix}j_{t}-r_{t}\\i_{t}-r_{t}\end{pmatrix}\right)(-1)^{\sum_{t=1}^{2}(j_{t}-i_{t})}\tilde{s}_{\boldsymbol{m}}^{[\boldsymbol{i}]}\\
=\sum_{(a_{1},a_{2})\in R_{1,m_{1}}\times R_{2,m_{2}}}\chi_{1}(a_{1})^{-j_{1}}\chi_{2}(a_{2})^{-j_{2}}l_{f,M}^{(m_{f}(\boldsymbol{m}))}\left(T_{p}(\theta^{(\boldsymbol{j})}_{\boldsymbol{m}}(a_{1},a_{2}))\right)[a_{1},a_{2}]
\end{multline}
where $l_{f,M}^{(m)}: N^{\leq \lfloor \frac{k-1}{2}\rfloor,\mathrm{cusp}}_{k}(Mp^{m},\psi;\K)\rightarrow \K$ is the $\K$-linear map defined in \eqref{definition of lf,L(n)} for each positive integer $m$ such that $m\geq m(f)$ and $[a_{1},a_{2}]\in \mathcal{O}_{\K}[[(\Delta\times \Gamma_{1})\times \Gamma_{2}]]$ is the class of $(a_{1},a_{2})\in (\Delta\times \Gamma_{1})\times \Gamma_{2}$. By the definition of  $l_{f,M}^{(m_{f}(\boldsymbol{m}))}$, we have 
\begin{equation}\label{admissible prop lfmcfyomgeq firs lmftptheta}
l_{f,M}^{(m_{f}(\boldsymbol{m}))}T_{p}(\theta^{(\boldsymbol{j})}_{\boldsymbol{m}}(a_{1},a_{2}))=a_{p}(f)^{-(m_{f}(\boldsymbol{m})-m(f))}l_{f,M}^{(m(f))}T_{p}^{m_{f}(\boldsymbol{m})-m(f)+1}(\theta^{(\boldsymbol{j})}_{\boldsymbol{m}}(a_{1},a_{2}))
\end{equation}
 for every $\boldsymbol{j}\in[\boldsymbol{r},\boldsymbol{s}]$ and for every $(a_{1},a_{2})\in R_{1,m_{2}}\times R_{2,m_{2}}$ with $\boldsymbol{m}\in \mathbb{Z}_{\geq 0}^{2}$. We regard $N_{k}^{\lfloor\frac{k-1}{2}\rfloor,\mathrm{cusp}}(Mp^{m(f)},\psi;\K)$ as a $\K$-Banach space by the valuation $v_{N_{k}^{\leq\lfloor\frac{k-1}{2}\rfloor}(Mp^{m(f)},\psi)}$ defined in \eqref{valuation of nearly holomorphic modular forms}. Since $l_{f,M}^{(m(f))}: N_{k}^{\lfloor\frac{k-1}{2}\rfloor,\mathrm{cusp}}(Mp^{m(f)},\psi;\K)\rightarrow \K$ is bounded, we have $v_{\mathfrak{L}}(l_{f,M}^{(m(f))})>-\infty$ where $v_{\mathfrak{L}}$ is the valuation defined by the condition \eqref{valuation of bounded operator}. Let $\alpha=\ord_{p}(a_{p}(f))$. By \eqref{admissible prop lfmcfyomgeq firs lmftptheta}, we see that 
\begin{multline}\label{p-adic admissible prop in T}
\ord_{p}\left(l_{f,M}^{(m_{f}(\boldsymbol{m}))}T_{p}(\theta^{(\boldsymbol{j})}_{\boldsymbol{m}}(a_{1},a_{2}))\right)
\\ 
\geq -(m_{f}(\boldsymbol{m})-m(f))\alpha+v_{\mathfrak{L}}(l_{f,M}^{(m(f))})+v_{N_{k}^{\leq\lfloor\frac{k-1}{2}\rfloor}(Mp^{m(f)},\psi)}(T_{p}^{m_{f}(\boldsymbol{m})-m(f)+1}(\theta^{(\boldsymbol{j})}_{\boldsymbol{m}}(a_{1},a_{2}))),
\end{multline}
for every $\boldsymbol{j}\in [\boldsymbol{r},\boldsymbol{s}]$ and for every $(a_{1},a_{2})\in R_{1,m_{1}}\times R_{2,m_{2}}$ with $\boldsymbol{m}\in \mathbb{Z}_{\geq 0}^{2}$.
Let $\iota: \K[[X]][[q]]\rightarrow K[[q]]$ be the $\K$-linear homomorphism defined in \eqref{definition of formal iota and d}. Denote by $\iota_{m(f)}$ the restriction of $\iota$ on $N_{k}^{\leq\lfloor \frac{k-1}{2}\rfloor,\mathrm{cusp}}(Mp^{m(f)},\psi;\K)$. {By Proposition \ref{classical ver of urban}, the map $\iota_{m(f)}:N_{k}^{\leq\lfloor \frac{k-1}{2}\rfloor,\mathrm{cusp}}(Mp^{m(f)},\psi;\K)\rightarrow \K[[q]]$ induces a $\K$-linear isomorphism 
$$
\iota_{m(f)}: N_{k}^{\leq\lfloor \frac{k-1}{2}\rfloor,\mathrm{cusp}}(Mp^{m(f)},\psi;\K)\stackrel{\sim}{\rightarrow}\iota_{m(f)}\left(N_{k}^{\leq\lfloor \frac{k-1}{2}\rfloor,\mathrm{cusp}}(Mp^{m(f)},
\psi;\K)\right).
$$
} 
By the diagram \eqref{formal commutative delta and d and Tp and Up}, 
we have 
\begin{equation}\label{amidssible prop tpm-cf=iotacfTpformal m-cf}
T_{p}^{m_{f}(\boldsymbol{m})-m(f)+1}(\theta^{(\boldsymbol{j})}_{\boldsymbol{m}}(a_{1},a_{2}))=\iota_{m(f)}^{-1}T_{p}^{m_{f}(\boldsymbol{m})-m(f)+1}\iota(\theta^{(\boldsymbol{j})}_{\boldsymbol{m}}(a_{1},a_{2})).
\end{equation}
We regard $\iota_{m(f)}\left(N_{k}^{\leq\lfloor \frac{k-1}{2}\rfloor,\mathrm{cusp}}(Mp^{m(f)},\psi;\K)\right)$ as a $\K$-Banach space by the valuation $v_{\iota_{m(f)}}$ defined by
\begin{equation}
v_{\iota_{m(f)}}(g)=\inf_{n\in \mathbb{Z}_{\geq 1}}\{\ord_{p}(a_{n}(g))\}
\end{equation}
for each $g=\sum_{n=1}^{+\infty}a_{n}(g)q^{n}\in \iota_{m(f)}\left(N_{k}^{\leq\lfloor \frac{k-1}{2}\rfloor,\mathrm{cusp}}(Mp^{m(f)},\psi;\K)\right)$. We note that $\iota_{m(f)}^{-1}$ is bounded. Then, by \eqref{p-adic admissible prop in T} and \eqref{amidssible prop tpm-cf=iotacfTpformal m-cf}, we have 
\begin{multline}\label{p-adic admissible prop in T iota}
\ord_{p}\left(l_{f,M}^{(m_{f}(\boldsymbol{m}))}T_{p}(\theta^{(\boldsymbol{j})}_{\boldsymbol{m}}(a_{1},a_{2}))\right) \\ 
\geq -(m_{f}(\boldsymbol{m})-m(f))\alpha+v_{\mathfrak{L}}(l_{f,M}^{(m(f))})+v_{\mathfrak{L}}(\iota_{m(f)}^{-1})+v_{\iota_{m(f)}}(T_{p}^{m_{f}(\boldsymbol{m})-m(f)+1}\iota(\theta^{(\boldsymbol{j})}_{\boldsymbol{m}}(a_{1},a_{2})))
\end{multline}
for every $\boldsymbol{j}\in [\boldsymbol{r},\boldsymbol{s}]$ and for every $(a_{1},a_{2})\in R_{1,m_{1}}\times R_{2,m_{2}}$ with $\boldsymbol{m}\in \mathbb{Z}_{\geq 0}^{2}$. 
We define a power series $\Phi^{(\boldsymbol{i})}_{\iota}([a_{1}]_{m_{1}};\psi,G\vert_{[M\slash N^{\prime}]})\in \mathcal{O}_{\K}[[\Gamma_{2}]][[q]]$ to be
\small 
\begin{align*}
&\Phi^{(\boldsymbol{i})}_{\iota}([a_{1}]_{m_{1}};\psi,G\vert_{[M\slash N^{\prime}]})=\sum_{b\in \Delta\times \Gamma_{1}\slash  \Gamma_{1}^{p^{m_{1}}}}(G\vert_{[M\slash N^{\prime}]})_{\equiv [a_{1}]_{m_{1}}b^{2} (p^{m_{1}+1})}\times
\\ 
&\displaystyle{
\begin{cases}
\displaystyle{\sum_{\substack{c\in \Delta_{M}\times (\Gamma_{1}\slash  \Gamma_{1}^{p^{\max\{m_{1},m(f)-1\}}})\\ p_{m_{1}}^{(M)}(c)=b}}} (\psi\xi^{-1})(c)d^{i_{1}}\left(F_{(0,k-2i_{1}-1,0)}(q)\right)\ &\mathrm{if}\ 0\leq i_{1}<\frac{1}{2}(k-i_{2}),
\\
\displaystyle{\sum_{\substack{c\in \Delta_{M}\times (\Gamma_{1}\slash  \Gamma_{1}^{p^{\max\{m_{1},m(f)-1\}}})\\ p_{m_{1}}^{(M)}(c)=b}}} (\psi\xi^{-1})(c)d^{k-i_{1}-i_{2}-1}
\left( F_{(-k+2i_{1}+1,0,i_{2})}(q) \right)\ &\mathrm{if}\ \frac{1}{2}(k-i_{2})\leq i_{1}<k-i_{2},
\end{cases}}
\end{align*}
\normalsize 
where we denote the power series $F_{(n_1,n_2,n_3 ),\Gamma_2}(c;Mp^{\max\{m_{1}+1,m(f)\}})$ defined in \eqref{Definition of Fm,a} by $F_{n_1 ,n_2,n_3} (q)$ for short. 
The map $p_{m}^{(M)}: \Delta_{M}\times\Gamma_{1}\slash \Gamma_{1}^{p^{\max\{m,m(f)-1\}}}\rightarrow \Delta\times \Gamma_{1}\slash \Gamma_{1}^{p^{m}}$ is the natural projection and $d:\mathcal{O}_{\K}[[\Gamma_{2}]][[q]]\rightarrow \mathcal{O}_{\K}[[\Gamma_{2}]][[q]]$ is the operator defined by $d=q\frac{d}{dq}$. By the diagram of \eqref{formal commutative delta and d and Tp and Up}, we see that 
\begin{equation}\label{admiss pro image ofphiiota}
\iota\left(\int_{a_{2}\Gamma_{2}^{p^{m_{2}}}}\chi_{2}(x_{2})^{i_{2}}d\mu_{\Phi^{(\boldsymbol{i})}([a_{1}]_{m_{1}};\psi,G\vert_{[M\slash N^{\prime}]})}\right)=\int_{a_{2}\Gamma_{2}^{p^{m_{2}}}}\chi_{2}(x_{2})^{i_{2}}d\mu_{\Phi_{\iota}^{(\boldsymbol{i})}([a_{1}]_{m_{1}};\psi,G\vert_{[M\slash N^{\prime}]})},
\end{equation}
where $\mu_{\Phi^{(\boldsymbol{i})}_{\iota}([a_{1}]_{m_{1}};\psi,G\vert_{[M\slash N^{\prime}]})}\in \mathrm{Hom}_{\mathcal{O}_{\K}}(C(\Gamma_{2},\mathcal{O}_{\K}),\mathcal{O}_{\K}[[q]])$ is the inverse image of the element $\Phi^{(\boldsymbol{i})}_{\iota}([a_{1}]_{m_{1}};\psi,G\vert_{[M\slash N^{\prime}]})\in \mathcal{O}_{\K}[[\Gamma_{2}]][[q]]$ by the isomorphism $\mathrm{Hom}_{\mathcal{O}_{\K}}(C(\Gamma_{2},\mathcal{O}_{\K}),\linebreak\mathcal{O}_{\K}[[q]])\stackrel{\sim}{\rightarrow}\mathcal{O}_{\K}[[\Gamma_{2}]][[q]]$ defined in \eqref{hom banach iwasawa continuous}. By \eqref{admiss pro image ofphiiota}, we have
\begin{align*}
\small 
&\iota(\theta^{(\boldsymbol{j})}_{\boldsymbol{m}}(a_{1},a_{2})) \\ 
& =\sum_{\boldsymbol{i}\in[\boldsymbol{r},\boldsymbol{j}]}\begin{pmatrix}j_{1}-r_{1}\\ i_{1}-r_{1}\end{pmatrix}\begin{pmatrix}j_{2}-r_{2}\\ i_{2}-r_{2}\end{pmatrix}(-1)^{i_{1}}\int_{a_{2}\Gamma_{2}^{p^{\nu_{2}}}}\chi_{2}(x_{2})^{i_{2}}d\mu_{\Phi_{\iota}^{(\boldsymbol{i})}([a_{1}]_{m_{1}};\psi,G\vert_{[M\slash N^{\prime}]})}\\
&(-\chi_{1}(a_{1}))^{j_{1}-i_{1}}(-\chi_{2}(a_{2}))^{j_{2}-i_{2}}\\
& =\sum_{i_{1}=r_{1}}^{j_{1}}\begin{pmatrix}j_{1}-r_{1}\\ i_{1}-r_{1}\end{pmatrix}(-\chi_{1}(a_{1}))^{j_{1}-i_{1}}(-1)^{i_{1}}\int_{a_{2}\Gamma_{2}^{p^{\nu_{2}}}}(\chi_{2}(x_{2})-\chi_{2}(a_{2}))^{j_{2}-r_{2}}\\
&\chi_{2}(x_{2})^{r_{2}}d\mu_{\Phi^{(\boldsymbol{i})}_{\iota}([a_{1}]_{m_{1}};\psi,G\vert_{[M\slash N^{\prime}]})}.
\normalsize
\end{align*}
We put $\iota(\theta^{(\boldsymbol{j})}_{\boldsymbol{m}}(a_{1},a_{2}))=\sum_{n=1}^{+\infty}a_{n}((a_{1},a_{2}),\boldsymbol{j})q^{n}$ with $a_{n}((a_{1},a_{2}),\boldsymbol{j})\in \mathcal{O}_{\K}$. By the definition of $T_{p}$, we have $T_{p}^{m_{f}(\boldsymbol{m})-m(f)+1}(\iota(\theta^{(\boldsymbol{j})}_{\boldsymbol{m}}(a_{1},a_{2})))=\sum_{n=1}^{+\infty}a_{p^{m_{f}(\boldsymbol{m})-m(f)+1}n}((a_{1},a_{2}),\boldsymbol{j})q^{n}$. For each $n\in \mathbb{Z}_{\geq 1}$, $p^{m_{f}(\boldsymbol{m})-m(f)+1}n$-th coefficinet of $\Phi^{(\boldsymbol{i})}_{\iota}([a_{1}]_{m_{1}};\psi,G\vert_{[M\slash N^{\prime}]})$ is given by 
\small 
\begin{multline}\label{pl(boldsymbolm)-1nth coefficient of phiiotaa1m1G}
\sum_{b\in \Delta\times \Gamma_{1}\slash  \Gamma_{1}^{p^{m_{1}}}}\sum_{\substack{c\in \Delta_{M}\times (\Gamma_{1}\slash  \Gamma_{1}^{p^{\max\{m_{1},m(f)-1\}}})\\ p_{m_{1}}^{(M)}(c)=b}} (\psi\xi^{-1})(c)\sum_{\substack{n_{1}+n_{2}=p^{m_{f}(\boldsymbol{m})-m(f)+1}n\\ n_{1}\equiv a_{1}b^{2}\ \mathrm{mod}\ p^{m_{1}+1}}}n_{2}^{i_{1}}\\
\sum_{\substack{t\vert n_{2}\\ t\equiv c\ \mathrm{mod}\ Mp^{\max\{m_{1}+1,m(f)\}}}}t^{k-2i_{1}-1}a_{n_{1},G\vert_{[M\slash N^{\prime}]}}\langle t\rangle^{-1}.
\end{multline}
\normalsize
For each $H=\sum_{n=0}^{+\infty}a_{n}(H)q^{n}\in \mathcal{O}_{\K}[[\Gamma_{2}]][[q]]$ with $a_{n}(H)\in \mathcal{O}_{\K}[[\Gamma_{2}]]$ and for each $\phi\in C(\Gamma_{2},\mathcal{O}_{\K})$, we have 
the following equality in $\mathcal{O}_{\K}[[q]]$: 
\begin{equation}\label{eq continuous measures formal power and each coeffici relatinon}
\int_{\Gamma_{2}}\phi(x)d\mu_{H}=\sum_{n=0}^{+\infty}\left(\int_{\Gamma_{2}}\phi(x)\mu_{a_{n}(H)}\right)q^{n} 
\end{equation}
where $\mu_{H}\in \mathrm{Hom}_{\mathcal{O}_{\K}}(C(\Gamma_{2},\mathcal{O}_{\K}),\mathcal{O}_{\K}[[q]])$ and $\mu_{a_{n}(H)}\in \mathrm{Hom}_{\mathcal{O}_{\K}}(C(\Gamma_{2},\mathcal{O}_{\K}),\mathcal{O}_{\K})$ are the inverse images of $H\in \mathcal{O}_{\K}[[\Gamma_{2}]][[q]]$ and $a_{n}(H)\in \mathcal{O}_{\K}[[q]]$ by the isomorphisms 
\begin{align}\label{eq two isomorphism hom and iwasawa in proof}
\begin{split}
& \mathrm{Hom}_{\mathcal{O}_{\K}}(C(\Gamma_{2},\mathcal{O}_{\K}),\mathcal{O}_{\K}[[q]])\stackrel{\sim}{\rightarrow}\mathcal{O}_{\K}[[\Gamma_{2}]][[q]] \\
& \mathrm{Hom}_{\mathcal{O}_{\K}}(C(\Gamma_{2},\mathcal{O}_{\K}),\mathcal{O}_{\K})\stackrel{\sim}{\rightarrow}\mathcal{O}_{\K}[[\Gamma_{2}]] 
\end{split}
\end{align}
in \eqref{hom banach iwasawa continuous} respectively. 
By applying \eqref{eq continuous measures formal power and each coeffici relatinon} to 
\begin{align*}
& H=\Phi^{(\boldsymbol{i})}_{\iota}([a_{1}]_{m_{1}};\psi,G\vert_{[M\slash N^{\prime}]}) \\ 
& \phi(x_{2})=\sum_{i_{1}=r_{1}}^{j_{1}}\begin{pmatrix}j_{1}-r_{1}\\ i_{1}-r_{1}\end{pmatrix}(-\chi_{1}(a_{1}))^{j_{1}-i_{1}}(-1)^{i_{1}}(\chi_{2}(x_{2})-\chi_{2}(a_{2}))^{j_{2}-r_{2}}\chi_{2}(x_{2})^{r_{2}}1_{a_{2}\Gamma_{2}^{p^{m_{2}}}}(x_{2}) 
\end{align*} 
 with the  characteristic function $1_{a_{2}\Gamma_{2}^{p^{m_{2}}}}(x_{2})$ on $a_{2}\Gamma_{2}^{p^{m_{2}}}$, we have
\begin{multline}\label{applyint general theta(boldsymbolj)boldsymbolm iota(a1a2)}
\iota(\theta^{(\boldsymbol{j})}_{\boldsymbol{m}}(a_{1},a_{2}))=\sum_{n=1}^{+\infty}\bigg(\sum_{i_{1}=r_{1}}^{j_{1}}\begin{pmatrix}j_{1}-r_{1}\\ i_{1}-r_{1}\end{pmatrix}(-\chi_{1}(a_{1}))^{j_{1}-i_{1}}(-1)^{i_{1}}\\
\int_{a_{2}\Gamma_{2}^{p^{\nu_{2}}}}(\chi_{2}(x_{2})-\chi_{2}(a_{2}))^{j_{2}-r_{2}}\chi_{2}(x_{2})^{r_{2}}d\mu_{n,\Phi^{(\boldsymbol{i})}_{\iota}([a_{1}]_{m_{1}};\psi,G\vert_{[M\slash N^{\prime}]})}\bigg)q^{n}
\end{multline}
where $\mu_{n,\Phi^{(\boldsymbol{i})}_{\iota}([a_{1}]_{m_{1}};\psi,G\vert_{[M\slash N^{\prime}]})}\in \mathrm{Hom}_{\mathcal{O}_{\K}}(C(\Gamma_{2},\mathcal{O}_{\K}),\mathcal{O}_{\K})$ is the inverse image of the $n$-th coefficient of $\Phi^{(\boldsymbol{i})}_{\iota}([a_{1}]_{m_{1}};\psi,G\vert_{[M\slash N^{\prime}]})$ by \eqref{eq two isomorphism hom and iwasawa in proof}.  
By \eqref{pl(boldsymbolm)-1nth coefficient of phiiotaa1m1G} and \eqref{applyint general theta(boldsymbolj)boldsymbolm iota(a1a2)}, for each $n\in \mathbb{Z}_{\geq 1}$, we have
\begin{multline}\label{admissible prodeqref apmntheta long equation}
a_{p^{m_{f}(\boldsymbol{m})-m(f)+1}n}((a_{1},a_{2}),\boldsymbol{j})=(-1)^{j_{1}-r_{1}}\sum_{b\in \Delta\times \Gamma_{1}\slash  \Gamma_{1}^{p^{m_{1}}}}\sum_{\substack{c\in \Delta_{M}\times (\Gamma_{1}\slash  \Gamma_{1}^{p^{\max\{m_{1},m(f)-1\}}})\\ p_{m_{1}}^{(M)}(c)=b}}(\psi\xi^{-1})(c)\\
\sum_{\substack{n_{1}+n_{2}=p^{m_{f}(\boldsymbol{m})-m(f)+1}n\\ n_{1}\equiv a_{1}b^{2}\ \mathrm{mod}\ p^{m_{1}+1}}}\sum_{\substack{t\vert n_{2}\\ t\equiv c\ \mathrm{mod}\ Mp^{\max\{m_{1}+1,m(f)\}}}}t^{k-1}
\int_{a_{2}\Gamma_{2}^{p^{m_{2}}}}(\chi_{2}(x_{2})-\chi_{2}(a_{2}))^{j_{2}-r_{2}}\\
\chi_{2}(x_{2})^{r_{2}}d\mu_{n_{1},t}\sum_{i_{1}=r_{1}}^{j_{1}}\begin{pmatrix}j_{1}-r_{1}\\i_{1}-r_{1}\end{pmatrix}(\frac{n_{2}}{t^{2}})^{i_{1}-r_{1}}\chi_{1}(a_{1})^{j_{1}-i_{1}}
\end{multline}
where $\mu_{n_{1},t}\in \mathrm{Hom}_{\mathcal{O}_{\K}}(C(\Gamma_{2},\mathcal{O}_{\K}),\mathcal{O}_{\K})$ is the inverse image of $a_{n_{1}}(G\vert_{[M\slash N^{\prime}]})\langle t\rangle^{-1}\in \mathcal{O}_{\K}[[\Gamma_{2}]]$ 
by \eqref{eq two isomorphism hom and iwasawa in proof}. 
By \eqref{continuous measure and spectral norm}, we have 
\begin{align}\label{ordpleft int a2gamma2 (chi2)x2geq (m2+1)(j2-r2)}
& \ord_{p}\left(\int_{a_{2}\Gamma_{2}^{p^{m_{2}}}}(\chi_{2}(x_{2})-\chi_{2}(a_{2}))^{j_{2}-r_{2}}\chi_{2}(x_{2})^{r_{2}}d\mu_{n_{1},t}\right)\\ 
& \geq 
\inf\{(\chi_{2}(x_{2})-\chi_{2}(a_{2}))^{j_{2}-r_{2}}\chi_{2}(x_{2})^{r_{2}}1_{a_{2}\Gamma_{2}^{p^{m_{2}}}}(x_{2})\}_{x_{2}\in \Gamma_{2}} \notag \\
&=(m_{2}+1)(j_{2}-r_{2}). \notag 
\end{align}
Let $b\in \Delta\times \Gamma_{1}\slash  \Gamma_{1}^{p^{m_{1}}}$, $c\in \Delta_{M}\times \Gamma_{1}\slash  \Gamma_{1}^{p^{\max\{m_{1},m(f)-1\}}}$ and $t\in \mathbb{Z}_{\geq 1}$ 
 be elements satisfying $p_{m_{1}}^{(M)}(c)=b$ and $t\equiv c\ \mathrm{mod}\ Mp^{\max\{m_{1}+1,m(f)\}}$. Since we have $p_{m_{1}}^{(M)}(c)=b$ and $t\equiv c\ \mathrm{mod}\ Mp^{\max\{m_{1}+1,m(f)\}}$, the element $b\in \Delta\times \Gamma_{1}\slash  \Gamma_{1}^{p^{m_{1}}}$ is sent to $[t]\in (\mathbb{Z}\slash p^{m_{1}+1}\mathbb{Z})^{\times}$ by the isomorphism $\Delta\times\Gamma_{1}\slash  \Gamma_{1}^{p^{m_{1}}}\simeq (\mathbb{Z}\slash p^{m_{1}+1}\mathbb{Z})^{\times}$ induced by $\chi_{1}$. That is, we have
\begin{equation}\label{eq addmisible prop tequiv chi19b0}
t\equiv \chi_{1}(b)\ \mathrm{mod}\  p^{m_{1}+1}.
\end{equation}
Let $(n_{1},n_{2})\in \mathbb{Z}_{\geq 1}^{2}$ be a pair of elements satisfying $n_{1}\equiv a_{1}b^{2}\ \mathrm{mod}\ p^{m_{1}+1}$ and $n_{1}+n_{2}\equiv 0\ \mathrm{mod}\ p^{m_{f}(\boldsymbol{m})-m(f)+1}$. Then we have 
\begin{equation}\label{eq admisibble prop n1/chi1(b)2equiv chi1(a)}
\frac{n_{1}}{\chi_{1}(b)^{2}}\equiv \chi_{1}(a)\ \mathrm{mod}\ p^{m_{1}+1}\ \mathrm{and}\ n_{2}\equiv -n_{1}\ \mathrm{mod}\ p^{m_{f}(\boldsymbol{m})-m(f)+1}.
\end{equation}
Assume that $t\vert n_{2}$. By combining \eqref{eq addmisible prop tequiv chi19b0} and \eqref{eq admisibble prop n1/chi1(b)2equiv chi1(a)}, we have $\frac{n_{2}}{t^{2}}\equiv \frac{n_{2}}{\chi_{1}(b)^{2}}\equiv \frac{-n_{1}}{\chi_{1}(b)^{2}}\equiv -\chi_{1}(a_{1})\ \mathrm{mod}\ p^{\min\{m_{f}(\boldsymbol{m})-m(f),m_{1}\}+1}$, 
 which implies that 
\begin{multline}\label{eq sum i1=r1 modulo p(m1+1)(j1-r1)}
\sum_{i_{1}=r_{1}}^{j_{1}}
\begin{pmatrix}j_{1}-r_{1}\\i_{1}-r_{1}\end{pmatrix}\left(\frac{n_{2}}{t^{2}}\right)^{i_{1}-r_{1}}\chi_{1}(a_{1})^{j_{1}-i_{1}}=
\left(\frac{n_{2}}{t^{2}}+\chi_{1}(a_{1})\right)^{j_{1}-r_{1}}\\
\equiv 0\ \mathrm{mod}\ p^{(j_{1}-r_{1})(\min\{m_{f}(\boldsymbol{m})-m(f),m_{1}\}+1)}.
\end{multline}
 By \eqref{admissible prodeqref apmntheta long equation}, \eqref{ordpleft int a2gamma2 (chi2)x2geq (m2+1)(j2-r2)} and \eqref{eq sum i1=r1 modulo p(m1+1)(j1-r1)}, we
have 
$$\ord_{p}(a_{p^{m_{f}(\boldsymbol{m})-m(f)+1}n}((a_{1},a_{2}),\boldsymbol{j}))\geq (j_{1}-r_{1})(\min\{m_{f}(\boldsymbol{m})-m(f),m_{1}\}+1)+(j_{2}-r_{2})(m_{2}+1)
$$ 
for each $n\in \mathbb{Z}_{\geq 1}$. Thus, we see that 
\begin{multline}\label{admissible prod viotam(f) geq j2-r2 m2}
v_{\iota_{m(f)}}(T_{p}^{m_{f}(\boldsymbol{m})-m(f)+1}(\iota(\theta^{(\boldsymbol{j})}_{\boldsymbol{m}}(a_{1},a_{2}))) =\inf_{n\in \mathbb{Z}_{\geq 1}}\{\ord_{p}(a_{p^{m_{f}(\boldsymbol{m})-m(f)+1}n}((a_{1},a_{2}),\boldsymbol{j}))\}\\ 
 \geq (j_{1}-r_{1})(\min\{m_{f}(\boldsymbol{m})-m(f),m_{1}\}+1)+(j_{2}-r_{2})(m_{2}+1)
\end{multline}
for every $\boldsymbol{j}\in [\boldsymbol{r},\boldsymbol{s}]$ and for every $(a_{1},a_{2})\in R_{1,m_{1}}\times R_{2,m_{2}}$ with $\boldsymbol{m}\in \mathbb{Z}_{\geq 0}^{2}$.
By \eqref{p-adic admissible prop in T iota} and \eqref{admissible prod viotam(f) geq j2-r2 m2}, we have
\small 
\begin{align}\label{admissible prod ordp geq geq geq v-(2-cf)aloha-(s1-r1)}
\begin{split}
& \ord_{p}\left((l_{f,M}^{(m_{f}(\boldsymbol{m}))}T_{p}(\theta^{(\boldsymbol{j})}_{\boldsymbol{m}}(a_{1},a_{2}))\right) \\
& \geq -(m_{f}(\boldsymbol{m})-m(f))\alpha+v_{\mathfrak{L}}(l_{f,M}^{(m(f))})+v_{\mathfrak{L}}(\iota_{m(f)}^{-1})+(j_{1}-r_{1})(\min\{m_{f}(\boldsymbol{m})-m(f),m_{1}\}+1)\\
&+(j_{2}-r_{2})(m_{2}+1) 
\\
& \geq -(2m_{1}+m_{2}+1)\alpha+v_{\mathfrak{L}}(l_{f,M}^{(m(f))})+v_{\mathfrak{L}}(\iota_{m(f)}^{-1})+(j_{1}-r_{1})(m_{1}+1-m(f))+(j_{2}-r_{2})(m_{2}+1)\\
&\geq-\langle \boldsymbol{m},\boldsymbol{h}-(\boldsymbol{j}-\boldsymbol{r})\rangle_{2}+v_{\mathfrak{L}}(l_{f,M}^{(m(f))})+v_{\mathfrak{L}}(\iota_{m(f)}^{-1})-\alpha-(s_{1}-r_{1})(m(f)-1)
\end{split}
\end{align}
\normalsize
for every $\boldsymbol{j}\in [\boldsymbol{r},\boldsymbol{s}]$ and for every $(a_{1},a_{2})\in R_{1,m_{1}}\times R_{2,m_{2}}$ with $\boldsymbol{m}\in \mathbb{Z}_{\geq 0}^{2}$. Let $n^{[\boldsymbol{r},\boldsymbol{s}]}(f)$ be a non-negative integer satisfying the following condition:
\begin{equation}\label{admissible prod prisice n 1}
v_{\mathfrak{L}}(l_{f,M}^{(m(f))})+v_{\mathfrak{L}}(\iota_{m(f)}^{-1})-\alpha-(s_{1}-r_{1})(m(f)-1)\geq -n^{[\boldsymbol{r},\boldsymbol{s}]}(f).
\end{equation}
Then, by \eqref{admissible prod ordp geq geq geq v-(2-cf)aloha-(s1-r1)}, we have 
\begin{equation}\label{admissible prop lfmcf(bldsybmol)jgeq -ncomp}
\ord_{p}\left((l_{f,M}^{(m_{f}(\boldsymbol{m}))}T_{p}(\theta^{(\boldsymbol{j})}_{\boldsymbol{m}}(a_{1},a_{2}))\right)+\langle \boldsymbol{m},\boldsymbol{h}-(\boldsymbol{j}-\boldsymbol{r})\rangle_{2}\geq -n^{[\boldsymbol{r},\boldsymbol{s}]}(f)
\end{equation}
 for every $\boldsymbol{j}\in [\boldsymbol{r},\boldsymbol{s}]$ and for every $(a_{1},a_{2})\in R_{1,m_{1}}\times R_{2,m_{2}}$ with $\boldsymbol{m}\in \mathbb{Z}_{\geq 0}^{2}$. 
Thus, by \eqref{admissible prop display prodt=12j1 sum R1m1 R2m2} and \eqref{admissible prop lfmcf(bldsybmol)jgeq -ncomp}, we see that
\begin{align*}
&p^{\langle \boldsymbol{m},\boldsymbol{h}-(\boldsymbol{j}-\boldsymbol{r})\rangle_{2}}\displaystyle{\sum_{\boldsymbol{i}\in [\boldsymbol{r},\boldsymbol{j}]}}\left(\prod_{t=1}^{2}\begin{pmatrix}j_{t}-r_{t}\\i_{t}-r_{t}\end{pmatrix}\right)(-1)^{\sum_{t=1}^{2}(j_{t}-i_{t})}\tilde{s}_{\boldsymbol{m}}^{[\boldsymbol{i}]}\\
&=\sum_{(a_{1},a_{2})\in R_{1,m_{1}}\times R_{2,m_{2}}}\chi_{1}(a_{1})^{-j_{1}}\chi_{2}(a_{2})^{-j_{2}}p^{\langle \boldsymbol{m},\boldsymbol{h}-(\boldsymbol{j}-\boldsymbol{r})\rangle_{2}}l_{f,M}^{(m_{f}(\boldsymbol{m}))}\left(T_{p}\theta^{(\boldsymbol{j})}_{\boldsymbol{m}}(a_{1},a_{2})\right)[a_{1},a_{2}]
\end{align*}
is in $\mathcal{O}_{\K}[[(\Delta\times \Gamma_{1})\times\Gamma_{2}]]\otimes_{\mathcal{O}_{\K}}p^{-n^{[\boldsymbol{r},\boldsymbol{s}]}(f)}\mathcal{O}_{\K}$ for every $\boldsymbol{j}\in [\boldsymbol{r},\boldsymbol{s}]$ and for every $\boldsymbol{m}\in \mathbb{Z}_{\geq 0}^{2}$. This completes the proof of the proposition.
\end{proof}

\subsubsection*{{\bf{Definition of the two-variable admissible distribution}}}
Let $f\in S_{k}(Np^{m(f)},\psi;\K)$ be a normalized cuspidal Hecke eigenform which is new away from $p$ with $m(f)\in \mathbb{Z}_{\geq 1}$ and $G\in S(N^{\prime}p,\xi;\mathcal{O}_{\K}[[\Gamma_{2}]])$. We assume that $m(f)$ is the smallest positive integer $m$ such that $f\in S_{k}(Np^{m},\psi;\K)$. Put $\boldsymbol{h}=(2\alpha,\alpha)$ with $\alpha=\ord_{p}(a_{p}(f))$. Let $M$ be the least common multiple of $N$ and $N^{\prime}$. We assume the following conditions:
\begin{enumerate}
\item We have $k>\lfloor 2\alpha\rfloor+\lfloor \alpha\rfloor+2$.
\item All $M$-th roots of unity and Fourier coefficients of $f^{0}$ are contained in $\K$, where $f^{0}$ is the primitive form associated with $f$.
\end{enumerate}
Let $\boldsymbol{d}=(0,2)$, $\boldsymbol{e}=(k-3, k-1)$. Let $\boldsymbol{r},\boldsymbol{s}\in \mathbb{Z}^{2}$ be elements such that $\boldsymbol{s}\geq \boldsymbol{r}$, $[\boldsymbol{r},\boldsymbol{s}]\subset [\boldsymbol{d},\boldsymbol{e}]$ and $s_{1}+s_{2}<k$. Let $s_{\boldsymbol{m}}^{[\boldsymbol{i}]}$ be the element defined in \eqref{definition of qm1m2i1i1} for each $\boldsymbol{m}\in \mathbb{Z}_{\geq 0}^{2}$. By Lemma \ref{multivariable litfing prop deformation} and Proposition \ref{admissible condition of p-adic l}, there exists a unique element
\begin{equation}
s_{\boldsymbol{m}}^{[\boldsymbol{r},\boldsymbol{s}]}\in \frac{\mathcal{O}_{\K}[[(\Delta\times \Gamma_{1})\times\Gamma_{2}]]}{(\Omega_{\boldsymbol{m}}^{[\boldsymbol{r},\boldsymbol{s}]})\mathcal{O}_{\K}[[(\Delta\times \Gamma_{1})\times\Gamma_{2}]]}\otimes_{\mathcal{O}_{\K}}\K
\end{equation}
for each $\boldsymbol{m}\in \mathbb{Z}_{\geq 0}^{2}$ such that the image of $s_{\boldsymbol{m}}^{[\boldsymbol{r},\boldsymbol{s}]}$ by the projection $\frac{\mathcal{O}_{\K}[[(\Delta\times \Gamma_{1})\times\Gamma_{2}]]}{(\Omega_{\boldsymbol{m}}^{[\boldsymbol{r},\boldsymbol{s}]})\mathcal{O}_{\K}[[(\Delta\times \Gamma_{1})\times\Gamma_{2}]]}\otimes_{\mathcal{O}_{\K}}\K\rightarrow \frac{\mathcal{O}_{\K}[[(\Delta\times \Gamma_{1})\times\Gamma_{2}]]}{(\Omega_{\boldsymbol{m}}^{[\boldsymbol{i}]})\mathcal{O}_{\K}[[(\Delta\times \Gamma_{1})\times\Gamma_{2}]]}\otimes_{\mathcal{O}_{\K}}\K$ is equal to $s_{\boldsymbol{m}}^{[\boldsymbol{i}]}$ for every $\boldsymbol{i}\in [\boldsymbol{r},\boldsymbol{s}]$ and we have \linebreak$(p^{\langle \boldsymbol{h},\boldsymbol{m}\rangle_{2}}s_{\boldsymbol{m}}^{[\boldsymbol{r},\boldsymbol{s}]})_{\boldsymbol{m}\in \mathbb{Z}_{\geq 0}}\in \left(\prod_{\boldsymbol{m}\in \mathbb{Z}_{\geq 0}^{2}}\frac{\mathcal{O}_{\K}[[(\Delta\times \Gamma_{1})\times\Gamma_{2}]]}{(\Omega_{\boldsymbol{m}}^{[\boldsymbol{r},\boldsymbol{s}]})\mathcal{O}_{\K}[[(\Delta\times \Gamma_{1})\times\Gamma_{2}]]}\right)\otimes_{\mathcal{O}_{\K}}\K$. By Proposition \ref{distribution of mu[boldsymbolr,boldsymbols]}, we see that $(s_{\boldsymbol{m}}^{[\boldsymbol{r},\boldsymbol{s}]})_{\boldsymbol{m}\in \mathbb{Z}_{\geq 0}^{2}}\in \varprojlim_{\boldsymbol{m}\in \mathbb{Z}_{\geq 0}}\left(\frac{\mathcal{O}_{\K}[[(\Delta\times \Gamma_{1})\times\Gamma_{2}]]}{(\Omega_{\boldsymbol{m}}^{[\boldsymbol{r},\boldsymbol{s}]})\mathcal{O}_{\K}[[(\Delta\times \Gamma_{1})\times\Gamma_{2}]]}\otimes_{\mathcal{O}_{\K}}\K\right)$. Then, we have
\begin{equation}\label{definition of q[boldsymbolr,boldsymbols]}
s^{[\boldsymbol{r},\boldsymbol{s}]}=(s_{\boldsymbol{m}}^{[\boldsymbol{r},\boldsymbol{s}]})_{\boldsymbol{m}\in \mathbb{Z}_{\geq 0}^{2}}\in I_{\boldsymbol{h}}^{[\boldsymbol{r},\boldsymbol{s}]}\otimes_{\mathcal{O}_{\K}[[\Gamma_{1}\times \Gamma_{2}]]}\mathcal{O}_{\K}[[(\Delta\times \Gamma_{1})\times \Gamma_{2}]].
\end{equation}
Let $I_{\boldsymbol{h}}^{[\boldsymbol{d},\boldsymbol{e}]}\otimes_{\mathcal{O}_{\K}[[\Gamma_{1}\times \Gamma_{2}]]}\mathcal{O}_{\K}[[(\Delta\times \Gamma_{1})\times \Gamma_{2}]]\rightarrow I_{\boldsymbol{h}}^{[\boldsymbol{d},\boldsymbol{e}_{\alpha}]}\otimes_{\mathcal{O}_{\K}[[\Gamma_{1}\times \Gamma_{2}]]}\mathcal{O}_{\K}[[(\Delta\times \Gamma_{1})\times \Gamma_{2}]]$ be the natural projection, where $\boldsymbol{e}_{\alpha}=(\lfloor2\alpha\rfloor,\lfloor\alpha\rfloor+2)$. As mentioned in \eqref{projection I is isom if e-dgeq h}, the above projection is an isomorphism. Then, we can define the inverse image
\begin{equation}\label{definition of inverse q[d,ealpha]}
s_{(f,G)}\in I_{\boldsymbol{h}}^{[\boldsymbol{d},\boldsymbol{e}]}\otimes_{\mathcal{O}_{\K}[[\Gamma_{1}\times \Gamma_{2}]]}\mathcal{O}_{\K}[[(\Delta\times \Gamma_{1})\times \Gamma_{2}]]
\end{equation}
of $s^{[\boldsymbol{d},\boldsymbol{e}_{\alpha}]}\in I_{\boldsymbol{h}}^{[\boldsymbol{d},\boldsymbol{e}_{\alpha}]}\otimes_{\mathcal{O}_{\K}[[\Gamma_{1}\times \Gamma_{2}]]}\mathcal{O}_{\K}[[(\Delta\times \Gamma_{1})\times \Gamma_{2}]]$ by the projection. 
\subsubsection*{{\bf{Verification of the interpolation formula of}\ $s_{(f,G)}$}}
For each $\kappa\in\mathfrak{X}_{\mathcal{O}_{\K}[[(\Delta\times \Gamma_{1})\times \Gamma_{2}]]}$, let $\phi_{\kappa,1}: (\Delta\times \Gamma_{1})\rightarrow \overline{\K}^{\times}$ and $\phi_{\kappa,2}: \Gamma_{2}\rightarrow \overline{\K}^{\times}$ be the finite characters which satisfy
\begin{equation}\label{definition phiQ,ZLtimes}
\kappa\vert_{(\Delta\times \Gamma_{1})\times \Gamma_{2}}((x_{1},x_{2}))=\phi_{\kappa,1}(x_{1})\chi_{1}(x_{1})^{w_{\kappa,1}}\phi_{\kappa,2}(x_{2})\chi_{2}(x_{2})^{w_{\kappa,2}}
\end{equation}
 for each $(x_{1},x_{2})\in (\Delta\times \Gamma_{1})\times \Gamma_{2}$. Here, $\boldsymbol{w}_{\kappa}=(w_{\kappa,1},w_{\kappa,2})\in \mathbb{Z}^{2}$ is the weight of $\kappa$. For each $\kappa\in\mathfrak{X}_{\mathcal{O}_{\K}[[(\Delta\times \Gamma_{1})\times \Gamma_{2}]]}$, we denote by $m_{\kappa,i}$ the smallest integer $m$ such that $\phi_{\kappa,i}$ factors through $\Gamma_{i}\slash \Gamma_{i}^{p^{m}}$ with $i=1,2$ and put
 \begin{equation}\label{twovariable ZLGamma2 conductor}
 \boldsymbol{m}_{\kappa}=(m_{\kappa,1},m_{\kappa,2}).
 \end{equation}
 Let $\tau_{L}$ be the matrix defined in \eqref{definition of tauM} for each $L\in \mathbb{Z}_{\geq 1}$ and $f^{\rho}$ the cusp form defined in \eqref{definition of frho}. 
\begin{lem}\label{proposition of interpolation of L(f,G)}
Let $N$ and $N^{\prime}$ be positive integers which are prime to $p$.
Let $f\in S_{k}(Np^{m(f)},\linebreak\psi;\K)$ be a normalized cuspidal Hecke eigenform which is new away from $p$ with $m(f)\in \mathbb{Z}_{\geq 1}$. 
Assume that $m(f)$ is the smallest positive integer $m$ such that $f\in S_{k}(Np^{m},\psi)$. 
Let $G\in S(N^{\prime}p,\xi;\mathcal{O}_{\K}[[\Gamma_{2}]])$. Put $\boldsymbol{h}=(2\alpha,\alpha)$ with $\alpha=\ord_{p}(a_{p}(f))$. Let $M$ be the least common multiple of $N$ and $N^{\prime}$. We assume the following conditions:
\begin{enumerate}
\item We have $k>\lfloor 2\alpha\rfloor+\lfloor \alpha\rfloor+2$.
\item All $M$-th roots of unity and Fourier coefficients of $f^{0}$ are contained in $\K$, where $f^{0}$ is the primitive form associated with $f$.
\end{enumerate}
Then the element $s_{(f,G)}=(s_{(f,G),\boldsymbol{m}})_{\boldsymbol{m}\in \mathbb{Z}_{\geq 0}^{2}}\in I_{\boldsymbol{h}}^{[\boldsymbol{d},\boldsymbol{e}]}\otimes_{\mathcal{O}_{\K}[[\Gamma_{1}\times \Gamma_{2}]]}\mathcal{O}_{\K}[[(\Delta\times \Gamma_{1})\times \Gamma_{2}]]$ defined in \eqref{definition of inverse q[d,ealpha]} with $s_{(f,G),\boldsymbol{m}}\in\frac{ \mathcal{O}_{\K}[[(\Delta\times \Gamma_{1})\times \Gamma_{2}]]}{(\Omega_{\boldsymbol{m}}^{[\boldsymbol{d},\boldsymbol{e}]})\mathcal{O}_{\K}[[(\Delta\times \Gamma_{1})\times \Gamma_{2}]]}\otimes_{\mathcal{O}_{\K}}\K$ satisfies the following interpolation property for every $\kappa\in \mathfrak{X}_{\mathcal{O}_{\K}[[(\Delta\times \Gamma_{1})\times \Gamma_{2}]]}^{[\boldsymbol{d},\boldsymbol{e}]}$ satisfying $w_{\kappa,1}+w_{\kappa,2}<k$: 
\begin{equation}\label{lemma interpolation L(f,G) eq}
\kappa(\tilde{s}_{(f,G),\boldsymbol{m}_{\kappa}})=(-1)^{w_{\kappa,1}}\phi_{\kappa,1}(M\slash N^{\prime})l_{f,M}\circ T_{p}\left((\kappa\vert_{\mathcal{O}_{\K}[[\Gamma_{2}]]}(G)\otimes\phi_{\kappa,1})\vert_{[M\slash N^{\prime}]}H_{\kappa}\right)
\end{equation}
where $\tilde{s}_{(f,G),\boldsymbol{m}_{\kappa}}\in \mathcal{O}_{\K}[[(\Delta\times \Gamma_{1})\times\Gamma_{2}]]\otimes_{\mathcal{O}_{\K}}\K$ is a lift of the element 
$s_{(f,G),\boldsymbol{m}_{\kappa}}$, 
the map $l_{f,M}: \cup_{n=m(f)}^{+\infty}N^{\leq \lfloor \frac{k-1}{2}\rfloor,\mathrm{cusp}}_{k}(Mp^{n},\psi;\K)\rightarrow \K$ is the $\K$-linear map defined in \eqref{classical lf map}, $T_{p}$ is the $p$-th Hecke operator, $\vert_{[M\slash N^{\prime}]}$ is the operator defined in \eqref{power series and qn mapst qNn}, $\kappa\vert_{\mathcal{O}_{\K}[[\Gamma_{2}]]}(G\vert_{[M\slash N^{\prime}]})\otimes\phi_{\kappa,1}$ is the twist of $\kappa\vert_{\mathcal{O}_{\K}[[\Gamma_{2}]]}(G\vert_{[M\slash N^{\prime}]})$ and 
\begin{equation}\label{lemma interpolation L(f,G) eisenstein}
H_{\kappa}=\begin{cases}\delta^{(w_{\kappa,1})}_{t_{\kappa}^{(1)}}\left(F_{t_{\kappa}^{(1)}}(\boldsymbol{1},\psi_{\kappa})\right)\ &\mathrm{if}\ 0\leq w_{\kappa,1}<\frac{1}{2}(k-w_{\kappa,2}),\\
\delta^{(k-w_{\kappa,1}-w_{\kappa,2}-1)}_{t_{\kappa}^{(2)}}\left(F_{t_{\kappa}^{(2)}}(\psi_{\kappa},\boldsymbol{1})\right)\ &\mathrm{if}\ \frac{1}{2}(k-w_{\kappa,2})\leq w_{\kappa,1}<k-w_{\kappa,2},
\end{cases}
\end{equation}
with $t_{\kappa}^{(1)}=k-2w_{\kappa,1}-w_{\kappa,2}$, $t_{\kappa}^{(2)}=w_{\kappa,2}-k+2w_{\kappa,1}+2$ and $\psi_{\kappa}=\psi\xi^{-1}\phi_{\kappa,1}^{-2}\omega^{w_{\kappa,2}}\phi_{\kappa,2}^{-1}$. Here $\boldsymbol{1}$ is the trivial character modulo $1$, $F_{t_{\kappa}^{(1)}}(\boldsymbol{1},\psi_{\kappa})$ and $F_{t_{\kappa}^{(2)}}(\psi_{\kappa},\boldsymbol{1})$ are the $q$-expansions of the Eisenstein series defined in \eqref{definition of eisenstein seriesnonap} and \eqref{another Eisenstein seriesnonap}, $\phi_{\kappa,1}$ and $\phi_{\kappa,2}$ are finite characters defined in \eqref{definition phiQ,ZLtimes} and $\delta_{m}^{(r)}$ is the differential operator defined in \eqref{shimura operator} with $m\in \mathbb{Z}$ and $r\in \mathbb{Z}_{\geq 0}$.
\end{lem}
\begin{proof}
Let $\boldsymbol{d}=(0,2)$, $\boldsymbol{e}=(k-3, k-1)$ and $\boldsymbol{e}_{\alpha}=(\lfloor2\alpha\rfloor,\lfloor\alpha\rfloor+2)$. The weights of the arithmetic specializations $\kappa$ in the range of interpolation is given as follows: 
$$
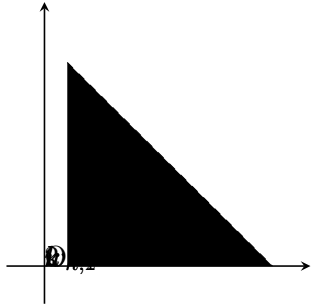
\begin{tikzpicture}
 \draw[->,>=stealth,semithick] (-0.5,0)--(3.5,0)node[above]{$w_{\kappa,2}$}; 
 \draw[->,>=stealth,semithick] (0,-0.5)--(0,3.5)node[right]{$w_{\kappa,1}$}; 
 \draw (0,0)node[above  left]{O}; 
 \draw (0.3,0)node[below]{2};
 \draw (3,0)node[below]{$k$}; 
 \draw(0,2.7)node[left]{$k-2$};
  \fill[lightgray] (0.3,0)--(0.3,2.7)--(3,0);
 \draw[dashed,domain=0.3:3] plot(\x,3-\x);
   \draw (1.65,-0.5)node[below]{};
\end{tikzpicture} . 
$$
The range of the interpolation is triangular, but our theory covers only the rectangular region.  
So we will cover the rectangular region $[\boldsymbol{d},\boldsymbol{e}_{\alpha}]$ which is contained in the above triangular region in Step 1 below. 
In Step 2, we will extend this rectangular region to the vertical direction to cover the upper subtriangle which was not coverd in Step 1. 
In Step 3, we will extend the region which was covered in Step 1 and Step 2 to the horizontal direction to cover the right subtriangle which was not covered in Step 1 and Step2. 
\\ 
$\mathbf{Step\ 1}$. Let $\boldsymbol{r},\boldsymbol{s}\in \mathbb{Z}^{2}$ be elements satisfying $\boldsymbol{s}\geq \boldsymbol{r}$, $[\boldsymbol{r},\boldsymbol{s}]\subset [\boldsymbol{d},\boldsymbol{e}]$ and $s_{1}+s_{2}<k$. Let $s^{[\boldsymbol{r},\boldsymbol{s}]}=(s_{\boldsymbol{m}}^{[\boldsymbol{r},\boldsymbol{s}]})_{\boldsymbol{m}\in \mathbb{Z}_{\geq 0}^{2}}\in I_{\boldsymbol{h}}^{[\boldsymbol{r},\boldsymbol{s}]}\otimes_{\mathcal{O}_{\K}[[\Gamma_{1}\times \Gamma_{2}]]}\mathcal{O}_{\K}[[(\Delta\times \Gamma_{1})\times \Gamma_{2}]]$ be the element defined in \eqref{definition of q[boldsymbolr,boldsymbols]}. We will prove that, for each $\kappa\in \mathfrak{X}_{\mathcal{O}_{\K}[[(\Delta\times \Gamma_{1})\times \Gamma_{2}]]}^{[\boldsymbol{r},\boldsymbol{s}]}$, we have
\begin{equation}\label{lemma interpolation q[r,s]mQ eq}
\kappa(\tilde{s}^{[\boldsymbol{r},\boldsymbol{s}]}_{\boldsymbol{m}_{\kappa}})=(-1)^{w_{\kappa,1}}\phi_{\kappa,1}(M\slash N^{\prime})l_{f,M}\circ T_{p}\left((\kappa\vert_{\mathcal{O}_{\K}[[\Gamma_{2}]]}(G)\otimes\phi_{\kappa,1})\vert_{[M\slash N^{\prime}]}H_{\kappa}\right)
\end{equation}
where $\tilde{s}^{[\boldsymbol{r},\boldsymbol{s}]}_{\boldsymbol{m}_{\kappa}}\in \mathcal{O}_{\K}[[(\Delta\times \Gamma_{1})\times \Gamma_{2}]]\otimes_{\mathcal{O}_{\K}}\K$ is a lift of $s^{[\boldsymbol{r},\boldsymbol{s}]}_{\boldsymbol{m}_{\kappa}}$ and $\boldsymbol{m}_{\kappa}$ is the pair of non-negative integers defined in \eqref{twovariable ZLGamma2 conductor}. Let $\kappa\in \mathfrak{X}_{\mathcal{O}_{\K}[[(\Delta\times \Gamma_{1})\times \Gamma_{2}]]}^{[\boldsymbol{r},\boldsymbol{s}]}$ and $s^{[\boldsymbol{w}_{\kappa}]}_{\boldsymbol{m}_{\kappa}}\in \frac{\mathcal{O}_{\K}[[(\Delta\times \Gamma_{1})\times \Gamma_{2}]]}{(\Omega_{\boldsymbol{m}_{\kappa}}^{[\boldsymbol{w}_{\kappa},\boldsymbol{w}_{\kappa}]})\mathcal{O}_{\K}[[(\Delta\times \Gamma_{1})\times \Gamma_{2}]]}\otimes_{\mathcal{O}_{\K}}\K$ be the element defined in \eqref{definition of qm1m2i1i1}. 
By the definition of $s^{[\boldsymbol{r},\boldsymbol{s}]}$, we have $\kappa(\tilde{s}^{[\boldsymbol{r},\boldsymbol{s}]}_{\boldsymbol{m}_{\kappa}})=\kappa(\tilde{s}^{[\boldsymbol{w}_{\kappa}]}_{\boldsymbol{m}_{\kappa}})$. By \eqref{definition of qm1m2i1i1}, we see that
\begin{align*}
\kappa(\tilde{s}^{[\boldsymbol{r},\boldsymbol{s}]}_{\boldsymbol{m}_{\kappa}})=\kappa(\tilde{s}^{[\boldsymbol{w}_{\kappa}]}_{\boldsymbol{m}_{\kappa}})=\sum_{(a_{1},a_{2})\in (\Delta\times \Gamma_{1}\slash  \Gamma_{1}^{p^{m_{1}}})\times \Gamma_{2}\slash \Gamma_{2}^{p^{m_{2}}}}s^{[\boldsymbol{w}_{\kappa}]}(a_{1},a_{2})\phi_{\kappa,1}(a_{2})\phi_{\kappa,2}(a_{2})
\end{align*}
where $s^{[\boldsymbol{w}_{\kappa}]}(a_{1},a_{2})\in \K$ is the element defined in \eqref{definition of qboldysmboli a12}. By Proposition \ref{proposition for interpolation of mu[r,s]} and \eqref{definition of qboldysmboli a12}, we have
\begin{align*}
\kappa(\tilde{s}^{[\boldsymbol{r},\boldsymbol{s}]}_{\boldsymbol{m}_{\kappa}})&=\sum_{(a_{1},a_{2})\in (\Delta\times \Gamma_{1}\slash  \Gamma_{1}^{p^{m_{1}}})\times \Gamma_{2}\slash \Gamma_{2}^{p^{m_{2}}}}s^{[\boldsymbol{w}_{\kappa}]}(a_{1},a_{2})\phi_{\kappa,1}(a_{2})\phi_{\kappa,2}(a_{2})\\
&=l_{f,M}\left(\sum_{(a_{1},a_{2})\in (\Delta\times \Gamma_{1}\slash  \Gamma_{1}^{p^{m_{1}}})\times \Gamma_{2}\slash \Gamma_{2}^{p^{m_{2}}}}\phi_{1}(a_{1})\phi_{2}(a_{2})\phi^{[\boldsymbol{w}_{\kappa}]}((a_{1},a_{2});\psi,G\vert_{[M\slash N^{\prime}]})\right)\\
&=(-1)^{w_{\kappa,1}}l_{f,M}\circ T_{p}\left(\kappa\vert_{\mathcal{O}_{\K}[[\Gamma_{2}]]}(G\vert_{[M\slash N^{\prime}]})\otimes\phi_{\kappa,1}H_{\kappa}\right)\\
&=(-1)^{w_{\kappa,1}}\phi_{\kappa,1}(M\slash N^{\prime})l_{f,M}\circ T_{p}\left((\kappa\vert_{\mathcal{O}_{\K}[[\Gamma_{2}]]}(G)\otimes\phi_{\kappa,1})\vert_{[M\slash N^{\prime}]}H_{\kappa}\right).
\end{align*}

Therefore, we have \eqref{lemma interpolation q[r,s]mQ eq}. By the definition of $s_{(f,G)}$, we see that $\kappa(\tilde{s}_{(f,G),\boldsymbol{m}_{\kappa}})=\kappa(\tilde{s}^{[\boldsymbol{d},\boldsymbol{e}_{\alpha}]}_{\boldsymbol{m}_{\kappa}})$ for every $\kappa\in \mathfrak{X}^{[\boldsymbol{d},\boldsymbol{e}_{\alpha}]}_{_{\mathcal{O}_{\K}[[(\Delta\times \Gamma_{1})\times \Gamma_{2}]]}}$. Then, by \eqref{lemma interpolation q[r,s]mQ eq}, we have
$$\kappa(\tilde{s}_{(f,G),\boldsymbol{m}_{\kappa}})=(-1)^{w_{\kappa,1}}\phi_{\kappa,1}(M\slash N^{\prime})l_{f,M}\circ T_{p}\left((\kappa\vert_{\mathcal{O}_{\K}[[\Gamma_{2}]]}(G)\otimes\phi_{\kappa,1})\vert_{[M\slash N^{\prime}]}H_{\kappa}\right).$$
 for every $\kappa\in \mathfrak{X}^{[\boldsymbol{d},\boldsymbol{e}_{\alpha}]}_{_{\mathcal{O}_{\K}[[(\Delta\times \Gamma_{1})\times \Gamma_{2}]]}}$.

$\mathbf{Step\ 2}$. We will prove that $\kappa(\tilde{s}_{(f,G),\boldsymbol{m}_{\kappa}})$ is equal to 
the right-hand side of \eqref{lemma interpolation L(f,G) eq} for each $\kappa\in \mathfrak{X}_{\mathcal{O}_{\K}[[(\Delta\times \Gamma_{1})\times \Gamma_{2}]]}^{[\boldsymbol{d},\boldsymbol{e}]}$ such that $w_{\kappa,2}\in [2,\lfloor \alpha\rfloor+2]$ and $w_{\kappa,1}+w_{\kappa,2}<k$.
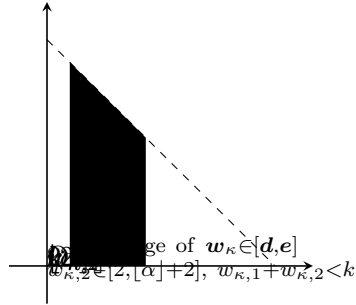
$$
\begin{tikzpicture}
 \draw[->,>=stealth,semithick] (-0.5,0)--(3.5,0)node[above]{$w_{\kappa,2}$}; 
 \draw[->,>=stealth,semithick] (0,-0.5)--(0,3.5)node[right]{$w_{\kappa,1}$}; 
 \draw (0,0)node[above  left]{O}; 
 \draw(1.3,0)node[below]{$\lfloor\alpha\rfloor+2$};
 \draw (0.3,0)node[below]{2};
 \draw(0,1)node[left]{$\lfloor 2\alpha\rfloor$};
 \draw (3,0)node[below]{$k$}; 
 \draw(0,3)node[left]{$k$};
  \fill[lightgray] (0.3,0)--(0.3,2.7)--(1.3,1.7)--(1.3,0);
 \draw[dashed,domain=0:3] plot(\x,3-\x);
  \draw (1.65,-0.5)node[below]{$\substack{\mathrm{Range\ of}\ \boldsymbol{w}_{\kappa}\in [\boldsymbol{d},\boldsymbol{e}]\\ w_{\kappa,2}\in [2,\lfloor \alpha\rfloor+2],\ w_{\kappa,1}+w_{\kappa,2}<k}$};
\end{tikzpicture}
$$
Let $\kappa\in \mathfrak{X}_{\mathcal{O}_{\K}[[(\Delta\times \Gamma_{1})\times \Gamma_{2}]]}^{[\boldsymbol{d},\boldsymbol{e}]}$ such that $w_{\kappa,2}\in [2,\lfloor \alpha\rfloor+2]$ and $w_{\kappa,1}+w_{\kappa,2}<k$. We define a continuous $\mathcal{O}_{\K}[[\Delta\times \Gamma_{1}]]$-module homomorphism
\begin{equation}\label{definition of r(wkappa,2,phikappa2)in lemma}
r_{(w_{\kappa,2},\phi_{\kappa,2})}: \mathcal{O}_{\K}[[(\Delta\times \Gamma_{1})\times \Gamma_{2}]]\rightarrow \mathcal{O}_{\K(\phi_{\kappa,2})}[[(\Delta\times \Gamma_{1})]]
\end{equation}
to be $r_{(w_{\kappa,2},\phi_{\kappa,2})}\vert_{(\Delta\times \Gamma_{1})\times\Gamma_{2}}((a_{1},a_{2}))=\kappa\vert_{\mathcal{O}_{\K}[[\Gamma_{2}]]}(a_{2})[a_{1}]$ for each $(a_{1},a_{2})\in (\Delta\times \Gamma_{1})\times \Gamma_{2}$, where $[a_{1}]\in\mathcal{O}_{\K(\phi_{\kappa,2})}[[(\Delta\times \Gamma_{1})]]$ is the class of $a_{1}\in (\Delta\times \Gamma_{1})$. We remark that we have 
\begin{equation}\label{proposition of interpolation of L(f,G) interpolation form rkappa}
\kappa^{\prime}\vert_{\mathcal{O}_{\K}[[\Delta\times \Gamma_{1}]]}(r_{(w_{\kappa,2},\phi_{\kappa,2})}(s))=\kappa^{\prime}(s)
\end{equation}
for every $s\in \mathcal{O}_{\K}[[(\Delta\times \Gamma_{1})\times \Gamma_{2}]]$ and for every $\kappa^{\prime}\in \mathfrak{X}_{\mathcal{O}_{\K}[[(\Delta\times \Gamma_{1})\times \Gamma_{2}]]}$ 
such that $\kappa^{\prime}\vert_{\mathcal{O}_{\K}[[\Gamma_{2}]]}=\kappa\vert_{\mathcal{O}_{\K}[[\Gamma_{2}]]}$. Let $\boldsymbol{r},\boldsymbol{s}\in \mathbb{Z}^{2}$ 
be elements satisfying $\boldsymbol{s}\geq \boldsymbol{r}$ and $w_{\kappa,2}\in [r_{2},s_{2}]$. Then, $r_{(w_{\kappa,2},\phi_{\kappa,2})}$ induces an $\mathcal{O}_{\K}[[\Delta\times \Gamma_{1}]]$-module homomorphism
$$r_{m,(w_{\kappa,2},\phi_{\kappa,2})}^{[\boldsymbol{r},\boldsymbol{s}]}:
\frac{\mathcal{O}_{\K}[[(\Delta\times \Gamma_{1})\times \Gamma_{2}]]}{(\Omega_{(m,m_{\kappa,2})}^{[\boldsymbol{r},\boldsymbol{s}]})\mathcal{O}_{\K}[[(\Delta\times \Gamma_{1})\times \Gamma_{2}]]}
\rightarrow \frac{\mathcal{O}_{\K(\phi_{\kappa,2})}[[(\Delta\times \Gamma_{1})]]}{(\Omega_{m}^{[r_{1},s_{1}]})\mathcal{O}_{\K(\phi_{\kappa,2})}[[(\Delta\times \Gamma_{1})]]}$$
for each $m\in \mathbb{Z}_{\geq 0}$. We put $I_{h_{1},\K(\phi_{\kappa,2})}^{[r_{1},s_{1}]}=I_{h_{1}}^{[r_{1},s_{1}]}\otimes_{\K}\K(\phi_{\kappa,2})$ and define an $\mathcal{O}_{\K}[[\Delta\times \Gamma_{1}]]\otimes_{\mathcal{O}_{\K}}\K$-module homomorphism 
$$r_{(w_{\kappa,2},\phi_{\kappa,2})}^{[\boldsymbol{r},\boldsymbol{s}]}:I_{\boldsymbol{h}}^{[\boldsymbol{r},\boldsymbol{s}]}\otimes_{\mathcal{O}_{\K}[[\Gamma_{1}\times \Gamma_{2}]]}\mathcal{O}_{\K}[[(\Delta\times \Gamma_{1})\times\Gamma_{2}]]\rightarrow I_{h_{1},\K(\phi_{\kappa,2})}^{[r_{1},s_{1}]}\otimes_{\mathcal{O}_{\K(\phi_{\kappa,2})}[[\Gamma_{1}]]}\mathcal{O}_{\K(\phi_{\kappa,2})}[[(\Delta\times \Gamma_{1})]]$$
by setting $r_{(w_{\kappa,2},\phi_{\kappa,2})}^{[\boldsymbol{r},\boldsymbol{s}]}((s_{\boldsymbol{m}})_{\boldsymbol{m}\in \mathbb{Z}_{\geq 0}^{2}})=(r_{m,(w_{\kappa,2},\phi_{\kappa,2})}^{[\boldsymbol{r},\boldsymbol{s}]}(s_{(m,m_{\kappa,2})}))_{m\in \mathbb{Z}_{\geq 0}}$ for each $(s_{\boldsymbol{m}})_{\boldsymbol{m}\in \mathbb{Z}_{\geq 0}^{2}}\in I_{\boldsymbol{h}}^{[\boldsymbol{r},\boldsymbol{s}]}\otimes_{\mathcal{O}_{\K}[[\Gamma_{1}\times \Gamma_{2}]]}\mathcal{O}_{\K}[[(\Delta\times \Gamma_{1})\times \Gamma_{2}]]$. Let 
 \begin{multline*}
 \mathrm{pr}^{[\boldsymbol{r}^{(1)},\boldsymbol{s}^{(1)}]}_{[\boldsymbol{r}^{(2)},\boldsymbol{s}^{(2)}]}:I_{\boldsymbol{h}}^{[\boldsymbol{r}^{(1)},\boldsymbol{s}^{(1)}]}\otimes_{\mathcal{O}_{\K}[[\Gamma_{1}\times \Gamma_{2}]]}\mathcal{O}_{\K}[[(\Delta\times \Gamma_{1})\times \Gamma_{2}]]\\
 \rightarrow I_{\boldsymbol{h}}^{[\boldsymbol{r}^{(2)},\boldsymbol{s}^{(2)}]}\otimes_{\mathcal{O}_{\K}[[\Gamma_{1}\times \Gamma_{2}]]}\mathcal{O}_{\K}[[(\Delta\times \Gamma_{1})\times \Gamma_{2}]]
 \end{multline*}
be the projection for each $\boldsymbol{r}^{(i)},\boldsymbol{s}^{(i)}\in \mathbb{Z}_{\geq 0}^{2}$ such that $[\boldsymbol{r}^{(2)},\boldsymbol{s}^{(2)}]\subset [\boldsymbol{r}^{(1)},\boldsymbol{s}^{(1)}]\subset [\boldsymbol{d},\boldsymbol{e}]$ with $i=1,2$. By the definition of $s^{[\boldsymbol{r}^{(i)},\boldsymbol{s}^{(i)}]}$ with $i=1,2$, we have
\begin{equation}\label{lemma interpolation formula L(f,g) image of pro1s1r2s2}
\mathrm{pr}^{[\boldsymbol{r}^{(1)},\boldsymbol{s}^{(1)}]}_{[\boldsymbol{r}^{(2)},\boldsymbol{s}^{(2)}]}(s^{[\boldsymbol{r}^{(1)},\boldsymbol{s}^{(1)}]})=s^{[\boldsymbol{r}^{(2)},\boldsymbol{s}^{(2)}]}.
\end{equation}
Let $e_{\kappa,1}=k-w_{\kappa,2}-1$. By \eqref{lemma interpolation q[r,s]mQ eq} and \eqref{proposition of interpolation of L(f,G) interpolation form rkappa}, we see that $\kappa\vert_{\mathcal{O}_{\K}[[\Delta\times\Gamma_{1}]]} r_{(w_{\kappa,2},\phi_{\kappa,2})}(\linebreak\tilde{s}^{[(0,w_{\kappa,2}),(e_{\kappa,1},w_{\kappa,2})]}_{\boldsymbol{m}_{\kappa}})$ is equal to the right-hand side of \eqref{lemma interpolation L(f,G) eq}. Further, we have $\kappa(\tilde{s}_{(f,G),\boldsymbol{m}_{\kappa}})=\kappa\vert_{\mathcal{O}_{\K}[[(\Delta\times \Gamma_{1})]]}r_{(w_{\kappa,2},\phi_{\kappa,2})}(\tilde{s}_{(f,G)})$. Then, to prove that $\kappa(\tilde{s}_{(f,G),\boldsymbol{m}_{\kappa}})$ is equal to the right-hand side of \eqref{lemma interpolation L(f,G) eq}, it suffices to prove that 
\begin{equation}\label{lemma interpolation L(f,G) eq step2 sufficient rkappa}
r_{(w_{\kappa,2},\phi_{\kappa,2})}^{[(0,w_{\kappa,2}),(e_{\kappa,1},w_{\kappa,2})]}\circ\mathrm{pr}^{[\boldsymbol{d},\boldsymbol{e}]}_{[(0,w_{\kappa,2}),(e_{\kappa,1},w_{\kappa,2})]}(s_{(f,G)})=r_{(w_{\kappa,2},\phi_{\kappa,2})}^{[(0,w_{\kappa,2}),(e_{\kappa,1},w_{\kappa,2})]}(s^{[(0,w_{\kappa,2}),(e_{\kappa,1},w_{\kappa,2})]})
\end{equation}
in $I_{h_{1},\K(\phi_{\kappa,2})}^{[0,e_{\kappa,1}]}\otimes_{\mathcal{O}_{\K(\phi_{\kappa,2})}[[\Gamma_{1}]]}\mathcal{O}_{\K(\phi_{\kappa,2})}[[(\Delta\times \Gamma_{1})]]$.  As mentioned in \eqref{projection I is isom if e-dgeq h}, the projection $p^{[0,e_{\kappa,1}]}_{[0,\lfloor 2\alpha\rfloor]}:I_{h_{1},\K(\phi_{\kappa,2})}^{[0,e_{\kappa,1}]}\otimes_{\mathcal{O}_{\K(\phi_{\kappa,2})}[[\Gamma_{1}]]}\mathcal{O}_{\K(\phi_{\kappa,2})}[[(\Delta\times \Gamma_{1})]]\rightarrow I_{h_{1},\K(\phi_{\kappa,2})}^{[0,\lfloor 2\alpha\rfloor]}\otimes_{\mathcal{O}_{\K(\phi_{\kappa,2})}[[\Gamma_{1}]]}\mathcal{O}_{\K(\phi_{\kappa,2})}[[(\Delta\times \Gamma_{1})]]$ is an isomorphism. Then, to prove \eqref{lemma interpolation L(f,G) eq step2 sufficient rkappa}, it suffices to show that we have 
\begin{multline}\label{lemma interpolation L(f,G) eq step2 sufficient rkappa 2alphaver}
\mathrm{pr}^{[0,e_{\kappa,1}]}_{[0,\lfloor2\alpha\rfloor]}\circ r_{(w_{\kappa,2},\phi_{\kappa,2})}^{[(0,w_{\kappa,2}),(e_{\kappa,1},w_{\kappa,2})]}\circ\mathrm{pr}^{[\boldsymbol{d},\boldsymbol{e}]}_{[(0,w_{\kappa,2}),(e_{\kappa,1},w_{\kappa,2})]}(s_{(f,G)})
\\
=\mathrm{pr}^{[0,e_{\kappa,1}]}_{[0,\lfloor2\alpha\rfloor]}\circ r_{(w_{\kappa,2},\phi_{\kappa,2})}^{[(0,w_{\kappa,2}),(e_{\kappa,1},w_{\kappa,2})]}(s^{[(0,w_{\kappa,2}),(e_{\kappa,1},w_{\kappa,2})]}).
\end{multline}
It is easy to see that the following diagram is commutative:
\small 
\begin{equation}\label{lemma interpolation L(f,G) commutative diagram}
\hspace*{-1cm}
\xymatrix{
I^{(1)}\ar[d]_{\mathrm{pr}^{[(0,w_{\kappa,2}),(e_{\kappa,1},w_{\kappa,2})]}_{[(0,w_{\kappa,2}),(\lfloor 2\alpha\rfloor,w_{\kappa,2})]}}\ar[rrr]^{\ \ \ \ \ \ r_{(w_{\kappa,2},\phi_{\kappa,2})}^{[(0,w_{\kappa,2}),(e_{\kappa,1},w_{\kappa,2})]}}&&&I^{(3)}\ar[d]^{\mathrm{pr}^{[0,e_{\kappa,1}]}_{[0,\lfloor2\alpha\rfloor]}}\\
I^{(2)}\ar[rrr]^{\ \ \ \ \ \ r_{(w_{\kappa,2},\phi_{\kappa,2})}^{[(0,w_{\kappa,2}),(\lfloor 2\alpha\rfloor,w_{\kappa,2})]}}&&&I^{(4)}}
\end{equation}
\normalsize 
where $I^{(1)}=I_{\boldsymbol{h}}^{[(0,w_{\kappa,2}),(e_{\kappa,1},w_{\kappa,2})]}\otimes_{\mathcal{O}_{\K}[[\Gamma_{1}\times \Gamma_{2}]]}\mathcal{O}_{\K}[[(\Delta\times \Gamma_{1})\times\Gamma_{2}]]$, $I^{(2)}=I_{\boldsymbol{h}}^{[(0,w_{\kappa,2}),(\lfloor 2\alpha\rfloor,w_{\kappa,2})]}\linebreak\otimes_{\mathcal{O}_{\K}[[\Gamma_{1}\times \Gamma_{2}]]}\mathcal{O}_{\K}[[(\Delta\times \Gamma_{1})\times\Gamma_{2}]]$, $I^{(3)}=I_{h_{1},\K(\phi_{\kappa,2})}^{[0,e_{\kappa,1}]}\otimes_{\mathcal{O}_{\K(\phi_{\kappa,2})}[[\Gamma_{1}]]}\mathcal{O}_{\K(\phi_{\kappa,2})}[[(\Delta\times \Gamma_{1})]]$ and $I^{(4)}=I_{h_{1},\K(\phi_{\kappa,2})}^{[0,\lfloor2\alpha\rfloor]}\otimes_{\mathcal{O}_{\K(\phi_{\kappa,2})}[[\Gamma_{1}]]}\mathcal{O}_{\K(\phi_{\kappa,2})}[[(\Delta\times \Gamma_{1})]]$.
By \eqref{lemma interpolation formula L(f,g) image of pro1s1r2s2} and \eqref{lemma interpolation L(f,G) commutative diagram}, we have 
\begin{align}\label{lemma interpolation L(f,g) eq pro0eekappa1,0,lfloor21}
\begin{split}
&\mathrm{pr}^{[0,e_{\kappa,1}]}_{[0,\lfloor2\alpha\rfloor]}\circ r_{(w_{\kappa,2},\phi_{\kappa,2})}^{[(0,w_{\kappa,2}),(e_{\kappa,1},w_{\kappa,2})]}\circ\mathrm{pr}^{[\boldsymbol{d},\boldsymbol{e}]}_{[(0,w_{\kappa,2}),(e_{\kappa,1},w_{\kappa,2})]}(s_{(f,G)})\\
&=r_{(w_{\kappa,2},\phi_{\kappa,2})}^{[(0,w_{\kappa,2}),(\lfloor 2\alpha\rfloor,w_{\kappa,2})]}\circ\mathrm{pr}^{[(0,w_{\kappa,2}),(e_{\kappa,1},w_{\kappa,2})]}_{[(0,w_{\kappa,2}),(\lfloor 2\alpha\rfloor,w_{\kappa,2})]}\circ \mathrm{pr}^{[\boldsymbol{d},\boldsymbol{e}]}_{[(0,w_{\kappa,2}),(e_{\kappa,1},w_{\kappa,2})]}(s_{(f,G)})\\
&=r_{(w_{\kappa,2},\phi_{\kappa,2})}^{[(0,w_{\kappa,2}),(\lfloor 2\alpha\rfloor,w_{\kappa,2})]}\circ\mathrm{pr}^{[\boldsymbol{d},\boldsymbol{e}]}_{[(0,w_{\kappa,2}),(\lfloor 2\alpha\rfloor,w_{\kappa,2})]}(s_{(f,G)})
\end{split}
\end{align}
and 
\begin{align}\label{lemma interpolation L(f,g) eq pro0eekappa1,0,lfloor22}
\begin{split}
&\mathrm{pr}^{[0,e_{\kappa,1}]}_{[0,\lfloor2\alpha\rfloor]}\circ r_{(w_{\kappa,2},\phi_{\kappa,2})}^{[(0,w_{\kappa,2}),(e_{\kappa,1},w_{\kappa,2})]}(s^{[(0,w_{\kappa,2}),(e_{\kappa,1},w_{\kappa,2})]})\\
&=r_{(w_{\kappa,2},\phi_{\kappa,2})}^{[(0,w_{\kappa,2}),(\lfloor 2\alpha\rfloor,w_{\kappa,2})]}\circ\mathrm{pr}^{[(0,w_{\kappa,2}),(e_{\kappa,1},w_{\kappa,2})]}_{[(0,w_{\kappa,2}),(\lfloor 2\alpha\rfloor,w_{\kappa,2})]}(s^{[(0,w_{\kappa,2}),(e_{\kappa,1},w_{\kappa,2})]})\\
&=r_{(w_{\kappa,2},\phi_{\kappa,2})}^{[(0,w_{\kappa,2}),(\lfloor 2\alpha\rfloor,w_{\kappa,2})]}(s^{[(0,w_{\kappa,2}),(\lfloor 2\alpha\rfloor,w_{\kappa,2})]}).
\end{split}
\end{align}
By the definition of $s_{(f,G)}$, we have $\mathrm{pr}^{[\boldsymbol{d},\boldsymbol{e}]}_{[\boldsymbol{d}_,\boldsymbol{e}_{\alpha}]}(s_{(f,G)})=s^{[\boldsymbol{d},\boldsymbol{e}_{\alpha}]}$, where $\boldsymbol{e}_{\alpha}=(\lfloor 2\alpha\rfloor,\lfloor \alpha\rfloor +2)$. Then, by \eqref{lemma interpolation formula L(f,g) image of pro1s1r2s2}, we have
\begin{align}\label{lemma interpolation L(f,g) image s(f,G) by pro is s0,wkappalfloor}
\begin{split}
\mathrm{pr}^{[\boldsymbol{d},\boldsymbol{e}]}_{[(0,w_{\kappa,2}),(\lfloor 2\alpha\rfloor,w_{\kappa,2})]}(s_{(f,G)})&=\mathrm{pr}^{[\boldsymbol{d},\boldsymbol{e}_{\alpha}]}_{[(0,w_{\kappa,2}),(\lfloor 2\alpha\rfloor,w_{\kappa,2})]}\circ \mathrm{pr}^{[\boldsymbol{d},\boldsymbol{e}]}_{[\boldsymbol{d}_,\boldsymbol{e}_{\alpha}]}(s_{(f,G)})\\
&=\mathrm{pr}^{[\boldsymbol{d},\boldsymbol{e}_{\alpha}]}_{[(0,w_{\kappa,2}),(\lfloor 2\alpha\rfloor,w_{\kappa,2})]}(s^{[\boldsymbol{d},\boldsymbol{e}_{\alpha}]})\\
&=s^{[(0,w_{\kappa,2}),(\lfloor 2\alpha\rfloor,w_{\kappa,2})]}.
\end{split}
\end{align}
By \eqref{lemma interpolation L(f,g) eq pro0eekappa1,0,lfloor21}, \eqref{lemma interpolation L(f,g) eq pro0eekappa1,0,lfloor22} and \eqref{lemma interpolation L(f,g) image s(f,G) by pro is s0,wkappalfloor}, we have \eqref{lemma interpolation L(f,G) eq step2 sufficient rkappa 2alphaver}.
\\ 
$\mathbf{Step\ 3}$. We will prove that for every $\kappa\in \mathfrak{X}_{\mathcal{O}_{\K}[[(\Delta\times \Gamma_{1})\times \Gamma_{2}]]}^{[\boldsymbol{d},\boldsymbol{e}]}$ satisfying  $w_{\kappa,1}+w_{\kappa,2}<k$ and $w_{\kappa,1}+(\lfloor \alpha\rfloor+2)<k$, $\kappa(\tilde{s}_{(f,G),\boldsymbol{m}_{\kappa}})$ is equal to the right-hand side of \eqref{lemma interpolation L(f,G) eq}.
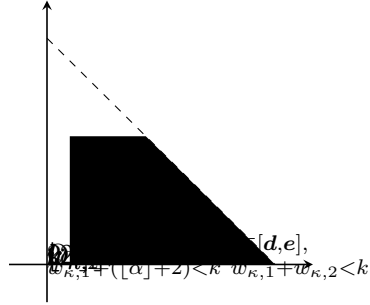
$$
\begin{tikzpicture}
 \draw[->,>=stealth,semithick] (-0.5,0)--(3.5,0)node[above]{$w_{\kappa,2}$}; 
 \draw[->,>=stealth,semithick] (0,-0.5)--(0,3.5)node[right]{$w_{\kappa,1}$}; 
 \draw (0,0)node[above  left]{O}; 
 \draw(1.3,0)node[below]{$\lfloor\alpha\rfloor+2$};
 \draw (0.3,0)node[below]{2};
 \draw(0,1)node[left]{$\lfloor 2\alpha\rfloor$};
 \draw (3,0)node[below]{$k$}; 
 \draw(0,3)node[left]{$k$};
  \fill[lightgray] (0.3,0)--(0.3,1.7)--(1.3,1.7)--(3,0);
 \draw[dashed,domain=0:3] plot(\x,3-\x);
  \draw (1.65,-0.5)node[below]{$\substack{\mathrm{Range\ of}\ \boldsymbol{w}_{\kappa}\in [\boldsymbol{d},\boldsymbol{e}],\\w_{\kappa,1}+(\lfloor \alpha\rfloor+2)<k\ w_{\kappa,1}+w_{\kappa,2}<k}$};
\end{tikzpicture}
$$
Let us fix an element $\kappa\in \mathfrak{X}_{\mathcal{O}_{\K}[[(\Delta\times \Gamma_{1})\times \Gamma_{2}]]}^{[\boldsymbol{d},\boldsymbol{e}]}$ satisfying 
$w_{\kappa,1}+w_{\kappa,2}<k$ and $w_{\kappa,1}+(\lfloor \alpha\rfloor+2)<k$. For each $(a_{1},a_{2})\in (\Delta\times \Gamma_{1})\times \Gamma_{2}$, we define a continuous $\mathcal{O}_{\K}[[\Gamma_{2}]]$-module homomorphism
\begin{equation}
r_{(w_{\kappa,1},\phi_{\kappa,1})}: \mathcal{O}_{\K}[[(\Delta\times \Gamma_{1})\times \Gamma_{2}]]\rightarrow \mathcal{O}_{\K(\phi_{\kappa,1})}[[\Gamma_{2}]]
\end{equation}
to be $r_{(w_{\kappa,1},\phi_{\kappa,1})}\vert_{(\Delta\times \Gamma_{1})\times\Gamma_{2}}((a_{1},a_{2}))=\kappa\vert_{\mathcal{O}_{\K}[[(\Delta\times \Gamma_{1})]}(a_{1})[a_{2}]$. In the same way as \eqref{proposition of interpolation of L(f,G) interpolation form rkappa}, we have
\begin{equation}\label{lemma interpolatino L(f,G) interpolation rwkappa1,phkappa1}
\kappa^{\prime}\vert_{\mathcal{O}_{\K}[[\Gamma_{2}]]}(r_{(w_{\kappa,1},\phi_{\kappa,1})}(s))=\kappa^{\prime}(s)
\end{equation}
for every $s\in \mathcal{O}_{\K}[[(\Delta\times \Gamma_{1})\times \Gamma_{2}]]$ and for every $\kappa^{\prime}\in \mathfrak{X}_{\mathcal{O}_{\K}[[(\Delta\times \Gamma_{1})\times \Gamma_{2}]]}$  such that $\kappa^{\prime}\vert_{\mathcal{O}_{\K}[[\Delta\times \Gamma_{1}]]}=\kappa\vert_{\mathcal{O}_{\K}[[\Delta\times \Gamma_{1}]]}$. Let $\boldsymbol{r},\boldsymbol{s}\in \mathbb{Z}^{2}$ such that $\boldsymbol{s}\geq \boldsymbol{r}$ and $w_{\kappa,1}\in [r_{1},s_{1}]$. Then, $r_{(w_{\kappa,1},\phi_{\kappa,1})}$ induces an $\mathcal{O}_{\K}[[\Gamma_{2}]]$-module homomorphism
$$r_{(w_{\kappa,1},\phi_{\kappa,1}),m}^{[\boldsymbol{r},\boldsymbol{s}]}:\frac{\mathcal{O}_{\K}[[(\Delta\times \Gamma_{1})\times \Gamma_{2}]]}{(\Omega_{(m_{\kappa,1},m)}^{[\boldsymbol{r},\boldsymbol{s}]})\mathcal{O}_{\K}[[(\Delta\times \Gamma_{1})\times \Gamma_{2}]]}\rightarrow \mathcal{O}_{\K(\phi_{\kappa,1})}[[\Gamma_{2}]]\slash (\Omega_{m}^{[r_{2},s_{2}]})$$
for each $m\in \mathbb{Z}_{\geq 0}$. We define an $\mathcal{O}_{\K}[[\Gamma_{2}]]\otimes_{\mathcal{O}_{\K}}\K$-module homomorphism 
$$r_{(w_{\kappa,1},\phi_{\kappa,1})}^{[\boldsymbol{r},\boldsymbol{s}]}:I_{\boldsymbol{h}}^{[\boldsymbol{r},\boldsymbol{s}]}\otimes_{\mathcal{O}_{\K}[[\Gamma_{1}\times \Gamma_{2}]]}\mathcal{O}_{\K}[[(\Delta\times \Gamma_{1})\times\Gamma_{2}]]\rightarrow I_{h_{2},\K(\phi_{\kappa,1})}^{[r_{2},s_{2}]}$$
to  be $r_{(w_{\kappa,1},\phi_{\kappa,1})}^{[\boldsymbol{r},\boldsymbol{s}]}((s_{\boldsymbol{m}})_{\boldsymbol{m}\in \mathbb{Z}_{\geq 0}^{2}})=(r_{(w_{\kappa,1},\phi_{\kappa,1}),m}^{[\boldsymbol{r},\boldsymbol{s}]}(s_{(m_{\kappa,1},m)}))_{m\in \mathbb{Z}_{\geq 0}}$ for each $(s_{\boldsymbol{m}})_{\boldsymbol{m}\in \mathbb{Z}_{\geq 0}^{2}}\in I_{\boldsymbol{h}}^{[\boldsymbol{r},\boldsymbol{s}]}\linebreak\otimes_{\mathcal{O}_{\K}[[\Gamma_{1}\times \Gamma_{2}]]}\mathcal{O}_{\K}[[(\Delta\times \Gamma_{1})\times \Gamma_{2}]]$. Let $e_{\kappa,2}=k-w_{\kappa,1}-1$. In the same way as \eqref{lemma interpolation L(f,G) eq step2 sufficient rkappa}, if we have 
\begin{equation}\label{lemma interpolation L(f,G) eq step3 sufficient rkappa}
r_{(w_{\kappa,1},\phi_{\kappa,1})}^{[(w_{\kappa,1},2),(w_{\kappa,1},e_{\kappa,2})]}\circ \mathrm{pr}^{[\boldsymbol{d},\boldsymbol{e}]}_{[(w_{\kappa,1},2),(w_{\kappa,1},e_{\kappa,2})]}(s_{(f,G)})=r_{(w_{\kappa,1},\phi_{\kappa,1})}^{[(w_{\kappa,1},2),(w_{\kappa,1},e_{\kappa,2})]}(s^{[(w_{\kappa,1},2),(w_{\kappa,1},e_{\kappa,2})]})
\end{equation}
in $I_{h_{2},\K(\phi_{\kappa,1})}^{[2,e_{\kappa,2}]}$, we see that $\kappa(\tilde{s}_{(f,G),\boldsymbol{m}_{\kappa}})$ is equal to the right-hand side of \eqref{lemma interpolation L(f,G) eq}. Then, we will prove that we have \eqref{lemma interpolation L(f,G) eq step3 sufficient rkappa}. Since $w_{\kappa,1}+(\lfloor \alpha\rfloor+2)<k$, we see that $\lfloor \alpha\rfloor+2\leq e_{\kappa,2}$. Then, as mentioned in \eqref{projection I is isom if e-dgeq h}, the projection $\mathrm{pr}^{[2,e_{\kappa,2}]}_{[2,\lfloor\alpha\rfloor+2]}:I_{h_{2},\K(\phi_{\kappa,1})}^{[2,e_{\kappa,2}]}\rightarrow I_{h_{2},\K(\phi_{\kappa,1})}^{[2,\lfloor\alpha\rfloor+2]}$ is an isomorphism. To prove \eqref{lemma interpolation L(f,G) eq step3 sufficient rkappa}, it suffices to prove that we have 
\begin{multline}\label{lemma interpolation L(f,G) eq step3 sufficient rkappa 2alphaver}
\mathrm{pr}^{[2,e_{\kappa,2}]}_{[2,\lfloor\alpha\rfloor+2]}\circ r_{(w_{\kappa,1},\phi_{\kappa,1})}^{[(w_{\kappa,1},2),(w_{\kappa,1},e_{\kappa,2})]}\circ \mathrm{pr}^{[\boldsymbol{d},\boldsymbol{e}]}_{[(w_{\kappa,1},2),(w_{\kappa,1},e_{\kappa,2})]}(s_{(f,G)})
\\ 
=\mathrm{pr}^{[2,e_{\kappa,2}]}_{[2,\lfloor\alpha\rfloor+2]}\circ r_{(w_{\kappa,1},\phi_{\kappa,1})}^{[(w_{\kappa,1},2),(w_{\kappa,1},e_{\kappa,2})]}(s^{[(w_{\kappa,1},2),(w_{\kappa,1},e_{\kappa,2})]})
\end{multline}
in $I_{h_{2},\K(\phi_{\kappa,1})}^{[2,\lfloor \alpha\rfloor+2]}$. 
In the same way as \eqref{lemma interpolation L(f,g) eq pro0eekappa1,0,lfloor21} and \eqref{lemma interpolation L(f,g) eq pro0eekappa1,0,lfloor22}, we have 
\begin{multline}\label{lemma interpolation L(f,G) step3 p2,ekappa2 ,0rw s(f,G)}
\mathrm{pr}^{[2,e_{\kappa,2}]}_{[2,\lfloor\alpha\rfloor+2]}\circ r_{(w_{\kappa,1},\phi_{\kappa,1})}^{[(w_{\kappa,1},2),(w_{\kappa,1},e_{\kappa,2})]}\circ \mathrm{pr}^{[\boldsymbol{d},\boldsymbol{e}]}_{[(w_{\kappa,1},2),(w_{\kappa,1},e_{\kappa,2})]}(s_{(f,G)})
\\ =r_{(w_{\kappa,1},\phi_{\kappa,1})}^{[(w_{\kappa,1},2),(w_{\kappa,1},\lfloor \alpha\rfloor+2)]}\circ \mathrm{pr}^{[\boldsymbol{d},\boldsymbol{e}]}_{[(w_{\kappa,1},2),(w_{\kappa,1},\lfloor \alpha\rfloor+2)]}(s_{(f,G)})
\end{multline}
and
\begin{align}\label{lemma interpolatin L(f,G) step3 p2swkappa1,2, swkappa1}
\begin{split}
&\mathrm{pr}^{[2,e_{\kappa,2}]}_{[2,\lfloor\alpha\rfloor+2]}\circ r_{(w_{\kappa,1},\phi_{\kappa,1})}^{[(w_{\kappa,1},2),(w_{\kappa,1},e_{\kappa,2})]}(s^{[(w_{\kappa,1},2),(w_{\kappa,1},e_{\kappa,2})]})\\
&= r_{(w_{\kappa,1},\phi_{\kappa,1})}^{[(w_{\kappa,1},2),(w_{\kappa,1},\lfloor \alpha\rfloor+2)]}\circ \mathrm{pr}^{[(w_{\kappa,1},2),(w_{\kappa,1},e_{\kappa,2})]}_{[(w_{\kappa,1},2),(w_{\kappa,1},\lfloor \alpha\rfloor+2)]}(s^{[(w_{\kappa,1},2),(w_{\kappa,1},e_{\kappa,2})]})\\
&= r_{(w_{\kappa,1},\phi_{\kappa,1})}^{[(w_{\kappa,1},2),(w_{\kappa,1},\lfloor \alpha\rfloor+2)]}(s^{[(w_{\kappa,1},2),(w_{\kappa,1},\lfloor \alpha\rfloor+2)]}).
\end{split}
\end{align}
By \eqref{lemma interpolation L(f,G) step3 p2,ekappa2 ,0rw s(f,G)} and \eqref{lemma interpolatin L(f,G) step3 p2swkappa1,2, swkappa1}, \eqref{lemma interpolation L(f,G) eq step3 sufficient rkappa 2alphaver} is equivalent to 
\begin{multline}\label{lemma interpolation L(f,G) last rwkapp1,phikappa1most left}
r_{(w_{\kappa,1},\phi_{\kappa,1})}^{[(w_{\kappa,1},2),(w_{\kappa,1},\lfloor \alpha\rfloor+2)]}\circ \mathrm{pr}^{[\boldsymbol{d},\boldsymbol{e}]}_{[(w_{\kappa,1},2),(w_{\kappa,1},\lfloor \alpha\rfloor+2)]}(s_{(f,G)})\\
=r_{(w_{\kappa,1},\phi_{\kappa,1})}^{[(w_{\kappa,1},2),(w_{\kappa,1},\lfloor \alpha\rfloor+2)]}(s^{[(w_{\kappa,1},2),(w_{\kappa,1},\lfloor \alpha\rfloor+2)]}).
\end{multline}
By the results of Step\ 1 and Step\ 2, we see that $\kappa^{\prime}(\tilde{s}_{(f,G),\boldsymbol{m}_{\kappa^{\prime}}})$ and $\kappa^{\prime}(\tilde{s}_{\boldsymbol{m}_{\kappa^{\prime}}}^{[(w_{\kappa,1},2),(w_{\kappa,1},\lfloor \alpha\rfloor+2)]})$ are equal to the right-hand side of \eqref{lemma interpolation L(f,G) eq} for every $\kappa^{\prime}\in \mathfrak{X}_{\mathcal{O}_{\K}[[(\Delta\times \Gamma_{1})\times \Gamma_{2}]]}^{[(w_{\kappa,1},2),(w_{\kappa,1},\lfloor \alpha\rfloor+2)]}$. Then, we see that
\begin{equation}\label{lemma interpolation L(f,G) step3 by setp12kappaprime equal}
\kappa^{\prime}(\tilde{s}_{(f,G),\boldsymbol{m}_{\kappa^{\prime}}})=\kappa^{\prime}(\tilde{s}_{\boldsymbol{m}_{\kappa^{\prime}}}^{[(w_{\kappa,1},2),(w_{\kappa,1},\lfloor \alpha\rfloor+2)]})
\end{equation}
for every $\kappa^{\prime}\in \mathfrak{X}_{\mathcal{O}_{\K}[[(\Delta\times \Gamma_{1})\times \Gamma_{2}]]}^{[(w_{\kappa,1},2),(w_{\kappa,1},\lfloor \alpha\rfloor+2)]}$. By \eqref{lemma interpolatino L(f,G) interpolation rwkappa1,phkappa1} and \eqref{lemma interpolation L(f,G) step3 by setp12kappaprime equal}, we have 
\begin{equation}\label{lasteqatino lemma ofsection6.1twovariableiwasawarakin}
\kappa^{\prime\prime}r_{(w_{\kappa,1},\phi_{\kappa,1})}(\tilde{s}_{(f,G),(m_{\kappa,1},m_{\kappa^{\prime\prime}})})=\kappa^{\prime\prime}r_{(w_{\kappa,1},\phi_{\kappa,1})}(\tilde{s}_{(m_{\kappa,1},m_{\kappa^{\prime\prime}})}^{[(w_{\kappa,1},2),(w_{\kappa,1},\lfloor \alpha\rfloor+2)]})
\end{equation}
for every $\kappa^{\prime\prime}\in \mathfrak{X}_{\mathcal{O}_{\K}[[\Gamma_{2}]]}^{[2,\lfloor \alpha\rfloor+2]}$. Since each element $s=(s_{m})_{m\in \mathbb{Z}_{\geq 0}}\in I_{h_{2},\K(\phi_{\kappa,1})}^{[2,\lfloor \alpha\rfloor+2]}$ is characterized by the specializations $\kappa^{\prime\prime}(\tilde{s}_{m_{\kappa^{\prime\prime}}})$ for every $\kappa^{\prime\prime}\in \mathfrak{X}_{\mathcal{O}_{\K}[[\Gamma_{2}]]}^{[2,\lfloor \alpha\rfloor+2]}$,  by \eqref{lasteqatino lemma ofsection6.1twovariableiwasawarakin}, we have \eqref{lemma interpolation L(f,G) last rwkapp1,phikappa1most left}.
\end{proof}

\subsection{Construction of a two-variable $p$-adic Rankin-Selberg $L$-series}\label{Construction of a two-variable padic rankin selber l series}
Let $\Gamma_{1}$ and $\Gamma_{2}$ be $p$-adic Lie groups which are isomorphic to $1+p\mathbb{Z}_{p}$. Set $\Delta=(\mathbb{Z}\slash p\mathbb{Z})^{\times}$. We fix continuous characters $\chi_{1}: \Delta\times \Gamma_{1}\rightarrow \mathbb{Q}_{p}^{\times}$ and $\chi_{2}: \Gamma_{2}\rightarrow \mathbb{Q}_{p}^{\times}$ which induce $\chi_{1}: \Delta\times \Gamma_{1}\stackrel{\sim}{\rightarrow}\mathbb{Z}_{p}^{\times}$ and $\chi_{i}:\Gamma_{i}\stackrel{\sim}{\rightarrow}1+p\mathbb{Z}_{p}$ for $i=1,2$. Let $\mathbf{I}$ be a finite free extension of $\mathcal{O}_{\K}[[\Gamma_{2}]]$. Let $N$ and $N^{\prime}$ be positive integers which are prime to $p$. Let $M$ be the least common multiple of $N$ and $N^{\prime}$. For each $\kappa\in \mathfrak{X}_{\mathcal{O}_{\K}[[\Delta\times\Gamma_{1}]]\widehat{\otimes}_{\mathcal{O}_{\K}}\mathbf{I}}$, we denote by $\phi_{\kappa,1}: \Delta\times \Gamma_{1}\rightarrow \overline{\K}^{\times}$ and $\phi_{\kappa,2}: \Gamma_{2}\rightarrow \overline{\K}^{\times}$ are finite characters which satisfy
\begin{equation}\label{subsectionconstrution two-variable rankin finitekappa1,2}
\kappa\vert_{(\Delta\times \Gamma_{1})\times \Gamma_{2}}((x_{1},x_{2}))=\phi_{\kappa,1}(x_{1})\chi_{1}(x_{1})^{w_{\kappa,1}}\phi_{\kappa,2}(x_{2})\chi_{2}(x_{2})^{w_{\kappa,2}}
\end{equation}
for each $(x_{1},x_{2})\in (\Delta\times \Gamma_{1})\times \Gamma_{2}$, where $\boldsymbol{w}_{\kappa}=(w_{\kappa,1},w_{\kappa,2})$ is the weight of $\kappa$. Further, we denote by $m_{\kappa,i}$ the smallest integer $m$ such that $\phi_{\kappa,i}$ factors through $\Gamma_{i}\slash \Gamma_{i}^{p^{m}}$ with $i=1,2$. Let $\xi$ be a Dirichlet character modulo $N^{\prime}p$. In this subsection, we prove the following theorem:
\begin{thm}\label{two variable rankin selberg l series of hida family}
Let $f\in S_{k}(Np^{m(f)},\psi;\K)$ be a normalized cuspidal Hecke eigenform which is new away from $p$ with $m(f)\in \mathbb{Z}_{\geq 1}$ and $G\in eS(N^{\prime}p,\xi;\mathbf{I})$ an $\mathbf{I}$-adic primitive Hida family of tame level $N^{\prime}$ and character $\xi$. Here, $\psi$ is a Dirichlet charactere modulo $Np^{m(f)}$ and $\mathbf{I}$ is a finite free extension of $\mathcal{O}_{\K}[[\Gamma_{2}]]$ such that $\mathbf{I}$ is an integral domain. Put $\boldsymbol{h}=(2\alpha,\alpha)$ with $\alpha=\ord_{p}(a_{p}(f))$. We assume the following conditions:
\begin{enumerate}
\item We have $k>\lfloor2\alpha\rfloor+\lfloor \alpha\rfloor +2$.
\item All $M$-roots of unity, the root number of $f^{0}$ and Fourier coefficients of $f^{0}$ are contained in $\K$, where $f^{0}$ is the primitive form associated with $f$.
\end{enumerate}
Let $\boldsymbol{d}=(0,2)$ and $\boldsymbol{e}=(k-3,k-1)$. We denote by $\xi_{(p)}$ the restriction of $\xi$ on $(\mathbb{Z}\slash p\mathbb{Z})^{\times}$. Then, there exists a unique element $L_{(f,G),p}\in \mathcal{D}_{\boldsymbol{h}}^{[\boldsymbol{d},\boldsymbol{e}]}(\Gamma_{1}\times \Gamma_{2},\K)\otimes_{\mathcal{O}_{\K}[[\Gamma_{1}\times\Gamma_{2}]]}(\mathcal{O}_{\K}[[\Delta \times \Gamma_{1}]]\widehat{\otimes}_{\mathcal{O}_{\K}}\mathbf{I})$ which satisfies
\begin{align}
\begin{split}\label{eq two variable rankin selberg l series of hida family}
\kappa(L_{(f,G),p})&={N^{\prime}}^{\frac{w_{\kappa,2}}{2}}\sqrt{-1}^{w_{\kappa,2}}(-1)^{w_{\kappa,1}}w^{\prime}(\kappa\vert_{\mathbf{I}}(G)^{0})G(\phi_{\kappa,1})G(\omega^{-w_{\kappa,2}}\xi_{(p)}\phi_{\kappa,1}\phi_{\kappa,2})\\
&\ \ \times E_{p,\phi_{\kappa,1}}(w_{\kappa,1}+w_{\kappa,2},f,\kappa\vert_{\mathbf{I}}(G))\frac{\Lambda\left(w_{\kappa,1}+w_{\kappa,2},f,
 \left(\kappa\vert_{\mathbf{I}}(G)\otimes\phi_{\kappa,1}\right)^{\rho}\right)}{\langle f^{0},f^{0}\rangle_{k,c_{f}}}
\end{split}
\end{align}
for every $\kappa\in \mathfrak{X}^{[\boldsymbol{d},\boldsymbol{e}]}_{\mathcal{O}_{\K}[[\Delta\times \Gamma_{1}]]\widehat{\otimes}_{\mathcal{O}_{\K}}\mathbf{I}}$ such that $w_{\kappa,1}+w_{\kappa,2}<k$, where $\kappa\vert_{\mathbf{I}}(G)^{0}$ is the primitive form associated with $\kappa\vert_{\mathbf{I}}(G)$, $\kappa\vert_{\mathbf{I}}(G)\otimes\phi_{\kappa,1}$ is the twist of $\kappa\vert_{\mathbf{I}}(G)$ by $\phi_{\kappa,1}$ and $c_{f}$ is the conductor of $f$, $w^{\prime}(\kappa\vert_{\mathbf{I}}(G)^{0})$ is the constant defined in \eqref{notp part of rootnumber}, $G(\phi_{\kappa,1})$ and $G(\omega^{-w_{\kappa,2}}\xi_{(p)}\phi_{\kappa,1}\phi_{\kappa,2})$ are the Gauss sums defined in \eqref{definition of gauss sum}, $E_{p,\phi_{\kappa,1}}(s,f,\kappa\vert_{\mathbf{I}}(G))$ is the Euler factor defined in \eqref{for comaptibilityof PC Euler factor} and $\Lambda\left(s,f,
 \left(\kappa\vert_{\mathbf{I}}(G)\otimes\phi_{\kappa,1}\right)^{\rho}\right)$ is the Rankin-Selberg $L$-series defined in \eqref{rankin attached primitives}. Here, $\phi_{\kappa,1}$ and $\phi_{\kappa,2}$ are finite characters defined in \eqref{subsectionconstrution two-variable rankin finitekappa1,2}.
\end{thm}
\begin{proof}
We can assume that $m(f)$ is the smallest positive integer $m$ such that $f\in S_{k}(Np^{m},\psi)$. Let $\alpha_{1},\ldots, \alpha_{n}$ be a basis of $\mathbf{I}$ over $\mathcal{O}_{\K}[[\Gamma_{2}]]$. By \eqref{eSIisom eSLambda}, we have an expression $G=\sum_{i=1}^{n}G_{i}\alpha_{i}\in eS(N^{\prime}p,\xi;\mathbf{I})$ with $G_{i}\in eS(N^{\prime}p,\xi;\mathcal{O}_{\K}[[\Gamma_{2}]])$. 
We define $L_{(f,G),p}\in \mathcal{D}_{\boldsymbol{h}}^{[\boldsymbol{d},\boldsymbol{e}]}(\Gamma_{1}\times\Gamma_{2},\K)\otimes_{\mathcal{O}_{\K}[[\Gamma_{1}\times\Gamma_{2}]]}(\mathcal{O}_{\K}[[\Delta\times \Gamma_{1}]]{\otimes}_{\widehat{\mathcal{O}}_{\K}}\mathbf{I})$ to be
$$L_{(f,G),p}=\sum_{i=1}^{n}\Psi(s_{(f,G_{i})})\alpha_{i},$$
where $s_{(f,G_{i})}\in I_{\boldsymbol{h}}^{[\boldsymbol{d},\boldsymbol{e}]}\otimes_{\mathcal{O}_{\K}[[\Gamma_{1}\times \Gamma_{2}]]}\mathcal{O}_{\K}[[(\Delta\times \Gamma_{1})\times \Gamma_{2}]]$ is the element defined in \eqref{definition of inverse q[d,ealpha]} and $\Psi: I_{\boldsymbol{h}}^{[\boldsymbol{d},\boldsymbol{e}]}\otimes_{\mathcal{O}_{\K}[[\Gamma_{1}\times \Gamma_{2}]]}\mathcal{O}_{\K}[[(\Delta\times \Gamma_{1})\times \Gamma_{2}]]\stackrel{\sim}{\rightarrow}\mathcal{D}_{\boldsymbol{h}}^{[\boldsymbol{d},\boldsymbol{e}]}(\Gamma_{1}\times \Gamma_{2},\K)\otimes_{\mathcal{O}_{\K}[[\Gamma_{1}\times\Gamma_{2}]]}\mathcal{O}_{\K}[[(\Delta\times \Gamma_{1})\times \Gamma_{2}]]$ is the isomorphism defined in \eqref{equation:multi-variable results on admissible distributions deformation ver}. Let $\kappa\in \mathfrak{X}^{[\boldsymbol{d},\boldsymbol{e}]}_{\mathcal{O}_{\K}[[\Delta\times \Gamma_{1}]]\widehat{\otimes}_{\mathcal{O}_{\K}}\mathbf{I}}$ such that $w_{\kappa,1}+w_{\kappa,2}<k$. By Lemma \ref{proposition of interpolation of L(f,G)}, we see that 
\begin{align}\label{main two variable proof 1}
\kappa(L_{(f,G),p})
& =\sum_{i=1}^{n}(-1)^{w_{\kappa,1}}\phi_{\kappa,1}(M\slash N^{\prime})l_{f,M}\circ T_{p}\left((\kappa\vert_{\mathcal{O}_{\K}[[\Gamma_{2}]]}(G_{i})\otimes\phi_{\kappa,1})\vert_{[M\slash N^{\prime}]}H_{\kappa}\right)\\
& =(-1)^{w_{\kappa,1}}\phi_{\kappa,1}(M\slash N^{\prime})l_{f,M}\circ T_{p}\left((\kappa\vert_{\mathbf{I}}(G)\otimes\phi_{\kappa,1})\vert_{[M\slash N^{\prime}]}H_{\kappa}\right) \notag 
\end{align}
where $H_{\kappa}$ is the nearly holomorphic modular form defined in \eqref{lemma interpolation L(f,G) eisenstein}, $l_{f,M}: \cup_{m=m(f)}^{+\infty}\linebreak N^{\leq \lfloor \frac{k-1}{2}\rfloor,\mathrm{cusp}}_{k}(Mp^{m},\psi;\K)\rightarrow \K$ is the $\K$-linear map defined in \eqref{classical lf map} and $T_{p}$ is the $p$-th Hecke operator, $\vert_{[M\slash N^{\prime}]}$ is the operator defined in \eqref{power series and qn mapst qNn}. Let $\beta_{\kappa}$ be the smallest positive integer $m$ so that $\kappa\vert_{\mathbf{I}}(G)\otimes\phi_{\kappa,1}\in S_{w_{\kappa,2}}(N^{\prime}p^{m},\xi\omega^{-w_{\kappa,2}}\phi_{\kappa,2}\phi_{\kappa,1}^{2})$. Let $\pi_{\kappa\vert_{\mathbf{I}}(G)}=\otimes_{l}\pi_{\kappa\vert_{\mathbf{I}}(G),l}$ be the automorphic representation attached to $\kappa\vert_{\mathbf{I}}(G)$. By \cite[Proposition 2.2]{ML Hsieh}, we see that $\pi_{\kappa\vert_{\mathbf{I}}(G),p}$ is the special representation $\chi\mathrm{St}$ attached to an unramified character $\chi$ or the  principal series $\pi(\chi,\chi^{\prime})$ attached to an unramified character $\chi$ and a character $\chi^{\prime}$. Then, $\pi_{\kappa\vert_{\mathbf{I}}(G),p}\otimes \phi_{\kappa,1}$ is the special representation $\chi\phi_{\kappa,1}\mathrm{St}$ or the princiapl series $\pi(\chi\phi_{\kappa,1},\chi^{\prime}\phi_{\kappa,1})$. If $m_{\kappa,2}\geq 1$, the conductor of $\pi_{\kappa\vert_{\mathbf{I}}(G),p}$ is equal to $m_{\kappa,2}+1$ and if $m_{\kappa,2}=0$, the conductor of $\pi_{\kappa\vert_{\mathbf{I}}(G),p}$ is equal to $1$ or $0$. Then, by the table in \cite[page 8]{Schmid02}, we have $\beta_{\kappa}\geq \max\{m_{\kappa_{1}},m_{\kappa_{2}}\}+1$. Since $H_{\kappa}$ is a modular form of level $Mp^{\max\{m_{\kappa_{1}},m_{\kappa_{2}},m(f)-1\}+1}$, we have
\begin{equation}\label{betakkageq mkappa1,mkappa2+1}
(\kappa\vert_{\mathbf{I}}(G)\otimes\phi_{\kappa,1})\vert_{[M\slash N^{\prime}]}H_{\kappa}\in N^{\leq \lfloor \frac{k-1}{2}\rfloor,\mathrm{cusp}}_{k}(Mp^{\max\{\beta_{\kappa},m(f)\}},\psi).
\end{equation}
We will prove that 
\begin{align}\label{main two variable proof 2}
& l_{f,M}\circ T_{p}\left((\kappa\vert_{\mathbf{I}}(G)\otimes\phi_{\kappa,1})\vert_{[\frac{M}{N^{\prime}}]}H_{\kappa}\right)\\
& = a_{p}(f)^{m(f)+1}(Np^{m(f)})^{1-\frac{k}{2}}(-1)^{w_{\kappa,2}}(N^{\prime})^{\frac{w_{\kappa,2}}{2}}2^{-k-1}M^{w_{\kappa,1}}(\sqrt{-1})^{k-w_{\kappa,2}} \notag \\
& \ \ \times w^{\prime}(\kappa\vert_{\mathbf{I}}(G)^{0})G(\phi_{\kappa,1})G(\omega^{-w_{\kappa,2}}\xi_{(p)}\phi_{\kappa,1}\phi_{\kappa,2})E_{p,\phi_{\kappa,1}}(w_{\kappa,1}+w_{\kappa,2},f,\kappa\vert_{\mathbf{I}}(G))\phi_{\kappa,1}(N^{\prime})\notag \\
& \ \ \times \frac{\Lambda\left(w_{\kappa,1}+w_{\kappa,2},f,
 \left(\kappa\vert_{\mathbf{I}}(G)\otimes\phi_{\kappa,1}\right)^{\rho}\right)}{\mathcal{E}(f)\langle f^{0},f^{0}\rangle_{k,c_{f}}} \notag 
 \end{align}
 where 
 $$\mathcal{E}(f)=
\begin{cases}(-1)^{k}w(f^{0})\ &\mathrm{if}\ f=f^{0},\\
p^{-\frac{k}{2}+1}a_{p}(f)\left(1-\frac{\psi_{0}(p)p^{k-1}}{a_{p}(f)^{2}}\right)\left(1-\frac{\psi_{0}(p)p^{k-2}}{a_{p}(f)^{2}}\right)(-1)^{k}w(f^{0})\ &\mathrm{if}\ f\neq f^{0} .
\end{cases}$$
Here, $w(f^{0})$ is the root number of $f^{0}$ and $\psi_{0}$ is the primitive Dirichlet character associated with $\psi$.

$\mathbf{Case}\ \beta_{\kappa}\geq m(f)$. Assume that $\beta_{\kappa}\geq m(f)$. Let $\tau_{m}=\begin{pmatrix}0&-1\\ m&0\end{pmatrix}$ for each $m\in \mathbb{Z}_{\geq 1}$. By \eqref{betakkageq mkappa1,mkappa2+1} and the assumption $\beta_{\kappa}\geq m(f)$, we have 
$$
l_{f,M}\circ T_{p}\left((\kappa\vert_{\mathbf{I}}(G)\otimes\phi_{\kappa,1})\vert_{[\frac{M}{N^{\prime}}]}H_{\kappa}\right)=l_{f,M}^{(\beta_{\kappa})}\circ T_{p}\left((\kappa\vert_{\mathbf{I}}(G)\otimes\phi_{\kappa,1})\vert_{[\frac{M}{N^{\prime}}]}H_{\kappa}\right), 
$$ 
where $l_{f,M}^{(\beta_{\kappa})}$ is the map defined in \eqref{definition of lf,L(n)}. By the definition of $l_{f,M}^{(\beta_{\kappa})}$ and $\eqref{hecke opeator trace Mp}$, we have
\begin{align}\label{main two variable 2.5}
\small 
\begin{split}
&l_{f,M}^{(\beta_{\kappa})}\circ T_{p}\left((\kappa\vert_{\mathbf{I}}(G)\otimes\phi_{\kappa,1})\vert_{[M\slash N^{\prime}]}W_{\kappa}\right)=a_{p}(f)^{-\beta_{\kappa}+m(f)}\\
&\frac{
\left\langle f^{\rho}\vert_{k}\tau_{Np^{m(f)}},\mathrm{Tr}_{Mp^{m(f)}\slash Np^{m(f)}}\circ T_{p}^{\beta_{\kappa}+1-m(f)}\left((\kappa\vert_{\mathbf{I}}(G)\otimes\phi_{\kappa,1})\vert_{[M\slash N^{\prime}]}H_{\kappa}\right)
\right\rangle_{k,Np^{m(f)}}}{\langle f^{\rho}\vert_{k}\tau_{Np^{m(f)}},f\rangle_{k,Np^{m(f)}}}\\
&=a_{p}(f)^{-\beta_{\kappa}+m(f)}(M\slash N)^{\frac{k}{2}-1}\frac{
\left\langle f^{\rho}\vert_{k}\tau_{Mp^{m(f)}},T_{p}^{\beta_{\kappa}+1-m(f)}\left((\kappa\vert_{\mathbf{I}}(G)\otimes\phi_{\kappa,1})\vert_{[M\slash N^{\prime}]}H_{\kappa}\right)
\right\rangle_{k,Mp^{m(f)}}}{\langle f^{\rho}\vert_{k}\tau_{ Np^{m(f)}},f\rangle_{k,Np^{m(f)}}}
\end{split}
\normalsize
\end{align}
where $\mathrm{Tr}_{Mp\slash Np}$ is the trace operator defined in \eqref{definition of trace operator}, $f^{\rho}=\sum_{n=1}^{\infty}\overline{a_{n}(f)}q^{n}$. By \eqref{relaTl anddiagl1} and \cite[Theorem 4.5.5]{Miyake89}, we see that 
\begin{multline*}
\left\langle f^{\rho}\vert_{k}\tau_{Mp^{m(f)}},T_{p}^{\beta_{\kappa}+1-m(f)}\left((\kappa\vert_{\mathbf{I}}(G)\otimes\phi_{\kappa,1})\vert_{[M\slash N^{\prime}]}H_{\kappa}\right)
\right\rangle_{k,Mp^{m(f)}}
\\ 
=a_{p}(f)p^{(\beta_{\kappa}-m(f))(\frac{k}{2}-1)}
\left\langle 
f^{\rho}\vert_{k}\tau_{Mp^{m(f)}}\begin{pmatrix}p^{\beta_{\kappa}-m(f)}&0\\0&1\end{pmatrix},(\kappa\vert_{\mathbf{I}}(G)\otimes\phi_{\kappa,1})\vert_{[M\slash N^{\prime}]}H_{\kappa}
\right\rangle_{k,Mp^{\beta_{\kappa}}}
\\
=a_{p}(f)p^{(\beta_{\kappa}-m(f))(\frac{k}{2}-1)}
\left\langle 
f^{\rho}\vert_{k}\tau_{Mp^{\beta_{\kappa}}},(\kappa\vert_{\mathbf{I}}(G)\otimes\phi_{\kappa,1})\vert_{[M\slash N^{\prime}]}H_{\kappa}
\right\rangle_{k,Mp^{\beta_{\kappa}}}
\end{multline*}
and by \eqref{main two variable 2.5} and Lemma \ref{lemma for modefied euler factor}, we have
\begin{align}\label{main two variable proof 2.7}
\begin{split}
&l_{f,M}\circ T_{p}\left((\kappa\vert_{\mathbf{I}}(G)\otimes\phi_{\kappa,1})\vert_{[\frac{M}{N^{\prime}}]}H_{\kappa}\right)\\
&=l_{f,M}^{(\beta_{\kappa})}\circ T_{p}\left((\kappa\vert_{\mathbf{I}}(G)\otimes\phi_{\kappa,1})\vert_{[M\slash N^{\prime}]}H_{\kappa}\right)\\
&=a_{p}(f)^{-\beta_{\kappa}+m(f)+1}(Mp^{\beta_{\kappa}-m(f)}\slash N)^{\frac{k}{2}-1}\frac{
\left\langle 
f^{\rho}\vert_{k}\tau_{Mp^{\beta_{\kappa}}},(\kappa\vert_{\mathbf{I}}(G)\otimes\phi_{\kappa,1})\vert_{[M\slash N^{\prime}]}H_{\kappa}
\right\rangle_{k,Mp^{\beta_{\kappa}}}}{\langle f^{\rho}\vert_{k}\tau_{ Np^{m(f)}},f\rangle_{k,Np^{m(f)}}}\\
&=a_{p}(f)^{-\beta_{\kappa}+m(f)+1}(Mp^{\beta_{\kappa}-m(f)}\slash N)^{\frac{k}{2}-1}\frac{
\left\langle 
f^{\rho}\vert_{k}\tau_{Mp^{\beta_{\kappa}}},(\kappa\vert_{\mathbf{I}}(G)\otimes\phi_{\kappa,1})\vert_{[M\slash N^{\prime}]}H_{\kappa}
\right\rangle_{k,Mp^{\beta_{\kappa}}}}{\mathcal{E}(f)\langle f^{0},f^{0}\rangle_{k,c_{f}}}.
\end{split}
\end{align}
 Since we have 
$$\left(\kappa\vert_{\mathbf{I}}(G)\otimes\phi_{\kappa,1}\right)\Big\vert_{w_{\kappa,2}}\tau_{N^{\prime}p^{\beta_{\kappa}}}\tau_{Mp^{\beta_{\kappa}}}
=(-1)^{w_{\kappa,2}}\left(\tfrac{M}{N^{\prime}}\right)^{\frac{w_{\kappa,2}}{2}}(\kappa\vert_{\mathbf{I}}(G)\otimes\phi_{\kappa,1})\vert_{[\frac{M}{N^{\prime}}]},
$$
by Lemma \ref{lemma for interpolation formula}, we see that
\begin{multline}\label{main two variable proof 3}
\Lambda_{Mp^{\beta_{\kappa}}}\left(w_{\kappa,1}+w_{\kappa,2},f,\left(\kappa\vert_{\mathcal{O}_{\K}[[\Gamma_{2}]]}(G)\otimes\phi_{\kappa,1}\right)\Big\vert_{w_{\kappa,2}}\tau_{N^{\prime}p^{\beta_{\kappa}}}\right)=(-1)^{w_{\kappa,2}}\left(\tfrac{M}{N^{\prime}}\right)^{\frac{w_{\kappa,2}}{2}}\\
\times 2^{k+1}(Mp^{\beta_{\kappa}})^{\frac{1}{2}(k-w_{\kappa,2}-2w_{\kappa,1}-2)}(\sqrt{-1})^{w_{\kappa,2}-k}\\
\times\left\langle f^{\rho}\vert_{k}\tau_{Mp^{\beta_{\kappa}}},(\kappa\vert_{\mathcal{O}_{\K}[[\Gamma_{2}]]}(G)\otimes\phi_{\kappa,1})\vert_{[M\slash N^{\prime}]}H_{\kappa}
\right\rangle_{k,Mp^{\beta_{\kappa}}}.
\end{multline}
By \eqref{main two variable proof 2.7}, \eqref{main two variable proof 3} and Lemma \ref{for comaptibilityof PC}, we have
\begin{align}\label{main two variable proof 4}
& l_{f,M}\circ T_{p}\left((\kappa\vert_{\mathbf{I}}(G)\otimes\phi_{\kappa,1})\vert_{[\frac{M}{N^{\prime}}]}H_{\kappa}\right)  \\ 
& =a_{p}(f)^{-\beta_{\kappa}+m(f)+1}(Np^{m(f)})^{1-\frac{k}{2}}(-1)^{w_{\kappa,2}}\left(\tfrac{N^{\prime}}{M}\right)^{\frac{w_{\kappa,2}}{2}} \notag  \\
& \ \ \times 2^{-k-1}(Mp^{\beta_{\kappa}})^{\frac{1}{2}(w_{\kappa,2}+2w_{\kappa,1})}(\sqrt{-1})^{k-w_{\kappa,2}}\notag \\
& \ \ \times \frac{\Lambda_{Mp^{\beta_{\kappa}}}\left(w_{\kappa,1}+w_{\kappa,2},f,\left(\kappa\vert_{\mathcal{O}_{\K}[[\Gamma_{2}]]}(G)\otimes\phi_{\kappa,1}\right)\Big\vert_{w_{\kappa,2}}\tau_{N^{\prime}p^{\beta_{\kappa}}}\right)}{\mathcal{E}(f)\langle f^{0},f^{0}\rangle_{k,c_{f}}} \notag \\
& =a_{p}(f)^{m(f)+1}(Np^{m(f)})^{1-\frac{k}{2}}(-1)^{w_{\kappa,2}}\left(N^{\prime}\right)^{\frac{w_{\kappa,2}}{2}}2^{-k-1}M^{w_{\kappa,1}}(\sqrt{-1})^{k-w_{\kappa,2}} \notag \\
& \ \ \times w^{\prime}(\kappa\vert_{\mathbf{I}}(G)^{0})G(\phi_{\kappa,1})G(\omega^{-w_{\kappa,2}}\xi_{(p)}\phi_{\kappa,1}\phi_{\kappa,2})E_{p,\phi_{\kappa,1}}(w_{\kappa,1}+w_{\kappa,2},f,\kappa\vert_{\mathbf{I}}(G))\phi_{\kappa,1}(N^{\prime}) \notag \\
& \ \ \times\frac{\Lambda\left(w_{\kappa,1}+w_{\kappa,2},f,
 \left(\kappa\vert_{\mathbf{I}}(G)\otimes\phi_{\kappa,1}\right)^{\rho}\right)}{\mathcal{E}(f)\langle f^{0},f^{0}\rangle_{k,c_{f}}}. \notag 
\end{align}
$\mathbf{Case}\ \beta_{\kappa}<m(f)$. We assume that $\beta_{\kappa}<m(f)$. By \eqref{betakkageq mkappa1,mkappa2+1} and the assumption $\beta_{\kappa}<m(f)$, the form $(\kappa\vert_{\mathbf{I}}(G)\otimes\phi_{\kappa,1})\vert_{[M\slash N^{\prime}]}H_{\kappa}$ is in $N^{\leq \lfloor \frac{k-1}{2}\rfloor,\mathrm{cusp}}_{k}(Mp^{m(f)},\psi)$. 
By \cite[Theorem 4.5.5]{Miyake89}, $\eqref{hecke opeator trace Mp}$ and Lemma \ref{lemma for modefied euler factor}, we have
\begin{align}\label{main two variable proof 5}
& l_{f,M}\circ T_{p}\left((\kappa\vert_{\mathbf{I}}(G)\otimes\phi_{\kappa,1})\vert_{[\frac{M}{N^{\prime}}]}H_{\kappa}\right)   \\ 
& =l_{f,M}^{(m(f))}\circ T_{p}\left((\kappa\vert_{\mathbf{I}}(G)\otimes\phi_{\kappa,1})\vert_{[\frac{M}{N^{\prime}}]}H_{\kappa}\right) \notag \\
& =a_{p}(f)(M\slash N)^{\frac{k}{2}-1}\frac{
\left\langle 
f^{\rho}\vert_{k}\tau_{Mp^{m(f)}},(\kappa\vert_{\mathbf{I}}(G)\otimes\phi_{\kappa,1})\vert_{[M\slash N^{\prime}]}H_{\kappa}
\right\rangle_{k,Mp^{m(f)}}}{\mathcal{E}(f)\langle f^{0},f^{0}\rangle_{k,c_{f}}}. \notag 
\end{align}
 Since we have 
$$\left(\kappa\vert_{\mathbf{I}}(G)\otimes\phi_{\kappa,1}\right)\Big\vert_{w_{\kappa,2}}\tau_{N^{\prime}p^{m(f)}}\tau_{Mp^{m(f)}}
=(-1)^{w_{\kappa,2}}\left(\tfrac{M}{N^{\prime}}\right)^{\frac{w_{\kappa,2}}{2}}(\kappa\vert_{\mathbf{I}}(G)\otimes\phi_{\kappa,1})\vert_{[\frac{M}{N^{\prime}}]},
$$
by Lemma \ref{lemma for interpolation formula}, we see that
\begin{align}\label{main two variable proof 6}
& \Lambda_{Mp^{m(f)}}\left(w_{\kappa,1}+w_{\kappa,2},f,\left(\kappa\vert_{\mathcal{O}_{\K}[[\Gamma_{2}]]}(G)\otimes\phi_{\kappa,1}\right)\Big\vert_{w_{\kappa,2}}\tau_{N^{\prime}p^{m(f)}}\right) \\ 
& =(-1)^{w_{\kappa,2}}\left(\tfrac{M}{N^{\prime}}\right)^{\frac{w_{\kappa,2}}{2}}
\times 2^{k+1}(Mp^{m(f)})^{\frac{1}{2}(k-w_{\kappa,2}-2w_{\kappa,1}-2)}(\sqrt{-1})^{w_{\kappa,2}-k} \notag\\ 
& \ \ \times \left\langle f^{\rho}\vert_{k}\tau_{Mp^{m(f)}},(\kappa\vert_{\mathcal{O}_{\K}[[\Gamma_{2}]]}(G)\otimes\phi_{\kappa,1})\vert_{[M\slash N^{\prime}]}H_{\kappa} \notag 
\right\rangle_{k,Mp^{m(f)}}.
\end{align}
 By Lemma \ref{lemma 1 of shimura rankin selberg}, we have
\begin{multline}\label{main two variable proof 7}
\Lambda_{Mp^{m(f)}}\left(w_{\kappa,1}+w_{\kappa,2},f,\left(\kappa\vert_{\mathcal{O}_{\K}[[\Gamma_{2}]]}(G)\otimes\phi_{\kappa,1}\right)\Big\vert_{w_{\kappa,2}}\tau_{N^{\prime}p^{m(f)}}\right)\\
=p^{\frac{1}{2}(\beta_{\kappa}-m(f))(2w_{\kappa,1}+w_{\kappa_{2}})}a_{p}(f)^{m(f)-\beta_{\kappa}}\\
\times\Lambda_{Mp^{m(f)}}\left(w_{\kappa,1}+w_{\kappa,2},f,\left(\kappa\vert_{\mathcal{O}_{\K}[[\Gamma_{2}]]}(G)\otimes\phi_{\kappa,1}\right)\Big\vert_{w_{\kappa,2}}\tau_{N^{\prime}p^{\beta_{\kappa}}}\right).
\end{multline}
By \eqref{main two variable proof 6}, \eqref{main two variable proof 7} and Lemma \ref{for comaptibilityof PC}, we have
\begin{align}\label{main two variable proof 8}
& \left\langle f^{\rho}\vert_{k}\tau_{Mp^{m(f)}},(\kappa\vert_{\mathcal{O}_{\K}[[\Gamma_{2}]]}(G)\otimes\phi_{\kappa,1})\vert_{[M\slash N^{\prime}]}H_{\kappa}
\right\rangle_{k,Mp^{m(f)}} \\ 
& =(-1)^{w_{\kappa,2}}\left(\tfrac{N^{\prime}}{M}\right)^{\frac{w_{\kappa,2}}{2}} \notag
2^{-(k+1)}(Mp^{m(f)})^{-\frac{1}{2}(k-w_{\kappa,2}-2w_{\kappa,1}-2)}(\sqrt{-1})^{k-w_{\kappa,2}}\\
& \ \ \times p^{-\frac{1}{2}m(f)(2w_{\kappa,1}+w_{\kappa_{2}})}a_{p}(f)^{m(f)}w^{\prime}(\kappa\vert_{\mathbf{I}}(G)^{0})G(\phi_{\kappa,1})G(\omega^{-w_{\kappa,2}}\xi_{(p)}\phi_{\kappa,1}\phi_{\kappa,2}) \notag \\
& \ \ \times E_{p,\phi_{\kappa,1}}(w_{\kappa,1}+w_{\kappa,2},f,\kappa\vert_{\mathbf{I}}(G))\phi_{\kappa,1}(N^{\prime})\Lambda\left(w_{\kappa,1}+w_{\kappa,2},f,
 \left(\kappa\vert_{\mathbf{I}}(G)\otimes\phi_{\kappa,1}\right)^{\rho}\right). \notag 
\end{align}
By \eqref{main two variable proof 5} and \eqref{main two variable proof 8}, we have 
\begin{align}\label{main two variable proof 9}
&l_{f,M}\circ T_{p}\left((\kappa\vert_{\mathbf{I}}(G)\otimes\phi_{\kappa,1})\vert_{[\frac{M}{N^{\prime}}]}H_{\kappa}\right)\\
& = a_{p}(f)^{m(f)+1}(Np^{m(f)})^{1-\frac{k}{2}}(-1)^{w_{\kappa,2}}(N^{\prime})^{\frac{w_{\kappa,2}}{2}}2^{-k-1}M^{w_{\kappa,1}}(\sqrt{-1})^{k-w_{\kappa,2}}
\notag \\
& \ \ \times w^{\prime}(\kappa\vert_{\mathbf{I}}(G)^{0})G(\phi_{\kappa,1})G(\omega^{-w_{\kappa,2}}\xi_{(p)}\phi_{\kappa,1}\phi_{\kappa,2})
E_{p,\phi_{\kappa,1}}(w_{\kappa,1}+w_{\kappa,2},f,\kappa\vert_{\mathbf{I}}(G))\phi_{\kappa,1}(N^{\prime}) \notag \\
& \ \ \times \frac{\Lambda\left(w_{\kappa,1}+w_{\kappa,2},f,
 \left(\kappa\vert_{\mathbf{I}}(G)\otimes\phi_{\kappa,1}\right)^{\rho}\right)}{\mathcal{E}(f)\langle f^{0},f^{0}\rangle_{k,c_{f}}}. \notag
 \end{align}
By \eqref{main two variable proof 4} and \eqref{main two variable proof 9}, we have \eqref{main two variable proof 2}. We define a group homomorphism 
$$\langle\ \rangle_{1}: \mathbb{Z}_{p}^{\times}\rightarrow \mathbb{Z}_{p}[[\Delta\times \Gamma_{1}]]^{\times}$$ 
to be $z\mapsto [\chi_{1}^{-1}(z)]$ for each $z\in \mathbb{Z}_{p}^{\times}$, where $[\ ]: \Delta\times \Gamma_{1}\rightarrow \mathbb{Z}_{p}[[\Delta\times \Gamma_{1}]]^{\times}$ is the 
tautological inclusion. We replace $L_{(f,G),p}$ with $L_{(f,G),p}a_{p}(f)^{-(m(f)+1)}(Np^{m(f)})^{\frac{k}{2}-1}2^{k-1}\sqrt{-1}^{-k}\mathcal{E}(f)\langle M\rangle_{1}^{-1}$. 
By \eqref{main two variable proof 1} and \eqref{main two variable proof 2}, $L_{(f,G),p}$ satisfies the following interpolation property: 
\begin{align*}
\kappa(L_{(f,G),p})
&={N^{\prime}}^{\frac{w_{\kappa,2}}{2}}\sqrt{-1}^{w_{\kappa,2}}(-1)^{w_{\kappa,1}}w^{\prime}(\kappa\vert_{\mathbf{I}}(G)^{0})G(\phi_{\kappa,1})G(\omega^{-w_{\kappa,2}}\xi_{(p)}\phi_{\kappa,1}\phi_{\kappa,2})\\ 
&\ \ \times E_{p,\phi_{\kappa,1}}(w_{\kappa,1}+w_{\kappa,2},f,\kappa\vert_{\mathbf{I}}(G))\frac{\Lambda\left(w_{\kappa,1}+w_{\kappa,2},f,
 \left(\kappa\vert_{\mathbf{I}}(G)\otimes\phi_{\kappa,1}\right)^{\rho}\right)}{\langle f^{0},f^{0}\rangle_{k,c_{f}}}
\end{align*}
for every $\kappa\in \mathfrak{X}^{[\boldsymbol{d},\boldsymbol{e}]}_{\mathcal{O}_{\K}[[\Delta\times \Gamma_{1}]]\widehat{\otimes}_{\mathcal{O}_{\K}}\mathbf{I}}$ 
satisfying $w_{\kappa,1}+w_{\kappa,2}<k$. The uniqueness of $L_{(f,G),p}$ follows from Proposition \ref{admissible distribution characterization prop}.

\end{proof}

\begin{rem}
Let $N$ and $N'$ be positive integers relatively prime to $p$. Let 
$\psi$ $($resp. $\xi )$ be a Dirichlet character modulo $N$ $($resp. $N^{\prime})$. Let $f\in S_{k}(N,\psi;\K)$ and $g\in S_{l}(N^{\prime},\xi;K)$ be primitive forms of weight $k$ and $l$. We assume that we have $k>l\geq 2$.  Assume that $g$ is ordinary at $p$ and the inequality $k>\lfloor 2\alpha\rfloor+\lfloor \alpha\rfloor +2$ 
is valid with $\alpha=\ord_{p}(a_{p}(f))$. Let $\alpha_{1}(f)$ and $\alpha_{2}(f)$ $($resp. $\alpha_{1}(g)$ and $\alpha_{2}(g))$ be the roots of the polynomial $X^{2}-a_{p}(f)X+\psi(p)p^{k-1}$ 
$($resp. $X^{2}-a_{p}(g)X+\xi(p)p^{l-1})$ satisfying $\ord_{p}(\alpha_{1}(f))\leq \ord_{p}(\alpha_{2}(f))$ $($resp. $\ord_{p}(\alpha_{1}(g))\leq \ord_{p}(\alpha_{2}(g)))$.

Let $G$ be the primitie Hida deformation which extends the primitie form $g$. 
By specializing the two-variable $p$-adic $L$-function $L_{(f,G),p}$ constructed in Theorem \ref{two variable rankin selberg l series of hida family} at the point 
$g$ of $G$, we obtain a one-variable $p$-adic $L$-function 
$L_{(f,g),p}\in \mathcal{D}_{2\alpha}^{[0,k-l-1]}(\Gamma_{1},\K)\otimes_{\mathcal{O}_{\K}[[\Gamma]]}\mathcal{O}_{\K}[[\Delta\times\Gamma_{1}]]$. Let $w(g)^{\prime}$ be the constant defined in \eqref{notp part of rootnumber}. By replacing $L_{(f,g),p}$ with ${N^{\prime}}^{-l\slash 2}\sqrt{-1}^{-l\slash 2}w^{\prime}(g)^{-1}L_{(f,G),p}$, we have the following interpolation property:
\begin{multline}\label{remark perrin coatseq}
\kappa(L_{(f,g),p})\\
=(-1)^{w_{\kappa}}G(\phi_{\kappa})^{2}\Phi_{p}\left(f,g,\phi_{\kappa}^{-1},l+w_{\kappa}\right)
\prod_{i=1}^{2}\left(\frac{p^{l+w_{\kappa}-1}}{\alpha_{1}(f)\alpha_{i}(g)^{\rho}}\right)^{r_{\kappa}}
\frac{\Lambda(l+w_{\kappa},f,g^{\rho}\otimes\phi_{\kappa}^{-1})
}{\langle f^{0},f^{0}\rangle_{k,c_{f}}}
\end{multline}
for every $\kappa\in \mathfrak{X}_{\mathcal{O}_{\K}[[\Delta\times \Gamma_{1}]]}^{[0,k-l-1]}$ where 
$$
\Phi_{p}\left(f,g,\phi_{\kappa}^{-1},s\right)=\prod_{i=1}^{2}(1-\alpha_{2}(f)\alpha_{i}(g)^{\rho}\phi_{\kappa,0}^{-1}(p)p^{-s})\prod_{j=1}^{2}\left(1-\left(\frac{p}{\alpha_{1}(f)\alpha_{j}(g)^{\rho}}\right)\phi_{\kappa,0}(p)p^{s-2}\right)$$
and 
$$r_{\kappa}=\begin{cases}m_{\kappa}+1\ &\mathrm{if}\ \phi_{\kappa}\ \mathrm{is\ not\ trivial},\\
0\ &\mathrm{if}\ \phi_{\kappa}\ \mathrm{is\ trivial}. 
\end{cases}$$
Here $\phi_{\kappa}$ is the unique finite character on $\Delta\times \Gamma_{1}$ which satisfies $\kappa\vert_{\Delta\times \Gamma}(x)=\phi_{\kappa}(x)\chi_{1}(x)^{w_{\kappa}}$, $m_{\kappa}$ is the smallest non-negative integer $m$ such that $\phi_{\kappa}$ factors through $\Gamma_{1}\slash \Gamma_{1}^{p^{m}}$. 
{We see that the interpolation formula of \eqref{remark perrin coatseq} of the one-variable $p$-adic $L$-function $L_{(f,g),p}$ is compatible with the Coates--Perrin-Riou's principal conjecture given in \cite[(4.14)]{CoatesPerrin}.}

\end{rem}
\begin{rem}
In Theorem \ref{two variable rankin selberg l series of hida family}, we constructed a two-variable $p$-adic $L$-function $L_{(f,G),p}$  
which is associated to a normalized cuspidal Hecke eigenform $f$ and an $\mathbf{I}$-adic Hida family $G$.  
\begin{enumerate}
\item 
By the reason related to the uniqueness and the construction of $L_{(f,G),p}$, 
we imposed the condition 
$$
k>\lfloor 2\alpha\rfloor+\lfloor \alpha\rfloor +2$$
where $k$ is the weight of the fixed cuspform $f$ and we set $\alpha=\ord_{p}(a_{p}(f))$ 
(see the proof of Theorem \ref{two variable rankin selberg l series of hida family}, \eqref{definition of inverse q[d,ealpha]} and Step 2 and Step 3 of the proof of Lemma \ref{proposition of interpolation of L(f,G)}).  At the moment, we do not know how much we can relax the 
above condition for the uniueness and the construction of $L_{(f,G),p}$.  
%
\item 
By the technical reason related to the construction of $L_{(f,G),p}$, we can only show that 
$L_{(f,G),p}\in \mathcal{D}_{\boldsymbol{h}}^{[\boldsymbol{d},\boldsymbol{e}]}(\Gamma_{1}\times \Gamma_{2},\K)\otimes_{\mathcal{O}_{\K}[[\Gamma_{1}\times\Gamma_{2}]]}(\mathcal{O}_{\K}[[\Delta \times \Gamma_{1}]]\widehat{\otimes}_{\mathcal{O}_{\K}}\mathbf{I})$ with $\boldsymbol{h}=(2\alpha,\alpha)$ where $\Delta=\left(\mathbb{Z}\slash p\mathbb{Z}\right)^{\times}$. 
At the moment, we do not know what should be the minimal $(h_{1},h_{2})\in \ord_{p}(\mathcal{O}_{\K}\backslash \{0\})^{2}$ so that we have 
$L_{(f,G),p}\in \mathcal{D}_{(h_{1},h_{2})}^{[\boldsymbol{d},\boldsymbol{e}]}(\Gamma_{1}\times \Gamma_{2},\K)\otimes_{\mathcal{O}_{\K}[[\Gamma_{1}\times\Gamma_{2}]]}(\mathcal{O}_{\K}[[\Delta \times \Gamma_{1}]]\widehat{\otimes}_{\mathcal{O}_{\K}}\mathbf{I})$. 
%
\end{enumerate}
\end{rem}
\begin{rem}
The two-variable $p$-adic $L$-function $L_{(f,G),p}$ associated to a non $p$-ordinary normalized cuspidal Hecke eigenform $f$ and 
a primitive $\mathbf{I}$-adic Hida family $G$ which we constructed in Theorem \ref{two variable rankin selberg l series of hida family} 
has the following triangular range of interpolation.  
$$
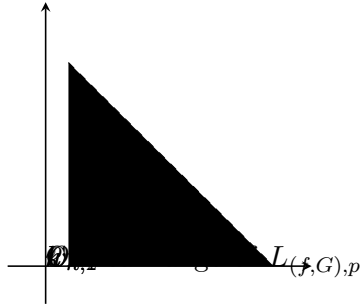
\begin{tikzpicture}
 \draw[->,>=stealth,semithick] (-0.5,0)--(3.5,0)node[above]{$w_{\kappa,2}$}; 
 \draw[->,>=stealth,semithick] (0,-0.5)--(0,3.5)node[right]{$w_{\kappa,1}$}; 
 \draw (0,0)node[above  left]{O}; 
 \draw (0.3,0)node[below]{2};
 \draw (3,0)node[below]{$k$}; 
 \draw(0,2.7)node[left]{$k-2$};
  \fill[lightgray] (0.3,0)--(0.3,2.7)--(3,0);
 \draw[dashed,domain=0.3:3] plot(\x,3-\x);
   \draw (1.65,-0.5)node[below]{$\mathrm{Critical\ range\ of}\ L_{(f,G),p}$};
\end{tikzpicture}
$$
Recall that a non $p$-ordinary normalized cuspidal Hecke eigenform $f$ of weight $k_0$ and level $Np$ such that $a_{p}(f)\not=0$ 
extends to a $p$-adic family $F=\{f_{k}\}_{\substack{k\in U\cap \mathbb{Z}\\ k>\alpha+1}}$ called a Coleman family where $U$ 
is a closed subdisk of $\mathbb{C}_{p}$ and $f_{k}$ is a normalized cuspidal Hecke eigenform of weight $k$ and level $Np$ 
such that $\mathrm{ord}_p (a_p (f)) = \mathrm{ord}_p (a_p (f_k ))$ for each $k \in U\cap \mathbb{Z} $ satisfying $k> \mathrm{ord}_p (a_p (f)) +1$ (see \cite{coleman 1997}). 
It is known that the Coleman family has a formal model $\sum A_n q^n \in \mathbf{A}_{\K}[[q]]$ where $\mathcal{K}$ is a $p$-adic field and $\mathbf{A}_{\K}=\mathcal{O}_{\K}[[\frac{T-k_{0}}{e_{0}}]]\otimes_{\mathcal{O}_{\K}}\K$ with $e_{0}\in \K^{\times}$ (see \cite[Thm 3.2]{NOR2023}). 
%
%
\par 
We expect that there exists a unique three-variable $p$-adic $L$-function $L_{(F,G),p}$ which coincides with the two-variable $p$-adic 
$L$-function $L_{(f_k ,G),p}$ when specialized to $f_k$ for every $k \in U\cap \mathbb{Z} $ satisfying $k> \lfloor 2\alpha\rfloor+\lfloor\alpha\rfloor +1$ with $\alpha=\ord_{p}(a_{p}(f))$. 
The expected range of interpolation of the three-variable $p$-adic $L$-function $L_{(F,G),p}$ 
is given as follows: 
$$
\left\{(k,\kappa)\in (U\cap \mathbb{Z})\times \mathfrak{X}_{\mathcal{O}_{\K}[[\Delta\times \Gamma_{1}]]\widehat{\otimes}_{\mathcal{O}_{\K}}\mathbf{I}}\ 
\Big\vert \  k>\alpha+1,\ 0\leq w_{\kappa,1}+w_{\kappa,2}<k,\ w_{\kappa_{2}}\geq 2\right\} . 
$$ 
Note that the above range of interpolation is unbounded and it will be constructed as an element of 
$\mathcal{D}_{\boldsymbol{h}}(\Gamma_{1}\times\Gamma_{2},\mathbf{A}_{\K})\otimes_{\mathcal{O}_{\K}[[\Gamma_{1}\times \Gamma_{2}]]}(\mathcal{O}_{\K}[[\Delta\times \Gamma_{1}]]\widehat{\otimes}_{\mathcal{O}_{\K}}\mathbf{I})$ given in $\eqref{not-bounded admissible distributions}$. 




\end{rem}


\begin{thebibliography}{BlRo94}
\bibitem{amicevelu}
Y.~Amice, J.~V\'{e}lu, 
\textit{Distributions p-adiques associees aux s\'{e}ries de Hecke}, 
Ast\'{e}risque 24-25, pp.~119--131, 1975.
\bibitem{Bell2021}
J.~Bella{\"\i}che, \textit{The eigenbook: Eigenvarieties, families of Galois representations, $p$-adic $L$-functions}, Pathways in Mathematics, Birkh\"{a}user/Springer, Cham, 2021.
\bibitem{BGR1984}
S.~Bosch, U.~G$\ddot{\mathrm{u}}$ntzer, R.~Remmert,
\textit{Non-Archimedian Analysis}, Springer-Verlag, 1984.
\bibitem{CoatesPerrin}
J. Coates, B. Perrin-Riou,\textit{On $p$-adic $L$-functions Attached to Motives over $\mathbb{Q}$}, Advanced Studies in Pure Math., 17, 23--54, 1989.
\bibitem{coleman 1997}
R.~Coleman, \textit{$p$-adic Banach spaces and families of modular forms}, Invent. math, 127, 417--479, 1997.
\bibitem{colmez2010}
P.~Colmez, 
\textit{Fonctions d'une variable $p$-adique}, 
Ast\'{e}risque, tome 330, pp.~13--59, 2010.
\bibitem{ML Hsieh}
M-L.~Hsieh,\textit{Hida families and $p$-adic triple product $L$-functions}, Amer. J. Math, 143, 411--532, 2021.
\bibitem{Hida1988}
H. Hida, \textit{A $p$-adic measure attached to the zeta functions associated with two elliptic modular forms. I\hspace{-1.2pt}I}, Ann. Inst. Fourier Grenoble 38,1988, no. 3, 1-83.
\bibitem{JaLa1970}
H.~Jacquet, R. P. Langlands, \textit{Automorphic forms on GL(2)}, Lecture Notes in Mathematics, Vol. 114, Springer Verlag, Berlin, 1970.
\bibitem{Jerome2013}
P. J\'{e}r\^{o}me,
\textit{Les espaces de Berkovich sont ang\'{e}liques}, Bull. Soc. math. France, 141, No2, pp.~267--297, 2013.
\bibitem{Miyake89}
T.\ Miyake, {\it Modular Forms}, Springer-Verlag, (1989).
\bibitem{NOR2023}
F.~Nuccio, T.~Ochiai, J.~Ray, 
\textit{A formal model of Coleman families and applications to Iwasawa invariants}, Annals mathematiques du Quebec, Volume 48, pages 453--475, (2024). 
\bibitem{Panchishkin03}
A.~A.~Panchishkin, \textit{Two-variable $p$-adic $L$-functions attached to eigenfamilies of positive slope}, Invent. math. 154, 551--615, 2003.
\bibitem{perrinriou1994} 
B.~Perrin-Riou, 
\textit{Th\'{e}orie d'Iwasawa des repr\'{e}sentations $p$-adiques sur un corps local}, Invent. Math. 115, No.1, pp.~81--149, 1994.
\bibitem{Schmid02}
R. Schmidt, \textit{Some remarks on local newforms for $\mathrm{GL}(2)$}, J. Ramanujan Math. Soc. 17, no. 2, 115--147, 2002.
\bibitem{Shimura1971}
G. Shimura, \textit{Introduction to the arithmetic theory of automorphic functions}, Tokyo Princeton: Iwanami Shoten and Princeton University Press, 1971.
\bibitem{Shimura1976}
G. Shimura,
\textit{The special values of the zeta functions associated with cusp forms}, Commun. Pure Appl. Math. 29, 783--804, 1976.
\bibitem{Shimura1977}
G. Shimura, \textit{On the periods of modular forms}, Math. Ann. 229, 211-221, 1977. 
\bibitem{shimura 07}G. Shimura, \textit{Elementary Dirichlet Series and Modular Forms}, Springer, Monographs in Mathematics, Springer, New York, 2007.
\bibitem{Shimura2012}
{G. Shimura, \textit{Modular Forms: Basis and Beyond}, Springer Monographs in Mathematics, Springer, New York, 2012.}
\bibitem{Urban2014}
E. Urban, \textit{Nearly overconvergent modular forms}, Iwasawa theory 2012, 401--441, Contrib. Math. Comput. Sci., 7, Springer, Heidelberg, 2014.
\bibitem{vishik1976} M.~Visik, 
\textit{Non-archimedean measures attached to Dirichlet series (Russian)},
Mat. Sb., N. Ser. 99(141), pp.~248--260, 1976.

\bibitem{washinton}L.C.~Washington,
\textit{Introduction to Cyclotomic Fields}, second edition, Grad. Texts in Math., 83,
Springer-Verlag, New York, 1997.
\end{thebibliography}
\end{document}